\documentclass[12pt,a4paper]{book}
\usepackage{fourier} 
\usepackage{combelow}
\usepackage[english]{babel}
\usepackage{color}
\usepackage{amsmath,amsthm}
\usepackage{amsfonts}
\usepackage{amssymb}
\usepackage{longtable}
\usepackage{hyperref}
\usepackage{natbib}
\usepackage{indentfirst}
\usepackage{graphicx}
\usepackage{varioref}
\usepackage{makeidx}
\usepackage{fancyhdr}
\usepackage{fancyhdr}
\usepackage{titletoc}
\usepackage{titlesec}
\newtheorem{thm}{Theorem}[chapter]
\newtheorem{cor}[thm]{Corollary}
\newtheorem{lem}[thm]{Lemma}
\newtheorem{prop}[thm]{Proposition}
\theoremstyle{definition}
\newtheorem{defn}[thm]{Def{}inition}

\newtheorem{exem}[thm]{Example}
\theoremstyle{remark}
\newtheorem{rem}[thm]{Remark}
\newtheorem{obs}[thm]{Observation}
\newtheorem{prog}[thm]{Program}
\newtheorem{func}[thm]{Function}
\newtheorem{apl}[thm]{Aplication}
\newtheorem{alg}[thm]{Algorthm}

\newcommand{\abs}[1]{\left\vert#1\right\vert}
\newcommand{\set}[1]{\left\{#1\right\}}
\newcommand{\Real}{\mathbb R}

\newcommand{\Na}{\mathbb{N}}
\newcommand{\Ns}{\mathbb{N}^*}
\newcommand{\Zs}{\mathbb{Z}^*}

\newcommand{\NP}[1]{\mathbb{P}_{\ge#1}}

\newcommand{\Ind}[1]{I_{#1}}
\newcommand{\md}[1]{\left(mod\ #1\right)}

\newcommand{\dsnd}{if and only if }

\newcommand{\desp}[2][\alpha]{\ensuremath{p_1^{#1_1}\cdot p_2^{#1_2}\cdots p_#2^{#1_{#2}}}}

\makeindex

\begin{document}
\title{Various Arithmetic Functions \\and their Applications}
\author{Octavian Cira and Florentin Smarandache}
\maketitle
\frontmatter

\pagenumbering{Roman}\setcounter{page}{1}
\chapter{Preface}

Over 300 sequences and many unsolved problems and conjectures related to them are presented herein. These notions, def{}initions, unsolved problems, questions, theorems corollaries, formulae, conjectures, examples, mathematical
criteria, etc. on integer sequences, numbers, quotients, residues, exponents, sieves, pseudo-primes squares cubes factorials, almost primes, mobile periodicals, functions, tables, prime square factorial bases, generalized factorials,
generalized palindromes, so on, have been extracted from the Archives of American Mathematics (University of Texas at Austin) and Arizona State University (Tempe): "\emph{The Florentin Smarandache papers}" special collections, University of Craiova Library, and Arhivele Statului (Filiala V\^alcea \& Filiala Dolj, Rom\^ania).

The book is based on various articles in the theory of numbers (starting from 1975), updated many times. Special thanks  to  C. Dumitrescu\index{Dumitrescu C.} and V. Seleacu\index{Seleacu V.}from the University of Craiova (see their edited book  "\emph{Some  Notions  and Questions  in  Number  Theory}", Erhus Press, Glendale, 1994), M. Bencze\index{Bencze M.}, L. Tu\-te\-scu\index{Tutescu L.}, E. Burton\index{Burton E.}, M. Coman\index{Coman M.}, F. Russo\index{Russo F.}, H.  Ibstedt\index{Ibstedt}, C.  Ashbacher\index{Ashbacher C.}, S. M. Ruiz\index{Ruiz S. M.},  J. Sandor\index{Sandor J.}, G.  Policarp\index{Policarp G.}, V. Iovan\index{Iovan V.}, N. Ivaschescu\index{Ivaschescu N.}, etc. who helped incollecting and editing this material.

This book was born from the collaboration of the two authors, which started in 2013. The f{}irst common work was the volume "\emph{Solving Diophantine Equations}", published in 2014. The contribution of the authors can be summarized as follows: Florentin Smarandache came with his extraordinary ability to propose new areas of study in number theory, and Octavian Cira -- with his algorithmic thinking and knowledge of Mathcad.

The work has been edited in \LaTeX.
\vspace{5mm}
\begin{flushleft}
\date{\today}
\end{flushleft}
\begin{flushright}
Authors
\end{flushright}

\newpage
\

{\small\tableofcontents}
\newpage\addcontentsline{toc}{chapter}{Contents}\label{Contens}

{\small\listoffigures}
\renewcommand\listfigurename{List of Figure}
\newpage\addcontentsline{toc}{chapter}{List of Figure}\label{ListaFig}

{\small\listoftables}
\renewcommand\listtablename{List of Table}
\newpage\addcontentsline{toc}{chapter}{List of Table}\label{ListaTabelelor}
\chapter{Introduction}

In this we will analyze other functions than the classical functions, \citep{Hardy+Wright2008}\index{Hardy G. H.}\index{Wright E. M.}:
\begin{itemize}
  \item Multiplicative and additive functions;
  \item $\Omega$, $\omega$, $\nu_p$ -- prime power decomposition;
  \item Multiplicative functions:
  \begin{itemize}
    \item $\sigma_k$, $\tau$, $d$ -- divisor sums,
    \item $\varphi$ -- Euler\index{Euler L.} totient function,
    \item $J_k$ -- Jordan\index{Jordan C.} totient function,
    \item $\mu$ -- M\"{o}bius\index{ M\"{o}bius A. F.} function,
    \item $\tau$ -- Ramanujan\index{Ramanujan S.} $\tau$ function,
    \item $c_q$ -- Ramanujan\rq{s}\index{Ramanujan S.} sum;
  \end{itemize}
  \item Completely multiplicative functions:
  \begin{itemize}
    \item $\lambda$ -- Liouville\index{Liouville J.} function,
    \item $\chi$ -- characters;
  \end{itemize}
  \item Additive functions:
  \begin{itemize}
    \item $\omega$ -- distinct prime divisors;
  \end{itemize}
  \item Completely additive functions:
  \begin{itemize}
    \item $\Omega$ -- prime divisors,
    \item $\nu_p$ -- prime power dividing $n$;
  \end{itemize}
  \item Neither multiplicative nor additive:
  \begin{itemize}
    \item $\pi$, $\Pi$, $\theta$, $\psi$ -- prime count functions,
    \item $\Lambda$ -- von Mangoldt\index{von Mangoldt H. H.} function,
    \item $p$ -- partition function,
    \item $\lambda$ -- Carmichael\index{Carmichael R. D.} function,
    \item $h$ -- Class number,
    \item $r_k$ -- Sum of $k$ squares;
  \end{itemize}
  \item Summation functions,
  \item Dirichlet\index{Dirichlet L.} convolution,
  \item Relations among the functions;
  \begin{itemize}
    \item Dirichlet\index{Dirichlet L.} convolutions,
    \item Sums of squares,
    \item Divisor sum convolutions,
    \item Class number related,
    \item Prime--count related,
    \item Menon\rq{s}\index{Menon P. K.} identity.
  \end{itemize}
\end{itemize}

In the book we have extended the following functions:
\begin{itemize}
  \item of counting the digits in base $b$,
  \item digits the number in base $b$,
  \item primes counting using Smarandache\rq{s} function,
  \item multifactorial,
  \item digital
      \begin{itemize}
        \item sum in base $b$,
        \item sum in base $b$ to the power $k$,
        \item product in base $b$,
      \end{itemize}
  \item divisor product,
  \item proper divisor product,
  \item $n$--multiple power free sieve,
  \item irrational root sieve,
  \item $n$--ary power sieve,
  \item $k$--ary consecutive sieve,
  \item consecutive sieve,
  \item prime part, square part, cubic part, factorial part, function part,
  \item primorial,
  \item Smarandache type functions: 
  \begin{itemize}
    \item Smarandache--Cira function of order $k$,
    \item Smarandache--Kurepa,
    \item Smarandache--Wagstaf{}f,
    \item Smarandache near to $k$--primorial,
    \item Smarandache ceil,
    \item Smarandache--Mersenne,
    \item Smarandache--X-na\-cci,
    \item pseudo--Smarandache,
    \item alternative pseudo--Smarandache,
    \item Smarandache functions of the $k$-th kind,
  \end{itemize}
  \item factorial for real numbers,
  \item analogues of the Smarandache,
  \item $m-$powers,
  \item and we have also introduced alternatives of them.
\end{itemize}

The next chapter of the book is dedicated to primes. Algorithms are presented: sieve of Eratosthenes, sieve of Sundram, sieve of Atkin. In the section dedicated to the criteria of primality, the Smarandache primality criterion is
introduced. The next section concentrates on Luhn prime numbers of f{}irst, second and third rank. The odd primes have the f{}inal digits 1, 3, 7 or 9. Another section studies the number of primes\rq{} f{}inal digits. The dif{}ference between two primes is called gap. It seems that gaps of length 6 are the most numerous. In the last section, we present the polynomial which generates primes.

The second chapter (the main chapter of the this book) is dedicated to arithmetical functions.

The third chapter is dedicated to numbers\rq{} sequences: consecutive sequen\-ce, circular sequence, symmetric sequence, deconstructive sequence, concate\-nated sequences, permutation sequence, combinatorial sequences.

The fourth chapter discusses special numbers. The f{}irst section presents numeration bases and Smarandache numbers, Smarandache quotients, primitive numbers, $m-$power residues, exponents of power $m$, almost prime, pseudo--primes,
permutation--primes, pseudo--squares, pseudo--cubes, pseudo--$m$--pow\-ers, pseudo--factorials, pseudo--divisors, pseudo--odd numbers, pseudo--tri\-an\-gu\-lar numbers, pseudo--even numbers, pseudo--multiples of $prime$, progressions, palindromes, Smarandache--Wellin primes.

The f{}ifth chapter treats about a series of numbers that have applicability in sciences.

The sixth chapter approximates some constants that are connected to the series proposed in this volume.

The carpet numbers are discussed in the seventh chapter, suggesting some algorithms to generate these numbers.

All open issues that have no conf{}irmation were included in the eighth chapter, \emph{Conjecture}.

The ninth chapter includes algorithms that generate series of numbers with some special properties.

The tenth chapter comprises some Mathcad documents that have been created for this volume. For the reader interested in a particular issue, we provide an adequate Mathcad document.

The book includes a chapter of \emph{Indexes}: notation, Mathcad utility functions used in this work, user functions that have been called in this volume, series generation programs and an index of names.

\mainmatter
\sloppy
\chapter{Prime Numbers}

\section{Generating Primes}

Generating primes can be obtained by means of several deterministic algorithms, known in the literature as sieves:
Eratosthenes\index{Eratosthenes}, Euler\index{Euler L.}, Sundaram\index{Sundaram S. P.}, Atkin\index{Atkin A. O. L.}, etc.

\subsection{Sieve of Eratosthenes }

The linear variant of the Sieve of Eratosthenes implemented by \cite{Pritchard1987}\index{Pritchard P.}, given by the code, has the inconvenience that it uselessly repeats operations.

Our implementation optimizes Pritchard\rq{s} algorithm, lowering to minimum the number of putting to zero in the vector \emph{is\_prime} and reducing to maximum the used memory. The speed of generating primes up to the limit $L$ is
remarkable.

\begin{prog}\label{ProgramSEPC} $SEPC$ (Sieve of Erathostenes, linear version of Prithcard, optimized of Cira) program of generating primes up to $L$.
\begin{tabbing}
  $SEPC(L):=$\=\vline\ $\lambda\leftarrow \emph{f{}loor}\left(\dfrac{L}{2}\right)$\\
  \>\vline\ $f$\=$or\ j\in1..\lambda$\\
  \>\vline\ \>\ $is\_prime_j\leftarrow1$\\
  \>\vline\ $prime\leftarrow(2\ 3\ 5\ 7)^\textrm{T}$\\
  \>\vline\ $i\leftarrow 5$\\
  \>\vline\ $f$\=$or\ j\in4,7..\lambda$\\
  \>\vline\ \>\ $is\_prime_j\leftarrow0$\\
  \>\vline\ $k\leftarrow3$\\
  \>\vline\ $w$\=$hile\ \ (prime_k)^2\le L$\\
  \>\vline\ \>\vline\ $f\leftarrow\dfrac{(prime_k)^2-1}{2}$\\
  \>\vline\ \>\vline\ $f$\=$or\ j\in f,f+\emph{prime}_k..\lambda$\\
  \>\vline\ \>\vline\ \>\ $is\_prime_j\leftarrow0$\\
  \>\vline\ \>\vline\ $s\leftarrow\dfrac{(prime_{k-1})^2+1}{2}$\\
  \>\vline\ \>\vline\ $f$\=$or\ j\in s..f$\\
  \>\vline\ \>\vline\ \>\ $if$\=\ $is\_prime_j\textbf{=}1$\\
  \>\vline\ \>\vline\ \>\ \>\vline\ $prime_i\leftarrow 2\cdot j+1$\\
  \>\vline\ \>\vline\ \>\ \>\vline\ $i\leftarrow i+1$\\
  \>\vline\ \>\vline\ $k\leftarrow k+1$\\
  \>\vline\ $j\leftarrow f$\\
  \>\vline\ $w$\=$hile\ \ j<\lambda$\\
  \>\vline\ \>\vline\ $if$\=$\ is\_prime_j\textbf{=}1$\\
  \>\vline\ \>\vline\ \>\vline\ $prime_i\leftarrow 2\cdot j+1$\\
  \>\vline\ \>\vline\ \>\vline\ $i\leftarrow i+1$\\
  \>\vline\ \>\vline\ $j\leftarrow j+1$\\
  \>\vline\ $\emph{return}\ \ prime$
\end{tabbing}

It is known that for $L<10^{10}$ the Erathostene\rq{s} sieve in the linear variant of Pritchard is the fastest primes\rq{} generator algorithm, \citep{Cira+Smarandache2014}. Then, the program $SEPC$, \ref{ProgramSEPC}, is more performant.
\end{prog}

\subsection{Sieve of Sundaram}

The Sieve of Sundaram\index{Sundaram S. P.} is a simple deterministic algorithm for f{}inding the primes up to a given natural number. This algorithm was presented by \citep{Sundaram+Aiyar1934}\index{Sundaram S. P.}. As it is known, the Sieve of Sundaram\index{Sundaram S. P.} uses $O(L\log(L))$ operations in order to f{}ind the primes up to $L$. The algorithm of the Sieve of Sundaram in Mathcad is:

\begin{prog}\label{ProgSundram}\
\begin{tabbing}
  $SS(L):=$\ \= \vline\ $m\leftarrow \emph{f{}loor}\left(\dfrac{L}{2}\right)$\\
  \>\vline\ $fo$\=$r\ k\in1..m$\\
  \>\vline\ \>\ $is\_prime_k\leftarrow1$\\
  \>\vline\ $fo$\=$r\ k\in1..m$\\
  \>\vline\ \>\ $fo$\=$r\ j\in 1..\emph{ceil}\left(\dfrac{m-k}{2\cdot k+1}\right)$\\
  \>\vline\ \>\ \>\ $is\_prime_{k+j+2\cdot k\cdot j}\leftarrow0$\\
  \>\vline\ $prime_1\leftarrow2$\\
  \>\vline\ $j\leftarrow1$\\
  \>\vline\ $fo$\=$r\ k\in1..m$\\
  \>\vline\ \>\ $if$\=$\ is\_prime_k\textbf{=}1$\\
  \>\vline\ \>\ \>\vline\ $j\leftarrow j+1$\\
  \>\vline\ \>\ \>\vline\ $prime_j\leftarrow2\cdot k+1$\\
  \>\vline\ $\emph{return}\ \ prime$
\end{tabbing}
\end{prog}

\subsection{Sieve of Atkin}

Until recently, i.e. till the appearance of the Sieve of Atkin\index{Atkin A. O. L.}, \citep{Atkin+Bernstein2004}\index{Bernstein D. J.}, the Sieve of Eratosthenes\index{Eratosthenes} was considered the most ef{}f{}icient algorithm that generates all the primes up to a limit $L>10^{10}$.

\begin{prog}\label{ProgAtkin}
\emph{SAOC} (Sieve of Atkin Optimized of Cira) program of generating primes up to $L$.
\begin{tabbing}
  $\emph{SAOC}(L):=$\ \= \vline\ $is\_prime_L\leftarrow0$\\
  \>\vline\ $\lambda\leftarrow \emph{f{}loor}\big(\sqrt{L}\big)$\\
  \>\vline\ $fo$\=$r\ j\in1..\emph{ceil}(\lambda)$\\
  \>\vline\ \>\vline\ $fo$\=$r\ k\in1..\emph{ceil}\left(\dfrac{\sqrt{L-j^2}}{2}\right)$\\
  \>\vline\ \>\vline\ \>\vline\ $n\leftarrow4k^2+j^2$\\
  \>\vline\ \>\vline\ \>\vline\ $m\leftarrow \mod(n,12)$\\
  \>\vline\ \>\vline\ \>\vline\ $is\_prime_n\leftarrow\neg is\_prime_n\ \ \emph{if}\ \ n\le L\wedge(m\textbf{=}1\vee m\textbf{=}5)$\\
  \>\vline\ \>\vline\ $fo$\=$r\ k\in1..\emph{ceil}\left(\sqrt{\dfrac{L-j^2}{3}}\right)$\\
  \>\vline\ \>\vline\ \>\vline\ $n\leftarrow3k^2+j^2$\\
  \>\vline\ \>\vline\ \>\vline\ $is\_prime_n\leftarrow\neg is\_prime_n\ \ \emph{if}\ \ n\le L\wedge \mod(n,12)\textbf{=}7$\\
  \>\vline\ \>\vline\ $fo$\=$r\ k\in j+1..\emph{ceil}\left(\sqrt{\dfrac{L+j^2}{3}}\right)$\\
  \>\vline\ \>\vline\ \>\vline\ $n\leftarrow3k^2-j^2$\\
  \>\vline\ \>\vline\ \>\vline\ $is\_prime_n\leftarrow\neg is\_prime_n\ \ \emph{if}\ \ n\le L\wedge \mod(n,12)\textbf{=}11$\\
  \>\vline\ $fo$\=$r\ j\in5,7..\lambda$\\
  \>\vline\ \>\ $fo$\=$r\ k\in1,3..\dfrac{L}{j^2}\ \ \emph{if}\ \ is\_prime_j$\\
  \>\vline\ \>\ \>\ $is\_prime_{k\cdot j^2}\leftarrow0$\\
  \>\vline\ $prime_1\leftarrow2$\\
  \>\vline\ $prime_2\leftarrow3$\\
  \>\vline\ $fo$\=$r\ n\in5,7..L$\\
  \>\vline\ \>\ $if$\=$\ is\_prime_n$\\
  \>\vline\ \>\ \>\vline\ $prime_j\leftarrow n$\\
  \>\vline\ \>\ \>\vline\ $j\leftarrow j+1$\\
  \>\vline\ $\emph{return}\ \ prime$
\end{tabbing}
\end{prog}

\section{Primality Criteria}

\subsection{Smarandache Primality Criterion}

Let $S:\mathbb{N}\to\mathbb{N}$ be Smarandache\index{Smarandache F.} function \citep{Sondow+Weisstein}, \citep{Smarandache1999,Smarandache1999a}, that gives the smallest value for a given $n$ at which $n\mid S(n)!$ (i.e.  $n$ divides $S(n)!$).

\begin{thm}\label{T1Smarandache}
  Let $n$ be an integer $n>4$. Then $n$ is prime if and only if $S(n)=n$.
\end{thm}
\begin{proof}
  See \cite[p. 31]{Smarandache1999a}.
\end{proof}
As seen in Theorem \ref{T1Smarandache}, we can use as primality test the computing of the value of $S$ function. For $n>4$, if relation $S(n)=n$ is satisf{}ied, it follows that $n$ is prime. In other words, the primes (to which number $4$ is added) are f{}ixed points for $S$ function. In this study we will use this primality test.

\begin{prog}\label{Program TS}
The program returns the value $0$ if the number is not prime and the value $1$ if the number is prime. File $\eta.prn$ (contains the values function Smarandache) is read and assigned to vector $\eta$~.
\[
 ORIGIN:=1\ \ \ \ \eta:=READPRN("\ldots\backslash\eta.prn")
\]
\begin{tabbing}
  $\emph{TS}(n):=$\=\ \vline\ $\emph{return}\ \ "\emph{Error.}\ n<1\ or\ not\ integer"\ \ \emph{if}\ \ n<1\vee n\neq trunc(n)$\\
  \>\ \vline\ $i$\=$f\ n>4$\\
  \>\ \vline\ \>\vline\ $\emph{return}\ \ 0\ \ \emph{if}\ \ \eta_n\neq n$\\
  \>\ \vline\ \>\vline\ $\emph{return}\ \ 1\ otherwise$\\
  \>\ \vline\ $o$\=$therwise$\\
  \>\ \vline\ \>\vline\ $\emph{return}\ \ 0\ \ \emph{if}\ \ n\textbf{=}1\vee n\textbf{=}4$\\
  \>\ \vline\ \>\vline\ $\emph{return}\ \ 1\ otherwise$
\end{tabbing}
By means of the program $\emph{TS}$, \ref{Program TS} was realized the following test.
\[
 n:=499999\ \ k:=1..n\ \ v_k:=2\cdot k+1
\]
\[
 last(v)=499999\ \ v_1=3\ \ v_{last(v)}=999999
\]
\[
 t_0:=time(0)\ \ w_k:=TS(v_k)\ \ t_1:=time(1)
\]
\[
 (t_1-t_0)sec=0.304s\ \ \ \sum w=78497~.
\]
The number of primes up to $10^6$ is $78798$, and the sum of non-zero components (equal to 1) is $78797$, as $2$ was not
counted as prime number because it is an even number.
\end{prog}

\section{Luhn primes}

The number 229 is the smallest prime which summed with its inverse gives also a prime. Indeed, 1151 is a prime, and $1151=229+922$. The f{}irst to note this special property of 229, on the website \emph{Prime Curios}, was Norman Luhn\index{Luhn N.} (9 Feb. 1999), \citep{Luhn2013,Caldwell+HonacherPrimeCurios}.

\begin{func}\label{FunctiaReverse} The function that returns the reverse of number $n_{(10)}$ in base $b$.
   \[
     \emph{Reverse}(n[,b])=\emph{sign}(n)\cdot \emph{reverse}(dn(\abs{n},b))\cdot \emph{Vb}(b,nrd(\abs{n},b))~,
   \]
  where $\emph{reverse}(v)$ is the Mathcad function that returns the inverse of vector $v$, $\emph{dn}(n,b)$ is the program \ref{ProgramDn} which returns the digits of number $n_{(10)}$ in numeration base $b$, $\emph{nrd}(n,b)$ is the function \ref{FunctionNrd} which returns the number of digits of the number $n_{(10)}$ in numeration base $b$, and the program $\emph{Vb}(b,m)$ returns the vector $(b^m\ b^{m-1}\ \ldots b^0)^\textrm{T}$. If the argument $b$ lacks when calling the function $\emph{Reverse}$ (the notation $[,b]$ shows that the argument is optional) we have a numeration base $10$.
\end{func}

\begin{defn}\label{DefinitiaNPLo}
   The primes $p$ for which $p^o+\emph{Reverse}(p)^o\in\NP{2}$ are called \emph{Luhn primes} of \emph{o} rank. We simply call \emph{Luhn primes} of f{}irst rank ($o=1$) simple \emph{Luhn primes}.
\end{defn}

\begin{prog}\label{ProgramLuhn} The $pL$ program for determining \emph{Luhn primes} of $o$ rank up to the limit $L$.
  \begin{tabbing}
    $\emph{pL}(o,L):=$\=\vline\ $\emph{return}\ \ "\emph{Error.}\ L<11"\ \ \emph{if}\ \ L<11$\\
    \>\vline\ $p\leftarrow \emph{SEPC}(L)$\\
    \>\vline\ $j\leftarrow1$\\
    \>\vline\ $f$\=$or\ k\in1..\emph{last}(p)$\\
    \>\vline\ \>\vline\ $d\leftarrow \emph{trunc}(p_k\cdot10^{-nrd(p_k,10)+1})$\\
    \>\vline\ \>\vline\ $i$\=$f\ d\textbf{=}2\vee d\textbf{=}4\vee d\textbf{=}6\vee d\textbf{=}8$\\
    \>\vline\ \>\vline\ \>\vline\ $q_j\leftarrow p_k$\\
    \>\vline\ \>\vline\ \>\vline\ $j\leftarrow j+1$\\
    \>\vline\ $j\leftarrow1$\\
    \>\vline\ $f$\=$or\ k\in1..\emph{last}(q)$\\
    \>\vline\ \>\vline\ $s\leftarrow (q_k)^o+\emph{Reverse}(q_k)^o$\\
    \>\vline\ \>\vline\ $i$\=$f\ TS(s)\textbf{=}1$\\
    \>\vline\ \>\vline\ \>\vline\ $v_j\leftarrow q_k$\\
    \>\vline\ \>\vline\ \>\vline\ $j\leftarrow j+1$\\
    \>\vline\ $\emph{return}\ \ v$
  \end{tabbing}
  where $TS$ is the primality Smarandache test, \ref{Program TS}, and $\emph{Reverse}(n)$ is the function \ref{FunctiaReverse}. In the f{}irst part of the program, the primes that have an odd digit as the f{}irst digit are dropped.
\end{prog}

\subsection{Luhn Primes of First Rank}

Up to $L<3\cdot10^4$ we have 321 \emph{Luhn primes} of f{}irst rank: 229, 239, 241, 257, 269, 271, 277, 281, 439, 443, 463, 467, 479, 499, 613, 641, 653, 661, 673, 677, 683, 691, 811, 823, 839, 863, 881, 20011, 20029, 20047, 20051, 20101, 20161, 20201, 20249, 20269, 20347, 20389, 20399, 20441, 20477, 20479, 20507, 20521, 20611, 20627, 20717, 20759, 20809, 20879, 20887, 20897, 20981, 21001, 21019, 21089, 21157, 21169, 21211, 21377, 21379, 21419, 21467, 21491, 21521, 21529, 21559, 21569, 21577, 21601, 21611, 21617, 21647, 21661, 21701, 21727, 21751, 21767, 21817, 21841, 21851, 21859, 21881, 21961, 21991, 22027, 22031, 22039, 22079, 22091, 22147, 22159, 22171, 22229, 22247, 22291, 22367, 22369, 22397, 22409, 22469, 22481, 22501, 22511, 22549, 22567, 22571, 22637, 22651, 22669, 22699, 22717, 22739, 22741, 22807, 22859, 22871, 22877, 22961, 23017, 23021, 23029, 23081, 23087, 23099, 23131, 23189, 23197, 23279, 23357, 23369, 23417, 23447, 23459, 23497, 23509, 23539, 23549, 23557, 23561, 23627, 23689, 23747, 23761, 23831, 23857, 23879, 23899, 23971, 24007, 24019, 24071, 24077, 24091, 24121, 24151, 24179, 24181, 24229, 24359, 24379, 24407, 24419, 24439, 24481, 24499, 24517, 24547, 24551, 24631, 24799, 24821, 24847, 24851, 24889, 24979, 24989, 25031, 25057, 25097, 25111, 25117, 25121, 25169, 25171, 25189, 25219, 25261, 25339, 25349, 25367, 25409, 25439, 25469, 25471, 25537, 25541, 25621, 25639, 25741, 25799, 25801, 25819, 25841, 25847, 25931, 25939, 25951, 25969, 26021, 26107, 26111, 26119, 26161, 26189, 26209, 26249, 26251, 26339, 26357, 26417, 26459, 26479, 26489, 26591, 26627, 26681, 26701, 26717, 26731, 26801, 26849, 26921, 26959, 26981, 27011, 27059, 27061, 27077, 27109, 27179, 27239, 27241, 27271, 27277, 27281, 27329, 27407, 27409, 27431, 27449, 27457, 27479, 27481, 27509, 27581, 27617, 27691, 27779, 27791, 27809, 27817, 27827, 27901, 27919, 28001, 28019, 28027, 28031, 28051, 28111, 28229, 28307, 28309, 28319, 28409, 28439, 28447, 28571, 28597, 28607, 28661, 28697, 28711, 28751, 28759, 28807, 28817, 28879, 28901, 28909, 28921, 28949, 28961, 28979, 29009, 29017, 29021, 29027, 29101, 29129, 29131, 29137, 29167, 29191, 29221, 29251, 29327, 29389, 29411, 29429, 29437, 29501, 29587, 29629, 29671, 29741, 29759, 29819, 29867, 29989~.

The number of \emph{Luhn primes} up to the limit $L$ is given in Table \ref{NumbersOfLuhnPrimes}:
\begin{table}[hb]
  \centering
  \begin{tabular}{|l|c|c|c|c|c|c|c|c|}
     \hline
     $L$ & $3\cdot10^2$ & $5\cdot10^2$ & $7\cdot10^2$ & $9\cdot10^2$ & $3\cdot10^4$ & $5\cdot10^4$ & $7\cdot10^4$ & $9\cdot10^4$ \\ \hline
      & 8 & 14 & 22 & 27 & 321 & 586 & 818 & 1078 \\
     \hline
   \end{tabular}
  \caption{Numbers of Luhn primes}\label{NumbersOfLuhnPrimes}
\end{table}

Up to the limit $L=2\cdot10^7$ the number of \emph{Luhn primes} is 50598, \citep{Cira+Smarandache2015}.

\subsection{Luhn Primes of Second Rank}

Are there \emph{Luhn primes} of second rank? Yes, indeed. 23 is a \emph{Luhn prime} number of second rank because 1553 is a prime and we have $1553=23^2+32^2$. Up to $3\cdot10^4$ we have 158 \emph{Luhn primes} of second rank: 23, 41, 227, 233, 283, 401, 409, 419, 421, 461, 491, 499, 823, 827, 857, 877, 2003, 2083, 2267, 2437, 2557, 2593, 2617, 2633, 2677, 2857, 2887, 2957, 4001, 4021, 4051, 4079, 4129, 4211, 4231, 4391, 4409, 4451, 4481, 4519, 4591, 4621, 4639, 4651, 4871, 6091, 6301, 6329, 6379, 6521, 6529, 6551, 6781, 6871, 6911, 8117, 8243, 8273, 8317, 8377, 8543, 8647, 8713, 8807, 8863, 8963, 20023, 20483, 20693, 20753, 20963, 20983, 21107, 21157, 21163, 21383, 21433, 21563, 21587, 21683, 21727, 21757, 21803, 21863, 21937, 21997, 22003, 22027, 22063, 22133, 22147, 22193, 22273, 22367, 22643, 22697, 22717, 22787, 22993, 23057, 23063, 23117, 23227, 23327, 23473, 23557, 23603, 23887, 24317, 24527, 24533, 24547, 24623, 24877, 24907, 25087, 25237, 25243, 25453, 25523, 25693, 25703, 25717, 25943, 26053, 26177, 26183, 26203, 26237, 26357, 26407, 26513, 26633, 26687, 26987, 27043, 27107, 27397, 27583, 27803, 27883, 28027, 28297, 28513, 28607, 28643, 28753, 28807, 29027, 29063, 29243, 29303, 29333, 29387, 29423, 29537, 29717, 29983~.

\begin{prop}
  The digit of unit for sum $q^2+\emph{Reverse}(q)^2$ is 3 or 7 for all numbers $q$ \emph{Luhn primes} of second rank.
\end{prop}
\begin{proof}
  The square of an even number is an even number, the square of an odd number is an odd number. The sum of an odd number with an even number is an odd number. If $q$ is a \emph{Luhn prime} number of second rank, then its reverse must be necessarily an even number, because $\sigma=q^2+Reverse(q)^2$ an odd number. The prime $q$ has unit digit 1, 3, 7 or 9 and $\emph{Reverse}(q)$ will obligatory have the digit unit, 2, 4, 6 or 8. Then, $q^2$ will have the unit digit, respectively 1, 9, 9 or 1, and unit digit of $\emph{Reverse}(q)^2$ will be respectively 4, 6, 6 or 4. Then, we consider all possible combinations of summing units $1+4=5$, $1+6=7$, $9+4=13$, $9+6=15$. Sums ending in 5 does not suit, therefore only endings of sum $\sigma$ that does suit, as $\sigma$ can eventually be a prime, are 3 or 7.
\end{proof}

This sentence can be used to increase the speed determination algorithm of \emph{Luhn prime} numbers of second rank, avoiding the primality test or sums $\sigma=q^2+\emph{Reverse}(q)^2$ which end in digit 5.

Up to the limit $L=3\cdot10^4$ \emph{Luhn prime} numbers of $o$ rank, $o=3$ were not found.

Up to $3\cdot10^4$ we have 219 \emph{Luhn prime} numbers of $o$ rank, $o=4$: 23, 43, 47, 211, 233, 239, 263, 419, 431, 487, 491, 601, 683, 821, 857, 2039, 2063, 2089, 2113, 2143, 2203, 2243, 2351, 2357, 2377, 2417, 2539, 2617, 2689, 2699, 2707, 2749, 2819, 2861, 2917, 2963, 4051, 4057, 4127, 4129, 4409, 4441, 4481, 4603, 4679, 4733, 4751, 4951, 4969, 4973, 6053, 6257, 6269, 6271, 6301, 6311, 6353, 6449, 6547, 6551, 6673, 6679, 6691, 6803, 6869, 6871, 6947, 6967, 8081, 8123, 8297, 8429, 8461, 8521, 8543, 8627, 8731, 8741, 8747, 8849, 8923, 8951, 8969, 20129, 20149, 20177, 20183, 20359, 20369, 20593, 20599, 20639, 20717, 20743, 20759, 20903, 20921, 21017, 21019, 21169, 21211, 21341, 21379, 21419, 21503, 21611, 21613, 21661, 21727, 21803, 21821, 21841, 21881, 21893, 21929, 21937, 22031, 22073, 22133, 22171, 22277, 22303, 22343, 22349, 22441, 22549, 22573, 22741, 22817, 22853, 22877, 22921, 23029, 23071, 23227, 23327, 23357, 23399, 23431, 23531, 23767, 23827, 23917, 23977, 24019, 24023, 24113, 24179, 24197,24223, 24251, 24421, 24481, 24527, 24593, 24659, 24683, 24793, 25171, 25261, 25303, 25307, 25321, 25343, 25541, 25643, 25673, 25819, 25873, 25969, 26083, 26153, 26171, 26267, 26297, 26561, 26833, 26839, 26953, 26993, 27103, 27277, 27337, 27427, 27551, 27617, 27749, 27751, 27791, 27823, 27901, 27919, 27953, 28019, 28087, 28211, 28289, 28297, 28409, 28547, 28631, 28663, 28723, 28793, 28813, 28817, 28843, 28909, 28927, 28949, 28979, 29063, 29173, 29251, 29383, 29663, 29833, 29881, 29989~.

The number $23^4+32^4\rightarrow1328417$ is a prime, the number $43^4+34^4\rightarrow4755137$ is a prime, \ldots, the number $29989^4+98992^4\rightarrow96837367848621546737$ is a prime.

\begin{rem}
  Up to $3\cdot10^4$, the numbers: 23, 233, 419, 491, 857, 2617, 4051, 4129, 4409, 4481, 6301, 6551, 6871, 8543, 21727, 21803, 21937, 22133, 23227, 23327, 24527, 28297, 29063 are \emph{Luhn prime} numbers of 2nd and 4th rank.
\end{rem}

Questions:
\begin{enumerate}
  \item There are an inf{}inite number of \emph{Luhn primes} of f{}irst rank?
  \item There are an inf{}inite number of \emph{Luhn primes} of second rank?
  \item There are \emph{Luhn primes} of third rank?
  \item There are an inf{}inite number of \emph{Luhn primes} of fourth rank?
  \item There are Luhn prime numbers of $o$ rank, $o>4$?
\end{enumerate}

\section{Endings the Primes}

Primes, with the exception of 2 and 5, have their unit digit equal to 1, 3, 7 or 9. It has been counting the f{}irst 200,000 primes, i.e. from 2 to $prime_{200000}=2750159$. The f{}inal units 3 and 7 "dominate" the f{}inal digits 1 and 9 to primes.

For this observation, we present the counting program for f{}inal digits on primes (Number of appearances as the Final Digit):
\begin{prog}\label{Programnfd} Program for counting the f{}inal digits on primes.
  \begin{tabbing}
    $\emph{nfd}(k_{max}):=$\ \vline\=\ $\emph{return}\ \ "\emph{Error.}\ k_{max}>last(prime)"\ \ \emph{if}\ \ k_{max}>last(prime)$\\
    \>\vline\ $nrd\leftarrow(0\ 1\ 0\ 0\ 0\ 0\ 0\ 0\ 0)$\\
    \>\vline\ $f$\=$or\ k\in2..k_{max}$\\
    \>\vline\ \>\ \vline\ $fd\leftarrow prime_k-\emph{Trunc}(prime_k,10)$\\
    \>\vline\ \>\ \vline\ $f$\=$or\ j=1..9$\\
    \>\vline\ \>\ \vline\ \>\ \vline\ $nrd_{k,j}\leftarrow nrd_{k-1,j}+1\ \ \emph{if}\ \ j\textbf{=}fd$\\
    \>\vline\ \>\ \vline\ \>\ \vline\ $nrd_{k,j}\leftarrow nrd_{k-1,j}\ \ otherwise$\\
    \>\vline\ $\emph{return}\ \ nrd$
   \end{tabbing}

Calling the program \emph{nrfd=nfd}$(2\cdot10^5)$, it provides a matrix of 9 columns and 200,000 lines. The component \emph{nrfd}$_{k,1}$ give us the number of primes that have digit unit 1 to the limit $prime_k$, the component \emph{nrfd}$_{k,3}$ provides the number of primes that have the unit digit 3 up to the limit $prime_k$, and so on. Therefore, we have \emph{nrfd}$_{15,1}=3$ (11, 31 and 41 have digit unit 1), \emph{nrfd}$_{15,3}=4$ (3, 13, 23 and 43 have digit unit 3), \emph{nrfd}$_{15,7}=4$ (7, 17, 37 and 47 have digit unit 7) and \emph{nrfd}$_{15,9}=2$ (19 and 29 have digit unit 9). We recall that $prime_{15}=47$.
\end{prog}

For presenting the fact that the digits 3 and 7 "\emph{dominates}" the digits 1 and 9, we use the function $\emph{umd}:\Ns\to\Ns$ (Upper limit of Mean Digits appearances), $\emph{umd}(L)=\left\lceil\frac{L}{4}\right\rceil$. This function is the superior limit of the mean of digits 1, 3, 7 appearances and 9 as endings of primes (exception for 2 and 5 which occur each only once).

In Figure \ref{nrfd-umd} we present the graphs of functions: \emph{nrfd}$_{k,1}-\emph{umd}(k)$ (red), \emph{nrfd}$_{k,3}-\emph{umd}(k)$ (blue), \emph{nrfd}$_{k,7}-\emph{umd}(k)$ (green) and \emph{nrfd}$_{k,9}-\emph{umd}(k)$ (black).

\begin{figure}[h]
  \centering
  \includegraphics[scale=0.48]{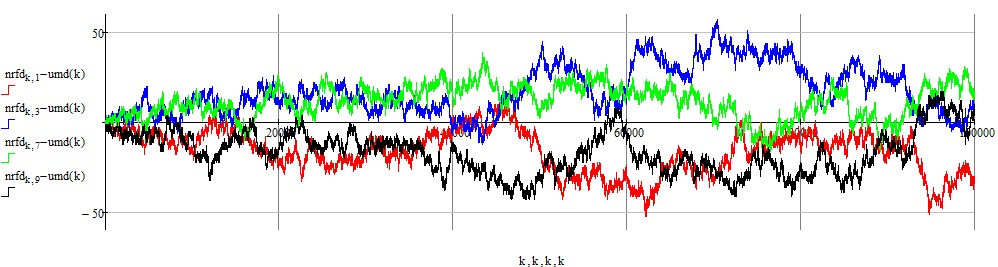}\\
  \includegraphics[scale=0.48]{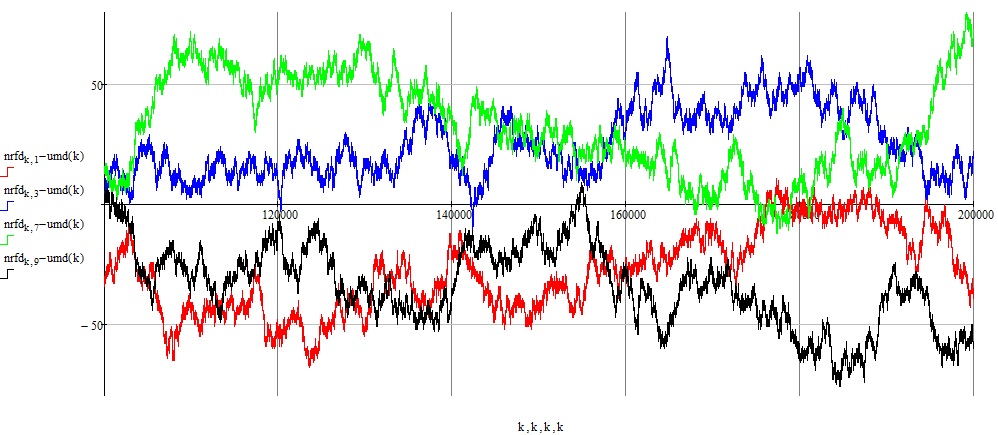}\\
  \caption{Graphic of terminal digits for primes}\label{nrfd-umd}
\end{figure}

\section{Numbers of Gap between Primes}

\begin{defn}\label{DefinitionGap}
  The dif{}ference between two successive primes is called \emph{gap}. The \emph{gap} is denoted by $g$, but to specify which specif{}ic \emph{gap}, it is available the notation $g_n=p_{n+1}-p_n$, where $p_n,p_{n+1}\in\NP{2}$.
\end{defn}

\begin{obs}
  There are authors that considers the \emph{gap} is given by the formula $g_n=p_{n+1}-p_n-1$, where $p_n,p_{n+1}\in\NP{2}$.
\end{obs}

\begin{table}[h]
  \centering
  \begin{tabular}{|r|r|r|r|r|r|r|r|}
  \hline
  g  & $3\cdot10^7$ & $2\cdot10^7$ & $10^7$ & $10^6$ & $10^5$ & $10^4$ & $10^3$ \\ \hline
  2  & 152891 & 107407 & 58980 &  8169 & 1224 & 205 & 35 \\
  4  & 152576 & 107081 & 58621 &  8143 & 1215 & 202 & 40 \\
  6  & \textbf{263423} & \textbf{183911} & \textbf{99987} & \textbf{13549} & \textbf{1940} & \textbf{299} & \textbf{44} \\
  8  & 113368 &  78792 & 42352 &  5569 &  773 & 101 & 15 \\
  10 & 145525 & 101198 & 54431 &  7079 &  916 & 119 & 16 \\
  12 & \textbf{178927} & \textbf{123410} & \textbf{65513} &  \textbf{8005} &  \textbf{965} & \textbf{105} & \textbf{8} \\
  14 &  96571 &  66762 & 35394 &  4233 &  484 &  54 & 7 \\
  16 &  70263 &  47951 & 25099 &  2881 &  339 &  33 & 0 \\
  18 & \textbf{124171} &  \textbf{84782} & \textbf{43851} &  \textbf{4909} &  \textbf{514} &  \textbf{40} & \textbf{1} \\
  20 &  63966 &  43284 & 22084 &  2402 &  238 &  15 & 1 \\
  \hline
\end{tabular}
  \caption{Number of \emph{gaps} of length 2, 4, \ldots, 20}\label{NumberGaps}
\end{table}

In Table \ref{NumberGaps}, it was displayed the number of \emph{gaps} of length 2, 4, 6, 8, 10, 12 for primes lower than $3\cdot10^7$, $2\cdot10^7$, $10^7$, $10^6$, $10^5$, $10^4$ and $10^3$.

Let us notice the the \emph{gaps} of length 6 are more frequent than the \emph{gaps} of length 2 or 4. Generally, the \emph{gaps} of multiple of 6 length are more frequent than the \emph{gaps} of comparable length.

Given this observation, we can enunciate the following conjecture: the \emph{gap of length 6 is the most common gap}.

\section{Polynomials Generating Prime Numbers}

These algebraic polynomials have the property that for $n=0,1,\ldots,m-1$ value of the polynomial, eventually in module, are $m$ primes.

\begin{enumerate}
  \item Polynomial $P(n)=n^3+n^2+17$ generates 11 primes: 17, 19, 29, 53, 97, 167, 269, 409, 593, 827, 1117,
  \cite[A050266 ]{SloaneOEIS}.
  \item Polynomial $P(n)=2n^2+11$ generates 11 primes: 11, 13, 19, 29, 43, 61, 83, 109, 139, 173, 211,
  \cite[A050265 ]{SloaneOEIS}.
  \item Honaker\index{Honaker Jr. G. L.} polynomial, $P(n)=4n^2+4n+59$, generates 14 primes: 59, 67, 83, 107, 139, 179, 227, 283, 347, 419, 499, 587, 683, 787, \cite[A048988]{SloaneOEIS}.
  \item Legendre\index{Legendre A.-M.} polynomial, $P(n)=n^2+n+17$, generates 16 primes: 17, 19, 23, 29, 37, 47, 59, 73, 89, 107, 127, 149, 173, 199, 227, 257~, \citep{Wells1986}, \cite[A007635]{SloaneOEIS}.
  \item \cite{Bruno2009}\index{Bruno A.} polynomial, $P(n)=3n^2+39n+37$, generates 18 primes: 37, 79, 127, 181, 241, 307, 379, 457, 541, 631, 727, 829, 937, 1051, 1171, 1297, 1429, 1567~.
  \item \cite{Pegg2005}\index{Pegg Jr. E.} polynomial, $P(n)=n^4+29n^2+101$, generates 20 primes: 101, 131, 233, 443, 821, 1451, 2441, 3923, 6053, 9011, 13001, 18251, 25013, 33563, 44201, 57251, 73061, 92003, 114473, 140891~.
  \item \cite{Gobbo2005}\index{Gobbo F.} polynomial, $P(n)=\abs{7n^2-371n+4871}$, generates 24 primes: 4871, 4507, 4157, 3821, 3499, 3191, 2897, 2617, 2351, 2099, 1861, 1637, 1427, 1231, 1049, 881, 727, 587, 461, 349, 251, 167, 97, 41~.
  \item Legendre\index{Legendre A.-M.} polynomial (1798), $P(n)=2n^2+29$ generates 29 primes: 29, 31, 37, 47, 61, 79, 101, 127, 157, 191, 229, 271, 317, 367, 421, 479, 541, 607, 677, 751, 829, 911, 997, 1087, 1181, 1279, 1381, 1487, 1597~, \cite[A007641]{SloaneOEIS}.
  \item \cite{Brox2006}\index{Brox J.}  polynomial, $P(n)=6n^2-342n+4903$, generates 29 primes: 4903, 4567, 4243, 3931, 3631, 3343, 3067, 2803, 2551, 2311, 2083, 1867, 1663, 1471, 1291, 1123, 967, 823, 691, 571, 463, 367, 283, 211, 151, 103, 67, 43, 31~.
  \item \cite{Gobbo2005a}\index{Gobbo F.}  polynomial, $P(n)=\abs{8n^2-488n+7243}$, generates 31 primes: 7243, 6763, 6299, 5851, 5419, 5003, 4603, 4219, 3851, 3499, 3163, 2843, 2539, 2251, 1979, 1723 1483, 1259, 1051, 859, 683, 523, 379, 251, 139, 43, 37, 101, 149, 181, 197~.
  \item \cite{Brox2006}\index{Brox J.}  polynomial, $P(n)=43n^2-537n+2971$, generates 35 primes: 2971, 2477, 2069, 1747, 1511, 1361, 1297, 1319, 1427, 1621, 1901, 2267, 2719, 3257, 3881, 4591, 5387, 6269 7237, 8291, 9431, 10657, 11969, 13367, 14851, 16421, 18077, 19819, 21647, 23561, 25561, 27647, 29819, 32077, 34421~.
  \item Wroblewski and Meyrignac polynomial, \citep{Wroblewski+Meyrignac2006}\index{Wroblewski J.}\index{Meyrignac J.-C.},
     \[
      P(n)=\abs{42n^3+270n^2-26436n+250703}~,
     \]
     generates 40 primes: 250703, 224579, 199247, 174959, 151967, 130523, 110879, 93287, 77999, 65267, 55343, 48479, 44927, 44939 48767, 56663, 68879, 85667, 107279, 133967, 165983, 203579, 247007, 296519, 352367, 414803, 484079, 560447  644159, 735467, 834623, 941879, 1057487, 1181699, 1314767, 1456943, 1608479, 1769627, 1940639, 2121767~.
  \item Euler\rq{s} polynomial\index{Euler L.} (1772), $P(n)=n^2+n+41$, generates 40 primes: 41, 43, 47, 53, 61, 71, 83, 97, 113, 131, 151, 173, 197, 223, 251, 281, 313, 347, 383, 421, 461 503, 547, 593, 641, 691, 743, 797, 853, 911, 971, 1033, 1097, 1163, 1231, 1301, 1373, 1447, 1523, 1601, \cite[A005846]{SloaneOEIS}.
  \item Legendre\index{Legendre A.-M.} (1798)  polynomial, $P(n)=n^2-n+41$, generates 40 primes: 41, 41, 43, 47, 53, 61, 71, 83, 97, 113, 131, 151, 173, 197, 223, 251, 281, 313, 347, 383, 421 461, 503, 547, 593, 641, 691, 743, 797, 853, 911, 971, 1033, 1097, 1163, 1231, 1301, 1373, 1447, 1523, 1601~. In the list there are 41 primes, but the number 41 repeats itself.
  \item \cite{Speiser2005}\index{Speiser R.} polynomial, $P(n)=\abs{103n^2-4707+50383}$, generates 43 primes: 50383, 45779, 41381, 37189, 33203, 29423, 25849, 22481, 19319, 16363, 13613, 11069, 8731, 6599 4673, 2953, 1439, 131, 971, 1867, 2557, 3041, 3319, 3391, 3257, 2917, 2371, 1619 661, 503, 1873, 3449, 5231, 7219, 9413, 11813, 14419, 17231, 20249, 23473, 26903, 30539, 34381~.
  \item Fung and Ruby polynomial, \citep{Fung+Williams1990}\index{Fung G. W.}\index{Ruby R.}, \citep{Guy2004},
      \[
       P(n)=\abs{47n^2-1701n+10181}~,
      \]
      generates 43 primes: 10181, 8527, 6967, 5501, 4129, 2851, 1667, 577, 419, 1321, 2129, 2843, 3463, 3989  4421, 4759, 5003, 5153, 5209, 5171, 5039, 4813, 4493, 4079, 3571, 2969, 2273, 1483, 599, 379, 1451, 2617, 3877, 5231, 6679, 8221, 9857, 11587, 13411, 15329, 17341, 19447, 21647, \cite[A050268]{SloaneOEIS}.
  \item \cite{Ruiz2005}\index{Riuz S. M.} polynomial, $P(n)=\abs{3n^3-183n^2+3318n-18757}$, generates 43 primes: 18757, 15619, 12829, 10369, 8221, 6367, 4789, 3469, 2389, 1531, 877, 409, 109, 41, 59, 37, 229, 499, 829, 1201, 1597, 1999, 2389, 2749, 3061, 3307, 3469, 3529, 3469, 3271, 2917, 2389 1669, 739, 419, 1823, 3491, 5441, 7691, 10259, 13163, 16421, 20051, 24071, 28499, 33353, 38651~. There are 47 primes, but 2389 and 3469 are tripled, therefore it rests only 43 of distinct primes.
  \item Fung\index{Fung G. W.} and Ruby\index{Ruby R.} polynomial, \citep{Fung+Williams1990}, $P(n)=\abs{36n^2-810n+2753}$, generates 45 primes: 2753, 1979, 1277, 647, 89, 397, 811, 1153, 1423, 1621, 1747, 1801, 1783, 1693, 1531, 1297 991, 613, 163, 359, 953, 1619, 2357, 3167, 4049, 5003, 6029, 7127, 8297, 9539, 10853, 12239, 13697, 15227, 16829, 18503, 20249, 22067, 23957, 25919, 27953, 30059, 32237, 34487, 36809~.
  \item Kazmenko and Trofimov polynomial, \citep{Kazmenko+Trofimov2006}\index{Kazmenko I.}\index{Trofimov V.}
     \[
      P(n)=\abs{-66n^3+3845n^2-60897n+251831}~,
     \]
     generates 46 primes: 251831, 194713, 144889, 101963, 65539, 35221, 10613, 8681, 23057, 32911, 38639, 40637, 39301, 35027, 28211, 19249 8537, 3529, 16553, 30139, 43891, 57413, 70309, 82183, 92639, 101281, 107713, 111539, 112363, 109789, 103421, 92863, 77719, 57593, 32089, 811, 36637, 80651, 131627, 189961, 256049, 330287, 413071, 504797, 605861, 716659~.
  \item Wroblewski and Meyrignac polynomial, \citep{Wroblewski+Meyrignac2006}\index{Wroblewski J.}\index{Meyrignac J.-C.}
     \[
      P(n)=\abs{n^5-99n^4+3588n^3-56822n^2+348272n-286397}~,
     \]
     generates 47 primes: 286397, 8543, 210011, 336121, 402851, 424163, 412123, 377021, 327491, 270631, 212123, 156353, 106531, 64811, 32411, 9733, 3517, 8209, 5669, 2441, 14243, 27763, 41051, 52301, 59971, 62903, 60443, 52561, 39971, 24251, 7963, 5227, 10429, 1409, 29531, 91673, 196003, 355331, 584411, 900061, 1321283, 1869383, 2568091, 3443681, 4525091, 5844043, 7435163~.
  \item \cite{Beyleveld2006}\index{Beyleveld M.} polynomial,
     \[
      P(n)=\abs{n^4-97n^3+329n^2-45458n+213589}~,
     \]
     generates 49 primes: 213589, 171329, 135089, 104323, 78509, 57149, 39769, 25919, 15173, 7129, 1409, 2341, 4451, 5227, 4951, 3881, 2251, 271, 1873, 4019, 6029, 7789, 9209, 10223, 10789, 10889, 10529, 9739, 8573, 7109, 5449, 3719, 2069, 673, 271, 541, 109, 1949, 5273, 10399, 17669, 27449, 40129, 56123, 75869, 99829, 128489, 162359, 201973, 247889~. In the list, we have 50 primes, but the number 271 repeats once.
  \item Wroblewski and Meyrignac polynomial, \citep{Wroblewski+Meyrignac2006}\index{Wroblewski J.}\index{Meyrignac J.-C.}
     \begin{multline*}
       P(n)=\left|\frac{n^6-126n^5+6217n^4-153066n^3}{36}\right. \\
       +\left.\frac{1987786n^2-1305531n+34747236}{36}\right|~,
     \end{multline*}
     generates 55 primes: 965201, 653687, 429409, 272563, 166693, 98321, 56597, 32969, 20873, 15443, 13241, 12007, 10429, 7933, 4493, 461, 3583, 6961, 9007, 9157, 7019, 2423, 4549, 13553, 23993, 35051, 45737, 54959, 61613, 64693, 63421, 57397, 46769, 32423, 16193, 1091, 8443, 6271, 15733, 67993, 163561, 318467, 552089, 887543, 1352093, 1977581, 2800877, 3864349, 5216353, 6911743, 9012401, 11587787, 14715509, 18481913, 22982693~.
  \item Dress\index{Dress F.}, Laudreau\index{Laudreau B.} and Gupta\index{Gupta H.} polynomial, \citep{Dress+Landreau2002,Gupta2006},
     \[
       P(n)=\abs{\frac{n^5-133n^4+6729n^3-158379n^2+1720294n-6823316}{4}}~,
     \]
     generates 57 primes: 1705829, 1313701, 991127, 729173, 519643, 355049, 228581, 134077, 65993, 19373, 10181, 26539, 33073, 32687, 27847, 20611, 12659, 5323, 383, 3733, 4259, 1721, 3923, 12547, 23887, 37571, 53149, 70123, 87977, 106207, 124351, 142019, 158923, 174907, 189977, 204331, 218389, 232823, 248587, 266947, 289511, 318259, 355573, 404267, 467617, 549391, 653879, 785923, 950947, 1154987, 1404721, 1707499, 2071373, 2505127, 3018307, 3621251, 4325119~.
\end{enumerate}

\section{Primorial}

Let $p_n$ be the $n$th prime, then the \emph{primorial} is def{}ined by
\begin{equation}\label{Primorial}
  p_n\#=\prod_{k=1}^np_k~.
\end{equation}
By def{}inition we have that $1\#=1$.

The values of $1\#$, $2\#$, $3\#$, \ldots, $43\#$ are: 1, 2, 6, 30, 210, 2310, 30030, 510510, 9699690, 223092870, 6469693230, 200560490130, 7420738134810, 304250263527210, 13082761331670030, \cite[A002110]{SloaneOEIS}~.

This list can also be obtained with the program \ref{Program kP}, which generates the multiprimorial.

\begin{prog}\label{Program kP} for generating the multiprimorial, $p\#=kP(p,1)$ or $p\#\#=kP(p,2)$ or $p\#\#\#=kP(p,3)$.
  \begin{tabbing}
    $\emph{kP}(p,k):=$\=\vline\ $\emph{return}\ \ 1\ \ \emph{if}\ \ p\textbf{=}1$\\
    \>\vline\ $\emph{return}\ \ "\emph{Error.}\ p\ not\ \emph{prime}\ \ \emph{if}\ \ \emph{TS}(p)\textbf{=}0$\\
    \>\vline\ $q\leftarrow \mod(p,k+1)$\\
    \>\vline\ $\emph{return}\ \ p\ \ \emph{if}\ \ q\textbf{=}0$\\
    \>\vline\ $pk\leftarrow1$\\
    \>\vline\ $pk\leftarrow2\ \ \emph{if}\ \ k\textbf{=}1$\\
    \>\vline\ $j\leftarrow1$\\
    \>\vline\ $w$\=$hile\ \ \emph{prime}_j\le p$\\
    \>\vline\ \>\vline\ $pk\leftarrow pk\cdot \emph{prime}_j\ \ \emph{if}\ \ \mod(\emph{prime}_j,k+1)\textbf{=}q$\\
    \>\vline\ \>\vline\ $j\leftarrow j+1$\\
    \>\vline\ $\emph{return}\ \ pk$\\
  \end{tabbing}
  The program calls the primality test $\emph{TS}$, \ref{Program TS}.
\end{prog}

It is sometimes convenient to def{}ine the primorial $n\#$ for values other than just the primes, in which case it is taken to be given by the product of all primes less than or equal to $n$, i.e.
\begin{equation}\label{Primorial N}
  n\#=\prod_{k=1}^{\pi(n)}p_k~,
\end{equation}
where $\pi$ is the prime counting function.

For $1$, \ldots, $30$ the f{}irst few values of $n\#$ are: 1, 2, 6, 6, 30, 30, 210, 210, 210, 210, 2310, 2310, 30030, 30030, 30030, 30030, 510510, 510510, 9699690, 9699690, 9699690, 9699690, 223092870, 223092870, 223092870, 223092870, 223092870, 223092870, 6469693230, 6469693230, \cite[A034386]{SloaneOEIS}.

The decomposition in prime factors of numbers $\emph{prime}_n\#-1$ and $\emph{prime}_n\#+1$ for $n=1, 2, \ldots, 15$ are in Tables 1.3 and respectively 1.4
\begin{center}
 \begin{longtable}{|c|c|}
   \caption{The decomposition in factors of $\emph{prime}_n\#-1$}\label{Primorial-1}\\
   \hline
   $\emph{prime}_n\#-1$  & $\emph{factorization}$ \\
   \hline
  \endfirsthead
   \hline
   $\emph{prime}_n\#-1$  & $\emph{factorization}$ \\
   \hline
  \endhead
   \hline \multicolumn{2}{r}{\textit{Continued on next page}} \\
  \endfoot
   \hline
  \endlastfoot
   $0$ & $0$ \\
   $1$ & $1$ \\
   $5$ & $5$ \\
   $29$ & $29$ \\
   $209$ & $11\cdot19$ \\
   $2309$ & $2309$ \\
   $30029$ & $30029$ \\
   $510509$ & $61\cdot8369$ \\
   $9699689$ & $53\cdot197\cdot929$ \\
   $223092869$ & $37\cdot131\cdot46027$ \\
   $6469693229$ & $79\cdot81894851$ \\
   $200560490129$ & $228737\cdot876817$ \\
   $7420738134809$ & $229\cdot541\cdot1549\cdot38669$ \\
   $304250263527209$ & 304250263527209 \\
   $13082761331670029$ & $141269\cdot92608862041$ \\
   \hline
 \end{longtable}
\end{center}

\begin{center}
 \begin{longtable}{|c|c|}
   \caption{The decomposition in factors of $\emph{prime}_n\#+1$}\label{Primorial+1}\\
   \hline
   $\emph{prime}_n\#+1$  & $\emph{factorization}$ \\
   \hline
  \endfirsthead
   \hline
   $\emph{prime}_n\#+1$  & $\emph{factorization}$ \\
   \hline
  \endhead
   \hline \multicolumn{2}{r}{\textit{Continued on next page}} \\
  \endfoot
   \hline
  \endlastfoot
   $2$ & $2$ \\
   $3$ & $3$ \\
   $7$ & $7$ \\
   $31$ & $31$ \\
   $211$ & $211$ \\
   $2311$ & $2311$ \\
   $30031$ & $59\cdot509$ \\
   $510511$ & $19\cdot97\cdot277$ \\
   $9699691$ & $347\cdot27953$ \\
   $223092871$ & $317\cdot703763$ \\
   $6469693231$ & $331\cdot571\cdot34231$ \\
   $200560490131$ & $200560490131$ \\
   $7420738134811$ & $181\cdot60611\cdot676421$ \\
   $304250263527211$ & $61\cdot450451\cdot11072701$ \\
   $13082761331670031$ & $167\cdot78339888213593$ \\
   \hline
 \end{longtable}
\end{center}

\subsection{Double Primorial}

Let $p\in\NP{2}$, then the \emph{double primorial} is def{}ined by
\begin{equation}\label{Double Primorial}
  p\#\#=p_1\cdot p_2\cdots p_m \ \ \textnormal{with}\ \mod(p_j,3)=\mod(p,3)~,
\end{equation}
where $p_j\in\NP{2}$ and $p_j<p$, for any $j=1,2,\ldots,m-1$ and $p_m=p$. By def{}inition we have that $1\#\#=1$.

Examples: $2\#\#=2$, $3\#\#=3$, $5\#\#=2\cdot5=10$ because $\mod(5,3)=\mod(2,3)=2$, $13\#\#=7\cdot13=91$ because $\mod(13,3)=\mod(7,3)=1$, $17\#\#=2\cdot5\cdot11\cdot17=1870$ because $\mod(17,3)=\mod(11,3)=\mod(5,3)=\mod(2,3)=2$, etc.

The list of values $1\#\#$, $2\#\#$, \ldots, $67\#\#$, obtained with the program $kP$, \ref{Program kP}, is: 1, 2, 3, 10, 7, 110, 91, 1870, 1729, 43010, 1247290, 53599, 1983163, 51138890, 85276009, 2403527830, 127386974990, 7515831524410, 5201836549, 348523048783~.

The decomposition in prime factors of the numbers $\emph{prime}_n\#\#-1$ and $\emph{prime}_n\#\#+1$ for $n=1, 2, \ldots, 19$ are in Tables \ref{DoublePrimorial-1} and respectively \ref{DoublePrimorial+1}:

\begin{center}
 \begin{longtable}{|c|c|}
   \caption{The decomposition in factors of $\emph{prime}_n\#\#-1$}\label{DoublePrimorial-1}\\
   \hline
   $\emph{prime}_n\#\#-1$  & $\emph{factorization}$ \\
   \hline
  \endfirsthead
   \hline
   $\emph{prime}_n\#\#-1$  & $\emph{factorization}$ \\
   \hline
  \endhead
   \hline \multicolumn{2}{r}{\textit{Continued on next page}} \\
  \endfoot
   \hline
  \endlastfoot
   $1$ & $1$ \\
   $2$ & $2$ \\
   $9$ & $3^2$ \\
   $6$ & $2\cdot3$ \\
   $109$ & 109 \\
   $90$ & $2\cdot3^2\cdot5$ \\
   $1869$ & $3\cdot7\cdot89$ \\
   $1728$ & $2^6\cdot3^3$ \\
   $43009$ & $41\cdot1049$ \\
   $1247289$ & $3\cdot379\cdot1097$ \\
   $53598$ & $2\cdot3\cdot8933$ \\
   $1983162$ & $2\cdot3\cdot103\cdot3209$ \\
   $51138889$ & $67\cdot763267$ \\
   $85276008$ & $2^3\cdot3^2\cdot29\cdot40841$ \\
   $2403527829$ & $3\cdot23537\cdot34039$ \\
   $127386974989$ & $19\cdot59\cdot113636909$ \\
   $7515831524409$ & $3^2\cdot13\cdot101\cdot19211\cdot33107$ \\
   $5201836548$ & $2^2\cdot3\cdot9277\cdot46727$ \\
   $348523048782$ & $2\cdot3^3\cdot23\cdot280614371$ \\
   \hline
 \end{longtable}
\end{center}

\begin{center}
 \begin{longtable}{|c|c|}
   \caption{The decomposition in factors of $\emph{prime}_n\#\#+1$}\label{DoublePrimorial+1}\\
   \hline
   $\emph{prime}_n\#\#+1$ & $\emph{factorization}$ \\
   \hline
  \endfirsthead
   \hline
   $\emph{prime}_n\#\#+1$  & $\emph{factorization}$ \\
   \hline
  \endhead
   \hline \multicolumn{2}{r}{\textit{Continued on next page}} \\
  \endfoot
   \hline
  \endlastfoot
   $3$ & $3$ \\
   $4$ & $2^2$ \\
   $11$ & $11$ \\
   $8$ & $2^3$ \\
   $111$ & $3\cdot37$ \\
   $92$ & $2^2\cdot23$ \\
   $1871$ & $1871$ \\
   $1730$ & $2\cdot5\cdot173$ \\
   $43011$ & $3^6\cdot59$ \\
   $1247291$ & $1247291$ \\
   $53600$ & $2^5\cdot5^2\cdot67$ \\
   $1983164$ & $2^2\cdot495791$ \\
   $51138891$ & $3^3\cdot1894033$ \\
   $85276010$ & $2\cdot5\cdot8527601$ \\
   $2403527831$ & $12889\cdot186479$ \\
   $127386974991$ & $3\cdot349\cdot121668553$ \\
   $7515831524411$ & $7\cdot1367\cdot785435419$ \\
   $5201836550$ & $2\cdot5^2\cdot104036731$ \\
   $348523048784$ & $2^4\cdot113\cdot192767173$ \\
   \hline
 \end{longtable}
\end{center}

\subsection{Triple Primorial}

Let $p\in\NP{2}$, then the \emph{triple primorial} is def{}ined by
\begin{equation}\label{Triplu Primorial}
  p\#\#\#=p_1\cdot p_2\cdots p_m\ \ \textnormal{with}\ \ \mod(p_j,4)=\mod(p,4)
\end{equation}
where $p_j\in\NP{2}$, $p_j<p$, for any $j=1,2,\ldots,m-1$ and $p_m=p$. By def{}inition we have that $1\#\#\#=1$.

The list of values $1\#\#\#$, $2\#\#\#$, \ldots, $73\#\#\#$, obtained with the program $kP$, \ref{Program kP}, is: 1, 2, 3, 5, 21, 231, 65, 1105, 4389, 100947, 32045, 3129357, 1185665, 48612265, 134562351, 6324430497, 2576450045, 373141399323, 157163452745, 25000473754641, 1775033636579511, 11472932050385~.

The decomposition in prime factors of $\emph{prime}_n\#\#\#-1$ and $\emph{prime}_n\#\#\#+1$ for $n=1, 2, \ldots, 19$ are in Tables \ref{TriplePrimorial-1} and respectively \ref{TriplePrimorial+1}:

\begin{center}
 \begin{longtable}{|c|c|}
   \caption{The decomposition in factors of $prime_n\#\#\#-1$}\label{TriplePrimorial-1}\\
   \hline
   $\emph{prime}_n\#\#\#-1$  &  $\emph{factorization}$ \\
   \hline
  \endfirsthead
   \hline
   $\emph{prime}_n\#\#\#-1$  &  $\emph{factorization}$ \\
   \hline
  \endhead
   \hline \multicolumn{2}{r}{\textit{Continued on next page}} \\
  \endfoot
   \hline
  \endlastfoot
  $1$ & $1$ \\
  $2$ & $2$ \\
  $4$ & $2^2$ \\
  $20$ & $2^2\cdot5$ \\
  $230$ & $2\cdot5\cdot23$ \\
  $64$ & $2^6$ \\
  $1104$ & $2^4\cdot3\cdot23$ \\
  $4388$ & $2^2\cdot1097$ \\
  $100946$ & $2\cdot17\cdot2969$ \\
  $32044$ & $2^2\cdot8011$ \\
  $3129356$ & $2^2\cdot782339$ \\
  $1185664$ & $2^7\cdot59\cdot157$ \\
  $48612264$ & $2^3\cdot3\cdot2025511$ \\
  $134562350$ & $2\cdot5^2\cdot13\cdot241\cdot859$ \\
  $6324430496$ & $2^5\cdot197638453$ \\
  $2576450044$ & $2^2\cdot7\cdot263\cdot349871$ \\
  $373141399322$ & $2\cdot186570699661$ \\
  $157163452744$ & $2^3\cdot193\cdot421\cdot241781$ \\
  $25000473754640$ & $2^4\cdot5\cdot312505921933$ \\
  \hline
 \end{longtable}
\end{center}

\begin{center}
 \begin{longtable}{|c|c|}
   \caption{The decomposition in factors of $prime_n\#\#\#+1$}\label{TriplePrimorial+1}\\
   \hline
   $\emph{prime}_n\#\#\#+1$  & $\emph{factorization}$ \\
   \hline
  \endfirsthead
   \hline
   $\emph{prime}_n\#\#\#+1$  & $\emph{factorization}$ \\
   \hline
  \endhead
   \hline \multicolumn{2}{r}{\textit{Continued on next page}} \\
  \endfoot
   \hline
  \endlastfoot
  $2$ & $2$ \\
  $3$ & $3$ \\
  $4$ & $2^2$ \\
  $6$ & $2\cdot3$ \\
  $22$ & $2\cdot11$ \\
  $232$ & $2^3\cdot29$ \\
  $66$ & $2\cdot3\cdot11$ \\
  $1106$ & $2\cdot7\cdot79$ \\
  $4390$ & $2\cdot5\cdot439$ \\
  $100948$ & $2^2\cdot25237$ \\
  $32046$ & $2\cdot3\cdot7^2\cdot109$ \\
  $3129358$ & $2\cdot1564679$ \\
  $1185666$ & $2\cdot3\cdot73\cdot2707$ \\
  $48612266$ & $2\cdot131\cdot185543$ \\
  $134562352$ & $2^4\cdot1123\cdot7489$ \\
  $6324430498$ & $2\cdot37\cdot6269\cdot13633$ \\
  $2576450046$ & $2\cdot3\cdot19\cdot61\cdot163\cdot2273$ \\
  $373141399324$ & $2^2\cdot433\cdot215439607$ \\
  $157163452746$ & $2\cdot3\cdot5647\cdot4638553$ \\
  $25000473754642$ & $2\cdot109\cdot6199\cdot18499931$ \\
  \hline
 \end{longtable}
\end{center}

\chapter{Arithmetical Functions}

\section{Function of Counting the Digits}
Mathcad user functions required in the following.
\begin{func}\label{FunctionNrd}
  Function of counting the digits of number $n_{(10)}$ in base $b$:
  \begin{tabbing}
    $\emph{nrd}(n,b):=$\ \=\ \vline\ $\emph{return}\ \ 1\ \ \emph{if}\ \ n\textbf{=}0$\\
    \>\ \vline\ $\emph{return}\ \ 1+\emph{floor}\big(\log(n,b)\big)\ \ \emph{otherwise}$
  \end{tabbing}
\end{func}

\section{Digits the Number in Base $b$}
\begin{prog}\label{ProgramDn}
  Program providing the digits in base $b$ of the number $n_{(10)}$:
  \begin{tabbing}
    $\emph{dn}(n,b):=$\ \=\ \vline\ $f$\=$or\ k\in \emph{ORIGIN}..\emph{nrd}(n,b)-1$\\
    \>\ \vline\ \> \vline\ $t\leftarrow\emph{Trunc}(n,b)$\\
    \>\ \vline\ \> \vline\ $cb_k\leftarrow n-t$\\
    \>\ \vline\ \> \vline\ $n\leftarrow\dfrac{t}{b}$\\
    \>\ \vline\ $\emph{return}\ \ reverse(cb)$
  \end{tabbing}
\end{prog}

\section{Prime Counting Function}

By means of Smarandache\rq{s} function, we obtain a formula for counting the prime numbers less or equal to $n$, \citep{Seagull1995}.
\begin{prog}\label{Functia Pi} of counting the primes to $n$.
  \begin{tabbing}
    $\pi(n):=$\=\vline\ $\emph{return}\ \ 0\ \ \emph{if}\ \ n=1$\\
    \>\vline\ $\emph{return}\ \ 1\ \ \emph{if}\ \ n=2$\\
    \>\vline\ $\emph{return}\ \ 2\ \ \emph{if}\ \ n=3$\\
    \>\vline\ $\emph{return}\ \ -1+\displaystyle\sum_{k=2}^n \emph{floor}\left(\dfrac{S_k}{k}\right)$
  \end{tabbing}
\end{prog}

\section{Digital Sum}

\begin{func}\label{FunctionDks}
  Function of summing the digits in base $b$ of power $k$ of the number $n$ written in base $10$.
  \begin{equation}\label{Functia dks}
    \emph{dks}(n,b,k):=\sum \overrightarrow{dn(n,b)^k}~.
  \end{equation}
\end{func}

\begin{figure}
  \centering
  \includegraphics[scale=0.8]{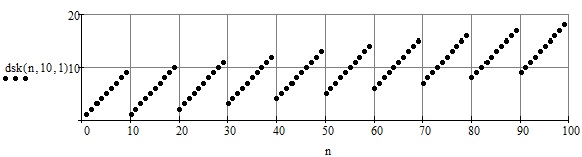}\\
  \caption{The digital sum function}\label{FigDigitalSum}
\end{figure}

Examples:
\begin{enumerate}
  \item Example $\emph{dks}(76,8,1)=6$, verif{}ied by the identity $76_{(10)}=114_{(8)}$ and by the fact than $1^1+1^1+4^1=6$;
  \item Example $\emph{dks}(1234,16,1)=19$, verif{}ied by the identity $1234_{(10)}=4d2_{(16)}$ and by the fact than $4^1+d^1+2^1=4+13+2=19$;
  \item Example $\emph{dks}(15,2,1)=4$, verif{}ied by the identity $15_{(10)}=1111_{(2)}$ and by the fact than $1^1+1^1+1^1+1^1=4$.
  \item Example $\emph{dks}(76,8,2)=18$, verif{}ied by the identity $76_{(10)}=114_{(8)}$ and by the fact than $1^2+1^2+4^2=18$;
  \item Example $\emph{dks}(1234,16,2)=189$, verif{}ied by the identity $1234_{(10)}=4d2_{(16)}$ and by the fact than $4^2+d^2+2^2=4^2+13^2+2^2=189$;
  \item Example $\emph{dks}(15,2)=4$, verif{}ied by the identity $15_{(10)}=1111_{(2)}$ and by the fact than $1^2+1^2+1^2+1^2=4$.
\end{enumerate}

\begin{figure}[h]
  \centering
  \includegraphics[scale=0.8]{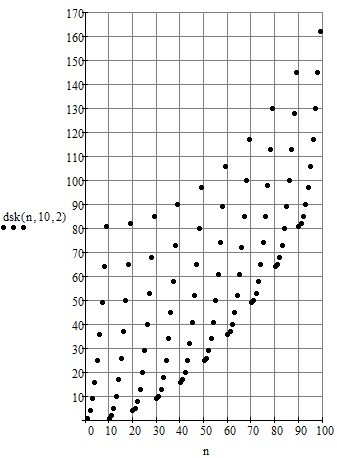}\\
  \caption{Digital sum function of power $2$ of the number $n_{(10)}$}\label{FigDigitalMSum}
\end{figure}

\subsection{Narcissistic Numbers}

We can apply the function $\emph{dks}$, given by (\ref{Functia dks}), for determining Narcissistic numbers (Armstrong numbers, or Plus Perfect numbers), \citep{MathWorldNarcissisticNumber}, \cite[ A005188, A003321, A010344, A010346, A010348, A010350, A010353, A010354, A014576, A023052, A032799, A046074, A101337 and A11490]{SloaneOEIS}, \cite[p. 105]{Hardy1993}, \cite[pp. 163--173]{Madachy1979}, \cite[p. 35]{Roberts1992}, \cite[pp. 169--170]{Pickover1995}, \cite[pp. 204--205]{Pickover2001}.

\begin{defn}\label{DefnNarcissisticNumber}
  A number having $m$ digits $d_k$ in base $b$ ($b\in\Ns$, $b\ge2$) is Narcissistic if
  \begin{equation}\label{NarNumber}
    \overline{d_1d_2\dots d_m}_{(b)}=d_1^m+d_2^m+\ldots+d_m^m~,
  \end{equation}
  where $d_k\in\set{0,1,\ldots,b-1}$, for $k\in\Ind{m}=\set{1,2,\ldots,m}$~.
\end{defn}
Program \ref{ProgramDn} can be used for determining Narcissistic numbers in numeration base $b$, \citep{CO+CCMpuav2010}. The numbers $\set{1,2,\ldots,b-1}$ are ordinary Narcissistic numbers for any $b\ge2$.

To note that the search domain for a base $b\in\Na$, $b\ge2$ are f{}inite. For any number $n$, with $m$ digits in base $b$ for which we have met the condition $\log_b\big(m(b-1)^m\big)>\log_b(b^{m-1})$ Narcissistic number search
makes no sense. For example, the search domain for numbers in base 3 makes sense only for the numbers: 3, 4, \ldots, 2048.

Let the function $h:\Ns\times\Ns\to\Real$ given by the formula
\begin{equation}\label{FunctiahDc}
  h(b,m)=\log_b\big(m(b-1)^m\big)-\log_b(b^{m-1})~.
\end{equation}

We represent the function for $b=3,4,\ldots,16$ and $m=1,2,\ldots,120$, see Figures \ref{FigFh3-16}. Using the function $h$ we can determine the Narcissistic numbers\rq{} search domains.

Let the search domains be:
\begin{eqnarray}
  Dc_2&=&\set{2},\label{Dc2Narcisissistic} \\
  Dc_3&=&\set{3,4,\ldots,2048},\ \textnormal{where}\ 2048=8\cdot2^8,\label{Dc3Narcisissistic}
\end{eqnarray}
\begin{equation}\label{DcGenNarcisissistic}
  Dc_b=\set{b,b+1,\ldots,10^7}\ \ \textnormal{for}\ \ b=\set{4,5,\ldots,16}~.
\end{equation}
\begin{figure}[h]
  \centering
  \includegraphics[scale=0.7]{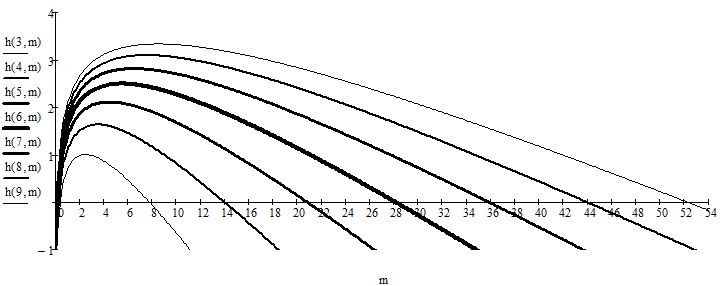}\\
  \includegraphics[scale=0.7]{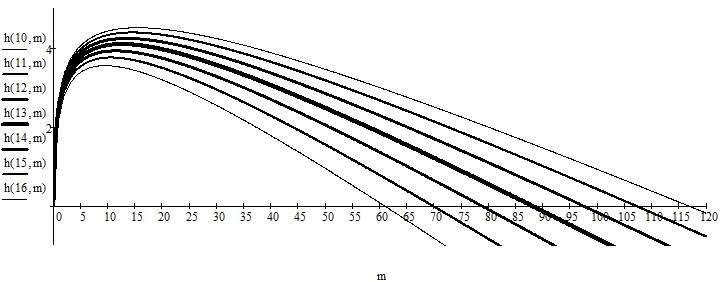}\\
  \caption{Function $h$ for $b=3,4,\ldots,16$ and $m=1,2,\ldots,120$}\label{FigFh3-16}
\end{figure}

\begin{center}
  \begin{longtable}{|rrl|}
    \caption{Narcissistic numbers with $b=3$ of (\ref{Dc3Narcisissistic})}\\ \hline
    \endfirsthead
    \endhead
    \multicolumn{3}{r}{\textit{Continued on next page}} \\
    \endfoot
    \endlastfoot
      $5=$&$12_{(3)}=$&$1^2+2^2$\\
      $8=$&$22_{(3)}=$&$2^2+2^2$\\ \hline
      $17=$&$122_{(3)}=$&$1^3+2^3+2^3$\\ \hline
    \end{longtable}
    \end{center}

    \begin{center}
    \begin{longtable}{|rrl|}
    \caption{Narcissistic numbers with $b=4$ of (\ref{DcGenNarcisissistic})}\\ \hline
    \endfirsthead
    \endhead
    \multicolumn{3}{r}{\textit{Continued on next page}} \\
    \endfoot
    \endlastfoot
      $28=$&$130_{(4)}=$&$1^3+3^3+0^3$\\
      $29=$&$131_{(4)}=$&$1^3+3^3+1^3$\\
      $35=$&$203_{(4)}=$&$2^3+0^3+3^3$\\
      $43=$&$223_{(4)}=$&$2^3+2^3+3^3$\\
      $55=$&$313_{(4)}=$&$3^3+1^3+3^3$\\
      $62=$&$332_{(4)}=$&$3^3+3^3+2^3$\\ \hline
      $83=$&$1103_{(4)}=$&$1^4+1^4+0^4+3^4$\\
      $243=$&$3303_{(4)}=$&$3^4+3^4+0^4+3^4$\\ \hline
  \end{longtable}
\end{center}

\begin{center}
  \begin{longtable}{|rrl|}
    \caption{Narcissistic numbers with $b=5$ of (\ref{DcGenNarcisissistic})}\\ \hline
    \endfirsthead
    \endhead
    \multicolumn{3}{r}{\textit{Continued on next page}} \\
    \endfoot
    \endlastfoot
       $13=$&$23_{(5)}=$&$2^2+3^2$\\
       $18=$&$33_{(5)}=$&$3^2+3^2$\\ \hline
       $28=$&$103_{(5)}=$&$1^3+0^3+3^3$\\
       $118=$&$433_{(5)}=$&$4^3+3^3+3^3$\\ \hline
       $289=$&$2124_{(5)}=$&$2^4+1^4+2^4+4^4$\\
       $353=$&$2403_{(5)}=$&$2^4+4^4+0^4+3^4$\\
       $419=$&$3134_{(5)}=$&$3^4+1^4+3^4+3^4$\\ \hline
       $4890=$&$124030_{(5)}=$&$1^6+2^6+4^6+0^6+3^6+0^6$\\
       $4891=$&$124031_{(5)}=$&$1^6+2^6+4^6+0^6+3^6+1^6$\\
       $9113=$&$242423_{(5)}=$&$2^6+4^6+2^6+4^6+2^6+3^6$\\ \hline
  \end{longtable}
\end{center}

\begin{center}
  \begin{longtable}{|rrl|}
    \caption{Narcissistic numbers with $b=6$ of (\ref{DcGenNarcisissistic})}\\ \hline
    \endfirsthead
    \endhead
    \multicolumn{3}{r}{\textit{Continued on next page}} \\
    \endfoot
    \endlastfoot
       $99=$&$243_{(6)}=$&$2^3+4^3+3^3$\\
       $190=$&$514_{(6)}=$&$5^3+1^3+4^3$\\ \hline
       $2292=$&$14340_{(6)}=$&$1^5+4^5+3^5+4^5+0^5$\\
       $2293=$&$14341_{(6)}=$&$1^5+4^5+3^5+4^5+1^5$\\
       $2324=$&$14432_{(6)}=$&$1^5+4^5+4^5+3^5+2^5$\\
       $3432=$&$23520_{(6)}=$&$2^5+3^5+5^5+2^5+0^5$\\
       $3433=$&$23521_{(6)}=$&$2^5+3^5+5^5+2^5+1^5$\\
       $6197=$&$44405_{(6)}=$&$4^5+4^5+4^5+0^5+5^5$\\ \hline
       $36140=$&$435152_{(6)}=$&$4^6+3^6+5^6+1^6+5^6+2^6$\\ \hline
       $269458=$&$5435254_{(6)}=$&$5^7+4^7+3^7+5^7+2^7+5^7+4^7$\\ \hline
       $391907=$&$12222215_{(6)}=$&$1^8+2^8+2^8+2^8+2^8+2^8+1^8+5^8$\\ \hline
  \end{longtable}
\end{center}

\begin{center}
  \begin{longtable}{|rrl|}
    \caption{Narcissistic numbers with $b=7$ of (\ref{DcGenNarcisissistic})}\\ \hline
    \endfirsthead
    \endhead
    \multicolumn{3}{r}{\textit{Continued on next page}} \\
    \endfoot
    \endlastfoot
       $10=$&$13_{(7)}=$&$1^2+3^2$\\
       $25=$&$34_{(7)}=$&$3^2+4^2$\\
       $32=$&$44_{(7)}=$&$4^2+4^2$\\
       $45=$&$63_{(7)}=$&$6^2+3^2$\\ \hline
       $133=$&$250_{(7)}=$&$2^3+5^3+0^3$\\
       $134=$&$251_{(7)}=$&$2^3+5^3+1^3$\\
       $152=$&$305_{(7)}=$&$3^3+0^3+5^3$\\
       $250=$&$505_{(7)}=$&$5^3+0^3+5^3$\\ \hline
       $3190=$&$12205_{(7)}=$&$1^5+2^5+2^5+0^5+5^5$\\
       $3222=$&$12252_{(7)}=$&$1^5+2^5+2^5+5^5+2^5$\\
       $3612=$&$13350_{(7)}=$&$1^5+3^5+3^5+5^5+0^5$\\
       $3613=$&$13351_{(7)}=$&$1^5+3^5+3^5+5^5+1^5$\\
       $4183=$&$15124_{(7)}=$&$1^5+5^5+1^5+2^5+4^5$\\
       $9286=$&$36034_{(7)}=$&$3^5+6^5+0^5+3^5+4^5$\\ \hline
       $35411=$&$205145_{(7)}=$&$2^6+0^6+5^6+1^6+4^6+5^6$\\ \hline
       $191334=$&$1424553_{(7)}=$&$1^7+4^7+2^7+4^7+5^7+5^7+3^7$\\
       $193393=$&$1433554_{(7)}=$&$1^7+4^7+3^7+3^7+5^7+5^7+4^7$\\
       $376889=$&$3126542_{(7)}=$&$3^7+1^7+2^7+6^7+5^7+4^7+2^7$\\
       $535069=$&$4355653_{(7)}=$&$4^7+3^7+5^7+5^7+6^7+5^7+3^7$\\
       $794376=$&$6515652_{(7)}=$&$6^7+5^7+1^7+5^7+6^7+5^7+2^7$\\ \hline
  \end{longtable}
\end{center}

\begin{center}
  \begin{longtable}{|rrl|}
   \caption{Narcissistic numbers with $b=8$ of (\ref{DcGenNarcisissistic})}\\ \hline
   \endfirsthead
   \endhead
   \multicolumn{3}{r}{\textit{Continued on next page}} \\
   \endfoot
   \endlastfoot
       $20=$&$24_{(8)}=$&$2^2+4^2$\\
       $52=$&$64_{(8)}=$&$6^2+4^2$\\ \hline
       $92=$&$134_{(8)}=$&$1^3+3^3+4^3$\\
       $133=$&$205_{(8)}=$&$2^3+0^3+5^3$\\
       $307=$&$463_{(8)}=$&$4^3+6^3+3^3$\\
       $432=$&$660_{(8)}=$&$6^3+6^3+0^3$\\
       $433=$&$661_{(8)}=$&$6^3+6^3+1^3$\\ \hline
       $16819=$&$40663_{(8)}=$&$4^5+0^5+6^5+6^5+3^5$\\
       $17864=$&$42710_{(8)}=$&$4^5+2^5+7^5+1^5+0^5$\\
       $17865=$&$42711_{(8)}=$&$4^5+2^5+7^5+1^5+1^5$\\
       $24583=$&$60007_{(8)}=$&$6^5+0^5+0^5+0^5+7^5$\\
       $25639=$&$62047_{(8)}=$&$6^5+2^5+0^5+4^5+7^5$\\ \hline
       $212419=$&$636703_{(8)}=$&$6^6+3^6+6^6+7^6+0^6+3^6$\\ \hline
       $906298=$&$3352072_{(8)}=$&$3^7+3^7+5^7+2^7+0^7+7^7+2^7$\\
       $906426=$&$3352272_{(8)}=$&$3^7+3^7+5^7+2^7+2^7+7^7+2^7$\\
       $938811=$&$3451473_{(8)}=$&$3^7+4^7+5^7+1^7+4^7+7^7+3^7$\\ \hline
  \end{longtable}
\end{center}

\begin{center}
  \begin{longtable}{|rrl|}
    \caption{Narcissistic numbers with $b=9$ of (\ref{DcGenNarcisissistic})}\\ \hline
    \endfirsthead
    \endhead
    \multicolumn{3}{r}{\textit{Continued on next page}} \\
    \endfoot
    \endlastfoot
       $41=$&$45_{(9)}=$&$4^2+5^2$\\
       $50=$&$55_{(9)}=$&$5^2+5^2$\\ \hline
       $126=$&$150_{(9)}=$&$1^3+5^3+0^3$\\
       $127=$&$151_{(9)}=$&$1^3+5^3+1^3$\\
       $468=$&$570_{(9)}=$&$5^3+7^3+0^3$\\
       $469=$&$571_{(9)}=$&$5^3+7^3+1^3$\\ \hline
       $1824=$&$2446_{(9)}=$&$2^4+4^4+4^4+6^4$\\ \hline
       $8052=$&$12036_{(9)}=$&$1^5+2^5+0^5+3^5+6^5$\\
       $8295=$&$12336_{(9)}=$&$1^5+2^5+3^5+3^5+6^5$\\
       $9857=$&$14462_{(9)}=$&$1^5+4^5+4^5+6^5+2^5$\\ \hline
  \end{longtable}
\end{center}

\begin{center}
  \begin{longtable}{|rl|}
    \caption{Narcissistic numbers with $b=10$ of (\ref{DcGenNarcisissistic})}\label{TabelNumereNarcisissticBaza10}\\  \hline
    \endfirsthead
    \endhead
    \multicolumn{2}{r}{\textit{Continued on next page}} \\
    \endfoot
    \endlastfoot
       $153=$&$1^3+5^3+3^3$\\
       $370=$&$3^3+7^3+0^3$\\
       $371=$&$3^3+7^3+1^3$\\
       $407=$&$4^3+0^3+7^3$\\ \hline
       $1634=$&$1^4+6^4+3^4+4^4$\\
       $8208=$&$8^4+2^4+0^4+8^4$\\
       $9474=$&$9^4+4^4+7^4+4^4$\\ \hline
       $54748=$&$5^5+4^5+7^5+4^5+8^5$\\
       $92727=$&$9^5+2^5+7^5+2^5+7^5$\\
       $93084=$&$9^5+3^5+0^5+8^5+4^5$\\ \hline
       $548834=$&$5^6+4^6+8^6+8^6+3^6+4^6$\\ \hline
  \end{longtable}
\end{center}

The list of all Narcissistic numbers in base 10 is known, \citep{MathWorldNarcissisticNumber}. Besides those in Table \ref{TabelNumereNarcisissticBaza10} we still have the following Narcissistic numbers: 1741725, 4210818, 9800817,  9926315, 24678050, 24678051, 88593477, 146511208, 472335975, 534494836, 912985153, 4679307774, 32164049650, 32164049651, \ldots \cite[A005188]{SloaneOEIS}.

The biggest Narcissistic numbers have 39 digits; they are:
\begin{multline*}
  115132219018763992565095597973971522400,\\
   115132219018763992565095597973971522401~.
\end{multline*}

\begin{obs}
  As known, the digits of numeration base 16 are: 0, 1, 2, 3, 4, 5, 6, 7, 8, 9, 10 denoted respectively by $a$, 11 with $b$, 12 with $c$, 13 with $d$, 14 with $e$ and $15$ with $f$. Naturally, for bases bigger than 16 we use the following digits notation: $16=g$, $17=h$, $18=i$, $19=j$, $20=k$, $21=\ell$, $22=m$, $23=n$, $24=o$, $25=p$, $26=q$, $27=r$, $28=s$, $29=t$, $30=u$, $31=v$, $32=w$, $33=x$, $34=y$, $35=z$, $36=A$, $37=B$, $38=C$, $39=D$, $40=E$, $41=F$, $42=G$, $43=H$, $44=I$, \ldots~.
\end{obs}

\begin{center}
  \begin{longtable}{|rrl|}
    \caption{Narcissistic numbers with $b=11$ of (\ref{DcGenNarcisissistic})}\\  \hline
    \endfirsthead
    \endhead
    \multicolumn{3}{r}{\textit{Continued on next page}} \\
    \endfoot
    \endlastfoot
       $61=$&$56_{(11)}=$&$5^2+6^2$\\
       $72=$&$66_{(11)}=$&$6^2+6^2$\\ \hline
       $126=$&$105_{(11)}=$&$1^3+0^3+5^3$\\
       $370=$&$307_{(11)}=$&$3^3+0^3+7^3$\\
       $855=$&$708_{(11)}=$&$7^3+0^3+8^3$\\
       $1161=$&$966_{(11)}=$&$9^3+6^3+6^3$\\
       $1216=$&$a06_{(11)}=$&$10^3+0^3+6^3$\\
       $1280=$&$a64_{(11)}=$&$10^3+6^3+4^3$\\ \hline
       $10657=$&$8009_{(11)}=$&$8^4+0^4+0^4+9^4$\\ \hline
       $16841=$&$11720_{(11)}=$&$1^5+1^5+7^5+2^5+0^5$\\
       $16842=$&$11721_{(11)}=$&$1^5+1^5+7^5+2^5+1^5$\\
       $17864=$&$12470_{(11)}=$&$1^5+2^5+4^5+7^5+0^5$\\
       $17865=$&$12471_{(11)}=$&$1^5+2^5+4^5+7^5+1^5$\\
       $36949=$&$25840_{(11)}=$&$2^5+5^5+8^5+4^5+0^5$\\
       $36950=$&$25841_{(11)}=$&$2^5+5^5+8^5+4^5+1^5$\\
       $63684=$&$43935_{(11)}=$&$4^5+3^5+9^5+3^5+5^5$\\
       $66324=$&$45915_{(11)}=$&$4^5+5^5+9^5+1^5+5^5$\\
       $71217=$&$49563_{(11)}=$&$4^5+9^5+5^5+6^5+3^5$\\
       $90120=$&$61788_{(11)}=$&$6^5+1^5+7^5+8^5+8^5$\\
       $99594=$&$68910_{(11)}=$&$6^5+8^5+9^5+1^5+0^5$\\
       $99595=$&$68911_{(11)}=$&$6^5+8^5+9^5+1^5+1^5$\\
       $141424=$&$97288_{(11)}=$&$9^5+7^5+2^5+8^5+8^5$\\
       $157383=$&$a8276_{(11)}=$&$10^5+8^5+2^5+7^5+6^5$\\ \hline
  \end{longtable}
\end{center}

\begin{center}
  \begin{longtable}{|rrl|}
    \caption{Narcissistic numbers with $b=12$ of (\ref{DcGenNarcisissistic})}\\ \hline
    \endfirsthead
    \endhead
    \multicolumn{3}{r}{\textit{Continued on next page}} \\
    \endfoot
    \endlastfoot
       $29=$&$25_{(12)}=$&$2^2+5^2$\\
       $125=$&$a5_{(12)}=$&$10^2+5^2$\\ \hline
       $811=$&$577_{(12)}=$&$5^3+7^3+7^3$\\
       $944=$&$668_{(12)}=$&$6^3+6^3+8^3$\\
       $1539=$&$a83_{(12)}=$&$10^3+8^3+3^3$\\ \hline
       $28733=$&$14765_{(12)}=$&$1^5+4^5+7^5+6^5+5^5$\\
       $193084=$&$938a4_{(12)}=$&$9^5+3^5+8^5+10^5+4^5$\\ \hline
       $887690=$&$369862_{(12)}=$&$3^6+6^6+9^6+8^6+6^6+2^6$\\ \hline
  \end{longtable}
\end{center}

\begin{center}
  \begin{longtable}{|rrl|}
    \caption{Narcissistic numbers with $b=13$ of (\ref{DcGenNarcisissistic})}\\ \hline
    \endfirsthead
    \endhead
    \multicolumn{3}{r}{\textit{Continued on next page}} \\
    \endfoot
    \endlastfoot
       $17=$&$14_{(13)}=$&$1^2+4^2$\\
       $45=$&$36_{(13)}=$&$3^2+6^2$\\
       $85=$&$67_{(13)}=$&$6^2+7^2$\\
       $98=$&$77_{(13)}=$&$7^2+7^2$\\
       $136=$&$a6_{(13)}=$&$10^2+6^2$\\
       $160=$&$c4_{(13)}=$&$12^2+4^2$\\ \hline
       $793=$&$490_{(13)}=$&$4^3+9^3+0^3$\\
       $794=$&$491_{(13)}=$&$4^3+9^3+1^3$\\
       $854=$&$509_{(13)}=$&$5^3+0^3+9^3$\\
       $1968=$&$b85_{(13)}=$&$11^3+8^3+5^3$\\ \hline
       $8194=$&$3964_{(13)}=$&$3^4+9^4+6^4+4^4$\\ \hline
       $62481=$&$22593_{(13)}=$&$2^5+2^5+5^5+9^5+3^5$\\
       $167544=$&$5b350_{(13)}=$&$5^5+11^5+3^5+5^5+0^5$\\
       $167545=$&$5b351_{(13)}=$&$5^5+11^5+3^5+5^5+1^5$\\
       $294094=$&$a3b28_{(13)}=$&$10^5+3^5+11^5+2^5+8^5$\\
       $320375=$&$b2a93_{(13)}=$&$11^5+2^5+10^5+9^5+3^5$\\
       $323612=$&$b43b3_{(13)}=$&$11^5+4^5+3^5+11^5+3^5$\\
       $325471=$&$b51b3_{(13)}=$&$11^5+5^5+1^5+11^5+3^5$\\
       $325713=$&$b533b_{(13)}=$&$11^5+5^5+3^5+3^5+11^5$\\
       $350131=$&$c34a2_{(13)}=$&$12^5+3^5+4^5+10^5+2^5$\\
       $365914=$&$ca723_{(13)}=$&$12^5+10^5+7^5+2^5+3^5$\\ \hline
  \end{longtable}
\end{center}

\begin{center}
  \begin{longtable}{|rrl|}
    \caption{Narcissistic numbers with $b=14$ of (\ref{DcGenNarcisissistic})}\\  \hline
    \endfirsthead
    \endhead
    \multicolumn{3}{r}{\textit{Continued on next page}} \\
    \endfoot
    \endlastfoot
       $244=$&$136_{(14)}=$&$1^3+3^3+6^3$\\
       $793=$&$409_{(14)}=$&$4^3+0^3+9^3$\\ \hline
       $282007=$&$74ab5_{(14)}=$&$7^5+4^5+10^5+11^5+5^5$\\ \hline
    \end{longtable}
    \end{center}

   \begin{center}
    \begin{longtable}{|rrl|}
    \caption{Narcissistic numbers with $b=15$ of (\ref{DcGenNarcisissistic})}\\  \hline
    \endfirsthead
    \endhead
    \multicolumn{3}{r}{\textit{Continued on next page}} \\
    \endfoot
    \endlastfoot
       $113=$&$78_{(15)}=$&$7^2+8^2$\\
       $128=$&$88_{(15)}=$&$8^2+8^2$\\ \hline
       $2755=$&$c3a_{(15)}=$&$12^3+3^3+10^3$\\
       $3052=$&$d87_{(15)}=$&$13^3+8^3+7^3$\\ \hline
       $5059=$&$1774_{(15)}=$&$1^4+7^4+7^4+4^4$\\
       $49074=$&$e819_{(15)}=$&$14^4+8^4+1^4+9^4$\\
       $49089=$&$e829_{(15)}=$&$14^4+8^4+2^4+9^4$\\ \hline
       $386862=$&$7995c_{(15)}=$&$7^5+9^5+9^5+5^5+12^5$\\
       $413951=$&$829bb_{(15)}=$&$8^5+2^5+9^5+11^5+11^5$\\
       $517902=$&$a36bc_{(15)}=$&$10^5+3^5+6^5+11^5+12^5$\\ \hline
  \end{longtable}
\end{center}

\begin{center}
  \begin{longtable}{|rrl|}
    \caption{Narcissistic numbers with $b=16$ of (\ref{DcGenNarcisissistic})}\\ \hline
    \endfirsthead
    \endhead
    \multicolumn{3}{r}{\textit{Continued on next page}} \\
    \endfoot
    \endlastfoot
       $342=$&$156_{(16)}=$&$1^3+5^3+6^3$\\
       $371=$&$173_{(16)}=$&$1^3+7^3+3^3$\\
       $520=$&$208_{(16)}=$&$2^3+0^3+8^3$\\
       $584=$&$248_{(16)}=$&$2^3+4^3+8^3$\\
       $645=$&$285_{(16)}=$&$2^3+8^3+5^3$\\
       $1189=$&$4a5_{(16)}=$&$4^3+10^3+5^3$\\
       $1456=$&$5b0_{(16)}=$&$5^3+11^3+0^3$\\
       $1457=$&$5b1_{(16)}=$&$5^3+11^3+1^3$\\
       $1547=$&$60b_{(16)}=$&$6^3+0^3+11^3$\\
       $1611=$&$64b_{(16)}=$&$6^3+4^3+11^3$\\
       $2240=$&$8c0_{(16)}=$&$8^3+12^3+0^3$\\
       $2241=$&$8c1_{(16)}=$&$8^3+12^3+1^3$\\
       $2458=$&$99a_{(16)}=$&$9^3+9^3+10^3$\\
       $2729=$&$aa9_{(16)}=$&$10^3+10^3+9^3$\\
       $2755=$&$ac3_{(16)}=$&$10^3+12^3+3^3$\\
       $3240=$&$ca8_{(16)}=$&$12^3+10^3+8^3$\\
       $3689=$&$e69_{(16)}=$&$14^3+6^3+9^3$\\
       $3744=$&$ea0_{(16)}=$&$14^3+10^3+0^3$\\
       $3745=$&$ea1_{(16)}=$&$14^3+10^3+1^3$\\ \hline
       $47314=$&$b8d2_{(16)}=$&$11^4+8^4+13^4+2^4$\\ \hline
       $79225=$&$13579_{(16)}=$&$1^5+3^5+5^5+7^5+9^5$\\
       $177922=$&$2b702_{(16)}=$&$2^5+11^5+7^5+0^5+2^5$\\
       $177954=$&$2b722_{(16)}=$&$2^5+11^5+7^5+2^5+2^5$\\
       $368764=$&$5a07c_{(16)}=$&$5^5+10^5+0^5+7^5+12^5$\\
       $369788=$&$5a47c_{(16)}=$&$5^5+10^5+4^5+7^5+12^5$\\
       $786656=$&$c00e0_{(16)}=$&$12^5+0^5+0^5+14^5+0^5$\\
       $786657=$&$c00e1_{(16)}=$&$12^5+0^5+0^5+14^5+1^5$\\
       $787680=$&$c04e0_{(16)}=$&$12^5+0^5+4^5+14^5+0^5$\\
       $787681=$&$c04e1_{(16)}=$&$12^5+0^5+4^5+14^5+1^5$\\
       $811239=$&$c60e7_{(16)}=$&$12^5+6^5+0^5+14^5+7^5$\\
       $812263=$&$c64e7_{(16)}=$&$12^5+6^5+4^5+14^5+7^5$\\
       $819424=$&$c80e0_{(16)}=$&$12^5+8^5+0^5+14^5+0^5$\\
       $819425=$&$c80e1_{(16)}=$&$12^5+8^5+0^5+14^5+1^5$\\
       $820448=$&$c84e0_{(16)}=$&$12^5+8^5+4^5+14^5+0^5$\\
       $820449=$&$c84e1_{(16)}=$&$12^5+8^5+4^5+14^5+1^5$\\
       $909360=$&$de030_{(16)}=$&$13^5+14^5+0^5+3^5+0^5$\\
       $909361=$&$de031_{(16)}=$&$13^5+14^5+0^5+3^5+1^5$\\
       $910384=$&$de430_{(16)}=$&$13^5+14^5+4^5+3^5+0^5$\\
       $910385=$&$de431_{(16)}=$&$13^5+14^5+4^5+3^5+1^5$\\
       $964546=$&$eb7c2_{(16)}=$&$14^5+11^5+7^5+12^5+2^5$\\ \hline
  \end{longtable}
\end{center}

The following numbers are Narcissistic in two dif{}ferent numeration bases:
\begin{center}
  \begin{longtable}{|rrr|}
    \caption{Narcissistic numbers in two bases}\\ \hline
    \endfirsthead
    \endhead
    \multicolumn{3}{r}{\textit{Continued on next page}} \\
    \endfoot
    \endlastfoot
     $n_{(10)}$ & $n_{(b_1)}$ & $n_{(b_2)}$ \\
     \hline
     17 & $122_{(3)}$ & $14_{(13)}$ \\
     28 & $130_{(4)}$ & $103_{(5)}$ \\
     29 & $131_{(4)}$ & $25_{(12)}$ \\
     45 & $63_{(7)}$ & $36_{(13)}$ \\
     126 & $150_{(9)}$ & $105_{(11)}$ \\
     133 & $250_{(7)}$ & $205_{(8)}$ \\
     370 & $370_{(10)}$ & $307_{(11)}$ \\
     371 & $371_{(10)}$ & $173_{(16)}$ \\
     793 & $490_{(13)}$ & $409_{(14)}$ \\
     2755 & $c3a_{(15)}$ & $ac3_{(16)}$ \\
     17864 & $42710_{(8)}$ & $12470_{(11)}$ \\
     17865 & $42711_{(8)}$ & $12471_{(11)}$ \\
     \hline
 \end{longtable}
\end{center}

The question is whether there are multiple Narcissistic numbers in dif{}ferent numeration bases. There is the number $10261_{(10)}$ which is triple Narcissistic, in numeration bases $b=31, 32$ and 49, indeed
\begin{equation*}
  \begin{array}{rcrcl}
       10261_{(10)}&=&al0_{(31)}&=&10^3+21^3+0^3~,\\
       10261_{(10)}&=&a0\ell_{(32)}&=&10^3+0^3+21^3~,\\
       10261_{(10)}&=&4dk_{(49)}&=&4^3+13^3+20^3~.
  \end{array}
\end{equation*}

If we consider the numeration bases less or equal to 100 we have 3 triple Narcissistic numbers, in dif{}ferent numeration bases:
\begin{center}
  \begin{tabular}{|r|c|c|c|}
     \hline
     $n_{(10)}$ & $n_{(b_1)}$ & $n_{(b_2)}$ & $n_{(b_3)}$ \\
     \hline
      125 & $a5_{(12)}$ & $5a_{(23)}$ & $2b_{(57)}$ \\
      2080 & $Ic_{(47)}$ & $As_{(57)}$ & $sA_{(73)}$ \\
      10261 & $a\ell0_{(31)}$ & $a0\ell_{(32)}$ & $4dk_{(49)}$ \\
     \hline
  \end{tabular}
\end{center}
where $k=20$, $\ell=21$, $s=28$, $A=36$ and $I=44$.

To note that for $n\in\Ns$, $n\le10^6$ we don\rq{t} have Narcissistic numbers in four dif{}ferent bases, where $b\in\Ns$ and $2\le b\le100$.

\subsection{Inverse Narcissistic Numbers}

\begin{defn}\label{DefnInverse Narcissistic}
  A number in base $b$ ($b\in\Ns$, $b\ge2$) is \emph{inverse Narcissistic} if
  \begin{equation}\label{ConditiaInvNarNumber}
    \overline{d_1d_2\dots d_m}_{(b)}=m^{d_1}+m^{d_2}+\ldots+m^{d_m}~,
  \end{equation}
  where $d_k\in\set{0,1,\ldots,b-1}$, for $k\in\Ind{m}$.
\end{defn}

Let
\begin{equation}\label{EcInvNarcisissistic}
    \log_b(b^{m-1})=\log_b\left(m^b\right)~.
  \end{equation}

\begin{lem}\label{LemmaInvNarcisissistic}
  The numbers $n_{(b)}$ with the property \emph{(\ref{ConditiaInvNarNumber})} can\rq{t} have more than $n_d$ digits, where $n_d=\lceil s\rceil$, and $s$ is the solution equation \emph{(\ref{EcInvNarcisissistic})}.
\end{lem}
\begin{proof}
  The smallest number in base $b$ with $m$ digits is $b^{m-1}$. The biggest number with property (\ref{ConditiaInvNarNumber}) is $m\cdot m^{b-1}$. If the smallest number in base $b$ with $m$ digits is bigger than the biggest number with property (\ref{ConditiaInvNarNumber}), i.e. $b^{m-1}\ge m^b$, then this inequation provides the limit of numbers which can meet the condition (\ref{ConditiaInvNarNumber}).

  If we logarithmize both terms of the inequation, we get an inequation which establishes the limit of possible digits number. This leads to solving the equation (\ref{EcInvNarcisissistic}). Let $s$ be the solution of the equation (\ref{EcInvNarcisissistic}), but keeping into account that the digits number of a number is an integer resulting that $n_d=\lceil s\rceil$.
\end{proof}

\begin{cor}
  The maximum number of digits of the numbers in base $b$, which meet the condition \emph{(\ref{EcInvNarcisissistic})} are given in the Table \emph{\ref{NcInvNarcisissistic}}
\begin{table}[h]
  \centering
  \begin{tabular}{|c|c|c|c|c|c|c|c|c|c|c|c|c|c|c|c|}
    \hline
    b & $2$ & $3$ & $4$ & $5$ & $6$ & $7$ & $8$ & $9$ & $10$ & $11$ & $12$ & $13$ & $14$ & $15$ & $16$ \\ \hline
    $n_d$ & $7$ & $6$ & $7$ & $8$ & $8$ & $9$ & $10$ & $11$ & $12$ & $13$ & $14$ & $15$ & $16$ & $17$ & $18$ \\
    \hline
  \end{tabular}
  \caption{The maximum number of digits of the numbers in base $b$}\label{NcInvNarcisissistic}
\end{table}
\end{cor}

Let the search domain be def{}ined by
\begin{equation}\label{DcbInvNarcisissistic}
  Dc_b=\set{b,b+1,\ldots,n_d^b}~,
\end{equation}
where $n_d^b$ is not bigger than $10^7$, and $n_d$ are given in the Table (\ref{NcInvNarcisissistic}). We avoid
single digit numbers $0$, $1$, $2$, \ldots, $b-1$, because $1=1^1$ is a trivial solution, and $0\neq1^0$, $2\neq1^2$, \ldots $b-1\neq1^{b-1}$, for any base $b$.

Therefore, the search domain are:
\begin{eqnarray}
  Dc_2&=&\set{2,3,\ldots,49},\ \textnormal{where}\ 49=7^2,\label{Dc2InvNarcisissitic} \\
  Dc_3&=&\set{3,4,\ldots,216},\ \textnormal{where}\ 216=6^3,\label{Dc3InvNarcisissitic} \\
  Dc_4&=&\set{4,5,\ldots,2401},\ \textnormal{where}\ 2401=7^4,\label{Dc4InvNarcisissitic} \\
  Dc_5&=&\set{5,6,\ldots,32768},\ \textnormal{where}\ 32768=8^5,\label{Dc5InvNarcisissitic} \\
  Dc_6&=&\set{6,7,\ldots,262144},\ \textnormal{where}\ 262144=8^6,\label{Dc6InvNarcisissitic} \\
  Dc_7&=&\set{7,8,\ldots,4782969},\ \textnormal{where}\ 4782969=9^7,\label{Dc7InvNarcisissitic}
\end{eqnarray}
and for $b=8,9,\ldots,16$,
\begin{equation}\label{DcGenInvNarcisissistic}
  Dc_b=\set{b,b+1,\ldots,10^7}\textnormal{for}\ \ b=8,9,\ldots,16~.
\end{equation}

We determined all the inverse Narcissistic numbers in numeration bases $b=2,3,\ldots,16$ on the search domains (\ref{Dc2InvNarcisissitic} -- \ref{DcGenInvNarcisissistic}).

\begin{center}
  \begin{longtable}{|rrl|}
    \caption{Inverse narcissistic numbers of (\ref{Dc2InvNarcisissitic} -- \ref{DcGenInvNarcisissistic})}\\ \hline
    \endfirsthead
    \endhead
    \multicolumn{3}{r}{\textit{Continued on next page}} \\
    \endfoot
    \endlastfoot
       $10=$&$1010_{(2)}=$&$4^1+4^0+4^1+4^0$\\
       $13=$&$1101_{(2)}=$&$4^1+4^1+4^0+4^1$\\ \hline
       $3=$&$10_{(3)}=$&$2^1+2^0$\\
       $4=$&$11_{(3)}=$&$2^1+2^1$\\
       $8=$&$22_{(3)}=$&$2^2+2^2$\\ \hline
       $6=$&$12_{(4)}=$&$2^1+2^2$\\
       $39=$&$213_{(4)}=$&$3^2+3^1+3^3$\\ \hline
       $33=$&$113_{(5)}=$&$3^1+3^1+3^3$\\
       $117=$&$432_{(5)}=$&$3^4+3^3+3^2$\\ \hline
       $57=$&$133_{(6)}=$&$3^1+3^3+3^3$\\
       $135=$&$343_{(6)}=$&$3^3+3^4+3^3$\\
       $340=$&$1324_{(6)}=$&$4^1+4^3+4^2+4^4$\\
       $3281=$&$23105_{(6)}=$&$5^2+5^3+5^1+5^0+5^5$\\ \hline
       $10=$&$13_{(7)}=$&$2^1+2^3$\\
       $32=$&$44_{(7)}=$&$2^4+2^4$\\
       $245=$&$500_{(7)}=$&$3^5+3^0+3^0$\\
       $261=$&$522_{(7)}=$&$3^5+3^2+3^2$\\ \hline
       $20=$&$24_{(8)}=$&$2^2+2^4$\\
       $355747=$&$1266643_{(8)}=$&$7^1+7^2+7^6+7^6+7^6+7^4+7^3$\\ \hline
       $85=$&$104_{(9)}=$&$3^1+3^0+3^4$\\
       $335671=$&$561407_{(9)}=$&$6^5+6^6+6^1+6^4+6^0+6^7$\\
       $840805=$&$1521327_{(9)}=$&$7^1+7^5+7^2+7^1+7^3+7^2+7^7$\\
       $842821=$&$1524117_{(9)}=$&$7^1+7^5+7^2+7^4+7^1+7^1+7^7$\\
       $845257=$&$1527424_{(9)}=$&$7^1+7^5+7^2+7^7+7^4+7^2+7^4$\\ \hline
       $4624=$&$4624_{(10)}=$&$4^4+4^6+4^2+4^4$\\ \hline
       $68=$&$62_{(11)}=$&$2^6+2^8$\\
       $16385=$&$11346_{(11)}=$&$5^1+5^1+5^3+5^4+5^6$\\ \hline
        & --- & \\ \hline
        & --- & \\ \hline
       $18=$&$14_{(14)}=$&$2^1+2^4$\\
       $93905=$&$26317_{(14)}=$&$5^2+5^6+5^3+5^1+5^7$\\
       $156905=$&$41277_{(14)}=$&$5^4+5^1+5^2+5^7+5^7$\\
       $250005=$&$67177_{(14)}=$&$5^6+5^7+5^1+5^7+5^7$\\ \hline
        & --- & \\ \hline
       $4102=$&$1006_{(16)}=$&$4^1+4^0+4^0+4^6$\\ \hline
  \end{longtable}
\end{center}

\subsection{M\"{u}nchhausen Numbers}

M\"{u}nchhausen numbers are a subchapter of Visual Representation Numbers, \cite[pp. 163--175]{Madachy1979},
\cite[pp. 169--171]{Pickover1995}, \citep{Pickover2001}, \cite[A046253]{SloaneOEIS}, \citep{MathWorldMunchhausenNumber}.

\begin{defn}\label{DefnMunchhausenNumber}
  A number in base $b$ ($b\in\Ns$, $b\ge2$) is \emph{M\"{u}nchhausen number} if
  \begin{equation}\label{MunchhausenNumber}
    \overline{d_1d_2\dots d_m}_{(b)}=d_1^{d_1}+d_2^{d_2}+\ldots+d_m^{d_m}~,
  \end{equation}
  where $d_k\in\set{0,1,\ldots,b-1}$, for $k\in\Ind{m}$.
\end{defn}
We specify that by convention we have $0^0=1$. Let the search domains be:
\begin{eqnarray}
  Dc_2 &=& \set{2,3},\ \textnormal{where}\ 3=3\cdot1^1~,\label{Dc2Munchhaussen} \\
  Dc_3 &=& \set{3,4,\ldots,16},\ \textnormal{where}\ 16=4\cdot2^2~,\label{Dc3Munchhaussen} \\
  Dc_4 &=& \set{4,5,\ldots,135},\ \textnormal{where}\ 135=5\cdot3^3~,\label{Dc4Munchhaussen} \\
  Dc_5 &=& \set{5,6,\ldots,1536},\ \textnormal{where}\ 1536=6\cdot4^4~,\label{Dc5Munchhaussen} \\
  Dc_6 &=& \set{6,7,\ldots,21875},\ \textnormal{where}\ 21875=7\cdot5^5~,\label{Dc6Munchhaussen} \\
  Dc_7 &=& \set{7,8,\ldots,373248},\ \textnormal{where}\ 373248=8\cdot6^6~,\label{Dc7Munchhaussen} \\
  Dc_8 &=& \set{8,9,\ldots,7411887},\ \textnormal{where}\ 7411887=9\cdot7^7~,\label{Dc8Munchhaussen}
\end{eqnarray}
and
\begin{equation}\label{DcGenMunchhaussen}
  Dc_b=\set{b,b+1,\ldots,2\cdot10^7}\ \ \textnormal{for}\ \ b=9,10,\ldots,16~.
\end{equation}

These search domains were determined by inequality $b^{m-1}\ge m\cdot(b-1)^{b-1}$, where $b^{m-1}$ is the smallest number of $m$ digits in numeration base $b$, and $m\cdot (b-1)^{b-1}$ is the biggest number that fulf{}ills the condition (\ref{MunchhausenNumber}) for being a \emph{M\"{u}nchhaussen number}. The solution, in relation to $m$, of the equation $\log_b(b^{m-1})=\log_b(m\cdot (b-1)^{b-1})$ is the number of digits in base $b$ of the number from
which we can\rq{t} have anymore \emph{M\"{u}nchhaussen number}. The number of digits in base $b$ for \emph{M\"{u}nchhaussen number} are in Table \ref{NcMunchhaussen}.
\begin{table}[h]
  \centering
   \begin{tabular}{|c|c|c|c|c|c|c|c|c|c|c|c|c|c|c|c|}
    \hline
    $b$ & $2$ & $3$ & $4$ & $5$ & $6$ & $7$ & $8$ & $9$ & $10$ & $11$ & $12$ & $13$ & $14$ & $15$ & $16$ \\ \hline
    $n_d$ & $3$ & $4$ & $5$ & $6$ & $7$ & $8$ & $9$ & $10$ & $11$ & $12$ & $13$ & $14$ & $15$ & $16$ & $17$ \\
    \hline
  \end{tabular}
  \caption{The number of digits in base $b$ for M\"{u}nchhaussen number}\label{NcMunchhaussen}
\end{table}

The limits from where the search for \emph{M\"{u}nchhaussen numbers} is useless are $n_d(b-1)^{b-1}$, where $n_d$ is taken from the Table \ref{NcMunchhaussen} corresponding to $b$.  The limits are to be found in the search domains (\ref{Dc2Munchhaussen}--\ref{Dc8Munchhaussen}). For limits bigger than $2\cdot10^7$ we considered the limit $2\cdot10^7$.

There were determined all \emph{M\"{u}nchhausen numbers} in base $b=2,3,\ldots,16$ on search domains (\ref{Dc2Munchhaussen}--\ref{DcGenMunchhaussen})~.

\begin{center}
  \begin{longtable}{|rrl|}
    \caption{M\"{u}nchhausen numbers of (\ref{Dc2Munchhaussen}--\ref{DcGenMunchhaussen})}\\ \hline
    \endfirsthead
    \endhead
    \multicolumn{3}{r}{\textit{Continued on next page}} \\
    \endfoot
    \endlastfoot
       $2=$&$10_{(2)}=$&$1^1+0^0$\\ \hline
       $5=$&$12_{(3)}=$&$1^1+2^2$\\
       $8=$&$22_{(3)}=$&$2^2+2^2$\\ \hline
       $29=$&$131_{(4)}=$&$1^1+3^3+1^1$\\
       $55=$&$313_{(4)}=$&$3^3+1^1+3^3$\\ \hline
       &---&\\ \hline
       $3164=$&$22352_{(6)}=$&$2^2+2^2+3^3+5^5+2^2$\\
       $3416=$&$23452_{(6)}=$&$2^2+3^3+4^4+5^5+2^2$\\ \hline
       $3665=$&$13454_{(7)}=$&$1^1+3^3+4^4+5^5+4^4$\\ \hline
       &---&\\ \hline
       $28=$&$31_{(9)}=$&$3^3+1^1$\\
       $96446=$&$156262_{(9)}=$&$1^1+5^5+6^6+2^2+6^6+2^2$\\
       $923362=$&$1656547_{(9)}=$&$1^1+6^6+5^5+6^6+5^5+4^4+7^7$\\ \hline
       $3435=$&$3435_{(10)}=$&$3^3+4^4+3^3+5^5$\\ \hline
       &---&\\ \hline
       &---&\\ \hline
       $93367=$&$33661_{(13)}=$&$3^3+3^3+6^6+6^6+1^1$\\ \hline
       $31=$&$23_{(14)}=$&$2^2+3^3$\\
       $93344=$&$26036_{(14)}=$&$2^2+6^6+0^0+3^3+6^6$\\ \hline
       &---&\\ \hline
       &---&\\ \hline
  \end{longtable}
\end{center}

Of course, to these solutions we can also add the trivial solution 1 in any numeration base. If we make the convention $0^0=0$ then in base $10$ we also have the \emph{M\"{u}nchhausen number} $438579088$, \cite[A046253 ]{SloaneOEIS}.

Returning to the convention $0^0=1$we can say that the number $\overline{d_1d_2\dots d_m}_{(b)}$ is almost \emph{M\"{u}nchhausen number} meaning that
\[
 \abs{\overline{d_1d_2\dots d_m}_{(b)}-\big(d_1^{d_1}+d_2^{d_2}+\ldots+d_m^{d_m}\big)}\le\varepsilon~,
\]
where $\varepsilon$ can be $0$ and then we have \emph{M\"{u}nchhausen numbers}, or if we have $\varepsilon=1,2,\ldots$ then we can say that we have almost \emph{M\"{u}nchhausen numbers} with the dif{}ference of most $1$, $2$, \ldots~. In this regard, the number $438579088$ is an almost \emph{M\"{u}nchhausen number} with the dif{}ference of most $1$.

\subsection{Numbers with Digits Sum in Ascending Powers}

The numbers which fulf{}ill the condition
\begin{equation}\label{ConditiaSCPC}
  n_{(b)}=\overline{d_1d_2\ldots d_m}_{(b)}=\sum_{k=1}^m d_k^k
\end{equation}
are \emph{numbers with digits sum in ascending powers}. In base $10$ we have the following \emph{numbers with digits sum in ascending powers} $89=8^1+9^2$, $135=1^1+3^2+5^3$, $175=1^1+7^2+5^3$, $518=5^1+1^2+8^3$, $598=5^1+9^2+8^3$, $1306=1^1+3^2+0^3+6^4$, $1676=1^1+6^2+7^3+6^4$, $2427=2^1+4^2+2^3+7^4$, \citep{MathWorldNarcissisticNumber}, \cite[A032799]{SloaneOEIS}.

We also have the trivial solutions $1$, $2$, \ldots, $b-1$, numbers with property (\ref{ConditiaSCPC}) for any base $b\ge2$.

Let $n_{(b)}$ a number in base $b$ with $m$ digits and the equation
\begin{equation}\label{EcSCPC}
  \log_b(b^{m-1})=\log_b\left(\frac{b-1}{b-2}[(b-1)^m-1]\right)~.
\end{equation}
\begin{lem}\label{LemmaSCPC}
  The numbers $n_{(b)}$ with property \emph{(\ref{ConditiaSCPC})} can\rq{t} have more than $n_d$ digits, where $n_d=\lceil s\rceil$, and $s$ is the solution of the equation \emph{(\ref{EcSCPC})}.
\end{lem}
\begin{proof} The smallest number in base $b$ with $m$ digits is $b^{m-1}$.The biggest number
with property (\ref{ConditiaSCPC}) is
  \[
    (b-1)^1+(b-1)^2+\ldots+(b-1)^m=(b-1)\frac{(b-1)^m-1}{(b-1)-1}=\frac{b-1}{b-2}[(b-1)^m-1]~.
  \]
  If the smallest number in base $b$ with $m$ digits is bigger than the bigger number with property (\ref{ConditiaSCPC}), i.e.
  \[
   b^{m-1}\ge\frac{b-1}{b-2}[(b-1)^m-1]~,
  \]
  then the inequality provides the limit of numbers that can fulf{}ill the condition (\ref{ConditiaSCPC}).

  If we logarithmize both terms of the inequation, we obtain an inequation that establishes the limit of possible digits number. It gets us to the solution of the equation (\ref{EcSCPC}). Let $s$ be the solution of the equation (\ref{EcSCPC}), but keeping into account that the digits number of a number is an integer, it follows that $n_d=\lceil s\rceil$.
\end{proof}

\begin{cor}
  The maximum digits number of numbers in base $b\ge3$, which fulf{}ill the condition \emph{(\ref{ConditiaSCPC})} are given in the Table \emph{\ref{NcSCPC}}.
\begin{table}[h]
  \centering
    \begin{tabular}{|c|c|c|c|c|c|c|c|c|c|c|c|c|c|c|}
      \hline
      b & $3$ & $4$ & $5$ & $6$ & $7$ & $8$ & $9$ & $10$ & $11$ & $12$ & $13$ & $14$ & $15$ & $16$ \\ \hline
      $n_d$ & $5$ & $7$ & $9$ & $12$ & $14$ & $17$ & $20$ & $23$ & $27$ & $30$ & $34$ & $37$ & $41$ & $45$ \\
      \hline
    \end{tabular}
  \caption{The maximum digits number of numbers in base $b\ge3$}\label{NcSCPC}
\end{table}
\end{cor}

Let the search domains be def{}ined by
\begin{equation}\label{DcbBSCPC}
  Dc_b=\set{b,b+1,\ldots,\frac{b-1}{b-2}\left[(b-1)^{n_d}-1\right]}~,
\end{equation}
where $\frac{b-1}{b-2}\left[(b-1)^{n_d}-1\right]$, $n_{(b)}\le2\cdot10^7$, and $n_d$ are given in the Table (\ref{NcSCPC}). We avoid numbers having only one digit $0$, $1$, $2$, \ldots, $b-1$, since they are trivial solutions, because $0=0^1$, $1=1^1$, $2=2^1$, \ldots, $b-1=(b-1)^1$, for any base $b$.

Then, the search domains are:
\begin{eqnarray}
  Dc_3 &=& \set{3,4,\ldots,62},\ \textnormal{where}\ 62=2(2^5-1)~,\label{Dc3SumaCifrelorLaPuteriCrescatoare} \\
  Dc_4 &=& \set{4,5,\ldots,3279},\ \textnormal{where}\ 3279=3\left(3^7-1\right)/2~,\label{Dc4SumaCifrelorLaPuteriCrescatoare} \\
  Dc_5 &=& \set{5,6,\ldots,349524},\ \textnormal{where}\ 349524=4\left(4^9-1\right)/3~,\label{Dc5SumaCifrelorLaPuteriCrescatoare}
\end{eqnarray}
and for $b=6,7,\ldots,16$ the search domains are:
\begin{equation}\label{DcGenSumaCifrelorLaPuteriCrescatoare}
  Dc_b=\set{b,b+1,\ldots,2\cdot10^7}~.
\end{equation}

All the \emph{numbers with digits sum in ascending powers} in numeration bases $b=2,3,\ldots,16$ on search domains (\ref{Dc3SumaCifrelorLaPuteriCrescatoare}--\ref{DcGenSumaCifrelorLaPuteriCrescatoare}) are given in the following table.
\begin{center}
  \begin{longtable}{|rrl|}
    \caption{Numbers with the property (\ref{ConditiaSCPC}) of (\ref{Dc3SumaCifrelorLaPuteriCrescatoare}--\ref{DcGenSumaCifrelorLaPuteriCrescatoare})}\\ \hline
    \endfirsthead
    \endhead
    \multicolumn{3}{r}{\textit{Continued on next page}} \\
    \endfoot
    \endlastfoot
    $5=$&$12_{(3)}=$&$1^1+2^2$\\ \hline
    $11=$&$23_{(4)}=$&$2^1+3^2$\\
    $83=$&$1103_{(4)}=$&$1^1+1^2+0^3+3^4$\\
    $91=$&$1123_{(4)}=$&$1^1+1^2+2^3+3^4$\\ \hline
    $19=$&$34_{(5)}=$&$3^1+4^2$\\
    $28=$&$103_{(5)}=$&$1^1+0^2+3^3$\\
    $259=$&$2014_{(5)}=$&$2^1+0^2+1^3+4^4$\\
    $1114=$&$13424_{(5)}=$&$1^1+3^2+4^3+2^4+4^5$\\
    $81924=$&$10110144_{(5)}=$&$1^1+0^2+1^3+1^4+0^5+4^6+4^7$\\ \hline
    $29=$&$45_{(6)}=$&$4^1+5^2$\\ \hline
    $10=$&$13_{(7)}=$&$1^1+3^2$\\
    $18=$&$24_{(7)}=$&$2^1+4^2$\\
    $41=$&$56_{(7)}=$&$5^1+6^2$\\
    $74=$&$134_{(7)}=$&$1^1+3^2+4^3$\\
    $81=$&$144_{(7)}=$&$1^1+4^2+4^3$\\
    $382=$&$1054_{(7)}=$&$1^1+0^2+5^3+4^4$\\
    $1336=$&$3616_{(7)}=$&$3^1+6^2+1^3+6^4$\\
    $1343=$&$3626_{(7)}=$&$3^1+6^2+2^3+6^4$\\ \hline
    $55=$&$67_{(8)}=$&$6^1+7^2$\\
    $8430=$&$20356_{(8)}=$&$2^1+0^2+3^3+5^4+6^5$\\
    $46806=$&$133326_{(8)}=$&$1^1+3^2+3^3+3^4+2^5+6^6$\\ \hline
    $71=$&$78_{(9)}=$&$7^1+8^2$\\
    $4445=$&$6078_{(9)}=$&$6^1+0^2+7^3+8^4$\\
    $17215=$&$25547_{(9)}=$&$2^1+5^2+5^3+4^4+7^5$\\
    $17621783$&$36137478_{(9)}=$&$3^1+6^2+1^3+3^4+7^5+4^6+7^7+8^8$\\ \hline
    $89=$&$89_{(10)}=$&$8^1+9^2$\\
    $135=$&$135_{(10)}=$&$1^1+3^2+5^3$\\
    $175=$&$175_{(10)}=$&$1^1+7^2+5^3$\\
    $518=$&$518_{(10)}=$&$5^1+1^2+8^3$\\
    $598=$&$598_{(10)}=$&$5^1+9^2+8^3$\\
    $1306=$&$1306_{(10)}=$&$1^1+3^2+0^3+6^4$\\
    $1676=$&$1676_{(10)}=$&$1^1+6^2+7^3+6^4$\\
    $2427=$&$2427_{(10)}=$&$2^1+4^2+2^3+7^4$\\
    $2646798=$&$2646798_{(10)}=$&$2^1+6^2+4^3+6^4+7^5+9^6+8^7$\\ \hline
    $27=$&$25_{(11)}=$&$2^1+5^2$\\
    $39=$&$36_{(11)}=$&$3^1+6^2$\\
    $109=$&$9a_{(11)}=$&$9^1+10^2$\\
    $126=$&$105_{(11)}=$&$1^1+0^2+5^3$\\
    $525=$&$438_{(11)}=$&$4^1+3^2+8^3$\\
    $580=$&$488_{(11)}=$&$4^1+8^2+8^3$\\
    $735=$&$609_{(11)}=$&$6^1+0^2+9^3$\\
    $1033=$&$85a_{(11)}=$&$8^1+5^2+10^3$\\
    $1044=$&$86a_{(11)}=$&$8^1+6^2+10^3$\\
    $2746=$&$2077_{(11)}=$&$2^1+0^2+7^3+7^4$\\
    $59178=$&$40509_{(11)}=$&$4^1+0^2+5^3+0^4+9^5$\\
    $63501=$&$43789_{(11)}=$&$4^1+3^2+7^3+8^4+9^5$\\ \hline
    $131=$&$ab_{(12)}=$&$10^1+11^2$\\ \hline
    $17=$&$14_{(13)}=$&$1^1+4^2$\\
    $87=$&$69_{(13)}=$&$6^1+9^2$\\
    $155=$&$bc_{(13)}=$&$11^1+12^2$\\
    $253=$&$166_{(13)}=$&$1^1+6^2+6^3$\\
    $266=$&$176_{(13)}=$&$1^1+7^2+6^3$\\
    $345=$&$207_{(13)}=$&$2^1+0^2+7^3$\\
    $515=$&$308_{(13)}=$&$3^1+0^2+8^3$\\
    $1754=$&$a4c_{(13)}=$&$10^1+4^2+12^3$\\
    $1819=$&$a9c_{(13)}=$&$10^1+9^2+12^3$\\
    $250002=$&$89a3c_{(13)}=$&$8^1+9^2+10^3+3^4+12^5$\\
    $1000165=$&$29031a_{(13)}=$&$2^1+9^2+0^3+3^4+1^5+10^6$\\ \hline
    $181=$&$cd_{(14)}=$&$12^1+13^2$\\
    $11336=$&$41ba_{(14)}=$&$4^1+1^2+11^3+10^4$\\
    $4844251=$&$90157d_{(14)}=$&$9^1+0^2+1^3+5^4+7^6+13^7$\\ \hline
    $52=$&$37_{(15)}=$&$3^1+7^2$\\
    $68=$&$48_{(15)}=$&$4^1+8^2$\\
    $209=$&$de_{(15)}=$&$13^1+14^2$\\
    $563=$&$278_{(15)}=$&$2^1+7^2+8^3$\\
    $578=$&$288_{(15)}=$&$2^1+8^2+8^3$\\
    $15206=$&$478b_{(15)}=$&$4^1+7^2+8^3+11^4$\\
    $29398=$&$8a9d_{(15)}=$&$8^1+10^2+9^3+13^4$\\
    $38819=$&$b77e_{(15)}=$&$11^1+7^2+7^3+14^4$\\ \hline
    $38=$&$26_{(16)}=$&$2^1+6^2$\\
    $106=$&$6a_{(16)}=$&$6^1+10^2$\\
    $239=$&$ef_{(16)}=$&$14^1+15^2$\\
    $261804=$&$3feac_{(16)}=$&$3^1+15^2+14^3+10^4+12^5$\\ \hline
  \end{longtable}
\end{center}

\begin{obs}
  T he numbers $\overline{(b-2)(b-1)}_{(b)}$ are \emph{numbers with digits sum in ascending powers} for any base $b\in\Ns$, $b\ge2$. Indeed, we have the identity $(b-2)^1+(b-1)^2=(b-2)b+(b-1)$ true for any base $b\in\Ns$, $b\ge2$, identity which proves the assertion.
\end{obs}

\subsection{Numbers with Digits Sum in Descending Powers}

Let the number $n$ with $m$ digits in base $b$, i.e. $n_{(b)}=\overline{d_1d_2\ldots d_m}_{(b)}$, where $d_k\in\set{0,1,\ldots,b-1}$ for any $k\in\Ind{m}$. Let us determine the numbers that fulf{}ill the
condition
\[
  n_{(b)}=d_1\cdot b^{m-1}+d_2\cdot b^{m-2}+\ldots+d_m\cdot b^0=d_1^m+d_2^{m-1}+\ldots+d_m^1~.
\]
Such numbers do not exist because
\[
  n_{(b)}=d_1\cdot b^{m-1}+d_2\cdot b^{m-2}+\ldots+d_m\cdot b^0>d_1^m+d_2^{m-1}+\ldots+d_m^1~.
\]

Naturally, we will impose the following condition for numbers with digits sum in descending powers:
\begin{equation}\label{ConditiaCifreLaPuteriDescrescatoare}
  n_{(b)}=d_1\cdot b^{m-1}+d_2\cdot b^{m-2}+\ldots+d_m\cdot b^0=d_1^{m+1}+d_2^m+\ldots+d_m^2~.
\end{equation}

The biggest number with $m$ digits in base $b$ is $b^m-1$. The biggest number with $m$ digits sum in descending powers (starting with power $m+1$) is $(b-1)^{m+1}+(b-1)^m+\ldots+(b-1)^2$. If the inequation $(b-1)^2[(b-1)^m-1]/(b-2)\le b^m-1$, relation to $m$, has integer solutions $\ge1$, then the condition (\ref{ConditiaCifreLaPuteriDescrescatoare}) makes sens. The inequation reduces to solving the equation
\[
 (b-1)^2\frac{(b-1)^m-1}{b-2}=b^m-1~,
\]
which, after logarithmizing in base $b$ can be solved in relation to $m$. The solution $m$ represents the digits number of the number $n_{(b)}$, therefore we should take the superior integer part of the solution..

The maximum digits number of the numbers in base $b\ge3$, which fulf{}ills the condition (\ref{ConditiaCifreLaPuteriDescrescatoare}) are given Table \ref{NcSDPC}.
\begin{table}[h]
  \centering
   \begin{tabular}{|c|c|c|c|c|c|c|c|c|c|c|c|c|c|c|}
    \hline
    b & $3$ & $4$ & $5$ & $6$ & $7$ & $8$ & $9$ & $10$ & $11$ & $12$ & $13$ & $14$ & $15$ & $16$ \\ \hline
    $n_d$ & $7$ & $11$ & $15$ & $20$ & $26$ & $32$ & $38$ & $44$ & $51$ & $58$ & $65$ & $72$ & $79$ & $86$ \\
    \hline
  \end{tabular}
  \caption{The maximum digits number of the numbers in base $b\ge3$}\label{NcSDPC}
\end{table}

Let the search domain be def{}ined by
\begin{equation}\label{DcbBSDPC}
  Dc_b=\set{b,b+1,\ldots,\frac{(b-1)^2}{b-2}\left[(b-1)^m-1\right]}~,
\end{equation}
where $(b-1)^2\left[(b-1)^m-1\right]/(b-2)$ is not bigger than $2\cdot10^7$, and $n_d$ are given in the Table \ref{NcSDPC}. We avoid numbers $1$, $2$, \ldots, $b-1$, because $0=0^2$ and $1=1^2$ are trivial solutions, and $2\neq2^2$, $3\neq3^2$, \ldots, $b-1\neq (b-1)^2$.

Therefore, the search domains are:
\begin{eqnarray}
  Dc_3 &=& \set{3,4,\ldots,508},\ \textnormal{where}\ 508=2^2(2^7-1)\label{Dc3SumaCifrelorLaPuteriDescrescatoare} \\
  Dc_4 &=& \set{4,5,\ldots,797157},\ \textnormal{where}\ 797157=3^2(3^{11}-1)/2\label{Dc4SumaCifrelorLaPuteriDescrescatoare}
\end{eqnarray}
and for $b=5,6,\ldots,16$ the search domains are
\begin{equation}\label{DcGenSumaCifrelorLaPuteriDescrescatoare}
  Dc_b=\set{b,b+1,\ldots,2\cdot10^7}~.
\end{equation}

All the numbers having the digits sum in descending powers in numeration bases $b=2,3,\ldots,16$ on the search domain (\ref{Dc3SumaCifrelorLaPuteriDescrescatoare}--\ref{DcGenSumaCifrelorLaPuteriDescrescatoare}) are given in the table
below.
\begin{center}
  \begin{longtable}{|rrl|}
    \caption{Numbers with the property (\ref{ConditiaCifreLaPuteriDescrescatoare}) of (\ref{Dc3SumaCifrelorLaPuteriDescrescatoare}--\ref{DcGenSumaCifrelorLaPuteriDescrescatoare})}\\ \hline
    \endfirsthead
    \endhead
    \multicolumn{3}{r}{\textit{Continued on next page}} \\
    \endfoot
    \endlastfoot
    $5=$&$12_{(3)}=$&$1^3+2^2$\\
    $20=$&$202_{(3)}=$&$2^4+0^3+2^2$\\
    $24=$&$220_{(3)}=$&$2^4+2^3+0^2$\\
    $25=$&$221_{(3)}=$&$2^4+2^3+1^2$\\ \hline
    $8=$&$20_{(4)}=$&$2^3+0^2$\\
    $9=$&$21_{(4)}=$&$2^3+1^2$\\
    $28=$&$130_{(4)}=$&$1^4+3^3+0^2$\\
    $29=$&$131_{(4)}=$&$1^4+3^3+1^2$\\
    $819=$&$30303_{(4)}=$&$3^6+0^5+3^4+0^3+3^2$\\
    $827=$&$30323_{(4)}=$&$3^6+0^5+3^4+2^3+3^2$\\
    $983=$&$33113_{(4)}=$&$3^6+3^5+1^4+1^3+3^2$\\ \hline
    $12=$&$22_{(5)}=$&$2^3+2^2$\\
    $44=$&$134_{(5)}=$&$1^4+3^3+4^2$\\
    $65874=$&$4101444_{(5)}=$&$4^8+1^7+0^6+1^5+4^4+4^3+4^2$\\ \hline
    $10=$&$13_{(7)}=$&$1^3+3^2$\\
    $17=$&$23_{(7)}=$&$2^3+3^2$\\
    $81=$&$144_{(7)}=$&$1^4+4^3+4^2$\\
    $181=$&$346_{(7)}=$&$3^4+4^3+6^2$\\ \hline
    $256=$&$400_{(8)}=$&$4^4+0^3+0^2$\\
    $257=$&$401_{(8)}=$&$4^4+0^3+1^2$\\
    $1683844=$&$6330604_{(8)}=$&$6^8+3^7+3^6+0^5+6^4+0^3+4^2$\\
    $1683861=$&$6330625_{(8)}=$&$6^8+3^7+3^6+0^5+6^4+2^3+5^2$\\
    $1685962=$&$6334712_{(8)}=$&$6^8+3^7+3^6+4^5+7^4+1^3+2^2$\\ \hline
    $27=$&$30_{(9)}=$&$3^3+0^2$\\
    $28=$&$31_{(9)}=$&$3^3+1^2$\\
    $126=$&$150_{(9)}=$&$1^4+5^3+0^2$\\
    $127=$&$151_{(9)}=$&$1^4+5^3+1^2$\\
    $297=$&$360_{(9)}=$&$3^4+6^3+0^2$\\
    $298=$&$361_{(9)}=$&$3^4+6^3+1^2$\\
    $2805=$&$3756_{(9)}=$&$3^5+7^4+5^3+6^2$\\
    $3525=$&$4746_{(9)}=$&$4^5+7^4+4^3+6^2$\\
    $4118=$&$5575_{(9)}=$&$5^5+5^4+7^3+6^2$\\ \hline
    $24=$&$24_{(10)}=$&$2^3+4^2$\\
    $1676=$&$1676_{(10)}=$&$1^5+6^4+7^3+6^2$\\
    $4975929=$&$4975929_{(10)}=$&$4^8+9^7+7^6+5^5+9^4+2^3+9^2$\\ \hline
    $36=$&$33_{(11)}=$&$3^3+3^2$\\
    $8320=$&$6284_{(11)}=$&$6^5+2^4+8^3+4^2$\\ \hline
    $786=$&$556_{(12)}=$&$5^4+5^3+6^2$\\
    $8318=$&$4992_{(12)}=$&$4^5+9^4+9^3+2^2$\\
    $11508=$&$67b0_{(12)}=$&$6^5+7^4+11^3+0^2$\\
    $11509=$&$67b1_{(12)}=$&$6^5+7^4+11^3+1^2$\\ \hline
    $17=$&$14_{(13)}=$&$1^3+4^2$\\
    $43=$&$34_{(13)}=$&$3^3+4^2$\\
    $253=$&$166_{(13)}=$&$1^4+6^3+6^2$\\
    $784=$&$484_{(13)}=$&$4^4+8^3+4^2$\\ \hline
    $33=$&$25_{(14)}=$&$2^3+5^2$\\
    $1089=$&$57b_{(14)}=$&$5^4+7^3+11^2$\\
    $7386=$&$2998_{(14)}=$&$2^5+9^4+9^3+8^2$\\
    $186307=$&$4bc79_{(14)}=$&$4^6+11^5+12^4+7^3+9^2$\\ \hline
    $577=$&$287_{(15)}=$&$2^4+8^3+7^2$\\
    $810=$&$390_{(15)}=$&$3^4+9^3+0^2$\\
    $811=$&$391_{(15)}=$&$3^4+9^3+1^2$\\
    $1404=$&$639_{(15)}=$&$6^4+3^3+9^2$\\
    $16089=$&$4b79_{(15)}=$&$4^5+11^4+7^3+9^2$\\
    $22829=$&$6b6e_{(15)}=$&$6^5+11^4+6^3+14^2$\\ \hline
    $64=$&$40_{(16)}=$&$4^3+0^2$\\
    $65=$&$41_{(16)}=$&$4^3+1^2$\\
    $351=$&$15f_{(16)}=$&$1^4+5^3+15^2$\\
    $32768=$&$8000_{(16)}=$&$8^5+0^4+0^3+0^2$\\
    $32769=$&$8001_{(16)}=$&$8^5+0^4+0^3+1^2$\\
    $32832=$&$8040_{(16)}=$&$8^5+0^4+4^3+0^2$\\
    $32833=$&$8041_{(16)}=$&$8^5+0^4+4^3+1^2$\\
    $33119=$&$815f_{(16)}=$&$8^5+1^4+5^3+15^2$\\
    $631558=$&$9a306_{(16)}=$&$9^6+10^5+3^4+0^3+6^2$\\
    $631622=$&$9a346_{(16)}=$&$9^6+10^5+3^4+4^3+6^2$\\
    $631868=$&$9a43c_{(16)}=$&$9^6+10^5+4^4+3^3+12^2$\\ \hline
  \end{longtable}
\end{center}

\section{Multifactorial}

\begin{func}\label{FunctionMultiFactorial}
  By def{}inition, the multifactorial function, \citep{MathWorldMultifactorial}, is
  \begin{equation}\label{FunctiaKF}
    kf(n,k)=\prod_{j=1}^{\lfloor\frac{n}{k}\rfloor}\big(j\cdot k+\mod(n,k)\big)~.
  \end{equation}
\end{func}

For this function, in general bibliography, the commonly used notation are $n!$ for factorial, $n!!$ for double factorial, \ldots~.

The well known factorial function is $n!=1\cdot2\cdot3\cdots n$, \cite[A000142]{SloaneOEIS}. In general, for $0!$ the convention $0!=1$ is used.

The double factorial function, \cite[A006882]{SloaneOEIS}, can also be def{}ined in the following way:
\[
 n!!=\left\{
  \begin{array}{ll}
    1\cdot3\cdot5\cdots n, & \hbox{if $\mod(n,2)=1$;}\\
    2\cdot4\cdot6\cdots n, & \hbox{if $\mod(n,2)=0$.}
  \end{array}
\right.
\]
It is natural to consider the convention $0!!=1$. Let us note that $n!!$ is not the same as $(n!)!$.

The triple factorial function, \cite[A007661]{SloaneOEIS}, is def{}ined by:
\[
 n!!!=\left\{
  \begin{array}{ll}
    1\cdot4\cdot7\cdots n, & \hbox{if $\mod(n,3)=1$;} \\
    2\cdot5\cdot8\cdots n, & \hbox{if $\mod(n,3)=2$;} \\
    3\cdot6\cdot9\cdots n, & \hbox{if $\mod(n,3)=0$.}
  \end{array}
\right.
\]
We will use the same convention for the triple factorial function, $0!!!=1$.

\begin{prog}\label{Programkf} Program for calculating the multifactorial
  \begin{tabbing}
    $\emph{kf}(n,k):=$\=\ \vline\ $\emph{return}\ \ 1\ \ \emph{if}\ \ n\textbf{=}0$\\
    \>\ \vline\ $f\leftarrow1$\\
    \>\ \vline\ $r\leftarrow\mod(n,k)$\\
    \>\ \vline\ $i\leftarrow k\ \ \emph{if}\ \ r\textbf{=}0$\\
    \>\ \vline\ $i\leftarrow r\ \ \emph{otherwise}$\\
    \>\ \vline\ $f$\=$or\ j=i,i+k..n$\\
    \>\ \vline\ \>\ $f\leftarrow f\cdot j$\\
    \>\ \vline\ $\emph{return}\ \ f$
  \end{tabbing}
  This function allows us to calculate $n!$ by the call $\emph{kf}(n,1)$, $n!!$ by the call $\emph{kf}(n,2)$, etc.
\end{prog}

\subsection{Factorions}

The numbers that fulf{}ill the condition
\begin{equation}\label{ConditiaFactorion}
  \overline{d_1d_2\ldots d_m}_{(b)}=\sum_{k=1}^n d_k!
\end{equation}
are called \emph{factorion numbers}, \cite[p. 61 and 64]{Gardner1978}, \cite[167]{Madachy1979}, \cite[pp. 169--171 and 319--320]{Pickover1995}, \cite[A014080]{SloaneOEIS}.

Let $n_{(b)}$ a number in base $b$ with $m$ digits and the equation
\begin{equation}\label{EcFactorion}
  \log_b(b^{m-1})=\log_b\left(m(b-1)!\right)~.
\end{equation}
\begin{lem}\label{LemmaFactorin}
  The numbers $n_{(b)}$ with the property \emph{(\ref{ConditiaFactorion})} can\rq{t} have more then $n_d$ digits, where $n_d=\lceil s\rceil$, and $s$ is the solution of the equation \emph{(\ref{EcFactorion})}.
\end{lem}
\begin{proof}
  The smallest number in base $b$ with $m$ digits is $b^{m-1}$. The biggest number in base $b$ with the property (\ref{ConditiaFactorion}) is the number $m(b-1)!$. Therefore, the inequality $b^{m-1}\ge m(b-1)!$ provides the limit of numbers that can fulf{}ill the condition (\ref{ConditiaFactorion}).

  If we logarithmize both terms of inequation $b^{m-1}\ge m(b-1)!$ we get an inequation that establishes the limit of possible digits number for the numbers that fulf{}ill the condition (\ref{ConditiaFactorion}). It drives us to the solution of the equation (\ref{EcFactorion}). Let $s$ be the solution of the equation (\ref{EcFactorion}), but keeping into account that the digits number of a number is an integer, it follows that $n_d=\lceil s\rceil$.
\end{proof}

\begin{cor}
  The maximum digits number of the numbers in base $b$, which fulf{}ills the condition \emph{(\ref{ConditiaFactorion})} are given below.
\begin{table}[h]
  \centering
   \begin{tabular}{|c|c|c|c|c|c|c|c|c|c|c|c|c|c|c|c|}
      \hline
      b & $2$ & $3$ & $4$ & $5$ & $6$ & $7$ & $8$ & $9$ & $10$ & $11$ & $12$ & $13$ & $14$ & $15$ & $16$ \\ \hline
      $n_d$ & $2$ & $3$ & $4$ & $4$ & $5$ & $6$ & $6$ & $7$ & $8$ & $9$ & $9$ & $10$ & $11$ & $12$ & $12$ \\
      \hline
    \end{tabular}
  \caption{The maximum digits number of the numbers in base $b$}\label{NcFactorion}
\end{table}
\end{cor}

Let the search domains be def{}ined by
\begin{equation}\label{DcbBFactorini}
  Dc_b=\set{b,b+1,\ldots,n_d(b-1)!}~,
\end{equation}
where $n_d(b-1)!$ is not bigger than $2\cdot10^7$, and $n_d$ are values from the Table \ref{NcFactorion}.

Threfore, the search domains are:
\begin{eqnarray}
  Dc_2 &=& \set{2}~,\label{Dc2Factorioni}\\
  Dc_3 &=& \set{3,4,\ldots,6}~,\label{Dc3Factorioni}\\
  Dc_4 &=& \set{4,5,\ldots,24}~,\label{Dc4Factorioni}\\
  Dc_5 &=& \set{5,6,\ldots,96}~,\label{Dc5Factorioni}\\
  Dc_6 &=& \set{6,7,\ldots,600}~,\label{Dc6Factorioni}\\
  Dc_7 &=& \set{7,8,\ldots,4320}~,\label{Dc7Factorioni}\\
  Dc_8 &=& \set{8,9,\ldots,30240}~,\label{Dc8Factorioni}\\
  Dc_9 &=& \set{9,10,\ldots,282240}~,\label{Dc9Factorioni}\\
  Dc_{10} &=& \set{10,11,\ldots,2903040}~,\label{Dc10Factorioni}
\end{eqnarray}
and for $b=10,11,\ldots,16$,
\begin{equation}\label{DcGenFactorioni}
  Dc_b=\set{b,b+1,\ldots,2\cdot10^7}~.
\end{equation}

In base $10$ we have only 4 factorions $1=1!$, $2=2!$, $145=1!+4!+5!$ and $40585=4!+0!+5!+8!+5!$, with the observation that by convention we have $0!=1$. To note that we have the trivial solutions $1=1!$ and $2=2!$ in any base of
numeration $b\ge3$, and $3!\neq3$, \ldots, $(b-1)!\neq b-1$ for any $b\ge3$. The list of factorions in numeration base $b=2,3,\ldots,16$ on the search domain given by (\ref{Dc2Factorioni}--\ref{DcGenFactorioni}) is:
\begin{center}
  \begin{longtable}{|rrl|}
    \caption{Numbers with the property (\ref{ConditiaFactorion}) of (\ref{Dc2Factorioni}--\ref{DcGenFactorioni})}\\ \hline
    \endfirsthead
    \endhead
    \multicolumn{3}{r}{\textit{Continued on next page}} \\
    \endfoot
    \endlastfoot
    $2=$&$10_{(2)}=$&$1!+0!$\\ \hline
    ---& &\\ \hline
    $7=$&$13_{(4)}=$&$1!+3!$\\ \hline
    $49=$&$144_{(5)}=$&$1!+4!+4!$\\ \hline
    $25=$&$41_{(6)}=$&$4!+1!$\\
    $26=$&$42_{(6)}=$&$4!+2!$\\ \hline
    ---& &\\ \hline
    ---& &\\ \hline
    $41282=$&$62558_{(9)}=$&$6!+2!+5!+5!+8!$\\ \hline
    $145=$&$145_{(10)}=$&$1!+4!+5!$\\
    $40585=$&$40585_{(10)}=$&$4!+0!+5!+8!+5!$\\ \hline
    $26=$&$24_{(11)}=$&$2!+4!$\\
    $48=$&$44_{(11)}=$&$4!+4!$\\
    $40472=$&$28453_{(11)}=$&$2!+8!+4!+5!+3!$\\ \hline
    ---& &\\ \hline
    ---& &\\ \hline
    ---& &\\ \hline
    $1441=$&$661_{(15)}=$&$6!+6!+1!$\\
    $1442=$&$662_{(15)}=$&$6!+6!+2!$\\ \hline
    ---& &\\ \hline
  \end{longtable}
\end{center}

\subsection{Double Factorions}

We can also def{}ine the \emph{duble factorion} numbers, i.e. the numbers which fulf{}ills the condition
\begin{equation}\label{ConditiaDubluFactorion}
  \overline{d_1d_2\ldots d_m}_{(b)}=\sum_{k=1}^m d_k!!~.
\end{equation}

Let $n_{(b)}$ be a number in base $b$ with $m$ digits and the equation
\begin{equation}\label{EcDubluFactorion}
    \log_b(b^{m-1})=\log_b\left(m\cdot b!!\right)~.
\end{equation}

\begin{lem}\label{LemmaDubluFactorin}
  The numbers $n_{(b)}$ with the property \emph{(\ref{ConditiaDubluFactorion})} can\rq{t} have more than $n_d$ digits, where $n_d=\lceil s\rceil$, and $s$ is the solution of the equation \emph{(\ref{EcDubluFactorion})}.
\end{lem}
\begin{proof}
  The smallest number in base $b$ with $m$ digits is $b^{m-1}$~. The biggest number in base $b$ with the property (\ref{ConditiaDubluFactorion}) is the number $m\cdot b!!$~. Therefore, the inequality $b^{m-1}\ge m\cdot b!!$ provides a limit of numbers that can fulf{}ill the condition (\ref{ConditiaDubluFactorion}).

  If we logarithmize both terms of inequation $b^{m-1}\ge m\cdot b!!$ we get an inequation which establishes the limit of possible digits number for the numbers which fulf{}ills the condition (\ref{ConditiaDubluFactorion}).It drives us to the solution of the equation (\ref{EcDubluFactorion}). Let $s$ the solution of the equation (\ref{EcDubluFactorion}), but keeping into account that the digits number of a number is an integer, it follows that $n_d=\lceil s\rceil$.
\end{proof}

\begin{cor}
  The maximum digits of numbers in base $b$, which fulf{}ills the condition \emph{(\ref{ConditiaDubluFactorion})} are given in the table below.
\begin{table}[h]
  \centering
   \begin{tabular}{|c|c|c|c|c|c|c|c|c|c|c|c|c|c|c|c|}
      \hline
      b & $2$ & $3$ & $4$ & $5$ & $6$ & $7$ & $8$ & $9$ & $10$ & $11$ & $12$ & $13$ & $14$ & $15$ & $16$ \\ \hline
      $n_d$ & $1$ & $2$ & $2$ & $3$ & $3$ & $3$ & $4$ & $4$ & $4$ & $5$ & $5$ & $6$ & $6$ & $6$ & $7$ \\
      \hline
    \end{tabular}
  \caption{The maximum digits of numbers in base $b$}\label{NcDubluFactorion}
\end{table}
\end{cor}

Let the search domains be def{}ined by
\begin{equation}\label{DcbBDubluFactorini}
  Dc_b=\set{b,b+1,\ldots,n_d(b-1)!!}~,
\end{equation}
where $n_d$ are values from the Table (\ref{NcDubluFactorion}).

Therefore, the search domains are:
\begin{eqnarray}
  Dc_2 &=& \set{2},\label{Dc2DubluFactorioni}\\
  Dc_3 &=& \set{3,4},\label{Dc3DubluFactorioni}\\
  Dc_4 &=& \set{4,5,6},\label{D42DubluFactorioni}\\
  Dc_5 &=& \set{5,6,\ldots,24},\label{Dc5DubluFactorioni}\\
  Dc_6 &=& \set{6,7,\ldots,45},\label{Dc6DubluFactorioni}\\
  Dc_7 &=& \set{7,8,\ldots,144},\label{Dc7DubluFactorioni}\\
  Dc_8 &=& \set{8,9,\ldots,420},\label{Dc8DubluFactorioni}\\
  Dc_9 &=& \set{9,10,\ldots,1536},\label{Dc9DubluFactorioni}\\
  Dc_{10} &=& \set{10,11,\ldots,3780},\label{Dc10DubluFactorioni}\\
  Dc_{11} &=& \set{11,12,\ldots,19200},\label{Dc11DubluFactorioni}\\
  Dc_{12} &=& \set{12,13,\ldots,51975},\label{Dc12DubluFactorioni}\\
  Dc_{13} &=& \set{13,14,\ldots,276480},\label{Dc13DubluFactorioni}\\
  Dc_{14} &=& \set{14,15,\ldots,810810},\label{Dc14DubluFactorioni}\\
  Dc_{15} &=& \set{15,16,\ldots,3870720},\label{Dc15DubluFactorioni}\\
  Dc_{16} &=& \set{16,17,\ldots,14189175},\label{Dc16DubluFactorioni}
\end{eqnarray}
In the Table \ref{TabelDubluFactorini} we have all the double factorions for numeration bases $b=2,3,\ldots,16$.
\begin{center}
  \begin{longtable}{|rrl|}
    \caption{Numbers with the property (\ref{ConditiaDubluFactorion}) of (\ref{Dc2DubluFactorioni}--\ref{Dc16DubluFactorioni})}\label{TabelDubluFactorini}\\ \hline
    \endfirsthead
    \endhead
    \multicolumn{3}{r}{\textit{Continued on next page}} \\
    \endfoot
    \endlastfoot
    $2=$&$10_{(2)}=$&$1!!+0!!$\\ \hline
    --&&\\ \hline
    --&&\\ \hline
    $9=$&$14_{(5)}=$&$1!!+4!!$\\ \hline
    $17=$&$25_{(6)}=$&$2!!+5!!$\\ \hline
    $97=$&$166_{(7)}=$&$1!!+6!!+6!!$\\ \hline
    $49=$&$61_{(8)}=$&$6!!+1!!$\\
    $50=$&$62_{(8)}=$&$6!!+2!!$\\
    $51=$&$63_{(8)}=$&$6!!+3!!$\\ \hline
    $400=$&$484_{(9)}=$&$4!!+8!!+4!!$\\ \hline
    $107=$&$107_{(10)}=$&$1!!+0!!+7!!$\\ \hline
    $16=$&$15_{(11)}=$&$1!!+5!!$\\ \hline
    $1053=$&$739_{(12)}=$&$7!!+3!!+9!!$\\ \hline
    --&&\\ \hline
    $1891=$&$991_{(14)}=$&$9!!+9!!+1!!$\\
    $1892=$&$992_{(14)}=$&$9!!+9!!+2!!$\\
    $1893=$&$993_{(14)}=$&$9!!+9!!+3!!$\\
    $191666=$&$4dbc6_{(14)}=$&$4!!+13!!+11!!+12!!+6!!$\\ \hline
    $51=$&$36_{(15)}=$&$3!!+6!!$\\
    $96=$&$66_{(15)}=$&$6!!+6!!$\\
    $106=$&$71_{(15)}=$&$7!!+1!!$\\
    $107=$&$72_{(15)}=$&$7!!+2!!$\\
    $108=$&$73_{(15)}=$&$7!!+3!!$\\
    $181603=$&$38c1d_{(15)}=$&$3!!+8!!+12!!+1!!+13!!$\\ \hline
    $2083607=$&$1fcb17_{(16)}=$&$1!!+15!!+12!!+11!!+1!!+7!!$\\ \hline
  \end{longtable}
\end{center}

\subsection{Triple Factorials}

The \emph{triple factorion} numbers are the numbers fulf{}illing the condition
\begin{equation}\label{ConditiaTripleFactorion}
  \overline{d_1d_2\ldots d_m}_{(b)}=\sum_{k=1}^m d_k!!!~.
\end{equation}
Let $n_{(b)}$ a number in base $b$ with $m$ and the equation
\begin{equation}\label{EcTripleFactorion}
    \log_b(b^{m-1})=\log_b\left(m\cdot b!!!\right)~.
\end{equation}
\begin{lem}\label{LemmaTripleFactorin}
  The numbers $n_{(b)}$ with property \emph{(\ref{ConditiaTripleFactorion})} can\rq{t} have more $n_d$ digits, where $n_d=\lceil s\rceil$, and $s$ is the solution of the equation \emph{(\ref{EcTripleFactorion})}.
\end{lem}
\begin{proof}
  The smallest number in base $b$ with $m$ digits is $b^{m-1}$~. The biggest number in base $b$ with property (\ref{ConditiaTripleFactorion}) is the number $m\cdot b!!!$~. Therefore, the inequality $b^{m-1}\ge m\cdot b!!!$ provides a limit of numbers that can fulf{}ill the condition(\ref{ConditiaTripleFactorion}).

  If we logarithmize both terms of inequation $b^{m-1}\ge m\cdot b!!!$ we get an inequation that establishes the limit of possible digits number of the numbers that fulf{}ill the condition (\ref{ConditiaTripleFactorion}). It drives us to the solution of the equation (\ref{EcTripleFactorion}). Let $s$ be the solution of the equation (\ref{EcTripleFactorion}), but keeping into account that the digits number of a number is an integer, it follows that $n_d=\lceil s\rceil$.
\end{proof}

\begin{cor}
  The maximum numbers of numbers, in base $b$, which fulf{}ills the condition \emph{(\ref{ConditiaTripleFactorion})} are given below.
\begin{table}[h]
  \centering
   \begin{tabular}{|c|c|c|c|c|c|c|c|c|c|c|c|c|c|c|c|}
      \hline
      b & $2$ & $3$ & $4$ & $5$ & $6$ & $7$ & $8$ & $9$ & $10$ & $11$ & $12$ & $13$ & $14$ & $15$ & $16$ \\ \hline
      $n_d$ & $2$ & $3$ & $3$ & $3$ & $3$ & $4$ & $4$ & $4$ & $4$ & $4$ & $5$ & $5$ & $5$ & $6$ & $6$ \\
      \hline
    \end{tabular}
  \caption{The maximum numbers of numbers in base $b$}\label{NcTripleFactorion}
\end{table}
\end{cor}

Let the search domains be def{}ined by
\begin{equation}\label{DcbBTripleFactorini}
  Dc_b=\set{b,b+1,\ldots,n_d(b-1)!!!}~,
\end{equation}
where $n_d$ are the values in the Table \ref{NcTripleFactorion}.

Therefore, the search domains are:
\begin{eqnarray}
  Dc_2 &=& \set{2}, \label{Dc2TripluFactorioni}\\
  Dc_3 &=& \set{3,4,5,6}, \label{Dc3TripluFactorioni}\\
  Dc_4 &=& \set{4,5,\ldots,9}, \label{Dc4TripluFactorioni}\\
  Dc_5 &=& \set{5,6,\ldots,12}, \label{Dc5TripluFactorioni}\\
  Dc_6 &=& \set{6,7,\ldots,30}, \label{Dc6TripluFactorioni}\\
  Dc_7 &=& \set{7,8,\ldots,72}, \label{Dc7TripluFactorioni}\\
  Dc_8 &=& \set{8,9,\ldots,112}, \label{Dc8TripluFactorioni}\\
  Dc_9 &=& \set{9,10,\ldots,320}, \label{Dc9TripluFactorioni}\\
  Dc_{10} &=& \set{10,11,\ldots,648}, \label{Dc10TripluFactorioni}\\
  Dc_{11} &=& \set{11,12,\ldots,1120}, \label{Dc11TripluFactorioni}\\
  Dc_{12} &=& \set{12,13,\ldots,4400}, \label{Dc12TripluFactorioni}\\
  Dc_{13} &=& \set{13,14,\ldots,9720}, \label{Dc13TripluFactorioni}\\
  Dc_{14} &=& \set{14,15,\ldots,18200}, \label{Dc14TripluFactorioni}\\
  Dc_{15} &=& \set{15,16,\ldots,73920}, \label{Dc15TripluFactorioni}\\
  Dc_{16} &=& \set{16,17,\ldots,174960}. \label{Dc16TripluFactorioni}
\end{eqnarray}

In Table \ref{TabelTripleFactorini} we have all the triple factorions for numeration bases $b=2,3,\ldots,16$.
\begin{center}
  \begin{longtable}{|rrl|}
    \caption{Numbers with the property (\ref{ConditiaTripleFactorion}) of (\ref{DcbBTripleFactorini})}\label{TabelTripleFactorini}\\ \hline
    \endfirsthead
    \endhead
    \multicolumn{3}{r}{\textit{Continued on next page}} \\
    \endfoot
    \endlastfoot
    $2=$&$10_{(2)}=$&$1!!!+0!!!$\\ \hline
    --&&\\ \hline
    --&&\\ \hline
    --&&\\ \hline
    $11=$&$15_{(6)}=$&$1!!!+5!!!$\\ \hline
    $20=$&$26_{(7)}=$&$2!!!+6!!!$\\ \hline
    $31=$&$37_{(8)}=$&$3!!!+7!!!$\\ \hline
    $161=$&$188_{(9)}=$&$1!!!+8!!!+8!!!$\\ \hline
    $81=$&$81_{(10)}=$&$8!!!+1!!!$\\
    $82=$&$82_{(10)}=$&$8!!!+2!!!$\\
    $83=$&$83_{(10)}=$&$8!!!+3!!!$\\
    $84=$&$84_{(10)}=$&$8!!!+4!!!$\\ \hline
    $285=$&$23a_{(11)}=$&$2!!!+3!!!+10!!!$\\ \hline
    --&&\\ \hline
    $19=$&$16_{(13)}=$&$1!!!+6!!!$\\ \hline
    --&&\\ \hline
    $98=$&$68_{(15)}=$&$6!!!+8!!!$\\
    $1046=$&$49b_{(15)}=$&$4!!!+9!!!+11!!!$\\
    $3804=$&$11d9_{(15)}=$&$1!!!+1!!!+13!!!+9!!!$\\ \hline
    $282=$&$11a_{(16)}=$&$1!!!+1!!!+10!!!$\\
    $1990=$&$7c6_{(16)}=$&$7!!!+12!!!+6!!!$\\
    $15981=$&$3e6d_{(16)}=$&$3!!!+14!!!+6!!!+13!!!$\\ \hline
  \end{longtable}
\end{center}

Similarly, one can obtain \emph{quadruple factorions} and \emph{quintuple factorions}. In numeration base $10$ we only have factorions $49=4!!!!+9!!!!$ and $39=3!!!!!+9!!!!!$.

\subsection{Factorial Primes}

An important class of numbers that are prime numbers are the \emph{factorial primes}.
\begin{defn}\label{FactorialPrime}
  Numbers of the form $n!\pm1$ are called \emph{factorial primes}.
\end{defn}

In Table \ref{TabelFactorialprimes} we have all the \emph{factorial primes}, for $n\le30$, which are primes.
\begin{center}
  \begin{longtable}{|rcl|}
    \caption{Factorial primes that are primes}\label{TabelFactorialprimes}\\ \hline
    \endfirsthead
    \endhead
    \multicolumn{3}{r}{\textit{Continued on next page}} \\
    \endfoot
    \endlastfoot
    $1!+1$&=&$2$\\
    $2!+1$&=&$3$\\
    $3!-1$&=&$5$\\
    $3!+1$&=&$7$\\
    $4!-1$&=&$23$\\
    $6!-1$&=&$719$\\
    $7!-1$&=&$5039$\\
    $11!+1$&=&$39916801$\\
    $12!-1$&=&$479001599$\\
    $14!-1$&=&$8717821199$\\
    $27!+1$&=&$10888869450418352160768000001$\\
    $30!-1$&=&$265252859812191058636308479999999$\\
    \hline
  \end{longtable}
\end{center}
Similarly, we can def{}ine \emph{double factorial primes}.
\begin{defn}\label{DoubleFactorialPrimes}
  The numbers of the form $n!!\pm1$ are called \emph{double factorial primes}.
\end{defn}

In Table \ref{TabelDoublefactorialprimes} we have all the numbers that are \emph{double factorial primes}, for $n\le30$, that are primes.
\begin{center}
  \begin{longtable}{|rcl|}
    \caption{Double factorial primes that are primes}\label{TabelDoublefactorialprimes}\\ \hline
    \endfirsthead
    \endhead
    \multicolumn{3}{r}{\textit{Continued on next page}} \\
    \endfoot
    \endlastfoot
    $3!!-1$&=&$2$\\
    $2!!+1$&=&$3$\\
    $4!!-1$&=&$7$\\
    $6!!-1$&=&$47$\\
    $8!!-1$&=&$383$\\
    $16!!-1$&=&$10321919$\\
    $26!!-1$&=&$51011754393599$\\
    \hline
  \end{longtable}
\end{center}

\begin{defn}\label{TripleFactorialPrimes}
  The numbers of the form $n!!!\pm1$ are called \emph{triple factorial primes}.
\end{defn}

In Table \ref{TabelTriplefactorialprimes} we have all the numbers that are \emph{triple factorial primes}, foru $n\le30$, that are primes.
\begin{center}
  \begin{longtable}{|rcl|}
    \caption{Triple factorial primes that are primes}\label{TabelTriplefactorialprimes}\\ \hline
    \endfirsthead
    \endhead
    \multicolumn{3}{r}{\textit{Continued on next page}} \\
    \endfoot
    \endlastfoot
    $3!!!-1$&=&$2$\\
    $4!!!-1$&=&$3$\\
    $4!!!+1$&=&$5$\\
    $5!!!+1$&=&$11$\\
    $6!!!-1$&=&$17$\\
    $6!!!+1$&=&$19$\\
    $7!!!+1$&=&$29$\\
    $8!!!-1$&=&$79$\\
    $9!!!+1$&=&$163$\\
    $10!!!+1$&=&$281$\\
    $11!!!+1$&=&$881$\\
    $17!!!+1$&=&$209441$\\
    $20!!!-1$&=&$4188799$\\
    $24!!!+1$&=&$264539521$\\
    $26!!!-1$&=&$2504902399$\\
    $29!!!+1$&=&$72642169601$\\
    \hline
  \end{longtable}
\end{center}

\begin{defn}\label{QuadrupleFactorialPrimes}
  The numbers of the form $n!!!!\pm1$ are called \emph{quadruple factorial primes}.
\end{defn}

In Table \ref{TabelQuadruplefactorialprimes} we have all the numbers that are \emph{quadruple factorial primes},for $n\le30$, that are primes.
\begin{center}
  \begin{longtable}{|rcl|}
    \caption{Quadruple factorial primes that are primes}\label{TabelQuadruplefactorialprimes}\\ \hline
    \endfirsthead
    \endhead
    \multicolumn{3}{r}{\textit{Continued on next page}} \\
    \endfoot
    \endlastfoot
    $4!!!!-1$&=&$3$\\
    $4!!!!+1$&=&$5$\\
    $6!!!!-1$&=&$11$\\
    $6!!!!+1$&=&$13$\\
    $8!!!!-1$&=&$31$\\
    $9!!!!+1$&=&$163$\\
    $12!!!!-1$&=&$383$\\
    $16!!!!-1$&=&$6143$\\
    $18!!!!+1$&=&$30241$\\
    $22!!!!-1$&=&$665279$\\
    $24!!!!-1$&=&$2949119$\\
    \hline
  \end{longtable}
\end{center}

\begin{defn}\label{QuintupleFactorialPrimes}
  The numbers of the form $n!!!!!\pm1$ are called \emph{quintuple factorial primes}.
\end{defn}

In Table \ref{TabelQuintuplefactorialprimes} we have all the numbers that are \emph{quintuple factorial primes}, for $n\le30$, that are primes.
\begin{center}
  \begin{longtable}{|rcl|}
    \caption{Quintuple factorial primes that are primes}\label{TabelQuintuplefactorialprimes}\\ \hline
    \endfirsthead
    \endhead
    \multicolumn{3}{r}{\textit{Continued on next page}} \\
    \endfoot
    \endlastfoot
    $6!!!!!-1$&=&$5$\\
    $6!!!!!+1$&=&$7$\\
    $7!!!!!-1$&=&$13$\\
    $8!!!!!-1$&=&$23$\\
    $9!!!!!+1$&=&$37$\\
    $11!!!!!+1$&=&$67$\\
    $12!!!!!-1$&=&$167$\\
    $13!!!!!-1$&=&$311$\\
    $13!!!!!+1$&=&$313$\\
    $14!!!!!-1$&=&$503$\\
    $15!!!!!+1$&=&$751$\\
    $17!!!!!-1$&=&$8857$\\
    $23!!!!!+1$&=&$129169$\\
    $26!!!!!+1$&=&$576577$\\
    $27!!!!!-1$&=&$1696463$\\
    $28!!!!!-1$&=&$3616703$\\
    \hline
  \end{longtable}
\end{center}

\begin{defn}\label{SixFactorialPrimes}
  The numbers of the form $n!!!!!!\pm1$ are called \emph{sextuple factorial primes}.
\end{defn}

In Table \ref{TabelSixfactorialprimes} we have all the numbers that are \emph{sextuple factorial primes}, for $n\le30$, that are primes.
\begin{center}
  \begin{longtable}{|rcl|}
    \caption{Sextuple factorial primes that are primes}\label{TabelSixfactorialprimes}\\ \hline
    \endfirsthead
    \endhead
    \multicolumn{3}{r}{\textit{Continued on next page}} \\
    \endfoot
    \endlastfoot
    $6!!!!!!-1$&=&$5$\\
    $6!!!!!!+1$&=&$7$\\
    $8!!!!!!+1$&=&$17$\\
    $10!!!!!!+1$&=&$41$\\
    $12!!!!!!-1$&=&$71$\\
    $12!!!!!!+1$&=&$73$\\
    $14!!!!!!-1$&=&$223$\\
    $16!!!!!!+1$&=&$641$\\
    $18!!!!!!+1$&=&$1297$\\
    $20!!!!!!+1$&=&$4481$\\
    $22!!!!!!+1$&=&$14081$\\
    $28!!!!!!+1$&=&$394241$\\
    \hline
  \end{longtable}
\end{center}

\section{Digital Product}

We consider the function $\emph{dp}$ product of the digits\rq{} number $n_{(b)}$~.
\begin{prog}\label{ProgramDp}
  The function $\emph{dp}$ is given of program
  \begin{tabbing}
    $\emph{dp}(n,b):=$\=\ \vline\ $v\leftarrow \emph{dn}(n,b)$\\
    \>\ \vline\ $p\leftarrow1$\\
    \>\ \vline\ $f$\=$or\ j\in \emph{ORIGIN}..\emph{last}(v)$\\
    \>\ \vline\ \>\ $p\leftarrow p\cdot v_j$\\
    \>\ \vline\ $\emph{return}\ \ p$
  \end{tabbing}
  The program \ref{ProgramDp} calls the program \ref{ProgramDn}.

  Examples:
  \begin{enumerate}
    \item The call $\emph{dp}(76,8)=4$ verif{}ies with the identity $76_{(10)}=114_{(8)}$ and by the fact than $1\cdot1\cdot4=4$;
    \item The call $\emph{dp}(1234,16)=104$ verif{}ies with the identity $1234_{(10)}=4d2_{(16)}$ and by the fact than $4\cdot d\cdot2=4\cdot13\cdot2=104$;
    \item The call $\emph{dp}(15,2)=1$ verif{}ies with the identity $15_{(10)}=1111_{(2)}$ and by the fact than $1\cdot1\cdot1\cdot1=1$.
  \end{enumerate}
\end{prog}

\begin{figure}[h]
  \centering
  \includegraphics[scale=1]{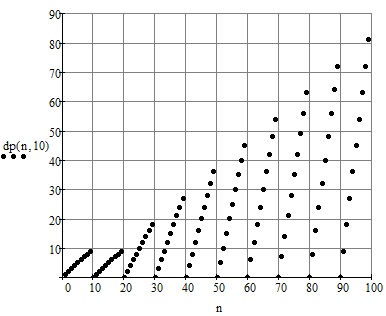}\\
  \caption{The digital product function}\label{FigDigitalProduct}
\end{figure}

We suggest to resolve the Diophantine equations
\begin{equation}\label{EcDiophanitineDsk+Dp=n}
  \alpha\cdot\left(d_1^k+d_2^k+\ldots+d_m^k\right)+\beta\cdot(d_1\cdot d_2\cdots d_m)=\overline{d_1d_2\ldots d_m}_{(b)}~,
\end{equation}
where $d_1,d_2,\ldots d_m\in\set{0,1,\ldots,b-1}$, iar $\alpha,\beta\in\Na$, $b\in\Ns$, $b\ge2$.

\begin{prog}\label{ProgDsk+Dp=n} Program for determining the natural numbers which verif{}ies the equation (\ref{EcDiophanitineDsk+Dp=n}):
  \begin{tabbing}
    $\emph{ED}(L,b,k,\alpha,\beta):=$\=\ \vline\
    $S\leftarrow("n_{(10)}"\ \ \ "n_{(b)}"\ \ \ "s"\ \ \ "p"\ \ \ "\alpha\cdot s"\ \ \ "\beta\cdot p")$\\
    \>\ \vline\ $f$\=$or\ n\in b..L$\\
    \>\ \vline\ \>\ \vline\ $s\leftarrow\emph{dsk}(n,b,k)$\\
    \>\ \vline\ \>\ \vline\ $p\leftarrow\emph{dp}(n,b)$\\
    \>\ \vline\ \>\ \vline\ $i$\=$f\ \alpha\cdot s+\beta\cdot p\textbf{=}n$\\
    \>\ \vline\ \>\ \vline\ \>\ \vline\ $v\leftarrow(n\ \ \ \emph{dn}(n,b)^\textrm{T}\ \ \ s\ \ \ p\ \ \ \alpha\cdot s\ \ \  \beta\cdot p)$\\
    \>\ \vline\ \>\ \vline\ \>\ \vline\ $S\leftarrow stack[S,v]$\\
    \>\ \vline\ $\emph{return}\ \ S$
  \end{tabbing}
\end{prog}

\begin{exem}\
  \begin{enumerate}
    \item The case $b=10$, $\alpha=2$, $\beta=1$ and $k=1$ with $n\le10^3$ has the solutions:
        \begin{enumerate}
          \item $14=2\cdot(1+4)+1\cdot(1\cdot4)$~;
          \item $36=2\cdot(3+6)+1\cdot(3\cdot6)$~;
          \item $77=2\cdot(7+7)+1\cdot(7\cdot7)$~.
        \end{enumerate}
    \item The case $b=10$, $\alpha=2$, $\beta=1$ and $k=3$ with $n\le10^3$ has the solutions:
        \begin{enumerate}
          \item $624=2\cdot(6^3+2^3+4^3)+1\cdot(6\cdot2\cdot4)$~;
          \item $702=2\cdot(7^3+0^3+2^3)+1\cdot(7\cdot0\cdot2)$~.
    \end{enumerate}
    \item  The case $b=11$, $\alpha=2$, $\beta=1$ and $k=3$ with $n\le10^3$ has the solutions:
    \begin{enumerate}
      \item $136_{(10)}=114_{(11)}=2\cdot(1^3+1^3+4^3)+1\cdot(1\cdot1\cdot4)$~.
    \end{enumerate}
    \item The case $b=15$, $\alpha=2$, $\beta=1$ and $k=3$ with $n\le10^3$ has the solutions:
    \begin{enumerate}
      \item $952_{(10)}=437_{(15)}=2\cdot(4^3+3^3+7^3)+1\cdot(4\cdot3\cdot)$~.
    \end{enumerate}
    \item The case $b=10$, $\alpha=1$, $\beta=0$ and $k=3$ with $n\le10^3$ has the solutions:
    \begin{enumerate}
      \item $153_{(10)}=1^3+5^3+3^3$~;
      \item $370_{(10)}=3^3+7^3+0^3$~;
      \item $371_{(10)}=3^3+7^3+1^3$~;
      \item $407_{(10)}=4^3+0^3+7^3$~.
    \end{enumerate}
    These are the Narcissistic numbers in base $b=10$ of 3 digits, see Table \ref{TabelNumereNarcisissticBaza10}.
  \end{enumerate}
\end{exem}

\section{Sum--Product}

\begin{defn}{\citep{MathWorldSum-ProductNumber},\citep[A038369]{SloaneOEIS}}\label{Definitia Sum-Prod}
  The natural numbers $n$, in the base $b$, $n_{(b)}=\overline{d_1d_2\ldots d_m}$, where $d_k\in\set{0,1,\ldots,b-1}$ which verif{}ies the equality
  \begin{equation}\label{ConditionSum-ProductNumbers}
    n=\prod_{k=1}^md_k\sum_{k=1}^md_k~,
  \end{equation}
  are called \emph{sum--product numbers}.
\end{defn}

Let the function $sp$, def{}ined on $\Ns\times\Na_{\ge2}$ with values on $\Ns$:
\begin{equation}\label{FunctiaSum-Product}
  sp(n,b):=dp(n,b)\cdot dks(n,b,1)~,
\end{equation}
where the digital product function, $\emph{dp}$, is given by \ref{ProgramDp} and the digital sum--product function of power $k$, $dks$, is def{}ined in relation (\ref{Functia dks}). The graphic of the function is shown in \ref{FigFunctiasp}.

\begin{figure}[h]
  \centering
  \includegraphics[scale=1]{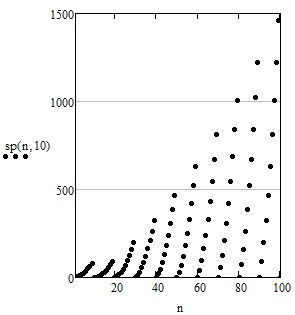}\\
  \caption{Function $sp$}\label{FigFunctiasp}
\end{figure}

\begin{prog}\label{ProgramPsp} for determining the \emph{sum--product} numbers in base $b$.
  \begin{tabbing}
    $\emph{Psp}(L,b,\varepsilon):=$\=\vline\ $j\leftarrow1$\\
    \>\vline\ $f$\=$or\ n\in1..L$\\
    \>\vline\ \>\ $i$\=$f\ \abs{\emph{sp}(n,b)-n}\le\varepsilon$\\
    \>\vline\ \>\ \>\vline\ $q_{j,1}\leftarrow n$\\
    \>\vline\ \>\ \>\vline\ $q_{j,2}\leftarrow \emph{dn}(n,b)^\textrm{T}$\\
    \>\vline\ $\emph{return}\ \ q$
  \end{tabbing}
\end{prog}
In the table below we show all \emph{sum--product} numbers in bases $b=2,3,\ldots,16$, up to the limit $L=10^6$. It is obvious that for any bases $b\ge2$ the number $n_{(b)}=1$ is a \emph{sum--product} number, therefore we do not show this trivial solution. The call of program \ref{ProgramSp} is made by command $\emph{Psp}(10^6,7,0)=$.

\begin{center}
  \begin{longtable}{|l|}
    \caption{Sum--product numbers}\label{TabelSum-ProductNumbers}\\ \hline
    \endfirsthead
    \endhead
    \multicolumn{1}{r}{\textit{Continued on next page}} \\
    \endfoot
    \endlastfoot
    $6=12_{(4)}=(1+2)\cdot1\cdot2$\\ \hline
    $96=341_{(5)}=(3+4+1)\cdot3\cdot4\cdot1$\\ \hline
    $16=22_{(7)}=(2+2)\cdot2\cdot2$\\
    $128=242_{(7)}=(2+4+2)\cdot2\cdot4\cdot2$\\
    $480=1254_{(7)}=(1+2+5+4)\cdot1\cdot2\cdot5\cdot4$\\
    $864=2343_{(7)}=(2+3+4+3)\cdot2\cdot3\cdot4\cdot3$\\
    $21600=116655_{(7)}=(1+1+6+6+5+5)\cdot1\cdot1\cdot6\cdot6\cdot5\cdot5$\\
    $62208=346236_{(7)}=(3+4+6+2+3+6)\cdot3\cdot4\cdot6\cdot2\cdot3\cdot6$\\
    $73728=424644_{(7)}=(4+2+4+6+4+4)\cdot4\cdot2\cdot4\cdot6\cdot4\cdot4$\\ \hline
    $12=13_{(9)}=(1+3)\cdot1\cdot9$\\
    $172032=281876_{(9)}=(2+8+1+8+7+6)\cdot2\cdot8\cdot1\cdot8\cdot7\cdot6$\\
    $430080=724856_{(9)}=(7+2+4+8+5+6)\cdot7\cdot2\cdot4\cdot8\cdot5\cdot6$\\ \hline
    $135=135_{(10)}=(1+3+5)\cdot1\cdot3\cdot5$\\
    $144=144_{(10)}=(1+4+4)\cdot1\cdot4\cdot4$\\ \hline
    $300=253_{(11)}=(2+5+3)\cdot2\cdot5\cdot3$\\
    $504=419_{(11)}=(4+1+9)\cdot4\cdot1\cdot9$\\
    $2880=2189_{(11)}=(2+1+8+9)\cdot2\cdot1\cdot8\cdot9$\\
    $10080=7634_{(11)}=(7+6+3+4)\cdot7\cdot6\cdot3\cdot4$\\
    $120960=82974_{(11)}=(8+2+9+7+4)\cdot8\cdot2\cdot9\cdot7\cdot4$\\ \hline
    $176=128_{(12)}=(1+2+8)\cdot1\cdot2\cdot8$\\
    $231=173_{(12)}=(1+7+3)\cdot1\cdot7\cdot3$\\
    $495=353_{(12)}=(3+5+3)\cdot3\cdot5\cdot3$\\ \hline
    $720=435_{(13)}=(4+3+5)\cdot4\cdot3\cdot5$\\
    $23040=a644_{(13)}=(10+6+4+4)\cdot10\cdot6\cdot4\cdot4$\\
    $933120=268956_{(13)}=(2+6+8+9+5+6)\cdot2\cdot6\cdot8\cdot9\cdot5\cdot6$\\ \hline
    $624=328_{(14)}=(3+2+8)\cdot3\cdot2\cdot8$\\
    $1040=544_{(14)}=(5+4+4)\cdot5\cdot4\cdot4$\\
    $22272=818c_{(14)}=(8+1+8+12)\cdot8\cdot1\cdot8\cdot12$\\ \hline
    $8000=2585_{(15)}=(2+5+8+5)\cdot2\cdot5\cdot8\cdot5$\\ \hline
    $20=14_{(16)}=(1+4)\cdot1\cdot4$\\ \hline
  \end{longtable}
\end{center}

The program \ref{ProgramPsp} $\emph{Psp}$ has 3 input parameters. If the parameter $\varepsilon$ is $0$ then we will obtain \emph{sum--product} numbers, if $\varepsilon=1$ then we will obtain \emph{sum--product} numbers and \emph{almost sum--product} numbers. For example, in base $b=7$ we have the \emph{almost sum--product} numbers:
\[
 43=61_{(7)}=(6+1)\cdot6=42~,
\]
\[
 3671=13463_{(7)}=(1+3+4+6+3)\cdot3\cdot4\cdot6\cdot3=3672~,
\]
\[
 5473=21646_{(7)}=(2+1+6+4+6)\cdot2\cdot1\cdot6\cdot4\cdot6=5472~,
\]
\[
 10945=43624_{(7)}=(4+3+6+2+4)\cdot4\cdot3\cdot6\cdot2\cdot4=10944~.
\]

It is clear that the \emph{sum--product} numbers can not be prime numbers. In base 10, up to the limit $L=10^6$, there is only one \emph{almost sum--product} number which is a prime, that is 13. Maybe there are other \emph{almost sum--product} numbers that are primes?

The number 144 has the quality of being a \emph{sum--product} number and a perfect square. This number is also called "\emph{gross number}" (rough number). There is also another \emph{sum--product} number that is perfect square? At least between numbers displayed in Table \ref{TabelSum-ProductNumbers} there is no perfect square excepting 144.

\section{Code Puzzle}

Using the following letter-to-number code:

\begin{center}
  \begin{tabular}{|c|c|c|c|c|c|c|c|c|c|c|c|c|}
    \hline
    A & B & C & D & E & F & G & H & I & J & K & L & M \\ \hline
    01 & 05 & 06 & 07 & 08 & 09 & 10 & 11 & 12 & 13 & 14 & 15 & 16 \\
    \hline
  \end{tabular}
\end{center}
\begin{center}
  \begin{tabular}{|c|c|c|c|c|c|c|c|c|c|c|c|c|}
     \hline
     N & O & P & Q & R & S & T & U & V & W & X & Y & Z \\ \hline
     14 & 15 & 16 & 17 & 18 & 19 & 20 & 21 & 22 & 23 & 24 & 25 & 26 \\
     \hline
   \end{tabular}
\end{center}
then $c_p(n)=$ the numerical code for the spelling of n in English language; for example: $c_p(ONE)=151405$, $c_p(TWO)=202315$, etc.

\section{Pierced Chain}

Let the function
\[
 c(n)=101\cdot1\underbrace{0001}_1\underbrace{0001}_2\ldots\underbrace{0001}_{n-1}~,
\]
then $c(1)=101$, $c(10001)=1010101$, $c(100010001)=10101010101$, \ldots~.

How many $c(n)/101$ are primes? \citep{SmarandacheArizona,Smarandache1979,Smarandache1993,Smarandache2006}.

\section{Divisor Product}

Let $P_d(n)$ is the product of all positive divisors of $n$.
\begin{eqnarray*}
  P_d(1) &=& 1=1~, \\
  P_d(2) &=& 1\cdot2=2~, \\
  P_d(3) &=& 1\cdot3=3~, \\
  P_d(4) &=& 1\cdot2\cdot4=8~, \\
  P_d(5) &=& 1\cdot5=5~, \\
  P_d(6) &=& 1\cdot2\cdot3\cdot6=36~, \\
  \vdots & &
\end{eqnarray*}
thus, the sequence: 1, 2, 3, 8, 5, 36, 7, 64, 27, 100, 11, 1728, 13, 196, 225, 1024, 17, 5832, 19, 8000, 441, 484, 23, 331776, 125, 676, 729, 21952, 29, 810000, 31, 32768, 1089, 1156, 1225, 100776, 96, 37, 1444, 1521, 2560000, 41, \ldots~.

\section{Proper Divisor Products}

Let $P_{dp}(n)$ is the product of all positive proper divisors of $n$.
\begin{eqnarray*}
  P_{dp}(1) &=& 1~, \\
  P_{dp}(2) &=& 1~, \\
  P_{dp}(3) &=& 1~, \\
  P_{dp}(4) &=& 2~, \\
  P_{dp}(5) &=& 1~, \\
  P_{dp}(6) &=& 2\cdot3=6~, \\
  \vdots & &
\end{eqnarray*}
thus, the sequence: 1, 1, 1, 2, 1, 6, 1, 8, 3, 10, 1, 144, 1, 14, 15, 64, 1, 324, 1, 400, 21, 22, 1, 13824, 5, 26, 27, 784, 1, 27000, 1, 1024, 33, 34, 35, 279936, 1, 38, 39, 64000, 1, \ldots~.

\section{n -- Multiple Power Free Sieve}

\begin{defn} The sequence of positive integer numbers $\set{2,3,\ldots,L}$ from which take of{}f numbers $k\cdot p^n$, where $p\in\NP{2}$, $n\in\Ns$, $n\ge3$ and $k\in\Ns$ such that $k\cdot p^n\le L$ (take of{}f all multiples of all $n$ -- power primes) is called \emph{$n$ -- power free sieve}.
\end{defn}

The list of numbers without primes to multiple cubes up to $L=125$ is: 2, 3, 4, 5, 6, 7, 9, 10, 11, 12, 13, 14, 15, 17, 18, 19, 20, 21, 22, 23, 25, 26, 28, 29, 30, 31, 33, 34, 35, 36, 37, 38, 39, 41, 42, 43, 44, 45, 46, 47, 49, 50, 51, 52, 53, 55, 57, 58, 59, 60, 61, 62, 63, 65, 66, 67, 68, 69, 70, 71, 73, 74, 75, 76, 77, 78, 79, 82, 83, 84, 85, 86, 87, 89, 90, 91, 92, 93, 94, 95, 97, 98, 99, 100, 101, 102, 103, 105, 106, 107, 109, 110, 111, 113, 114, 115, 116, 117, 118, 119, 121, 122, 123, 124~. We eliminated the numbers: 8, 16, 24, 32, 40, 48, 56, 64, 72, 80, 88, 96, 104, 112, 120 (multiples of $2^3$), 27, 54, 81, 108 (multiples of $3^3$), 125 (multiples of $5^3$)~.

The list of numbers without multiples of order 4 powers of primes to $L=125$ is: 2, 3, 4, 5, 6, 7, 8, 9, 10, 11, 12, 13, 14, 15, 17, 18, 19, 20, 21, 22, 23, 24, 25, 26, 27, 28, 29, 30, 31, 33, 34, 35, 36, 37, 38, 39, 40, 41, 42, 43, 44, 45, 46, 47, 49, 50, 51, 52, 53, 54, 55, 56, 57, 58, 59, 60, 61, 62, 63, 65, 66, 67, 68, 69, 70, 71, 72, 73, 74, 75, 76, 77, 78, 79, 82, 83, 84, 85, 86, 87, 88, 89, 90, 91, 92, 93, 94, 95, 97, 98, 99, 100, 101, 102, 103, 104, 105, 106, 107, 108, 109, 110, 111, 113, 114, 115, 116, 117, 118, 119, 120, 121, 122, 123, 124, 125~.
We eliminated the numbers: 16, 32, 48, 64, 80, 96, 112 (multiples of $2^4$) and 81 (multiples of $3^4$)~.

\section{Irrational Root Sieve}

\begin{defn}{\citep{SmarandacheArizona,Smarandache1993,Smarandache2006}} The sequence of positive integer numbers $\set{2,3,\ldots,L}$ from which we take of{}f numbers $j\cdot k^2$, where $k=2,3,\ldots,\lfloor L\rfloor$ and $j=1,2,\ldots,\lfloor L/k^2\rfloor$ is the \emph{free sequence of multiples perfect squares}.
\end{defn}

The list of numbers free of perfect squares multiples for $L=71$ is: 2, 3, 5, 6, 7, 10, 11, 13, 14, 15, 17, 19, 21, 22, 23, 26, 29, 30, 31, 33, 34, 35, 37, 38, 39, 41, 42, 43, 46, 47, 51, 53, 55, 57, 58, 59, 61, 62, 65, 66, 67, 69, 70, 71~.

The number of numbers free of perfect squares multiples to the limit $L$ is given in the Table \ref{TabelLungimeaSirurilorLibereDeMultipliPatratelorPerfecte}
\begin{center}
  \begin{longtable}{|r|r|}
    \caption{{\small The length of the free of perfect squares multiples}} \label{TabelLungimeaSirurilorLibereDeMultipliPatratelorPerfecte}\\ \hline
    \endfirsthead
    \endhead
    \multicolumn{2}{r}{\textit{Continued on next page}} \\
    \endfoot
    \endlastfoot
    $L$ & $\emph{length}$ \\ \hline
    $10$ & $6$ \\
    $100$ & $60$ \\
    $1000$ & $607$ \\
    $10000$ & $6082$ \\
    $100000$ & $60793$ \\
    $1000000$ & $607925$ \\
    $2000000$ & $1215876$ \\
    $3000000$ & $1823772$ \\
    $4000000$ & $2431735$ \\
    $5000000$ & $3039632$ \\
    $6000000$ & $3647556$ \\
    $7000000$ & $4255503$ \\
    $8000000$ & $4863402$ \\
    $9000000$ & $5471341$ \\
    $10000000$ & $6079290$ \\
    $20000000$ & $12518574$ \\
    $30000000$ & $18237828$ \\
    \vdots & \vdots \\
    \hline
  \end{longtable}
\end{center}

\section{Odd Sieve}

\begin{defn}\label{DefnOddSieve}
  All odd numbers that are not equal to the fractional of two primes.
\end{defn}

\begin{obs}
  The dif{}ference between an odd number and an even number is an odd number; indeed $(2k_1+1)-2k_2=2(k_1-k_2)+1$. The number 2 is the only even prime number, all the other primes are odd. Then the dif{}ference between a prime number and 2 is always an odd number.
\end{obs}

The series generation algorithm, \citep{Le+Smarandache1999}, with the property from Def{}inition \ref{DefnOddSieve} is:
\begin{enumerate}
  \item Take all prime numbers up to the limit $L$.
  \item From every prime number subtract $2$. This series becomes a temporary series.
  \item Eliminate all numbers that are on the temporary list from the odd numbers list.
\end{enumerate}

The list of the odd numbers that are not the dif{}ference of two prime numbers up to the limit $L=150$ is: 7, 13, 19, 23, 25, 31, 33, 37, 43, 47, 49, 53, 55, 61, 63, 67, 73, 75, 79, 83, 85, 89, 91, 93, 97, 103, 109, 113, 115, 117, 119, 121, 123, 127, 131, 133, 139, 141, 143, 145~.

The length of the odd-numbered series that are not the dif{}ference of two prime numbers up to the limit 10, $10^2$, $10^3$, $10^4$, $10^5$, $10^6$ and $10^7$ is, respectively:1, 25, 333, 3772, 40409, 421503, 4335422~.

\section{$n$ -- ary Power Sieve}

The list of the odd numbers that are not the dif{}ference of two prime numbers up to the limit $L$, we delete all the $n$-th term, from the remaining series, we delete all the $n^2$-th term, and so on until possible.

\begin{prog}\label{ProgramnPS} Program for generating the series up to the limit $L$.
  \begin{tabbing}
    $\emph{nPS}(L,n):=$\ \=\ \vline\ $f$\=$or\ j\in1..L$\\
    \>\ \vline\ \>\ $S_j\leftarrow j$\\
    \>\ \vline\ $f$\=$or\ k\in1..\emph{floor}(\log(L,n))$\\
    \>\ \vline\ \>\ \vline\ $break\ \ \emph{if}\ \ n^k>\emph{last}(S)$\\
    \>\ \vline\ \>\ \vline\ $f$\=$or\ j\in1..\emph{floor}\left(\frac{last(S)}{n^k}\right)$\\
    \>\ \vline\ \>\ \vline\ \>\ $S_{j\cdot n^k}\leftarrow0$\\
    \>\ \vline\ \>\ \vline\ $i\leftarrow1$\\
    \>\ \vline\ \>\ \vline\ $f$\=$or\ j\in1..\emph{last}(S)$\\
    \>\ \vline\ \>\ \vline\ \>\ $i$\=$f\ S_j\neq0$\\
    \>\ \vline\ \>\ \vline\ \>\ \>\ \vline\ $Q_i\leftarrow S_j$\\
    \>\ \vline\ \>\ \vline\ \>\ \>\ \vline\ $i\leftarrow i+1$\\
    \>\ \vline\ \>\ \vline\ $S\leftarrow Q$\\
    \>\ \vline\ \>\ \vline\ $Q\leftarrow0$\\
    \>\ \vline\ $\emph{return}\ \ S$
  \end{tabbing}
\end{prog}

The series for $L=135$ and $n=2$ is: 1, $\fbox{3}$, $\fbox{5}$, 9, $\fbox{11}$, $\fbox{13}$, $\fbox{17}$, 21, 25, 27, $\fbox{29}$, 33, 35, $\fbox{37}$, $\fbox{43}$, 49, 51, $\fbox{53}$, 57, $\fbox{59}$, 65, $\fbox{67}$, 69, $\fbox{73}$, 75, 77, 81, 85, $\fbox{89}$, 91, $\fbox{97}$, $\fbox{101}$, $\fbox{107}$, $\fbox{109}$, $\fbox{113}$, 115, 117, 121, 123, 129, $\fbox{131}$, 133, where the numbers that appear in the box are primes. To obtain this list we call $\emph{nPS}(135,2)=$, where $\emph{nPS}$ is the program \ref{ProgramnPS}.

The length of the series for $L=10$, $L=10^2$, \ldots, $L=10^6$, respectively, is: 4 (2 primes), 31 (14 primes), 293 (97 primes), 2894 (702 primes), 28886 (5505 primes), 288796 (45204 primes)~.

The series for $L=75$ and $n=3$ is: 1, $\fbox{2}$, 4, $\fbox{5}$, $\fbox{7}$, 8, 10, $\fbox{11}$, 14, 16, $\fbox{17}$, $\fbox{19}$, 20, 22, $\fbox{23}$, 25, 28, $\fbox{29}$, $\fbox{31}$, 32, 34, 35, $\fbox{37}$, 38, $\fbox{41}$, $\fbox{43}$, 46, $\fbox{47}$, 49, 50, 52, 55, 56, 58, $\fbox{59}$, $\fbox{61}$, 62, 64, 65, 68, 70, $\fbox{71}$, $\fbox{73}$, 74~, where the numbers that appear in the box are primes. To obtain this list, we call $\emph{nPS}(75,3)=$, where $nPS$ is the program \ref{ProgramnPS}.

The length of the series for $L=10$, $L=10^2$, \ldots, $L=10^6$, respectively, is: 7 (3 primes), 58 (20 primes), 563 (137 primes), 5606 (1028 primes), 56020 (8056 primes), 560131 (65906 primes)~.

The series for $L=50$ and $n=5$ is: 1, $\fbox{2}$, $\fbox{3}$, 4, 6, $\fbox{7}$, 8, 9, $\fbox{11}$, 12, $\fbox{13}$, 14, 16, $\fbox{17}$, 18, $\fbox{19}$, 21, 22, $\fbox{23}$, 24, 26, 27, 28, $\fbox{29}$, 32, 33, 34, 36, $\fbox{37}$, 38, 39, $\fbox{41}$, 42, $\fbox{43}$, 44, 46, $\fbox{47}$, 48, 49~. To obtain this list, we call $\emph{nPS}(50,5)=$, where $\emph{nPS}$ is the program \ref{ProgramnPS}.

The length of the series for $n=5$ and $L=10$, $L=10^2$, \ldots, $L=10^6$, respectively, is: 8 (3 primes), 77 (23 primes), 761 (161 primes), 7605 (1171 primes), 76037 (9130 primes), 760337 (74631 primes)~.

For counting the primes, we used the Smarandache primality test, \citep{Cira+Smarandache2014}.

Conjectures:
\begin{enumerate}
  \item There are an inf{}inity of primes that belong to this sequence.
  \item There are an inf{}inity of numbers of this sequence which are not prime.
\end{enumerate}

\section{$k$ -- ary Consecutive Sieve}

The series of positive integers to the imposed limit $L$, we delete all the $k$-th term ($k\ge2$), from the remaining series we delete all the $(k+1)$-th term, and so on until possible,, \citep{Le+Smarandache1999}.

\begin{prog}\label{ProgramConS} The program for generating the series to the limit $L$.
  \begin{tabbing}
    $\emph{kConsS}(L,k):=$\ \=\ \vline\ $f$\=$or\ j\in1..L$\\
    \>\ \vline\ \>\ $S_j\leftarrow j$\\
    \>\ \vline\ $f$\=$or\ n\in k..L$\\
    \>\ \vline\ \>\ \vline\ $\emph{break}\ \ \emph{if}\ \ n>\emph{last}(S)$\\
    \>\ \vline\ \>\ \vline\ $f$\=$or\ j\in 1..\emph{floor}\left(\frac{last(S)}{n}\right)$\\
    \>\ \vline\ \>\ \vline\ \>\ $S_{j\cdot n}\leftarrow0$\\
    \>\ \vline\ \>\ \vline\ $i\leftarrow1$\\
    \>\ \vline\ \>\ \vline\ $f$\=$or\ j\in1..\emph{last}(S)$\\
    \>\ \vline\ \>\ \vline\ \>\ $i$\=$f\ S_j\neq0$\\
    \>\ \vline\ \>\ \vline\ \>\ \>\ \vline\ $Q_i\leftarrow S_j$\\
    \>\ \vline\ \>\ \vline\ \>\ \>\ \vline\ $i\leftarrow i+1$\\
    \>\ \vline\ \>\ \vline\ $S\leftarrow Q$\\
    \>\ \vline\ \>\ \vline\ $Q\leftarrow0$\\
    \>\ \vline\ $\emph{return}\ \ S$
   \end{tabbing}
\end{prog}

The series for $k=2$ and $L=10^3$ is: 1, $\fbox{3}$,  $\fbox{7}$,  $\fbox{13}$,  $\fbox{19}$, 27, 39, 49, 63,  $\fbox{79}$, 91,  $\fbox{109}$, 133, 147,  $\fbox{181}$, 207,  $\fbox{223}$, 253, 289,  $\fbox{307}$,  $\fbox{349}$, 387, 399, 459, 481, 529, 567,  $\fbox{613}$, 649,  $\fbox{709}$, 763, 807, 843, 927, 949, where the numbers that appear in a box are primes. This series was obtained by the call $s:=\emph{kConsS}(10^3,2)$.

The length of the series for $k=2$ and $L=10$, $L=10^2$, \ldots, $L=10^6$, respectively, is: 3 (2 primes), 11 (5 primes), 35 (12 primes), 112 (35 primes), 357 (88 primes), 1128 (232 primes)~.

The series for $k=3$ and $L=500$ is: 1, $\fbox{2}$, 4, $\fbox{7}$, 10, 14, 20, 25, 32, 40, 46, 55, $\fbox{67}$, 74, 91, 104, 112, $\fbox{127}$, 145, 154, 175, 194, 200, 230, $\fbox{241}$, 265, 284, $\fbox{307}$, 325, 355, 382, 404, 422, 464, 475, where the numbers that appear in a box are primes. This series was obtained by the call $s:=\emph{kConsS}(500,3)$.

The length of the series for $k=3$ and $L=10$, $L=10^2$, \ldots, $L=10^6$, respectively, is: 5 (2 primes), 15 (3 primes), 50 (10 primes), 159 (13 primes), 504 (30 primes), 1595 (93 primes)~. To count the primes, we used Smarandache
primality test, \citep{Cira+Smarandache2014}.

The series for $k=5$ and $L=300$ is: 1, $\fbox{2}$, $\fbox{3}$, 4, 6, 8, $\fbox{11}$, $\fbox{13}$, $\fbox{17}$, 21, 24, 28, 34, 38, 46, $\fbox{53}$, 57, 64, $\fbox{73}$, 78, 88, 98, $\fbox{101}$, 116, 121, 133, 143, 154, $\fbox{163}$, 178, 192, 203, 212, $\fbox{233}$, 238, 253, 274, 279, 298, where the numbers that appear in a box are primes. This series was obtained by the call $s:=\emph{kConsS}(300,5)$.

The length of the series for $k=5$ and $L=10$, $L=10^2$, \ldots, $L=10^6$, respectively, is: 6 (2 primes), 22 (7 primes), 71 (19 primes), 225 (42 primes), 713 (97 primes), 2256 (254 primes)~. To count the primes, we used Smarandache primality test, \citep{Cira+Smarandache2014}.

\section{Consecutive Sieve}

From the series of positive natural numbers, we eliminate the terms given by the following algorithm. Let $k\ge1$ and $i=k$. Starting with the element $k$ we delete the following $i$ terms. We do $i=i+1$ and $k=k+i$ and repeat this step as many times as possible.

\begin{prog}\label{ProgramConsS} Program for generating the series specif{}ied by the above algorithm.
  \begin{tabbing}
     $\emph{ConsS}(L,k):=$\ \=\vline\ $f$\=$or\ j\in1..L$\\
     \>\vline\ \>\ $S_j\leftarrow j$\\
     \>\vline\ $i\leftarrow k$\\
     \>\vline\ $w$\=$\emph{hile}\ \ k\le L$\\
     \>\vline\ \>\ \vline\ $f$\=$or\ j\in1..i$\\
     \>\vline\ \>\ \vline\ \>\ $S_{k+j}\leftarrow0$\\
     \>\vline\ \>\ \vline\ $i\leftarrow i+1$\\
     \>\vline\ \>\ \vline\ $k\leftarrow k+i$\\
     \>\vline\ $i\leftarrow1$\\
     \>\vline\ $f$\=$or\ j\in1..\emph{last}(S)$\\
     \>\vline\ \>\ $i$\=$f\ S_j\neq0$\\
     \>\vline\ \>\ \>\ \vline\ $Q_i\leftarrow S_j$\\
     \>\vline\ \>\ \>\ \vline\ $i\leftarrow i+1$\\
     \>\vline\ $\emph{return}\ \ Q$
   \end{tabbing}
\end{prog}

The call of program \ref{ProgramConsS} by command $\emph{ConsS}(700,1)$ generates the series: 1, $\fbox{3}$, 6, 10, 15, 21, 28, 36, 45, 55, 66, 78, 91, 105, 120, 136, 153, 171, 190, 210, 231, 253, 276, 300, 325, 351, 378, 406, 435, 465, 496, 528, 561, 595, 630, 666, which has only one prime number.

The length of the series for $k=1$ and $L=10$, $L=10^2$, \ldots, $L=10^6$, respectively, is: 4 (1 prime), 13 (1 prime), 44 (1 prime), 140 (1 prime), 446 (1 prime), 1413 (1 prime)~.

The call of program \ref{ProgramConsS} by command $\emph{ConsS}(700,2)$ generates the series: 1, $\fbox{2}$, $\fbox{5}$, 9, 14, 20, 27, 35, 44, 54, 65, 77, 90, 104, 119, 135, 152, 170, 189, 209, 230, 252, 275, 299, 324, 350, 377, 405, 434, 464, 495, 527, 560, 594, 629, 665, which has 2 primes.

The length of the series for $k=2$ and $L=10$, $L=10^2$, \ldots, $L=10^6$, respectively, is: 4 (2 primes), 13 (2 primes), 44 (2 primes), 140 (2 primes), 446 (2 primes), 1413 (2 primes)~.

For counting the primes, we used the Smarandache primality test, \citep{Cira+Smarandache2014}.

\section{Prime Part}

\subsection{Inferior and Superior Prime Part}

We consider the function $\emph{ipp}:[2,\infty)\to\Na$, $\emph{ipp}(x)=p$, where $p$ is the biggest prime number $p$, $p<x$.

Using the list of primes up to $10^7$ generated by program \ref{ProgramSEPC} in the vector $prime$, we can write a program for the function $\emph{ipp}$.

\begin{prog}\label{Programipp} The program for function $\emph{ipp}$.
  \begin{tabbing}
    $\emph{ipp}(x):=$\ \=\ \vline\ $\emph{return}\ \ "\emph{undef{}ined}"\ \ \emph{if}\ \ x<2\vee x>10^7$\\
    \>\ \vline\ $f$\=$or\ k\in1..\emph{last}(\emph{prime})$\\
    \>\ \vline\ \>\ $\emph{break}\ \ \emph{if}\ \ x<\emph{prime}_k$\\
    \>\ \vline\ $\emph{return}\ \ \emph{prime}_{k-1}$
  \end{tabbing}
\end{prog}

For $n=2,3,\ldots,100$ the values of the function $\emph{ipp}$ are: 2, 3, 3, 5, 5, 7, 7, 7, 7, 11, 11, 13, 13, 13, 13, 17, 17, 19, 19, 19, 19, 23, 23, 23, 23, 23, 23, 29, 29, 31, 31, 31, 31, 31, 31, 37, 37, 37, 37, 41, 41, 43, 43, 43, 43, 47, 47, 47, 47, 47, 47, 53, 53, 53, 53, 53, 53, 59, 59, 61, 61, 61, 61, 61, 61, 67, 67, 67, 67, 71, 71, 73, 73, 73, 73, 73, 73, 79, 79, 79, 79, 83, 83, 83, 83, 83, 83, 89, 89, 89, 89, 89, 89, 89, 89, 97, 97, 97, 97~. The
graphic of the function on the interval $[2,100)$ is given in the Figure \ref{Figippspp}.
\begin{figure}[h]
  \centering
  \includegraphics[scale=0.7]{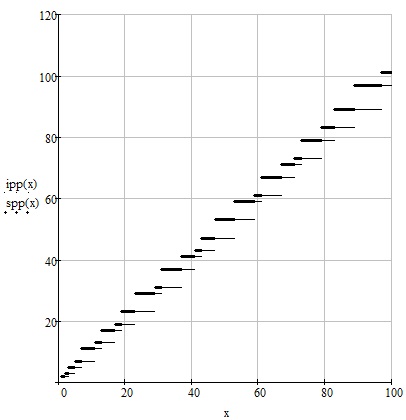}\\
  \caption{Function $\emph{ipp}$ and $\emph{spp}$}\label{Figippspp}
\end{figure}

Another program for the function $\emph{ipp}$ is based on the Smarandache primality criterion, \citep{Cira+Smarandache2014}.

\begin{prog}\label{ProgramippTestulSmarandache} The program for function $\emph{ipp}$ using he Smarandache primality
test (program \ref{Program TS}).
  \begin{tabbing}
    $\emph{ipp}(x):=$\ \=\ \vline\ $\emph{return}\ \ "\emph{nedef{}ined}"\ \ \emph{if}\ \ x<2\vee x>10^7$\\
    \>\ \vline\ $f$\=$or\ k\in \emph{floor}(x)..1$\\
    \>\ \vline\ \>\ $\emph{return}\ \ k\ \ \emph{if}\ \ \emph{TS}(k)\textbf{=}1$\\
    \>\ \vline\ $\emph{return}\ \ "\emph{Error.}"$
  \end{tabbing}
\end{prog}

We consider the function $\emph{spp}:[1,\infty)\to\Na$, $spp(x)=p$, where $p$ is the smallest prime number $p$, $p\ge x$.

Using the list of primes up to $10^7$ generated by program \ref{ProgramSEPC} in the vector $prime$, we can write a program for the function $\emph{spp}$, see Figure \ref{Figispssp}.

\begin{prog}\label{Programspp} Program for the function $\emph{spp}$.
  \begin{tabbing}
    $\emph{spp}(x):=$\=\ \vline\ $\emph{return}\ \ "\emph{undef{}ined}"\ \ \emph{if}\ \ x<1\vee x>10^7$\\
    \>\ \vline\ $f$\=$or\ k\in1..\emph{last}(\emph{prime})$\\
    \>\ \vline\ \>\ $\emph{break}\ \ \emph{if}\ \ x\le \emph{prime}_k$\\
    \>\ \vline\ $\emph{return}\ \ \emph{prime}_k$
  \end{tabbing}
\end{prog}

For $n=1,2,\ldots,100$ the values of the function $\emph{spp}$ are: 2, 2, 3, 5, 5, 7, 7, 11, 11, 11, 11, 13, 13, 17, 17, 17, 17, 19, 19, 23, 23, 23, 23, 29, 29, 29, 29, 29, 29, 31, 31, 37, 37, 37, 37, 37, 37, 41, 41, 41, 41, 43, 43, 47, 47, 47, 47, 53, 53, 53, 53, 53, 53, 59, 59, 59, 59, 59, 59, 61, 61, 67, 67, 67, 67, 67, 67, 71, 71, 71, 71, 73, 73, 79, 79, 79, 79, 79, 79, 83, 83, 83, 83, 89, 89, 89, 89, 89, 89, 97, 97, 97, 97, 97, 97, 97, 97, 101, 101, 101~. The graphic of the function on the interval $[1,100)$ in the Figure \ref{Figippspp}.

\begin{prog}\label{ProgramsppTestulSmarandache} Program for the function $\emph{spp}$ using Smarandache primality test (program \ref{Program TS}).
  \begin{tabbing}
    $\emph{spp}(x):=$\ \=\ \vline\ $\emph{return}\ \ "\emph{nedef{}ined}"\ \ \emph{if}\ \ x<1\vee x>10^7$\\
    \>\ \vline\ $f$\=$or\ k\in \emph{ceil}(x)..\emph{last}(S)$\\
    \>\ \vline\ \>\ $\emph{return}\ \ k\ \ \emph{if}\ \ \emph{TS}(k)\textbf{=}1$\\
    \>\ \vline\ $\emph{return}\ \ "\emph{Error.}"$
  \end{tabbing}
\end{prog}

\begin{apl}
  Determine prime numbers that have among themselves 120 and the length of the gap that contains the number 120.
  \[
   \emph{ipp}(120)=113~,\ \ \ \emph{spp}(120)=127~,\ \ \ \emph{spp}(120)-\emph{ipp}(120)=14~.
  \]
\end{apl}
\begin{apl}
  Write a program that provides the series of maximal gap up to $10^6$. The distance between two consecutive prime numbers is called gap, $g=g_n=g(p_n)=p_{n+1}-p_n$, where $p_n\in\NP{2}$, for any $n\in\Ns$. The series of maximal gaps is the series $\set{g_n}$ with property $g_n>g_k$, for any $n>k$, is: 1, 2, 4, 6, 8, 14, 18, \ldots, where $1=3-2$, $2=5-3$, $4=11-7$, $6=29-23$, $8=97-89$, $14=127-113$, $18=541-523$, \ldots~.
  \begin{tabbing}
    $\emph{Sgm}(L):=$\ \=\ \vline\ $s\leftarrow\left(
                                          \begin{array}{cc}
                                            2 & 1 \\
                                            3 & 2 \\
                                          \end{array}
                                        \right)$\\
    \>\ \vline\ $k\leftarrow2$\\
    \>\ \vline\ $n\leftarrow4$\\
    \>\ \vline\ $w$\=$\emph{hile}\ \ n\le L$\\
    \>\ \vline\ \>\ \vline\ $p_i\leftarrow \emph{ipp}(n)$\\
    \>\ \vline\ \>\ \vline\ $p_s\leftarrow \emph{spp}(n)$\\
    \>\ \vline\ \>\ \vline\ $g\leftarrow p_s-p_i$\\
    \>\ \vline\ \>\ \vline\ $i$\=$f\ g>s_{k,2}$\\
    \>\ \vline\ \>\ \vline\ \>\ \vline\ $k\leftarrow k+1$\\
    \>\ \vline\ \>\ \vline\ \>\ \vline\ $s_{k,1}\leftarrow p_i$\\
    \>\ \vline\ \>\ \vline\ \>\ \vline\ $s_{k,2}\leftarrow g$\\
    \>\ \vline\ \>\ \vline\ $n\leftarrow p_s+1$\\
    \>\ \vline\ $\emph{return}\ \ s$
  \end{tabbing}
  By the call $mg:=\emph{Sgm}(999983)$ (where $999983$ is the biggest prime smaller than $10^6$) we get the result:
  \begin{multline}\label{MaximalGaps}
    mg^\textrm{T}=\left[\begin{array}{rrrrrrrrrrrr}
           2&3&7&23&89&113&523&887&1129&1327&9551&15683 \\
           1&2&4&6&8&14&18&20&22&34&36&44
          \end{array}\right.\\
    \left.\begin{array}{rrrrrr}
      19609&31397&155921&360653&370261&492113 \\
      52&72&86&96&112&114
    \end{array}\right]
  \end{multline}

  So far, we know 75 terms of the series of maximal gaps, \citep{WeissteinPrimeGaps}, \citep{Oliveira2014}. The last maximal gap is 1476 (known as of December, 14, 2014), and the inferior prime is 1425172824437699411. To note that getting a new maximal gap is considered an important mathematical result. For determining these gaps, leading researchers were involved, as: Tom\'{a}s Oliveira e Silva, Donald E. Knuth, Siegfried Herzog, \citep{Oliveira2014}.

  The series (\ref{MaximalGaps}) can also be obtained directly using series of primes.
\end{apl}

\subsection{Inferior and Superior Fractional Prime Part}

\begin{func}\label{Functia ppi} The function \emph{inferior fractional prime part}, $\emph{ppi}:[2,\infty)\to\Real_+$, is
def{}ined by the formula (see Figure \ref{FunctiilePpiPps}):
  \[
  \emph{ppi}(x):=x-ipp(x)~.
  \]
\end{func}

\begin{func}\label{Functia pps} The function \emph{superior fractional prime part}, $\emph{pps}:[2,\infty)\to\Real_+$, is def{}ined by the formula (see Figure \ref{FunctiilePpiPps}):
  \[
  \emph{pps}(x):=\emph{ssp}(x)-x~.
  \]
\end{func}
\begin{figure}[h]
  \centering
  \includegraphics[scale=0.8]{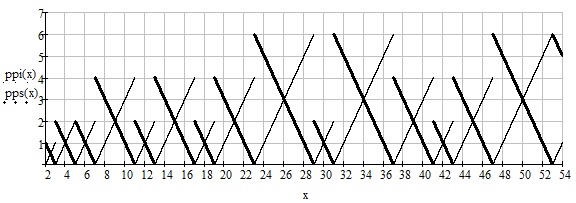}\\
  \caption{The graphic of the functions $\emph{ppi}$ and $\emph{pps}$ }\label{FunctiilePpiPps}
\end{figure}

Examples of calls of the functions $\emph{ppi}$ and $\emph{pps}$:
\[
 \emph{ppi}(\pi^3+e^5)=0.41943578287637706~,\ \ \ \emph{pps}(\pi^3+e^5)=1.580564217123623~.
\]

\section{Square Part}

\subsection{Inferior and Superior Square Part}

The functions $\emph{isp},\emph{ssp}:\Real_+\to\Na$, are the inferior square part and respectively
the superior square part of the number $x$, \citep{Popescu+Nicolescu1996}.
\begin{figure}[h]
  \centering
  \includegraphics[scale=0.8]{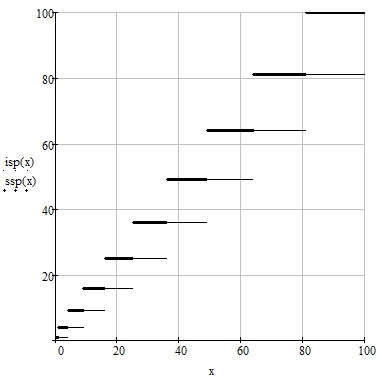}\\
  \caption{The graphic of the functions $\emph{isp}$ and $\emph{ssp}$}\label{Figispssp}
\end{figure}
\begin{func}\label{Functia isp} The function $\emph{isp}$ is given by the formula (see Figure \ref{Figispssp}):
  \[
   \emph{isp}(x):=\left(\emph{floor}(\sqrt{x})\right)^2~.
  \]
\end{func}
\begin{func}\label{Functia ssp} The function $\emph{ssp}$ is given by the formula (see Figure \ref{Figispssp}):
  \[
   \emph{ssp}(x):=\left(\emph{ceil}(\sqrt{x})\right)^2~.
  \]
\end{func}
\begin{apl}
  To determine between what perfect squares we f{}ind the irrational numbers: $\pi^\phi$, $\pi^{\phi^2}$, $e^{\phi+1}$, $e^{2\phi+3}$, $\phi^{e^2}$, $\phi^{\pi^3}$, $e^{\pi+\phi}$, where $\phi$ is the golden number $\phi=(1+\sqrt{5})/2$. We f{}ind the answer in the Table \ref{AplicatieIspSsp}.
  \begin{table}[h]
    \centering
    \begin{tabular}{|r|c|r|}
      \hline
      $\emph{isp}(x)$ & $x$ & $\emph{ssp}(x)$ \\ \hline
      4 & $\pi^\phi$ & 9 \\
      16 & $\pi^{\phi^2}$ & 25 \\
      9 & $e^{\phi+1}$ & 16 \\
      474 & $e^{2\phi+3}$ & 529 \\
      25 & $\phi^{e^2}$ & 36 \\
      3017169 & $\phi^{\pi^3}$ & 3020644 \\
      100 & $e^{\pi+\phi}$ & 121 \\
      \hline
    \end{tabular}
    \caption{Applications to functions $\emph{isp}$ and $\emph{ssp}$}\label{AplicatieIspSsp}
  \end{table}
\end{apl}

\subsection{Inferior and Superior Fractional Square Part}

\begin{func}\label{Functia spi} The function \emph{inferior fractional square part}, $\emph{spi}:\Real\to\Real_+$, is given
by the formula (see Figure \ref{FunctiileSpiSps}):
  \[
  \emph{spi}(x):=x-\emph{isp}(x)~.
  \]
\end{func}

\begin{func}\label{Functia sps} The function \emph{superior fractional square part}, $\emph{sps}:\Real\to\Real_+$, is given by the formula (see Figure \ref{FunctiileSpiSps}):
  \[
  \emph{sps}(x):=\emph{ssp}(x)-x~.
  \]
\end{func}
\begin{figure}
  \centering
  \includegraphics[scale=0.8]{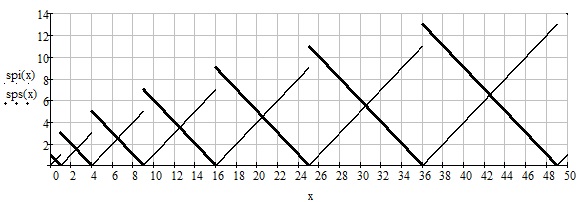}\\
  \caption{The graphic of the functions $\emph{spi}$ and $\emph{sps}$ }\label{FunctiileSpiSps}
\end{figure}

Examples of calls of the functions $\emph{spi}$ and $\emph{sps}$:
\[
 \emph{spi}(\pi^3)=6.006276680299816,\ \ \emph{sps}(\pi^3+e^3)=12.90818639651252~.
\]

\section{Cubic Part}

\subsection{Inferior and Superior Cubic Part}

The functions $\emph{icp},\emph{scp}:\Real\to\mathbb{Z}$, are the \emph{inferior cubic part} and respectively the
\emph{superior cubic part} of the number $x$, \citep{Popescu+Seleacu1996}.
\begin{func}\label{Functia icp} The function $\emph{icp}$ is given by the formula (see Figure \ref{Figicpscp}):
  \[
   \emph{icp}(x):=\left(\emph{floor}(\sqrt[3]{x})\right)^3~.
  \]
\end{func}
\begin{func}\label{Functia scp} The function $\emph{scp}$ is given by the formula (see Figure \ref{Figicpscp}):
  \[
   \emph{scp}(x):=\left(\emph{ceil}(\sqrt[3]{x})\right)^3~.
  \]
\end{func}

\begin{figure}[h]
  \centering
  \includegraphics[scale=0.8]{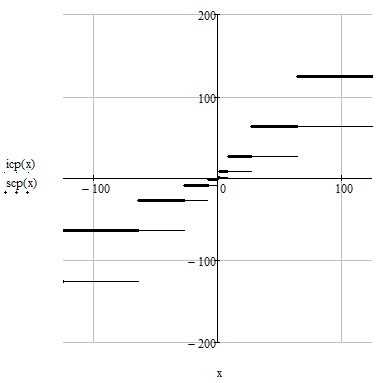}\\
  \caption{The functions $\emph{icp}$ and $\emph{scp}$}\label{Figicpscp}
\end{figure}

\begin{apl}
  Determine where between the perfect cubes we f{}ind the irrational numbers: $\phi^2+e^3+\pi^4$, $\phi^3+e^4+\pi^5$, $\phi^4+e^5+\pi^6$, $e^{\pi\sqrt{58}}$, $e^{\pi\sqrt{163}}$ where $\phi=(1+\sqrt{5})/2$ is the golden number. The answer is in the Table \ref{TabelAplRamanujan}.
  \begin{table}[h]
    \centering
    \begin{tabular}{|r|c|r|}
      \hline
      $\emph{icp}(x)$ & $x$ & $\emph{scp}(x)$ \\ \hline
      64 & $\phi^2+e^3+\pi^4$ & 125 \\
      343 & $\phi^3+e^4+\pi^5$ & 512 \\
      1000 & $\phi^4+e^5+\pi^6$ & 1331 \\
      24566036643 & $e^{\pi\sqrt{58}}$ & 24591397312 \\
      262537412640768000 & $e^{\pi\sqrt{163}}$ & 262538642671796161 \\
      \hline
    \end{tabular}
    \caption{Applications to the functions $\emph{icp}$ and $\emph{scp}$}\label{TabelAplRamanujan}
  \end{table}
The constants $e^{\sqrt{58}\pi}$ and $e^{\sqrt{163}\pi}$ are related to the results of the noted indian mathematician Srinivasa Ramanujan\index{Ramanujan S.} and we have $262537412640768000=640320^3$, $262538642671796161=640321^3$, $24566036643=2907^3$ and $24591397312=2908^3$. As known the number $e^{\pi\sqrt{163}}$ is an \emph{almost integer} of $640320^3+744$ or of $\big(\emph{icp}(e^{\pi\sqrt{163}})\big)^3+744$ because
\[
 \abs{e^{\pi\sqrt{163}}-(640320^3+744)}\approx7.49927460489676830923677642\times10^{-13}~.
\]
\end{apl}

\subsection{Inferior and Superior Fractional Cubic Part}

\begin{func}\label{Functia cpi} The \emph{inferior fractional cubic part} function, $\emph{cpi}:\Real\to\Real_+$, is def{}ined by the formula (see Figure \ref{FunctiileCpiCps}):
  \[
  \emph{cpi}(x):=x-\emph{icp}(x)~.
  \]
\end{func}

\begin{func}\label{Functia cps} The \emph{superior fractional cubic part} function, $\emph{cps}:\Real\to\Real_+$, is given
by the formula (see Figure \ref{FunctiileCpiCps}):
  \[
  \emph{cps}(x):=\emph{scp}(x)-x~.
  \]
\end{func}
\begin{figure}[h]
  \centering
  \includegraphics[scale=0.7]{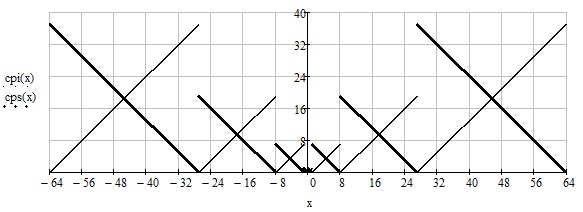}\\
  \caption{The graphic of the functions $\emph{cpi}$ and $\emph{cps}$ }\label{FunctiileCpiCps}
\end{figure}

Examples of calling the functions $\emph{cpi}$ and $\emph{cps}$:
\begin{multline*}
  cpi\left(\sqrt{\pi^3+\pi^5}\right)=10.358266842640163~,\\
  cpi\left(e^{\pi\sqrt{163}}\right)=743.99999999999925007253951~,\\
  cps\left(\sqrt{\pi^3+\pi^5}\right)=8.641733157359837~.
\end{multline*}

\section{Factorial Part}

\subsection{Inferior and Superior Factorial Part}

 The functions $\emph{ifp},\emph{sfp}:\Real_+\to\Na$, are the \emph{inferior factorial part} and respectively the \emph{superior factorial part} of the number $x$, \citep{Dumitrescu+Seleacu1994}.
\begin{prog}\label{Programifp} Program for function $\emph{ifp}$.
  \begin{tabbing}
    $\emph{ifp}(x):=$\ \=\vline\ $\emph{return}\ "\emph{undef{}ined}"\ \emph{if}\ x<0\vee x>18!$\\
    \>\vline\ $f$\=$or\ k\in1..18$\\
    \>\vline\ \>\ $\emph{return}\ (k-1)!\ \emph{if}\ x<k!$\\
    \>\vline\ $\emph{return}\ "\emph{Err}."$
  \end{tabbing}
\end{prog}
\begin{prog}\label{Programsfp} Program for function $sfp$.
  \begin{tabbing}
    $\emph{sfp}(x):=$\ \=\vline\ $\emph{return}\ "\emph{undef{}ined}"\ \emph{if}\ x<0\vee x>18!$\\
    \>\vline\ $f$\=$or\ k\in1..18$\\
    \>\vline\ \>\ $\emph{return}\ k!\ \emph{if}\ x<k!$\\
    \>\vline\ $\emph{return}\ "\emph{Error.}."$
  \end{tabbing}
\end{prog}
\begin{apl}
  Determine in what factorial the numbers $e^{k\pi}$ for $k=1,2,\ldots,11$.
  \begin{center}
    \begin{longtable}{|c|}
      \caption{Factorial parts for $e^{k\pi}$}\label{Aplicatie ifp sfp}\\ \hline
      \endfirsthead
      \endhead
      \multicolumn{1}{c}{\textit{Continued on next page}} \\
      \endfoot
      \endlastfoot
      $6=3!>e^\pi>4!=24$\\
      $120=5!>e^{2\pi}>6!=720$\\
      $5040=7!>e^{3\pi}>8!=40320$\\
      $40320=8!>e^{4\pi}>9!=362880$\\
      $3628800=10!>e^{5\pi}>11!=39916800$\\
      $39916800=11!>e^{6\pi}>12!=479001600$\\
      $479001600=12!>e^{7\pi}>13!=6227020800$\\
      $6227020800=13!>e^{8\pi}>14!=87178291200$\\
      $1307674368000=15!>e^{9\pi}>16!=20922789888000$\\
      $20922789888000=16!>e^{10\pi}>17!=355687428096000$\\
      $355687428096000=17!>e^{11\pi}>18!=6402373705728000$\\
      \hline
    \end{longtable}
  \end{center}
\end{apl}

\begin{func}\label{Functia fpi} The \emph{inferior factorial dif{}ference part}, $fpi:\Real_+\to\Real_+$, is def{}ined by the formula:
  \[
   \emph{fpi}(x):=x-\emph{ifp}(x)!~.
  \]
\end{func}

\begin{func}\label{Functia fps} The \emph{superior factorial dif{}ference part}, $fps:\Real_+\to\Real_+$, is given by the formula:
  \[
   \emph{fps}(x):=\emph{sfp}(x)!-x~.
  \]
\end{func}

\begin{apl}
  Determine the \emph{inferior} and \emph{superior factorial dif{}ference parts} for the numbers $e^{k\pi}$ for $k=1,2,\ldots,10$.
  \begin{center}
    \begin{longtable}{|c|}
      \caption{Factorial dif{}ference parts for $e^{k\pi}$}\label{Aplicatie fp1 fps}\\ \hline
      \endfirsthead
      \endhead
      \multicolumn{1}{c}{\textit{Continued on next page}} \\
      \endfoot
      \endlastfoot
      $\emph{fpi}\left(e^\pi\right)=17.140692632779263$\\
      $\emph{fpi}\left(e^{2\pi}\right)=415.4916555247644$\\
      $\emph{fpi}\left(e^{3\pi}\right)=7351.647807916686$\\
      $\emph{fpi}\left(e^{4\pi}\right)=246431.31313665299$\\
      $\emph{fpi}\left(e^{5\pi}\right)=3006823.9993411247$\\
      $\emph{fpi}\left(e^{6\pi}\right)=113636135.39544642$\\
      $\emph{fpi}\left(e^{7\pi}\right)=3074319680.8470373$\\
      $\emph{fpi}\left(e^{8\pi}\right)=75999294785.594800$\\
      $\emph{fpi}\left(e^{9\pi}\right)=595099527292.15620$\\
      $\emph{fpi}\left(e^{10\pi}\right)=23108715972631.90$\\ \hline
      $\emph{fps}\left(e^\pi\right)=0.85930736722073680$\\
      $\emph{fps}\left(e^{2\pi}\right)=184.50834447523562$\\
      $\emph{fps}\left(e^{3\pi}\right)=27928.352192083315$\\
      $\emph{fps}\left(e^{4\pi}\right)=76128.686863347010$\\
      $\emph{fps}\left(e^{5\pi}\right)=33281176.000658877$\\
      $\emph{fps}\left(e^{6\pi}\right)=325448664.60455360$\\
      $\emph{fps}\left(e^{7\pi}\right)=2673699519.1529627$\\
      $\emph{fps}\left(e^{8\pi}\right)=4951975614.4051970$\\
      $\emph{fps}\left(e^{9\pi}\right)=19020015992707.844$\\
      $\emph{fps}\left(e^{10\pi}\right)=311655922235368.10$\\ \hline
    \end{longtable}
  \end{center}
\end{apl}

\section{Function Part}

Let $f$ be a strictly ascending real on the interval $[a,b]$, where $a,b\in\Real$, $a<b$. We can generalize the notion of part (inferior or superior) in relation to the function $f$, \citep{Castillo}.

\subsection{Inferior and Superior Function Part}

We def{}ine the function, $\emph{ip}:[a,b]\to\Real$, \emph{inferior part relative to the function} $f$.

\begin{prog}\label{ProgramIp} Program for the function $\emph{ip}$. We have to def{}ine the function $f$ in
relation to which we consider the inferior part function. It remains the responsibility of the user as the function $f$ to be strictly ascending on $[a,b]\subset\Real$.
  \begin{tabbing}
    $\emph{ip}(f,a,b,x):=$\=\vline\ $\emph{return}\ "\emph{undef{}ined}"\ \ \emph{if}\ \ x<a\vee x>b$\\
    \>\vline\ $\emph{return}\ "a\ or\ b\ \emph{not}\ \emph{integer}"\ \ \emph{if}\ \ a\neq \emph{trunc}(a)\vee b\neq \emph{trunc}(b)$\\
    \>\vline\ $f$\=$or\ z\in a..b$\\
    \>\vline\ \>\ $\emph{return}\ z-1\ \ \emph{if}\ \ x<f(z)$\\
    \>\vline\ $\emph{return}\ \ "\emph{Error.}"$
  \end{tabbing}
We want to determine the inferior part of $e^\pi$ in relation to the function $f(z):=2z+\ln(z^2+z+1)$ on the interval $[0,10^6]$. The function is strictly ascending on the interval $[0,10^6]$ so it makes sense to consider the \emph{inferior part function} in relation to $f$ and we have $\emph{ip}(f,0,10^6,e^\pi)=22.51085950651685$. Other examples of calling the function $\emph{ip}$:
\begin{eqnarray*}
  g(z):=z+\sqrt{z} && \emph{ip}\left(g,0,10^2,e^\pi\right)=22.242640687119284~, \\
  h(z):=z+3\arctan(z) && \emph{ip}\left(h,-6,6,e^{\sqrt{\pi}}\right)=5.321446153382271~,\\
  && \emph{ip}\left(h,-6,6,e^{2\sqrt{\pi}}\right)="undef{}ined"~,\\
  && \emph{ip}\left(h,-36,36,e^{2\sqrt{\pi}}\right)=34.61242599274995~,\\
\end{eqnarray*}
\end{prog}

We def{}ine the function, $\emph{sp}:[a,b]\to\Real$, \emph{superior part relative to the function} $f$.

\begin{prog}\label{ProgramSp} Program for function $\emph{sp}$. We have to def{}ine the function $f$ related to which we consider the function a superior part. It remains the responsibility of the user as the function $f$ to be strictly increasing $[a,b]\subset\Real$.
  \begin{tabbing}
    $\emph{sp}(f,a,b,x):=$\=\vline\ $\emph{return}\ "\emph{undef{}ined}"\ \emph{if}\ x<a\vee x>b$\\
    \>\vline\ $\emph{return}\ "a\ or\ b\ not\ \emph{integer}"\ \emph{if}\ a\neq \emph{trunc}(a)\vee b\neq \emph{trunc}(b)$\\
    \>\vline\ $f$\=$or\ z\in a..b$\\
    \>\vline\ \>\ $\emph{return}\ z\ \emph{if}\ x<f(z)$\\
    \>\vline\ $\emph{return}\ "\emph{Err}."$
  \end{tabbing}
We want to determine the superior part of $e^\pi$ related to the function $f(z):=2z+\ln(z^2+z+1)$ on the interval $[0,10^6]$. The function is strictly ascending on the interval $[0,10^6]$ so it makes sense to consider the part function in relation to $f$ and we have $\emph{sp}(f,0,10^6,e^\pi)=24.709530201312333$. Other examples of function $\emph{sp}$:
\begin{eqnarray*}
  g(z):=z+\sqrt{z} && \emph{sp}\left(g,0,10^2,e^\pi\right)=23.358898943540673~, \\
  h(z):=z+3\arctan(z) && \emph{sp}\left(h,-6,6,e^{\sqrt{\pi}}\right)=6.747137317194763~,\\
  && \emph{sp}\left(h,-6,6,e^{2\sqrt{\pi}}\right)="undef{}ined"~,\\
  && \emph{sp}\left(h,-36,36,e^{2\sqrt{\pi}}\right)=35.61564833307893~,\\
\end{eqnarray*}
\end{prog}

\begin{obs}\label{ObservatiaPrecizia}
  All values displayed by functions $\emph{ip}$ and $\emph{sp}$ have an accuracy of mathematical computing, given by software implementation, of $10^{-15}$. To obtain better accuracy it is necessary to turn to symbolic computation.
\end{obs}

\subsection{Inferior and Superior Fractional Function Part}

The \emph{fractional inferior part} function in relation to the function $f$,  $\emph{ipd}:[a,b]\subset\Real$, is given by the formula $\emph{ipd}(f,a,b,x):=x-\emph{ip}(f,a,b,x)$. Before the call of function $\emph{ipd}$ we have to def{}ine the strictly ascending function $f$ on the real interval $[a,b]$. Examples of calls of function $\emph{ipd}$:
\begin{eqnarray*}
  f(z):=2z+\ln(z^2+z+1) && \emph{ipd}\left(f,0,10^6,e^\pi\right)=0.6298331262624117~,\\
  g(z):=z+\sqrt{z} && \emph{ipd}\left(g,0,10^2,e^\pi\right)=0.8980519456599794~, \\
  h(z):=z+3\arctan(z) && \emph{ipd}\left(h,-6,6,e^{\sqrt{\pi}}\right)=0.5638310966357558~,\\
\end{eqnarray*}

The \emph{fractional superior part} function in relation to the function $f$,  $\emph{spd}:[a,b]\subset\Real$, is given by the formula $\emph{sdp}(f,a,b,x):=\emph{sp}(f,a,b,x)-x$. As with the function $\emph{ipd}$ before the call of function $\emph{spd}$ we have to def{}ine the strictly ascending function $f$ on the real interval $[a,b]$. Examples of calls of function $\emph{spd}$:
\begin{eqnarray*}
  f(z):=2z+\ln(z^2+z+1) && \emph{spd}\left(f,0,10^6,e^\pi\right)=1.5688375685330698~,\\
  g(z):=z+\sqrt{z} && \emph{spd}\left(g,0,10^2,e^\pi\right)=0.21820631076140984~, \\
  h(z):=z+3\arctan(z) && \emph{spd}\left(h,-6,6,e^{\sqrt{\pi}}\right)=0.8618600671767362~,\\
\end{eqnarray*}

The remark taken in Observation \ref{ObservatiaPrecizia} is valid for the functions $\emph{ipd}$ and $\emph{spd}$.

\section{Smarandache type Functions}

\subsection{Smarandache Function}

The function that associates to each natural number $n$ the smallest natural number $m$ which has the property that $m!$ is a multiple of $n$ was considered for the f{}irst time by \cite{Lucas1883}\index{Lucas F. E. A.}. Other authors who have
considered this function in their works are: \cite{Neuberg1887}\index{Neuberg J.}, \cite{Kempner1918}\index{Kempner A. J.}. This function was rediscovered by \cite{Smarandache1980}\index{Smarandache F.}.

Therefore, function $S:\Ns\to\Ns$, $S(n)=m$, where $m$ is the smallest natural that has the property that $n$ divides $m!$, (or $m!$ is a multiple of $n$) is known in the literature as \emph{Smarandache\rq{s} function}, \citep{Hazewinkel2011}, \citep{DeWikipediaSF,DeWikipediaSK}.
\begin{figure}[h]
  \centering
  \includegraphics[scale=1]{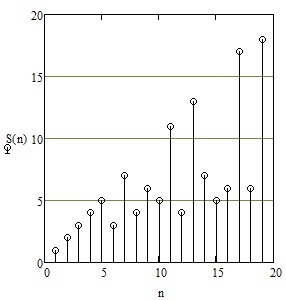}\\
  \caption{$S$ function}
\end{figure}
The values of the function, for $n=1,2,\ldots,18$, are: $1$, $2$, $3$, $4$, $5$, $3$, $7$, $4$, $6$, $5$, $11$, $4$, $13$, $7$, $5$, $6$, $17$, $6$ obtained by means of an algorithm that results from the def{}inition of function $S$, as follows:

\begin{prog}\label{ProgramS}\
  \begin{tabbing}
   $S(n)=$\=\ \vline\ $f$\=$or\ m=1..n$\\
   \>\ \vline\ \>\ $\emph{return}\ m\ \ \emph{if}\ \ \mod(m!,n)\textbf{=}0$
 \end{tabbing}
\end{prog}

\subsubsection{Properties of the function $S$}

\begin{enumerate}
  \item $S(1)=1$, where it should be noted that \cite[A002034]{SloaneOEIS} def{}ines $S(1)=1$, while \cite{Ashbacher1995} and \cite[p. 4]{Russo2000} take $S(1)=0$;
  \item $S(n!)=n$, \citep{Sondow+Weisstein};
  \item $S(p^m)=m\cdot p$, where $p\in\NP{2}$ and $1\le m<p$, \citep{Kempner1918}; Particular case: $m=1$, then $S(p)=p$, for all $p\in\NP{2}$;
  \item $S(p_1\cdot p_2\cdots p_m)=p_m$, where $p_1<p_2<\ldots<p_m$ and $p_k\in\NP{2}$, for all $k\in\Ns$, \citep{Sondow+Weisstein};
  \item $S(p^{p^m})=p^{m+1}-p^m+1$, for all $p\in\NP{2}$, \citep{Ruiz1999b};
  \item $S(P_p)=M_p$, if $P_p$ is the $p$th even \emph{perfect number} and $M_p$ is the corresponding \emph{Mersenne prime}, \citep{Ashbacher1997,Ruiz1999a}.
\end{enumerate}

\subsection{Smarandache Function of Order k}

\begin{defn}\label{DefinitiaFunctiaSmarandachek}
The function $S_k:\Ns\to\Ns$, $m=S_k(n)$ is the smallest integer $m$ such that
\[
 n\mid m\underbrace{!!\ldots !}_{k\ times}\ \ \textnormal{or}\ \ \ n\mid\emph{kf}(m,k)~.
\]
\end{defn}

The function multifactorial $\emph{kf}$ is given by \ref{Programkf}. For $k=1$ we can say that $S_1(n)=S(n)$, given by \ref{ProgramS}.

\begin{prog}\label{Program Smarandache} for calculating the values of the Smarandache function of $k$ rank.
  \begin{tabbing}
    $S(n,k):=$\=\vline\ $f$\=$or\ m\in1..n$\\
    \>\vline\ \>\ $\emph{return}\ \ m\ \ \emph{if}\ \ \mod(\emph{kf}(m,k),n)=0$\\
  \end{tabbing}
  which uses the function $\emph{kf}$ given by \ref{Programkf}.
\end{prog}
To display the f{}irst 30 values of the functions $S$ we use the sequence of commands: $n=1..30$ $S(n,2)\rightarrow$ and $S(n,3)\rightarrow$.
\begin{enumerate}
  \item The values of the function $S_2(n)$ are: 1, 2, 3, 4, 5, 6, 7, 4, 9, 10, 11, 6, 13, 14, 5, 6, 17, 12, 19, 10, 7, 22, 23, 6, 15, 26, 9, 14, 29, 10~.
  \item The values of the function $S_3(n)$ are: 1, 2, 3, 4, 5, 6, 7, 8, 6, 5, 11, 12, 13, 7, 15, 8, 17, 6, 19, 8, 21, 11, 23, 12, 20, 13, 9, 7, 29, 15~.
\end{enumerate}

\subsubsection{Properties of the function $S_k$, $k>1$}

The question is whether we have the same set of properties as the Smarandache function $S_k$?
\begin{enumerate}
  \item $S_k(1)=1$, may be taken by convention;
  \item $S_2(n!!)=n$ and $S_3(n!!!)=n$, result from def{}inition;
  \item $S_k(p^\alpha)=[k\cdot\alpha-(k-1)]p$, where $p\in\NP{3}$ and $2\le \alpha<p$;
      \begin{enumerate}
        \item If $k=2$
             \[
               p^\alpha\mid1\cdots p\cdots3p\cdots5p\cdots(2\alpha-1)p=[(2\alpha-1)p]!!
             \]
             and therefore
             \[
              S_2(p^{\alpha})=(2\alpha-1)p~.
             \]
        \item If $k=3$
            \[
             p^\alpha\mid1\cdots p\cdots7p\cdots13p\cdots(3\alpha-2)p=[(3\alpha-2)p]!!!
            \]
            and therefore
            \[
             S_3(p^\alpha)=(3\alpha-2)p~.
            \]
      \end{enumerate}
      In the case: $\alpha=1$, then $S_k(p^\alpha)=S_k(p)==p$, for $k>1$ and all $p\in\NP{2}$;
  \item $S_2(2\cdot p_1\cdot p_2\cdots p_m)=2\cdot p_m$? Where $p_1<p_2<\ldots<p_m$ and $p_k\in\NP{3}$, for all $k\in\Ns$?
  \item $S_2(3\cdot p_1\cdot p_2\cdots p_m)=3\cdot p_m$? Where $p_1<p_2<\ldots<p_m$ and $p_k\in\NP{5}$, for all $k\in\Ns$?
  \item $S_3(p_1\cdot p_2\cdots p_m)=2\cdot p_m$? Where $p_1<p_2<\ldots<p_m$, $p_1\neq3$, and $p_k\in\NP{2}$, for all $k\in\Ns$?
  \item $S_3(3\cdot p_1\cdot p_2\cdots p_m)=3\cdot p_m$? Where $p_1<p_2<\ldots<p_m$ and $p_k\in\NP{5}$, for all $k\in\Ns$?
\end{enumerate}

\subsection{Smarandache--Cira Function of Order k}

Function $\emph{SC}:\Ns\times\Ns\to\Ns$, $m=\emph{SC}(n,k)$ is the smallest integer $m$ such that $n\mid 1^k\cdot2^k\cdots m^k$ (or $n\mid (m!)^k$). For $k=1$ we can say that $S(n,1)=S(n)$, given by \ref{ProgramS}.

\begin{prog}\label{Program Smarandache-Cira Order k} for calculating the values of Smarandache--Cira function of $k$ rank.
  \begin{tabbing}
    $\emph{SC}(n,k):=$\=\vline\ $f$\=$or\ m\in1..n$\\
    \>\vline\ \>\ $\emph{return}\ m\ \ \emph{if}\ \ \mod((m!)^k,n)=0$\\
  \end{tabbing}
\end{prog}

The values given by the function $\emph{SC}(n,1)$ The values given by the function Smarandache $S$, \ref{ProgramS}.

To display the f{}irst 113 values of the functions $\emph{SC}$ we use the sequence of commands: $n=1..113$, $\emph{SC}(n,2)\rightarrow$ and $\emph{SC}(n,3)\rightarrow$.
\begin{enumerate}
  \item The values of the function $\emph{SC}(n,2)$ are: 1, 2, 3, 2, 5, 3, 7, 4, 3, 5, 11, 3, 13, 7, 5, 4, 17, 3, 19, 5, 7, 11, 23, 4, 5, 13, 6, 7, 29, 5, 31, 4, 11, 17, 7, 3, 37, 19, 13, 5, 41, 7, 43, 11, 5, 23, 47, 4, 7, 5, 17, 13, 53, 6, 11, 7, 19, 29, 59, 5, 61, 31, 7, 4, 13, 11, 67, 17, 23, 7, 71, 4, 73, 37, 5, 19, 11, 13, 79, 5, 6, 41, 83, 7, 17, 43, 29, 11, 89, 5, 13, 23, 31, 47, 19, 4, 97, 7, 11, 5, 101, 17, 103, 13, 7, 53, 107, 6, 109, 11, 37, 7, 113~.
  \item The values of the function $\emph{SC}(n,3)$ are: 1, 2, 3, 2, 5, 3, 7, 2, 3, 5, 11, 3, 13, 7, 5, 4, 17, 3, 19, 5, 7, 11, 23, 3, 5, 13, 3, 7, 29, 5, 31, 4, 11, 17, 7, 3, 37, 19, 13, 5, 41, 7, 43, 11, 5, 23, 47, 4, 7, 5, 17, 13, 53, 3, 11, 7, 19, 29, 59, 5, 61, 31, 7, 4, 13, 11, 67, 17, 23, 7, 71, 3, 73, 37, 5, 19, 11, 13, 79, 5, 6, 41, 83, 7, 17, 43, 29, 11, 89, 5, 13, 23, 31, 47, 19, 4, 97, 7, 11, 5, 101, 17, 103, 13, 7, 53, 107, 3, 109, 11, 37, 7, 113~.
\end{enumerate}

\section{Smarandache--Kurepa Functions}

\subsection{Smarandache--Kurepa Function of Order 1}
We notation
\begin{equation}\label{SumaDeFactoriali}
  \Sigma_1(n)=\sum_{k=1}^nk!~.
\end{equation}

\begin{prog}\label{Program Sigma(k,n)} for calculating the sum (\ref{SumaDeFactoriali}), (\ref{SumaDeDubluFactoriali}) and \ref{SumaDeTripluFactoriali}.
  \begin{tabbing}
    $\Sigma(k,n):=$\=\vline\ $s\leftarrow0$\\
    \>\vline\ $f$\=$or\ j\in1..n$\\
    \>\vline\ \>\ $s\leftarrow s+\emph{kf}(j,k)$\\
    \>\vline\ $\emph{return}\ s$\\
   \end{tabbing}
   The program uses the subprogram $\emph{kf}$, \ref{Programkf}.
\end{prog}

With commands $n:=1..20$ and $\Sigma(1,n)\rightarrow$, result f{}irst 20 values of the function $\Sigma_1$:
\begin{multline*}
  1, 3, 9, 33, 153, 873, 5913, 46233, 409113, 4037913, 43954713,\\
  522956313, 6749977113, 93928268313, 1401602636313, 22324392524313,\\
  378011820620313, 6780385526348313, 128425485935180313,\\
  2561327494111820313~.
\end{multline*}

\begin{defn}[\citep{Mudge1996,Mudge1996a,Ashbacher1997}]
  The function $SK_1:\NP{2}\to\Ns$, $m=SK_1(p)$ is the smallest $m\in\Ns$ such that $p\mid[1+\Sigma_1(m-1)]$.
\end{defn}

\begin{prop}
  If $p\nmid[1+\Sigma_1(m-1)]$, for all $m\le p$, then $p$ never divides any sum for all $m>p$~.
\end{prop}
\begin{proof}
  If $p\nmid[1+\Sigma_1(m-1)]$, for all $m\le p$, then $1+\Sigma_1(p-1)=\mathcal{M}\cdot p+r$, with $1\le r<p$.

  Let $m>p$, then
  \begin{multline*}
    1+\Sigma_1(m-1)=1+\Sigma_1(p-1)+p!+(p+1)!+\ldots+(m-1)!\\
    =\mathcal{M}\cdot p+r+p!+(p+1)!+\ldots+(m-1)!\\
    =[\mathcal{M}+(p-1)!\big(1+(p+1)+\ldots+(p+1)\cdots(m-1)\big)]p+r\\
    =\mathcal{M}\cdot p+r~,
  \end{multline*}
  then $p\nmid[1+\Sigma_1(m-1)]$ for all $m>p$.
\end{proof}

\begin{prog}\label{Program SK(k,p)} for calculating the values of functions $SK_1$, $SK_2$ and $SK_3$.
  \begin{tabbing}
    $SK(k,p):=$\=\vline\ $f$\=$or\ m\in2..k\cdot p-1$\\
    \>\vline\ \>\ $\emph{return}\ \ m\ \ \emph{if}\ \ \mod(1+\Sigma(k,m-1),p)\textbf{=}0$\\
    \>\vline\ $\emph{return}\ \ -1$\\
  \end{tabbing}
  The program uses the subprogram $\Sigma$, \ref{Program Sigma(k,n)}, and the utilitarian function Mathcad $mod$.
\end{prog}

With commands $k:=1..25$ and $SK(1,prime_k)\rightarrow$ are obtained f{}irst 25 values of the function $SK_1$:
\begin{multline}\label{Vector sk1}
  \begin{tabular}{|l|r|r|r|r|r|r|r|r|r|r|r|r|}
    \hline
    $p$       & 2 &    3 & 5 & 7 & 11 &   13 & 17 & 19 & 23 & 29   & 31 & 37 \\
    $SK_1(p)$ & 2 & $-1$ & 4 & 6 &  6 & $-1$ &  5 &  7 &  7 & $-1$ & 12 & 22 \\
    \hline
  \end{tabular}\\
  \begin{tabular}{|r|r|r|r|r|r|r|r|r|r|r|r|r|}
    \hline
    41 &   43 &   47 &   53 &   59 & 61 &  67 & 71 & 73 &   79 &   83 & 89 & 97 \\
    16 & $-1$ & $-1$ & $-1$ & $-1$ & 55 & $-1$& 54 & 42 & $-1$ & $-1$ & 24 & $-1$\\
    \hline
  \end{tabular}
\end{multline}
If $SK_1(p)=-1$, then for $p$ the function $SK_1$ is undef{}ined, \citep{WeissteinSmarandacheKurepaFunction}.

\subsection{Smarandache--Kurepa Function of order 2}

We notation
\begin{equation}\label{SumaDeDubluFactoriali}
  \Sigma_2(n)=\sum_{k=1}^nk!!~.
\end{equation}

With commands $n:=1..20$ and $\Sigma(2,n)\rightarrow$, (using the program \ref{Program Sigma(k,n)}) result f{}irst 20 values of the function $\Sigma_2$:
\begin{multline*}
  1, 3, 6, 14, 29, 77, 182, 566, 1511, 5351, 15746, 61826, \\
  196961, 842081, 2869106, 13191026, 47650451, 233445011, \\
  888174086, 4604065286~.
\end{multline*}

\begin{defn}
  The function $SK_2:\NP{2}\to\Ns$, $m=SK_2(p)$ is the smallest $m\in\Ns$ such that $p\mid[1+\Sigma_2(m-1)]$.
\end{defn}

\begin{prop}
  If $p\nmid[1+\Sigma_2(m-1)]$, for all $m\le2p$, then $p$ never divides any sum for all $m>2p$
\end{prop}
\begin{proof}
 If for all $m$, $m\le2p$, $p\nmid[1+\Sigma_2(m-1)]$, then $p\nmid[1+\Sigma_2(2p-1)]$ i.e. $1+\Sigma_2(2p-1)=\mathcal{M}\cdot p+r$, with $1\le r<p$.

  Let $m=2p+1$,
  \begin{multline*}
    1+\Sigma_2(m-1)=1+\Sigma_2(2p-1)+2p!!\\
    =\mathcal{M}\cdot p+r+2\cdot4\cdots(p-1(p+1)\cdots2p\\
    =[\mathcal{M}+2\cdot4\cdots(p-1)(p+1)\cdots2]p+r=\mathcal{M}\cdot p+r~,
  \end{multline*}
  then $1+\Sigma_2(2p)=\mathcal{M}p+r$, with $1\le r<p$.

  Let $m=2p+2$ and using the above statement, we have that
  \begin{multline*}
    1+\Sigma_2(m-1)=1+\Sigma_2(2p)+(2p+1)!!\\
    =\mathcal{M}\cdot p+r+1\cdot3\cdots (p-2)p(p+2)\cdots(2p+1)\\
    =[\mathcal{M}+1\cdot3\cdots(p-2)(p+2)\cdots(2p+1)]p+r=\mathcal{M}\cdot p+r~,
  \end{multline*}
  then $1+\Sigma_2(2p+1)=\mathcal{M}p+r$, with $1\le r<p$.

  Through complete induction, it follows that $p\nmid\Sigma_2(m)$, for all $m>2p$.
\end{proof}

With commands $k:=1..25$ and $SK(2,prime_k)\rightarrow$ (using the program \ref{Program SK(k,p)}) are obtained f{}irst 25 values of the function $SK_2$:
\begin{multline}\label{Vector sk2}
  \begin{tabular}{|l|r|r|r|r|r|r|r|r|r|r|r|r|}
    \hline
    $p$       & 2 & 3 & 5 & 7 &   11 & 13 & 17 &   19 & 23 & 29 & 31 & 37 \\
    $SK_2(p)$ & 2 & 5 & 5 & 4 & $-1$ &  7 & 14 & $-1$ & 31 & 12 & 17 & 13 \\
    \hline
  \end{tabular}\\
  \begin{tabular}{|r|r|r|r|r|r|r|r|r|r|r|r|r|}
    \hline
      41 &   43 & 47 & 53 & 59 & 61 & 67 & 71 & 73 &   79 &   83 & 89 &  97 \\
    $-1$ & $-1$ & 20 & 43 & 40 &  8 & 17 & 50 & 17 & $-1$ & $-1$ & 46 & 121 \\
    \hline
  \end{tabular}~.
\end{multline}
If $SK_2(p)=-1$, then for $p$ function $SK_2$ is undef{}ined.

\subsection{Smarandache--Kurepa Function of Order 3}

We notation
\begin{equation}\label{SumaDeTripluFactoriali}
  \Sigma_3(n)=\sum_{k=1}^nk!!!~.
\end{equation}

With commands $n:=1..20$ and $\Sigma(3,n)\rightarrow$, (using the program \ref{Program Sigma(k,n)}) result f{}irst 20 values of the function $\Sigma_3$:
\begin{multline*}
  1, 3, 6, 10, 20, 38, 66, 146, 308, 588, 1468, 3412, 7052,\\
  19372, 48532, 106772, 316212, 841092, 1947652, 6136452~.
\end{multline*}
\begin{defn}
  The function $SK_3:\NP{2}\to\Ns$, $m=SK_3(p)$ is the smallest $m\in\Ns$ such that $p\mid[1+\Sigma_3(m-1)]$.
\end{defn}

\begin{thm}\label{Teorema1pentruSK3}
  Let $p\in\NP{5}$, then exists $k_1,k_2\in\set{0,1,\ldots,p-1}$ for which $\mod(3k_1+1,p)=0$ and $\mod(3k_2+1,p)=0$.
\end{thm}
\begin{proof}
 If $p\in\NP{5}$, then $M =\set{0, 1, 2, ..., p-1}$ is a complete system of residual classes $\md{p}$. Let $q$ be a relative prime with $p$, \big(i.e. $gcd(p,q)=1$\big), then $q\cdot M$ is also a complete system of residual classes $\md{p}$ and $M=q\cdot M$, \cite[T 1.14]{Smarandache1999}.

 Because 3 is relative prime with $p$, then there exists $k_1\in 3\cdot M$ such that $\mod(3k_1,p-1)=0$, i.e. $\mod(3k_1+1,p)=0$. Also, there exists $k_2\in3\cdot M$ such that $\mod(3k_2,p-2)=0$, i.e. $\mod(3k_2+2,p)=0$.
\end{proof}

\begin{prop}\label{Propozitie pentru SK3}
  If $p\nmid[1+\Sigma_3(m-1)]$, for all $m\le 3p$, then $p$ never divides any sum for all $m>3p$~.
\end{prop}
\begin{proof}
  If for all $m$, $m\le3p$, $p\nmid[1+\Sigma_3(m-1)]$, then $p\nmid[1+\Sigma_3(3p-1)]$ i.e. $1+\Sigma_3(3p-1)=\mathcal{M}\cdot p+r$, with $1\le r<p$.

  Let $m=3p+1$,
  \begin{multline*}
    1+\Sigma_3(m-1)=1+\Sigma_3(3p-1)+3p!!!=\mathcal{M}\cdot p+r+3\cdot6\cdots3(p-1)\cdot3p\\
    =[\mathcal{M}+3\cdot6\cdots3(p-1)\cdot3]p+r=\mathcal{M}\cdot p+r~,
  \end{multline*}
  then $1+\Sigma_3(3p)=\mathcal{M}p+r$, with $1\le r<p$.

  Let $m=3p+2$ and using the above statement, we have that
  \begin{multline*}
    1+\Sigma_3(m-1)=1+\Sigma_3(3p)+(3p+1)!!!\\
    =\mathcal{M}\cdot p+r+1\cdot4\cdots\alpha p\cdots(3p+1)\\
    =[\mathcal{M}+1\cdot4\cdots\alpha\cdots(3p+1)]p+r=\mathcal{M}\cdot p+r~,
  \end{multline*}
  because exists $k$, according to Theorem \ref{Teorema1pentruSK3}, $k\in\set{0,1,\ldots,p-1}$, for which $3k+1=\alpha p$, then $1+\Sigma_3(3p+1)=\mathcal{M}p+r$, with $1\le r<p$.

  Let $m=3p+3$ and using the above statement, we have that
  \begin{multline*}
    1+\Sigma_3(m-1)=1+\Sigma_3(3p+1)+(3p+2)!!!\\
    =\mathcal{M}\cdot p+r+2\cdot5\cdots\alpha p\cdots(3p+2)\\
    =[\mathcal{M}+2\cdot5\cdots\alpha\cdots(3p+2)]p+r=\mathcal{M}\cdot p+r~,
  \end{multline*}
  because exists $k$, according to Theorem \ref{Teorema1pentruSK3}, $k\in\set{0,1,\ldots,p-1}$, for which $3k+2=\alpha p$, then $1+\Sigma_3(3p+2)=\mathcal{M}p+r$, with $1\le r<p$.

  Through complete induction, it follows that $p\nmid\Sigma_3(m-1)$, for all $m>3p$.
\end{proof}

With commands $k:=1..25$ and $SK(3,\emph{prime}_k)\rightarrow$ (using the program \ref{Program SK(k,p)}) are obtained f{}irst 25 values of the function $SK_3$:
\begin{multline}\label{Vector sk3}
  \begin{tabular}{|l|r|r|r|r|r|r|r|r|r|r|r|r|}
    \hline
    $p$       & 2 & 3 &    5 & 7 & 11 & 13 & 17 & 19 & 23 & 29 & 31 & 37 \\
    $SK_3(p)$ & 2 & 6 & $-1$ & 4 &  5 &  7 & 22 & 11 & 61 & 70 & 11 & 55 \\
    \hline
  \end{tabular}\\
  \begin{tabular}{|r|r|r|r|r|r|r|r|r|r|r|r|r|}
    \hline
    41 & 43 & 47 &   53 & 59 &  61 & 67 & 71 & 73 & 79 & 83 & 89 & 97 \\
    80 & 32 & 29 &  154 & 24 & 145 &  8 & 98 & 21 & 30 & 24 & 22 & 90 \\
    \hline
  \end{tabular}~.
\end{multline}
If $SK_3(p)=-1$, then for $p$ function $SK_3$ is undef{}ined.

\section{Smarandache--Wagstaf{}f Functions}

\subsection{Smarandache--Wagstaf{}f Function of Order 1}

\begin{defn}[\citep{Mudge1996,Mudge1996a,Ashbacher1997}]
  The function $SW_1:\NP{2}\to\Ns$, $m=SW_1(p)$ is the smallest $m\in\Ns$ such that $p\mid\Sigma_1(m)$, where $\Sigma_1(m)$ is def{}ined by (\ref{SumaDeFactoriali}).
\end{defn}

\begin{prop}
  If $p\nmid\Sigma_1(m)$, for all $m<p$, then $p$ never divides any sum for all $m\in\Ns$.
\end{prop}
\begin{proof}
  If for all $m$, $m<p$, $p\nmid\Sigma_1(m)$, then $p\nmid\Sigma_1(p-1)$ i.e. $\Sigma_1(p-1)=\mathcal{M}\cdot p+r$, with $1\le r<p$. Let $m\ge p$,
  \begin{multline*}
    \Sigma_1(m)=\Sigma_1(p-1)+p!+(p+1)!+\ldots+m!\\
    =\mathcal{M}\cdot p+r+(p-1)![1+(p+1)+\ldots+(p+1)(p+2)\cdots m]p\\
    =[\mathcal{M}+(p-1)!\big(1+(p+1)+\ldots+(p+1)(p+2)\cdots m\big)]p+r\\
    =\mathcal{M}\cdot p+r~,
  \end{multline*}
  then one obtains that $p\nmid\Sigma(m)$, for all $m$, $m\ge p$.
\end{proof}

\begin{prog}\label{Program SW(k,p)} for calculating the values of functions $SW_1$, $SW_2$ and $SW_3$.
  \begin{tabbing}
    $\emph{SW}(k,p):=$\=\vline\ $f$\=$or\ m\in2..k\cdot p-1$\\
    \>\vline\ \>\ $\emph{return}\ m\ \ \emph{if}\ \ \mod(\Sigma(k,m),p)\textbf{=}0$\\
    \>\vline\ $\emph{return}\ -1$\\
  \end{tabbing}
  The program uses the subprograms $\Sigma$, \ref{Program Sigma(k,n)}, and the utilitarian function Mathcad $mod$.
\end{prog}

With commands $k:=1..25$ and $SW(1,prime_k)\rightarrow$ (using the program \ref{Program SW(k,p)}) are obtained f{}irst 25 values of the function $SW_1$, \citep{WeissteinSmarandache-WagstaffFunction}:
\begin{multline}\label{Vector sw1}
  \begin{tabular}{|l|r|r|r|r|r|r|r|r|r|r|r|r|}
    \hline
    $p$       & 2    & 3 &    5 &    7 & 11 &   13 & 17 &   19 & 23 & 29 &   31 & 37 \\
    $SW_1(p)$ & $-1$ & 2 & $-1$ & $-1$ &  4 & $-1$ &  5 & $-1$ & 12 & 19 & $-1$ & 24 \\
    \hline
  \end{tabular}\\
  \begin{tabular}{|r|r|r|r|r|r|r|r|r|r|r|r|r|}
    \hline
    41 & 43 &   47 & 53 &   59 &   61 & 67 &   71 & 73 & 79 &   83 &   89 & 97 \\
    32 & 19 & $-1$ & 20 & $-1$ & $-1$ & 20 & $-1$ &  7 & 57 & $-1$ & $-1$ &  6 \\
    \hline
  \end{tabular}
\end{multline}
If $SW_1(p)=-1$, then for $p$ function $SW_1$ is undef{}ined.

\subsection{Smarandache--Wagstaf{}f Function of Order 2}

\begin{defn}
  The function $SW_2:\NP{2}\to\Ns$, $m=SW_2(p)$ is the smallest $m\in\Ns$ such that $p\mid\Sigma_2(m)$, where $\Sigma_2(m)$ is def{}ined by (\ref{SumaDeDubluFactoriali}).
\end{defn}

\begin{prop}
  If $p\nmid\Sigma_2(m)$, for all $m<2p$, then $p$ never divides any sum for all $m\in\Ns$.
\end{prop}
\begin{proof}
  If for all $m$, $m<2p$, $p\nmid\Sigma_2(m)$, then $p\nmid\Sigma_2(2p-1)$ i.e. $\Sigma_2(2p-1)=\mathcal{M}\cdot p+r$, with $1\le r<p$.

  Let $m=2p$,
  \begin{multline*}
    \Sigma_2(m)=\Sigma_2(2p-1)+2p!!\\
    =\mathcal{M}\cdot p+r+2\cdot4\cdots(p-1)(p+1)\cdots2p\\
    =[\mathcal{M}+2\cdot4\cdots(p-1)(p+1)\cdots2(p-1)\cdot2]p+r\\
    =\mathcal{M}\cdot p+r~,
  \end{multline*}
  then $\Sigma_2(2p)=\mathcal{M}p+r$, with $1\le r<p$.

  Let $m=2p+1$ and using the above statement, we have that
  \begin{multline*}
    \Sigma_2(m)=\Sigma_2(2p)+(2p+1)!!\\
    =\mathcal{M}\cdot p+r+1\cdot3\cdots (p-2)\cdot p\cdot(p+2)\cdots(2p+1)\\
    =[\mathcal{M}+1\cdot3\cdots(p-2)(p+2)\cdots2(p+1)]p+r\\
    =\mathcal{M}\cdot p+r~,
  \end{multline*}
  then complete induction, it follows that $p\nmid\Sigma_2(m)$, for all $m$, $m\ge 2p$.
\end{proof}
With commands $k:=1..25$ and $SW(2,prime_k)\rightarrow$ (using the program \ref{Program SW(k,p)}) are obtained f{}irst 20 values of the function $SW_2$:
\begin{multline}\label{Vector sw2}
  \begin{tabular}{|l|r|r|r|r|r|r|r|r|r|r|r|r|}
    \hline
    $p$       & 2 & 3 &    5 &  7 & 11 & 13 &   17 & 19 & 23 & 29 & 31 & 37 \\
    $SW_2(p)$ & 3 & 2 & $-1$ &  4 &  6 &  7 & $-1$ & 12 & 34 &  5 & 26 & 52 \\
    \hline
  \end{tabular}\\
  \begin{tabular}{|r|r|r|r|r|r|r|r|r|r|r|r|r|}
    \hline
    41 &   43 & 47 & 53 & 59 &   61 & 67 &   71 &   73 & 79 & 83 &   89 & 97 \\
    36 & $-1$ & 43 & 23 & 88 & $-1$ & 21 & $-1$ & $-1$ & 59 & 48 & $-1$ & 67 \\
    \hline
  \end{tabular}~.
\end{multline}
If $SW_2(p)=-1$, then for $p$ function $SW_2$ is undef{}ined.

\subsection{Smarandache--Wagstaf{}f Function of Order 3}

\begin{defn}
  The function $SW_3:\NP{2}\to\Ns$, $m=SW_3(p)$ is the smallest $m\in\Ns$ such that $p\mid\Sigma_3(m)$, where $\Sigma_3(m)$ is def{}ined by (\ref{SumaDeTripluFactoriali}).
\end{defn}

\begin{prop}
  If $p\nmid\Sigma_3(m)$, for all $m<3p$, then $3p$ never divides any sum for all $m\in\Ns$.
\end{prop}
\begin{proof}
  If for all $m$, $m<3p$, $p\nmid\Sigma_3(m)$, then $p\nmid\Sigma_3(3p-1)$ i.e. $\Sigma_3(3p-1)=\mathcal{M}\cdot p+r$, with $1\le r<p$.

  Let $m=3p$,
  \begin{multline*}
    \Sigma_3(m)=\Sigma_3(3p-1)+3p!!!=\mathcal{M}\cdot p+r+3\cdot6\cdots3(p-1)3p\\
    =[\mathcal{M}+3\cdot6\cdots3(p-1)\cdot3]p+r=\mathcal{M}\cdot p+r~,
  \end{multline*}
  then $\Sigma_3(3p)=\mathcal{M}p+r$, with $1\le r<p$.

  Let $m=3p+1$ and using the above statement, we have that
  \begin{multline*}
    \Sigma_3(m)=\Sigma_3(3p)+(3p+1)!!!\\
    =\mathcal{M}\cdot p+r+1\cdot4\cdots\alpha p\cdots(3p+1)\\
    =[\mathcal{M}+1\cdot4\cdots\alpha\cdots(3p+1)]p+r=\mathcal{M}\cdot p+r~,
  \end{multline*}
  because exists $k$, according to Theorem \ref{Teorema1pentruSK3}, $k\in\set{0,1,\ldots,p-1}$, for which $3k+1=\alpha p$, then $\Sigma_3(3p+1)=\mathcal{M}p+r$, with $1\le r<p$.

   Let $m=3p+2$ and using the above statement, we have that
  \begin{multline*}
    \Sigma_3(m)=\Sigma_3(3p)+(3p+2)!!!\\
    =\mathcal{M}\cdot p+r+2\cdot5\cdots\alpha p\cdots(3p+2)\\
    =[\mathcal{M}+2\cdot5\cdots\alpha\cdots(3p+2)]p+r=\mathcal{M}\cdot p+r~,
  \end{multline*}
   because exists, according to Theorem \ref{Teorema1pentruSK3}, $k\in\set{0,1,\ldots,p-1}$, for which $3k+2=\alpha p$, then $\Sigma_3(3p+2)=\mathcal{M}p+r$, with $1\le r<p$.

  Through complete induction, it follows that $p\nmid\Sigma_3(m)$, for all $m$, $m\ge 3p$.
\end{proof}

With commands $k:=1..25$ and $SW(3,prime_k)\rightarrow$ (using the program \ref{Program SW(k,p)}) are obtained f{}irst 25 values of the function $SW_3$:
\begin{multline}\label{Vector sw3}
  \begin{tabular}{|l|r|r|r|r|r|r|r|r|r|r|r|r|}
    \hline
    $p$       & 2 & 3 & 5 & 7 & 11 & 13 & 17 & 19 &   23 & 29 & 31 &   37 \\
    $SW_3(p)$ & 3 & 2 & 4 & 9 &  7 & 17 & 18 &  6 & $-1$ & 14 & 18 & $-1$ \\
    \hline
  \end{tabular}\\
  \begin{tabular}{|r|r|r|r|r|r|r|r|r|r|r|r|r|}
    \hline
    41 & 43 & 47 &   53 & 59 & 61 & 67 & 71 & 73 & 79 &   83 &  89 & 97 \\
    13 & 13 & 73 & $-1$ & 40 & 49 & 37 & 55 &  8 & 73 & $-1$ & 132 & 72 \\
    \hline
  \end{tabular}~.
\end{multline}
If $SW_3(p)=-1$, then for $p$ function $SW_3$ is undef{}ined.

\section[Smarandache Near to k--Primorial Functions]{Smarandache Near to k--Primorial \\Functions}

\subsection{Smarandache Near to Primorial Function}

Let be the function $\emph{SNtP}:\Ns\to\NP{2}\cup\set{1}$.
\begin{defn}\label{Definitia SNtP}
  The number $p=\emph{SNtP}(n)$ is the smallest prime, $p\le n$, such that $mod[p\#-1,n]=0\vee mod[p\#,n]=0\vee mod[p\#+1,n]=0$, where $p\#$ is the \emph{primorial} of $p$, given by \ref{Primorial}~.
\end{defn}
\cite{Ashbacher1996}\index{Ashbacher C.} shows that $\emph{SNtP}(n)$ only exists, \citep{WeissteinSmarandacheNear-to-PrimorialFunction}.

\begin{prog}\label{Program SNtkP} for calculating the function $SNtkP$.
  \begin{tabbing}
    $\emph{SNtkP}(n,k):=$\=\vline\ $\emph{return}\ 1\ \ \emph{if}\ \ n\textbf{=}1$\\
    \>\vline\ $m\leftarrow1$\\
    \>\vline\ $w$\=$hile\ \ \emph{prime}_m\le k\cdot n$\\
    \>\vline\ \>\vline\ $kp\leftarrow kP(prime_m,k)$\\
    \>\vline\ \>\vline\ $\emph{return}\ \emph{prime}_m\ \ \emph{if}\ \ \mod(kp,n)\textbf{=}0$\\
    \>\vline\ \>\vline\ $\emph{return}\ \emph{prime}_m\ \ \emph{if}\ \ \mod(kp-1,n)\textbf{=}0$\\
    \>\vline\ \>\vline\ $\emph{return}\ \emph{prime}_m\ \ \emph{if}\ \ \mod(kp+1,n)\textbf{=}0$\\
    \>\vline\ $\emph{return}\ \ -1$\\
  \end{tabbing}
\end{prog}
For $n=1$, 2, \ldots, 45 the f{}irst few values of $\emph{SNtP}(n)=\emph{SNtkP}(n,1)$ are: 1, 2, 2, -1, 3, 3, 3, -1, -1, 5, 7, -1, 13, 7, 5, 43, 17, 47, 7, 47, 7, 11, 23, 47, 47, 13, 43, 47, 5, 5, 5, 47, 11, 17, 7, 47, 23, 19, 13, 47, 41, 7, 43, 47, 47~.
If $\emph{SNtkP}(n,1)=-1$, then for $n$ function $\emph{SNtkP}$ is undef{}ined.

For examples $\emph{SNtkP}(4)=-1$ because $4\nmid(2\#-1)=1$, $4\nmid2\#=2$, $4\nmid(2\#+1)=3$, $4\nmid(3\#-1)=5$, $4\nmid3\#=6$, $4\nmid(3\#+1)=7$.

\subsection{Smarandache Near to Double Primorial Function}

Let be the function $\emph{SNtDP}:\Ns\to\NP{2}\cup\set{1}$.
\begin{defn}\label{Definitia SNtDP}
  The number $p=\emph{SNtDP}(n)$ is the smallest prime, $p\le2n$, such that $\mod(p\#\#-1,n)=0\vee \mod(p\#\#,n)=0\vee \mod(p\#\#+1,n)=0$, where $p\#\#$ is the \emph{double primorial} of $p$, given by \ref{Double Primorial}~.
\end{defn}

For $n=1$, 2, \ldots, 45 the f{}irst few values of $\emph{SNtDP}(n)=\emph{SNtkP}(n,2)$, \ref{Program SNtkP}, are: 2, 2, 2, 3, 5, -1, 7, 13, 5, 5, 5, 83, 13, 83, 83, 13, 13, 83, 19, 7, 7, 7, 23, 83, 37, 83, 23, 83, 29, 83, 31, 83, 89, 13, 83, 83, 11, 97, 13, 71, 23, 83, 43, 89, 89~. If $\emph{SNtkP}(n,2)=-1$, then for $n$ function $\emph{SNtkP}$ is undef{}ined.

\subsection{Smarandache Near to Triple Primorial Function}

Let be the function $\emph{SNtTP}:\Ns\to\NP{2}\cup\set{1}$.
\begin{defn}\label{Definitia SNtTP}
  The number $p=\emph{SNtTP}(n)$ is the smallest prime, $p\le3n$, such that $\mod(p\#\#\#-1,n)=0\vee \mod(p\#\#\#,n)=0\vee \mod(p\#\#\#+1,n)=0$, where $p\#\#\#$ is the \emph{triple primorial} of $p$, given by \ref{Triplu Primorial}~.
\end{defn}

For $n=1$, 2, \ldots, 40 the f{}irst few values of $\emph{SNtTP}(n)=\emph{SNtkP}(n,3)$,  given by \ref{Program SNtkP}, are: 2, 2, 2, 3, 5, 5, 7, 11, 23, 43, 11, 89, 7, 7, 7, 11, 11, 23, 19, 71, 37, 13, 23, 89, 71, 127, 97, 59, 29, 127, 31, 11, 11, 11, 127, 113, 37, 103, 29, 131, 41, 37, 31, 23, 131~. If $\emph{SNtkP}(n,3)=-1$, then for $n$ function $\emph{SNtkP}$ is undef{}ined.

We can generalize further this function as Smarandache Near to $k$-Pri\-mordial Function by using
\[
 p\underbrace{\#\ldots\#}_{k\ times}~,
\]
def{}ined analogously to \ref{Triplu Primorial}, instead of $p\#\#\#$. Alternatives to $\emph{SNtkP}(n)$ can be the following: $p\#\ldots\#$, or $p\#\ldots\#\pm1$, or $p\#\ldots\#\pm2$, or \ldots $p\#\ldots\#\pm s$ (where $s$ is a positive odd integer is a multiple of $n$).

\section{Smarandache Ceil Function}

Let the function $S_k:\Ns\to\Ns$.
\begin{defn}\label{Definitia Sk}
  The number $m=S_k(n)$ is the smallest $m\in\Ns$ such that $n\mid m^k$.
\end{defn}

This function has been treated in the works \citep{Smarandache1993,Begay1997,Smarandache1997,WeissteinSmarandacheCeilFunction}.

\begin{prog}\label{Program Sk} for the function $S_k$.
  \begin{tabbing}
    $\emph{Sk}(n,k):=$\=\vline\ $f$\=$or\ m\in1..n$\\
    \>\vline\ \>\ $\emph{return}\ \ m\ \ \emph{if}\ \ \mod(m^k,n)\textbf{=}0$\\
  \end{tabbing}
\end{prog}

If $n:=1..100$ then:
\begin{itemize}\label{sk1-sk6}
  \item[] $Sk(n,1)\rightarrow$ 1, 2, 3, 4, 5, 6, 7, 8, 9, 10, 11, 12, 13, 14, 15, 16, 17, 18, 19, 20, 21, 22, 23, 24, 25, 26, 27, 28, 29, 30, 31, 32, 33, 34, 35, 36, 37, 38, 39, 40, 41, 42, 43, 44, 45, 46, 47, 48, 49, 50, 51, 52, 53, 54, 55, 56, 57, 58, 59, 60, 61, 62, 63, 64, 65, 66, 67, 68, 69, 70, 71, 72, 73, 74, 75, 76, 77, 78, 79, 80, 81, 82, 83, 84, 85, 86, 87, 88, 89, 90, 91, 92, 93, 94, 95, 96, 97, 98, 99, 100 \citep[A000027]{SloaneOEIS};
  \item[] $Sk(n,2)\rightarrow$ 1, 2, 3, 2, 5, 6, 7, 4, 3, 10, 11, 6, 13, 14, 15, 4, 17, 6, 19, 10, 21, 22, 23, 12, 5, 26, 9, 14, 29, 30, 31, 8, 33, 34, 35, 6, 37, 38, 39, 20, 41, 42, 43, 22, 15, 46, 47, 12, 7, 10, 51, 26, 53, 18, 55, 28, 57, 58, 59, 30, 61, 62, 21, 8, 65, 66, 67, 34, 69, 70, 71, 12, 73, 74, 15, 38, 77, 78, 79, 20, 9, 82, 83, 42, 85, 86, 87, 44, 89, 30, 91, 46, 93, 94, 95, 24, 97, 14, 33, 10 \citep[A019554]{SloaneOEIS};
  \item[] $Sk(n,3)\rightarrow$ 1, 2, 3, 2, 5, 6, 7, 2, 3, 10, 11, 6, 13, 14, 15, 4, 17, 6, 19, 10, 21, 22, 23, 6, 5, 26, 3, 14, 29, 30, 31, 4, 33, 34, 35, 6, 37, 38, 39, 10, 41, 42, 43, 22, 15, 46, 47, 12, 7, 10, 51, 26, 53, 6, 55, 14, 57, 58, 59, 30, 61, 62, 21, 4, 65, 66, 67, 34, 69, 70, 71, 6, 73, 74, 15, 38, 77, 78, 79, 20, 9, 82, 83, 42, 85, 86, 87, 22, 89, 30, 91, 46, 93, 94, 95, 12, 97, 14, 33, 10 \cite[A019555]{SloaneOEIS};
  \item[] $Sk(n,4)\rightarrow$ 1, 2, 3, 2, 5, 6, 7, 2, 3, 10, 11, 6, 13, 14, 15, 2, 17, 6, 19, 10, 21, 22, 23, 6, 5, 26, 3, 14, 29, 30, 31, 4, 33, 34, 35, 6, 37, 38, 39, 10, 41, 42, 43, 22, 15, 46, 47, 6, 7, 10, 51, 26, 53, 6, 55, 14, 57, 58, 59, 30, 61, 62, 21, 4, 65, 66, 67, 34, 69, 70, 71, 6, 73, 74, 15, 38, 77, 78, 79, 10, 3, 82, 83, 42, 85, 86, 87, 22, 89, 30, 91, 46, 93, 94, 95, 12, 97, 14, 33, 10 \cite[A053166 ]{SloaneOEIS};
  \item[] $Sk(n,5)\rightarrow$ 1, 2, 3, 2, 5, 6, 7, 2, 3, 10, 11, 6, 13, 14, 15, 2, 17, 6, 19, 10, 21, 22, 23, 6, 5, 26, 3, 14, 29, 30, 31, 2, 33, 34, 35, 6, 37, 38, 39, 10, 41, 42, 43, 22, 15, 46, 47, 6, 7, 10, 51, 26, 53, 6, 55, 14, 57, 58, 59, 30, 61, 62, 21, 4, 65, 66, 67, 34, 69, 70, 71, 6, 73, 74, 15, 38, 77, 78, 79, 10, 3, 82, 83, 42, 85, 86, 87, 22, 89, 30, 91, 46, 93, 94, 95, 6, 97, 14, 33, 10 \citep[A007947 ]{SloaneOEIS};
  \item[] $Sk(n,6)\rightarrow$ 1, 2, 3, 2, 5, 6, 7, 2, 3, 10, 11, 6, 13, 14, 15, 2, 17, 6, 19, 10, 21, 22, 23, 6, 5, 26, 3, 14, 29, 30, 31, 2, 33, 34, 35, 6, 37, 38, 39, 10, 41, 42, 43, 22, 15, 46, 47, 6, 7, 10, 51, 26, 53, 6, 55, 14, 57, 58, 59, 30, 61, 62, 21, 2, 65, 66, 67, 34, 69, 70, 71, 6, 73, 74, 15, 38, 77, 78, 79, 10, 3, 82, 83, 42, 85, 86, 87, 22, 89, 30, 91, 46, 93, 94, 95, 6, 97, 14, 33, 10;
\end{itemize}

\section{Smarandache--Mersenne Functions}

\subsection{Smarandache--Mersenne Left Function}

Let the function $\emph{SML}:2\Na+1\to\Ns$, where $2\Na+1=\set{1,3,\ldots}$ is the set of natural numbers odd.
\begin{defn}\label{Functia SML}
  The number $m=\emph{SML}(\omega)$ is the smallest $m\in\Ns$ such that $\omega\mid 2^m-1$.
\end{defn}
\begin{prog}\label{Program SML} for generating the values of function $\emph{SML}$.
   \begin{tabbing}
     $\emph{SML}(\omega):=$\=\vline\ $f$\=$or\ m=1..\omega$\\
     \>\vline\ \>\ $\emph{return}\ \ m\ \ \emph{if}\ \ \mod(2^m-1,\omega)\textbf{=}0$\\
     \>\vline\ $\emph{return}\ \ -1$\\
   \end{tabbing}
   If $\emph{SML}(\omega)=-1$, then for $\omega$ the function $\emph{SML}$ is undef{}ined.
\end{prog}
If $n:=1..40$ then:
\begin{itemize}\label{SML(prime)+SML(2n-1)}
  \item[] $SML(\emph{prime}_n)\rightarrow$ $-1$, 2, 4, 3, 10, 12, 8, 18, 11, 28, 5, 36, 20, 14, 23, 52, 58, 60, 66, 35, 9, 39, 82, 11, 48, 100, 51, 106, 36, 28, 7, 130, 68, 138, 148, 15, 52, 162, 83, 172,
  \item[] $SML(2n-1)\rightarrow$: 1, 2, 4, 3, 6, 10, 12, 4, 8, 18, 6, 11, 20, 18, 28, 5, 10, 12, 36, 12, 20, 14, 12, 23, 21, 8, 52, 20, 18, 58, 60, 6, 12, 66, 22, 35, 9, 20, 30, 39~.
\end{itemize}

\subsection{Smarandache--Mersenne Right Function}

Let the function $\emph{SMR}:2\Na+1\to\Ns$, where $2\Na+1=\set{1,3,\ldots}$ is the set of natural numbers odd.
\begin{defn}\label{Functia SMR}
  The number $m=\emph{SMR}(\omega)$ is the smallest $m\in\Ns$ such that $\omega\mid 2^m+1$.
\end{defn}
\begin{prog}\label{Program SMR} for generating the values of function $\emph{SMR}$.
   \begin{tabbing}
     $\emph{SMR}(\omega):=$\=\vline\ $f$\=$or\ m=1..\omega$\\
     \>\vline\ \>\ $\emph{return}\ m\ \emph{if}\ \mod(2^m+1,\omega)=0$\\
     \>\vline\ $\emph{return}\ -1$\\
   \end{tabbing}
   If $\emph{SMR}(\omega)=-1$, then for $\omega$ the function $\emph{SMR}$ is undef{}ined.
\end{prog}
If $n:=1..40$ then:
\begin{itemize}\label{SMR(prime)+SMR(2n-1)}
  \item[] $\emph{SMR}(prime_n)\rightarrow$ $-1$, 1, 2, $-1$, 5, 6, 4, 9, $-1$, 14, $-1$, 18, 10, 7, $-1$, 26, 29, 30, 33, $-1$, $-1$, $-1$, 41, $-1$, 24, 50, $-1$, 53, 18, 14, $-1$, 65, 34, 69, 74, $-1$, 26, 81, $-1$, 86,
  \item[] $\emph{SMR}(2n-1)\rightarrow$ 1, 1, 2, $-1$, 3, 5, 6, $-1$, 4, 9, $-1$, $-1$, 10, 9, 14, $-1$, 5, $-1$, 18, $-1$, 10, 7, $-1$, $-1$, $-1$, $-1$, 26, $-1$, 9, 29, 30, $-1$, 6, 33, $-1$, $-1$, $-1$, $-1$, $-1$, -1~.
\end{itemize}

\section{Smarandache--X-nacci Functions}

\subsection{Smarandache--Fibonacci Function}

Let the function $\emph{SF}:\Ns\to\Ns$ and Fibonacci sequence def{}ined by formula $f_1:=1$, $f_2:=1$, $k=1,2,\ldots, 120$, $f_{k+2}:=f_{k+1}+f_k$.
\begin{defn}\label{Functia SF}
  The number $m=\emph{SF}(n)$ is the smallest $m\in\Ns$ such that $n\mid f_m$.
\end{defn}
\begin{prog}\label{Program SF} for generating the values of function $\emph{SF}$.
   \begin{tabbing}
     $\emph{SF}(n):=$\=\vline\ $f$\=$or\ \ m\in1..\emph{last}(f)$\\
     \>\vline\ \>\ $\emph{return}\ \ m\ \ \emph{if}\ \ \mod(f_m,n)=0$\\
     \>\vline\ $\emph{return}\ \ -1$\\
   \end{tabbing}
   If $\emph{SF}(n)=-1$, then for $n$ the function $\emph{SF}$ is undef{}ined for $last(f)=120$.
\end{prog}
If $n:=1..80$ then $\emph{SF}(n)^\textrm{T}\rightarrow$ 1, 3, 4, 6, 5, 12, 8, 6, 12, 15, 10, 12, 7, 24, 20, 12, 9, 12, 18, 30, 8, 30, 24, 12, 25, 21, 36, 24, 14, 60, 30, 24, 20, 9, 40, 12, 19, 18, 28, 30, 20, 24, 44, 30, 60, 24, 16, 12, 56, 75, 36, 42, 27, 36, 10, 24, 36, 42, 58, 60, 15, 30, 24, 48, 35, 60, 68, 18, 24, 120, 70, 12, 37, 57, 100, 18, 40, 84, 78, 60~.

\subsection{Smarandache--Tribonacci Function}

Let the function $\emph{STr}:\Ns\to\Ns$ and Tribonacci sequence def{}ined by formula $t_1:=1$, $t_2:=1$, $t_3:=2$, $k=1,2,\ldots, 130$, $t_{k+3}:=t_{k+2}+t_{k+1}+t_k$.
\begin{defn}\label{Functia STr}
  The number $m=\emph{STr}(n)$ is the smallest $m\in\Ns$ such that $n\mid t_m$.
\end{defn}
\begin{prog}\label{Program STr} for generating the values of function $\emph{STr}$.
   \begin{tabbing}
     $\emph{STr}(n):=$\=\vline\ $f$\=$or\ m=1..\emph{last}(t)$\\
     \>\vline\ \>\ $\emph{return}\ \ m\ \ \emph{if}\ \ \mod(t_m,n)=0$\\
     \>\vline\ $\emph{return}\ \ -1$\\
   \end{tabbing}
   If $\emph{STr}(n)=-1$, then for $n$ the function $\emph{STr}$ is undef{}ined for $last(t)=100$.
\end{prog}
If $n:=1..80$ then $\emph{STr}(n)^\textrm{T}\rightarrow$ 1, 3, 7, 4, 14, 7, 5, 7, 9, 19, 8, 7, 6, 12, 52, 15, 28, 12, 18, 31, 12, 8, 29, 7, 30, 39, 9, 12, 77, 52, 14, 15, 35, 28, 21, 12, 19, 28, 39, 31, 35, 12, 82, 8, 52, 55, 29, 64, 15, 52, 124, 39, 33, 35, 14, 12, 103, 123, 64, 52, 68, 60, 12, 15, 52, 35, 100, 28, 117, 31, 132, 12, 31, 19, 52, 28, 37, 39, 18, 31~.

\subsection{Smarandache--Tetranacci Function}

Let the function $\emph{STe}:\Ns\to\Ns$ and Tetranacci sequence def{}ined by formula $T_1:=1$, $T_2:=1$, $T_3:=2$, $T_4:=4$, $k=1,2,\ldots, 300$, $T_{k+4}:=T_{k+3}+T_{k+2}+T_{k+1}+T_k$.
\begin{defn}\label{Functia STe}
  The number $m=\emph{STe}(n)$ is the smallest $m\in\Ns$ such that $n\mid T_m$.
\end{defn}
\begin{prog}\label{Program STe} for generating the values of function $\emph{STe}$.
   \begin{tabbing}
     $\emph{STe}(n):=$\=\vline\ $f$\=$or\ m=1..\emph{last}(T)$\\
     \>\vline\ \>\ $\emph{return}\ m\ \emph{if}\ \mod(T_m,n)=0$\\
     \>\vline\ $\emph{return}\ -1$\\
   \end{tabbing}
   If $\emph{STe}(n)=-1$, then for $n$ the function $\emph{STe}$ is undef{}ined for $last(T)=300$.
\end{prog}
If $n:=1..80$ then $\emph{STe}(n)^\textrm{T}\rightarrow$ 1, 3, 6, 4, 6, 9, 8, 5, 9, 13, 20, 9, 10, 8, 6, 10, 53, 9, 48, 28, 18, 20, 35, 18, 76, 10, 9, 8, 7, 68, 20, 15, 20, 53, 30, 9, 58, 48, 78, 28, 19, 18, 63, 20, 68, 35, 28, 18, 46, 108, 76, 10, 158, 9, 52, 8, 87, 133, 18, 68, 51, 20, 46, 35, 78, 20, 17, 138, 35, 30, 230, 20, 72, 58, 76, 48, 118, 78, 303, 30.

And so on, one can def{}ine the Smarandache--N-nacci function, where N-nacci sequence is 1, 1, 2, 4, 8, \ldots and $N_{n+k}=N_{n+k-1}+N_{n+k-2}+\ldots+N_n$ is the sum of the previous n terms. Then, the number $m=\emph{SN}(n)$ is the smallest m such that $n\mid N_m$.

\section{Pseudo--Smarandache Functions}

The functions in this section are similar to Smarandache $S$ function, \ref{ProgramS}, \citep{Smarandache1980,Cira+Smarandache2014}. The f{}irst authors who dealt with the def{}inition and properties of the pseudo--Smarandache function of f{}irst rank are: \cite{Ashbacher1995}\index{Ashbacher C.} and \cite{Kashihara1996}\index{Kashihara K.}, \citep{WeissteinPseudosmarandacheFunction}.

\subsection{Pseudo--Smarandache Function of the Order 1}

Let $n$ be a natural positive number and function $Z_1:\Ns\to\Ns$.
\begin{defn}\label{DefinitiaFunctieiZ1}
  The value $Z_1(n)$ is the smallest natural number $m=Z_1(n)$ for which the sum $1+2+\ldots+m$ divides by $n$.
\end{defn}

Considering that $1+2+\ldots+m=m(m+1)/2$ this def{}inition of the function $Z_1$ is equivalent with the fact that $m=Z_1(n)$ is the smallest natural number $n$ for which we have $m(m+1)=\mathcal{M}\cdot 2n$ i.e. $m(m+1)$ is multiple of $2n$ \big(or the equivalent relation $2n\mid m(m+1)$ i.e. $2n$ divides $m(m+1)$\big).

\begin{lem}\label{L1pentruZ1} Let $n,m\in\Ns$, $n\ge m$, if $n\mid[m(m+1)]/2$, then $m\ge\lceil s_1(n)\rceil$, where
 \begin{equation}\label{Solutia s1(n)}
  s_1(n):=\frac{\sqrt{8n+1}-1}{2}~.
 \end{equation}
\end{lem}
\begin{proof}
  The relation $n\mid[m(m+1)]/2$ is equivalent with $m(m+1)=\mathcal{M}\cdot2n$, with $\mathcal{M}=1,2,\ldots$~. The smallest multiplicity is for $\mathcal{M}=1$. The equation $m(m+1)=2n$ has as positive real solution $s_1(n)$ given by (\ref{Solutia s1(n)}). Considering that $m$ is a natural number, it follows that $m\ge\lceil s_1(n)\rceil$.
\end{proof}

The bound of the function $Z_1$ (see Figure \ref{FigFunctiaZ1}) is given by the theorem:
\begin{thm}\label{T1Kashihara}
  For any $n\in\Ns$ we have $\lceil s_1(n)\rceil\le Z_1(n)\le2n-1$.
\end{thm}
\begin{proof}
 T he inequality $\lceil s_1(n)\rceil\le Z_1(n)$, for any $n\in\Ns$ follows from Lemma \ref{L1pentruZ1}.

 The relation $n\mid[m(m+1)]/2$ is equivalent with $2n\mid[m(m+1)]$. Of the two factors of expression $m(m+1)$, in a sequential ascending scroll, f{}irst with value $2n$ is $m+1$. It follows that for  $m=2n-1$ we f{}irst met the condition $n\mid [m(m+1)/2]$, i.e. $m(m+1)/2n=(2n-1)2n/2n=2n-1\in\Ns$.
\end{proof}

\begin{thm}\label{T2Kashihara}
  For any $k\in\Ns$, it follows that  $Z_1(2^k)=2^{k+1}-1$.
\end{thm}
\begin{proof}
  We use the notation $Z_1(n)=m$. If $n=2^k$, we calculate $m(m+1)/(2n)$ for $m=2^{k+1}-1$.
  \[
   \frac{m(m+1)}{2}=\frac{(2^{k+1}-1)2^{k+1}}{2\cdot2^k}=2^{k+1}-1\in\Ns~.
  \]

  Let us prove that $m=2^{k+1}-1$ is the smallest $m$ for which $m(m+1)/(2n)\in\Ns$. It is obvious that $m$ has to be of form $2^\alpha-1$ sau $2^\alpha$, where $\alpha\in\Ns$, if we want $m(m+1)$ to divide by $2\cdot2^k$. Let $m=2^{k+1-j}-1$ with $j\in\Ns$, which is a number smaller than $2^{k+1}-1$. If we calculate
  \[
   \frac{m(m+1)}{2n}=\frac{(2^{k+1-j}-1)2^{k+1-j}}{2\cdot2^k}=(2^{k+1-j}-1)2^{-j}\notin\Ns~,
  \]
  therefore we can not have $Z_1(2^k)=m=2^{k+1-j}-1$, with $j\in\Ns$. Let $m=2^{k+1-j}$, with $j\in\Ns$, which is a number smaller than $2^{k+1}-1$. Calculating,
  \[
   \frac{m(m+1)}{2n}=\frac{2^{k+1-j}(2^{k+1-j}+1)}{2\cdot2^k}=2^{-j}(2^{k+1-j}+1)\notin\Ns~,
  \]
  therefore we can not have $Z_1(2^k)=m=2^{k+1-j}$, with $j\in\Ns$.

  It was proved that $m=2^{k+1}-1$ is the smallest number that has the property $n\mid m(m+1)/2$, therefore $Z_1(2^k)=2^{k+1}-1$.
\end{proof}

We present a theorem, \cite[T4, p. 36]{Kashihara1996}, on function values $Z_1$ for the powers of primes.
\begin{thm}\label{T3Kashihara}
  $Z_1(p^k)=p^k-1$ for any $p\in\NP{3}$ and $k\in\Ns$~.
\end{thm}
\begin{proof}
  From the sequential ascending completion of $m$,the f{}irst factor between $m$ and $m+1$, which divides $n=p^k$ is $m+1=p^k$. Then it follows that $m=p^k-1$. It can be proved by direct calculation that $m(m+1)/(2n)=(p^k-1)p^k/(2p^k)\in\Ns$ because $p^k-1$, since $p\in\NP{3}$, is always an even number. Therefore, $m=p^k-1$ is the smallest natural number for which $m(m+1)/2$ divides to $n=p^k$, then it follows that $Z_1(p^k)=p^k-1$ for any $p\in\NP{3}$.
\end{proof}
\begin{cor}{\emph{\cite[T3, p. 36]{Kashihara1996}}}\label{CorolarulT3Kashihara}
  For $k=1$ it follows that $Z_1(p)=p-1$ for any $p\in\NP{3}$.
\end{cor}

\begin{prog}\label{ProgramZ1} for the function $Z_1$.
  \begin{tabbing}
    $Z_1(n):=$\=\vline\ $\emph{return}\ n-1\ \ \emph{if}\ \ \emph{TS}(n)\textbf{=}1\wedge n>2$\\
    \>\vline\ $f$\=$or\ m\in \emph{ceil}(s_1(n))..n-1$\\
    \>\vline\ \>\ $\emph{return}\ m\ \emph{if}\ mod[m(m+1),2n]\textbf{=}0$\\
    \>\vline\ $\emph{return}\ 2n-1$\\
  \end{tabbing}
  Explanations for the program $Z_1$, for search of $m$ optimization.
  \begin{enumerate}
    \item The program $Z_1$ uses Smarandache primality test $\emph{TS}$, \ref{Program TS}. For $n\in\NP{3}$ the value of function $Z_1$ is $n-1$ according the Corollary \ref{CorolarulT3Kashihara}, without "searching" $m$ anymore, the value of the function $Z_1$, that fulf{}ills the condition $n\mid[m(m+1)/2]$.
    \item Searching for $m$ in program $Z_1$ starts from $\lceil s_1(n)\rceil$ according to the Theorem \ref{T1Kashihara}. Searching for $m$ ends when $m$ has at most value $n-1$.
    \item If the search for $m$ reached the value of $n-1$ and the condition $mod[m(m+1)/2,n]=0$ (i.e. the rest of division $m(m+1)/2$ to $n$ is 0) was fulf{}illed, then it follows that $n$ is of form $2^k$. Indeed, if $m=n-1$ and
        \[
         \frac{(n-1)(n-1+1)}{2n}=\frac{n-1}{2}\notin\Ns~,
        \]
        then it follows that $n$ is an even number, and that is $n=2q_1$. Calculating again
        \[
         \frac{(2q_1-1)(2q_1-1+1)}{2\cdot 2q_1}=\frac{2q_1-1}{2}\notin\Ns~,
        \]
        then it follows that $q_1$ is even, i.e. $q_1=2q_2$ and $n=2\cdot2q_2=2^2\cdot q_2$. After an identical reasoning, it follows that $n=2^3\cdot q_3$, and so on. But $n$ is a f{}inite number, then it follows that there exists k $k\in\Ns$ such that $n=2^k$. Therefore, according to the Theorem \ref{T2Kashihara} we have that $Z_1(n)=Z_1(2^k)=2^{k+1}-1=2n-1$, see Figure \ref{FigFunctiaZ1}.
  \end{enumerate}
\end{prog}

\begin{figure}[h]
  \centering
  \includegraphics[scale=0.7]{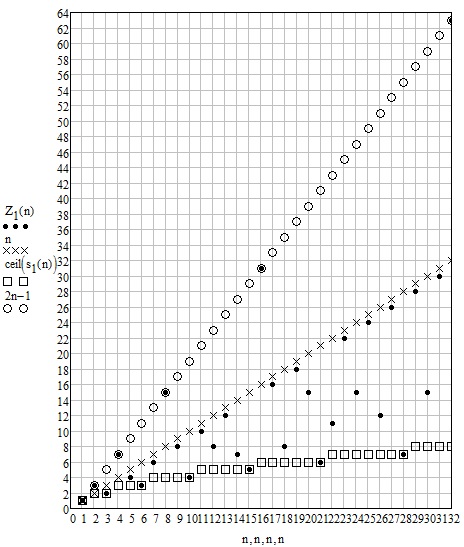}\\
  \caption{Function $Z_1$}\label{FigFunctiaZ1}
\end{figure}

\subsection{Pseudo--Smarandache Function of the Order 2}

We def{}ine the function $Z_2:\Ns\to\Ns$ and denote the value of the function $Z_2$ with $m$, i.e. $m=Z_2(n)$. The value of $m$ is the smallest natural number for which the sum $1^2+2^2+\ldots+m^2$ divides by $n$.

Considering that $1^2+2^2+\ldots+m^2=m(m+1)(2m+1)/6$ this def{}inition of the function $Z_2$ is equivalent with $m=Z_2(n)$ is the smallest natural number for which we have $m(m+1)(2m+1)/6=\mathcal{M}\cdot n$ i.e. $m(m+1)(2m+1)$ is multiple of $6n$ \big(or the equivalent relation $6n\mid m(m+1)(2m+1)$ i.e. $6n$ divides $m(m+1)(2m+1)$\big).

We consider the function $\tau$ given by the formula:
  \begin{equation}\label{FunctiaTau}
    \tau(n):=\sqrt[3]{3(108n+\sqrt{11664n^2-3})}~,
  \end{equation}
the real solution of the equation $m(m+1)(2m+1)=6n$ is
\begin{equation}\label{Solutia s2(n)}
  s_2(n):=\frac{1}{2}\left(\frac{1}{\tau(n)}+\frac{\tau(n)}{3}-1\right)~.
\end{equation}

\begin{lem}\label{Lemma1pentruZ2} Let $n,m\in\Ns$, $n\ge m$, if $n\mid[m(m+1)(2m+1)]/6$, then $m\ge\lceil s_2(n)\rceil$, with $s_2(n)$ is given by \emph{(\ref{Solutia s2(n)})}.
\end{lem}
\begin{proof}
  The relation $n\mid[m(m+1)(2m+1)]/6$ $\Leftrightarrow$ $m(m+1)(2m+1)=\mathcal{M}\cdot6n$, with $\mathcal{M}=1,2,\ldots$. The smallest multiplicity is for $\mathcal{M}=1$. The equation $m(m+1)(2m+1)=6n$ has as real positive solution $s_2(n)$ given by (\ref{Solutia s2(n)}). Considering that $m$ is a natural number, it follows that $m\ge\lceil s_2(n)\rceil$.
\end{proof}

\begin{lem}\label{Lemma2pentruZ2}
  The number $(2^{k+2}-1)(2^{k+1}-1)$ is multiple of 3 for any $k\in\Ns$.
\end{lem}
\begin{proof} Let us observe that for:
  \begin{itemize}
    \item $k=1$, $2^{k+1}-1=2^2-1=3$ and $2^{k+2}-1=2^3-1=7$,
    \item $k=2$, $2^{k+1}-1=2^3-1=7$ and $2^{k+2}-1=2^4-1=15$,
    \item \ldots,
    \item $k=2j-1$, $2^{2j}-1=3\cdot\mathcal{M}$ and $2^{2j+1}-1=?$,
    \item $k=2j$, $2^{2j+1}-1=?$ and $2^{2(j+1)}-1=3\cdot\mathcal{M}$,
    \item and so on.
  \end{itemize}
 We can say that the proof of lemma is equivalent with proof{}ing the fact that $2^{2j}-1$ is multiple of $3$ for anye $j\in\Ns$.

We make the proof by full induction.
\begin{itemize}
  \item For $j=1$ we have $2^2-1=3$, is multiple of $3$.
  \item We suppose that $2^{2j}-1=3\cdot\mathcal{M}$.
  \item Then we show that $2^{2(j+1)}-1=3\cdot\mathcal{M}$. Indeed
\begin{multline*}
  2^{2(j+1)}-1=2^2\cdot2^{2j}-1=2^2(\cdot2^{2j}-1)+2^2-1=2^2\cdot3\mathcal{M}+3\\
  =3\cdot(2^2\mathcal{M}+1)=3\cdot\mathcal{M}~.
\end{multline*}
\end{itemize}
\end{proof}

If $k=1,2,\ldots,10$, then
  \[(2^{k+2}-1)(2^{k+1}-1)=\left(\begin{array}{c}
                                   3\cdot7\\
                                   3\cdot5\cdot7\\
                                   3\cdot5\cdot31\\
                                   3^2\cdot7\cdot31\\
                                   3^2\cdot7\cdot127\\
                                   3\cdot5\cdot17\cdot127\\
                                   3\cdot5\cdot7\cdot17\cdot73\\
                                   3\cdot7\cdot11\cdot31\cdot73\\
                                   3\cdot11\cdot23\cdot31\cdot89\\
                                   3^2\cdot5\cdot7\cdot13\cdot23\cdot89\\
                                 \end{array}\right)~.
  \]

The bound of the function $Z_2$ (see Figure \ref{FigFunctiaZ2}) is given by the theorem:
\begin{thm}\label{T2Z2}
  For any $n\in\Ns$ we have $\lceil s_2(n)\rceil\le Z_2(n)\le2n-1$.
\end{thm}
\begin{proof}
 The inequality $\lceil s_2(n)\rceil\le Z_2(n)$, for any $n\in\Ns$ results from Lemma \ref{Lemma1pentruZ2}. If $n=2^k$ $m=2^{k+1}-1$, then
 \[
   \frac{m(m+1)(2m+1)}{6n}=\frac{(2^{k+1}-1)2^{k+1}(2^{k+2}-2+1)}{6\cdot2^k}=\frac{(2^{k+1}-1)(2^{k+2}-1)}{3}~,
 \]
 by according to Lemma \ref{Lemma2pentruZ2} the number $(2^{k+1}-1)(2^{k+2}-1)$ is multiple of $3$. Then it results that $6n\mid m(m+1)(2m+1)$, therefore we can say that $Z_2(n)=2n-1$, if $n=2^k$, for any $k\in\Ns$.

 Let us prove that if $Z_2(n)=2n-1$, then $n=2^k$. If $Z_2(n)=2n-1$, then it results that $6n\mid(2n-1)2n(4n-1)$, i.e. $(2n-1)(4n-1)=3\cdot\mathcal{M}$. Let us suppose that $n$ is of form $n=p^k$, where $p\in\NP{2}$ and $k\in\Ns$. We look for the pair $(p,\mathcal{M})$, of integer number, solution of the system:
 \begin{equation}\label{SistemNeliniarZ2}
   \left\{\begin{array}{l}
     (2p-1)(4p-1)=3\cdot\mathcal{M}~, \\
     (2p^2-1)(4p^2-1)=3\cdot q\cdot\mathcal{M}~
   \end{array}\right.
 \end{equation}
 for $q=1,2,\ldots$~. The f{}irst value of $q$ for which we also have a pair $(p,\ \mathcal{M})$ of integer numbers as solution of the nonlinear system (\ref{SistemNeliniarZ2}) is $q=5$. The nonlinear system:
 \[\left\{\begin{array}{l}
     (2p-1)(4p-1)=3\cdot\mathcal{M}~, \\
     (2p^2-1)(4p^2-1)=3\cdot 5\cdot\mathcal{M}~,
   \end{array}\right.
 \]
 has the solutions:
 \[
  \left(
    \begin{array}{cc}
      p &  \mathcal{M} \\
    \end{array}
  \right)=
  \left(
    \begin{array}{cc}
      2 & 7 \\ \\
      \dfrac{1}{2} & 0 \\ \\
      -\dfrac{\sqrt{33}}{4}-\dfrac{5}{4} & \dfrac{25}{2}+\dfrac{13\sqrt{33}}{6} \\ \\
      \dfrac{\sqrt{33}}{4}-\dfrac{5}{4} & \dfrac{25}{2}-\dfrac{13\sqrt{33}}{6}
    \end{array}
  \right)
 \]

   It follows that the f{}irst solution $n$ for which $(2n-1)(4n-1)$ is always multiple of 3 is $n=2^k$. As we have seen in Lemma \ref{Lemma2pentruZ2} for any $k\in\Ns$, $(2^{k+1}-1)(2^{k+2}-1)=3\cdot\mathcal{M}$~.
\end{proof}

\begin{figure}[h]
  \centering
  \includegraphics[scale=0.8]{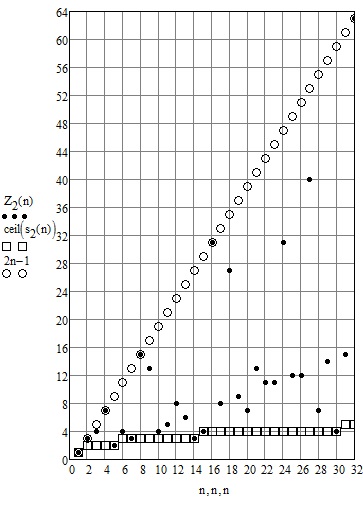}\\
  \caption{Function $Z_2$}\label{FigFunctiaZ2}
\end{figure}
\begin{prog}\label{ProgramZ2} for function $Z_2$.
  \begin{tabbing}
    $Z_2(n):=$\=\vline\ $f$\=$or\ m\in \emph{ceil}(s_2(n))..2n-1$\\
    \>\vline\ \>\ $\emph{return}\ m\ \ \emph{if}\ \ mod[m(m+1)(2m+1),6n]\textbf{=}0$\\
  \end{tabbing}
\end{prog}

\subsection{Pseudo--Smarandache Function of the Order 3}

We def{}ine the function $Z_3:\Ns\to\Ns$ and denote the value of the function $Z_3$ cu $m$, i.e. $m=Z_3(n)$. The value of $m$ is the smallest natural number for which the sum $1^3+2^3+\ldots+m^3$ is dividing by $n$.

Considering the fact that $1^3+2^3+\ldots+m^3=[m(m+1)/2]^2$ this def{}inition of the function $Z_3$ is equivalent with the fact that $m=Z_3(n)$ is the smallest natural number for which we have $[m(m+1)/2]^2=\mathcal{M}\cdot n$ i.e. $[m(m+1)]^2$ is multiple of $4n$ \big(or the equivalent relation $4n\mid [m(m+1)]^2$ i.e. $4n$ divides $[m(m+1)]^2$\big).

The function $s_3(n)$ is the real positive solution of the equation $m^2(m+1)^2=4n$.
\begin{equation}\label{Solutia s3(n)}
  s_3(n):=\frac{\sqrt{8\sqrt{n}+1}-1}{2}~.
\end{equation}

\begin{lem}\label{Lemma1pentruZ3} Let $n,m\in\Ns$, $n\ge m$, if $n\mid[m^2(m+1)^2]/4$, then $m\ge\lceil s_3(n)\rceil$, where $s_3(n)$ is given by \emph{(\ref{Solutia s3(n)})}.
\end{lem}
\begin{proof}
  The relation $n\mid[m^2(m+1)^2]/4$ $\Leftrightarrow$ $m^2(m+1)^2=\mathcal{M}\cdot4n$, with $\mathcal{M}=1,2,\ldots$~. The smallest multiplicity is for $\mathcal{M}=1$. The equation $m^2(m+1)^2=4n$ has as real positive solution $s_3(n)$ given by (\ref{Solutia s3(n)}). Considering that $m$ is a natural number, it results that $m\ge\lceil s_3(n)\rceil$.
\end{proof}

\begin{figure}[h]
  \centering
  \includegraphics[scale=0.8]{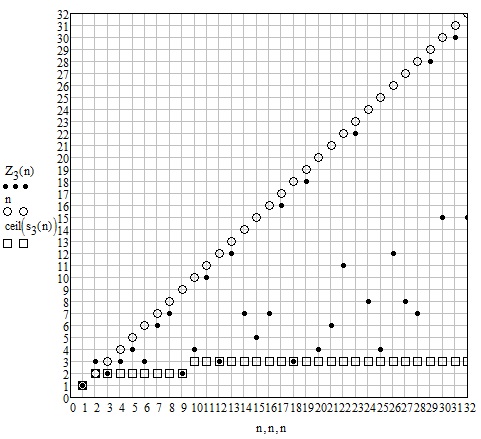}\\
  \caption{Function $Z_3$}\label{FigFunctiaZ3}
\end{figure}

\begin{thm}\label{T1Z3}
  For any number $p\in\NP{3}$, $Z_3(p)=p-1$.
\end{thm}
\begin{proof} We use the notation $Z_3(n)=m$. If $p\in\NP{3}$ then $p=2k+1$, i.e. $p$ is an odd number, and $p-1=2k$, i.e. $p-1$ is an even number. Calculating for $n=p$ the ratio
  \[
   \frac{(p-1)^2p^2}{4p}=\frac{4k^2(2k+1)^2}{4(2k+1)}=k^2(2k+1)\in\Ns~,
  \]
  it follows that, for $m=p-1$, $n=p$ divides $m^2(m+1)^2/4=1^1+2^3+\ldots+m^3$.

  Let us prove that $m=p-1$ is the smallest integer for which we have this property. Supposing that there is a $m=p-j$, where $j\ge2$, such that $Z_3(p)=p-j$, then it should that the number $(p-j)^2(p-j+1)^2/4$ divides $p$, i.e. $p\mid(p-j)$ or $p\mid(p-j+1)$ which is absurd. Therefore, $m=p-1$ is the smallest $m$ for which we have that $m^2(m+1)^2/4$ divides $p$.
\end{proof}

\begin{thm}\label{T2Z3}
  For any $n\in\Ns$, $n\ge3$, $Z_3(n)\le n-1$.
\end{thm}
\begin{proof} We use the notation $Z_3(n)=m$. Suppose that $Z_3(n)\ge n$. If $Z_3(n)=n$, then itshould that $4n\mid[n^2(n+1)^2]$, but
       \[
        \frac{n^2(n+1)^2}{4n}=\frac{n(n+1)^2}{4}~.
       \]
       \begin{enumerate}
         \item If $n=2n_1$, then
             \[
              \frac{n(n+1)^2}{4}=\frac{n_1(2n_1+1)^2}{2}
             \]
             \begin{enumerate}
               \item if $n_1=2n_2$, then
               \[
                \frac{n(n+1)^2}{4}=\frac{n_1(2n_1+1)^2}{2}=\frac{2n_2(4n_2+1)^2}{2}=n_2(4n_2+1)^2\in\Ns~,
               \]
               \item if $n_1=2n_2+1$ then
               \[
                \frac{n(n+1)^2}{4}=\frac{n_1(2n_1+1)^2}{2}=\frac{(2n_2+1)(4n_2+3)^2}{2}\notin\Ns~,
               \]
             \end{enumerate}
         from where it results that the supposition $Z_3(n)=n$ is false (true only if $n=4n_2$).
         \item If $n=2n_1+1$, then
         \[
          \frac{n(n+1)^2}{4}=(2n_1+1)(n_1+1)^2\in\Ns~,
         \]
         but that would imply that also for primes, which are odd numbers, we would have $Z_3(p)=p$ which contradicts the Theorem \ref{T1Z3}, so the supposition that $Z_3(n)=n$ is false.
       \end{enumerate}
  In conclusion, the supposition that $Z_3(n)=n$ is false. Similarly, one can prove that $Z_3(n)=n+j$, for $j=1,2,\ldots$, is false. Therefore, it follows that the equality $Z_3(n)\le n-1$ is true.
\end{proof}

\begin{obs}
  We have two exceptional cases $Z_3(1)=1$ and $Z_3(2)=3$.
\end{obs}

\begin{thm}
  For any $n\in\Ns$, $n\ge3$ and $n\notin\NP{3}$, we have $Z_3(n)\le\left\lfloor\frac{n}{2}\right\rfloor$.
\end{thm}
\begin{proof}
  Theorem to be proved!
\end{proof}

\begin{thm}\label{T4Z3}
  For any $k\in\Ns$, $Z_3(2^k)=2^{\lceil\frac{k+2}{2}\rceil}-1$.
\end{thm}
\begin{proof} We use the notation $Z_3(n)=m$. If $n=2^k$,then, by direct calculation, it verif{}ies for $m=2^{\lceil\frac{k+2}{2}\rceil}-1$, $m^2(m+1)^2$ divides by $4n$,
\[
 \frac{\left(2^{\left\lceil\frac{k+2}{2}\right\rceil}-1\right)^2\left(2^{\left\lceil\frac{k+2}{2}\right\rceil}\right)^2}{4\cdot2^k}=
 \frac{\left(2^{\left\lceil\frac{k+2}{2}\right\rceil}-1\right)^22^{k+2}}{2^{k+2}}
 =\left(2^{\left\lceil\frac{k+2}{2}\right\rceil}-1\right)^2\in\Ns~.
\]

 Let us prove that $m=2^{\lceil\frac{k+2}{2}\rceil}-1$ is the smallest natural number for which $m^2(m+1)^2$ divides by $4n$.
  We search for numbers $m$ of the forma $p^k-1$. From the divisibility conditions for $k=2$ and $k=4$, it follows the nonlinear system
  \begin{equation}\label{SistemNeliniarZ3}
    \left\{\begin{array}{l}
               (p^2-1)p^2=2^2\cdot\mathcal{M}~, \\
               (p^4-1)p^4=2^3\cdot q\cdot\mathcal{M}~, \\
             \end{array}
           \right.
  \end{equation}
  for $q=1,2,\ldots$~. The number $q=10$ is the f{}irst natural number for which the system (\ref{SistemNeliniarZ3}) has integer positive solution. We present the solution of the system using Mathcad symbolic computation
  \[
       \left[\begin{array}{c}
               (p^2-1)p^2=2^2\cdot\mathcal{M}~, \\
               (p^4-1)p^4=2^3\cdot10\cdot\mathcal{M}~, \\
             \end{array}\right]
       \left|\begin{array}{l}
               assume,p\textbf{=}integer \\
               assume,\mathcal{M}\textbf{=}integer \\
               solve,\left(\begin{array}{c}
                             p \\
                             \mathcal{M} \\
                           \end{array}\right)
             \end{array}\right.\rightarrow
                 \left(\begin{array}{cc}
                         0 & 0 \\
                         1 & 0 \\
                         -1 & 0 \\
                         2 & 3 \\
                         -2 & 3 \\
                   \end{array}\right)~.
  \]
  Of the 5 solutions only one solution is convenient $p=2$ and $\mathcal{M}=3$.It follows that $m=2^{f(k)}-1$. By direct verif{}ication it follows that $f(k)=\lceil\frac{k+2}{2}\rceil$. Therefore $m=2^{\lceil\frac{k+2}{2}\rceil}-1$ is the smallest natural number for which $m^2(m+1)^2$ divides by $4n$.
\end{proof}

\begin{prog}\label{ProgramZ3} for function $Z_3$.
  \begin{tabbing}
    $Z_3(n):=$\=\vline\ $\emph{return}\ 3\ \ \emph{if}\ \ n\textbf{=}2$\\
    \>\vline\ $\emph{return}\ n-1\ \emph{if}\ n>2\wedge \emph{TS}(n)\textbf{=}1$\\
    \>\vline\ $f$\=$or\ m\in ceil(s_3(n))..n$\\
    \>\vline\ \>\ $\emph{return}\ m\ \emph{if}\ mod\left[[m(m+1)]^2,4n\right]\textbf{=}0$\\
  \end{tabbing}
  Explanations for the program $Z_3$, \ref{ProgramZ3}, for search of m shortening $m$.
  \begin{enumerate}
    \item The program treats separately the exceptional case $Z_3(2)=3$.
    \item The search of $m$ begins from the value $\lceil s_3(n)\rceil$ according to the Lemma \ref{Lemma1pentruZ3}.
    \item The search of $m$ goes to the value $m=n$.
    \item The program uses the Smarandache primality test, \ref{Program TS}. If $\emph{TS}(n)=1$, then $n\in\NP{3}$ and $Z_3(n)=n-1$ according to the Theorem \ref{T1Z3}.
  \end{enumerate}
\end{prog}

\subsection{Alternative Pseudo--Smarandache Function}

We can def{}ine alternatives of the function $Z_k$, $k=1,2,3$. For example: $V_k:\Ns\to\Ns$, $m=V(n)$ is the smallest integer $m$ such that
\[
 n\mid 1^k-2^k+3^k-4^k+\ldots+(-1)^{m-1}m^k~.
\]

To note that:
\[
 1-2+3-4+\ldots+(-1)^{m-1}\cdot m=\frac{(-1)^{m+1}(2m+1)+1}{4}~,
\]
\[
 1^2-2^2+3^2-4^2+\ldots+(-1)^{m-1}\cdot m^2=\frac{(-1)^{m+1}m(m+1)}{2}~,
\]
and
\begin{multline*}
  1^3-2^3+3^3-4^3+\ldots+(-1)^{m-1}\cdot m^3\\
  =\frac{(-1)^{m+1}(2m+1)(2m^2+2m-1)-1}{8}~.
\end{multline*}

Or more versions of $Z_k$, $k=1,2,3$, by inserting in between the numbers 1, 2, 3, \ldots\ various operators.

\subsection{General Smarandache Functions}

Function $T_k:\Ns\to\Ns$, $m=T_k(n)$ is smallest integer $m$ such that $1^k\circ2^k\circ\ldots\circ m^k$ is divisible by $m$, where $\circ\in\set{+,\cdot,(-1)^{i-1}}$ (and more operators can be used).

If $\circ\equiv\cdot$ we have Smarandache\rq{s} functions, if $\circ\equiv+$ then result Pseudo--Smarandache functions and if $\circ\equiv(-1)^{i-1}$ then one obtains Alternative--Smarandache functions.

\section{Smarandache Functions of the $k$-th Kind}

\subsection{Smarandache Function of the First Kind}
Let the function $S_n:\Ns\to\Ns$ with $n\in\Ns$.
\begin{defn}\label{Definitia Sn}\
   \begin{enumerate}
     \item If $n=p^\alpha$, where $p\in\NP{2}\cup\set{1}$ and $\alpha\in\Ns$, then $m=S_n(a)$ is smallest positive integer such that $n^a\mid m!$~;
     \item If $n=\desp[\alpha]{s}$, where $p_j\in\NP{2}$ and $\alpha_j\in\Ns$ for $j=1,2,\ldots s$, then
     \[
      S_n(a)=\max_{1\le j\le s}\set{S_{p_j^{\alpha_j}}(a)}~.
     \]
   \end{enumerate}
\end{defn}

\subsection{Smarandache Function of the Second Kind}

Smarandache functions of the second kind: $S^k:\Ns\to\Ns$, $S^k(n)=S_n(k)$ for $k\in\Ns$ where $S_n$ are the Smarandache functions of the f{}irst kind.

\subsection{Smarandache Function of the Third Kind}

Smarandache function of the third kind: $S_a^b(n)=S_{a_n}(b_n)$, where $S_{a_n}$ is the Smarandache function of the f{}irst kind, and the sequences $\{a_n\}$ and $\{b_n\}$ are dif{}ferent from the following situations:
\begin{enumerate}
  \item $a_n=1$ and $b_n=n$, for $n\in\Ns$;
  \item $a_n=n$ and $b_n=1$, for $n\in\Ns$.
\end{enumerate}

\section{The Generalization of the Factorial}

\subsection{Factorial for Real Numbers}

Let $x\in\Real_+$, be positive real number. Then factorial of real number is def{}ined as, \citep{Smarandache1972}:
\begin{equation}\label{FactorialRealRatie1}
  x!=\prod_{k=0}^{\lfloor x\rfloor} (x-k)~,\ \ \textnormal{where}\ \ k\in\Na~.
\end{equation}

Examples:
\begin{enumerate}
  \item $2.5!=2.5(2.5-1)(2.5-2)=1.875$~,
  \item $4.37!=4.37(4.37-1)(4.37-2)(4.37-3)(4.37-4)=17.6922054957$~.
\end{enumerate}

\subsubsection{More generally.} Let $\delta\in\Real_+$ be positive real number, then we can introduce formula:
\begin{equation}\label{FactorialRealRatieOarecare}
  x!(\delta)=\prod_{k=0}^{k\cdot\delta<x} (x-k\cdot\delta)~,\ \ \textnormal{where}\ \ k\in\Na~.
\end{equation}
The notation (\ref{FactorialRealRatieOarecare}) means:
\[
 \prod_{k=0}^{k\cdot\delta<x} (x-k\cdot\delta)=x(x-\delta)(x-2\cdot\delta)\cdots(x-m\cdot\delta)~,
\]
where $m$ is the largest integer for which $m\cdot\delta<x$.

Examples:
\begin{multline*}
  4.37!(0.82)=\\
  4.37(4.47-0.82)(4.47-2\cdot0.82)(4.47-3\cdot0.82)\\
  \times(4.47-4\cdot0.82)(4.47-5\cdot0.82)=23.80652826961506~.
\end{multline*}

\subsubsection{And more generally.} Let $\lambda\in\Real$ be real number, then can consider formula:
\begin{equation}\label{FactorialReal}
  x!(\delta)(\lambda)=\prod_{k=0}^{\lambda+k\cdot\delta<x} (x-k\cdot\delta)~,\ \ \textnormal{where}\ \ k\in\Na~.
\end{equation}

Examples:
\begin{enumerate}
  \item\
\[
 6!(1.2)(1.5)=6(6-1.2)(6-2\cdot1.2)(6-3\cdot1.2)=248.83200000000002~,
\]
because $6-3\cdot1.2=2.4>1.5$ and $6-4\cdot1.2=1.2<1.5$~.
\item\
\begin{multline*}
  4.37!(0.82)(-3.25)=\\
  4.37(4.47-0.82)(4.47-2\cdot0.82)(4.47-3\cdot0.82)\\
  \times(4.47-4\cdot0.82)(4.47-5\cdot0.82)(4.47-6\cdot0.82)\\
  \times(4.47-7\cdot0.82)(4.47-8\cdot0.82)(4.47-9\cdot0.82)\\
  =118.24694616330815~,
\end{multline*}
\item\
\begin{multline*}
  4.37!(0.82)(-4.01)=\\
  4.37(4.47-0.82)(4.47-2\cdot0.82)(4.47-3\cdot0.82)\\
  \times(4.47-4\cdot0.82)(4.47-5\cdot0.82)(4.47-6\cdot0.82)\\
  \times(4.47-7\cdot0.82)(4.47-8\cdot0.82)(4.47-9\cdot0.82)\\
  \times(4.47-10\cdot0.82)=-452.8858038054701~.
\end{multline*}
\end{enumerate}
\begin{prog}\label{Program gf} the calculation of generalized factorial.
  \begin{tabbing}
    $\emph{gf}(x,\delta,\lambda):=$\=\vline\ $\emph{return}\ "\emph{Error.}"\ \ \emph{if}\ \ \delta<0$\\
    \>\vline\ $\emph{return}\ \ 1\ \emph{if}\ x\textbf{=}0$\\
    \>\vline\ $f\leftarrow x$\\
    \>\vline\ $k\leftarrow1$\\
    \>\vline\ $w$\=$\emph{hile}\ \ x-k\cdot\delta\ge\lambda$\\
    \>\vline\ \>\vline\ $f\leftarrow f\cdot(x-k\cdot\delta)\ \emph{if}\ x-k\cdot\delta\neq0$\\
    \>\vline\ \>\vline\ $k\leftarrow k+1$\\
    \>\vline\ $\emph{return}\ \ f$\\
   \end{tabbing}

   This program covers all formulas given by (\ref{FactorialRealRatie1}--\ref{FactorialReal}), as you can see in the following examples:
   \begin{enumerate}
     \item $\emph{gf}(7,1,0)=5040=7!$~,
     \item $\emph{gf}(2.5,1,0)=1.875$~,
     \item $\emph{gf}(4.37,1,0)=17.6922054957$~,
     \item $\emph{gf}(4.37,0.82,0)=23.80652826961506$~,
     \item $\emph{gf}(4.37,0.82,-3.25)=118.24694616330815$~,
     \item $\emph{gf}(4.37,0.82,-4.01)=-452.8858038054701$~.
   \end{enumerate}
\end{prog}

\subsection{Smarandacheial}

Let $n>k\ge1$ be two integers. Then the Smarandacheial, \citep{Smarandache2004}, is def{}ined as:
\begin{equation}\label{EquationSmarandacheial1}
  !n!_k=\prod_{i=0}^{0<\abs{n-i\cdot k}\le n}(n-i\cdot k)
\end{equation}

For examples:
\begin{enumerate}
  \item[1.] In the case $k=1$:
\end{enumerate}
  \begin{multline*}
    !n!_1\equiv!n!=\prod_{i=0}^{0<\abs{n-i}\le n}(n-i)\\
    =n(n-1)\cdots2\cdot1\cdot(-1)\cdot(-2)\cdots(-n+1)(-n)=(-1)^n(n!)^2~. \\
  \end{multline*}
  \[
    !5!=5\cdot4\cdot3\cdot2\cdot1\cdot(-1)\cdot(-2)\cdot(-3)\cdot(-4)\cdot(-5)\\
    =-14400=(-1)^5120^2~.
  \]
  To calculate $!n!$ can use the program $\emph{gf}$, given by \ref{Program gf}, as shown in the following example:
  \[
   \emph{gf}(-5,1,-5)=-14400~.
  \]
  The sequence of the f{}irst 20 numbers $!n!=\emph{gf}(n,1,-n)$ is found in following table.
  \begin{center}
  \begin{longtable}{|c|r|}
   \caption{Smarandacheial of order 1}\\
   \hline
   $n$  &  $\emph{gf}(n,1,-n)$ \\
   \hline
  \endfirsthead
   \hline
   $n$  &  $\emph{gf}(n,1,-n)$ \\
   \hline
  \endhead
   \hline \multicolumn{2}{r}{\textit{Continued on next page}} \\
  \endfoot
   \hline
  \endlastfoot
    1&--1\\
    2&4\\
    3&--36\\
    4&576\\
    5&--14400\\
    6&518400\\
    7&--25401600\\
    8&1625702400\\
    9&--131681894400\\
    10&13168189440000\\
    11&--1593350922240000\\
    12&229442532802560000\\
    13&--38775788043632640000\\
    14&7600054456551997440000\\
    15&--1710012252724199424000000\\
    16&437763136697395052544000000\\
    17&--126513546505547170185216000000\\
    18&40990389067797283140009984000000\\
    19&--14797530453474819213543604224000000\\
    20&5919012181389927685417441689600000000\\
\hline
\end{longtable}
\end{center}
\begin{enumerate}
  \item[2.] In case $k=2$:
      \begin{enumerate}
        \item[(a)] If $n$ is odd, then
      \end{enumerate}
\end{enumerate}
  \begin{multline*}
    !n!_2=\prod_{i=0}^{0<\abs{n-2i}\le n}(n-2i)\\
    =n(n-2)\cdots3\cdot1\cdot(-1)\cdot(-3)\cdots(-n+2)(-n)=(-1)^\frac{n+1}{2}(n!!)^2~. \\
  \end{multline*}
  \[
   !5!_2=5(5-2)(5-4)(5-6)(5-8)(5-10)=-225=(-1)^3 15^2~.
  \]
  This result can be achieved with function $\emph{gf}$, given by \ref{Program gf},
  \[
   \emph{gf}(5,2,-5)=-225~.
  \]
\begin{enumerate}
  \item[]
      \begin{enumerate}
        \item[(b)] If $n$ is even, then
      \end{enumerate}
\end{enumerate}
\begin{multline*}
    !n!_2=\prod_{i=0}^{0<\abs{n-2i}\le n}(n-2i)\\
    =n(n-2)\cdots4\cdot2\cdot(-2)\cdot(-4)\cdots(-n+2)(-n)=(-1)^\frac{n}{2}(n!!)^2~. \\
  \end{multline*}
  \[
   !6!_2=6(6-2)(6-4)(6-8)(6-10)(6-12)=-2304=(-1)^3 48^2,
  \]

This result can be achieved with function $\emph{gf}$, given by \ref{Program gf},
\[
 \emph{gf}(6,2,-6)=-2304~.
\]

The sequence of the f{}irst 20 numbers $!n!_2=\emph{gf}(n,2,-n)$ is found in following table.
  \begin{center}
  \begin{longtable}{|c|r|}
   \caption{Smarandacheial of order 2}\\
   \hline
   $n$  &  $\emph{gf}(n,2,-n)$ \\
   \hline
  \endfirsthead
   \hline
   $n$  &  $\emph{gf}(n,2,-n)$ \\
   \hline
  \endhead
   \hline \multicolumn{2}{r}{\textit{Continued on next page}} \\
  \endfoot
   \hline
  \endlastfoot
    1&--1\\
    2&--4\\
    3&9\\
    4&64\\
    5&--225\\
    6&--2304\\
    7&11025\\
    8&147456\\
    9&--893025\\
    10&--14745600\\
    11&108056025\\
    12&2123366400\\
    13&--18261468225\\
    14&--416179814400\\
    15&4108830350625\\
    16&106542032486400\\
    17&--1187451971330625\\
    18&--34519618525593600\\
    19&428670161650355625\\
    20&13807847410237440000\\
\hline
\end{longtable}
\end{center}

The sequence of the f{}irst 20 numbers $!n!_3=\emph{gf}(n,3,-n)$ is found in following table.
  \begin{center}
  \begin{longtable}{|c|r|}
   \caption{Smarandacheial of order 3}\\
   \hline
   $n$  &  $\emph{gf}(n,3,-n)$ \\
   \hline
  \endfirsthead
   \hline
   $n$  &  $\emph{gf}(n,3,-n)$ \\
   \hline
  \endhead
   \hline \multicolumn{2}{r}{\textit{Continued on next page}} \\
  \endfoot
   \hline
  \endlastfoot
    1&1\\
    2&--2\\
    3&--9\\
    4&--8\\
    5&40\\
    6&324\\
    7&280\\
    8&--2240\\
    9&--26244\\
    10&--22400\\
    11&246400\\
    12&3779136\\
    13&3203200\\
    14&--44844800\\
    15&--850305600\\
    16&--717516800\\
    17&12197785600\\
    18&275499014400\\
    19&231757926400\\
    20&--4635158528000\\
\hline
\end{longtable}
\end{center}

For $n:=1..20$, one obtains:
\begin{enumerate}
  \item[] $\emph{gf}(n,4,-n)^\textrm{T}\rightarrow$ 1, --4, --3, --16, --15, 144, 105, 1024, 945, --14400, --10395, --147456, --135135, 2822400, 2027025, 37748736, 34459425, --914457600, --654729075, --15099494400~;
  \item[] $\emph{gf}(n,5,-n)^\textrm{T}\rightarrow$ 1, 2, --6, --4, --25, --24, --42, 336, 216, 2500, 2376, 4032, --52416, --33264, --562500, --532224, --891072, 16039296, 10112256, 225000000~;
  \item[] $\emph{gf}(n,6,-n)^\textrm{T}\rightarrow$ 1, 2, --9, --8, --5, --36, --35, --64, 729, 640, 385, 5184, 5005, 8960, --164025, --143360, --85085, --1679616, --1616615, --2867200~;
  \item[] $\emph{gf}(n,7,-n)^\textrm{T}\rightarrow$ 1, 2, 3, --12, --10, --6, --49, --48, --90, --120, 1320, 1080, 624, 9604, 9360, 17280, 22440, --403920, --328320, --187200~.
\end{enumerate}

We propose to proving the theorem:
\begin{thm}\label{TheoremOfProved}
  The formula
  \[
   !n!_k=(-1)^{\frac{n-1-\mod(n-1,k)}{k}+1}(n\underbrace{!!\ldots!}_{k\ \emph{times}})^2~,
  \]
  for $n,k\in\Ns$, $n>k\ge1$, is true.
\end{thm}

\begin{thm}
  For any integers $k\ge2$ and $n\ge k-1$ following equality
  \begin{equation}\label{Equality}
    n!=k^n\prod_{i=0}^{k-1}\left(\frac{n-i}{k}\right)!
  \end{equation}
  is true.
\end{thm}
\begin{proof}\

  Verif{}ication $k=n+1$, then
  \[
   (n+1)^n\prod_{i=0}^n\left(\frac{n-i}{n+1}\right)!=(n+1)^n\prod_{i=1}^n\frac{n-i}{n+1}=n!~.
  \]

  For any $n\ge k-1$ suppose that (\ref{Equality}) is true, to prove that
  \[
   (n+1)!=k^{n+1}\prod_{i=0}^{k-1}\left(\frac{n+1-i}{k}\right)!~.
  \]
  Really
  \begin{multline*}
    (n+1)!=(n+1)n!=(n+1)k^n\prod_{i=0}^{k-1}\left(\frac{n-i}{k}\right)!\\
    =k^{n+1}\frac{n+1}{k}\left(\frac{n-k+1}{k}\right)!\prod_{i=0}^{k-2}\left(\frac{n-i}{k}\right)!\\
    =k^{n+1}\frac{n+1}{k}\left(\frac{n+1}{k}-1\right)!\prod_{i=0}^{k-2}\left(\frac{n-i}{k}\right)!\\
    =k^{n+1}\left(\frac{n+1}{k}\right)!\prod_{i=0}^{k-2}\left(\frac{n-i}{k}\right)!
    =k^{n+1}\prod_{i=0}^{k-1}\left(\frac{n+1-i}{k}\right)!~.
  \end{multline*}
\end{proof}

\section{Analogues of the Smarandache Function}

Let $a:\Ns\to\Ns$ be a function, where $a(n)$ is the smallest number m such that $n\le m!$ \citep{Yuan+Wenpeng2005}, \cite[A092118]{SloaneOEIS}.

\begin{prog}\label{Program a} for function $a$.
  \begin{tabbing}
    $a(n):=$\=\vline\ $f$\=$or\ m\in1..1000$\\
    \>\vline\ \>\ $\emph{return}\ \ m\ \ \emph{if}\ \ m!\ge n$\\
    \>\vline\ $\emph{return}\ \ "\emph{Error.}"$\\
  \end{tabbing}
\end{prog}

For $n:=1..25$ one obtains $a(n)=$ 1, 1, 2, 3, 3, 3, 3, 4, 4, 4, 4, 4, 4, 4, 4, 4, 4, 4, 4, 4, 4, 4, 4, 4, 4, 5 and $a(10^{10})\rightarrow14$, $a(10^{20})\rightarrow22$, $a(10^{30})\rightarrow29$, $a(10^{40})\rightarrow35$, $a(10^{50})\rightarrow42$, $a(10^{60})\rightarrow48$, $a(10^{70})\rightarrow54$, $a(10^{80})\rightarrow59$, $a(10^{90})\rightarrow65$, $a(10^{100})\rightarrow70$.

\section{Power Function}

\subsection{Power Function of Second Order}

The function $\emph{SP2}:\Ns\to\Ns$, where $\emph{SP2}(n)$ is the smallest number $m$ such that $m^m$ is divisible by $n$.

\begin{prog}\label{Program SP2} for the function $\emph{SP2}$.
  \begin{tabbing}
    $\emph{SP2}(n):=$\=\vline\ $f$\=$or\ \ m\in1..n$\\
    \>\vline\ $\emph{return}\ \ m\ \ \ \emph{if}\ \ \ \mod(m^m,n)\textbf{=}0$\\
  \end{tabbing}
\end{prog}
For $n:=1..10^2$, the command $sp2_n:=\emph{SP2}(n)$, generate the sequence $sp2^\textrm{T}\rightarrow$(1\ \ 2\ \ 3\ \ 2\ \ 5\ \ 6\ \ 7\ \ 4\ \ 3\ \ 10\ \ 11\ \ 6\ \ 13\ \ 14\ \ 15\ \ 4\ \ 17\ \ 6\ \ 19\ \ 10\ \ 21\ \ 22\ \ 23\ \ 6\ \ 5\ \ 26\ \ 3\ \ 14\ \ 29\ \ 30\ \ 31\ \ 4\ \ 33\ \ 34\ \ 35\ \ 6\ \ 37\ \ 38\ \ 39\ \ 10\ \ 41\ \ 42\ \ 43\ \ 22\ \ 15\ \ 46\ \ 47\ \ 6\ \ 7\ \ 10\ \ 51\ \ 26\ \ 53\ \ 6\ \ 55\ \ 14\ \ 57\ \ 58\ \ 59\ \ 30\ \ 61\ \ 62\ \ 21\ \ 4\ \ 65\ \ 66\ \ 67\ \ 34\ \ 69\ \ 70\ \ 71\ \ 6\ \ 73\ \ 74\ \ 15\ \ 38\ \ 77\ \ 78\ \ 79\ \ 10\ \ 6\ \ 82\ \ 83\ \ 42\ \ 85\ \ 86\ \ 87\ \ 22\ \ 89\ \ 30\ \ 91\ \ 46\ \ 93\ \ 94\ \ 95\ \ 6\ \ 97\ \ 14\ \ 33\ \ 10).

\begin{rem} relating to function $\emph{SP2}$, \citep{Smarandache1998,Xu2006,Zhou2006}.
  \begin{enumerate}
    \item If $p\in\NP{2}$, then $\emph{SP2}(p)=p$;
    \item If $r$ is square free, then $\emph{SP2}(r)=r$;
    \item If $n=\desp[\alpha]{s}$ and $\alpha_k\le p_k$, for $k=1,2,\ldots s$, then $\emph{SP2}(n)=n$;
    \item If $n=p^\alpha$, where $p\in\NP{2}$, then:
    \[ \emph{SP2}(n)=\left\{\begin{array}{ll}
               p & \textnormal{if}\ \ 1\le\alpha\le p~, \\
               p^2 & \textnormal{if}\ \ p+1\le\alpha\le2\cdot p^2~, \\
               p^3 & \textnormal{if}\ \ 2p^2+1\le\alpha\le3\cdot p^3~, \\
               \vdots & \vdots \\
               p^s & \textnormal{if}\ \ (s-1)p^{s-1}+1\le\alpha\le s\cdot p^s~.\\
             \end{array}\right.
    \]
  \end{enumerate}
\end{rem}

\subsection{Power Function of Third Order}

The function $\emph{SP3}:\Ns\to\Ns$, where $\emph{SP3}(n)$ is the smallest number $m$ such that $m^{m^m}$ is divisible by $n$.

\begin{prog}\label{Program SP3} for the function $\emph{SP3}$.
  \begin{tabbing}
    $\emph{SP3}(n):=$\=\vline\ $f$\=$or\ \ m\in1..n$\\
    \>\vline\ $\emph{return}\ \ m\ \ \emph{if}\ \mod\left(m^{m^m},n\right)\textbf{=}0$\\
  \end{tabbing}
\end{prog}
For $n:=1..10^2$, the command $sp3_n:=\emph{SP3}(n)$, generate the sequence $sp3^\textrm{T}\rightarrow$(1\ \ 2\ \ 3\ \ 2\ \ 5\ \ 6\ \ 7\ \ 2\ \ 3\ \ 10\ \ 11\ \ 6\ \ 13\ \ 14\ \ 15\ \ 2\ \ 17\ \ 6\ \ 19\ \ 10\ \ 21\ \ 22\ \ 23\ \ 6\ \ 5\ \ 26\ \ 3\ \ 14\ \ 29\ \ 30\ \ 31\ \ 4\ \ 33\ \ 34\ \ 35\ \ 6\ \ 37\ \ 38\ \ 39\ \ 10\ \ 41\ \ 42\ \ 43\ \ 22\ \ 15\ \ 46\ \ 47\ \ 6\ \ 7\ \ 10\ \ 51\ \ 26\ \ 53\ \ 6\ \ 55\ \ 14\ \ 57\ \ 58\ \ 59\ \ 30\ \ 61\ \ 62\ \ 21\ \ 4\ \ 65\ \ 66\ \ 67\ \ 34\ \ 69\ \ 70\ \ 71\ \ 6\ \ 73\ \ 74\ \ 15\ \ 38\ \ 77\ \ 78\ \ 79\ \ 10\ \ 3\ \ 82\ \ 83\ \ 42\ \ 85\ \ 86\ \ 87\ \ 22\ \ 89\ \ 30\ \ 91\ \ 46\ \ 93\ \ 94\ \ 95\ \ 6\ \ 97\ \ 14\ \ 33\ \ 10).

\chapter[Sequences of Numbers]{Sequences of Numbers Involved in Unsolved Problems}

Here it is a long list of sequences, functions, unsolved problems, conjectures, theorems, relationships, operations, etc. Some of them are inter--connected \citep{Knuth2005}, \citep{SloaneOEIS}, \citep{SmarandacheSpecialColection1993}.

\section{Consecutive Sequence}

How many primes are there among these numbers? In a general form, the consecutive sequence is considered in an arbitrary numeration base $b$? \citep{SmarandacheArizona,Smarandache1979}
\begin{center}
 \begin{longtable}{|c|c|}
   \caption{Consecutive sequence}\\
   \hline
   \#  &  $n_{(10)}$ \\
   \hline
  \endfirsthead
   \hline
   \#  &  $n_{(10)}$ \\
   \hline
  \endhead
   \hline \multicolumn{2}{r}{\textit{Continued on next page}} \\
  \endfoot
   \hline
  \endlastfoot
   1 & 1\\
   2 & 12\\
   3 & 123\\
   4 & 1234\\
   5 & 12345\\
   6 & 123456\\
   7 & 1234567\\
   8 & 12345678\\
   9 & 123456789\\
  10 & 12345678910\\
  11 & 1234567891011\\
  12 & 123456789101112\\
  13 & 12345678910111213\\
  14 & 1234567891011121314\\
  15 & 123456789101112131415\\
  16 & 12345678910111213141516\\
  17 & 1234567891011121314151617\\
  18 & 123456789101112131415161718\\
  19 & 12345678910111213141516171819\\
  20 & 1234567891011121314151617181920\\
  21 & 123456789101112131415161718192021\\
  22 & 12345678910111213141516171819202122\\
  23 & 1234567891011121314151617181920212223\\
  24 & 123456789101112131415161718192021222324\\
  25 & 12345678910111213141516171819202122232425\\
  26 & 1234567891011121314151617181920212223242526\\
\hline
\end{longtable}
\end{center}

\begin{center}
 \begin{longtable}{|c|c|}
  \caption{Factored consecutive sequence}\\
  \hline
  \#  &  factors \\
  \hline
  \endfirsthead
   \hline
  \#  &  factors \\
  \hline
  \endhead
   \hline \multicolumn{2}{r}{\textit{Continued on next page}} \\
  \endfoot
   \hline
  \endlastfoot
   1 & 1\\
   2 & $2^2\cdot3$ \\
   3 & $3\cdot41$ \\
   4 & $2\cdot617$ \\
   5 & $3\cdot5\cdot823$ \\
   6 & $2^6\cdot3\cdot643$ \\
   7 & $127\cdot9721$ \\
   8 & $2\cdot3^2\cdot47\cdot14593$ \\
   9 & $3^2\cdot3607\cdot3803$ \\
  10 & $2\cdot5\cdot1234567891$ \\
  11 & $3\cdot7\cdot13\cdot67\cdot107\cdot630803$ \\
  12 & $2^3\cdot3\cdot2437\cdot2110805449$ \\
  13 & $113\cdot125693\cdot869211457$ \\
  14 & $2\cdot3\cdot205761315168520219$ \\
  15 & $3\cdot5\cdot8230452606740808761$ \\
  16 & $2^2\cdot2507191691\cdot1231026625769$ \\
  17 & $3^2\cdot47\cdot4993\cdot584538396786764503$ \\
  18 & $2\cdot3^2\cdot97\cdot88241\cdot801309546900123763$ \\
  19 & $13\cdot43\cdot79\cdot281\cdot1193\cdot833929457045867563$ \\
  20 & $2^5\cdot3\cdot5\cdot323339\cdot3347983\cdot2375923237887317$ \\
  21 & $3\cdot17\cdot37\cdot43\cdot103\cdot131\cdot140453\cdot802851238177109689$ \\
  22 & $2\cdot7\cdot1427\cdot3169\cdot85829\cdot2271991367799686681549$ \\
  23 & $3\cdot41\cdot769\cdot13052194181136110820214375991629$ \\
  24 & $2^2\cdot3\cdot7\cdot978770977394515241\cdot1501601205715706321$ \\
  25 & $5^2\cdot15461\cdot31309647077\cdot1020138683879280489689401$ \\
  26 & $2\cdot3^4\cdot21347\cdot2345807\cdot982658598563\cdot154870313069150249$ \\
  \hline
 \end{longtable}
\end{center}
In base $10$, with the "digits" $\in\set{1,2,\ldots,26}$ not are primes.

\begin{center}
 \begin{longtable}{|c|c|}
   \caption{Binary consecutive sequence in base 2}\\
   \hline
   \#  &  $n_{(2)}$ \\
   \hline
  \endfirsthead
   \hline
   \#  &  $n_{(2)}$ \\
   \hline
  \endhead
   \hline \multicolumn{2}{r}{\textit{Continued on next page}} \\
  \endfoot
  \hline
  \endlastfoot
   1 & 1 \\
   2 & 110 \\
   3 & 11011 \\
   4 & 11011100 \\
   5 & 11011100101 \\
   6 & 11011100101110 \\
   7 & 11011100101110111 \\
   8 & 110111001011101111000 \\
   9 & 1101110010111011110001001 \\
  10 & 11011100101110111100010011010 \\
  11 & 110111001011101111000100110101011 \\
  12 & 1101110010111011110001001101010111100 \\
  13 & 11011100101110111100010011010101111001101 \\
  14 & 110111001011101111000100110101011110011011110 \\
  15 & 1101110010111011110001001101010111100110111101111 \\
  \hline
 \end{longtable}
\end{center}

\begin{center}
 \begin{longtable}{|c|c|c|}
   \caption{Binary consecutive sequence in base 10}\\
   \hline
   \#  &  $n_{(10)}$ & factors \\
   \hline
  \endfirsthead
   \hline
   \#  &  $n_{(10)}$ & factors \\
   \hline
  \endhead
   \hline \multicolumn{3}{r}{\textit{Continued on next page}} \\
  \endfoot
   \hline
  \endlastfoot
   1 & 1 & 1 \\
   2 & 6 & $2\cdot3$ \\
   3 & 27 & $3^3$ \\
   4 & 220 & $2^2\cdot5\cdot11$ \\
   5 & 1765 & $5\cdot353$ \\
   6 & 14126 & $2\cdot7\cdot1009$ \\
   7 & 113015 & $5\cdot7\cdot3229$ \\
   8 & 1808248 & $2^3\cdot13\cdot17387$ \\
   9 & 28931977 & $17\cdot1701881$ \\
  10 & 462911642 & $2\cdot167\cdot1385963$ \\
  11 & 7406586283 & $29\cdot53\cdot179\cdot26921$ \\
  12 & 118505380540 & $2^2\cdot5\cdot5925269027$ \\
  13 & 1896086088653 & $109\cdot509\cdot2971\cdot11503$ \\
  14 & 30337377418462 & $2\cdot15168688709231$ \\
  15 & 485398038695407 & \fbox{485398038695407} \\
  \hline
 \end{longtable}
\end{center}
The numbers given in the box are prime numbers. In base $2$, with the "digits" $\in\set{1,2,\ldots,15}$ the number \fbox{485398038695407} is a prime number.

\begin{center}
 \begin{longtable}{|c|c|}
   \caption{Ternary consecutive sequence in base 3}\\
   \hline
   \#  &  $n_{(3)}$ \\
   \hline
  \endfirsthead
   \hline
   \#  &  $n_{(3)}$ \\
   \hline
  \endhead
   \hline \multicolumn{2}{r}{\textit{Continued on next page}} \\
  \endfoot
   \hline
  \endlastfoot
   1 & 1\\
   2 & 12\\
   3 & 1210\\
   4 & 121011\\
   5 & 12101120\\
   6 & 1210112021\\
   7 & 121011202122\\
   8 & 121011202122100\\
   9 & 121011202122100101\\
  10 & 121011202122100101110\\
  11 & 121011202122100101110111\\
  12 & 121011202122100101110111200\\
  13 & 121011202122100101110111200201\\
  14 & 121011202122100101110111200201210\\
  15 & 121011202122100101110111200201210211\\
  16 & 121011202122100101110111200201210211220\\
  \hline
 \end{longtable}
\end{center}

\begin{center}
 \begin{longtable}{|c|c|c|}
   \caption{Ternary consecutive sequence in base 10}\\
   \hline
   \#  &  $n_{(10)}$ & factors \\
   \hline
  \endfirsthead
   \hline
   \#  &  $n_{(10)}$ & factors \\
   \hline
  \endhead
   \hline \multicolumn{3}{r}{\textit{Continued on next page}} \\
  \endfoot
   \hline
  \endlastfoot
   1 & 1 & 1\\
   2 & 5 & 5\\
   3 & 48 & $2^4\cdot3$\\
   4 & 436 & $2^2\cdot109$\\
   5 & 3929 & \fbox{3929}\\
   6 & 35367 & $3\cdot11789$\\
   7 & 318310 & $2\cdot5\cdot139\cdot229$\\
   8 & 2864798 & $2\cdot97\cdot14767$\\
   9 & 77349555 & $3^2\cdot5\cdot1718879$\\
  10 & 2088437995 & $5\cdot7\cdot59669657$\\
  11 & 56387825876 & $2^2\cdot14096956469$\\
  12 & 1522471298664 & $2^3\cdot3\cdot63436304111$\\
  13 & 41106725063941 & $6551\cdot11471\cdot547021$\\
  14 & 1109881576726421 & $41\cdot27070282359181$\\
  15 & 29966802571613382 & $2\cdot3\cdot17\cdot2935459\cdot100083899$\\
  \hline
 \end{longtable}
\end{center}
In base $3$, with the "digits" $\in\set{1,2,\ldots,15}$ number \fbox{3929} is prime number.

\begin{center}
 \begin{longtable}{|c|c|c|}
   \caption{Octal consecutive sequence}\\
   \hline
   $n_{(8)}$  &  $n_{(10)}$ & factors \\
   \hline
  \endfirsthead
   \hline
   $n_{(8)}$  &  $n_{(10)}$ & factors \\
   \hline
  \endhead
   \hline \multicolumn{3}{r}{\textit{Continued on next page}} \\
  \endfoot
   \hline
  \endlastfoot
   1 & 1 & 1 \\
   12 & 10 & $2\cdot5$ \\
   123 & 83 & \fbox{83} \\
   1234 & 668 & $2^2\cdot167$ \\
   12345 & 5349 & $3\cdot1783$ \\
   123456 & 42798 & $2\cdot3\cdot7\cdot1019$ \\
   1234567 & 342391 & $7\cdot41\cdot1193$\\
   123456710 & 21913032 & $2^3\cdot3\cdot31\cdot29453$ \\
   12345671011 & 1402434057 & $3\cdot17^2\cdot157\cdot10303$ \\
   1234567101112 & 89755779658 & $2\cdot44877889829$ \\
   123456710111213 & 5744369898123 & $3\cdot83\cdot23069758627$ \\
   12345671011121314 & 367639673479884 & $2^2\cdot3^2\cdot13\cdot29\cdot53\cdot511096199$ \\
   1234567101112131415 & 23528939102712588 & $2^2\cdot3\cdot461\cdot4253242787909$ \\
  \hline
 \end{longtable}
\end{center}
In octal, with the "digits" $\in\set{1,2,\ldots,15}$ only number \fbox{83} is prime number.

\begin{center}
 \begin{longtable}{|c|c|c|}
   \caption{Hexadecimal consecutive sequence}\\
   \hline
   $n_{(16)}$  &  $n_{(10)}$ & factors \\
   \hline
  \endfirsthead
   \hline
   $n_{(16)}$  &  $n_{(10)}$ & factors \\
   \hline
  \endhead
   \hline \multicolumn{3}{r}{\textit{Continued on next page}} \\
  \endfoot
   \hline
  \endlastfoot
   1 & 1 & 1 \\
   12 & 18 & $2\cdot3^2$ \\
   123 & 291 & $3\cdot97$ \\
   1234 & 4660 & $2^2\cdot5\cdot233$ \\
   12345 & 74565 & $3^2\cdot5\cdot1657$ \\
   123456 & 1193046 & $2\cdot3\cdot198841$ \\
   1234567 & 19088743 & $2621\cdot7283$ \\
   12345678 & 305419896 & $2^3\cdot3^5\cdot157109$ \\
   123456789 & 4886718345 & $3^2\cdot5\cdot23\cdot4721467$ \\
   123456789a & 78187493530 & $2\cdot5\cdot7818749353$ \\
   123456789ab & 1250999896491 & $3^2\cdot12697\cdot10947467$ \\
   123456789abc & 20015998343868 & $2^2\cdot3\cdot1242757\cdot1342177$ \\
   123456789abcd & 320255973501901 & \fbox{320255973501901} \\
   123456789abcde & 5124095576030430 &  $2\cdot3^2\cdot5\cdot215521\cdot264170987$ \\
  \hline
 \end{longtable}
\end{center}
In hexadecimal, with the "digits" $\in\set{1,2,\ldots,15}$ number \fbox{320255973501901} is prime.

\section{Circular Sequence}

\begin{center}
 \begin{longtable}{|r|r||r|r|}
 \caption{Circular sequence}\\
   \hline
   $n_{(10)}$ & factors & $n_{(10)}$ & factors \\
   \hline
 \endfirsthead
   \hline
   $n_{(10)}$ & factors & $n_{(10)}$ & factors \\
   \hline
  \endhead
   \hline \multicolumn{4}{r}{\textit{Continued on next page}} \\
  \endfoot
   \hline
  \endlastfoot
   12 & $2^2\cdot3$ & 13245 & $3\cdot5\cdot883$  \\
   21 & $3\cdot7$ & 13254 & $2\cdot3\cdot47^2$ \\
  \cline{1-2}
   123 & $3\cdot41$ & 13425 & $3\cdot5^2\cdot179$ \\
   132 & $2^2\cdot3\cdot11$ & 13452 & $2^2\cdot3\cdot19\cdot59$ \\
   213 & $3\cdot71$ & 13524 & $2^2\cdot3\cdot7^2\cdot23$ \\
   231 & $3\cdot7\cdot11$ & 13542 & $2\cdot3\cdot37\cdot61$ \\
   312 & $2^3\cdot3\cdot13$ & 14235 & $3\cdot5\cdot13\cdot73$ \\
   321 & $3\cdot107$ & 14253 & $3\cdot4751$ \\
  \cline{1-2}
   1234 & $2\cdot617$ & 14325 & $3\cdot5^2\cdot191$ \\
   1243 & $11\cdot113$ & 14352 & $2^4\cdot3\cdot13\cdot23$ \\
   1324 & $2^2\cdot331$ & 14523 & $3\cdot47\cdot103$ \\
   1342 & $2\cdot11\cdot61$ & 14532 & $2^2\cdot3\cdot7\cdot173$ \\
   1423 & \fbox{1423} & 15234 & $2\cdot3\cdot2539$ \\
   1432 & $2^3\cdot179$ & 15243 & $3\cdot5081$ \\
   2134 & $2\cdot11\cdot97$ & 15324 & $2^2\cdot3\cdot1277$ \\
   2143 & \fbox{2143} & 15342 & $2\cdot3\cdot2557$ \\
   2314 & $2\cdot13\cdot89$ & 15423 & $3\cdot53\cdot97$ \\
   2341 & \fbox{2341} & 15432 & $2^3\cdot3\cdot643$ \\
   2413 & $19\cdot127$ & 21345 & $3\cdot5\cdot1423$ \\
   2431 & $11\cdot13\cdot17$ & 21354 & $2\cdot3\cdot3559$ \\
   3124 & $2^2\cdot11\cdot71$ & 21435 & $3\cdot5\cdot1429$ \\
   3142 & $2\cdot1571$ & 21453 & $3\cdot7151$ \\
   3214 & $2\cdot1607$ & 21534 & $2\cdot3\cdot37\cdot97$ \\
   3241 & $7\cdot463$ & 21543 & $3\cdot43\cdot167$ \\
   3412 & $2^2\cdot853$ & 23145 & $3\cdot5\cdot1543$ \\
   3421 & $11\cdot311$ & 23154 & $2\cdot3\cdot17\cdot227$ \\
   4123 & $7\cdot19\cdot31$ & 23415 & $3\cdot5\cdot7\cdot223$ \\
   4132 & $2^2\cdot1033$ & 23451 & $3\cdot7817$ \\
   4213 & $11\cdot383$ & 23514 & $2\cdot3\cdot3919$ \\
   4231 & \fbox{4231} & 23541 & $3\cdot7\cdot19\cdot59$ \\
   4312 & $2^3\cdot7^2\cdot11$ & 24135 & $3\cdot5\cdot1609$ \\
   4321 & $29\cdot149$ & 24153 & $3\cdot83\cdot97$ \\
  \cline{1-2}
   12345 & $3\cdot5\cdot823$ & 24315 & $3\cdot5\cdot1621$ \\
   12354 & $2\cdot3\cdot29\cdot71$ & 24351 & $3\cdot8117$ \\
   12435 & $3\cdot5\cdot829$ & 24513 & $3\cdot8171$ \\
   12453 & $3\cdot7\cdot593$ & 24531 & $3\cdot13\cdot17\cdot37$ \\
   12534 & $2\cdot3\cdot2089$ & 25134 & $2\cdot3\cdot59\cdot71$ \\
   12543 & $3\cdot37\cdot113$ & 25143 & $3\cdot17^2\cdot29$ \\
   25314 & $2\cdot3\cdot4219$ & 42513 & $3\cdot37\cdot383$ \\
   25341 & $3\cdot8447$ & 42531 &  $3\cdot14177$ \\
   25413 & $3\cdot43\cdot197$ & 43125 & $3\cdot5^4\cdot23$ \\
   25431 & $3\cdot7^2\cdot173$ & 43152 & $2^4\cdot3\cdot29\cdot31$ \\
   31245 & $3\cdot5\cdot2083$ & 43215 & $3\cdot5\cdot43\cdot67$ \\
   31254 & $2\cdot3\cdot5209$ & 43251 &  $3\cdot13\cdot1109$ \\
   31425 & $3\cdot5^2\cdot419$ & 43512 &  $2^3\cdot3\cdot7^2\cdot37$ \\
   31452 & $2^2\cdot3\cdot2621$ & 43521 &  $3\cdot89\cdot163$ \\
   31524 & $2^2\cdot3\cdot37\cdot71$ & 45123 & $3\cdot13^2\cdot89$ \\
   31542 & $2\cdot3\cdot7\cdot751$ & 45132 & $2^2\cdot3\cdot3761$ \\
   32145 & $3\cdot5\cdot2143$ & 45213 & $3\cdot7\cdot2153$ \\
   32154 & $2\cdot3\cdot23\cdot233$ & 45231 & $3\cdot15077$ \\
   32415 & $3\cdot5\cdot2161$ & 45312 & $2^8\cdot3\cdot59$ \\
   32451 & $3\cdot29\cdot373$ & 45321 & $3\cdot15107$ \\
   32514 & $2\cdot3\cdot5419$ & 51234 & $2\cdot3\cdot8539$ \\
   32541 & $3\cdot10847$ & 51243 & $3\cdot19\cdot29\cdot31$ \\
   34125 & $3\cdot5^3\cdot7\cdot13$ & 51324 & $2^2\cdot3\cdot7\cdot13\cdot47$ \\
   34152 & $2^3\cdot3\cdot1423$ & 51342 & $2\cdot3\cdot43\cdot199$ \\
   34215 & $3\cdot5\cdot2281$ & 51423 & $3\cdot61\cdot281$ \\
   34251 & $3\cdot7^2\cdot233$ & 51432 & $2^3\cdot3\cdot2143$ \\
   34512 & $2^4\cdot3\cdot719$ & 52134 & $2\cdot3\cdot8689$ \\
   34521 & $3\cdot37\cdot311$ & 52143 & $3\cdot7\cdot13\cdot191$ \\
   35124 & $2^2\cdot3\cdot2927$ & 52314 & $2\cdot3\cdot8719$ \\
   35142 & $2\cdot3\cdot5857$ & 52341 & $3\cdot73\cdot239$ \\
   35214 & $2\cdot3\cdot5869$ & 52413 & $3\cdot17471$ \\
   35241 & $3\cdot17\cdot691$ & 52431 & $3\cdot17477$ \\
   35412 & $2^2\cdot3\cdot13\cdot227$ & 53124 & $2^2\cdot3\cdot19\cdot233$\\
   35421 & $3\cdot11807$ & 53142 & $2\cdot3\cdot17\cdot521$ \\
   41235 & $3\cdot5\cdot2749$ & 53214 & $2\cdot3\cdot7^2\cdot181$ \\
   41253 & $3\cdot13751$ & 53241 & $3\cdot17747$ \\
   41325 & $3\cdot5^2\cdot19\cdot29$ & 53412 & $2^2\cdot3\cdot4451$ \\
   41352 & $2^3\cdot3\cdot1723$ & 53421 & $3\cdot17807$ \\
   41523 & $3\cdot13841$ & 54123 & $3\cdot18041$ \\
   41532 & $2^2\cdot3\cdot3461$ & 54132 & $2^2\cdot3\cdot13\cdot347$ \\
   42135 & $3\cdot5\cdot53^2$ & 54213 & $3\cdot17\cdot1063$ \\
   42153 & $3\cdot14051$ & 54231 & $3\cdot18077$ \\
   42315 & $3\cdot5\cdot7\cdot13\cdot31$ & 54312 & $2^3\cdot3\cdot31\cdot73$\\
   42351 & $3\cdot19\cdot743$ & 54321 & $3\cdot19\cdot953$ \\
  \hline
 \end{longtable}
\end{center}
The numbers \fbox{1423}, \fbox{2143}, \fbox{2341} and \fbox{4231} are the only primes for circular sequences 12, 21, 123, \ldots, 321, 1234, \ldots, 54321~.

\section{Symmetric Sequence}

The sequence of symmetrical numbers was considered in the works \citep{Smarandache1979,SmarandacheArizona}

\begin{center}
 \begin{longtable}{|c|c|}
   \caption{Symmetric sequence}\\
   \hline
   $n_{(10)}$ & factors \\
   \hline
  \endfirsthead
   \hline
   $n_{(10)}$ & factors \\
   \hline
  \endhead
   \hline \multicolumn{2}{r}{\textit{Continued on next page}} \\
  \endfoot
   \hline
  \endlastfoot
   1 & 1 \\
   11 & \fbox{11} \\
   121 & $11^2$ \\
   1221 & $3\cdot11\cdot37$  \\
   12321 & $3^2\cdot37^2$ \\
   123321 & $3\cdot11\cdot37\cdot101$ \\
   1234321 & $11^2\cdot101^2$ \\
   12344321 & $11\cdot41\cdot101\cdot271$ \\
   123454321 & $41^2\cdot271^2$ \\
   1234554321 & $3\cdot7\cdot11\cdot13\cdot37\cdot41\cdot271$ \\
   12345654321 & $3^2\cdot7^2\cdot11^2\cdot13^2\cdot37^2$ \\
   123456654321 & $3\cdot7\cdot11\cdot13\cdot37\cdot239\cdot4649$ \\
   1234567654321 & $239^2\cdot4649^2$ \\
   12345677654321 & $11\cdot73\cdot101\cdot137\cdot239\cdot4649$ \\
   123456787654321 & $11^2\cdot73^2\cdot101^2\cdot137^2$ \\
   1234567887654321 & $3^2\cdot11\cdot37\cdot73\cdot101\cdot137\cdot333667$ \\
   12345678987654321 & $3^4\cdot37^2\cdot333667^2$ \\
   123456789987654321 & $3^2\cdot11\cdot37\cdot41\cdot271\cdot9091\cdot333667$ \\
   12345678910987654321 & \fbox{12345678910987654321} \\
   1234567891010987654321 & \fbox{1234567891010987654321} \\
   123456789101110987654321 & $7\cdot17636684157301569664903$ \\
   12345678910111110987654321 & $3\cdot43\cdot97\cdot548687\cdot1798162193492191$ \\
   1234567891011121110987654321 & $3^2\cdot7^2\cdot2799473675762179389994681$ \\
  \hline
 \end{longtable}
\end{center}

\section{Deconstructive Sequence}

Deconstructive sequence with the decimal digits $\set{1,2,\ldots,9}$, \citep{Smarandache1993,SmarandacheArizona}.

\begin{center}
 \begin{longtable}{|c|c|}
   \caption{Deconstructive sequence with $\set{1,2,\ldots,9}$}\\
   \hline
   $n_{(10)}$ & factors \\
   \hline
  \endfirsthead
   \hline
   $n_{(10)}$ & factors \\
   \hline
  \endhead
   \hline \multicolumn{2}{r}{\textit{Continued on next page}} \\
  \endfoot
   \hline
  \endlastfoot
   1 & 1\\
   23 & \fbox{23}\\
   456 & $2^3\cdot3\cdot19$\\
   7891 & $13\cdot607$\\
   23456 & $2^5\cdot733$\\
   789123 & $3\cdot17\cdot15473$\\
   4567891 & \fbox{4567891}\\
   23456789 & \fbox{23456789}\\
   123456789 & $3^2\cdot3607\cdot3803$\\
   1234567891 & \fbox{1234567891}\\
   23456789123 & $59\cdot397572697$\\
   456789123456 & $2^7\cdot3\cdot23\cdot467\cdot110749$\\
   7891234567891 & $37\cdot353\cdot604183031$\\
   23456789123456 & $2^7\cdot13\cdot23\cdot47\cdot13040359$\\
   789123456789123 & $3\cdot19\cdot13844271171739$\\
   4567891234567891 & $739\cdot1231\cdot4621\cdot1086619$\\
   23456789123456789 & \fbox{23456789123456788}\\
   123456789123456789 & $3^2\cdot7\cdot11\cdot13\cdot19\cdot3607$\\
                      & $\ \ \ \ \ \ \ \ \ \ \ \ \ \ \ \times3803\cdot52579$\\
   1234567891234567891 & $31\cdot241\cdot1019\cdot162166841159$\\
  \hline
 \end{longtable}
\end{center}

Deconstructive sequence with the decimal digits $\set{1,2,\ldots,9,0}$.

\begin{center}
 \begin{longtable}{|c|c|}
   \caption{Deconstructive sequence with $\set{1,2,\ldots,9,0}$}\\
   \hline
   $n_{(10)}$ & factors \\
   \hline
  \endfirsthead
   \hline
   $n_{(10)}$ & factors \\
   \hline
  \endhead
   \hline \multicolumn{2}{r}{\textit{Continued on next page}} \\
  \endfoot
   \hline
  \endlastfoot
   1 & 1\\
   23 & \fbox{23}\\
   456 & $2^3\cdot3\cdot19$\\
   7890 & $2\cdot3\cdot5\cdot263$\\
   12345 & $3\cdot5\cdot823$\\
   678901 & \fbox{678901}\\
   2345678 & $2\cdot23\cdot50993$\\
   90123456 & $2^6\cdot3\cdot367\cdot1279$\\
   789012345 & $3\cdot5\cdot11\cdot131\cdot173\cdot211$\\
   6789012345 & $3^2\cdot5\cdot150866941$\\
   67890123456 & $2^6\cdot3\cdot353594393$\\
   789012345678 & $2\cdot3\cdot19\cdot9133\cdot757819$\\
   9012345678901 & \fbox{9012345678901}\\
   23456789012345 & $5\cdot13\cdot19\cdot89\cdot213408443$\\
   678901234567890 & $2\cdot3\cdot5\cdot1901\cdot11904282563$\\
   1234567890123456 & $2^6\cdot3\cdot7^2\cdot301319\cdot435503$\\
  \hline
 \end{longtable}
\end{center}

\section{Concatenated Sequences}

Sequences obtained from concatenating the sequences of numbers: primes, Fibonacci\index{Fibonacci L.}, Mersenne\index{Mersenne M.}, etc.

\begin{prog}\label{Program ConS} for concatenation the terms of sequence.
\begin{tabbing}
  $\emph{ConS}(s,L):=$\=\vline\ $cs_1\leftarrow s_1$\\
  \>\vline\ $f$\=$or\ k\in2..L$\\
  \>\vline\ \>\ $cs_k\leftarrow conc(cs_{k-1},s_k)$\\
  \>\vline\ $\emph{return}\ \ \emph{cs}$\\
\end{tabbing}
\end{prog}

\begin{prog}\label{Program BConS} for back concatenation the terms of sequence.
\begin{tabbing}
  $\emph{BConS}(s,L):=$\=\vline\ $cs_1\leftarrow s_1$\\
  \>\vline\ $f$\=$or\ k\in2..L$\\
  \>\vline\ \>\ $cs_k\leftarrow conc(s_k,cs_{k-1})$\\
  \>\vline\ $\emph{return}\ \ \emph{cs}$\\
\end{tabbing}
\end{prog}
It was obtained by the programs $\emph{ConS}$, \ref{Program ConS} and $\emph{BConS}$, \ref{Program BConS} using the routine $\emph{conc}$, \ref{Functia conc}.

\subsection{Concatenated Prime Sequence}

Using the program \ref{Program ConS} one can generate a Concatenated Prime Sequence (called Smarandache--Wellin numbers)
\[
 L:=20\ \ p:=\emph{submatrix}(\emph{prime},1,L,1,1)\ \  \emph{cp}:=\emph{ConS}(p,L)
\]
then $\emph{cp}\rightarrow$ provides the vector:
\begin{center}
 \begin{longtable}{|r|c|}
   \caption{Concatenated Prime Sequence}\\
   \hline
   $k$&$cp_k$ \\
   \hline
  \endfirsthead
   \hline
   $k$&$cp_k$ \\
   \hline
  \endhead
   \hline \multicolumn{2}{r}{\textit{Continued on next page}} \\
  \endfoot
   \hline
  \endlastfoot
   1&2 \\ \hline
   2&23 \\ \hline
   3&235 \\ \hline
   4&2357 \\ \hline
   5&235711 \\ \hline
   6&23571113 \\ \hline
   7&2357111317 \\ \hline
   8&235711131719 \\ \hline
   9&23571113171923 \\ \hline
   10&2357111317192329 \\ \hline
   11&235711131719232931 \\ \hline
   12&23571113171923293137 \\ \hline
   13&2357111317192329313741 \\ \hline
   14&235711131719232931374143 \\ \hline
   15&23571113171923293137414347 \\ \hline
   16&2357111317192329313741434753 \\ \hline
   17&235711131719232931374143475359 \\ \hline
   18&23571113171923293137414347535961 \\ \hline
   19&2357111317192329313741434753596167 \\ \hline
   20&235711131719232931374143475359616771 \\
  \hline
 \end{longtable}
\end{center}

Factorization of the vector $cp$ is obtained with the command: $cp\ \emph{factor}\rightarrow$
\begin{center}
 \begin{longtable}{|r|c|}
   \caption{Factorization Concatenated Prime Sequence}\\
   \hline
   $k$&$cp_k$ \\
   \hline
  \endfirsthead
   \hline
   $k$&$cp_k$ \\
   \hline
  \endhead
   \hline \multicolumn{2}{r}{\textit{Continued on next page}} \\
  \endfoot
   \hline
  \endlastfoot
  1&\fbox{2} \\ \hline
  2&\fbox{23} \\ \hline
  3&$5\cdot47$ \\ \hline
  4&\fbox{2357} \\ \hline
  5&$7\cdot151\cdot223$ \\ \hline
  6&$23\cdot29\cdot35339$ \\ \hline
  7&$11\cdot214282847$ \\ \hline
  8&$7\cdot4363\cdot7717859$ \\ \hline
  9&$61\cdot478943\cdot806801$ \\ \hline
  10&$3\cdot4243\cdot185176472401$ \\ \hline
  11&$17\cdot283\cdot1787\cdot76753\cdot357211$ \\ \hline
  12&$7\cdot67^2\cdot151\cdot4967701595369$ \\ \hline
  13&$25391\cdot889501\cdot104364752351$ \\ \hline
  14&$6899\cdot164963\cdot7515281\cdot27558919$ \\ \hline
  15&$1597\cdot2801\cdot5269410931806332951$ \\ \hline
  16&$3\cdot2311\cdot1237278209\cdot274784055330749$ \\ \hline
  17&$17\cdot906133\cdot12846401\cdot1191126125288819$ \\ \hline
  18&$3\cdot13\cdot3390511326677\cdot178258515898000387$ \\ \hline
  19&$1019\cdot2313161253378144566969023310693$ \\ \hline
  20&$3^3\cdot8730041915527145606449758346652473$ \\
  \hline
 \end{longtable}
\end{center}

\subsection{Back Concatenated Prime Sequence}

Using the program \ref{Program BConS} one can generate a Back Concatenated Prime
Sequence (on short BCPS) $L:=20$ $p:=\emph{submatrix}(\emph{prime},2,L+1,1,1)^\textrm{T}=$ (3\ \ 5\ \ 7\ \ 11\ \ 13\ \ 17\ \ 19\ \ 23\ \ 29\ \ 31\ \ 37\ \ 41\ \ 43\ \ 47\ \ 53\ \ 59\ \ 61\ \ 67\ \ 71\ \ 73\ \ 79), $\emph{bcp}:=\emph{BConS}(p,L)$, then $\emph{bcp}\rightarrow$ provides the vector:
\begin{center}
 \begin{longtable}{|r|c|}
   \caption{Back Concatenated Prime Sequence}\\
   \hline
   $k$&$bcp_k$ \\
   \hline
  \endfirsthead
   \hline
   $k$&$bcp_k$ \\
   \hline
  \endhead
   \hline \multicolumn{2}{r}{\textit{Continued on next page}} \\
  \endfoot
   \hline
  \endlastfoot
  1&3 \\ \hline
  2&53 \\ \hline
  3&753 \\ \hline
  4&11753 \\ \hline
  5&1311753 \\ \hline
  6&171311753 \\ \hline
  7&19171311753 \\ \hline
  8&2319171311753 \\ \hline
  9&292319171311753 \\ \hline
  10&31292319171311753 \\ \hline
  11&3731292319171311753 \\ \hline
  12&413731292319171311753 \\ \hline
  13&43413731292319171311753 \\ \hline
  14&4743413731292319171311753 \\ \hline
  15&534743413731292319171311753 \\ \hline
  16&59534743413731292319171311753 \\ \hline
  17&6159534743413731292319171311753 \\ \hline
  18&676159534743413731292319171311753 \\ \hline
  19&71676159534743413731292319171311753 \\ \hline
  20&7371676159534743413731292319171311753 \\
  \hline
 \end{longtable}
\end{center}

\begin{center}
 \begin{longtable}{|r|c|}
   \caption{Factorization BCPS}\\
   \hline
   $k$&$bcp_k$ \\
   \hline
  \endfirsthead
   \hline
   $k$&$bcp_k$ \\
   \hline
  \endhead
   \hline \multicolumn{2}{r}{\textit{Continued on next page}} \\
  \endfoot
   \hline
  \endlastfoot
  1&\fbox{3} \\ \hline
  2&\fbox{53} \\ \hline
  3&$3\cdot251$ \\ \hline
  4&$7\cdot23\cdot73$ \\ \hline
  5&$3\cdot331\cdot1321$ \\ \hline
  6&\fbox{171311753} \\ \hline
  7&$3\cdot11^2\cdot52813531$ \\ \hline
  8&$19\cdot122061647987$ \\ \hline
  9&$75041\cdot3895459433$ \\ \hline
  10&$463\cdot44683\cdot1512566357$ \\ \hline
  11&$3\cdot15913\cdot1110103\cdot70408109$ \\ \hline
  12&$17\cdot347\cdot1613\cdot1709\cdot80449\cdot316259$ \\ \hline
  13&$3^2\cdot41\cdot557\cdot1260419\cdot167583251039$ \\ \hline
  14&$17^2\cdot37\cdot127309607\cdot3484418108803$ \\ \hline
  15&$67\cdot241249\cdot33083017882204960291$ \\ \hline
  16&$3\cdot7\cdot11\cdot13\cdot281\cdot15289778873\cdot4614319153627$ \\ \hline
  17&$1786103719753\cdot3448587377817864001$ \\ \hline
  18&$17\cdot83407\cdot336314747\cdot1417920375788952821$ \\ \hline
  19&$19989277303\cdot3585730411773627378513151$ \\ \hline
  20&$1613\cdot24574819\cdot75164149139\cdot2474177239668341$ \\
  \hline
 \end{longtable}
\end{center}

\subsection{Concatenated Fibonacci Sequence}

With the commands: $L:=20$ $f_1:=1$ $f_2:=1$ $k:=3..L$ $f_{k}:=f_{k-1}+f_{k-2}$ $\emph{cF}:=\emph{ConS}(f,l)$, resulting the vector:
\begin{center}
 \begin{longtable}{|r|c|}
   \caption{Concatenated Fibonacci Sequence}\\
   \hline
   $k$&$cF_k$ \\
   \hline
  \endfirsthead
   \hline
   $k$&$cF_k$ \\
   \hline
  \endhead
   \hline \multicolumn{2}{r}{\textit{Continued on next page}} \\
  \endfoot
   \hline
  \endlastfoot
  1&1 \\ \hline
  2&\fbox{11} \\ \hline
  3&112 \\ \hline
  4&\fbox{1123} \\ \hline
  5&11235 \\ \hline
  6&112358 \\ \hline
  7&11235813 \\ \hline
  8&1123581321 \\ \hline
  9&112358132134 \\ \hline
  10&11235813213455 \\ \hline
  11&1123581321345589 \\ \hline
  12&1123581321345589144 \\ \hline
  13&1123581321345589144233 \\ \hline
  14&1123581321345589144233377 \\ \hline
  15&1123581321345589144233377610 \\ \hline
  16&1123581321345589144233377610987 \\ \hline
  17&11235813213455891442333776109871597 \\ \hline
  18&112358132134558914423337761098715972584 \\ \hline
  19&1123581321345589144233377610987159725844181 \\ \hline
  20&11235813213455891442333776109871597258441816765 \\
  \hline
 \end{longtable}
\end{center}
where numbers in $\fbox{11}$, $\fbox{1123}$ are primes, \citep{Smarandache1975,Marimutha1997,Smarandache1997}.

\subsection{Back Concatenated Fibonacci Sequence}

With the commands: $L:=20$ $f_1:=1$ $f_2:=1$ $k:=3..L$ $f_{k}:=f_{k-1}+f_{k-2}$ $\emph{bcF}:=\emph{BConS}(f,l)$, resulting the vector:
\begin{center}
 \begin{longtable}{|r|c|}
   \caption{Back Concatenated Fibonacci Sequence}\\
   \hline
   $k$&$bcF_k$ \\
   \hline
  \endfirsthead
   \hline
   $k$&$bcF_k$ \\
   \hline
  \endhead
   \hline \multicolumn{2}{r}{\textit{Continued on next page}} \\
  \endfoot
   \hline
  \endlastfoot
  1&1 \\ \hline
  2&\fbox{11} \\ \hline
  3&\fbox{211} \\ \hline
  4&3211 \\ \hline
  5&53211 \\ \hline
  6&\fbox{853211} \\ \hline
  7&13853211 \\ \hline
  8&2113853211 \\ \hline
  9&342113853211 \\ \hline
  10&55342113853211 \\ \hline
  11&8955342113853211 \\ \hline
  12&1448955342113853211 \\ \hline
  13&2331448955342113853211 \\ \hline
  14&3772331448955342113853211 \\ \hline
  15&6103772331448955342113853211 \\ \hline
  16&9876103772331448955342113853211 \\ \hline
  17&15979876103772331448955342113853211 \\ \hline
  18&258415979876103772331448955342113853211 \\ \hline
  19&4181258415979876103772331448955342113853211 \\ \hline
  20&67654181258415979876103772331448955342113853211 \\
  \hline
 \end{longtable}
\end{center}

\subsection{Concatenated Tetranacci Sequence}

With the commands: $L:=20$ $t_1:=1$ $t_2:=1$ $t_3:=2$ $k:=4..L$ $t_{k}:=t_{k-1}+t_{k-2}+t_{k-3}$ $\emph{ct}:=\emph{ConS}(t,L)$ $\emph{bct}:=\emph{BConS}(t,L)$ resulting the vectors $\emph{ct}$ and $\emph{bct}$.

\begin{center}
 \begin{longtable}{|r|c|}
   \caption{Concatenated Tetranacci Sequence}\\
   \hline
   $k$&$ct_k$ \\
   \hline
  \endfirsthead
   \hline
   $k$&$ct_k$ \\
   \hline
  \endhead
   \hline \multicolumn{2}{r}{\textit{Continued on next page}} \\
  \endfoot
   \hline
  \endlastfoot
  1&1 \\
  2&\fbox{11} \\
  3&112 \\
  4&1124 \\
  5&11247 \\
  6&1124713 \\
  7&112471324 \\
  8&11247132444 \\
  9&1124713244481 \\
  10&1124713244481149 \\
  11&1124713244481149274 \\
  12&1124713244481149274504 \\
  13&1124713244481149274504927 \\
  14&11247132444811492745049271705 \\
  15&112471324448114927450492717053136 \\
  16&1124713244481149274504927170531365768 \\
  17&112471324448114927450492717053136576810609 \\
  18&11247132444811492745049271705313657681060919513 \\
  19&1124713244481149274504927170531365768106091951335890 \\
  20&112471324448114927450492717053136576810609195133589066012 \\
  \hline
 \end{longtable}
\end{center}

\begin{center}
 \begin{longtable}{|r|c|}
   \caption{Back Concatenated Tetranacci Sequence}\\
   \hline
   $k$&$bct_k$ \\
   \hline
  \endfirsthead
   \hline
   $k$&$bct_k$ \\
   \hline
  \endhead
   \hline \multicolumn{2}{r}{\textit{Continued on next page}} \\
  \endfoot
   \hline
  \endlastfoot
  1&1 \\ \hline
  2&\fbox{11} \\ \hline
  3&\fbox{211} \\ \hline
  4&\fbox{4211} \\ \hline
  5&74211 \\ \hline
  6&\fbox{1374211} \\ \hline
  7&241374211 \\ \hline
  8&44241374211 \\ \hline
  9&8144241374211 \\ \hline
  10&1498144241374211 \\ \hline
  11&2741498144241374211 \\\hline
  12&5042741498144241374211 \\ \hline
  13&9275042741498144241374211 \\ \hline
  14&17059275042741498144241374211 \\ \hline
  15&313617059275042741498144241374211 \\ \hline
  16&5768313617059275042741498144241374211 \\ \hline
  17&106095768313617059275042741498144241374211 \\ \hline
  18&19513106095768313617059275042741498144241374211 \\ \hline
  19&3589019513106095768313617059275042741498144241374211 \\ \hline
  20&660123589019513106095768313617059275042741498144241374211 \\
  \hline
 \end{longtable}
\end{center}

\subsection{Concatenated Mersenne Sequence}

With commands: $L:=17$\ \ $k:=1..L$\ \ $\ell M_k:=2^k-1$\ \ $rM_k:=2^k+1$\ \ $c\ell M:=\emph{ConS}(M\ell,L)$\ \ $\emph{crM}:=ConS(rM,L)$\ \ $\emph{bcM}\ell:=\emph{BConS}(\ell M,L)$\ \ $\emph{bcrM}:=\emph{BConS}(Mr,L)$\ resulting f{}iles:

\begin{center}
 \begin{longtable}{|r|c|}
   \caption{Concatenated Left Mersenne Sequence}\\
   \hline
   $k$&$c\ell M_k$ \\
   \hline
  \endfirsthead
   \hline
   $k$&$c\ell M_k$ \\
   \hline
  \endhead
   \hline \multicolumn{2}{r}{\textit{Continued on next page}} \\
  \endfoot
   \hline
  \endlastfoot
  1&1 \\ \hline
  2&\fbox{13}\\ \hline
  3&\fbox{137} \\ \hline
  4&13715 \\ \hline
  5&1371531 \\ \hline
  6&137153163 \\ \hline
  7&137153163127 \\ \hline
  8&137153163127255 \\ \hline
  9&\fbox{137153163127255511} \\ \hline
  10&1371531631272555111023 \\ \hline
  11&13715316312725551110232047 \\ \hline
  12&137153163127255511102320474095 \\ \hline
  13&1371531631272555111023204740958191 \\ \hline
  14&137153163127255511102320474095819116383 \\ \hline
  15&13715316312725551110232047409581911638332767 \\ \hline
  16&1371531631272555111023204740958191163833276765535 \\ \hline
  17&1371531631272555111023204740958191163833276765535131071 \\
  \hline
 \end{longtable}
\end{center}

\begin{center}
 \begin{longtable}{|r|c|}
   \caption{Back Concatenated Left Mersenne Sequence}\\
   \hline
   $k$&$bc\ell M_k$ \\
   \hline
  \endfirsthead
   \hline
   $k$&$bc\ell M_k$\\
   \hline
  \endhead
   \hline \multicolumn{2}{r}{\textit{Continued on next page}} \\
  \endfoot
   \hline
  \endlastfoot
  1&1 \\ \hline
  2&\fbox{31} \\ \hline
  3&731 \\ \hline
  4&\fbox{15731} \\ \hline
  5&3115731 \\ \hline
  6&633115731 \\ \hline
  7&127633115731 \\ \hline
  8&255127633115731 \\ \hline
  9&511255127633115731 \\ \hline
  10&1023511255127633115731 \\ \hline
  11&20471023511255127633115731 \\ \hline
  12&409520471023511255127633115731 \\ \hline
  13&8191409520471023511255127633115731 \\ \hline
  14&163838191409520471023511255127633115731 \\ \hline
  15&32767163838191409520471023511255127633115731 \\ \hline
  16&6553532767163838191409520471023511255127633115731 \\\hline
  17&1310716553532767163838191409520471023511255127633115731 \\
  \hline
 \end{longtable}
\end{center}

\begin{center}
 \begin{longtable}{|r|c|}
   \caption{Concatenated Right Mersenne Sequence}\\
   \hline
   $k$&$crM_k$ \\
   \hline
  \endfirsthead
   \hline
   $k$&$crM_k$\\
   \hline
  \endhead
   \hline \multicolumn{2}{r}{\textit{Continued on next page}} \\
  \endfoot
   \hline
  \endlastfoot
  1&\fbox{3} \\ \hline
  2&35 \\ \hline
  3&\fbox{359} \\ \hline
  4&35917 \\ \hline
  5&\fbox{3591733} \\ \hline
  6&359173365 \\ \hline
  7&359173365129 \\ \hline
  8&359173365129257 \\ \hline
  9&\fbox{359173365129257513} \\ \hline
  10&3591733651292575131025 \\ \hline
  11&35917336512925751310252049 \\ \hline
  12&359173365129257513102520494097 \\ \hline
  13&3591733651292575131025204940978193 \\ \hline
  14&359173365129257513102520494097819316385 \\ \hline
  15&35917336512925751310252049409781931638532769 \\ \hline
  16&3591733651292575131025204940978193163853276965537 \\ \hline
  17&3591733651292575131025204940978193163853276965537131073 \\
  \hline
 \end{longtable}
\end{center}

\begin{center}
 \begin{longtable}{|r|c|}
   \caption{Back Concatenated Right Mersenne Sequence}\\
   \hline
   $k$ & $bcrM_k$ \\
   \hline
  \endfirsthead
   \hline
   $k$ & $bcrM_k$  \\
   \hline
  \endhead
   \hline \multicolumn{2}{r}{\textit{Continued on next page}} \\
  \endfoot
   \hline
  \endlastfoot
  1&\fbox{3}\\ \hline
  2&\fbox{53} \\ \hline
  3&\fbox{953} \\ \hline
  4&17953 \\ \hline
  5&3317953 \\ \hline
  6&653317953 \\ \hline
  7&129653317953 \\ \hline
  8&257129653317953 \\ \hline
  9&513257129653317953 \\ \hline
  10&1025513257129653317953 \\ \hline
  11&20491025513257129653317953 \\ \hline
  12&409720491025513257129653317953 \\ \hline
  13&8193409720491025513257129653317953 \\ \hline
  14&163858193409720491025513257129653317953 \\ \hline
  15&32769163858193409720491025513257129653317953 \\ \hline
  16&6553732769163858193409720491025513257129653317953 \\ \hline
  17&1310736553732769163858193409720491025513257129653317953 \\
  \hline
 \end{longtable}
\end{center}

\subsection{Concatenated $6k-5$ Sequence}

With commands: $L:=25$\ \ $k:=1..L$\ \ $\emph{six}_k:=6k-5$\ \ $\emph{six}^\textrm{T}=$(1\ \ 7\ \ 13\ \ 19\ \ 25\ \ 31\ \ 37\ \ 43\ \ 49\ \ 55\ \ 61\ \ 67\ \ 73\ \ 79\ \ 85\ \ 91\ \ 97\ \ 103\ \ 109\ \ 115\ \ 121\ \ 127\ \ 133\ \ 139\ \ 145) result files $c6:=\emph{ConS}(\emph{six},L)$ and $\emph{bc6}:=\emph{BConS}(\emph{six},L)$.
\begin{center}
 \begin{longtable}{|r|c|}
   \caption{Concatenated $c6$ Sequence}\\
   \hline
   $k$ & $c6_k$ \\
   \hline
  \endfirsthead
   \hline
   $k$ & $c6_k$  \\
   \hline
  \endhead
   \hline \multicolumn{2}{r}{\textit{Continued on next page}} \\
  \endfoot
   \hline
  \endlastfoot
  1&1 \\ \hline
  2&\fbox{17} \\ \hline
  3&1713 \\ \hline
  4&171319 \\ \hline
  5&17131925 \\ \hline
  6&1713192531 \\ \hline
  7&171319253137 \\ \hline
  8&17131925313743 \\ \hline
  9&1713192531374349 \\ \hline
  10&171319253137434955 \\ \hline
  11&\fbox{17131925313743495561} \\ \hline
  12&1713192531374349556167 \\ \hline
  13&\fbox{171319253137434955616773} \\ \hline
  14&17131925313743495561677379 \\ \hline
  15&1713192531374349556167737985 \\ \hline
  16&171319253137434955616773798591 \\ \hline
  17&17131925313743495561677379859197 \\ \hline
  18&17131925313743495561677379859197103 \\ \hline
  19&17131925313743495561677379859197103109 \\ \hline
  20&17131925313743495561677379859197103109115 \\ \hline
  21&17131925313743495561677379859197103109115121 \\ \hline
  22&17131925313743495561677379859197103109115121127 \\ \hline
  23&17131925313743495561677379859197103109115121127133 \\ \hline
  24&17131925313743495561677379859197103109115121127133139 \\ \hline
  25&17131925313743495561677379859197103109115121127133139145 \\
  \hline
 \end{longtable}
\end{center}
where $\fbox{17}$, $\fbox{17131925313743495561}$ and $\fbox{171319253137434955616773}$ are primes.

\begin{center}
 \begin{longtable}{|r|c|}
   \caption{Back Concatenated $c6$ Sequence}\\
   \hline
   $k$ & $bc6_k$ \\
   \hline
  \endfirsthead
   \hline
   $k$ & $bc6_k$  \\
   \hline
  \endhead
   \hline \multicolumn{2}{r}{\textit{Continued on next page}} \\
  \endfoot
   \hline
  \endlastfoot
  1&1\\ \hline
  2&\fbox{71}\\ \hline
  3&1371\\ \hline
  4&191371\\ \hline
  5&25191371\\ \hline
  6&3125191371\\ \hline
  7&373125191371\\ \hline
  8&43373125191371\\ \hline
  9&4943373125191371\\ \hline
  10&554943373125191371\\ \hline
  11&61554943373125191371\\ \hline
  12&6761554943373125191371\\ \hline
  13&736761554943373125191371\\ \hline
  14&79736761554943373125191371\\ \hline
  15&8579736761554943373125191371\\ \hline
  16&918579736761554943373125191371\\ \hline
  17&97918579736761554943373125191371\\ \hline
  18&10397918579736761554943373125191371\\ \hline
  19&10910397918579736761554943373125191371\\ \hline
  20&11510910397918579736761554943373125191371\\ \hline
  21&12111510910397918579736761554943373125191371\\ \hline
  22&\fbox{12712111510910397918579736761554943373125191371}\\ \hline
  23&13312712111510910397918579736761554943373125191371\\ \hline
  24&13913312712111510910397918579736761554943373125191371\\ \hline
  25&14513913312712111510910397918579736761554943373125191371\\
  \hline
 \end{longtable}
\end{center}
where $\fbox{71}$ and $\fbox{12712111510910397918579736761554943373125191371}$ are primes.

\subsection{Concatenated Square Sequence}

$L:=23$\ \ $k:=1..L$\ \ $sq_k:=k^2$\ \ $\emph{csq}:=\emph{ConS}(sq,L)$.
\begin{center}
 \begin{longtable}{|r|c|}
   \caption{Concatenated Square Sequence}\\
   \hline
   $k$ & $csq_k$ \\
   \hline
  \endfirsthead
   \hline
   $k$ & $csq_k$  \\
   \hline
  \endhead
   \hline \multicolumn{2}{r}{\textit{Continued on next page}} \\
  \endfoot
   \hline
  \endlastfoot
  1&1\\ \hline
  2&14\\ \hline
  3&\fbox{149}\\ \hline
  4&14916\\ \hline
  5&1491625\\ \hline
  6&149162536\\ \hline
  7&14916253649\\ \hline
  8&1491625364964\\ \hline
  9&149162536496481\\ \hline
  10&149162536496481100\\ \hline
  11&149162536496481100121\\ \hline
  12&149162536496481100121144\\ \hline
  13&149162536496481100121144169\\ \hline
  14&149162536496481100121144169196\\ \hline
  15&149162536496481100121144169196225\\ \hline
  16&149162536496481100121144169196225256\\ \hline
  17&149162536496481100121144169196225256289\\ \hline
  18&149162536496481100121144169196225256289324\\ \hline
  19&149162536496481100121144169196225256289324361\\ \hline
  20&149162536496481100121144169196225256289324361400\\ \hline
  21&149162536496481100121144169196225256289324361400441\\ \hline
  22&149162536496481100121144169196225256289324361400441484\\ \hline
  23&149162536496481100121144169196225256289324361400441484529\\
  \hline
 \end{longtable}
\end{center}

\subsection{Back Concatenated Square Sequence}

$L:=23$\ \ $k:=1..L$\ \ $sq_k:=k^2$\ \ $\emph{bcsq}:=\emph{BConS}(sq,L)$.
\begin{center}
 \begin{longtable}{|r|c|}
   \caption{Back Concatenated Square Sequence}\\
   \hline
   $k$ & $bcsq_k$ \\
   \hline
  \endfirsthead
   \hline
   $k$ & $bcsq_k$  \\
   \hline
  \endhead
   \hline \multicolumn{2}{r}{\textit{Continued on next page}} \\
  \endfoot
   \hline
  \endlastfoot
  1&1\\ \hline
  2&\fbox{41}\\ \hline
  3&\fbox{941}\\ \hline
  4&16941\\ \hline
  5&2516941\\ \hline
  6&362516941\\ \hline
  7&49362516941\\ \hline
  8&6449362516941\\ \hline
  9&816449362516941\\ \hline
  10&100816449362516941\\ \hline
  11&121100816449362516941\\ \hline
  12&144121100816449362516941\\ \hline
  13&169144121100816449362516941\\ \hline
  14&\fbox{196169144121100816449362516941}\\ \hline
  15&225196169144121100816449362516941\\ \hline
  16&256225196169144121100816449362516941\\ \hline
  17&289256225196169144121100816449362516941\\ \hline
  18&324289256225196169144121100816449362516941\\ \hline
  19&361324289256225196169144121100816449362516941\\ \hline
  20&400361324289256225196169144121100816449362516941\\ \hline
  21&441400361324289256225196169144121100816449362516941\\ \hline
  22&484441400361324289256225196169144121100816449362516941\\ \hline
  23&529484441400361324289256225196169144121100816449362516941\\
  \hline
 \end{longtable}
\end{center}

\section{Permutation Sequence}

\begin{center}
 \begin{longtable}{|c|}
   \caption{Permutation sequence}\\
   \hline
   $n$\\
   \hline
  \endfirsthead
   \hline
   $n$\\
   \hline
  \endhead
   \hline \multicolumn{1}{r}{\textit{Continued on next page}} \\
  \endfoot
   \hline
  \endlastfoot
  12\\
  1342\\
  135642\\
  13578642\\
  13579108642\\
  135791112108642\\
  1357911131412108642\\
  13579111315161412108642\\
  135791113151718161412108642\\
  1357911131517192018161412108642\\
  \ldots\\
  \hline
 \end{longtable}
\end{center}
Questions:
\begin{enumerate}
  \item Is there any perfect power among these numbers?
  \item Their last digit should be: either 2 for exponents of the form $4k+1$, either 8 for exponents of the form $4k+3$, where $k\ge0$?
\end{enumerate}

\section{Generalized Permutation Sequence}

If $g:\Ns\to\Ns$, as a function, giving the number of digits of $a(n)$, and $F$ is a permutation of $g(n)$ elements, then:
$a(n)=F(1)F(2)\ldots F(g(n))$.

\section{Combinatorial Sequences}

Combinations of 4 taken by 3 for the set $\set{1,2,3,4}$: 123, 124, 134, 234, combinations of 5 taken by 3 for the set $\set{1,2,3,4,5}$: 123, 124, 125, 134, 135, 145, 234, 235, 245, 345, combinations of 6 taken by 3 for the set $\set{1,2,3,4,5,6}$: 123, 124, 125, 126, 134, 135, 136, 145, 146, 156, 234, 235, 236, 245, 246, 256, 345, 346, 356, 456, \ldots~.

Which is the set of 5 decimal digits for which we have the biggest number of primes for numbers obtained by combinations of 5 -- digits taken by 3?

We consider 5 digits of each in the numeration base 10, then, from mentioned sets, it results numbers by combinations of 5 -- digits taken by 3 as it follows:
\[
 \set{7,9,5,3,1}\Longrightarrow
   \begin{tabular}{|r|r|}
     \hline
       $n_{(10)}$ & factors\\
     \hline
       795 & $3\cdot5\cdot53$\\
       793 & $13\cdot61$\\
       791 & $7\cdot113$\\
       753 & $3\cdot251$\\
       751 & \fbox{751}\\
       731 & $17\cdot43$\\
       953 & \fbox{953}\\
       951 & $3\cdot317$\\
       931 & $7^2\cdot19$\\
       531 & $3^2\cdot59$\\
     \hline
\end{tabular}\ \ \
\set{9,7,5,3,1}\Longrightarrow
   \begin{tabular}{|r|r|}
     \hline
       $n_{(10)}$ & factors\\
     \hline
       975 & $3\cdot5^2\cdot13$\\
       973 & $7\cdot139$\\
       971 & \fbox{971}\\
       953 & \fbox{953}\\
       951 & $3\cdot317$\\
       931 & $7^2\cdot19$\\
       753 & $3\cdot251$\\
       751 & \fbox{751}\\
       731 & $17\cdot43$\\
       531 & $3^2\cdot59$\\
     \hline
\end{tabular}
\]

\[
 \set{5,7,1,9,3}\Longrightarrow
   \begin{tabular}{|r|r|}
     \hline
       $n_{(10)}$ & factors\\
     \hline
       571 & \fbox{571}\\
       579 & $3\cdot193$\\
       573 & $3\cdot191$\\
       519 & $3\cdot173$\\
       513 & $3^3\cdot19$\\
       593 & \fbox{593}\\
       719 & \fbox{719}\\
       713 & $23\cdot31$\\
       793 & $13\cdot61$\\
       193 & \fbox{193}\\
     \hline
\end{tabular}\ \ \
\set{3,1,5,7,9}\Longrightarrow
   \begin{tabular}{|r|r|}
     \hline
       $n_{(10)}$ & factors\\
     \hline
       315 & $3^2\cdot5\cdot7$\\
       317 & \fbox{317}\\
       319 & $11\cdot29$\\
       357 & $3\cdot7\cdot17$\\
       359 & \fbox{359}\\
       379 & \fbox{379}\\
       157 & \fbox{157}\\
       159 & $3\cdot53$\\
       179 & \fbox{179}\\
       579 & $3\cdot193$\\
     \hline
\end{tabular}
\]

\[
   \set{1,3,5,7,9}\Longrightarrow
   \begin{tabular}{|r|r|}
     \hline
       $n_{(10)}$ & factors\\
     \hline
       135 & $3^3\cdot5$\\
       137 & \fbox{137}\\
       139 & \fbox{139}\\
       157 & \fbox{157}\\
       159 & $3\cdot53$\\
       179 & \fbox{179}\\
       357 & $3\cdot7\cdot17$\\
       359 & \fbox{359}\\
       379 & \fbox{379}\\
       579 & $3\cdot193$\\
     \hline
   \end{tabular}\ \ \
   \set{1,4,3,7,9}\Longrightarrow
   \begin{tabular}{|r|r|}
     \hline
       $n_{(10)}$ & factors\\
     \hline
       143 & $11\cdot13$\\
       147 & $3\cdot7^2$\\
       149 & \fbox{149}\\
       137 & \fbox{137}\\
       139 & \fbox{139}\\
       179 & \fbox{179}\\
       437 & $19\cdot23$\\
       439 & \fbox{439}\\
       479 & \fbox{479}\\
       379 & \fbox{379}\\
     \hline
   \end{tabular}
\]

\[
   \set{1,3,0,7,9}\Longrightarrow
   \begin{tabular}{|r|r|}
     \hline
       $n_{(10)}$ & factors\\
     \hline
      130 & $2\cdot5\cdot13$\\
      137 & \fbox{137}\\
      139 & \fbox{139}\\
      107 & \fbox{107}\\
      109 & \fbox{109}\\
      179 & \fbox{179}\\
      307 & \fbox{307}\\
      309 & $3\cdot103$\\
      379 & \fbox{379}\\
       79 & \fbox{79}\\
     \hline
   \end{tabular}\ \ \
   \set{1,0,3,7,9}\Longrightarrow
   \begin{tabular}{|r|r|}
     \hline
       $n_{(10)}$ & factors\\
     \hline
       103 & \fbox{103}\\
       107 & \fbox{107}\\
       109 & \fbox{109}\\
       137 & \fbox{137}\\
       139 & \fbox{139}\\
       179 & \fbox{179}\\
        37 & \fbox{37}\\
        39 & $3\cdot13$\\
        79 & \fbox{79}\\
       379 & \fbox{379}\\
     \hline
   \end{tabular}
\]
In conclusion, it was determined that the set 1, 0, 3, 7, 9 generates primes by combining 5 -- digit taken by 3.

Generalization: which is the set of $m$ -- decimal digits for which we have the biggest number of primes for numbers obtained by combining $m$ -- digits taken by $n$?
\[
\set{2,4,0,6,8,5,1,3,7,9}\Longrightarrow
\left(\begin{array}{c}
         \vdots \\
  \set{2,0,6,8,3,9}\\
         \vdots
\end{array}\right)\Longrightarrow
   \begin{tabular}{|r|r|}
     \hline
       $n_{(10)}$ & factors\\
     \hline
       2068 & $2^2\cdot11\cdot47$\\
       2063 & \fbox{2063}\\
       2069 & \fbox{2069}\\
       2083 & \fbox{2083}\\
       2089 & \fbox{2089}\\
       2039 & \fbox{2039}\\
       2683 & \fbox{2683}\\
       2689 & \fbox{2689}\\
       2639 & $7\cdot13\cdot29$\\
       2839 & $17\cdot167$\\
        683 & \fbox{683}\\
        689 & $13\cdot53$\\
        639 & $3^2\cdot71$\\
        839 & \fbox{839}\\
       6839 & $7\cdot977$\\
     \hline
   \end{tabular}
\]

Of all $210=C_{10}^6$ combinations of 10 -- digits taken by 6, the most 4 -- digit primes are the numbers for digits 2, 0, 6, 8, 3, 9. Of all $15=C_6^4$ combinations of 6 -- digits taken by 4 we have 9 primes.

\section{Simple Numbers}

\begin{defn}[\citep{Smarandache2006}]\label{DefnSimpleNumbers}
  A number \emph{n} is called \emph{simple number} if the product of its proper divisors is less than or equal to \emph{n}.
\end{defn}

By analogy with the divisor function $\sigma_1(n)$, let
\begin{equation}\label{DivisorsProduct}
  \Pi_k(n)=\sum_{d\mid n}d~,
\end{equation}
denote the product of the divisors $d$ of $n$ (including $n$ itself). The function $\sigma_0:\Ns\to\Ns$ counted the divisors of $n$. If $n=\desp[\alpha]{m}$, then $\sigma_0(n)=(\alpha_1+1)(\alpha_2+1)\cdots(\alpha_m+1)$, \citep{MathWorldDivisorProduct}.

\begin{thm}\label{TheoremPiSigma0}
The divisor product \emph{(\ref{DivisorsProduct})} satisf{}ies the identity
\begin{equation}\label{IdentityPiSigma0}
  \Pi(n)=n^{\frac{\sigma_0(n)}{2}}~.
\end{equation}
\end{thm}
\begin{proof}
  Let $n=p_1^{\alpha_1}p_2^{\alpha_2}$, where $p_1,p_2\in\NP{2}$ si $n\in\Ns$, then divisors of $n$ are:
  \[
   \begin{array}{ccccc}
     p_1^0p_2^0~, & p_1^0p_2^1~, & p_1^0p_2^2~, & \ldots, & p_1^0p_2^{\alpha_2}~, \\
     p_1^1p_2^0~, & p_1^1p_2^1~, & p_1^1p_2^2~, & \ldots, & p_1^1p_2^{\alpha_2}~, \\
     p_1^2p_2^0~, & p_1^2p_2^1~, & p_1^2p_2^2~, & \ldots, & p_1^2p_2^{\alpha_2}~, \\
     \vdots & \vdots & \vdots &  & \vdots \\
     p_1^{\alpha_1}p_2^0~, & p_1^{\alpha_1}p_2^1~, & p_1^{\alpha_1}p_2^2~, & \ldots, & p_1^{\alpha_1}p_2^{\alpha_2}~.
   \end{array}
  \]

The products lines of divisors are:
\[
 \begin{array}{rcr}
    (p_1^0p_2^0)\cdot(p_1^0p_2^1)\cdot(p_1^0p_2^2)\cdots(p_1^0p_2^{\alpha_2})
    &=&p_1^0p_2^{\frac{\alpha2(\alpha2+1)}{2}}~, \\
    (p_1^1p_2^0)\cdot(p_1^1p_2^1)\cdot(p_1^1p_2^2)\cdots(p_1^1p_2^{\alpha_2})
    &=&p_1^{1(\alpha_2+1)}p_2^{\frac{\alpha2(\alpha2+1)}{2}}~, \\
    (p_1^2p_2^0)\cdot(p_1^2p_2^1)\cdot(p_1^2p_2^2)\cdots(p_1^2p_2^{\alpha_2})
    &=&p_1^{2(\alpha_2+1)}p_2^{\frac{\alpha2(\alpha2+1)}{2}}~, \\
    &\vdots \\
    (p_1^{\alpha_1}p_2^0)\cdot(p_1^{\alpha_1}p_2^1)\cdot(p_1^{\alpha_1}p_2^2)\cdots(p_1^{\alpha_1}p_2^{\alpha_2})
    &=&p_1^{\alpha_1(\alpha_2+1)}p_2^{\frac{\alpha2(\alpha2+1)}{2}}~.
 \end{array}
\]

Now we can write the product of all divisors:
\begin{multline*}
 \Pi(n)=(p_1^0p_2^0)\cdot(p_1^0p_2^1)\cdots(p_1^{\alpha_1}p_2^2)\cdots(p_1^{\alpha_1}p_2^{\alpha_2})\\
 =p_1^{\frac{\alpha_1(\alpha_1+1)}{2}(\alpha_2+1)}p_2^{(\alpha_1+1)\frac{\alpha_2(\alpha_2+1)}{2}}\\
 =p_1^{\frac{(\alpha_1+1)(\alpha_2+1)}{2}\alpha_1}p_2^{\frac{(\alpha_1+1)(\alpha_2+1)}{2}\alpha_2}\\
 =\left(p_1^{\alpha_1}p_2^{\alpha_2}\right)^{\frac{\sigma_0(n)}{2}}=n^{\frac{\sigma_0(n)}{2}}~.
\end{multline*}

By induction, it can be analogously proved the same identity for numbers that have the decomposition in m-prime factors $n=\desp[\alpha]{m}$.
\end{proof}

The table \ref{TabelLionnet} gives values of $n$ for which $\Pi(n)$ is a $m~th$ power. \cite{Lionnet1879}\index{Lionnet E.} considered the case $m=2$, \citep{MathWorldDivisorProduct}.
\begin{table}[h]
  \begin{center}
    \begin{tabular}{|c|c|l|}
     \hline
     $m$ & OEIS & $n$ \\
     \hline
     2 & \cite[A048943]{SloaneOEIS} &  1, 6, 8, 10, 14, 15, 16, 21, 22, 24, 26, \ldots  \\
     3 & \cite[A048944]{SloaneOEIS} & 1, 4, 8, 9, 12, 18, 20, 25, 27, 28, 32, \ldots  \\
     4 & \cite[A048945]{SloaneOEIS} & 1, 24, 30, 40, 42, 54, 56, 66, 70, 78, \ldots  \\
     5 & \cite[A048946]{SloaneOEIS}  & 1, 16, 32, 48, 80, 81, 112, 144, 162, \ldots \\
     \hline
   \end{tabular}
   \caption{Table which $\Pi(n)$ is a $m$--th power}\label{TabelLionnet}
  \end{center}
\end{table}

For $n=1,2,\ldots$ and $k=1$ the f{}irst few values are 1, 2, 3, 8, 5, 36, 7, 64, 27, 100, 11, 1728, 13, 196, \ldots, \cite[A007955]{SloaneOEIS}.

Likewise, we can def{}ine the function
\begin{equation}\label{ProperDivisorsProduct}
  P(n)=\sum_{d\mid n}d~,
\end{equation}
denoting the product of the proper divisors $d$ of $n$. Then def{}inition \ref{DefnSimpleNumbers} becomes
\begin{defn}\label{DefnSN}
  The number $n\in\Na$ is \emph{simple number} \dsnd{} $P(n)\le n$.
\end{defn}

We have a similar identity with (\ref{IdentityPiSigma0}), \cite[Ex. VI, p. 373]{Lucas1891}\index{Lucas E.}.
\begin{equation}\label{IdentityPSigma0}
  P(n)=n^{\frac{\sigma_0(n)}{2}-1}~.
\end{equation}

\begin{thm}\label{TheoremPSigma0}
 The product of proper divisors satisf{}ies the identity \emph{(\ref{IdentityPSigma0})}.
\end{thm}
\begin{proof}
 Let $n=p_1^{\alpha_1}p_2^{\alpha_2}$, where $p_1,p_2\in\NP{2}$ si $n\in\Ns$, then proper divisors of $n$ are:
  \[
   \begin{array}{cccccc}
      & p_1^0p_2^1~, & p_1^0p_2^2~, & \ldots, & p_1^0p_2^{\alpha_2-1}~, & p_1^0p_2^{\alpha_2}~,\\
     p_1^1p_2^0~, & p_1^1p_2^1~, & p_1^1p_2^2~, & \ldots, & p_1^1p_2^{\alpha_2-1}~, & p_1^1p_2^{\alpha_2}~,\\
     p_1^2p_2^0~, & p_1^2p_2^1~, & p_1^2p_2^2~, & \ldots, & p_1^2p_2^{\alpha_2-1}~, & p_1^2p_2^{\alpha_2}~,\\
     \vdots & \vdots & \vdots &  & \vdots \\
     p_1^{\alpha_1}p_2^0~, & p_1^{\alpha_1}p_2^1~, & p_1^{\alpha_1}p_2^2~, & \ldots, & p_1^{\alpha_1}p_2^{\alpha_2-1}~. &
   \end{array}
  \]

The products lines of divisors are:
\[
 \begin{array}{rcr}
    (p_1^0p_2^1)\cdot(p_1^0p_2^2)\cdots(p_1^0p_2^{\alpha_2-1})\cdot(p_1^0p_2^{\alpha_2})
    &=&p_1^0p_2^{\frac{\alpha2(\alpha2+1)}{2}}\\
    (p_1^1p_2^0)\cdot(p_1^1p_2^1)\cdot(p_1^1p_2^2)\cdots(p_1^1p_2^{\alpha_2-1})\cdot(p_1^1p_2^{\alpha_2})
    &=&p_1^{1(\alpha_2+1)}p_2^{\frac{\alpha2(\alpha2+1)}{2}}\\
    (p_1^2p_2^0)\cdot(p_1^2p_2^1)\cdot(p_1^2p_2^2)\cdots(p_1^2p_2^{\alpha_2-1})\cdot(p_1^2p_2^{\alpha_2})
    &=&p_1^{2(\alpha_2+1)}p_2^{\frac{\alpha2(\alpha2+1)}{2}}\\
    & \vdots & \\
    (p_1^{\alpha_1}p_2^0)\cdot(p_1^{\alpha_1}p_2^1)\cdot(p_1^{\alpha_1}p_2^2)\cdots(p_1^{\alpha_1}p_2^{\alpha_2-1})\ \ \ \ \ \ \ \ \ \ \
    &=&\frac{p_1^{\alpha_1(\alpha_2+1)}}{p_1^{\alpha_1}}\cdot\frac{p_2^{\frac{\alpha2(\alpha2+1)}{2}}}{p_2^{\alpha_2}}
 \end{array}
\]

Now we can write the product of all proper divisors:
\begin{multline*}
 P(n)=
 (p_1^0p_2^1)\cdots(p_1^{\alpha_1}p_2^2)\cdots(p_1^{\alpha_1}p_2^{\alpha_2-1})\\
 =\frac{p_1^{\frac{\alpha_1(\alpha_1+1)}{2}(\alpha_2+1)}}{p_1^{\alpha_1}} \frac{p_2^{(\alpha_1+1)\frac{\alpha_2(\alpha_2+1)}{2}}}{p_2^{\alpha_2}}
  =p_1^{\frac{\sigma_0(n)}{2}\alpha_1-\alpha_1}p_2^{\frac{\sigma_0(n)}{2}\alpha_2-\alpha_2}\\
 =\left(p_1^{\alpha_1}p_2^{\alpha_2}\right)^{\frac{\sigma_0(n)}{2}-1}=n^{\frac{\sigma_0(n)}{2}-1}~.
\end{multline*}

By induction, it can be analogously proved the same identity for numbers that have the decomposition in $m$ -- prime factors $n=\desp[\alpha]{m}$.
\end{proof}

\begin{obs}
  To note that $P(n)=1$ if and only if $n\in\NP{2}$.
\end{obs}

\begin{obs}
  Due to Theorem \emph{\ref{TheoremPSigma0}} one can give the following def{}inition of simple numbers: Any natural number that has more than 2 divisors by its own is a \emph{simple number}. Obviously, we can say that any natural number
that has more than 2 divisors by its own is a \emph{non-simple (complex) number}.
\end{obs}

Using the function \ref{ProperDivisorsProduct} one can write a simple program for determining the simple numbers: 2, 3, 4, 5, 6, 7, 8, 9, 10, 11, 13, 14, 15, 17, 19, 21, 22, 23, 25, 26, 27, 29, 31, 33, 34, 35, 37, 38, 39, 41, 43, 46, 47, 49, 51, 53, 55, 57, 58, 59, 61, 62, 65, 67, 69, 71, 73, 74, 77, 79, 82, 83, 85, 86, 87, 89, 91, 93, 94, 95, 97, 101, 103, 106, 107, 109, 111, 113, 115, 118, 119, 121, 122, 123, 125, 127, 129, 131, 133, 134, 137, 139, 141, 142, 143, 145, 146, 149, \ldots\\
or the non--simple (complex) numbers: 12, 16, 18, 20, 24, 28, 30, 32, 36, 40, 42, 44, 45, 48, 50, 52, 54, 56, 60, 63, 64, 66, 68, 70, 72, 75, 76, 78, 80, 81, 84, 88, 90, 92, 96, 98, 99, 100, 102, 104, 105, 108, 110, 112, 114, 116, 117, 120, 124, 126, 128, 130, 132, 135, 136, 138, 140, 144, 147, 148, 150, \ldots~.

How many simple numbers or non--simple (complex) numbers we have to the limit L=100, 1000, 10000, 20000, 30000,  40000,  50000, \ldots? The Table \ref{Simple-Nonsimple} answers the question:
\begin{table}[h]
  \centering
  \begin{tabular}{|r|r|r|}
     \hline
     L & Simple & Complex \\
     \hline
     100 & 61 & 38 \\
     1000 & 471 & 528 \\
     10000 & 3862 & 6137 \\
     20000 & 7352 & 12647 \\
     30000 & 10717 & 17282 \\
     40000 & 14004 & 25995 \\
     50000 & 17254 & 32745 \\
     \hline
   \end{tabular}
  \caption{How many simple numbers or non--simple}\label{Simple-Nonsimple}
\end{table}

\section{Pseudo--Smarandache Numbers}

Let $n$ be a positive natural number.
\begin{defn}\label{DefinitiaNumarPseudo-Smarandache}
  The \emph{pseudo--Smarandache} number of order $o$ ($o=1,2,\ldots$) of $n\in\Ns$ is the f{}irst natural number $m$ for which
  \begin{equation}\label{Sq(m)}
    S_o(m)=1^o+2^o+\ldots+m^o
  \end{equation}
  divides to $n$. The \emph{pseudo--Smarandache} number of f{}irst kind is simply called \emph{pseudo--Smarandache} number.
\end{defn}

\subsection{Pseudo--Smarandache Numbers of First Kind}
The \emph{pseudo--Smarandache} number of f{}irst kind to $L=50$ sunt: 1, 3, 2, 7, 4, 3, 6, 15, 8, 4, 10, 8, 12, 7, 5, 31, 16, 8, 18, 15, 6, 11, 22, 15, 24, 12, 26, 7, 28, 15, 30, 63, 11, 16, 14, 8, 36, 19, 12, 15, 40, 20, 42, 32, 9, 23, 46, 32, 48, 24~, obtained by calling the function $Z_1$ given by the program \ref{ProgramZ1}, $n:=1..L$, $Z_1(n)=$.

\begin{defn}\label{DefinitiaFunctieiZ1^k}
  The function $Z_1^k(n)=Z_1(Z_1(\ldots(Z_1(n))))$ is def{}ined, where composition repeats $k$--times.
\end{defn}

We present a list of issues related to the function $Z_1$, with total or partial solutions.

\begin{enumerate}
  \item Is the series
      \begin{equation}\label{Seria1Z1}
        \sum_{n=1}^\infty\frac{1}{Z_1(n)}
      \end{equation}
      convergent?

      Because
      \[
       \sum_{n=1}^\infty\frac{1}{Z_1(n)}\ge\sum_{n=1}^\infty\frac{1}{2n-1}=\infty~,
      \]
      it follows that the series (\ref{Seria1Z1}) is divergent.
  \item Is the series
      \begin{equation}\label{Seria2Z1}
       \sum_{n=1}^\infty\frac{Z_1(n)}{n}
      \end{equation}
      convergent?

      Because
      \[
       \sum_{n=1}^\infty\frac{Z_1(n)}{n}\ge\sum_{n=1}^\infty\frac{\sqrt{8n+1}-1}{2n}>\sum_{n=1}^\infty\frac{1}{n^\frac{1}{2}}~,
      \]
      and because the harmonic series
      \[
       \sum_{n=1}^\infty\frac{1}{n^\alpha}
      \]
      with $0<\alpha\le1$ is divergent, according to \citep{Acu2006}\index{Acu D.}, then it follows that the series (\ref{Seria2Z1}) is divergent.
  \item Is the series
  \begin{equation}\label{Seria3Z1}
       \sum_{n=1}^\infty\frac{Z_1(n)}{n^2}
      \end{equation}
      convergent?
  \item For a given pair of integers $k,m\in\Ns$, f{}ind all integers $n$ such that $Z_1^k(n)=m$. How many solutions are there?
     \begin{prog}\label{Program P2Z1} for determining the solutions of Diophantine equations $Z_1^2(n)=m$.
       \begin{tabbing}
         $\emph{P2Z1}(n_a,n_b,m_a,m_b):=$\=\vline\ $k\leftarrow1$\\
         \>\vline\ $f$\=$or\ m\in m_a..m_b$\\
         \>\vline\ \>\vline\ $j\leftarrow2$\\
         \>\vline\ \>\vline\ $f$\=$or\ n\in n_a..n_b$\\
         \>\vline\ \>\vline\ \>\ $i$\=$f\ m\textbf{=}Z_1(Z_1(n))$\\
         \>\vline\ \>\vline\ \>\ \>\vline\ $S_{k,j}\leftarrow n$\\
         \>\vline\ \>\vline\ \>\ \>\vline\ $j\leftarrow j+1$\\
         \>\vline\ \>\vline\ $i$\=$f\ j>3$\\
         \>\vline\ \>\vline\ \>\vline\ $S_{k,1}\leftarrow m$\\
         \>\vline\ \>\vline\ \>\vline\ $k\leftarrow k+1$\\
         \>\vline\ $\emph{return}\ S$\\
       \end{tabbing}
     \end{prog}
     Using the program \ref{Program P2Z1} we can determine all the solutions, for any $n\in\set{n_a,n_a+1,\ldots,n_b}$ and any $m\in\set{m_a,m_a+1,\ldots,m_b}$, where $m$ is the right part of the equation $Z_1^2(n)\textbf{=}m$. For example for $n\in\set{20,21,\ldots,100}$ and $m\in\set{12,13,\ldots,22}$ we give the solutions in Table \ref{SolutionZ1(2)=m}.
\begin{table}[h]
  \centering
  \begin{tabular}{|r|l|}
    \hline
    $m$ & $n$ \\ \hline
    12 & 27,\ 52,\ 79,\ 91;\\
    14 & 90;\\
    15 & 25,\ 31,\ 36,\ 41,\ 42,\ 50,\ 61,\ 70,\ 75,\ 82,\ 93,\ 100;\\
    16 & 51,\ 85;\\
    18 & 38,\ 95;\\
    20 & 43,\ 71;\\
    22 & 46,\ 69,\ 92;\\
    \hline
  \end{tabular}
  \caption{The solutions of Diophantine equations $Z_1^2(n)=m$}\label{SolutionZ1(2)=m}
\end{table}

Using a similar program \big(instead of condition $m\textbf{=}Z_1(Z_1(n))$ one puts the condition $m\textbf{=}Z_1(Z_1(Z_1(n)))$\big) we can obtain all the solutions for the equation $Z_1^3(n)=m$ with $n\in\set{20,21,\ldots,100}$ and $m\in\set{12,13,\ldots,22}$:
\begin{table}
  \centering
  \begin{tabular}{|r|l|}
    \hline
    m & n\\ \hline
    12 & 53,\ 54,\ 63;\\
    15 & 26,\ 37,\ 39,\ 43,\ 45,\ 57,\ 62,\ 65,\ 71,\ 74,\ 78,\ 83;\\
    20 & 86;\\
    22 & 47;\\
    \hline
  \end{tabular}~.
  \caption{The solutions of Diophantine equations $Z_1^3(n)=m$}\label{SOLutionZ1(3)=m}
\end{table}
  \item Are the following values bounded or unbounded
  \begin{enumerate}
    \item $\abs{Z_1(n+1)-Z_1(n)}$;
    \item $Z_1(n+1)/Z_1(n)$.
  \end{enumerate}
  We have the inequalities:
  \[
   \abs{Z_1(n+1)-Z_1(n)}\le\abs{2n+1-\left\lceil\frac{\sqrt{8n+1}-1}{2}\right\rceil}~,
  \]
  \[
   \frac{\lceil s_1(n+1)\rceil}{2n-1}\le\frac{Z_1(n+1)}{Z_1(n)}
   \le\frac{2n+1}{\lceil s_1(n)\rceil}~.
  \]
  \item Find all values of $n$ such that:
      \begin{enumerate}
        \item $Z_1(n)=Z_1(n+1)$~, for $n\in\set{1,2,\ldots,10^5}$ there is no $n,$ to verify the equality;
        \item $Z_1(n)\mid Z_1(n+1)$~, for $n\in\set{1,2,\ldots,10^3}$ obtain: 1, 6, 22, 28, 30, 46, 60, 66, 102, 120, 124, 138, 156, 166, 190, 262, 276, 282, 316, 348, 358, 378, 382, 399, 430, 478, 486, 498, 502, 508, 606, 630, 642, 700, 718, 732, 742, 750, 760, 786, 796, 822, 828, 838, 852, 858, 862, 886, 946, 979, 982;
        \item $Z_1(n+1)\mid Z_1(n)$~, for $n\in\set{1,2,\ldots,10^3}$ obtain: 9, 17, 25, 41, 49, 73, 81, 97, 113, 121, 169, 193, 233, 241, 257, 313, 337, 361, 401, 433, 457, 577, 593, 601, 617, 625, 673, 761, 841, 881, 977;
      \end{enumerate}
  \item Is there an algorithm that can be used to solve each of the following equations?
     \begin{enumerate}
       \item $Z_1(n)+Z_1(n+1)=Z_1(n+2)$ with the solutions: 609, 696, for $n\in\set{1,2,\ldots,10^3}$;
       \item $Z_1(n)=Z_1(n+1)+Z_1(n+2)$ with the solutions: 4, 13, 44, 83, 491, for $n\in\set{1,2,\ldots,10^3}$;
       \item $Z_1(n)\cdot Z_1(n+1)=Z_1(n+2)$ not are solutions, for $n\in\set{1,2,\ldots,10^3}$;
       \item $2\cdot Z_1(n+1)=Z_1(n)+Z_1(n+2)$ not are solutions, for $n\in\set{1,2,\ldots,10^3}$;
       \item $Z_1(n+1)^2=Z_1(n)\cdot Z_1(n+2)$ not are solutions, for $n\in\set{1,2,\ldots,10^3}$.
     \end{enumerate}
  \item There exists $n\in\Ns$ such that:
      \begin{enumerate}
        \item $Z_1(n)<Z_1(n+1)<Z_1(n+2)<Z_1(n+3)$? The following numbers, for $n\le10^3$, have the required propriety: 91, 159, 160, 164, 176, 224, 248, 260, 266, 308, 380, 406, 425, 469, 483, 484, 496, 551, 581, 590, 666, 695, 754, 790, 791, 805, 806, 812, 836, 903, 904. There are inf{}inite instances of 3 consecutive increasing terms in this sequence?
        \item $Z_1(n)>Z_1(n+1)>Z_1(n+2)>Z_1(n+3)$? Up to the limit $L=10^3$ we have the numbers: 97, 121, 142, 173, 214, 218, 219, 256, 257, 289, 302, 361, 373, 421, 422, 439, 529, 577, 578, 607, 669, 673, 686, 712, 751, 757, 761, 762, 773, 787, 802, 890, 907, 947 which verif{}ies the required condition. There are inf{}inite instances of 3 consecutive decreasing terms in this sequence?
        \item $Z_1(n)>Z_1(n+1)>Z_1(n+2)>Z_1(n+3)>Z_1(n+4)$? Up to $n\le10^3$ there are the numbers: 159, 483, 790, 805, 903, which verif{}ies the required condition. There are inf{}inite instances of 3 consecutive decreasing terms in this sequence?
        \item $Z_1(n)<Z_1(n+1)<Z_1(n+2)<Z_1(n+3)<Z_1(n+4)$? Up to $n\le10^3$ there are the numbers: 218, 256, 421, 577, 761, which verif{}ies the required condition. There are inf{}inite instances of 3 consecutive decreasing terms in this sequence?
      \end{enumerate}
  \item We denote by $S$, the \emph{Smarandache function}, the function that attaches to any $n\in\Ns$ the smallest natural number $m$ for which $m!$ is a multiple of $n$, \citep{Smarandache1980,Cira+Smarandache2014}. The question arises whether there are solutions to the equations:
      \begin{enumerate}
        \item $Z_1(n)=S(n)$? In general, if $Z_1(n)= S(n)=m$, then $n\mid[m(m+1)/2]$ and $n\mid m!$ must be satisf{}ied. So, in such cases, $m$ is sometimes the biggest prime factor of $n$, although that is not always the case. For $n\le10^2$ we have 19 solutions: 1, 6, 14, 15, 22, 28, 33, 38, 46, 51, 62, 66, 69, 86, 87, 91, 92, 94, 95~. There are an inf{}inite number of such solutions?
        \item $Z_1(n)=S(n)-1$? Let $p\in\NP{5}$. Since it is well-known that $S(p)=p$, for $p\in\NP{5}$, it follows from
            a previous result that $Z_1(p)+1=S(p)$. Of course, it is likely that other solutions may exist. Up to $10^2$ we have 37 solutions: 3, 5, 7, 10, 11, 13, 17, 19, 21, 23, 26, 29, 31, 34, 37, 39, 41, 43, 47, 53, 55, 57, 58, 59, 61, 67, 68, 71, 73, 74, 78, 79, 82, 83, 89, 93, 97~. Let us observe that there exists also primes as solutions, like: 10, 21, 26, 34, 39, 55, 57, 58, 68, 74, 78, 82, 93. There are an inf{}inite number of such solutions?
        \item $Z_1(n)=2\cdot S(n)$? This equation up to $10^3$ has 33 solutions: 12, 35, 85, 105, 117, 119, 185, 217, 235, 247, 279, 335, 351, 413, 485, 511, 535, 555, 595, 603, 635, 651, 685, 707, 741, 781, 785, 835, 893, 923, 925, 927, 985. In general, there are solutions for equation $Z_1(n)=k\cdot S(n)$, for $k=2,3,\ldots,16$ and $n\le10^3$. There are an inf{}inite number of such solutions?
      \end{enumerate}
\end{enumerate}

\subsection{Pseudo--Smarandache Numbers of Second Kind}

Pseudo--Smarandache numbers of second kind up to $L=50$ are: 1, 3, 4, 7, 2, 4, 3, 15, 13, 4, 5, 8, 6, 3, 4, 31, 8, 27, 9, 7, 13, 11, 11, 31, 12, 12, 40, 7, 14, 4, 15, 63, 22, 8, 7, 40, 18, 19, 13, 15, 20, 27, 21, 16, 27, 11, 23, 31, 24, 12~, obtained by calling the function $Z_2$ given by the program \ref{ProgramZ2}, $n:=1..L$, $Z_2(n)=$.

\begin{defn}\label{DefinitiaFunctieiZ2^k}
  The function $Z_2^k(n)=Z_2(Z_2(\ldots(Z_2(n))))$ is def{}ined by the composition which repeats of $k$ times.
\end{defn}

We present a list, similar to that of the function $Z_1$, of issues related to the function $Z_2$, with total or partial solutions.

\begin{enumerate}
  \item Is the series
      \begin{equation}\label{Seria1Z2}
        \sum_{n=1}^\infty\frac{1}{Z_2(n)}
      \end{equation}
      convergent?

      Because
      \[
       \sum_{n=1}^\infty\frac{1}{Z_2(n)}\ge\sum_{n=1}^\infty\frac{1}{2n-1}=\infty~,
      \]
      it follows that the series (\ref{Seria1Z2}) is divergent.
  \item Is the series
      \begin{equation}\label{Seria2Z2}
       \sum_{n=1}^\infty\frac{Z_2(n)}{n}
      \end{equation}
      convergent?

      Because
      \[
       \sum_{n=1}^\infty\frac{Z_2(n)}{n}\ge\sum_{n=1}^\infty\frac{s_2(n)}{n}>\sum_{n=1}^\infty\frac{1}{n^{\frac{2}{3}}}=\infty~,
      \]
      where $s_2(n)$ is given by (\ref{Solutia s2(n)}), and for harmonic series
      \[
       \sum_{n=1}^\infty\frac{1}{n^\alpha}
      \]
      with $0<\alpha\le1$ is divergent, according to \cite{Acu2006}\index{Acu D.}, then it follows that the series (\ref{Seria2Z2}) is divergent.
  \item Is the series
      \begin{equation}\label{Seria3Z2}
         \sum_{n=1}^\infty\frac{Z_2(n)}{n^2}
      \end{equation}
      convergent?
  \item For a given pair of integers $k,m\in\Ns$, f{}ind all integers $n$ such that $Z_2^k(n)=m$. How many solutions are there?

     Using a similar program with the program \ref{Program P2Z1} in which, instead of the condition $m=Z_1(Z_1(n))$ we put the condition $m=Z_2(Z_2(n))$. With this program we can determine all the solutions for any $n\in\set{n_a,n_a+1,\ldots,n_b}$ and any $m\in\set{m_a,m_a+1,\ldots,m_b}$, where $m$ is the right part of equation $Z_2^2(n)=m$. For example, for $n\in\set{20,21,\ldots,100}$ and $m\in\set{12,13,\ldots,22}$ we give the solution in the Table \ref{SolutionZ2(2)=m}.
\begin{table}[h]
  \centering
  \begin{tabular}{|r|l|}
    \hline
    $m$ & $n$  \\ \hline
    12 & 53;\\
    13 & 32,\ 43,\ 57,\ 79,\ 95,\ 96;\\
    14 & 59;\\
    15 & 24,\ 27,\ 34,\ 36,\ 48,\ 51,\ 54,\ 56,\ 60,\ 68,\ 84,\ 93;\\
    16 & 89;\\
    20 & 83;\\
    21 & 86;\\
    22 & 67;\\ \hline
  \end{tabular}~.
  \caption{The solutions of Diophantine equations $Z_2^2(n)=m$}\label{SolutionZ2(2)=m}
\end{table}

With a similar program \big(instead of condition $m=Z_2(Z_2(n))$ one puts the condition $m\textbf{=}Z_2(Z_2(Z_2(n)))$\big) one can obtain all solutions for equation $Z_2^3(n)=m$ with $n\in\set{20,21,\ldots,100}$ and $m\in\set{12,13,\ldots,22}$:
\begin{table}[h]
  \centering
  \begin{tabular}{|r|l|}
    \hline
    $m$ & $n$ \\ \hline
    13 & 38,\ 52,\ 64,\ 80,\ 86; \\
    15 & 25,\ 26,\ 42,\ 44,\ 45,\ 49,\ 50,\ 65,\ 66,\ 73,\ 74,\ 85,\ 88,\ 90,\ 97,\ 98,\ 99, 100;\\
    \hline
  \end{tabular}
  \caption{The solutions for equation $Z_2^3(n)=m$}\label{SolutionZ3(2)=m}
\end{table}
  \item Are the following values bounded or unbounded
  \begin{enumerate}
    \item $\abs{Z_2(n+1)-Z_2(n)}$;
    \item $Z_2(n+1)/Z_2(n)$.
  \end{enumerate}
  We have the inequalities:
  \[
   \abs{Z_2(n+1)-Z_2(n)}\le\abs{2n+1)-\lceil s_2(n)\rceil}~,
  \]
  \[
   \frac{\lceil s_2(n+1)\rceil}{2n-1}\le\frac{Z_2(n+1)}{Z_2(n)}\le\frac{2n+1}{\lceil s_2(n)\rceil}~.
  \]
  \item Find all values of $n$ such that:
      \begin{enumerate}
        \item $Z_2(n)=Z_2(n+1)$~, for $n\in\set{1,2,\ldots,5\cdot10^2}$ we obtained 10 solutions: 22, 25, 73, 121, 166, 262, 313, 358, 361, 457;
        \item $Z_2(n)\mid Z_2(n+1)$~, for $n\in\set{1,2,\ldots,5\cdot10^2}$ we obtained 21 solutions: 1, 5, 7, 22, 25, 28, 51, 70, 73, 95, 121, 143, 166, 190, 262, 313, 358, 361, 372, 457, 473;
        \item $Z_2(n+1)\mid Z_2(n)$~, for $n\in\set{1,2,\ldots,5\cdot10^2}$ we obtained 28 solutions: 13, 18, 22, 25, 49, 54, 61, 73, 97, 109, 121, 128, 157, 162, 166, 174, 193, 218, 241, 262, 289, 313, 337, 358, 361, 368, 397, 457;
      \end{enumerate}
  \item Is there an algorithm that can be used to solve each of the following equations?
     \begin{enumerate}
       \item $Z_2(n)+Z_2(n+1)=Z_2(n+2)$ with the solutions: 1, 2, for $n\le10^3$;
       \item $Z_2(n)=Z_2(n+1)+Z_2(n+2)$ with the solutions: 78, 116, 582, for $n\le10^3$;
       \item $Z_2(n)\cdot Z_2(n+1)=Z_2(n+2)$ not are solutions, for $n\in\le10^3$;
       \item $2\cdot Z_2(n+1)=Z_2(n)+Z_2(n+2)$ with the solution: 495, for $n\le10^3$;
       \item $Z_2(n+1)^2=Z_2(n)\cdot Z_2(n+2)$ not are solutions, for $n\le10^3$.
     \end{enumerate}
  \item There is $n\in\Ns$ such that:
      \begin{enumerate}
        \item $Z_2(n)<Z_2(n+1)<Z_2(n+2)<Z_2(n+3)$? The following 24 numbers, for $n\le5\cdot10^2$, have the required propriety: 1, 39, 57, 111, 145, 146, 147, 204, 275, 295, 315, 376, 380, 381, 391, 402, 406, 425, 445, 477, 494, 495, 496, 497. There are inf{}inite instances of 3 consecutive increasing terms in this sequence?
        \item $Z_2(n)>Z_2(n+1)>Z_2(n+2)>Z_2(n+3)$? Up to the limit $L=5\cdot10^2$ we have 21 numbers: 32, 48, 60, 162, 183, 184, 192, 193, 218, 228, 256, 257, 282, 332, 333, 342, 362, 422, 448, 449, 467, which verif{}ies the required condition. There are inf{}inite instances of 3 consecutive decreasing terms in this sequence?
        \item $Z_2(n)>Z_2(n+1)>Z_2(n+2)>Z_2(n+3)>Z_2(n+4)$? Up to $n\le10^3$ there are the numbers: 145, 146, 380, 494, 495, 496, 610, 805, 860, 930, 994, 995,which verif{}ies the required condition. There are inf{}inite instances of 3 consecutive increasing terms in this sequence?
        \item $Z_2(n)<Z_2(n+1)<Z_2(n+2)<Z_2(n+3)<Z_2(n+4)$? Up to $n\le10^3$ there are the numbers: 183, 192, 256, 332, 448, 547, 750, 751, which verif{}ies the required condition. There are inf{}inite instances of 3 consecutive increasing terms in this sequence?
      \end{enumerate}
  \item We denote by $S$, the \emph{Smarandache function}, i.e. the function which attaches to any $n\in\Ns$ the smallest natural number $m$ for which $m!$ is a multiple of $n$, \citep{Smarandache1980,Cira+Smarandache2014}. The question arises whether there are solutions to the equations:
      \begin{enumerate}
        \item $Z_2(n)=S(n)$? In general, if $Z_2(n)= S(n)=m$, then $n\mid[m(m+1)(2m+1)/6]$ and $n\mid m!$ must be satisf{}ied. So, in such cases, $m$ is sometimes the biggest prime factor of $n$, although that is not always the case. For $n\le3\cdot10^2$ we have 30 solutions: 1, 22, 28, 35, 38, 39, 70, 85, 86, 92, 93, 117, 118, 119, 134, 140, 166, 185, 186, 190, 201, 214, 217, 235, 247, 255, 262, 273, 278, 284~. There are an inf{}inite number of such solutions?
        \item $Z_2(n)=S(n)-1$? Let $p\in\NP{5}$. Since it is well--known that $S(p)=p$, for $p\in\NP{5}$, it follows from a previous result that $Z_1(p)+1=S(p)$. Of course, it is likely that other solutions may exist. Up to $3\cdot10^2$ we have 26 solutions: 10, 15, 26, 30, 58, 65, 69, 74, 77, 106, 115, 122, 123, 130, 136, 164, 177, 187, 202, 215, 218, 222, 246, 265, 292, 298. There are an inf{}inite number of such solutions?
        \item $Z_2(n)=2\cdot S(n)$? This equation up to $2\cdot10^2$ has 7 solutions: 12, 33, 87, 141, 165, 209, 249. In general, there exists solutions for equation $Z_2(n)=k\cdot S(n)$, for $k=2,3,4,5$ and $n\le10^3$. There are an inf{}inite number of such solutions?
      \end{enumerate}
\end{enumerate}

\subsection{Pseudo--Smarandache Numbers of Third Kind}

The \emph{pseudo--Smarandache} numbers of third rank up to $L=50$ are: 1, 3, 2, 3, 4, 3, 6, 7, 2, 4, 10, 3, 12, 7, 5, 7, 16, 3, 18, 4, 6, 11, 22, 8, 4, 12, 8, 7, 28, 15, 30, 15, 11, 16, 14, 3, 36, 19, 12, 15, 40, 20, 42, 11, 5, 23, 46, 8, 6, 4~, obtained by calling the function given by the program \ref{ProgramZ3}, $n:=1..L$, $Z_3(n)=$.

\begin{defn}\label{DefinitiaFunctieiZ3^k}
  In the function $Z_3^k(n)=Z_3(Z_3(\ldots(Z_3(n))))$ the composition repeats $k$ times.
\end{defn}

We present a list, similar to that of function $Z_1$, with issues concerning the function $Z_3$, with total or partial solutions.

\begin{enumerate}
   \item Is the series
      \begin{equation}\label{Seria1Z3}
        \sum_{n=1}^\infty\frac{1}{Z_3(n)}
      \end{equation}
      convergent?

      Because
      \[
       \sum_{n=1}^\infty\frac{1}{Z_3(n)}\ge\sum_{n=1}^\infty\frac{1}{n-1}=\infty~,
      \]
      it follows that the series (\ref{Seria1Z3}) is divergent.
  \item  Is the series
      \begin{equation}\label{Seria2Z3}
       \sum_{n=1}^\infty\frac{Z_3(n)}{n}
      \end{equation}
      convergent?

      Because
      \[
       \sum_{n=1}^\infty\frac{Z_3(n)}{n}\ge\sum_{n=1}^\infty\frac{s_3(n)}{n}>\sum_{n=1}^\infty\frac{1}{n^\frac{3}{4}}=\infty~,
      \]
      therefore it follows that the series (\ref{Seria2Z3}) is divergent (see Figure \ref{FigSeria2Z3}).
\begin{figure}[h]
  \centering
  \includegraphics[scale=0.75]{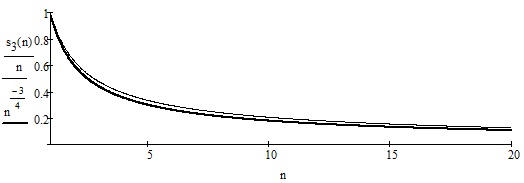}\\
  \caption{The functions $s_3(n)\cdot n^{-1}$ and $n^{-\frac{3}{4}}$}\label{FigSeria2Z3}
\end{figure}

  \item Is series
      \begin{equation}\label{Seria3Z3}
         \sum_{n=1}^\infty\frac{Z_3(n)}{n^2}
      \end{equation}
      convergent?
  \item For a given pair of integers $k,m\in\Ns$, f{}ind all integers $n$ such that $Z_3^k(n)=m$. How many solutions are there?

     Using a program similar to the program \ref{Program P2Z1} where instead of condition $m=Z_1(Z_1(n))$ we put the condition $m=Z_3(Z_3(n))$. Using this program, we can determine all solutions, or any $n\in\set{n_a,n_a+1,\ldots,n_b}$ and any $m\in\set{m_a,m_a+1,\ldots,m_b}$, where $m$ is the right member of the equation $Z_3^2(n)=m$. For example, for $n\in\set{20,21,\ldots,100}$ and $m\in\set{12,13,\ldots,22}$ we have the solutions in Table \ref{SolutionZ32(n)=m}:
\begin{table}[h]
  \centering
  \begin{tabular}{|r|l|}
    \hline
    $m$ & $n$  \\ \hline
    12 & 53,\ 79,\ 91;\\
    15 & 31,\ 41,\ 61,\ 82,\ 88,\ 93,\ 97;\\
    16 & 51,\ 85;\\
    18 & 38,\ 76,\ 95;\\
    20 & 43,\ 71;\\
    22 & 46,\ 69,\ 92;\\ \hline
  \end{tabular}~.
  \caption{The solutions of Diophantine equations $Z_3^2(n)=m$}\label{SolutionZ32(n)=m}
\end{table}

  Using a similar program \big(instead of condition $m=Z_3(Z_3(n))$ we put the condition $m\textbf{=}Z_3(Z_3(Z_3(n)))$\big) we can obtain all solutions for the equation $Z_3^3(n)=m$ with $n\in\set{20,21,\ldots,100}$ and $m\in\set{12,13,\ldots,22}$:
\begin{table}[h]
  \centering
  \begin{tabular}{|r|l|}
    \hline
    $m$ & $n$ \\ \hline
    15 & 62,\ 83,\ 89;\\
    20 & 86;\\
    22 & 47;\\ \hline
  \end{tabular}
  \caption{The solutions of Diophantine equations $Z_3^3(n)=m$}\label{SolutionZ33(n)=m}
\end{table}
  \item Are the following values bounded or unbounded
  \begin{enumerate}
    \item $\abs{Z_3(n+1)-Z_3(n)}$;
    \item $Z_3(n+1)/Z_3(n)$.
  \end{enumerate}
  We have the inequalites:
  \[
   \abs{Z_3(n+1)-Z_3(n)}\le\abs{(n+1)^2-s_3(n)}~,
  \]
  \[
   \frac{s_3(n+1)}{n^2}\le\frac{Z_3(n+1)}{Z_3(n)} \le\frac{(n+1)^2}{s_3(n)}~.
  \]
  \item Find all values of $n$ such that:
      \begin{enumerate}
        \item $Z_3(n)=Z_3(n+1)$~, for $n\in\set{1,2,\ldots,10^3}$ does not have solutions;
        \item $Z_3(n)\mid Z_3(n+1)$~, for $n\in\set{1,2,\ldots,3\cdot10^2}$ we obtained 35 solutions: 1, 6, 9, 12, 18, 22, 25, 28, 30, 36, 46, 60, 66, 72, 81, 100, 102, 112, 121, 138, 147, 150, 156, 166, 169, 172, 180, 190, 196, 198, 240, 262, 268, 276, 282;
        \item $Z_3(n+1)\mid Z_3(n)$~, for $n\in\set{1,2,\ldots,10^3}$ we obtained 25 solutions: 24, 31, 41, 73, 113, 146, 168, 193, 257, 313, 323, 337, 401, 433, 457, 506, 575, 577, 601, 617, 673, 728, 761, 881, 977;
      \end{enumerate}
  \item Is there an algorithm that can be used to solve each of the following equations?
     \begin{enumerate}
       \item $Z_3(n)+Z_3(n+1)=Z_3(n+2)$ with the solutions: 24, 132, 609, 979, for $n\in\set{1,2,\ldots,10^3}$;
       \item $Z_3(n)=Z_3(n+1)+Z_3(n+2)$ with the solution: 13, for $n\in\set{1,2,\ldots,10^3}$;
       \item $Z_3(n)\cdot Z_3(n+1)=Z_3(n+2)$ not are solutions, for $n\in\set{1,2,\ldots,10^3}$;
       \item $2\cdot Z_3(n+1)=Z_3(n)+Z_3(n+2)$ with the solution: 3, 48, 318, 350, for $n\in\set{1,2,\ldots,10^3}$;
       \item $Z_3(n+1)^2=Z_3(n)\cdot Z_3(n+2)$ not are solutions, for $n\in\set{1,2,\ldots,10^3}$.
     \end{enumerate}
  \item There exists $n\in\Ns$ such that:
      \begin{enumerate}
        \item $Z_3(n)<Z_3(n+1)<Z_3(n+2)<Z_3(n+3)$? The next 25 number, for$n\le8\cdot10^2$, have the required property: 20, 56, 91, 164, 175, 176, 236, 308, 350, 380, 405, 406, 468, 469, 496, 500, 644, 650, 656, 666, 679, 680, 715, 716, 775. Are there inf{}initely many instances of 3 consecutive increasing terms in this sequence?
        \item $Z_3(n)>Z_3(n+1)>Z_3(n+2)>Z_3(n+3)$? Up to the limit $L=8\cdot10^2$ we have 21 numbers: 47, 109, 113, 114, 118, 122, 123, 157, 181, 193, 257, 258, 317, 397, 401, 402, 487, 526, 534, 541, 547, 613, 622, 634, 669, 701, 723, 757, 761, 762, which verif{}ies the required condition. Are there inf{}initely many instances of 3 consecutive decreasing terms in this sequence?
        \item $Z_3(n)>Z_3(n+1)>Z_3(n+2)>Z_3(n+3)>Z_3(n+4)$? Up to $n\le10^3$ there are the numbers: 175, 405, 468, 679, 715, 805, 903, which verif{}ies the required condition. Are there inf{}initely many instances of 3 consecutive increasing terms in this sequence?
        \item $Z_3(n)<Z_3(n+1)<Z_3(n+2)<Z_3(n+3)<Z_3(n+4)$? Up to $n\le10^3$ there are the numbers: 113, 122, 257, 401, 761, 829, which verif{}ies the required condition. Are there inf{}initely many instances of 3 consecutive increasing terms in this sequence?
      \end{enumerate}
  \item We denote by $S$, the \emph{Smarandache function}, i.e. the function that attach to any $n\in\Ns$ the smallest natural number $m$ for which $m!$ is a multiple of $n$, \citep{Smarandache1980,Cira+Smarandache2014}. The question arises whether there are solutions to the equations:
      \begin{enumerate}
        \item $Z_3(n)=S(n)$? In general, if $Z_3(n)= S(n)=m$, then $n\mid[m(m+1)/2]$ and $n\mid m!$ must be satisf{}ied. So, in such cases, $m$ is sometimes the biggest prime factor of $n$, although that is not always the case. For $n\le10^2$ we have 23 solutions: 1, 6, 14, 15, 22, 28, 33, 38, 44, 46, 51, 56, 62, 66, 69, 76, 86, 87, 91, 92, 94, 95, 99~. There are an inf{}inite number of such solutions?
        \item $Z_3(n)=S(n)-1$? Let $p\in\NP{5}$. Since it is well-known that $S(p)=p$, for $p\in\NP{5}$, it follows from a previous result that $Z_3(p)+1=S(p)$. Of course, it is likely that other solutions may exist. Up to $10^2$ we have 46 solutions: 3, 4, 5, 7, 10, 11, 12, 13, 17, 19, 20, 21, 23, 26, 27, 29, 31, 34, 37, 39, 41, 43, 45, 47, 52, 53, 54, 55, 57, 58, 59, 61, 63, 67, 68, 71, 73, 74, 78, 79, 81, 82, 83, 89, 93, 97. There are an inf{}inite number of such solutions?
        \item $Z_3(n)=2\cdot S(n)$? This equation up to $10^2$ has 3 solutions: 24, 35, 85. There are an inf{}inite number of such solutions?
      \end{enumerate}
\end{enumerate}

\section{General Residual Sequence}

Let $x$, $n$ two integer numbers. The \emph{general residual} function is the product between $(x+C_1)(x+C_2)\cdots(x+C_{\varphi(n)})$, where $C_k$ are the residual class of $n$ which are relative primes to $n$. As we know, the relative prime factors of the number $n$ are in number of $\varphi(n)$, where $\varphi$ is Euler\rq{s}\index{Euler L.} totient function, \citep{WeissteinTotientFunction}. We can def{}ine the function $GR:\mathbb{Z}\times\mathbb{Z}\to\mathbb{Z}$,
\begin{equation}\label{Functia general residual}
  GR(x,n)=\prod_{k=1}^{\varphi(n)}(x+C_k),
\end{equation}
where the residual class $C_k\ \md{n}$ is the class $C_k$ for which we have $(C_k,n)=1$ ($C_k$ relative prime to $n$), \citep{Smarandache1992,Smarandache1995}.

The fact that the residual class $C_k$ is relative prime to $n$ can be written also in the form $gcd(C_k,n)=1$, i.e. the greatest common divisor of $C_k$ and $n$ is 1.

\begin{prog}\label{ProgramGR} General residual function
  \begin{tabbing}
    $GR(x,n):=$\=\vline\ $gr\leftarrow1$\\
    \>\vline\ $f$\=$or\ k\in1..n$\\
    \>\vline\ \>\ $gr\leftarrow gr(x+k)\ \ \emph{if}\ \ gcd(k,n)\textbf{=}1$\\
    \>\vline\ $\emph{return}\ \ gr$
  \end{tabbing}
\end{prog}
Let $n=2,3,\ldots,20$. General residual sequence for $x=0$ is:
\begin{multline*}
  1, 2, 3, 24, 5, 720, 105, 2240, 189, 3628800, 385, 479001600, 19305, 896896,\\
  2027025, 20922789888000, 85085, 6402373705728000, 8729721,
\end{multline*}
and for $x=1$ is:
\begin{multline*}
  2, 6, 8, 120, 12, 5040, 384, 12960, 640, 39916800, 1152, 6227020800,\\
  80640, 5443200, 10321920, 355687428096000, 290304,\\
  121645100408832000, 38707200
\end{multline*}
and for $x=2$ is:
\begin{multline*}
  3, 12, 15, 360, 21, 20160, 945, 45360, 1485, 239500800, 2457,\\
  43589145600, 225225, 20217600, 34459425, 3201186852864000, 700245,\\
  1216451004088320000, 115540425~.
\end{multline*}

\section{Goldbach--Smarandache Table}

Goldbach\rq{s} conjecture asserts that any even number $>4$ is the sum of two primes, \citep{Oliveira2016,WeissteinGoldbachConjecture}.

Let be the sequence of numbers be:
\begin{itemize}
  \item[] $t_1=6$ the largest even number such that any other even number, not exceeding it, is the sum of two of the f{}irst 1 (one) odd prime 3; $6=3+3$;
  \item[] $t_2=10$ the largest even number such that any other even number, not exceeding it, is the sum of two of the f{}irst 2 (two) odd primes 3, 5; $6=3+3$, $8=3+5$;
  \item[] $t_3=14$ the largest even number such that any other even number, not exceeding it, is the sum of two of the f{}irst 3 (three) odd primes 3, 5, 7; $3+3=6$, $3+5=8$, $5+5=10$, $7+5=12$; $7+7=14$;
  \item[] $t_4=18$ the largest even number such that any other even number, not exceeding it, is the sum of two of the f{}irst 4 (four) odd primes 3, 5, 7, 11; $6=3+3$, $8=3+5$, $10=3+7$, $12=5+7$, $14=7+7$, $16=5+11$, $18=11+7$;
  \item[] $t_5=$ \ldots~.
\end{itemize}
Thus we have the sequence: 6, 10, 14, 18, 26, 30, 38, 42, 42, 54, 62, 74, 74, 90, \ldots~.

Table \ref{GoldbachSmarandacheTable} contains the sum $\emph{prime}_k+\emph{prime}_j$, $k=5,6,\ldots, 15$, $j=\overline{2,k}$, where $\emph{prime}$ is the vector of prime numbers obtained by the program \ref{ProgramSEPC} using the call $\emph{prime}:=\emph{SEPC}(47)$. From this table one can see any even number $n$ ($11<n<47$) to what amount of prime numbers equals.
\begin{table}[h]
  \centering
  \begin{tabular}{|r|rrrrrrrrrrr|}
  \hline
  + & 11 & 13 & 17 & 19 & 23 & 29 & 31 & 37 & 41 & 43 & 47\\ \hline
  3 & 14 & 16 & 20 & 22 & 26 & 32 & 34 & 40 & 44 & 46 & 50\\
  5 & 16 & 18 & 22 & 24 & 28 & 34 & 36 & 42 & 46 & 48 & 52\\
  7 & 18 & 20 & 24 & 26 & 30 & 36 & 38 & 44 & 48 & 50 & 54\\
  11 & 22 & 24 & 28 & 30 & 34 & 40 & 42 & 48 & 52 & 54 & 58\\
  13 & 0 & 26 & 30 & 32 & 36 & 42 & 44 & 50 & 54 & 56 & 60\\
  17 & 0 & 0 & 34 & 36 & 40 & 46 & 48 & 54 & 58 & 60 & 64\\
  19 & 0 & 0 & 0 & 38 & 42 & 48 & 50 & 56 & 60 & 62 & 66\\
  23 & 0 & 0 & 0 & 0 & 46 & 52 & 54 & 60 & 64 & 66 & 70\\
  29 & 0 & 0 & 0 & 0 & 0 & 58 & 60 & 66 & 70 & 72 & 76\\
  31 & 0 & 0 & 0 & 0 & 0 & 0 & 62 & 68 & 72 & 74 & 78\\
  37 & 0 & 0 & 0 & 0 & 0 & 0 & 0 & 74 & 78 & 80 & 84\\
  41 & 0 & 0 & 0 & 0 & 0 & 0 & 0 & 0 & 82 & 84 & 88\\
  43 & 0 & 0 & 0 & 0 & 0 & 0 & 0 & 0 & 0 & 86 & 90\\
  47 & 0 & 0 & 0 & 0 & 0 & 0 & 0 & 0 & 0 & 0 & 94\\
  \hline
\end{tabular}
  \caption{Goldbach--Smarandache table}\label{GoldbachSmarandacheTable}
\end{table}

Table \ref{GoldbachSmarandacheTable} is generated by the program:
\begin{prog}\label{Program GSt} Goldbach--Smarandache.
   \begin{tabbing}
     $\emph{GSt}(a,b):=$\=\vline\ $\emph{prime}\leftarrow SEPC(b)$\\
     \>\vline\ $f$\=$or\ k\in2..\emph{last}(prime)$\\
     \>\vline\ \>\ $i$\=$f\ \emph{prime}_k\ge a$\\
     \>\vline\ \>\ \>\vline\ $ka\leftarrow k$\\
     \>\vline\ \>\ \>\vline\ $\emph{break}$\\
     \>\vline\ $f$\=$or\ k\in ka..\emph{last(prime)}$\\
     \>\vline\ \>\ $f$\=$or\ j\in2..k$\\
     \>\vline\ \>\ \>\ $g_{j-1,k-ka+1}\leftarrow \emph{prime}_k+\emph{prime}_j$\\
     \>\vline\ $\emph{return}\ \ g$\\
   \end{tabbing}
  The call of the program $\emph{GSt}$ for obtaining the Table \ref{GoldbachSmarandacheTable} is $\emph{GSt}(11,50)$.
\end{prog}

We present a program that determines all possible combinations (apart from the addition\rq{s} commutativity) of sums of two prime numbers that are equal to the given even number.
\begin{prog}\label{Program SGSt} search Goldbach--Smarandache table.
  \begin{tabbing}
    $\emph{SGSt}(n):=$\=\vline\ $\emph{return}\ \ "\emph{Error}\ n\ \emph{is}\ \emph{odd}"\ \ \emph{if}\ \ \mod(n,2)\neq0$\\
    \>\vline\ $\emph{prime}\leftarrow \emph{SEPC}(n+2)$\\
    \>\vline\ $f$\=$or\ k\in2..\frac{n}{2}$\\
    \>\vline\ \>\ $i$\=$f\ \emph{prime}_k\ge\frac{n}{2}$\\
    \>\vline\ \>\ \>\vline\ $ka\leftarrow k$\\
    \>\vline\ \>\ \>\vline\ $\emph{break}$\\
    \>\vline\ $h\leftarrow1$\\
    \>\vline\ $f$\=$or\ k\in \emph{last(prime)}..ka$\\
    \>\vline\ \>\ $f$\=$or\ j\in2..k$\\
    \>\vline\ \>\ \>\ $i$\=$f\ n\textbf{=}\emph{prime}_k+\emph{prime}_j$\\
    \>\vline\ \>\ \>\ \>\vline\ $g_{h,1}\leftarrow n$\\
    \>\vline\ \>\ \>\ \>\vline\ $g_{h,2}\leftarrow "="$\\
    \>\vline\ \>\ \>\ \>\vline\ $g_{h,3}\leftarrow \emph{prime}_k$\\
    \>\vline\ \>\ \>\ \>\vline\ $g_{h,4}\leftarrow "+"$\\
    \>\vline\ \>\ \>\ \>\vline\ $g_{h,5}\leftarrow \emph{prime}_j$\\
    \>\vline\ \>\ \>\ \>\vline\ $h\leftarrow h+1$\\
    \>\vline\ \>\ \>\ \>\vline\ $\emph{break}$\\
    \>\vline\ $\emph{return}\ \ g$\\
  \end{tabbing}
\end{prog}

We present two of the program \ref{Program SGSt}, calls, for numbers $556$ and $346$. The number $556$ decompose in 11 sums of prime numbers, and $346$ in $9$ sums of prime numbers.
\[\emph{SGSt}(556)=\left(\begin{array}{rrrrr}
                    556 & "=" & 509 & "+" & 47\\
                    556 & "=" & 503 & "+" & 53\\
                    556 & "=" & 467 & "+" & 89\\
                    556 & "=" & 449 & "+" & 107\\
                    556 & "=" & 443 & "+" & 113\\
                    556 & "=" & 419 & "+" & 137\\
                    556 & "=" & 389 & "+" & 167\\
                    556 & "=" & 383 & "+" & 173\\
                    556 & "=" & 359 & "+" & 197\\
                    556 & "=" & 317 & "+" & 239\\
                    556 & "=" & 293 & "+" & 263\\
              \end{array}\right)~.
\]

\[\emph{SGSt}(346)=\left(\begin{array}{rrrrr}
                    346 & "=" & 317 & "+" & 29\\
                    346 & "=" & 293 & "+" & 53\\
                    346 & "=" & 263 & "+" & 83\\
                    346 & "=" & 257 & "+" & 89\\
                    346 & "=" & 239 & "+" & 107\\
                    346 & "=" & 233 & "+" & 113\\
                    346 & "=" & 197 & "+" & 149\\
                    346 & "=" & 179 & "+" & 167\\
                    346 & "=" & 173 & "+" & 173\\
              \end{array}\right)~.
\]

The following program counts for any natural even number the number of possible unique decomposition (apart from the addition\rq{s} commutativity) in the sum of two primes.
\begin{prog}\label{Program NGSt} for counting the decompositions of $n$ (natural even number) in the sum of two primes.
   \begin{tabbing}
     $\emph{NGSt}(n):=$\=\vline\ $\emph{return}\ "\emph{Err}.\ n\ \emph{is}\ \emph{odd}"\ \emph{if}\ \mod(n,2)=1$\\
     \>\vline\ $\emph{prime}\leftarrow \emph{SEPC}(n)$\\
     \>\vline\ $h\leftarrow0$\\
     \>\vline\ $f$\=$or\ k\in \emph{last(prime)}..2$\\
     \>\vline\ \>\ $f$\=$or\ j\in k..2$\\
     \>\vline\ \>\ \>\ $i$\=$f\ n\textbf{=}\emph{prime}_k+\emph{prime}_j$\\
     \>\vline\ \>\ \>\ \>\vline\ $h\leftarrow h+1$\\
     \>\vline\ \>\ \>\ \>\vline\ $\emph{break}$\\
     \>\vline\ $\emph{return}\ h$
   \end{tabbing}
\end{prog}
The call of this program using the commands:
\[
 n:=6,8..100\ \ \ ngs_{\frac{n}{2}-2}:=NGSt(n)
\]
will provide the Goldbach--Smarandache series:
\begin{multline*}
  ngs^\textrm{T}=(1, 1, 2, 1, 2, 2, 2, 2, 3, 3, 3, 2, 3, 2, 4, 4, 2, 3, 4, 3, 4, 5, 4,\\ 3, 5, 3, 4, 6, 3, 5, 6, 2, 5, 6, 5, 5, 7, 4, 5, 8, 5, 4, 9, 4, 5, 7, 3, 6)~.
\end{multline*}

\section{Vinogradov--Smarandache Table}

Vinogradov conjecture: Every suf{}f{}iciently large odd number is a sum of three primes. Vinogradov\index{Vinogradov A. I.} proved in 1937 that any odd number greater than $3^{3^{15}}$ satisf{}ies this conjecture.

Waring\rq{s} prime number conjecture: Every odd integer $n$ is a prime or the sum of three primes.

Let be the sequence of numbers:
\begin{itemize}
  \item[] $v_1=$ is the largest odd number such that any odd number $\ge9$, not exceeding it, is the sum of three of the f{}irst 1 (one) odd prime, i.e. the odd prime 3;
  \item[] $v_2=$ is the largest odd number such that any odd number $\ge9$, not exceeding it, is the sum of three of the f{}irst 2 (two) odd prime, i.e. the odd primes 3, 5;
  \item[] $v_3=$ is the largest odd number such that any odd number $\ge9$, not exceeding it, is the sum of three of the f{}irst 3 (three) odd prime, i.e. the odd primes 3, 5, 7;
  \item[] $v_4=$ is the largest odd number such that any odd number $\ge9$, not exceeding it, is the sum of three of the f{}irst 4 (four) odd prime, i.e. the odd primes 3, 5, 7, 11;
  \item[] $v_5=$ \ldots~.
\end{itemize}

Thus we have the sequence: 9, 15, 21, 29, 39, 47, 57, 65, 71, 93, 99, 115, 129, 137, 143, 149, 183, 189, 205, 219, 225, 241, 251, 269, 287, 309, 317, 327, 335, 357, 371, 377, 417, 419, 441, 459, 465, 493, 503, 509, 543, 545, 567, 587, 597, 609, 621, 653, 657, 695, 701, 723, 725, 743, 749, 755, 785 \ldots~, \citep[A007962]{SloaneOEIS}.

The table gives you in how many dif{}ferent combinations an odd number is written as a sum of three odd primes, and in what combinations.
\begin{prog}\label{Program VSt} that generates the Vinogradov--Smarandache Table with $p$ prime f{}ixed between the limits $a$ and $b$.
   \begin{tabbing}
     $\emph{VSt}(p,a,b):=$\=\vline\ $\emph{prime}\leftarrow \emph{SEPC}(b)$\\
     \>\vline\ $f$\=$or\ k\in2..\emph{last(prime)}$\\
     \>\vline\ \>\ $i$\=$f\ \emph{prime}_k\ge a$\\
     \>\vline\ \>\ \>\vline\ $ka\leftarrow k$\\
     \>\vline\ \>\ \>\vline\ $\emph{break}$\\
     \>\vline\ $f$\=$or\ k\in ka..\emph{last(prime)}$\\
     \>\vline\ \>\ $f$\=$or\ j\in2..k$\\
     \>\vline\ \>\ \>\ $g_{j-1,k-ka+1}\leftarrow p+\emph{prime}_k+\emph{prime}_j$\\
     \>\vline\ $\emph{return}\ \ g$\\
   \end{tabbing}
 With this program, we generate the Vinogradov--Smarandache tables for $p=3,5,7,11$ and $a=13$, $b=45$.
\end{prog}
\[\emph{VSt}(3,13,45)=\left(\begin{array}{cccccccccc}
                       19 & 23 & 25 & 29 & 35 & 37 & 43 & 47 & 49\\
                       21 & 25 & 27 & 31 & 37 & 39 & 45 & 49 & 51\\
                       23 & 27 & 29 & 33 & 39 & 41 & 47 & 51 & 53\\
                       27 & 31 & 33 & 37 & 43 & 45 & 51 & 55 & 57\\
                       29 & 33 & 35 & 39 & 45 & 47 & 53 & 57 & 59\\
                       0 & 37 & 39 & 43 & 49 & 51 & 57 & 61 & 63\\
                       0 & 0 & 41 & 45 & 51 & 53 & 59 & 63 & 65\\
                       0 & 0 & 0 & 49 & 55 & 57 & 63 & 67 & 69\\
                       0 & 0 & 0 & 0 & 61 & 63 & 69 & 73 & 75\\
                       0 & 0 & 0 & 0 & 0 & 65 & 71 & 75 & 77\\
                       0 & 0 & 0 & 0 & 0 & 0 & 77 & 81 & 83\\
                       0 & 0 & 0 & 0 & 0 & 0 & 0 & 85 & 87\\
                       0 & 0 & 0 & 0 & 0 & 0 & 0 & 0 & 89\\
                 \end{array}\right)~,
\]

\[\emph{VSt}(5,13,45)=\left(\begin{array}{cccccccccc}
                       21 & 25 & 27 & 31 & 37 & 39 & 45 & 49 & 51\\
                       23 & 27 & 29 & 33 & 39 & 41 & 47 & 51 & 53\\
                       25 & 29 & 31 & 35 & 41 & 43 & 49 & 53 & 55\\
                       29 & 33 & 35 & 39 & 45 & 47 & 53 & 57 & 59\\
                       31 & 35 & 37 & 41 & 47 & 49 & 55 & 59 & 61\\
                       0 & 39 & 41 & 45 & 51 & 53 & 59 & 63 & 65\\
                       0 & 0 & 43 & 47 & 53 & 55 & 61 & 65 & 67\\
                       0 & 0 & 0 & 51 & 57 & 59 & 65 & 69 & 71\\
                       0 & 0 & 0 & 0 & 63 & 65 & 71 & 75 & 77\\
                       0 & 0 & 0 & 0 & 0 & 67 & 73 & 77 & 79\\
                       0 & 0 & 0 & 0 & 0 & 0 & 79 & 83 & 85\\
                       0 & 0 & 0 & 0 & 0 & 0 & 0 & 87 & 89\\
                       0 & 0 & 0 & 0 & 0 & 0 & 0 & 0 & 91\\
                     \end{array}\right)~,
\]

\[\emph{VSt}(7,13,45)=\left(\begin{array}{cccccccccc}
                       23 & 27 & 29 & 33 & 39 & 41 & 47 & 51 & 53\\
                       25 & 29 & 31 & 35 & 41 & 43 & 49 & 53 & 55\\
                       27 & 31 & 33 & 37 & 43 & 45 & 51 & 55 & 57\\
                       31 & 35 & 37 & 41 & 47 & 49 & 55 & 59 & 61\\
                       33 & 37 & 39 & 43 & 49 & 51 & 57 & 61 & 63\\
                       0 & 41 & 43 & 47 & 53 & 55 & 61 & 65 & 67\\
                       0 & 0 & 45 & 49 & 55 & 57 & 63 & 67 & 69\\
                       0 & 0 & 0 & 53 & 59 & 61 & 67 & 71 & 73\\
                       0 & 0 & 0 & 0 & 65 & 67 & 73 & 77 & 79\\
                       0 & 0 & 0 & 0 & 0 & 69 & 75 & 79 & 81\\
                       0 & 0 & 0 & 0 & 0 & 0 & 81 & 85 & 87\\
                       0 & 0 & 0 & 0 & 0 & 0 & 0 & 89 & 91\\
                       0 & 0 & 0 & 0 & 0 & 0 & 0 & 0 & 93\\
                     \end{array}\right)~,
\]

\[\emph{VSt}(11,13,45)=\left(\begin{array}{cccccccccc}
                       27 & 31 & 33 & 37 & 43 & 45 & 51 & 55 & 57\\
                       29 & 33 & 35 & 39 & 45 & 47 & 53 & 57 & 59\\
                       31 & 35 & 37 & 41 & 47 & 49 & 55 & 59 & 61\\
                       35 & 39 & 41 & 45 & 51 & 53 & 59 & 63 & 65\\
                       37 & 41 & 43 & 47 & 53 & 55 & 61 & 65 & 67\\
                       0 & 45 & 47 & 51 & 57 & 59 & 65 & 69 & 71\\
                       0 & 0 & 49 & 53 & 59 & 61 & 67 & 71 & 73\\
                       0 & 0 & 0 & 57 & 63 & 65 & 71 & 75 & 77\\
                       0 & 0 & 0 & 0 & 69 & 71 & 77 & 81 & 83\\
                       0 & 0 & 0 & 0 & 0 & 73 & 79 & 83 & 85\\
                       0 & 0 & 0 & 0 & 0 & 0 & 85 & 89 & 91\\
                       0 & 0 & 0 & 0 & 0 & 0 & 0 & 93 & 95\\
                       0 & 0 & 0 & 0 & 0 & 0 & 0 & 0 & 97\\
                     \end{array}\right)~.
\]

We present a program that determines all possible combinations of Vi\-no\-gra\-dov--Smarandache Tables, so that the odd number to be written as the sum of 3 primes. Because the Vinogradov--Smarandache Tables are tridimensional, we f{}ix the f{}irst factor in the triplet of prime numbers so that the program will determine the other two prime numbers, so that the sum of 3 primes to be $n$.
\begin{prog}\label{Program SVSt} search Vinogradov--Smarandache table.
  \begin{tabbing}
    $\emph{SVSt}(p,n):=$\=\vline\ $\emph{return}\ \ "\emph{Error}\ n\ \emph{is}\ \emph{odd}"\ \ \emph{if}\ \ \mod(n,2)\textbf{=}0$\\
    \>\vline\ $\emph{prime}\leftarrow \emph{SEPC}(n-p)$\\
    \>\vline\ $f$\=$or\ k\in2..\frac{n+1}{2}$\\
    \>\vline\ \>\ $i$\=$f\ \emph{prime}_k\ge\frac{n+1}{2}$\\
    \>\vline\ \>\ \>\vline\ $ka\leftarrow k$\\
    \>\vline\ \>\ \>\vline\ $\emph{break}$\\
    \>\vline\ $h\leftarrow1$\\
    \>\vline\ $f$\=$or\ k\in \emph{last(prime)}..ka$\\
    \>\vline\ \>\ $f$\=$or\ j\in2..k$\\
    \>\vline\ \>\ \>\ $i$\=$f\ n\textbf{=}p+\emph{prime}_k+\emph{prime}_j$\\
    \>\vline\ \>\ \>\ \>\vline\ $g_{h,1}\leftarrow n$\\
    \>\vline\ \>\ \>\ \>\vline\ $g_{h,2}\leftarrow "="$\\
    \>\vline\ \>\ \>\ \>\vline\ $g_{h,3}\leftarrow p$\\
    \>\vline\ \>\ \>\ \>\vline\ $g_{h,4}\leftarrow "+"$\\
    \>\vline\ \>\ \>\ \>\vline\ $g_{h,5}\leftarrow prime_k$\\
    \>\vline\ \>\ \>\ \>\vline\ $g_{h,6}\leftarrow "+"$\\
    \>\vline\ \>\ \>\ \>\vline\ $g_{h,7}\leftarrow prime_j$\\
    \>\vline\ \>\ \>\ \>\vline\ $h\leftarrow h+1$\\
    \>\vline\ \>\ \>\ \>\vline\ $\emph{break}$\\
    \>\vline\ $\emph{return}\ \ g$\\
  \end{tabbing}
\end{prog}

For illustration we determine all possible cases of three prime numbers that add up to be 559..

\[\emph{SVSt}(3,559)=\left(\begin{array}{ccccccc}
                      559 & "=" & 3 & "+" & 509 & "+" & 47\\
                      559 & "=" & 3 & "+" & 503 & "+" & 53\\
                      559 & "=" & 3 & "+" & 467 & "+" & 89\\
                      559 & "=" & 3 & "+" & 449 & "+" & 107\\
                      559 & "=" & 3 & "+" & 443 & "+" & 113\\
                      559 & "=" & 3 & "+" & 419 & "+" & 137\\
                      559 & "=" & 3 & "+" & 389 & "+" & 167\\
                      559 & "=" & 3 & "+" & 383 & "+" & 173\\
                      559 & "=" & 3 & "+" & 359 & "+" & 197\\
                      559 & "=" & 3 & "+" & 317 & "+" & 239\\
                      559 & "=" & 3 & "+" & 293 & "+" & 263\\
                \end{array}\right)~,
\]

\[\emph{SVSt}(5,559)=\left(\begin{array}{ccccccc}
                      559 & "=" & 5 & "+" & 547 & "+" & 7\\
                      559 & "=" & 5 & "+" & 541 & "+" & 13\\
                      559 & "=" & 5 & "+" & 523 & "+" & 31\\
                      559 & "=" & 5 & "+" & 487 & "+" & 67\\
                      559 & "=" & 5 & "+" & 457 & "+" & 97\\
                      559 & "=" & 5 & "+" & 397 & "+" & 157\\
                      559 & "=" & 5 & "+" & 373 & "+" & 181\\
                      559 & "=" & 5 & "+" & 331 & "+" & 223\\
                      559 & "=" & 5 & "+" & 313 & "+" & 241\\
                      559 & "=" & 5 & "+" & 283 & "+" & 271\\
                      559 & "=" & 5 & "+" & 277 & "+" & 277\\
                    \end{array}\right)~,
\]
\[
 \vdots
\]
\[
 \emph{SVSt}(181,559)=\left(\begin{array}{ccccccc}
                        559, & "=" & 181 & "+" & 197 & "+" & 181\\
                      \end{array}\right)~.
\]
\[
 \emph{SVSt}(191,559)=0~.
\]

\begin{prog}\label{Program NVSt} for counting the decomposition of $n$ (odd natural numbers $n\ge3$) in the sum of three primes.
   \begin{tabbing}
     $\emph{NVSt}(n):=$\=\vline\ $\emph{return}\ "\emph{Error}\ n\ \emph{is}\ \emph{even}"\ \ \emph{if}\ \ \mod(n,2)=0$\\
     \>\vline\ $\emph{prime}\leftarrow \emph{SEPC}(n+1)$\\
     \>\vline\ $h\leftarrow0$\\
     \>\vline\ $f$\=$or\ k\in \emph{last(prime)}..2$\\
     \>\vline\ \>\ $f$\=$or\ j\in k..2$\\
     \>\vline\ \>\ \>\ $f$\=$or\ i\in j..2$\\
     \>\vline\ \>\ \>\ \>\ $i$\=$f\ n\textbf{=}\emph{prime}_k+\emph{prime}_j+\emph{prime}_i$\\
     \>\vline\ \>\ \>\ \>\ \>\vline\ $h\leftarrow h+1$\\
     \>\vline\ \>\ \>\ \> \>\vline\ $\emph{break}$\\
     \>\vline\ $\emph{return}\ \ h$
   \end{tabbing}
\end{prog}
The call of this program using the controls
\[
 n:=3,5..100\ \ \ \emph{nvs}_{\frac{n-1}{2}}:=\emph{NVSt}(n)
\]
will provide the Vinogradov--Smarandache series
\begin{multline*}
  nvs^\textrm{T}=(0, 0, 0, 1, 1, 2, 2, 3, 3, 4, 4, 5, 6, 7, 6, 8, 7, 9, 10, 10, 10, 11, 12, 12,\\
  14, 16, 14, 16, 16, 16, 18, 20, 20, 20, 21, 21, 21, 27, 24, 25, 28,\\
  27, 28, 33, 29, 32, 35, 34, 30)~.
\end{multline*}

\section{Smarandacheian Complements}

Let $g:A\to A$ be a strictly increasing function, and let "$\sim$" be a given internal law on $A$. Then $f:A\to A$ is a smarandacheian complement with respect to the function $g$ and the internal law "$\sim$" if: $f(x)$ is the smallest $k$ such that there exists a $z$ in $A$ so that $x\sim k=g(z)$.

\subsection{Square Complements}

\begin{defn}\label{Definita Square Complements}
  For each integer $n$ to f{}ind the smallest integer $k$ such that $k\cdot n$ is a perfect square.
\end{defn}

\begin{obs}
  All these numbers are square free.
\end{obs}

Numbers square complements in between $10$ and $10^2$ are: 10, 11, 3, 13, 14, 15, 1, 17, 2, 19, 5, 21, 22, 23, 6, 1, 26, 3, 7, 29, 30, 31, 2, 33, 34, 35, 1, 37, 38, 39, 10, 41, 42, 43, 11, 5, 46, 47, 3, 1, 2, 51, 13, 53, 6, 55, 14, 57, 58, 59, 15, 61, 62, 7, 1, 65, 66, 67, 17, 69, 70, 71, 2, 73, 74, 3, 19, 77, 78, 79, 5, 1, 82, 83, 21, 85, 86, 87, 22, 89, 10, 91, 23, 93, 94, 95, 6, 97, 2, 11, 1~. These were generated with the program \ref{Program m-power Complements}, using the command: $\emph{mC}(2,10,100)^\textrm{T}=$~.

\subsection{Cubic Complements}

\begin{defn}\label{Definita Cubic Complements}
  For each integer $n$ to f{}ind the smallest integer $k$ such that $k\cdot n$ is a perfect cub.
\end{defn}

\begin{obs}
  All these numbers are cube free.
\end{obs}

Numbers cubic complements in between $1$ and $40$ are: 1, 2, 3, 2, 5, 6, 7, 1, 3, 10, 11, 12, 13, 14, 15, 16, 17, 18, 19, 20, 21, 22, 23, 24, 5, 26, 1, 28, 29, 30, 31, 32, 33, 34, 35, 36, 37, 38, 39, 40~. These were generated with the program \ref{Program m-power Complements}, using the command: $\emph{mC}(3,1,40)^\textrm{T}=$~.

\subsection{$m-$power Complements}

\begin{defn}\label{Definita m-power Complements}
  For each integer $n$ to f{}ind the smallest integer $k$ such that $k\cdot n$ is a $m-$power.
\end{defn}

\begin{obs}
  All these numbers are $m-$power free.
\end{obs}

\begin{prog}\label{Program m-power Complements} for generating the numbers $m-$power complements.
  \begin{tabbing}
    $\emph{mC}(m,n_a,n_b):=$\=\vline\ $f$\=$or\ n\in n_a..n_b$\\
    \>\vline\ \>\vline\ $f$\=$or\ k\in 1..n$\\
    \>\vline\ \>\vline\ \>\vline\ $kn\leftarrow k\cdot n$\\
    \>\vline\ \>\vline\ \>\vline\ $break\ \ \emph{if}\ \ \emph{trunc}\big(\sqrt[m]{kn}\big)^m\textbf{=}kn$\\
    \>\vline\ \>\vline\ $mc_{n-n_a+1}\leftarrow k$\\
    \>\vline\ $\emph{return}\ \emph{mc}$\\
  \end{tabbing}
  which uses Mathcad function $\emph{trunc}$.
\end{prog}

Numbers $5-$power complements in between $25$ and $65$ are: 25, 26, 9, 28, 29, 30, 31, 1, 33, 34, 35, 36, 37, 38, 39, 40, 41, 42, 43, 44, 45, 46, 47, 48, 49, 50, 51, 52, 53, 54, 55, 56, 57, 58, 59, 60, 61, 62, 63, 16, 65~. These were generated with the program \ref{Program m-power Complements}, using the command: $\emph{mC}(5,25,65)^\textrm{T}=$~.

\section{$m-$factorial Complements}

\begin{defn}\label{Definitia m-factorial complements}
  For each $n\in\Ns$ to f{}ind the smallest $k$ such that $k\cdot n$ is a $m-$factorial, where $m\in\Ns$~.
\end{defn}

\begin{prog}\label{Program m-factorial complements} for generating the series $m-$factorial complements.
  \begin{tabbing}
    $\emph{mfC}(m,n_a,n_b):=$\=\vline\ $f$\=$or\ n\in n_a..n_b$\\
    \>\vline\ \>\ $f$\=$or\ j\in1..n$\\
    \>\vline\ \>\ \>\vline\ $k\leftarrow\dfrac{kf(j,m)}{n}$\\
    \>\vline\ \>\ \>\vline\ $i$\=$f\ k\textbf{=}\emph{trunc}(k)$\\
    \>\vline\ \>\ \>\vline\ \>\vline\ $\emph{mfc}_{n-n_a+1}\leftarrow k$\\
    \>\vline\ \>\ \>\vline\ \>\vline\ $\emph{break}$\\
    \>\vline\ $\emph{return}\ \ \emph{mfc}$\\
  \end{tabbing}
  which uses the program \ref{Programkf} and Mathcad function $\emph{trunc}$.
\end{prog}

\begin{exem}
  The series of \emph{factorial complements} numbers from 1 to 25 is obtained with the command $\emph{mfc}(1,1,25)$: 1, 1, 2, 6, 24, 1, 720, 3, 80, 12, 3628800, 2, 479001600, 360, 8, 45, 20922789888000, 40, 6402373705728000, 6, 240, 1814400, 5288917409079652, 1, 145152~.
\end{exem}
\begin{exem}
  The series of \emph{double factorial complements} numbers from 1 to 30 is obtained with the command $\emph{mfc}(2,1,30)$: 1, 1, 1, 2, 3, 8, 15, 1, 105, 384, 945, 4, 10395, 46080, 1, 3, 2027025, 2560, 34459425, 192, 5, 3715891200, 13749310575, 2, 81081, 1961990553600, 35, 23040, 213458046676875, 128~.
\end{exem}
\begin{exem}
  The series of \emph{triple factorial complements} numbers from 1 to 35 is obtained with the command $\emph{mfc}(3,1,35)$: 1, 1, 1, 1, 2, 3, 4, 10, 2, 1, 80, 162, 280, 2, 1944, 5, 12320, 1, 58240, 4, 524880, 40, 4188800, 81, 167552, 140, 6, 1, 2504902400, 972, 17041024000, 385, 214277011200, 6160, 8~.
\end{exem}

\section{Prime Additive Complements}

\begin{defn}\label{Definitia Prime Additive Complements}
  For each $n\in\Ns$ to f{}ind the smallest $k$ such that $n+k\in\NP{2}$.
\end{defn}
\begin{prog}\label{Program prime additive Complements} for generating the series of \emph{prime additive complement}s numbers.
  \begin{tabbing}
    $\emph{paC}(n_a,n_b):=$\=\vline\ $f$\=$or\ n\in n_a..n_b$\\
    \>\vline\ \>\ $\emph{pac}_{n-n_a+1}\leftarrow \emph{spp}(n)-n$\\
    \>\vline\ $\emph{return}\ \ \emph{pac}$\\
  \end{tabbing}
  where $\emph{\emph{spp}}$ is the program \ref{Programspp}.
\end{prog}
\begin{exem}
  The series of \emph{prime additive complements} numbers between limits $n_a=1$ and $n_b=53$ are generated with the command $\emph{paC}(1,53)=$ are: 1, 0, 0, 1, 0, 1, 0, 3, 2, 1, 0, 1, 0, 3, 2, 1, 0, 1, 0, 3, 2, 1, 0, 5, 4, 3, 2, 1, 0, 1, 0, 5, 4, 3, 2, 1, 0, 3, 2, 1, 0, 1, 0, 3, 2, 1, 0, 5, 4, 3, 2, 1, 0~.
\end{exem}
\begin{exem}
  The series of \emph{prime additive complements} numbers between limits $n_a=114$ and $n_b=150$ are generated with the command $\emph{paC}(114,150)=$ are: 13, 12, 11, 10, 9, 8, 7, 6, 5, 4, 3, 2, 1, 0, 3, 2, 1, 0, 5, 4, 3, 2, 1, 0, 1, 0, 9, 8, 7, 6, 5, 4, 3, 2, 1, 0, 1~.
\end{exem}

\section{Sequence of Position}

Let $a=\overline{a_m a_{m-1}\ldots a_0}_{(10)}$ be a decimal integer, and $0\le k\le9$ a digit. The position is def{}ined as follows:
\begin{equation}\label{ukmina}
  u(k,"\emph{min}",a)=\left\{
           \begin{array}{ll}
             \displaystyle\min_{i=\overline{0,m}}\set{i\mid a_i=k}, & \hbox{if $\exists$ $i$ such that $a_i=k$}; \\
             -1, & \hbox{if $\nexists$ $i$ such that $a_i=k$}.
           \end{array}
         \right.
\end{equation}
\begin{equation}\label{ukmaxa}
  u(k,"\emph{max}",a)=\left\{
           \begin{array}{ll}
             \displaystyle\max_{i=\overline{0,m}}\set{i\mid a_i=k}, & \hbox{if $\exists$ $i$ such that $a_i=k$}; \\
             -1, & \hbox{if $\nexists$ $i$ such that $a_i=k$}.
           \end{array}
         \right.
\end{equation}
Let $x=\set{x_1,x_2,\ldots,x_n}$ be the sequence of positive integer numbers, then the sequence of position is def{}ined as follows:
\[
 U(k,"\emph{min}",x)=\set{u(k,"\emph{min}",x_1),u(k,"\emph{min}",x_2),\ldots,u(k,"\emph{min}",x_n)}
\]
or
\[
 U(k,"\emph{max}",x)=\set{u(k,"\emph{max}",x_1),u(k,"\emph{max}",x_2),\ldots,u(k,"\emph{max}",x_n)}~.
\]
\begin{prog}\label{Program u} for functions (\ref{ukmina}) and (\ref{ukmaxa}).
  \begin{tabbing}
    $u(k,\emph{minmax},a):=$\=\vline\ $d\leftarrow \emph{dn}(a,10)$\\
    \>\vline\ $m\leftarrow \emph{last(d)}$\\
    \>\vline\ $j\leftarrow0$\\
    \>\vline\ $f$\=$or\ i\in1..m$\\
    \>\vline\ \>\ $\emph{if}$\=\ $d_i\textbf{=}k$\\
    \>\vline\ \>\ \>\vline\ $j\leftarrow j+1$\\
    \>\vline\ \>\ \>\vline\ $u_j\leftarrow m-i$\\
    \>\vline\ $\emph{if}$\=\ $j>0$\\
    \>\vline\ \>\vline\ $\emph{return}\ \emph{min}(u)\ \emph{if}\ \emph{minmax}\textbf{=}"\emph{min}"$\\
    \>\vline\ \>\vline\ $\emph{return}\ \emph{max}(u)\ \emph{if}\ \emph{minmax}\textbf{=}"\emph{max}"$\\
    \>\vline\ $\emph{return}\ -1\ otherwise$\\
  \end{tabbing}
  This program uses the program $\emph{dn}$, \ref{ProgramDn}.
\end{prog}
\begin{prog}\label{Program U} generate sequence of position.
  \begin{tabbing}
    $U(k,\emph{minmax},x):=$\=\vline\ $\emph{return}\ "\emph{Err}."\ \emph{if}\ k<0\vee k>9$\\
    \>\vline\ $f$\=$or\ j\in1..\emph{last(x)}$\\
    \>\vline\ \>\vline\ $U_j\leftarrow u(k,\emph{minmax},x_j)$\\
    \>\vline\ $\emph{return}\ U$\\
  \end{tabbing}
\end{prog}

Examples of series of position:
\begin{enumerate}
  \item Random sequence of numbers with less than 5 digits: $n:=10$ $j:=1..n$ $x_j:=floor(rnd(10^{10}))$, where $\emph{rnd}(a)$ is the Mathcad function for generating random numbers with uniform distribution from 0 to $a$, then $x^\textrm{T}=$(2749840272, 8389146924, 3712396081, 2325329044, 2316651791, 9710168987, 229116575, 1518263844, 92574637, 6438703378). In this situation we obtain:
      \begin{itemize}
        \item[] $U(7,"\emph{min}",x)^\textrm{T}=$( 1, --1, 8, --1, 2, 0, 1, --1, 0, 1);
        \item[] $U(7,"\emph{max}",x)^\textrm{T}=$( 8, --1, 8, --1, 2, 8, 1, --1, 4, 5).
      \end{itemize}
  \item Sequence of prime numbers $p:=\emph{submatrix}(\emph{prime},1,33,1,1)$ i.e. $p^\textrm{T}=$(2, 3, 5, 7, 11, 13, 17, 19, 23, 29, 31, 37, 41, 43, 47, 53, 59, 61, 67, 71, 73, 79, 83, 89, 97, 101, 103, 107, 109, 113, 127, 131, 137), then:
      \begin{itemize}
        \item[] $U(1,"\emph{min}",p)^\textrm{T}=$( --1, --1, --1, --1, 0, 1, 1, 1, --1, --1, 0, --1, 0, --1, --1, --1, --1, 0, --1, 0, --1, --1, --1, --1, --1, 0, 2, 2, 2, 1, 2, 0, 2);
        \item[] $U(1,"\emph{max}",p)^\textrm{T}=$(--1, --1, --1, --1, 1, 1, 1, 1, --1, --1, 0, --1, 0, --1, --1, --1, --1, 0, --1, 0, --1, --1, --1, --1, --1, 2, 2, 2, 2, 2, 2, 2, 2).
      \end{itemize}
  \item Sequence of factorials $n:=16$ $j:=1..n$ $f_j:=j!$ i.e. $f^\textrm{T}=$(1, 2, 6, 24, 120, 720, 5040, 40320, 362880, 3628800, 39916800, 479001600, 6227020800, 87178291200, 1307674368000, 20922789888000), then:
      \begin{itemize}
        \item[]$U(2,"\emph{min}",f)^\textrm{T}=$(--1, 0, --1, 1, 1, 1, --1, 1, 3, 4, --1, --1, 4, 2, --1, 9);
        \item[]$U(2,"\emph{max}",f)^\textrm{T}=$(--1, 0, --1, 1, 1, 1, --1, 1, 3, 4, --1, --1, 8, 5, --1, 13).
      \end{itemize}
  \item Sequence of \emph{Left Mersenne} numbers $n:=22$ $j:=1..n$ $M\ell_j:=2^j-1$ i.e. $M\ell^\textrm{T}=$(1, 3, 7, 15, 31, 63, 127, 255, 511, 1023, 2047, 4095, 8191, 16383, 32767, 65535, 131071, 262143, 524287, 1048575, 2097151, 4194303), then:
      \begin{itemize}
        \item[] $U(1,"\emph{min}",M\ell)^\textrm{T}=$(0, --1, --1, 1, 0, --1, 2, --1, 0, 3, --1, --1, 0, 4, --1, --1, 0, 2, --1, 6, 0, 5);
        \item[] $U(1,"max",M\ell)^\textrm{T}=$(0, --1, --1, 1, 0, --1, 2, --1, 1, 3, --1, --1, 2, 4, --1, --1, 5, 2, --1, 6, 2, 5).
      \end{itemize}
  \item Sequence of \emph{Right Mersenne} numbers $n:=22$ $j:=1..n$ $Mr_j:=2^j+1$ i.e. $Mr^\textrm{T}=$(3, 5, 9, 17, 33, 65, 129, 257, 513, 1025, 2049, 4097, 8193, 16385, 32769, 65537, 131073, 262145, 524289, 1048577, 2097153, 4194305), then:
      \begin{itemize}
        \item[] $U(1,"\emph{min}",Mr)^\textrm{T}=$(--1, --1, --1, 1, --1, --1, 2, --1, 1, 3, --1, --1, 2, 4, --1, --1, 3, 2, --1, 6, 2, 5);
        \item[] $U(1,"\emph{max}",Mr)^\textrm{T}=$(--1, --1, --1, 1, --1, --1, 2, --1, 1, 3, --1, --1, 2, 4, --1, --1, 5, 2, --1, 6, 2, 5).
      \end{itemize}
  \item Sequence Fibonacci numbers $n:=24$ $j:=1..n$ $F_1:=1$ $F_2:=1$ $F_{j+2}:=F_{j+1}+F_j$ i.e. $F^\textrm{T}=$(1, 1, 2, 3, 5, 8, 13, 21, 34, 55, 89, 144, 233, 377, 610, 987, 1597, 2584, 4181, 6765, 10946, 17711, 28657, 46368, 75025, 121393), then:
      \begin{itemize}
        \item[] $U(1,"\emph{min}",F)^\textrm{T}=$(0, 0, --1, --1, --1, --1, 1, 0, --1, --1, --1, 2, --1, --1, 1, --1, 3, --1, 0, --1, 4, 0, --1, --1, --1, 3);
        \item[] $U(1,"\emph{max}",F)^\textrm{T}=$(0, 0, --1, --1, --1, --1, 1, 0, --1, --1, --1, 2, --1, --1, 1, --1, 3, --1, 2, --1, 4, 4, --1, --1, --1, 5).
      \end{itemize}
  \item Sequence of numbers type $n^n$: $n:=13$ $j:=1..n$ $N_j:=j^j$ i.e. $N^\textrm{T}=$(1, 4, 27, 256, 3125, 46656, 823543, 16777216, 387420489, 10000000000, 285311670611, 8916100448256, 302875106592253), then:
      \begin{itemize}
        \item[] $U(6,"\emph{min}",N)^\textrm{T}=$(--1, --1, --1, 0, --1, 0, --1, 0, --1, --1, 2, 0, 6);
        \item[] $U(6,"\emph{max}",N)^\textrm{T}=$(--1, --1, --1, 0, --1, 3, --1, 6, --1, --1, 5, 9, 6);
      \end{itemize}
  \item Sequence of primorial numbers $n:=33$ $j:=1..n$ $P_j:=\emph{kP}(p_j,1)$, where we used the program $\emph{kP}$, \ref{Program kP}, then:
      \begin{itemize}
        \item[] $U(6,"\emph{min}",P)^\textrm{T}=$(--1, 0, --1, --1, --1, --1, --1, 2, --1, 5, 7, --1, 7, 5, 17, --1, 16, 9, 10, 14, 23, 17, 23, 19, 28, 31, --1, 36, 36, 32, 42, 38, 44);
        \item[] $U(6,"\emph{max}",N)^\textrm{T}=$(--1, 0, --1, --1, --1, --1, --1, 5, --1, 9, 7, --1, 7, 10, 17, --1, 16, 9, 10, 15, 23, 26, 31, 31, 31, 34, --1, 40, 36, 44, 42, 45, 44).
      \end{itemize}
\end{enumerate}

Study:
\begin{enumerate}
  \item $\set{U(k,"\emph{min}",x)}$, where $\set{x}_n$ is the sequence of numbers (double factorial, triple factorial, \ldots, Fibonacci, Tribonacci, Tetranacci, \ldots, primorial, double primorial, triple primorial, \ldots, etc). Convergence, monotony.
  \item $\set{U(k,"\emph{max}",x)}$, where $\set{x}_n$ is the sequence of numbers (double factorial, triple factorial, \ldots, Fibonacci, Tribonacci, Tetranacci, \ldots, primorial, double primorial, triple primorial, \ldots, etc). Convergence, monotony.
\end{enumerate}

\section{The General Periodic Sequence}

Let $\mathcal{S}$ be a f{}inite set, and $f:\mathcal{S}\to \mathcal{S}$ be a function def{}ined for all elements of $\mathcal{S}$. There will always be a periodic sequence whenever we repeat the composition of the function $f$ with itself more times than $card(\mathcal{S})$, accordingly to the box principle of Dirichlet\index{Dirichlet L.}. The invariant sequence is considered a periodic sequence whose period length has one term.

Thus the General Periodic Sequence is def{}ined as:
\begin{itemize}
  \item $a_1=f(s)$, where $s\in \mathcal{S}$;
  \item $a_2=f(a_1)=f(f(s))$, where $s\in \mathcal{S}$;
  \item $a_3=f(a_2)=f(f(a_1))=f(f(f(s)))$, where $s\in \mathcal{S}$;
  \item and so on.
\end{itemize}

We particularize $\mathcal{S}$ and $f$ to study interesting cases of this type of sequences, \citep{Popov1996}.

\subsection{Periodic Sequences}

\subsubsection{The $n$--Digit Periodic Sequence}

Let $n_1$ be an integer of at most two digits and let $n_1'$ be its digital reverse. One def{}ines the absolute value $n_2=\abs{n_1-n_1'}$. And so on: $n_3=\abs{n_2-n_2'}$, etc. If a number $n$ has one digit only, one considers its reverse as $n\times10$ (for example: $5$, which is $05$, reversed will be $50$). This sequence is periodic. Except the case when the two digits are equal, and the sequence becomes: $n_1$, 0, 0, 0, \ldots the iteration always produces a loop of length 5, which starts on the second or the third term of the sequence, and the period is 9, 81, 63, 27, 45, or a cyclic permutation thereof.

\begin{func}\label{Functia PS} for periodic sequence.
  \[
   \emph{PS}(n):=\abs{n-\emph{Reverse}(n)}~,
  \]
  where it uses the function $\emph{Reverse}$, \ref{FunctiaReverse}.
\end{func}

\begin{prog}\label{Program PPS} of the application of function $\emph{PS}$, \ref{Functia PS} of $r$ times the elements vector $v$.
  \begin{tabbing}
    $\emph{PPS}(v,r):=$\=\vline\ $f$\=$or\ k\in1..\emph{last(v)}$\\
    \>\vline\ \>\vline\ $a_{k,1}\leftarrow v_k$\\
    \>\vline\ \>\vline\ $f$\=$or\ j=1..r-1$\\
    \>\vline\ \>\vline\ \>\vline\ $\emph{sw}\leftarrow0$\\
    \>\vline\ \>\vline\ \>\vline\ $f$\=$or\ i=1..j-1\ \ \emph{if}\ \ j\ge2$\\
    \>\vline\ \>\vline\ \>\vline\ \>\ $i$\=$f\ a_{k,i}\textbf{=}a_{k,j}$\\
    \>\vline\ \>\vline\ \>\vline\ \>\ \>\vline\ $a_{k,j}\leftarrow0$\\
    \>\vline\ \>\vline\ \>\vline\ \>\ \>\vline\ $\emph{break}$\\
    \>\vline\ \>\vline\ \>\vline\ $i$\=$f\ a_{k,j}\neq0$\\
    \>\vline\ \>\vline\ \>\vline\ \>\vline\ $a_{k,j+1}\leftarrow \emph{PS}(a_{k,j})$\\
    \>\vline\ \>\vline\ \>\vline\ \>\vline\ $\emph{sw}\leftarrow sw+1$\\
    \>\vline\ \>\vline\ \>\vline\ $\emph{break}\ \ \emph{if}\ \ sw\textbf{=}0$\\
    \>\vline\ $\emph{return}\ \ a$\\
  \end{tabbing}
\end{prog}

The following table has resulted from commands: $k:=10..99$\ \ $v_{k-9}:=k$\ \ $\emph{PPS}(v,7)=$~. In table through $\emph{PS}(n)^k$ understand $\emph{PS}(\emph{PS}(\ldots \emph{PS}(n)))$ of $k$ times.

\begin{center}
 \begin{longtable}{|c|c|c|c|c|c|c|}
   \caption{The two--digit periodic sequence}\\
   \hline
    $n$ & $\emph{PS}(n)$ & $\emph{PS}^2(n)$ & $\emph{PS}^3(n)$ & $\emph{PS}^4(n)$ & $\emph{PS}^5(n)$ & $\emph{PS}^6(n)$ \\
   \hline
  \endfirsthead
   \hline
    $n$ & $\emph{PS}(n)$ & $\emph{PS}^2(n)$ & $\emph{PS}^3(n)$ & $\emph{PS}^4(n)$ & $\emph{PS}^5(n)$ & $\emph{PS}^6(n)$ \\
   \hline
  \endhead
   \hline \multicolumn{7}{r}{\textit{Continued on next page}} \\
  \endfoot
   \hline
  \endlastfoot
   10 & 9 & 0 & 0 & 0 & 0 & 0\\
   11 & 0 & 0 & 0 & 0 & 0 & 0\\
   12 & 9 & 0 & 0 & 0 & 0 & 0\\
   13 & 18 & 63 & 27 & 45 & 9 & 0\\
   14 & 27 & 45 & 9 & 0 & 0 & 0\\
   15 & 36 & 27 & 45 & 9 & 0 & 0\\
   16 & 45 & 9 & 0 & 0 & 0 & 0\\
   17 & 54 & 9 & 0 & 0 & 0 & 0\\
   18 & 63 & 27 & 45 & 9 & 0 & 0\\
   19 & 72 & 45 & 9 & 0 & 0 & 0\\
   20 & 18 & 63 & 27 & 45 & 9 & 0\\
   21 & 9 & 0 & 0 & 0 & 0 & 0\\
   22 & 0 & 0 & 0 & 0 & 0 & 0\\
   23 & 9 & 0 & 0 & 0 & 0 & 0\\
   24 & 18 & 63 & 27 & 45 & 9 & 0\\
   25 & 27 & 45 & 9 & 0 & 0 & 0\\
   26 & 36 & 27 & 45 & 9 & 0 & 0\\
   27 & 45 & 9 & 0 & 0 & 0 & 0\\
   28 & 54 & 9 & 0 & 0 & 0 & 0\\
   29 & 63 & 27 & 45 & 9 & 0 & 0\\
   30 & 27 & 45 & 9 & 0 & 0 & 0\\
   31 & 18 & 63 & 27 & 45 & 9 & 0\\
   32 & 9 & 0 & 0 & 0 & 0 & 0\\
   33 & 0 & 0 & 0 & 0 & 0 & 0\\
   34 & 9 & 0 & 0 & 0 & 0 & 0\\
   35 & 18 & 63 & 27 & 45 & 9 & 0\\
   36 & 27 & 45 & 9 & 0 & 0 & 0\\
   37 & 36 & 27 & 45 & 9 & 0 & 0\\
   38 & 45 & 9 & 0 & 0 & 0 & 0\\
   39 & 54 & 9 & 0 & 0 & 0 & 0\\
   40 & 36 & 27 & 45 & 9 & 0 & 0\\
   41 & 27 & 45 & 9 & 0 & 0 & 0\\
   42 & 18 & 63 & 27 & 45 & 9 & 0\\
   43 & 9 & 0 & 0 & 0 & 0 & 0\\
   44 & 0 & 0 & 0 & 0 & 0 & 0\\
   45 & 9 & 0 & 0 & 0 & 0 & 0\\
   46 & 18 & 63 & 27 & 45 & 9 & 0\\
   47 & 27 & 45 & 9 & 0 & 0 & 0\\
   48 & 36 & 27 & 45 & 9 & 0 & 0\\
   49 & 45 & 9 & 0 & 0 & 0 & 0\\
   50 & 45 & 9 & 0 & 0 & 0 & 0\\
   51 & 36 & 27 & 45 & 9 & 0 & 0\\
   52 & 27 & 45 & 9 & 0 & 0 & 0\\
   53 & 18 & 63 & 27 & 45 & 9 & 0\\
   54 & 9 & 0 & 0 & 0 & 0 & 0\\
   55 & 0 & 0 & 0 & 0 & 0 & 0\\
   56 & 9 & 0 & 0 & 0 & 0 & 0\\
   57 & 18 & 63 & 27 & 45 & 9 & 0\\
   58 & 27 & 45 & 9 & 0 & 0 & 0\\
   59 & 36 & 27 & 45 & 9 & 0 & 0\\
   60 & 54 & 9 & 0 & 0 & 0 & 0\\
   61 & 45 & 9 & 0 & 0 & 0 & 0\\
   62 & 36 & 27 & 45 & 9 & 0 & 0\\
   63 & 27 & 45 & 9 & 0 & 0 & 0\\
   64 & 18 & 63 & 27 & 45 & 9 & 0\\
   65 & 9 & 0 & 0 & 0 & 0 & 0\\
   66 & 0 & 0 & 0 & 0 & 0 & 0\\
   67 & 9 & 0 & 0 & 0 & 0 & 0\\
   68 & 18 & 63 & 27 & 45 & 9 & 0\\
   69 & 27 & 45 & 9 & 0 & 0 & 0\\
   70 & 63 & 27 & 45 & 9 & 0 & 0\\
   71 & 54 & 9 & 0 & 0 & 0 & 0\\
   72 & 45 & 9 & 0 & 0 & 0 & 0\\
   73 & 36 & 27 & 45 & 9 & 0 & 0\\
   74 & 27 & 45 & 9 & 0 & 0 & 0\\
   75 & 18 & 63 & 27 & 45 & 9 & 0\\
   76 & 9 & 0 & 0 & 0 & 0 & 0\\
   77 & 0 & 0 & 0 & 0 & 0 & 0\\
   78 & 9 & 0 & 0 & 0 & 0 & 0\\
   79 & 18 & 63 & 27 & 45 & 9 & 0\\
   80 & 72 & 45 & 9 & 0 & 0 & 0\\
   81 & 63 & 27 & 45 & 9 & 0 & 0\\
   82 & 54 & 9 & 0 & 0 & 0 & 0\\
   83 & 45 & 9 & 0 & 0 & 0 & 0\\
   84 & 36 & 27 & 45 & 9 & 0 & 0\\
   85 & 27 & 45 & 9 & 0 & 0 & 0\\
   86 & 18 & 63 & 27 & 45 & 9 & 0\\
   87 & 9 & 0 & 0 & 0 & 0 & 0\\
   88 & 0 & 0 & 0 & 0 & 0 & 0\\
   89 & 9 & 0 & 0 & 0 & 0 & 0\\
   90 & 81 & 63 & 27 & 45 & 9 & 0\\
   91 & 72 & 45 & 9 & 0 & 0 & 0\\
   92 & 63 & 27 & 45 & 9 & 0 & 0\\
   93 & 54 & 9 & 0 & 0 & 0 & 0\\
   94 & 45 & 9 & 0 & 0 & 0 & 0\\
   95 & 36 & 27 & 45 & 9 & 0 & 0\\
   96 & 27 & 45 & 9 & 0 & 0 & 0\\
   97 & 18 & 63 & 27 & 45 & 9 & 0\\
   98 & 9 & 0 & 0 & 0 & 0 & 0\\
   99 & 0 & 0 & 0 & 0 & 0 & 0\\
\hline
\end{longtable}
\end{center}

\begin{enumerate}
  \item The $3$--digit periodic sequence (domain $10^2\le n_1\le10^3-1$):
      \begin{itemize}
        \item there are $90$ symmetric integers, $101$, $111$, $121$, \ldots, for which $n_2=0$;
        \item all other initial integers iterate into various entry points of the same periodic subsequence (or a cyclic permutation thereof) of f{}ive terms: $99$, $891$, $693$, $297$, $495$.
  \end{itemize}
  For example we take 3--digits prime numbers and we study periodic sequence. The following table has resulted from commands:
  $p:=\emph{submatrix}(\emph{prime},26,168,1,1)$\ \ $\emph{PPS}(p,7)=$~. In table through $\emph{PS}(n)^k$ understand $\emph{PS}(\emph{PS}(\ldots \emph{PS}(n)))$ of $k$ times.
  \begin{center}
  \begin{longtable}{|c|c|c|c|c|c|c|}
   \caption{Primes with 3--digits periodic sequences}\\
   \hline
    $n$ & $\emph{PS}(n)$ & $\emph{PS}^2(n)$ & $\emph{PS}^3(n)$ & $\emph{PS}^4(n)$ & $\emph{PS}^5(n)$ & $\emph{PS}^6(n)$ \\
   \hline
  \endfirsthead
   \hline
    $n$ & $\emph{PS}(n)$ & $\emph{PS}^2(n)$ & $\emph{PS}^3(n)$ & $\emph{PS}^4(n)$ & $\emph{PS}^5(n)$ & $\emph{PS}^6(n)$ \\
   \hline
  \endhead
   \hline \multicolumn{7}{r}{\textit{Continued on next page}} \\
  \endfoot
   \hline
  \endlastfoot
   101 & 0 & 0 & 0 & 0 & 0 & 0\\
   103 & 198 & 693 & 297 & 495 & 99 & 0\\
   107 & 594 & 99 & 0 & 0 & 0 & 0\\
   109 & 792 & 495 & 99 & 0 & 0 & 0\\
   113 & 198 & 693 & 297 & 495 & 99 & 0\\
   127 & 594 & 99 & 0 & 0 & 0 & 0\\
   131 & 0 & 0 & 0 & 0 & 0 & 0\\
   137 & 594 & 99 & 0 & 0 & 0 & 0\\
   139 & 792 & 495 & 99 & 0 & 0 & 0\\
   149 & 792 & 495 & 99 & 0 & 0 & 0\\
   151 & 0 & 0 & 0 & 0 & 0 & 0\\
   157 & 594 & 99 & 0 & 0 & 0 & 0\\
   163 & 198 & 693 & 297 & 495 & 99 & 0\\
   167 & 594 & 99 & 0 & 0 & 0 & 0\\
   173 & 198 & 693 & 297 & 495 & 99 & 0\\
   179 & 792 & 495 & 99 & 0 & 0 & 0\\
   181 & 0 & 0 & 0 & 0 & 0 & 0\\
   191 & 0 & 0 & 0 & 0 & 0 & 0\\
   193 & 198 & 693 & 297 & 495 & 99 & 0\\
   197 & 594 & 99 & 0 & 0 & 0 & 0\\
   199 & 792 & 495 & 99 & 0 & 0 & 0\\
   211 & 99 & 0 & 0 & 0 & 0 & 0\\
   223 & 99 & 0 & 0 & 0 & 0 & 0\\
   227 & 495 & 99 & 0 & 0 & 0 & 0\\
   229 & 693 & 297 & 495 & 99 & 0 & 0\\
   233 & 99 & 0 & 0 & 0 & 0 & 0\\
   239 & 693 & 297 & 495 & 99 & 0 & 0\\
   241 & 99 & 0 & 0 & 0 & 0 & 0\\
   251 & 99 & 0 & 0 & 0 & 0 & 0\\
   257 & 495 & 99 & 0 & 0 & 0 & 0\\
   263 & 99 & 0 & 0 & 0 & 0 & 0\\
   269 & 693 & 297 & 495 & 99 & 0 & 0\\
   271 & 99 & 0 & 0 & 0 & 0 & 0\\
   277 & 495 & 99 & 0 & 0 & 0 & 0\\
   281 & 99 & 0 & 0 & 0 & 0 & 0\\
   283 & 99 & 0 & 0 & 0 & 0 & 0\\
   293 & 99 & 0 & 0 & 0 & 0 & 0\\
   307 & 396 & 297 & 495 & 99 & 0 & 0\\
   311 & 198 & 693 & 297 & 495 & 99 & 0\\
   313 & 0 & 0 & 0 & 0 & 0 & 0\\
   317 & 396 & 297 & 495 & 99 & 0 & 0\\
   331 & 198 & 693 & 297 & 495 & 99 & 0\\
   337 & 396 & 297 & 495 & 99 & 0 & 0\\
   347 & 396 & 297 & 495 & 99 & 0 & 0\\
   349 & 594 & 99 & 0 & 0 & 0 & 0\\
   353 & 0 & 0 & 0 & 0 & 0 & 0\\
   359 & 594 & 99 & 0 & 0 & 0 & 0\\
   367 & 396 & 297 & 495 & 99 & 0 & 0\\
   373 & 0 & 0 & 0 & 0 & 0 & 0\\
   379 & 594 & 99 & 0 & 0 & 0 & 0\\
   383 & 0 & 0 & 0 & 0 & 0 & 0\\
   389 & 594 & 99 & 0 & 0 & 0 & 0\\
   397 & 396 & 297 & 495 & 99 & 0 & 0\\
   401 & 297 & 495 & 99 & 0 & 0 & 0\\
   409 & 495 & 99 & 0 & 0 & 0 & 0\\
   419 & 495 & 99 & 0 & 0 & 0 & 0\\
   421 & 297 & 495 & 99 & 0 & 0 & 0\\
   431 & 297 & 495 & 99 & 0 & 0 & 0\\
   433 & 99 & 0 & 0 & 0 & 0 & 0\\
   439 & 495 & 99 & 0 & 0 & 0 & 0\\
   443 & 99 & 0 & 0 & 0 & 0 & 0\\
   449 & 495 & 99 & 0 & 0 & 0 & 0\\
   457 & 297 & 495 & 99 & 0 & 0 & 0\\
   461 & 297 & 495 & 99 & 0 & 0 & 0\\
   463 & 99 & 0 & 0 & 0 & 0 & 0\\
   467 & 297 & 495 & 99 & 0 & 0 & 0\\
   479 & 495 & 99 & 0 & 0 & 0 & 0\\
   487 & 297 & 495 & 99 & 0 & 0 & 0\\
   491 & 297 & 495 & 99 & 0 & 0 & 0\\
   499 & 495 & 99 & 0 & 0 & 0 & 0\\
   503 & 198 & 693 & 297 & 495 & 99 & 0\\
   509 & 396 & 297 & 495 & 99 & 0 & 0\\
   521 & 396 & 297 & 495 & 99 & 0 & 0\\
   523 & 198 & 693 & 297 & 495 & 99 & 0\\
   541 & 396 & 297 & 495 & 99 & 0 & 0\\
   547 & 198 & 693 & 297 & 495 & 99 & 0\\
   557 & 198 & 693 & 297 & 495 & 99 & 0\\
   563 & 198 & 693 & 297 & 495 & 99 & 0\\
   569 & 396 & 297 & 495 & 99 & 0 & 0\\
   571 & 396 & 297 & 495 & 99 & 0 & 0\\
   577 & 198 & 693 & 297 & 495 & 99 & 0\\
   587 & 198 & 693 & 297 & 495 & 99 & 0\\
   593 & 198 & 693 & 297 & 495 & 99 & 0\\
   599 & 396 & 297 & 495 & 99 & 0 & 0\\
   601 & 495 & 99 & 0 & 0 & 0 & 0\\
   607 & 99 & 0 & 0 & 0 & 0 & 0\\
   613 & 297 & 495 & 99 & 0 & 0 & 0\\
   617 & 99 & 0 & 0 & 0 & 0 & 0\\
   619 & 297 & 495 & 99 & 0 & 0 & 0\\
   631 & 495 & 99 & 0 & 0 & 0 & 0\\
   641 & 495 & 99 & 0 & 0 & 0 & 0\\
   643 & 297 & 495 & 99 & 0 & 0 & 0\\
   647 & 99 & 0 & 0 & 0 & 0 & 0\\
   653 & 297 & 495 & 99 & 0 & 0 & 0\\
   659 & 297 & 495 & 99 & 0 & 0 & 0\\
   661 & 495 & 99 & 0 & 0 & 0 & 0\\
   673 & 297 & 495 & 99 & 0 & 0 & 0\\
   677 & 99 & 0 & 0 & 0 & 0 & 0\\
   683 & 297 & 495 & 99 & 0 & 0 & 0\\
   691 & 495 & 99 & 0 & 0 & 0 & 0\\
   701 & 594 & 99 & 0 & 0 & 0 & 0\\
   709 & 198 & 693 & 297 & 495 & 99 & 0\\
   719 & 198 & 693 & 297 & 495 & 99 & 0\\
   727 & 0 & 0 & 0 & 0 & 0 & 0\\
   733 & 396 & 297 & 495 & 99 & 0 & 0\\
   739 & 198 & 693 & 297 & 495 & 99 & 0\\
   743 & 396 & 297 & 495 & 99 & 0 & 0\\
   751 & 594 & 99 & 0 & 0 & 0 & 0\\
   757 & 0 & 0 & 0 & 0 & 0 & 0\\
   761 & 594 & 99 & 0 & 0 & 0 & 0\\
   769 & 198 & 693 & 297 & 495 & 99 & 0\\
   773 & 396 & 297 & 495 & 99 & 0 & 0\\
   787 & 0 & 0 & 0 & 0 & 0 & 0\\
   797 & 0 & 0 & 0 & 0 & 0 & 0\\
   809 & 99 & 0 & 0 & 0 & 0 & 0\\
   811 & 693 & 297 & 495 & 99 & 0 & 0\\
   821 & 693 & 297 & 495 & 99 & 0 & 0\\
   823 & 495 & 99 & 0 & 0 & 0 & 0\\
   827 & 99 & 0 & 0 & 0 & 0 & 0\\
   829 & 99 & 0 & 0 & 0 & 0 & 0\\
   839 & 99 & 0 & 0 & 0 & 0 & 0\\
   853 & 495 & 99 & 0 & 0 & 0 & 0\\
   857 & 99 & 0 & 0 & 0 & 0 & 0\\
   859 & 99 & 0 & 0 & 0 & 0 & 0\\
   863 & 495 & 99 & 0 & 0 & 0 & 0\\
   877 & 99 & 0 & 0 & 0 & 0 & 0\\
   881 & 693 & 297 & 495 & 99 & 0 & 0\\
   883 & 495 & 99 & 0 & 0 & 0 & 0\\
   887 & 99 & 0 & 0 & 0 & 0 & 0\\
   907 & 198 & 693 & 297 & 495 & 99 & 0\\
   911 & 792 & 495 & 99 & 0 & 0 & 0\\
   919 & 0 & 0 & 0 & 0 & 0 & 0\\
   929 & 0 & 0 & 0 & 0 & 0 & 0\\
   937 & 198 & 693 & 297 & 495 & 99 & 0\\
   941 & 792 & 495 & 99 & 0 & 0 & 0\\
   947 & 198 & 693 & 297 & 495 & 99 & 0\\
   953 & 594 & 99 & 0 & 0 & 0 & 0\\
   967 & 198 & 693 & 297 & 495 & 99 & 0\\
   971 & 792 & 495 & 99 & 0 & 0 & 0\\
   977 & 198 & 693 & 297 & 495 & 99 & 0\\
   983 & 594 & 99 & 0 & 0 & 0 & 0\\
   991 & 792 & 495 & 99 & 0 & 0 & 0\\
   997 & 198 & 693 & 297 & 495 & 99 & 0\\
  \hline
  \end{longtable}
  \end{center}
  \item The 4--digit periodic sequence (domain $10^3\le n_1\le 10^4-1$), \citep{Ibstedt1997}:
      \begin{itemize}
        \item the largest number of iterations carried out in order to reach the f{}irst member of the loop is 18, and it happens for $n_1=1019$
        \item iterations of 8818 integers result in one of the following loops (or a cyclic permutation thereof): 2178, 6534; or 90, 810, 630, 270, 450; or 909, 8181, 6363, 2727, 4545; or 999, 8991, 6993, 2997, 4995;
        \item the other iterations ended up in the invariant 0.
  \end{itemize}
  \item The 5--digit periodic sequence (domain $10^4\le n_1\le10^5-1$):
      \begin{itemize}
        \item there are 920 integers iterating into the invariant 0 due to symmetries;
        \item the other ones iterate into one of the following loops (or a cyclic permutation of these): 21978, 65934; or 990, 8910, 6930, 2970, 4950; or 9009, 81081, 63063, 27027, 45045; or 9999, 89991, 69993, 29997, 49995.
      \end{itemize}
  \item The 6--digit periodic sequence (domain $10^5\le n_1\le10^6-1$):
      \begin{itemize}
        \item there are 13667 integers iterating into the invariant 0 due to symmetries;
        \item the longest sequence of iterations before arriving at the f{}irst loop is 53 for $n_1=100720$;
        \item the loops have 2, 5, 9, or 18 terms.
      \end{itemize}
\end{enumerate}

\subsection{The Subtraction Periodic Sequences}

Let $c$ be a positive integer. Start with a positive integer $n$, and let $n'$ be its digital reverse. Put $n_1=\abs{n'-c}$, and let $n_1'$ be its digital reverse. Put $n_2=\abs{n_1'-c}$, and let $n_2'$ be its digital reverse. And so on. We shall eventually obtain a repetition.

For example, with $c=1$ and $n=52$ we obtain the sequence: 52, 24, 41, 13, 30, 02, 19, 90, 08, 79, 96, 68, 85, 57, 74, 46, 63, 35, 52, \ldots~. Here a repetition occurs after 18 steps, and the length of the repeating cycle is 18.

First example: $c=1$, $10\le n\le 999$. Every other member of this interval is an entry point into one of f{}ive cyclic periodic sequences (four of these are of length 18, and one of length 9). When $n$ is of the form $11k$ or $11k-1$, then the iteration process results in 0.

Second example: $1\le c\le 9$, $100\le n\le999$. For $c=1,2$, or 5 all iterations result in the invariant 0 after, sometimes, a large number of iterations. For the other values of c there are only eight dif{}ferent possible values for the length of the loops, namely 11, 22, 33, 50, 100, 167, 189, 200.

For $c=7$ and $n=109$ we have an example of the longest loop obtained: it has 200 elements, and the loop is closed after 286 iterations, \citep{Ibstedt1997}.

\begin{prog}\label{Program SPS} for the subtraction periodic sequences.
  \begin{tabbing}
    $\emph{SPS}(n,c):=$\=\vline\ $\emph{return}\ \ \abs{Reverse(n)\cdot10-c}\ \ \emph{if}\ \ n<10$\\
    \>\vline\ $\emph{return}\ \ \abs{Reverse(n)-c}\ \ otherwise$\\
  \end{tabbing}
  The program use the function $\emph{Reverse}$, \ref{FunctiaReverse}.
\end{prog}

\begin{prog}\label{Program PPScuf} of the application of function $f$, \ref{Program SPS} or \ref{Program MPS} of $r$ times the elements vector $v$ with constant $c$.
   \begin{tabbing}
     $\emph{PPS}(v,c,r,f):=$\=\vline\ $f$\=$or\ k\in1..\emph{last(v)}$\\
     \>\vline\ \>\vline\ $a_{k,1}\leftarrow v_k$\\
     \>\vline\ \>\vline\ $f$\=$or\ j=1..r-1$\\
     \>\vline\ \>\vline\ \>\vline\ $\emph{sw}\leftarrow0$\\
     \>\vline\ \>\vline\ \>\vline\ $f$\=$or\ i=1..j-1\ \ \emph{if}\ \ j\ge2$\\
     \>\vline\ \>\vline\ \>\vline\ \>\ $i$\=$f\ a_{k,i}\textbf{=}a_{k,j}$\\
     \>\vline\ \>\vline\ \>\vline\ \>\ \>\vline\ $a_{k,j}\leftarrow0$\\
     \>\vline\ \>\vline\ \>\vline\ \>\ \>\vline\ $\emph{break}$\\
     \>\vline\ \>\vline\ \>\vline\ $i$\=$f\ a_{k,j}\neq0$\\
     \>\vline\ \>\vline\ \>\vline\ \>\vline\ $a_{k,j+1}\leftarrow f(a_{k,j},c)$\\
     \>\vline\ \>\vline\ \>\vline\ \>\vline\ $sw\leftarrow \emph{sw}+1$\\
     \>\vline\ \>\vline\ \>\vline\ $i$\=$f\ sw\textbf{=}0$\\
     \>\vline\ \>\vline\ \>\vline\ \>\vline\ $a_{k,j}\leftarrow a_{k,i}\ \ \emph{if}\ \ a_{k,j}\textbf{=}0$\\
     \>\vline\ \>\vline\ \>\vline\ \>\vline\ $\emph{break}$\\
    \>\vline\ $\emph{return}\ \ a$\\
  \end{tabbing}

\end{prog}

With the commands $k:=10..35$\ \ $v_{k-9}:=k$\ \ $\emph{PPS}(v,6,10,\emph{SPS})=$, where $\emph{PPS}$, \ref{Program PPScuf}, one obtains the table:
\begin{center}
  {\footnotesize\begin{longtable}{|ccccccccccc|}
   \caption{2--digits substraction periodic sequences}\\
   \hline
    $n$&$f(n)$&$f^2(n)$&$f^3(n)$&$f^4(n)$&$f^5(n)$&$f^6(n)$&$f^7(n)$&$f^8(n)$&$f^9(n)$&$f^{10}(n)$\\
   \hline
  \endfirsthead
   \hline
    $n$&$f(n)$&$f^2(n)$&$f^3(n)$&$f^4(n)$&$f^5(n)$&$f^6(n)$&$f^7(n)$&$f^8(n)$&$f^9(n)$&$f^{10}(n)$\\
   \hline
  \endhead
   \hline \multicolumn{11}{r}{\textit{Continued on next page}} \\
  \endfoot
   \hline
  \endlastfoot
    10 & 5 & 44 & 38 & 77 & 71 & 11 & 5 & 0 & 0 & 0\\
    11 & 5 & 44 & 38 & 77 & 71 & 11 & 0 & 0 & 0 & 0\\
    12 & 15 & 45 & 48 & 78 & 81 & 12 & 0 & 0 & 0 & 0\\
    13 & 25 & 46 & 58 & 79 & 91 & 13 & 0 & 0 & 0 & 0\\
    14 & 35 & 47 & 68 & 80 & 2 & 14 & 0 & 0 & 0 & 0\\
    15 & 45 & 48 & 78 & 81 & 12 & 15 & 0 & 0 & 0 & 0\\
    16 & 55 & 49 & 88 & 82 & 22 & 16 & 0 & 0 & 0 & 0\\
    17 & 65 & 50 & 1 & 4 & 34 & 37 & 67 & 70 & 1 & 0\\
    18 & 75 & 51 & 9 & 84 & 42 & 18 & 0 & 0 & 0 & 0\\
    19 & 85 & 52 & 19 & 0 & 0 & 0 & 0 & 0 & 0 & 0\\
    20 & 4 & 34 & 37 & 67 & 70 & 1 & 4 & 0 & 0 & 0\\
    21 & 6 & 54 & 39 & 87 & 72 & 21 & 0 & 0 & 0 & 0\\
    22 & 16 & 55 & 49 & 88 & 82 & 22 & 0 & 0 & 0 & 0\\
    23 & 26 & 56 & 59 & 89 & 92 & 23 & 0 & 0 & 0 & 0\\
    24 & 36 & 57 & 69 & 90 & 3 & 24 & 0 & 0 & 0 & 0\\
    25 & 46 & 58 & 79 & 91 & 13 & 25 & 0 & 0 & 0 & 0\\
    26 & 56 & 59 & 89 & 92 & 23 & 26 & 0 & 0 & 0 & 0\\
    27 & 66 & 60 & 60 & 0 & 0 & 0 & 0 & 0 & 0 & 0\\
    28 & 76 & 61 & 10 & 5 & 44 & 38 & 77 & 71 & 11 & 5\\
    29 & 86 & 62 & 20 & 4 & 34 & 37 & 67 & 70 & 1 & 4\\
    30 & 3 & 24 & 36 & 57 & 69 & 90 & 3 & 0 & 0 & 0\\
    31 & 7 & 64 & 40 & 2 & 14 & 35 & 47 & 68 & 80 & 2\\
    32 & 17 & 65 & 50 & 1 & 4 & 34 & 37 & 67 & 70 & 1\\
    33 & 27 & 66 & 60 & 60 & 0 & 0 & 0 & 0 & 0 & 0\\
    34 & 37 & 67 & 70 & 1 & 4 & 34 & 0 & 0 & 0 & 0\\
    35 & 47 & 68 & 80 & 2 & 14 & 35 & 0 & 0 & 0 & 0\\
  \hline
 \end{longtable}}
\end{center}
where $f(n)=\abs{n'-c}$, $n'$ is digital reverse, $c=6$ and $f^k(n)=f(f(\ldots(f(n))))$ of $k$ times.

\subsection{The Multiplication Periodic Sequences}

Let $c>1$ be a positive integer. Start with a positive integer $n$, multiply each digit $x$ of $n$ by $c$ and replace that digit by the last digit of $c\cdot x$ to give $n_1$. And so on. We shall eventually obtain a repetition.

For example, with $c=7$ and $n=68$ we obtain the sequence: 68, 26, 42, 84, 68. Integers whose digits are all equal to $5$ are invariant under the given operation after one iteration.

One studies the one--\emph{digit multiplication periodic sequences} (short \emph{dmps}) only. For $c$ of two or more digits the problem becomes more complicated.

\begin{prog}\label{Program MPS} for the \emph{dmps}.
  \begin{tabbing}
    $\emph{MPS}(n,c):=$\=\vline\ $d\leftarrow \emph{reverse}(\emph{dn}(n,10))$\\
    \>\vline\ $m\leftarrow0$\\
    \>\vline\ $f$\=$or\ k\in1..\emph{last(d)}$\\
    \>\vline\ \>\ $m\leftarrow m+\mod(d_k\cdot c,10)\cdot10^{k-1}$\\
    \>\vline\ $\emph{return}\ m$
  \end{tabbing}
\end{prog}

$\emph{PPS}$ program execution to vector (68) is
\[
 \emph{PPS}((68),7,10,\emph{MPS})=(68\ \ 26\ \ 42\ \ 84\ \ 68)~.
\]

For example we use commands: $k:=10..19$\ \ $v_k:=k$.
\begin{enumerate}
  \item If $c:=2$, there are four term loops, starting on the f{}irst or second term and $\emph{PPS}(v,c,10,\emph{MPS})=$:
      \begin{center}
      \begin{longtable}{|cccccc|}
      \caption{2--\emph{dmps} with $c=2$}\\
      \hline
      $n$&$f(n)$&$f^2(n)$&$f^3(n)$&$f^4(n)$&$f^5(n)$\\
      \hline
      \endfirsthead
      \hline
      $n$&$f(n)$&$f^2(n)$&$f^3(n)$&$f^4(n)$&$f^5(n)$\\
      \hline
      \endhead
      \hline \multicolumn{6}{r}{\textit{Continued on next page}} \\
      \endfoot
      \hline
      \endlastfoot
      10 & 20 & 40 & 80 & 60 & 20\\
      11 & 22 & 44 & 88 & 66 & 22\\
      12 & 24 & 48 & 86 & 62 & 24\\
      13 & 26 & 42 & 84 & 68 & 26\\
      14 & 28 & 46 & 82 & 64 & 28\\
      15 & 20 & 40 & 80 & 60 & 20\\
      16 & 22 & 44 & 88 & 66 & 22\\
      17 & 24 & 48 & 86 & 62 & 24\\
      18 & 26 & 42 & 84 & 68 & 26\\
      19 & 28 & 46 & 82 & 64 & 28\\
      \hline
      \end{longtable}
      \end{center}
  where $f(n)=\emph{MPS}(n,2)$ and $f^k(n)=f(f(\ldots(f(n))))$ of $k$ times.
  \item If $c:=3$, there are four term loops, starting with the f{}irst term and $\emph{PPS}(v,c,10,\emph{MPS})=$:
      \begin{center}
      \begin{longtable}{|ccccc|}
      \caption{2--\emph{dmps} with $c=3$}\\
      \hline
      $n$&$f(n)$&$f^2(n)$&$f^3(n)$&$f^4(n)$\\
      \hline
      \endfirsthead
      \hline
      $n$&$f(n)$&$f^2(n)$&$f^3(n)$&$f^4(n)$\\
      \hline
      \endhead
      \hline \multicolumn{5}{r}{\textit{Continued on next page}} \\
      \endfoot
      \hline
      \endlastfoot
      10 & 30 & 90 & 70 & 10\\
      11 & 33 & 99 & 77 & 11\\
      12 & 36 & 98 & 74 & 12\\
      13 & 39 & 97 & 71 & 13\\
      14 & 32 & 96 & 78 & 14\\
      15 & 35 & 95 & 75 & 15\\
      16 & 38 & 94 & 72 & 16\\
      17 & 31 & 93 & 79 & 17\\
      18 & 34 & 92 & 76 & 18\\
      19 & 37 & 91 & 73 & 19\\
      \hline
      \end{longtable}
      \end{center}
  where $f(n)=\emph{MPS}(n,3)$ and $f^k(n)=f(f(\ldots(f(n))))$ of $k$ times.
  \item If $c:=4$, there are two term loops, starting on the f{}irst or second term (could be called switch or pendulum) and $\emph{PPS}(v,c,10,\emph{MPS})=$:
      \begin{center}
      \begin{longtable}{|cccc|}
      \caption{2--\emph{dmps} with $c=4$}\\
      \hline
      $n$&$f(n)$&$f^2(n)$&$f^3(n)$\\
      \hline
      \endfirsthead
      \hline
      $n$&$f(n)$&$f^2(n)$&$f^3(n)$\\
      \hline
      \endhead
      \hline \multicolumn{4}{r}{\textit{Continued on next page}} \\
      \endfoot
      \hline
      \endlastfoot
      10 & 40 & 60 & 40\\
      11 & 44 & 66 & 44\\
      12 & 48 & 62 & 48\\
      13 & 42 & 68 & 42\\
      14 & 46 & 64 & 46\\
      15 & 40 & 60 & 40\\
      16 & 44 & 66 & 44\\
      17 & 48 & 62 & 48\\
      18 & 42 & 68 & 42\\
      19 & 46 & 64 & 46\\
      \hline
      \end{longtable}
      \end{center}
  where $f(n)=\emph{MPS}(n,4)$ and $f^k(n)=f(f(\ldots(f(n))))$ of $k$ times.
  \item If $c:=5$, the sequence is invariant after one iteration and\\ $\emph{PPS}(v,c,10,\emph{MPS})=$:
   \begin{center}
      \begin{longtable}{|ccc|}
      \caption{2--\emph{dmps} with $c=5$}\\
      \hline
      $n$&$f(n)$&$f^2(n)$\\
      \hline
      \endfirsthead
      \hline
      $n$&$f(n)$&$f^2(n)$\\
      \hline
      \endhead
      \hline \multicolumn{3}{r}{\textit{Continued on next page}} \\
      \endfoot
      \hline
      \endlastfoot
      10 & 50 & 50\\
      11 & 55 & 55\\
      12 & 50 & 50\\
      13 & 55 & 55\\
      14 & 50 & 50\\
      15 & 55 & 55\\
      16 & 50 & 50\\
      17 & 55 & 55\\
      18 & 50 & 50\\
      19 & 55 & 55\\
      \end{longtable}
      \end{center}
  where $f(n)=\emph{MPS}(n,5)$ and $f^k(n)=f(f(\ldots(f(n))))$ of $k$ times.
  \item If $c:=6$, the sequence is invariant after one iteration and $\emph{PPS}(v,c,10,\emph{MPS})=$:.
   \begin{center}
      \begin{longtable}{|ccc|}
      \caption{2--\emph{dmps} with $c=6$}\\
      \hline
      $n$&$f(n)$&$f^2(n)$\\
      \hline
      \endfirsthead
      \hline
      $n$&$f(n)$&$f^2(n)$\\
      \hline
      \endhead
      \hline \multicolumn{3}{r}{\textit{Continued on next page}} \\
      \endfoot
      \hline
      \endlastfoot
      10 & 60 & 60\\
      11 & 66 & 66\\
      12 & 62 & 62\\
      13 & 68 & 68\\
      14 & 64 & 64\\
      15 & 60 & 60\\
      16 & 66 & 66\\
      17 & 62 & 62\\
      18 & 68 & 68\\
      19 & 64 & 64\\
      \end{longtable}
      \end{center}
  where $f(n)=\emph{MPS}(n,6)$ and $f^k(n)=f(f(\ldots(f(n))))$ of $k$ times.
  \item If $c:=7$, there are four term loops, starting with the f{}irst term and $\emph{PPS}(v,c,10,\emph{MPS})=$:
      \begin{center}
      \begin{longtable}{|ccccc|}
      \caption{2--\emph{dmps} with $c=7$}\\
      \hline
      $n$&$f(n)$&$f^2(n)$&$f^3(n)$&$f^4(n)$\\
      \hline
      \endfirsthead
      \hline
      $n$&$f(n)$&$f^2(n)$&$f^3(n)$&$f^4(n)$\\
      \hline
      \endhead
      \hline \multicolumn{5}{r}{\textit{Continued on next page}} \\
      \endfoot
      \hline
      \endlastfoot
      10 & 70 & 90 & 30 & 10\\
      11 & 77 & 99 & 33 & 11\\
      12 & 74 & 98 & 36 & 12\\
      13 & 71 & 97 & 39 & 13\\
      14 & 78 & 96 & 32 & 14\\
      15 & 75 & 95 & 35 & 15\\
      16 & 72 & 94 & 38 & 16\\
      17 & 79 & 93 & 31 & 17\\
      18 & 76 & 92 & 34 & 18\\
      19 & 73 & 91 & 37 & 19\\
      \hline
      \end{longtable}
      \end{center}
  where $f(n)=\emph{MPS}(n,7)$ and $f^k(n)=f(f(\ldots(f(n))))$ of $k$ times.
  \item If $c:=8$, there are four term loops, starting on the f{}irst or second term and $\emph{PPS}(v,c,10,\emph{MPS})=$:
      \begin{center}
      \begin{longtable}{|cccccc|}
      \caption{2--\emph{dmps} with $c=8$}\\
      \hline
      $n$&$f(n)$&$f^2(n)$&$f^3(n)$&$f^4(n)$&$f^5(n)$\\
      \hline
      \endfirsthead
      \hline
      $n$&$f(n)$&$f^2(n)$&$f^3(n)$&$f^4(n)$&$f^5(n)$\\
      \hline
      \endhead
      \hline \multicolumn{6}{r}{\textit{Continued on next page}} \\
      \endfoot
      \hline
      \endlastfoot
      10 & 80 & 40 & 20 & 60 & 80\\
      11 & 88 & 44 & 22 & 66 & 88\\
      12 & 86 & 48 & 24 & 62 & 86\\
      13 & 84 & 42 & 26 & 68 & 84\\
      14 & 82 & 46 & 28 & 64 & 82\\
      15 & 80 & 40 & 20 & 60 & 80\\
      16 & 88 & 44 & 22 & 66 & 88\\
      17 & 86 & 48 & 24 & 62 & 86\\
      18 & 84 & 42 & 26 & 68 & 84\\
      19 & 82 & 46 & 28 & 64 & 82\\
      \hline
      \end{longtable}
      \end{center}
  where $f(n)=\emph{MPS}(n,8)$ and $f^k(n)=f(f(\ldots(f(n))))$ of $k$ times.
  \item If $c:=9$, there are two term loops, starting with the f{}irst term (pendulum) and $\emph{PPS}(v,c,10,\emph{MPS})=$:
   \begin{center}
      \begin{longtable}{|ccc|}
      \caption{2--\emph{dmps} with $c=9$}\\
      \hline
      $n$&$f(n)$&$f^2(n)$\\
      \hline
      \endfirsthead
      \hline
      $n$&$f(n)$&$f^2(n)$\\
      \hline
      \endhead
      \hline \multicolumn{3}{r}{\textit{Continued on next page}} \\
      \endfoot
      \hline
      \endlastfoot
      10 & 90 & 10\\
      11 & 99 & 11\\
      12 & 98 & 12\\
      13 & 97 & 13\\
      14 & 96 & 14\\
      15 & 95 & 15\\
      16 & 94 & 16\\
      17 & 93 & 17\\
      18 & 92 & 18\\
      19 & 91 & 19\\
      \end{longtable}
      \end{center}
  where $f(n)=\emph{MPS}(n,9)$ and $f^k(n)=f(f(\ldots(f(n))))$ of $k$ times.
\end{enumerate}

\subsection{The Mixed Composition Periodic Sequences}

Let $n$ be a two-digits number. Add the digits, and add them again if the sum is greater than 10. Also take the absolute value of the dif{}ference of their digits. These are the f{}irst and second digits of $n_1$. Now repeat this.
\begin{prog}\label{Program MixPS} for the mixed composition periodic sequences.
  \begin{tabbing}
    $\emph{MixPS}(n,c):=$\=\vline\ $d\leftarrow \emph{dn}(n,10)$\\
    \>\vline\ $a_2\leftarrow\sum d$\\
    \>\vline\ $\sum dn(a_2,10)\ \ \emph{if}\ \ a_2>9$\\
    \>\vline\ $a_1\leftarrow\abs{d_2-d_1}$\\
    \>\vline\ $\emph{return}\ \ a_2\cdot10+a_1$\\
  \end{tabbing}
  where used program $\emph{dn}$, \ref{ProgramDn} and the function Mathcad $\sum v$, which sums vector elements $v$.
\end{prog}
For example, with $n=75$ we obtain the sequence:
\begin{multline*}
  \emph{PPS}[(75),0,20,\emph{MixPS}]=\\
  (75\ \ 32\ \ 51\ \ 64\ \ 12\ \ 31\ \ 42\ \ 62\ \ 84\ \ 34\ \ 71\\
  86\ \ 52\ \ 73\ \ 14\ \ 53\ \ 82\ \ 16\ \ 75)~,
\end{multline*}
(variable $c$ has no role, i.e. can be $c=0$).

There are no invariants in this case. Four numbers: 36, 90, 93, and 99 produce two-element loops. The longest loops have 18 elements. There also are loops of 4, 6, and 12 elements, \citep{Ibstedt1997}.

The mixed composition periodic sequences obtained with commands; $k:=5..25$\ \ $p_{k-4}:=\emph{prime}_k$\ \ $\emph{PPS}(p,0,21,\emph{MixPS})=$
\[
  {\scriptsize\left[\begin{array}{*{21}c}
  11 & 20 & 22 & 40 & 44 & 80 & 88 & 70 & 77 & 50 & 55 & 10 & 11 & 0 & 0 & 0 & 0 & 0 & 0 & 0 & 0\\
  13 & 42 & 62 & 84 & 34 & 71 & 86 & 52 & 73 & 14 & 53 & 82 & 16 & 75 & 32 & 51 & 64 & 12 & 31 & 42 & 0\\
  17 & 86 & 52 & 73 & 14 & 53 & 82 & 16 & 75 & 32 & 51 & 64 & 12 & 31 & 42 & 62 & 84 & 34 & 71 & 86 & 0\\
  19 & 18 & 97 & 72 & 95 & 54 & 91 & 18 & 0 & 0 & 0 & 0 & 0 & 0 & 0 & 0 & 0 & 0 & 0 & 0 & 0\\
  23 & 51 & 64 & 12 & 31 & 42 & 62 & 84 & 34 & 71 & 86 & 52 & 73 & 14 & 53 & 82 & 16 & 75 & 32 & 51 & 0\\
  29 & 27 & 95 & 54 & 91 & 18 & 97 & 72 & 95 & 0 & 0 & 0 & 0 & 0 & 0 & 0 & 0 & 0 & 0 & 0 & 0\\
  31 & 42 & 62 & 84 & 34 & 71 & 86 & 52 & 73 & 14 & 53 & 82 & 16 & 75 & 32 & 51 & 64 & 12 & 31 & 0 & 0\\
  37 & 14 & 53 & 82 & 16 & 75 & 32 & 51 & 64 & 12 & 31 & 42 & 62 & 84 & 34 & 71 & 86 & 52 & 73 & 14 & 0\\
  41 & 53 & 82 & 16 & 75 & 32 & 51 & 64 & 12 & 31 & 42 & 62 & 84 & 34 & 71 & 86 & 52 & 73 & 14 & 53 & 0\\
  43 & 71 & 86 & 52 & 73 & 14 & 53 & 82 & 16 & 75 & 32 & 51 & 64 & 12 & 31 & 42 & 62 & 84 & 34 & 71 & 0\\
  47 & 23 & 51 & 64 & 12 & 31 & 42 & 62 & 84 & 34 & 71 & 86 & 52 & 73 & 14 & 53 & 82 & 16 & 75 & 32 & 51\\
  53 & 82 & 16 & 75 & 32 & 51 & 64 & 12 & 31 & 42 & 62 & 84 & 34 & 71 & 86 & 52 & 73 & 14 & 53 & 0 & 0\\
  59 & 54 & 91 & 18 & 97 & 72 & 95 & 54 & 0 & 0 & 0 & 0 & 0 & 0 & 0 & 0 & 0 & 0 & 0 & 0 & 0\\
  61 & 75 & 32 & 51 & 64 & 12 & 31 & 42 & 62 & 84 & 34 & 71 & 86 & 52 & 73 & 14 & 53 & 82 & 16 & 75 & 0\\
  67 & 41 & 53 & 82 & 16 & 75 & 32 & 51 & 64 & 12 & 31 & 42 & 62 & 84 & 34 & 71 & 86 & 52 & 73 & 14 & 53\\
  71 & 86 & 52 & 73 & 14 & 53 & 82 & 16 & 75 & 32 & 51 & 64 & 12 & 31 & 42 & 62 & 84 & 34 & 71 & 0 & 0\\
  73 & 14 & 53 & 82 & 16 & 75 & 32 & 51 & 64 & 12 & 31 & 42 & 62 & 84 & 34 & 71 & 86 & 52 & 73 & 0 & 0\\
  79 & 72 & 95 & 54 & 91 & 18 & 97 & 72 & 0 & 0 & 0 & 0 & 0 & 0 & 0 & 0 & 0 & 0 & 0 & 0 & 0\\
  83 & 25 & 73 & 14 & 53 & 82 & 16 & 75 & 32 & 51 & 64 & 12 & 31 & 42 & 62 & 84 & 34 & 71 & 86 & 52 & 73\\
  89 & 81 & 97 & 72 & 95 & 54 & 91 & 18 & 97 & 0 & 0 & 0 & 0 & 0 & 0 & 0 & 0 & 0 & 0 & 0 & 0\\
  97 & 72 & 95 & 54 & 91 & 18 & 97 & 0 & 0 & 0 & 0 & 0 & 0 & 0 & 0 & 0 & 0 & 0 & 0 & 0 & 0\\
 \end{array}\right]}
\]

\subsection{Kaprekar Periodic Sequences}

Kaprekar\index{Kaprekar D. R.} proposed the following algorithm:
\begin{alg}\label{Algorithm Kaprekar}
  Fie $n\in\Ns$, we sort the number digits $n$ in decreasing order, thus the resulting number is $n'$, we sort the number digits $n$ in increasing order, thus the resulting number is $n''$. We denote by $K(n)$ the number $n'-n''$.
\end{alg}

Let the function $K:\set{1000,1001,\ldots,9999}$,
\begin{func}\label{FunctiaKaprekar} for the algorithm \ref{Algorithm Kaprekar}.
  \begin{multline}
  K(n)=\emph{reverse}(\emph{sort}(\emph{dn}(n,10)))\cdot\emph{Vb}(10,\emph{nrd}(n,10))\\
      -\emph{sort}(\emph{dn}(n,10))\cdot\emph{Vb}(10,\emph{nrd}(n,10))~,
\end{multline}
where function $\emph{dn}$, \ref{ProgramDn}, gives the vector with digits of numbers in the indicated numeration base and the Mathcad functions: $\emph{sort}$ for ascending sorting of a vector and $\emph{reverse}$ for reading a vector from tail to head. Function $Vb$ provides the vector $(10^{m-1}\ \ 10^{m-2}\ \ \ldots\ \ 1)^\textrm{T}$, where $m=nrd(n,10)$, i.e. the digits number of the number $n$ in base 10 10.
\end{func}
Examples: $K(7675)=2088$, $K(3215)=4086$, $K(5107)=7353$.

Since 1949 Kaprekar\index{Kaprekar D. R.} noted that if we apply several times to any number with four digits the above algorithm, we get the number 6147.  \cite{Kaprekar1955} conjectured that $K^m(n)=6147$ for $m\le7$ si $n\neq1111, 2222, \ldots, 9999$, the number 6147 is called \emph{Kaprekar constant}. The studies that followed, \citep{Deutsch+Goldman2004,WeissteinKaprekarRoutine}, conf{}irmed the Kaprekar conjecture. For exemplif{}ication, we consider the f{}irst 27 primes with 4 digits using the controls $k:=169..195$, $p_{k-168}:=\emph{prime}_k$, then by the call $\emph{PPS}(p,8,K)=$ we obtain the matrix:

\begin{center}
      \begin{longtable}{|cccccccc|}
      \caption{4--digits Kaprekar periodic sequences}\\
      \hline
      $n$&$K(n)$&$K^2(n)$&$K^3(n)$&$K^4(n)$&$K^5(n)$&$K^6(n)$&$K^7(n)$\\
      \hline
      \endfirsthead
      \hline
       $n$&$K(n)$&$K^2(n)$&$K^3(n)$&$K^4(n)$&$K^5(n)$&$K^6(n)$&$K^7(n)$\\
      \hline
      \endhead
      \hline \multicolumn{8}{r}{\textit{Continued on next page}} \\
      \endfoot
      \hline
      \endlastfoot
      1009 & 9081 & 9621 & 8352 & 6174 & 6174 & 0 & 0\\
      1013 & 2997 & 7173 & 6354 & 3087 & 8352 & 6174 & 6174\\
      1019 & 8991 & 8082 & 8532 & 6174 & 6174 & 0 & 0\\
      1021 & 1998 & 8082 & 8532 & 6174 & 6174 & 0 & 0\\
      1031 & 2997 & 7173 & 6354 & 3087 & 8352 & 6174 & 6174\\
      1033 & 3177 & 6354 & 3087 & 8352 & 6174 & 6174 & 0\\
      1039 & 9171 & 8532 & 6174 & 6174 & 0 & 0 & 0\\
      1049 & 9261 & 8352 & 6174 & 6174 & 0 & 0 & 0\\
      1051 & 4995 & 5355 & 1998 & 8082 & 8532 & 6174 & 6174\\
      1061 & 5994 & 5355 & 1998 & 8082 & 8532 & 6174 & 6174\\
      1063 & 6174 & 6174 & 0 & 0 & 0 & 0 & 0\\
      1069 & 9441 & 7992 & 7173 & 6354 & 3087 & 8352 & 6174\\
      1087 & 8532 & 6174 & 6174 & 0 & 0 & 0 & 0\\
      1091 & 8991 & 8082 & 8532 & 6174 & 6174 & 0 & 0\\
      1093 & 9171 & 8532 & 6174 & 6174 & 0 & 0 & 0\\
      1097 & 9531 & 8172 & 7443 & 3996 & 6264 & 4176 & 6174\\
      1103 & 2997 & 7173 & 6354 & 3087 & 8352 & 6174 & 6174\\
      1109 & 8991 & 8082 & 8532 & 6174 & 6174 & 0 & 0\\
      1117 & 5994 & 5355 & 1998 & 8082 & 8532 & 6174 & 6174\\
      1123 & 2088 & 8532 & 6174 & 6174 & 0 & 0 & 0\\
      1129 & 8082 & 8532 & 6174 & 6174 & 0 & 0 & 0\\
      1151 & 3996 & 6264 & 4176 & 6174 & 6174 & 0 & 0\\
      1153 & 4176 & 6174 & 6174 & 0 & 0 & 0 & 0\\
      1163 & 5175 & 5994 & 5355 & 1998 & 8082 & 8532 & 6174\\
      1171 & 5994 & 5355 & 1998 & 8082 & 8532 & 6174 & 6174\\
      1181 & 6993 & 6264 & 4176 & 6174 & 6174 & 0 & 0\\
      1187 & 7533 & 4176 & 6174 & 6174 & 0 & 0 & 0\\
      \end{longtable}
      \end{center}

\begin{itemize}
  \item For numbers with 2--digits. Applying the most 7 simple iteration, i.e.  $n=K(n)$, the Kaprekar algorithm becomes equal to one of \emph{Kaprekar constants}, and in these values the function $K$ becomes periodical of periods 1 or 5, as seen in the following Table:
      \[
      \begin{tabular}{|l|c|c|}
        \hline
        $\emph{CK}$ & $\nu$ & $p$ \\ \hline
        0 & 9 & 1\\
        $9=3^2$ & 24 & 5\\
        $27=3^3$ & 24 & 5\\
        $45=3^2\cdot5$ & 12 & 5\\
        $63=3^2\cdot7$ & 20 & 5\\
        $81=3^4$ & 1 & 5\\
        \hline
      \end{tabular}
      \]
      where $\emph{CK}=\emph{Kaprekar}\ \emph{constants}$, $\nu=\emph{frequent}$ and $p=\emph{periodicity}$. It follows that is f{}ixed point for the algorithm \ref{Algorithm Kaprekar} with frequency of 9 times (for 11, 22, \ldots, 99) with periodicity 1, i.e. $K(0)=0$; 9 is f{}ixed point for function $K$ with frequency of 24 times and with periodicity 5, i.e. $K^5(9)=9$; \ldots~.
  \item For numbers with 3--digits. Applying the most 7 simple iteration, i.e.  $n:=K(n)$, the algorithm \ref{Algorithm Kaprekar} becomes equal to one of the \emph{Kaprekar constants} from the following Table:
      \[
      \begin{tabular}{|l|c|c|}
        \hline
        $\emph{CK}$ & $\nu$ & $p$ \\ \hline
        0 & 9 & 1\\
        $495=3^2\cdot5\cdot11$ & 891 & 1\\
        \hline
      \end{tabular}
      \]
  and in these values the function $K$ becomes periodical of period 1, i.e.  $K(CK)=CK$, it follows that 0 and 495 are f{}ixed point for the function $K$.
  \item For numbers with 4--digits. Applying the most 7 simple iteration, i.e. $n:=K(n)$, the algorithm \ref{Algorithm Kaprekar} becomes equal to one of the \emph{Kaprekar constants} from the following Table:
      \[
      \begin{tabular}{|l|c|c|}
        \hline
        $CK$ & $\nu$ & $p$ \\ \hline
        0 & 9 & 1\\
        $6174=2\cdot3^2\cdot7^3$ & 8991 & 1\\
        \hline
      \end{tabular}
      \]
  and in these values the function $K$ becomes periodical of period 1, i.e. $K(CK)=CK$ , it follows that 0 and 6174 are f{}ixed point for the function $K$.
  \item For numbers with 5--digits. Applying the most 67 simple iteration, i.e.  $n:=K(n)$, $K(n)$ becomes equal to one of \emph{Kaprekar constants} from the following Table:
      \[
      \begin{tabular}{|l|c|c|}
        \hline
        $CK$ & $\nu$ & $p$ \\ \hline
        0 & 9 & 1\\
        $53955=3^2\cdot5\cdot11\cdot109$ & 2587 & 2\\
        $59994=2\cdot3^3\cdot11\cdot101$ & 415 & 2\\
        $61974=2\cdot3^2\cdot11\cdot313$ & 4770 & 4\\
        $62964=2^2\cdot3^3\cdot11\cdot53$ & 4754 & 4\\
        $63954=2\cdot3^2\cdot11\cdot17\cdot19$ & 24164 & 4\\
        $71973=3^2\cdot11\cdot727$ & 5816 & 4\\
        $74943=3^2\cdot11\cdot757$ & 27809 & 4\\
        $75933=3^2\cdot11\cdot13\cdot59$ & 9028 & 4\\
        $82962=2\cdot3^2\cdot11\cdot419$ & 5808 & 4\\
        $83952=2^4\cdot3^2\cdot11\cdot53$ & 4840 & 4\\
        \hline
      \end{tabular}
      \]
  and in these values the function $K$ becomes periodical of periods of 1, 2 or 4. It follows that 0 is f{}ixed point for \ref{Algorithm Kaprekar}, i.e. $K(0)=0$, 53955 and 59994 are of periods 2, equivalent to $K^2(\emph{CK})=K(K(\emph{CK}))=\emph{CK}$, and the rest are of period 4, i.e. $K^4(\emph{CK})=CK$. We note that all $\emph{CK}$ are multiples of $3^2\cdot11=99$.
  \item For numbers with 6--digits. Applying the most 11 simple iteration, i.e.  $n:=K(n)$, the function $K$ becomes equal to one of the \emph{Kaprekar constants} from the following Table:
      \[
      \begin{tabular}{|l|c|c|}
        \hline
        $CK$ & $\nu$ & $p$ \\ \hline
        0 & 4 & 1\\
        $4420876=2^2*3^5\cdot433$ & 154591 & 7\\
        $549945=3^2\cdot5\cdot11^2\cdot101$ & 840 & 1\\
        $631764=2^2\cdot3^2\cdot7\cdot23\cdot109$ & 24920 & 1\\
        $642654=2\cdot3^4\cdot3967$ & 13050 & 7\\
        $750843=3^3\cdot27809$ & 15845 & 7\\
        $840852=2^2\cdot3^2\cdot23357$ & 24370 & 7\\
        $851742=2\cdot3^3\cdot15773$ & 101550 & 7\\
        $860832=2^5\cdot3^2\cdot7^2\cdot61$ & 51730 & 7\\
        $862632=2^3\cdot3^2\cdot11981$ & 13100 & 7\\
        \hline
      \end{tabular}
      \]
and in these values the function $K$ becomes periodical of periods 1 or 7. It follows that 0, 549945 and 631764 are f{}ixed points for $K$, and the rest are of periods 7, i.e.  $K^7(\emph{CK})=\emph{CK}$.
\end{itemize}

\subsection{The Permutation Periodic Sequences}

A generalization of the regular functions would be the function resulting from the following algorithm
\begin{alg}\label{Algorithm PSP}
  Let $n\in\Ns$, be a number with $m$ digits, i.e. $n=\overline{d_1d_2\ldots d_m}$. We consider a permutation of the set $\set{1,2,\ldots,m}$, $pr=(i_1,i_2,\ldots,i_m)$, then the number $n'$ is given by the digits permutations of the numb $n$ using the permutation $\emph{pr}$, $n'=\overline{d_{i_1}d_{i_2}\ldots d_{i_m}}$. The new number equals $\abs{n-n'}$.
\end{alg}

\begin{prog}\label{Program PSP} for the algorithm \ref{Algorithm PSP}.
  \begin{tabbing}
    $\emph{PSP}(n,pr):=$\=\vline\ $d\leftarrow dn(n,10,last(pr)$\\
    \>\vline\ $f$\=$or\ k\in1..\emph{last}(d)$\\
    \>\vline\ \>\ $nd_k\leftarrow d_{(pr_k)}$\\
    \>\vline\ $\emph{return}\ \abs{n-nd\cdot Vb(10,last(d))}$\\
  \end{tabbing}
\end{prog}

For example, with 2 digits sequence by commands $k:=1..27$\ \ $v_k:=13+k$\ \ $\emph{PPS}\big(v,14,(2\ \ 1)^\textrm{T}\big)=$ generates the matrix:
\begin{center}
      \begin{longtable}{|ccccccccc|}
      \caption{2--digits permutation periodic sequences}\\
      \hline
      $n$&$f(n)$&$f^2(n)$&$f^3(n)$&$f^4(n)$&$f^5(n)$&$f^6(n)$&$f^7(n)$&$f^8(n)$\\
      \hline
      \endfirsthead
      \hline
       $n$&$f(n)$&$f^2(n)$&$f^3(n)$&$f^4(n)$&$f^5(n)$&$f^6(n)$&$f^7(n)$&$f^8(n)$\\
      \hline
      \endhead
      \hline \multicolumn{9}{r}{\textit{Continued on next page}} \\
      \endfoot
      \hline
      \endlastfoot
      14 & 27 & 45 & 9 & 81 & 63 & 27 & 0 & 5\\
      15 & 36 & 27 & 45 & 9 & 81 & 63 & 27 & 5\\
      16 & 45 & 9 & 81 & 63 & 27 & 45 & 0 & 5\\
      17 & 54 & 9 & 81 & 63 & 27 & 45 & 9 & 5\\
      18 & 63 & 27 & 45 & 9 & 81 & 63 & 0 & 5\\
      19 & 72 & 45 & 9 & 81 & 63 & 27 & 45 & 5\\
      20 & 18 & 63 & 27 & 45 & 9 & 81 & 63 & 5\\
      21 & 9 & 81 & 63 & 27 & 45 & 9 & 0 & 5\\
      22 & 22 & 0 & 0 & 0 & 0 & 0 & 0 & 1\\
      23 & 9 & 81 & 63 & 27 & 45 & 9 & 0 & 5\\
      24 & 18 & 63 & 27 & 45 & 9 & 81 & 63 & 5\\
      25 & 27 & 45 & 9 & 81 & 63 & 27 & 0 & 5\\
      26 & 36 & 27 & 45 & 9 & 81 & 63 & 27 & 5\\
      27 & 45 & 9 & 81 & 63 & 27 & 0 & 0 & 5\\
      28 & 54 & 9 & 81 & 63 & 27 & 45 & 9 & 5\\
      29 & 63 & 27 & 45 & 9 & 81 & 63 & 0 & 5\\
      30 & 27 & 45 & 9 & 81 & 63 & 27 & 0 & 5\\
      31 & 18 & 63 & 27 & 45 & 9 & 81 & 63 & 5\\
      32 & 9 & 81 & 63 & 27 & 45 & 9 & 0 & 5\\
      33 & 33 & 0 & 0 & 0 & 0 & 0 & 0 & 1\\
      34 & 9 & 81 & 63 & 27 & 45 & 9 & 0 & 5\\
      35 & 18 & 63 & 27 & 45 & 9 & 81 & 63 & 5\\
      36 & 27 & 45 & 9 & 81 & 63 & 27 & 0 & 5\\
      37 & 36 & 27 & 45 & 9 & 81 & 63 & 27 & 5\\
      38 & 45 & 9 & 81 & 63 & 27 & 45 & 0 & 5\\
      39 & 54 & 9 & 81 & 63 & 27 & 45 & 9 & 5\\
      40 & 36 & 27 & 45 & 9 & 81 & 63 & 27 & 5\\
      \end{longtable}
      \end{center}

Analysis tells us that this matrix of \emph{constants periodic sequences} associated to the permutation $(2\ \ 1)^\textrm{T}$ are the nonnull penultimate numbers in each row of the matrix: 27, 45, 63, 9, 22, 33 (in order of appearance in the matrix). The last column of the matrix represents the periodicity of each \emph{constant periodic sequences}. The constants 27, 45, 63 and 9 have the periodicity 5, i.e. $\emph{PSP}^5(27)=27$, $\emph{PSP}^5(45)=45$, etc. The constants $22$ and $33$ have a periodicity equals to 1, i.e. $\emph{PSP}(22)=0$ and $\emph{PSP}(33)=0$. One may count the frequency of occurrence of each \emph{constant periodic sequences}.

We present a study on permutation periodic sequences for 3--digits numbers having 3 digits relatively to the 6th permutation of the set $\set{1,2,3}$. The required commands are: $k:=100..999$\ \ $v_{k-99}:=k$\ \ $\emph{PPS}\big(v,20,pr\big)=$, where $\emph{pr}$ is a permutation of the set $\set{1,2,3}$.
\begin{enumerate}
  \item For the permutation $(2\ 3\ 1)^\textrm{T}$, the commands $k:=970..999$\ \ $v_{k-969}:=k$\ \ $\emph{PPS}\big(v,20,(2\ 3\ 1)^\textrm{T}\big)=$ generates the matrix:
  {\scriptsize\begin{center}
  \begin{longtable}{|ccccccccccccc|}
      \caption{3--digits PPS with permutation $(2\ 3\ 1)^\textrm{T}$}\\
      \hline
      $n$&$f(n)$&$f^2(n)$&$f^3(n)$&$f^4(n)$&$f^5(n)$&$f^6(n)$&$f^7(n)$&$f^8(n)$&$f^9(n)$&$f^{10}(n)$&$f^{11}(n)$&$f^{12}(n)$\\
      \hline
      \endfirsthead
      \hline
       $n$&$f(n)$&$f^2(n)$&$f^3(n)$&$f^4(n)$&$f^5(n)$&$f^6(n)$&$f^7(n)$&$f^8(n)$&$f^9(n)$&$f^{10}(n)$&$f^{11}(n)$&$f^{12}(n)$\\
      \hline
      \endhead
      \hline \multicolumn{13}{r}{\textit{Continued on next page}} \\
      \endfoot
      \hline
      \endlastfoot
      970 & 261 & 351 & 162 & 459 & 135 & 216 & 54 & 486 & 378 & 405 & 351 & 0\\
      971 & 252 & 270 & 432 & 108 & 27 & 243 & 189 & 702 & 675 & 81 & 729 & 432\\
      972 & 243 & 189 & 702 & 675 & 81 & 729 & 432 & 108 & 27 & 243 & 0 & 0\\
      973 & 234 & 108 & 27 & 243 & 189 & 702 & 675 & 81 & 729 & 432 & 108 & 0\\
      974 & 225 & 27 & 243 & 189 & 702 & 675 & 81 & 729 & 432 & 108 & 27 & 0\\
      975 & 216 & 54 & 486 & 378 & 405 & 351 & 162 & 459 & 135 & 216 & 0 & 0\\
      976 & 207 & 135 & 216 & 54 & 486 & 378 & 405 & 351 & 162 & 459 & 135 & 0\\
      977 & 198 & 783 & 54 & 486 & 378 & 405 & 351 & 162 & 459 & 135 & 216 & 54\\
      978 & 189 & 702 & 675 & 81 & 729 & 432 & 108 & 27 & 243 & 189 & 0 & 0\\
      979 & 180 & 621 & 405 & 351 & 162 & 459 & 135 & 216 & 54 & 486 & 378 & 405\\
      980 & 171 & 540 & 135 & 216 & 54 & 486 & 378 & 405 & 351 & 162 & 459 & 135\\
      981 & 162 & 459 & 135 & 216 & 54 & 486 & 378 & 405 & 351 & 162 & 0 & 0\\
      982 & 153 & 378 & 405 & 351 & 162 & 459 & 135 & 216 & 54 & 486 & 378 & 0\\
      983 & 144 & 297 & 675 & 81 & 729 & 432 & 108 & 27 & 243 & 189 & 702 & 675\\
      984 & 135 & 216 & 54 & 486 & 378 & 405 & 351 & 162 & 459 & 135 & 0 & 0\\
      985 & 126 & 135 & 216 & 54 & 486 & 378 & 405 & 351 & 162 & 459 & 135 & 0\\
      986 & 117 & 54 & 486 & 378 & 405 & 351 & 162 & 459 & 135 & 216 & 54 & 0\\
      987 & 108 & 27 & 243 & 189 & 702 & 675 & 81 & 729 & 432 & 108 & 0 & 0\\
      988 & 99 & 891 & 27 & 243 & 189 & 702 & 675 & 81 & 729 & 432 & 108 & 27\\
      989 & 90 & 810 & 702 & 675 & 81 & 729 & 432 & 108 & 27 & 243 & 189 & 702\\
      990 & 81 & 729 & 432 & 108 & 27 & 243 & 189 & 702 & 675 & 81 & 0 & 0\\
      991 & 72 & 648 & 162 & 459 & 135 & 216 & 54 & 486 & 378 & 405 & 351 & 162\\
      992 & 63 & 567 & 108 & 27 & 243 & 189 & 702 & 675 & 81 & 729 & 432 & 108\\
      993 & 54 & 486 & 378 & 405 & 351 & 162 & 459 & 135 & 216 & 54 & 0 & 0\\
      994 & 45 & 405 & 351 & 162 & 459 & 135 & 216 & 54 & 486 & 378 & 405 & 0\\
      995 & 36 & 324 & 81 & 729 & 432 & 108 & 27 & 243 & 189 & 702 & 675 & 81\\
      996 & 27 & 243 & 189 & 702 & 675 & 81 & 729 & 432 & 108 & 27 & 0 & 0\\
      997 & 18 & 162 & 459 & 135 & 216 & 54 & 486 & 378 & 405 & 351 & 162 & 0\\
      998 & 9 & 81 & 729 & 432 & 108 & 27 & 243 & 189 & 702 & 675 & 81 & 0\\
      999 & 0 & 0 & 0 & 0 & 0 & 0 & 0 & 0 & 0 & 0 & 0 & 0\\
      \end{longtable}
  \end{center}}
  We have the following list of \emph{constant periodic sequences}, with frequency of occurrence $\nu$ and periodicity $p$.
  \begin{center}
      \begin{longtable}{|l|c|c|}
      \caption{3--digits PPS with the permutation $(2\ 3\ 1)^\textrm{T}$}\\
      \hline
      $CPS$&$\nu$&$p$\\
      \hline
      \endfirsthead
      \hline
       $CPS$&$\nu$&$p$\\
      \hline
      \endhead
      \hline \multicolumn{1}{r}{\textit{Continued on next page}} \\
      \endfoot
      \hline
      \endlastfoot
       0 & 9 & 1\\
       $27=3^3$ & 67 & 9\\
       $54=2\cdot3^3$ & 68 & 9\\
       $81=3^4$ & 94 & 9\\
       $108=2^2\cdot3^3$ & 71 & 9\\
       $135=3^3\cdot5$ & 70 & 9\\
       $162=2\cdot3^4$ & 70 & 9\\
       $189=3^3\cdot7$ & 34 & 9\\
       $216=2^3\cdot3^3$ & 30 & 9\\
       $243=3^5$ & 45 & 9\\
       $333=3^2\cdot37$ & 11 & 1\\
       $351=3^3\cdot13$ & 50 & 9\\
       $378=2\cdot3^3\cdot7$ & 50 & 9\\
       $405=3^4\cdot5$ & 47 & 9\\
       $432=2^4\cdot3^3$ & 54 & 9\\
       $459=3^3\cdot17$ & 26 & 9\\
       $486=2\cdot3^5$ & 25 & 9\\
       $666=2\cdot3^2\cdot37$ & 5 & 1\\
       $675=3^3\cdot5^2$ & 50 & 9\\
       $702=2\cdot3^3\cdot13$ & 21 & 9\\
       $729=3^6$ & 3 & 9\\
      \end{longtable}
      \end{center}
      We note that all nonnull \emph{constants periodic sequences} are multiples of $3^2$.
  \item For permutation $(3\ 1\ 2)^\textrm{T}$ we have the following list of \emph{constant periodic sequences}, with frequency of occurrence $\nu$ and periodicity $p$.
      \begin{center}
      \begin{longtable}{|l|c|c|}
      \caption{3--digits PPS with the permutation $(3\ 1\ 2)^\textrm{T}$}\\
      \hline
      $CPS$&$\nu$&$p$\\
      \hline
      \endfirsthead
      \hline
       $CPS$&$\nu$&$p$\\
      \hline
      \endhead
      \hline \multicolumn{1}{r}{\textit{Continued on next page}} \\
      \endfoot
      \hline
      \endlastfoot
       0 & 9 & 1\\
       $27=3^3$ & 48 & 9\\
       $54=2\cdot3^3$ & 50 & 9\\
       $81=3^4$ & 52 & 9\\
       $108=2^2\cdot3^3$ & 76 & 9\\
       $135=3^3\cdot5$ & 70 & 9\\
       $162=2\cdot3^4$ & 71 & 9\\
       $189=3^3\cdot7$ & 70 & 9\\
       $216=2^3\cdot3^3$ & 50 & 9\\
       $243=3^5$ & 91 & 9\\
       $333=3^2\cdot37$ & 11 & 1\\
       $351=3^3\cdot13$ & 27 & 9\\
       $378=2\cdot3^3\cdot7$ & 28 & 9\\
       $405=3^4\cdot5$ & 45 & 9\\
       $432=2^4\cdot3^3$ & 54 & 9\\
       $459=3^3\cdot17$ & 47 & 9\\
       $486=2\cdot3^5$ & 48 & 9\\
       $666=2\cdot3^2\cdot37$ & 5 & 1\\
       $675=3^3\cdot5^2$ & 5 & 9\\
       $702=2\cdot3^3\cdot13$ & 22 & 9\\
       $729=3^6$ & 21 & 9\\
      \end{longtable}
      \end{center}
      We note that we have the same Kaprecar constants as in permutation $(2\ 3\ 1)^\textrm{T}$ just that there are other frequency of occurrence.
  \item For permutation $(1\ 3\ 2)^\textrm{T}$, the commands $k:=300..330$\ \ $v_{k-299}:=k$\ \ $\emph{PPS}\big(v,20,(1\ 3\ 2)^\textrm{T}\big)=$ generates the matrix:
  \begin{center}
  \begin{longtable}{|cccccccc|}
      \caption{3--digits PPS with permutation $(1\ 3\ 2)^\textrm{T}$}\\
      \hline
      $n$&$f(n)$&$f^2(n)$&$f^3(n)$&$f^4(n)$&$f^5(n)$&$f^6(n)$&$f^7(n)$\\
      \hline
      \endfirsthead
      \hline
       $n$&$f(n)$&$f^2(n)$&$f^3(n)$&$f^4(n)$&$f^5(n)$&$f^6(n)$&$f^7(n)$\\
      \hline
      \endhead
      \hline \multicolumn{8}{r}{\textit{Continued on next page}} \\
      \endfoot
      \hline
      \endlastfoot
      300 & 0 & 0 & 0 & 0 & 0 & 0 & 0\\
      301 & 9 & 81 & 63 & 27 & 45 & 9 & 0\\
      302 & 18 & 63 & 27 & 45 & 9 & 81 & 63\\
      303 & 27 & 45 & 9 & 81 & 63 & 27 & 0\\
      304 & 36 & 27 & 45 & 9 & 81 & 63 & 27\\
      305 & 45 & 9 & 81 & 63 & 27 & 45 & 0\\
      306 & 54 & 9 & 81 & 63 & 27 & 45 & 9\\
      307 & 63 & 27 & 45 & 9 & 81 & 63 & 0\\
      308 & 72 & 45 & 9 & 81 & 63 & 27 & 45\\
      309 & 81 & 63 & 27 & 45 & 9 & 81 & 0\\
      310 & 9 & 81 & 63 & 27 & 45 & 9 & 0\\
      311 & 0 & 0 & 0 & 0 & 0 & 0 & 0\\
      312 & 9 & 81 & 63 & 27 & 45 & 9 & 0\\
      313 & 18 & 63 & 27 & 45 & 9 & 81 & 63\\
      314 & 27 & 45 & 9 & 81 & 63 & 27 & 0\\
      315 & 36 & 27 & 45 & 9 & 81 & 63 & 27\\
      316 & 45 & 9 & 81 & 63 & 27 & 45 & 0\\
      317 & 54 & 9 & 81 & 63 & 27 & 45 & 9\\
      318 & 63 & 27 & 45 & 9 & 81 & 63 & 0\\
      319 & 72 & 45 & 9 & 81 & 63 & 27 & 45\\
      320 & 18 & 63 & 27 & 45 & 9 & 81 & 63\\
      321 & 9 & 81 & 63 & 27 & 45 & 9 & 0\\
      322 & 0 & 0 & 0 & 0 & 0 & 0 & 0\\
      323 & 9 & 81 & 63 & 27 & 45 & 9 & 0\\
      324 & 18 & 63 & 27 & 45 & 9 & 81 & 63\\
      325 & 27 & 45 & 9 & 81 & 63 & 27 & 0\\
      326 & 36 & 27 & 45 & 9 & 81 & 63 & 27\\
      327 & 45 & 9 & 81 & 63 & 27 & 45 & 0\\
      328 & 54 & 9 & 81 & 63 & 27 & 45 & 9\\
      329 & 63 & 27 & 45 & 9 & 81 & 63 & 0\\
      330 & 27 & 45 & 9 & 81 & 63 & 27 & 0\\
      \end{longtable}
  \end{center}
  We have the following list of \emph{constant periodic sequences}, with frequency of occurrences $\nu$ and periodicity $p$.
      \[
      \begin{tabular}{|l|c|c|}
        \hline
        $CPS$ & $\nu$ & $p$ \\ \hline
        0 & 90 & 1\\
        $9=3^2$ & 234 & 5\\
        $27=3^3$ & 234 & 5\\
        $45=3^2\cdot5$ & 126 & 5\\
        $63=3^2\cdot7$ & 198 & 5\\
        $81=3^4$ & 18 & 5\\
        \hline
      \end{tabular}
      \]
  \item For permutation $(2\ 1\ 3)^\textrm{T}$ we have the following list of \emph{constant periodic sequences}, with frequency of occurrences $\nu$ and periodicity $p$.
      \[
      \begin{tabular}{|l|c|c|}
        \hline
        $CPS$ & $\nu$ & $p$ \\ \hline
        0 & 90 & 1\\
        $90=2\cdot3^2\cdot5$ & 240 & 5\\
        $270=2\cdot3^3\cdot5$ & 240 & 5\\
        $450=2\cdot3^2\cdot5^2$ & 120 & 5\\
        $630=2\cdot5^2\cdot5\cdot7$ & 200 & 5\\
        $810=2\cdot3^4\cdot5$ & 10 & 5\\
        \hline
      \end{tabular}
      \]
  \item For permutation $(3\ 2\ 1)^\textrm{T}$ we have the following list of \emph{constant periodic sequences}, with frequency of occurrences $\nu$ and periodicity $p$.
      \[
      \begin{tabular}{|l|c|c|}
        \hline
        $CPS$ & $\nu$ & $p$ \\ \hline
        0 & 90 & 1\\
        $99=3^2\cdot11$ & 240 & 5\\
        $297=3^3\cdot11$ & 240 & 5\\
        $495=3^2\cdot5\cdot11$ & 120 & 5\\
        $693=3^2\cdot7\cdot11$ & 200 & 5\\
        $891=3^4\cdot11$ & 10 & 5\\
        \hline
      \end{tabular}
      \]
  \item For identical permutation, obvious, we have only \emph{constant periodic sequences} 0 with frequency 900 and periodicity  1.
\end{enumerate}

\section{Erd\"{o}s--Smarandache Numbers}

The solutions to the Diophantine equation $P(n)=S(n)$, where $P(n)$ is the largest prime factor which divides $n$, and $S(n)$ is the classical Smarandache function, are Erd\"{o}s--Smarandache numbers, \citep{Erdos+Ashbacher1997,Tabirca2000}, \cite[A048839]{SloaneOEIS}.

\begin{prog}\label{Program ES} generation the Erd\"{o}s--Smarandache numbers.
  \begin{tabbing}
    $\emph{ES}(a,b):=$\=\vline\ $j\leftarrow0$\\
    \>\vline\ $f$\=$or\ n\in a..b$\\
    \>\vline\ \>\vline\ $m\leftarrow \emph{max}\big(\emph{Fa}(n)^{\langle1\rangle}\big)$\\
    \>\vline\ \>\vline\ $i$\=$f\ S(n)\textbf{=}m$\\
    \>\vline\ \>\vline\ \>\vline\ $j\leftarrow j+1$\\
    \>\vline\ \>\vline\ \>\vline\ $s_j\leftarrow n$\\
    \>\vline\ $\emph{return}\ \ s$\\
  \end{tabbing}
  The program use $S$ (Smarandache function) and $\emph{Fa}$ (of factorization the numbers) programs.
\end{prog}

Erd\"{o}s--Smarandache numbers one obtains with the command $ES(2,200)^\textrm{T}=$ 2, 3, 5, 6, 7, 10, 11, 13, 14, 15, 17, 19, 20, 21, 22, 23, 26, 28, 29, 30, 31, 33, 34, 35, 37, 38, 39, 40, 41, 42, 43, 44, 46, 47, 51, 52, 53, 55, 56, 57, 58, 59, 60, 61, 62, 63, 65, 66, 67, 68, 69, 70, 71, 73, 74, 76, 77, 78, 79, 82, 83, 84, 85, 86, 87, 88, 89, 91, 92, 93, 94, 95, 97, 99, 101, 102, 103, 104, 105, 106, 107, 109, 110, 111, 112, 113, 114, 115, 116, 117, 118, 119, 120, 122, 123, 124, 126, 127, 129, 130, 131, 132, 133, 134, 136, 137, 138, 139, 140, 141, 142, 143, 145, 146, 148, 149, 151, 152, 153, 154, 155, 156, 157, 158, 159, 161, 163, 164, 165, 166, 167, 168, 170, 171, 172, 173, 174, 176, 177, 178, 179, 181, 182, 183, 184, 185, 186, 187, 188, 190, 191, 193, 194, 195, 197, 198, 199~.

\begin{prog}\label{Program ES1} generation the Erd\"{o}s--Smarandache numbers not prime.
  \begin{tabbing}
    $\emph{ES1}(a,b):=$\=\vline\ $j\leftarrow0$\\
    \>\vline\ $f$\=$or\ n\in a..b$\\
    \>\vline\ \>\vline\ $m\leftarrow \emph{max}\big(\emph{Fa}(n)^{\langle1\rangle}\big)$\\
    \>\vline\ \>\vline\ $i$\=$f\ S(n)\textbf{=}m\wedge S(n)\neq n$\\
    \>\vline\ \>\vline\ \>\vline\ $j\leftarrow j+1$\\
    \>\vline\ \>\vline\ \>\vline\ $s_j\leftarrow n$\\
    \>\vline\ $\emph{return}\ \ s$\\
  \end{tabbing}
  The program use $S$ (Smarandache function) and $Fa$ (of factorization the numbers) programs.
\end{prog}

Erd\"{o}s--Smarandache numbers, that not are primes, one obtains with command $ES1(2,200)^\textrm{T}=$ 6, 10, 14, 15, 20, 21, 22, 26, 28, 30, 33, 34, 35, 38, 39, 40, 42, 44, 46, 51, 52, 55, 56, 57, 58, 60, 62, 63, 65, 66, 68, 69, 70, 74, 76, 77, 78, 82, 84, 85, 86, 87, 88, 91, 92, 93, 94, 95, 99, 102, 104, 105, 106, 110, 111, 112, 114, 115, 116, 117, 118, 119, 120, 122, 123, 124, 126, 129, 130, 132, 133, 134, 136, 138, 140, 141, 142, 143, 145, 146, 148, 152, 153, 154, 155, 156, 158, 159, 161, 164, 165, 166, 168, 170, 171, 172, 174, 176, 177, 178, 182, 183, 184, 185, 186, 187, 188, 190, 194, 195, 198~.

\begin{prog}\label{Program ES2} generation the Erd\"{o}s--Smarandache with $\emph{nf}$ prime factors.
  \begin{tabbing}
    $\emph{ES2}(a,b,nf):=$\=\vline\ $j\leftarrow0$\\
    \>\vline\ $f$\=$or\ n\in a..b$\\
    \>\vline\ \>\vline\ $m\leftarrow \emph{max}\big(Fa(n)^{\langle1\rangle}\big)$\\
    \>\vline\ \>\vline\ $i$\=$f\ S(n)\textbf{=}m\wedge rows\big(Fa(n)^{\langle1\rangle}\big)\textbf{=}nf$\\
    \>\vline\ \>\vline\ \>\vline\ $j\leftarrow j+1$\\
    \>\vline\ \>\vline\ \>\vline\ $s_j\leftarrow n$\\
    \>\vline\ $\emph{return}\ \ s$\\
  \end{tabbing}
  The program use $S$ (Smarandache function) and $\emph{Fa}$ (of factorization the numbers) programs.
\end{prog}
\begin{enumerate}
  \item Erd\"{o}s--Smarandache numbers, that have exactly two prime factors, one obtains with the command $ES2(2,200,2)^\textrm{T}=$ 6, 10, 14, 15, 20, 21, 22, 26, 28, 33, 34, 35, 38, 39, 40, 44, 46, 51, 52, 55, 56, 57, 58, 62, 63, 65, 68, 69, 74, 76, 77, 82, 85, 86, 87, 88, 91, 92, 93, 94, 95, 99, 104, 106, 111, 112, 115, 116, 117, 118, 119, 122, 123, 124, 129, 133, 134, 136, 141, 142, 143, 145, 146, 148, 152, 153, 155, 158, 159, 161, 164, 166, 171, 172, 176, 177, 178, 183, 184, 185, 187, 188, 194~.
  \item Erd\"{o}s--Smarandache numbers, that have exactly three prime factors, one obtains with the command $ES2(2,200,3)^\textrm{T}=$ 30, 42, 60, 66, 70, 78, 84, 102, 105, 110, 114, 120, 126, 130, 132, 138, 140, 154, 156, 165, 168, 170, 174, 182, 186, 190, 195, 198~.
  \item Erd\"{o}s--Smarandache numbers, that have exactly four prime factors, one obtains with the command $ES2(2,1000,4)^\textrm{T}=$ 210, 330, 390, 420, 462, 510, 546, 570, 630, 660, 690, 714, 770, 780, 798, 840, 858, 870, 910, 924, 930, 966, 990~.
  \item Erd\"{o}s--Smarandache numbers, that have exactly f{}ive prime factors, one obtains with the command $ES2(2,5000,5)^\textrm{T}=$ 2310, 2730, 3570, 3990, 4290, 4620, 4830~.
\end{enumerate}

\section{Multiplicative Sequences}

\subsection{Multiplicative Sequence of First 2 Terms}

General def{}inition: if $m_1$, $m_2$, are the f{}irst two terms of the sequence, then $m_k$, for $k\ge3$, is the smallest number equal to the product of two previous distinct terms.
\begin{prog}\label{Program MS2} generate multiplicative sequence with f{}irst two terms.
  \begin{tabbing}
    $\emph{MS2}(m_1,m_2,L):=$\=\vline\ $m\leftarrow(m_1\ \ m_2)^\textrm{T}$\\
    \>\vline\ $j\leftarrow3$\\
    \>\vline\ $w$\=$\emph{hile}\ \ j\le L$\\
    \>\vline\ \>\vline\ $i\leftarrow1$\\
    \>\vline\ \>\vline\ $f$\=$or\ \ k_1\in1..\emph{last(m)}-1$\\
    \>\vline\ \>\vline\ \>\ $f$\=$or\ \ k_2\in k_1+1..\emph{last}(m)$\\
    \>\vline\ \>\vline\ \>\ \>\vline\ $y\leftarrow m_{k_1}\cdot m_{k_2}$\\
    \>\vline\ \>\vline\ \>\ \>\vline\ $i$\=$f\ \ y> max(m)$\\
    \>\vline\ \>\vline\ \>\ \>\vline\ \>\vline\ $x_i\leftarrow y$\\
    \>\vline\ \>\vline\ \>\ \>\vline\ \>\vline\ $i\leftarrow i+1$\\
    \>\vline\ \>\vline\ $m_j\leftarrow min(x)$\\
    \>\vline\ \>\vline\ $j\leftarrow j+1$\\
    \>\vline\ $\emph{return}\ \ m$\\
  \end{tabbing}
\end{prog}

Examples:
\begin{enumerate}
  \item $\emph{ms23}:=\emph{MS2}(2,3,25)$\ \ $ms23^\textrm{T}\rightarrow$(2\ \ 3\ \ 6\ \ 12\ \ 18\ \ 24\ \ 36\ \ 48\ \ 54\ \ 72\ \ 96\ \ 108\ \ 144\ \ 162\ \ 192\ \ 216\ \ 288\ \ 324\ \ 384\ \ 432\ \ 486\ \ 576\ \ 648\ \ 972\ \ 1458)~;
  \item $\emph{ms37}:=\emph{MS2}(3,7,25)$\ \ $ms37^\emph{T}\rightarrow$(3\ \ 7\ \ 21\ \ 63\ \ 147\ \ 189\ \ 441\ \ 567\ \ 1029\ \ 1323\ \ 1701\ \ 3087\ \ 3969\ \ 5103\ \ 7203\ \ 9261\ \ 11907\ \ 15309\ \ 21609\ \ 27783\ \ 35721\ \ 50421\ \ 64827\ \ 151263\ \ 352947)~;
  \item $\emph{ms1113}:=\emph{MS2}(11,13,25)$\ \ $ms1113^\textrm{T}\rightarrow$(11\ \ 13\ \ 143\ \ 1573\ \ 1859\ \ 17303\ \ 20449\ \ 24167\ \ 190333\ \ 224939\ \ 265837\ \ 314171\ \ 2093663\ \ 2474329\ \ 2924207\ \ 3455881\ \ 4084223\ \ 23030293\ \ 27217619\ \ 32166277\ \ 38014691)~.
\end{enumerate}

The program $\emph{MS2}$, \ref{Program MS2}, is equivalent with the program that generates the products $m_1$, $m_2$, $m_1\cdot m_2$, $m_1^2\cdot m_2^1$, $m_1^1\cdot m_2^2$, $m_1^3\cdot m_2^1$, $m_1^2\cdot m_2^2$, $m_1^1\cdot m_2^3$, \ldots, noting that the series f{}inally have to be ascending sorted because we have no guarantee that such terms are generated in ascending order. The program for this algorithm is simpler and probably faster..

\begin{prog}\label{Program VarMS2} generate multiplicative sequence with f{}irst two terms.
  \begin{tabbing}
    $\emph{VarMS2}(m_1,m_2,L):=$\=\vline\ $m\leftarrow(m_1\ \ m_2)^\textrm{T}$\\
    \>\vline\ $f$\=$or\ \ k\in1..\emph{ceil}\left(\frac{\sqrt{8L+1}-1}{2}\right)$\\
    \>\vline\ \>\vline\ $f$\=$or\ \ i\in1..k$\\
    \>\vline\ \>\vline\ \>\vline\ $m\leftarrow stack[m,(m_1)^{k+1-i}\cdot(m_2)^i]$\\
    \>\vline\ \>\vline\ \>\vline\ $\emph{return}\ \ \emph{sort}(m)\ \ \emph{if}\ \ \emph{last(m)}\le L$\\
  \end{tabbing}
\end{prog}

\subsection{Multiplicative Sequence of First 3 Terms}

General def{}inition: if $m_1$, $m_2$, $m_3$, are the f{}irst two terms of the sequence, then $m_k$, for $k\ge4$, is the smallest number equal to the product of three previous distinct terms.

\begin{prog}\label{Program MS3} generate multiplicative sequence with f{}irst three terms.
  \begin{tabbing}
    $\emph{MS3}(m_1,m_2,m_3,L):=$\=\vline\ $m\leftarrow(m_1\ \ m_2\ \ m_3)^\textrm{T}$\\
    \>\vline\ $j\leftarrow4$\\
    \>\vline\ $w$\=$\emph{hile}\ \ j\le L$\\
    \>\vline\ \>\vline\ $i\leftarrow1$\\
    \>\vline\ \>\vline\ $f$\=$or\ \ k_1\in1..\emph{last(m)}-1$\\
    \>\vline\ \>\vline\ \>\ $f$\=$or\ \ k_2\in k_1+1..\emph{last(m)}$\\
    \>\vline\ \>\vline\ \>\ \>\ $f$\=$or\ \ k_3\in k_2+1..\emph{last(m)}$\\
    \>\vline\ \>\vline\ \>\ \>\ \>\vline\ $y\leftarrow m_{k_1}\cdot m_{k_2}$\\
    \>\vline\ \>\vline\ \>\ \>\ \>\vline\ $i$\=$f\ \ y> max(m)$\\
    \>\vline\ \>\vline\ \>\ \>\ \>\vline\ \>\vline\ $x_i\leftarrow y$\\
    \>\vline\ \>\vline\ \>\ \>\ \>\vline\ \>\vline\ $i\leftarrow i+1$\\
    \>\vline\ \>\vline\ $m_j\leftarrow min(x)$\\
    \>\vline\ \>\vline\ $j\leftarrow j+1$\\
    \>\vline\ $\emph{return}\ \ m$\\
  \end{tabbing}
\end{prog}

Examples:
\begin{enumerate}
  \item $\emph{ms235}:=\emph{MS3}(2,3,5,30)$,\ $ms235^\textrm{T}\rightarrow$(2,\ 3,\ 5,\ 30,\ 180,\ 300,\ 450,\ 1080,\ 1800,\ 2700,\ 3000,\ 4500,\ 6480,\ 6750,\ 10800,\ 16200,\ 18000,\ 27000,\ 30000,\ 38880,\ 40500,\ 45000,\ 64800,\ 67500,\ 97200,\ 101250,\ 108000,\ 162000,\ 180000,\ 233280)~;
  \item $\emph{ms237}:=\emph{MS3}(2,3,7,30)$,\ $ms237^\textrm{T}\rightarrow$(2,\ 3,\ 7,\ 42,\ 252,\ 588,\ 882,\ 1512,\ 3528,\ 5292,\ 8232,\ 9072,\ 12348,\ 18522,\ 21168,\ 31752,\ 49392,\ 54432,\ 74088,\ 111132,\ 115248,\ 127008,\ 172872,\ 190512,\ 259308,\ 296352,\ 326592,\ 388962,\ 444528,\ 666792)~;
  \item $\emph{ms111317}:=\emph{MS3}(11,13,17,30)$,\ $ms111317^\textrm{T}\rightarrow$(11,\ 13,\ 17,\ 2431,\ 347633,\ 454597,\ 537251,\ 49711519,\ 65007371,\ 76826893,\ 85009639,\ 100465937,\ 118732471,\ 7108747217,\ 9296054053,\ 10986245699,\ 12156378377,\ 14366628991,\ 15896802493,\ 16978743353,\ 18787130219,\ 22202972077,\ 26239876091,\ 1016550852031,\ 1329335729579,\ 1571033134957,\ 1738362107911,\ 2054427945713,\ 2273242756499,\ 2427960299479)~.
\end{enumerate}

One can write a program similar to the program $\emph{VarMS2}$, \ref{Program VarMS2}, that genarete multiplicative sequence with f{}irst three terms.

\subsection{Multiplicative Sequence of First $k$ Terms}

General def{}inition: if $m_1$, $m_2$, \ldots, $m_k$ are the f{}irst $k$ terms of the sequence, then $m_j$, for $j\ge k+1$, is the smallest number equal to the product of $k$ previous distinct terms.

\section{Generalized Arithmetic Progression}

A classic arithmetic progression is def{}ined by: $a_1\in\Real$ the f{}irst term of progression, $r\neq0$, $r\in\Real$ ratio of progression (if $r>0$ then we have an ascending progression, if $r<0$ then we have a descending ratio), $a_{k+1}=a_k+r=a_1+k\cdot r$ for any $k\in\Ns$, the term of rank $k+1$. Obviously, we can consider ascending
progressions of integers (or natural numbers) or descending progressions of integers (or natural numbers), where $a_1\in\mathbb{Z}$, (or $a_1\in\Na$) $r\in\Zs$ (or $r\in\Na$) and $a_{k+1}=a_k+r=a_1+k\cdot r$.

We consider the following generalization of arithmetic progressions. Let $a_1\in\Real$ the f{}irst term of the arithmetic progression and $\set{r_k}\subset\Real$ series of positive numbers if ascending progressions and a series of negative numbers when descending progressions. This series we will call the series ratio. The term $a_{k+1}$ is def{}ined by formula:
\[
 a_{k+1}=a_k+r_k=a_1+\sum_{j=1}^kr_j~,
\]
for any $k\ge1$.
\begin{obs}
  It is obvious that a classical arithmetic progression is a particular case of the generalized arithmetic progression.
\end{obs}

Examples (we preferred to give examples of progression of integers for easier reading of text):
\begin{enumerate}
  \item Let $L:=35$\ \ $k:=1..L$\ \ $r_k:=3$, then
      \begin{itemize}
        \item[] $r^\textrm{T}\rightarrow$(3\ \ 3\ \ 3\ \ 3\ \ 3\ \ 3\ \ 3\ \ 3\ \ 3\ \ 3\ \ 3\ \ 3\ \ 3\ \ 3\ \ 3\ \ 3\ \ 3\ \ 3\ \ 3\ \ 3\ \ 3\ \ 3\ \ 3\ \ 3\ \ 3\ \ 3\ \ 3\ \ 3\ \ 3\ \ 3\ \ 3\ \ 3\ \ 3\ \ 3\ \ 3),
      \end{itemize}
      if $a_1:=2$ and $a_{k+1}:=a_k+r_k$, then result
      \begin{itemize}
        \item[] $a^\textrm{T}\rightarrow$(2\ \ 5\ \ 8\ \ 11\ \ 14\ \ 17\ \ 20\ \ 23\ \ 26\ \ 29\ \ 32\ \ 35\ \ 38\ \ 41\ \ 44\ \ 47\ \ 50\ \ 53\ \ 56\ \ 59\ \ 62\ \ 65\ \ 68\ \ 71\ \ 74\ \ 77\ \ 80\ \ 83\ \ 86\ \ 89\ \ 92\ \ 95\ \ 98\ \ 101\ \ 104\ \ 107)
      \end{itemize}
      is a classic ascending arithmetic progression with $a_1=2$ and $r=3$;
  \item Let $L:=35$\ \ $k:=1..L$\ \ $r_k:=k$, then
      \begin{itemize}
        \item[] $r^\textrm{T}\rightarrow$(1\ \ 2\ \ 3\ \ 4\ \ 5\ \ 6\ \ 7\ \ 8\ \ 9\ \ 10\ \ 11\ \ 12\ \ 13\ \ 14\ \ 15\ \ 16\ \ 17\ \ 18\ \ 19\ \ 20\ \ 21\ \ 22\ \ 23\ \ 24\ \ 25\ \ 26\ \ 27\ \ 28\ \ 29\ \ 30\ \ 31\ \ 32\ \ 33\ \ 34\ \ 35),
      \end{itemize}
      if $a_1:=2$ and $a_{k+1}:=a_k+r_k$, then result
      \begin{itemize}
        \item[] $a^\textrm{T}\rightarrow$(2\ \ 3\ \ 5\ \ 8\ \ 12\ \ 17\ \ 23\ \ 30\ \ 38\ \ 47\ \ 57\ \ 68\ \ 80\ \ 93\ \ 107\ \ 122\ \ 138\ \ 155\ \ 173\ \ 192\ \ 212\ \ 233\ \ 255\ \ 278\ \ 302\ \ 327\ \ 353\ \ 380\ \ 408\ \ 437\ \ 467\ \ 498\ \ 530\ \ 563\ \ 597\ \ 632);
      \end{itemize} then we have a generalized ascending arithmetic progression but it is not a classical ascending arithmetic progression.
  \item Let $L:=30$\ \ $k:=1..L$\ \ $r_k:=k^2$, then
      \begin{itemize}
        \item[] $r^\textrm{T}\rightarrow$(1\ \ 4\ \ 9\ \ 16\ \ 25\ \ 36\ \ 49\ \ 64\ \ 81\ \ 100\ \ 121\ \ 144\ \ 169\ \ 196\ \ 225\ \ 256\ \ 289\ \ 324\ \ 361\ \ 400\ \ 441\ \ 484\ \ 529\ \ 576\ \ 625\ \ 676\ \ 729\ \ 784\ \ 841\ \ 900),
      \end{itemize}
      if $a_1:=2$ and $a_{k+1}:=a_k+r_k$, then result
      \begin{itemize}
        \item[] $a^\textrm{T}\rightarrow$(2\ \ 3\ \ 7\ \ 16\ \ 32\ \ 57\ \ 93\ \ 142\ \ 206\ \ 287\ \ 387\ \ 508\ \ 652\ \ 821\ \ 1017\ \ 1242\ \ 1498\ \ 1787\ \ 2111\ \ 2472\ \ 2872\ \ 3313\ \ 3797\ \ 4326\ \ 4902\ \ 5527\ \ 6203\ \ 6932\ \ 7716\ \ 8557\ \ 9457);
      \end{itemize}
      then we have a generalized ascending arithmetic progression but it is not a classical arithmetic progression.
  \item Let $L:=25$\ \ $k:=1..L$\ \ $r_k:=k^3$, then
      \begin{itemize}
        \item[] $r^\textrm{T}\rightarrow$(1\ \ 8\ \ 27\ \ 64\ \ 125\ \ 216\ \ 343\ \ 512\ \ 729\ \ 1000\ \ 1331\ \ 1728\ \ 2197\ \ 2744\ \ 3375\ \ 4096\ \ 4913\ \ 5832\ \ 6859\ \ 8000\ \ 9261\ \ 10648\ \ 12167\ \ 13824\ \ 15625),
      \end{itemize}
      if $a_1:=2$ and $a_{k+1}:=a_k+r_k$, then result
      \begin{itemize}
        \item[] $a^\textrm{T}\rightarrow$(2\ \ 3\ \ 11\ \ 38\ \ 102\ \ 227\ \ 443\ \ 786\ \ 1298\ \ 2027\ \ 3027\ \ 4358\ \ 6086\ \ 8283\ \ 11027\ \ 14402\ \ 18498\ \ 23411\ \ 29243\ \ 36102\ \ 44102\ \ 53363\ \ 64011\ \ 76178\ \ 90002\ \ 105627);
      \end{itemize}
      then we have a generalized ascending arithmetic progression but it is not a classical arithmetic progression.
  \item Let $L:=38$\ \ $k:=1..L$\ \ $r_k:=1+\mod(k-1,6)$, then
      \begin{itemize}
        \item[] $r^\textrm{T}\rightarrow$(1\ \ 2\ \ 3\ \ 4\ \ 5\ \ 6\ \ 1\ \ 2\ \ 3\ \ 4\ \ 5\ \ 6\ \ 1\ \ 2\ \ 3\ \ 4\ \ 5\ \ 6\ \ 1\ \ 2\ \ 3\ \ 4\ \ 5\ \ 6\ \ 1\ \ 2\ \ 3\ \ 4\ \ 5\ \ 6\ \ 1\ \ 2\ \ 3\ \ 4\ \ 5\ \ 6\ \ 1\ \ 2),
      \end{itemize}
      if $a_1:=2$ and $a_{k+1}:=a_k+r_k$, then result
      \begin{itemize}
        \item[] $a^\textrm{T}\rightarrow$(2\ \ 3\ \ 5\ \ 8\ \ 12\ \ 17\ \ 23\ \ 24\ \ 26\ \ 29\ \ 33\ \ 38\ \ 44\ \ 45\ \ 47\ \ 50\ \ 54\ \ 59\ \ 65\ \ 66\ \ 68\ \ 71\ \ 75\ \ 80\ \ 86\ \ 87\ \ 89\ \ 92\ \ 96\ \ 101\ \ 107\ \ 108\ \ 110\ \ 113\ \ 117\ \ 122\ \ 128\ \ 129\ \ 131);
      \end{itemize}
      then we have a generalized ascending arithmetic progression but it is not a classical arithmetic progression.
  \item If
      \begin{itemize}
        \item[] $r:=$(2\ \ 3\ \ 5\ \ 7\ \ 11\ \ 13\ \ 17\ \ 19\ \ 23\ \ 29\ \ 31\ \ 37\ \ 41\ \ 43\ \ 47\ \ 53\ \ 59\ \ 61\ \ 67\ \ 71\ \ 73\ \ 79\ \ 83\ \ 89\ \ 97\ \ 101\ \ 103\ \ 107\ \ 109\ \ 113\ \ 127\ \ 131\ \ 137\ \ 139\ \ 149)$^\textrm{T}$
      \end{itemize}
      i.e. sequence of rations is sequence of prime numbers, and $a_1:=2$\ \ $L:=35$\ \ $k:=1..L$\ \ $a_{k+1}:=a_k+r_k$, then result
      \begin{itemize}
        \item[] $a^\textrm{T}\rightarrow$(2\ \ 4\ \ 7\ \ 12\ \ 19\ \ 30\ \ 43\ \ 60\ \ 79\ \ 102\ \ 131\ \ 162\ \ 199\ \ 240\ \ 283\ \ 330\ \ 383\ \ 442\ \ 503\ \ 570\ \ 641\ \ 714\ \ 793\ \ 876\ \ 965\ \ 1062\ \ 1163\ \ 1266\ \ 1373\ \ 1482\ \ 1595\ \ 1722\ \ 1853\ \ 1990\ \ 2129\ \ 2278);
      \end{itemize}
      then we have a generalized ascending arithmetic progression but it is not a classical arithmetic progression.
   \item Let $L:=35$\ \ $k:=1..L$\ \ $r_k:=-3$, then
      \begin{itemize}
        \item[] $r^\textrm{T}\rightarrow$(--3\ \ --3\ \ --3\ \ --3\ \ --3\ \ --3\ \ --3\ \ --3\ \ --3\ \ --3\ \ --3\ \ --3\ \ --3\ \ --3\ \ --3\ \ --3\ \ --3\ \ --3\ \ --3\ \ --3\ \ --3\ \ --3\ \ --3\ \ --3\ \ --3\ \ --3\ \ --3\ \ --3\ \ --3\ \ --3\ \ --3\ \ --3\ \ --3\ \ --3\ \ --3),
      \end{itemize}
      if $a_1:=150$ and $a_{k+1}:=a_k+r_k$, then result
      \begin{itemize}
        \item[] $a^\textrm{T}\rightarrow$(150\ \ 147\ \ 144\ \ 141\ \ 138\ \ 135\ \ 132\ \ 129\ \ 126\ \ 123\ \ 120\ \ 117\ \ 114\ \ 111\ \ 108\ \ 105\ \ 102\ \ 99\ \ 96\ \ 93\ \ 90\ \ 87\ \ 84\ \ 81\ \ 78\ \ 75\ \ 72\ \ 69\ \ 66\ \ 63\ \ 60\ \ 57\ \ 54\ \ 51\ \ 48\ \ 45)
      \end{itemize}
      is a classic descending arithmetic progression with $a_1=150$ and $r=-3$;
  \item Let $L:=30$\ \ $k:=1..L$\ \ $r_k:=-k$, then
      \begin{itemize}
        \item[] $r^\textrm{T}\rightarrow$(--1\ \ --2\ \ --3\ \ --4\ \ --5\ \ --6\ \ --7\ \ --8\ \ --9\ \ --10\ \ --11\ \ --12\ \ --13\ \ --14\ \ --15\ \ --16\ \ --17\ \ --18\ \ --19\ \ --20\ \ --21\ \ --22\ \ --23\ \ --24\ \ --25\ \ --26\ \ --27\ \ --28\ \ --29\ \ --30),
      \end{itemize}
      if $a_1:=500$ and $a_{k+1}:=a_k+r_k$, then result
      \begin{itemize}
        \item[] $a^\textrm{T}\rightarrow$(500\ \ 499\ \ 497\ \ 494\ \ 490\ \ 485\ \ 479\ \ 472\ \ 464\ \ 455\ \ 445\ \ 434\ \ 422\ \ 409\ \ 395\ \ 380\ \ 364\ \ 347\ \ 329\ \ 310\ \ 290\ \ 269\ \ 247\ \ 224\ \ 200\ \ 175\ \ 149\ \ 122\ \ 94\ \ 65\ \ 35);
      \end{itemize}
      then we have a generalized descending arithmetic progression but it is not a classical descending arithmetic progression.
  \item Let $L:=25$\ \ $k:=1..L$\ \ $r_k:=-k^2$, then
      \begin{itemize}
        \item[] $r^\textrm{T}\rightarrow$(--1\ \ --4\ \ --9\ \ --16\ \ --25\ \ --36\ \ --49\ \ --64\ \ --81\ \ --100\ \ --121\ \ --144\ \ --169\ \ --196\ \ --225\ \ --256\ \ --289\ \ --324\ \ --361\ \ --400\ \ --441\ \ --484\ \ --529\ \ --576\ \ -625),
      \end{itemize}
      if $a_1:=10000$ and $a_{k+1}:=a_k+r_k$, then result
      \begin{itemize}
        \item[] $a^\textrm{T}\rightarrow$(10000\ \ 9999\ \ 9995\ \ 9986\ \ 9970\ \ 9945\ \ 9909\ \ 9860\ \ 9796\ \ 9715\ \ 9615\ \ 9494\ \ 9350\ \ 9181\ \ 8985\ \ 8760\ \ 8504\ \ 8215\ \ 7891\ \ 7530\ \ 7130\ \ 6689\ \ 6205\ \ 5676\ \ 5100\ \ 4475);
      \end{itemize}
      then we have a generalized descending arithmetic progression but it is not a classical descending arithmetic progression.
  \item Let $L:=20$\ \ $k:=1..L$\ \ $r_k:=-k^3$, then
      \begin{itemize}
        \item[] $r^\textrm{T}\rightarrow$(--1\ \ --8\ \ --27\ \ --64\ \ --125\ \ --216\ \ --343\ \ --512\ \ --729\ \ --1000\ \ --1331\ \ --1728\ \ --2197\ \ --2744\ \ --3375\ \ --4096\ \ --4913\ \ --5832\ \ --6859\ \ --8000),
      \end{itemize}
      if $a_1:=50000$ and $a_{k+1}:=a_k+r_k$, then result
      \begin{itemize}
        \item[] $a^\textrm{T}\rightarrow$(50000\ \ 49999\ \ 49991\ \ 49964\ \ 49900\ \ 49775\ \ 49559\ \ 49216\ \ 48704\ \ 47975\ \ 46975\ \ 45644\ \ 43916\ \ 41719\ \ 38975\ \ 35600\ \ 31504\ \ 26591\ \ 20759\ \ 13900\ \ 5900);
      \end{itemize}
      then we have a generalized descending arithmetic progression but it is not a classical descending arithmetic progression.
  \item Let $L:=30$\ \ $k:=1..L$\ \ $r_k:=-(1+\mod(k-1,6))$, then
      \begin{itemize}
        \item[] $r^\textrm{T}\rightarrow$(--1\ \ --2\ \ --3\ \ --4\ \ --5\ \ --6\ \ --1\ \ --2\ \ --3\ \ --4\ \ --5\ \ --6\ \ --1\ \ --2\ \ --3\ \ --4\ \ --5\ \ --6\ \ --1\ \ --2\ \ --3\ \ --4\ \ --5\ \ --6\ \ --1\ \ --2\ \ --3\ \ --4\ \ --5\ \ --6),
      \end{itemize}
      if $a_1:=200$ and $a_{k+1}:=a_k+r_k$, then result
      \begin{itemize}
        \item[] $a^\textrm{T}\rightarrow$(200\ \ 199\ \ 197\ \ 194\ \ 190\ \ 185\ \ 179\ \ 178\ \ 176\ \ 173\ \ 169\ \ 164\ \ 158\ \ 157\ \ 155\ \ 152\ \ 148\ \ 143\ \ 137\ \ 136\ \ 134\ \ 131\ \ 127\ \ 122\ \ 116\ \ 115\ \ 113\ \ 110\ \ 106\ \ 101\ \ 95).
      \end{itemize}
      then we have a generalized descending arithmetic progression but it is not a classical descending arithmetic progression.
\end{enumerate}

\section{Non--Arithmetic Progression}

\cite{Smarandache2006}\index{Smarandache F.} def{}ines the series of numbers which are not arithmetic progression as the series of numbers that we have for all third term the relationship $a_{k+2}\neq a_{k+1}+r$, where $r=a_{k+1}-a_k$. This statement is equivalent to
that all rations third term of the sequence is dif{}ferent from the other two previous terms.

We suggest new def{}initions for non-arithmetic progression. If we have a series of real numbers $\set{a_k}$, $k=1,2,\ldots$ we say that the series $r_k=a_{k+1}-a_k$, $k=1,2,\ldots$ is the series of ratios\rq{} series $\set{a_k}$.

\begin{defn}\label{Definitia Non-progression generalization}
  The real series $\set{a_k}$, $k=1,2,\ldots$is a non--arithmetic progression generalization if the ratios\rq{} series is a series of numbers who do not have the same sign.
\end{defn}

\begin{defn}\label{Definitia Non-progression classical}
  The real series $\set{a_k}$, $k=1,2,\ldots$ is a non--arithmetic progression classical if the ratios\rq{} series is a series of non--constant numbers.
\end{defn}

\begin{obs}
  Def{}inition \ref{Definitia Non-progression classical} includes the \cite{Smarandache2006} def{}inition that we have non--arithmetic progression if all the third term of the series does not verify the equality $a_{k+2}=a_{k+1}+r=2a_{k+1}-a_k$ for any $k=1,2,\ldots$~.

  It is obvious that any classical arithmetic progression is also an arithmetic progression generalization. Therefore, any series of non--arithmetic progression generalization is also a classical non--arithmetic progression..
\end{obs}

\begin{prog}\label{Program PVap} test program if is arithmetic progression. We consider the following assignments of texts.
  \begin{itemize}
    \item[] $t_1:=$"\emph{Classical increasing arithmetic progression}";
    \item[] $t_2:=$"\emph{Classical decreasing arithmetic progression}";
    \item[] $t_3:=$"\emph{Generalized increasing arithmetic progression but not classical}";
    \item[] $t_4:=$"\emph{Generalized decreasing arithmetic progression but not classical}";
    \item[] $t_5:=$"\emph{Non--generalized arithmetic progression}";
  \end{itemize}
  \begin{tabbing}
    $\emph{PVap}(a):=$\=\vline\ $f$\=$\emph{or}\ \ k\in1..\emph{last}(a)-1$\\
    \>\vline\ \>\vline\ $r_k\leftarrow a_{k+1}-a_k$\\
    \>\vline\ $f$\=$\emph{or}\ \ k\in2..\emph{last}(r)$\\
    \>\vline\ \>\vline\ $\emph{npac}\leftarrow npac+1\ \ \emph{if}\ \ r_k\neq r_1$\\
    \>\vline\ \>\vline\ $\emph{npag}\leftarrow npag+1\ \ \emph{if}\ \ \emph{sign}(r_1\cdot r_k)\textbf{=}-1$\\
    \>\vline\ $\emph{if}\ \ npac\textbf{=}0$\\
    \>\vline\ \>\vline\ $\emph{return}\ t_1\ \emph{if}\ \emph{sign}(r_1)>0$\\
    \>\vline\ \>\vline\ $\emph{return}\ t_2\ \emph{if}\ \emph{sign}(r_1)<0$\\
    \>\vline\ \>\vline\ $\emph{return}\ "\emph{Error.}"\ \emph{if}\ r_1\textbf{=}0$\\
    \>\vline\ $\emph{if}\ \ \emph{npac}\neq0\wedge npag\textbf{=}0$\\
    \>\vline\ \>\vline\ $\emph{return}\ t_3\ \emph{if}\ \emph{sign}(r_1)>0$\\
    \>\vline\ \>\vline\ $\emph{return}\ t_4\ \emph{if}\ \emph{sign}(r_1)<0$\\
    \>\vline\ \>\vline\ $\emph{return}\ "\emph{Error.}"\ \emph{if}\ r_1\textbf{=}0$\\
    \>\vline\ $\emph{return}\  t_5\ \emph{if}\ \ \emph{npag}\neq0$\\
  \end{tabbing}
\end{prog}

Examples (for tracking easier the examples we considered only rows of integers):
\begin{enumerate}
  \item Let $L:=100$\ \ $k:=2..L$\ \ $r1_k:=\mod(k-1,2)$, then
  \begin{itemize}
    \item[] $r1^\textrm{T}\rightarrow$(1\ \ 2\ \ 1\ \ 2\ \ 1\ \ 2\ \ 1\ \ 2\ \ 1\ \ 2\ \ 1\ \ 2\ \ 1\ \ 2\ \ 1\ \ 2\ \ 1\ \ 2\ \ 1\ \ 2\ \ 1\ \ 2\ \ 1\ \ 2\ \ 1\ \ 2\ \ 1\ \ 2\ \ 1\ \ 2\ \ 1\ \ 2\ \ 1\ \ 2\ \ 1\ \ 2\ \ 1\ \ 2\ \ 1\ \ 2\ \ 1\ \ 2\ \ 1\ \ 2\ \ 1\ \ 2\ \ 1\ \ 2\ \ 1\ \ 2\ \ 1\ \ 2\ \ 1\ \ 2\ \ 1\ \ 2\ \ 1\ \ 2\ \ 1\ \ 2\ \ 1\ \ 2\ \ 1\ \ 2\ \ 1\ \ 2\ \ 1\ \ 2\ \ 1\ \ 2\ \ 1\ \ 2\ \ 1\ \ 2\ \ 1\ \ 2\ \ 1\ \ 2\ \ 1\ \ 2\ \ 1\ \ 2\ \ 1\ \ 2\ \ 1\ \ 2\ \ 1\ \ 2\ \ 1\ \ 2\ \ 1\ \ 2\ \ 1\ \ 2\ \ 1\ \ 2\ \ 1\ \ 2\ \ 1\ \ 2)
  \end{itemize}
  if $u1_1:=1$ and $u1_k:=u1_{k-1}+r1_{k-1}$, then result
  \begin{itemize}
    \item[] $u1^\textrm{T}\rightarrow$(1\ \ 2\ \ 4\ \ 5\ \ 7\ \ 8\ \ 10\ \ 11\ \ 13\ \ 14\ \ 16\ \ 17\ \ 19\ \ 20\ \ 22\ \ 23\ \ 25\ \ 26\ \ 28\ \ 29\ \ 31\ \ 32\ \ 34\ \ 35\ \ 37\ \ 38\ \ 40\ \ 41\ \ 43\ \ 44\ \ 46\ \ 47\ \ 49\ \ 50\ \ 52\ \ 53\ \ 55\ \ 56\ \ 58\ \ 59\ \ 61\ \ 62\ \ 64\ \ 65\ \ 67\ \ 68\ \ 70\ \ 71\ \ 73\ \ 74\ \ 76\ \ 77\ \ 79\ \ 80\ \ 82\ \ 83\ \ 85\ \ 86\ \ 88\ \ 89\ \ 91\ \ 92\ \ 94\ \ 95\ \ 97\ \ 98\ \ 100\ \ 101\ \ 103\ \ 104\ \ 106\ \ 107\ \ 109\ \ 110\ \ 112\ \ 113\ \ 115\ \ 116\ \ 118\ \ 119\ \ 121\ \ 122\ \ 124\ \ 125\ \ 127\ \ 128\ \ 130\ \ 131\ \ 133\ \ 134\ \ 136\ \ 137\ \ 139\ \ 140\ \ 142\ \ 143\ \ 145\ \ 146\ \ 148\ \ 149)
  \end{itemize}
  and $\emph{PVap}(u1)=$"\emph{Generalized increasing arithmetic progression but not classical}".
  \item Let $L:=100$\ \ $k:=2..L$\ \ $r2_k:=\mod(k-1,3)$, then
  \begin{itemize}
    \item[] $r2^\textrm{T}\rightarrow$(1\ \ 2\ \ 3\ \ 1\ \ 2\ \ 3\ \ 1\ \ 2\ \ 3\ \ 1\ \ 2\ \ 3\ \ 1\ \ 2\ \ 3\ \ 1\ \ 2\ \ 3\ \ 1\ \ 2\ \ 3\ \ 1\ \ 2\ \ 3\ \ 1\ \ 2\ \ 3\ \ 1\ \ 2\ \ 3\ \ 1\ \ 2\ \ 3\ \ 1\ \ 2\ \ 3\ \ 1\ \ 2\ \ 3\ \ 1\ \ 2\ \ 3\ \ 1\ \ 2\ \ 3\ \ 1\ \ 2\ \ 3\ \ 1\ \ 2\ \ 3\ \ 1\ \ 2\ \ 3\ \ 1\ \ 2\ \ 3\ \ 1\ \ 2\ \ 3\ \ 1\ \ 2\ \ 3\ \ 1\ \ 2\ \ 3\ \ 1\ \ 2\ \ 3\ \ 1\ \ 2\ \ 3\ \ 1\ \ 2\ \ 3\ \ 1\ \ 2\ \ 3\ \ 1\ \ 2\ \ 3\ \ 1\ \ 2\ \ 3\ \ 1\ \ 2\ \ 3\ \ 1\ \ 2\ \ 3\ \ 1\ \ 2\ \ 3\ \ 1\ \ 2\ \ 3\ \ 1\ \ 2\ \ 3\ \ 1)
  \end{itemize}
  if $u2_1:=1$ and $u2_k:=u2_{k-1}+r2_{k-1}$, then result
  \begin{itemize}
    \item[] $u2^\textrm{T}\rightarrow$(1\ \ 2\ \ 4\ \ 7\ \ 8\ \ 10\ \ 13\ \ 14\ \ 16\ \ 19\ \ 20\ \ 22\ \ 25\ \ 26\ \ 28\ \ 31\ \ 32\ \ 34\ \ 37\ \ 38\ \ 40\ \ 43\ \ 44\ \ 46\ \ 49\ \ 50\ \ 52\ \ 55\ \ 56\ \ 58\ \ 61\ \ 62\ \ 64\ \ 67\ \ 68\ \ 70\ \ 73\ \ 74\ \ 76\ \ 79\ \ 80\ \ 82\ \ 85\ \ 86\ \ 88\ \ 91\ \ 92\ \ 94\ \ 97\ \ 98\ \ 100\ \ 103\ \ 104\ \ 106\ \ 109\ \ 110\ \ 112\ \ 115\ \ 116\ \ 118\ \ 121\ \ 122\ \ 124\ \ 127\ \ 128\ \ 130\ \ 133\ \ 134\ \ 136\ \ 139\ \ 140\ \ 142\ \ 145\ \ 146\ \ 148\ \ 151\ \ 152\ \ 154\ \ 157\ \ 158\ \ 160\ \ 163\ \ 164\ \ 166\ \ 169\ \ 170\ \ 172\ \ 175\ \ 176\ \ 178\ \ 181\ \ 182\ \ 184\ \ 187\ \ 188\ \ 190\ \ 193\ \ 194\ \ 196\ \ 199)
  \end{itemize}
  and $\emph{PVap}(u2)=$"\emph{Generalized increasing arithmetic progression but not classical}".
  \item Let $L:=100$\ \ $k:=2..L$\ \ $r3_k:=\mod(k-1,4)$, then
  \begin{itemize}
    \item[] $r3^\textrm{T}\rightarrow$(1\ \ 2\ \ 3\ \ 4\ \ 1\ \ 2\ \ 3\ \ 4\ \ 1\ \ 2\ \ 3\ \ 4\ \ 1\ \ 2\ \ 3\ \ 4\ \ 1\ \ 2\ \ 3\ \ 4\ \ 1\ \ 2\ \ 3\ \ 4\ \ 1\ \ 2\ \ 3\ \ 4\ \ 1\ \ 2\ \ 3\ \ 4\ \ 1\ \ 2\ \ 3\ \ 4\ \ 1\ \ 2\ \ 3\ \ 4\ \ 1\ \ 2\ \ 3\ \ 4\ \ 1\ \ 2\ \ 3\ \ 4\ \ 1\ \ 2\ \ 3\ \ 4\ \ 1\ \ 2\ \ 3\ \ 4\ \ 1\ \ 2\ \ 3\ \ 4\ \ 1\ \ 2\ \ 3\ \ 4\ \ 1\ \ 2\ \ 3\ \ 4\ \ 1\ \ 2\ \ 3\ \ 4\ \ 1\ \ 2\ \ 3\ \ 4\ \ 1\ \ 2\ \ 3\ \ 4\ \ 1\ \ 2\ \ 3\ \ 4\ \ 1\ \ 2\ \ 3\ \ 4\ \ 1\ \ 2\ \ 3\ \ 4\ \ 1\ \ 2\ \ 3\ \ 4\ \ 1\ \ 2\ \ 3\ \ 4)
  \end{itemize}
  if $u3_1:=1$ and $u3_k:=u3_{k-1}+r3_{k-1}$, then result
  \begin{itemize}
    \item[] $u4^\textrm{T}\rightarrow$(1\ \ 2\ \ 4\ \ 7\ \ 11\ \ 12\ \ 14\ \ 17\ \ 21\ \ 22\ \ 24\ \ 27\ \ 31\ \ 32\ \ 34\ \ 37\ \ 41\ \ 42\ \ 44\ \ 47\ \ 51\ \ 52\ \ 54\ \ 57\ \ 61\ \ 62\ \ 64\ \ 67\ \ 71\ \ 72\ \ 74\ \ 77\ \ 81\ \ 82\ \ 84\ \ 87\ \ 91\ \ 92\ \ 94\ \ 97\ \ 101\ \ 102\ \ 104\ \ 107\ \ 111\ \ 112\ \ 114\ \ 117\ \ 121\ \ 122\ \ 124\ \ 127\ \ 131\ \ 132\ \ 134\ \ 137\ \ 141\ \ 142\ \ 144\ \ 147\ \ 151\ \ 152\ \ 154\ \ 157\ \ 161\ \ 162\ \ 164\ \ 167\ \ 171\ \ 172\ \ 174\ \ 177\ \ 181\ \ 182\ \ 184\ \ 187\ \ 191\ \ 192\ \ 194\ \ 197\ \ 201\ \ 202\ \ 204\ \ 207\ \ 211\ \ 212\ \ 214\ \ 217\ \ 221\ \ 222\ \ 224\ \ 227\ \ 231\ \ 232\ \ 234\ \ 237\ \ 241\ \ 242\ \ 244\ \ 247)
  \end{itemize}
  and $\emph{PVap}(u3)=$"\emph{Generalized increasing arithmetic progression but not classical}".
  \item Let $L:=100$\ \ $k:=2..L$\ \ $r4_k:=\mod(k-1,5)$, then
  \begin{itemize}
    \item[] $r4^\textrm{T}\rightarrow$(1\ \ 2\ \ 3\ \ 4\ \ 5\ \ 1\ \ 2\ \ 3\ \ 4\ \ 5\ \ 1\ \ 2\ \ 3\ \ 4\ \ 5\ \ 1\ \ 2\ \ 3\ \ 4\ \ 5\ \ 1\ \ 2\ \ 3\ \ 4\ \ 5\ \ 1\ \ 2\ \ 3\ \ 4\ \ 5\ \ 1\ \ 2\ \ 3\ \ 4\ \ 5\ \ 1\ \ 2\ \ 3\ \ 4\ \ 5\ \ 1\ \ 2\ \ 3\ \ 4\ \ 5\ \ 1\ \ 2\ \ 3\ \ 4\ \ 5\ \ 1\ \ 2\ \ 3\ \ 4\ \ 5\ \ 1\ \ 2\ \ 3\ \ 4\ \ 5\ \ 1\ \ 2\ \ 3\ \ 4\ \ 5\ \ 1\ \ 2\ \ 3\ \ 4\ \ 5\ \ 1\ \ 2\ \ 3\ \ 4\ \ 5\ \ 1\ \ 2\ \ 3\ \ 4\ \ 5\ \ 1\ \ 2\ \ 3\ \ 4\ \ 5\ \ 1\ \ 2\ \ 3\ \ 4\ \ 5\ \ 1\ \ 2\ \ 3\ \ 4\ \ 5\ \ 1\ \ 2\ \ 3\ \ 4\ \ 5)
  \end{itemize}
  if $u4_1:=1$ and $u4_k:=u4_{k-1}+r4_{k-1}$, then result
  \begin{itemize}
    \item[] $u4^\textrm{T}\rightarrow$(1\ \ 2\ \ 4\ \ 7\ \ 11\ \ 16\ \ 17\ \ 19\ \ 22\ \ 26\ \ 31\ \ 32\ \ 34\ \ 37\ \ 41\ \ 46\ \ 47\ \ 49\ \ 52\ \ 56\ \ 61\ \ 62\ \ 64\ \ 67\ \ 71\ \ 76\ \ 77\ \ 79\ \ 82\ \ 86\ \ 91\ \ 92\ \ 94\ \ 97\ \ 101\ \ 106\ \ 107\ \ 109\ \ 112\ \ 116\ \ 121\ \ 122\ \ 124\ \ 127\ \ 131\ \ 136\ \ 137\ \ 139\ \ 142\ \ 146\ \ 151\ \ 152\ \ 154\ \ 157\ \ 161\ \ 166\ \ 167\ \ 169\ \ 172\ \ 176\ \ 181\ \ 182\ \ 184\ \ 187\ \ 191\ \ 196\ \ 197\ \ 199\ \ 202\ \ 206\ \ 211\ \ 212\ \ 214\ \ 217\ \ 221\ \ 226\ \ 227\ \ 229\ \ 232\ \ 236\ \ 241\ \ 242\ \ 244\ \ 247\ \ 251\ \ 256\ \ 257\ \ 259\ \ 262\ \ 266\ \ 271\ \ 272\ \ 274\ \ 277\ \ 281\ \ 286\ \ 287\ \ 289\ \ 292\ \ 296)
  \end{itemize}
  and $\emph{PVap}(u4)=$"\emph{Generalized increasing arithmetic progression but not classical}".
\end{enumerate}

Examples non-arithmetic progression generalization:
\begin{itemize}
  \item Let $L:=100$\ \ $k:=2..L$\ \ $r5_k:=floor(-1+rnd(10))$, then
  \begin{itemize}
    \item[] $r5^\textrm{T}\rightarrow$(1\ \ 2\ \ 0\ \ --1\ \ 6\ \ 5\ \ 5\ \ --1\ \ --1\ \ 3\ \ 6\ \ 5\ \ 2\ \ --1\ \ 0\ \ 2\ \ 1\ \ 3\ \ 3\ \ 5\ \ 3\ \ 7\ \ 5\ \ --1\ \ 7\ \ 2\ \ --1\ \ 2\ \ --1\ \ 7\ \ 3\ \ 5\ \ 7\ \ 4\ \ 6\ \ 4\ \ 4\ \ 5\ \ --1\ \ 2\ \ 6\ \ 7\ \ 6\ \ 7\ \ 4\ \ 6\ \ 0\ \ 6\ \ 7\ \ 5\ \ 5\ \ 1\ \ 6\ \ 6\ \ --1\ \ 7\ \ 2\ \ 3\ \ 0\ \ 8\ \ 5\ \ 0\ \ 1\ \ 7\ \ 2\ \ 1\ \ 1\ \ 8\ \ --1\ \ 0\ \ --1\ \ 5\ \ 3\ \ --1\ \ 5\ \ 5\ \ 4\ \ --1\ \ 1\ \ 7\ \ --1\ \ 6\ \ 7\ \ 1\ \ 5\ \ 7\ \ 7\ \ 8\ \ 3\ \ 2\ \ 3\ \ --1\ \ 6\ \ --1\ \ 5\ \ 0\ \ --1\ \ --1\ \ 5\ \ 8)
  \end{itemize}
  if $u5_1:=10$ and $u5_k:=u5_{k-1}+r5_{k-1}$, then result
  \begin{itemize}
    \item[]$u5^\textrm{T}\rightarrow$(10\ \ 11\ \ 13\ \ 13\ \ 12\ \ 18\ \ 23\ \ 28\ \ 27\ \ 26\ \ 29\ \ 35\ \ 40\ \ 42\ \ 41\ \ 41\ \ 43\ \ 44\ \ 47\ \ 50\ \ 55\ \ 58\ \ 65\ \ 70\ \ 69\ \ 76\ \ 78\ \ 77\ \ 79\ \ 78\ \ 85\ \ 88\ \ 93\ \ 100\ \ 104\ \ 110\ \ 114\ \ 118\ \ 123\ \ 122\ \ 124\ \ 130\ \ 137\ \ 143\ \ 150\ \ 154\ \ 160\ \ 160\ \ 166\ \ 173\ \ 178\ \ 183\ \ 184\ \ 190\ \ 196\ \ 195\ \ 202\ \ 204\ \ 207\ \ 207\ \ 215\ \ 220\ \ 220\ \ 221\ \ 228\ \ 230\ \ 231\ \ 232\ \ 240\ \ 239\ \ 239\ \ 238\ \ 243\ \ 246\ \ 245\ \ 250\ \ 255\ \ 259\ \ 258\ \ 259\ \ 266\ \ 265\ \ 271\ \ 278\ \ 279\ \ 284\ \ 291\ \ 298\ \ 306\ \ 309\ \ 311\ \ 314\ \ 313\ \ 319\ \ 318\ \ 323\ \ 323\ \ 322\ \ 321\ \ 326)
  \end{itemize}
  and $\emph{PVap}(u5)=$"\emph{Non--generalized arithmetic progression}".
  \item Let $L:=100$\ \ $k:=2..L$\ \ $r6_k:=floor(2\sin(k)^2+3\cos(k)^3+5\sin(k)^2\cos(k)^3)$, then
  \begin{itemize}
    \item[] $r6^\textrm{T}\rightarrow$(1\ \ 1\ \ --3\ \ --1\ \ 2\ \ 3\ \ 3\ \ 1\ \ --3\ \ --3\ \ 1\ \ 3\ \ 3\ \ 1\ \ --2\ \ --3\ \ 1\ \ 2\ \ 3\ \ 2\ \ 0\ \ --3\ \ 0\ \ 2\ \ 3\ \ 2\ \ 1\ \ --3\ \ --2\ \ 1\ \ 3\ \ 3\ \ 1\ \ --3\ \ --3\ \ 1\ \ 3\ \ 3\ \ 2\ \ --1\ \ --3\ \ 1\ \ 2\ \ 3\ \ 2\ \ 1\ \ --3\ \ --1\ \ 2\ \ 3\ \ 3\ \ 1\ \ --3\ \ --2\ \ 1\ \ 3\ \ 3\ \ 1\ \ --2\ \ --3\ \ 1\ \ 2\ \ 3\ \ 2\ \ 0\ \ --3\ \ 0\ \ 2\ \ 3\ \ 2\ \ 1\ \ --3\ \ --2\ \ 1\ \ 3\ \ 3\ \ 1\ \ --3\ \ --3\ \ 1\ \ 3\ \ 3\ \ 1\ \ --1\ \ --3\ \ 1\ \ 2\ \ 3\ \ 2\ \ 0\ \ --3\ \ --1\ \ 2\ \ 3\ \ 3\ \ 1\ \ --3\ \ --2\ \ 1\ \ 3)
  \end{itemize}
  if $u6_1:=100$ and $u6_k:=u6_{k-1}+r6_{k-1}$, then result
  \begin{itemize}
    \item[]$u6^\textrm{T}\rightarrow$(100\ \ 101\ \ 102\ \ 99\ \ 98\ \ 100\ \ 103\ \ 106\ \ 107\ \ 104\ \ 101\ \ 102\ \ 105\ \ 108\ \ 109\ \ 107\ \ 104\ \ 105\ \ 107\ \ 110\ \ 112\ \ 112\ \ 109\ \ 109\ \ 111\ \ 114\ \ 116\ \ 117\ \ 114\ \ 112\ \ 113\ \ 116\ \ 119\ \ 120\ \ 117\ \ 114\ \ 115\ \ 118\ \ 121\ \ 123\ \ 122\ \ 119\ \ 120\ \ 122\ \ 125\ \ 127\ \ 128\ \ 125\ \ 124\ \ 126\ \ 129\ \ 132\ \ 133\ \ 130\ \ 128\ \ 129\ \ 132\ \ 135\ \ 136\ \ 134\ \ 131\ \ 132\ \ 134\ \ 137\ \ 139\ \ 139\ \ 136\ \ 136\ \ 138\ \ 141\ \ 143\ \ 144\ \ 141\ \ 139\ \ 140\ \ 143\ \ 146\ \ 147\ \ 144\ \ 141\ \ 142\ \ 145\ \ 148\ \ 149\ \ 148\ \ 145\ \ 146\ \ 148\ \ 151\ \ 153\ \ 153\ \ 150\ \ 149\ \ 151\ \ 154\ \ 157\ \ 158\ \ 155\ \ 153\ \ 154)
  \end{itemize}
  and $\emph{PV(u6)}=$"\emph{Non--generalized arithmetic progression}"
\end{itemize}

\section{Generalized Geometric Progression}

A classical generalized geometric progression is def{}ined by: $g_1\in\Real$, the f{}irst term of progression, $\rho\neq0$, $\rho\in\Real$ progression ratio (if $\rho>1$ then we have an ascending progression, if $0<\rho<1$ then we have a descending progression) and the formula $g_{k+1}=g_k\cdot\rho=g_1\cdot\rho^k$, for any $k\in\Ns$, where $g_{k+1}$ the term of rank $k+1$.

Obviously we can consider ascending progression of integers (or natural numbers) or descending progression of integers (or natural numbers), where $g_1\in\mathbb{Z}$, (or $g_1\in\Na$), $\rho\in\Zs$ (or $\rho\in\Na$) and $g_{k+1}=g_k\cdot\rho=g_1\cdot\rho^k$.

Consider the following generalization of geometric progressions. Let $g_1\in\Na$ the f{}irst term of the geometric progression and $\rho_k$ a series of positive real supraunitary numbers in ascending progressions or a series of real subunitary numbers in descending progression. We call the series $\set{\rho_k}$ il the series of generalized geometric progression ratios. The term $g_{k+1}$ is def{}ined by formula
\[
 g_{k+1}=g_k\cdot\rho_k=g_1\cdot\prod_{j=1}^k\rho_j~,
\]
for any $k\in\Ns$.

Examples (we preferred to give examples of progression of integers for easier browsing text):
\begin{enumerate}
  \item Let $L:=20$\ \ $k:=1..L$\ \ $\rho_k:=3$, then
      \begin{itemize}
        \item[] $\rho^\textrm{T}\rightarrow$(3\ \ 3\ \ 3\ \ 3\ \ 3\ \ 3\ \ 3\ \ 3\ \ 3\ \ 3\ \ 3\ \ 3\ \ 3\ \ 3\ \ 3\ \ 3\ \ 3\ \ 3\ \ 3\ \ 3)
      \end{itemize}
      if $g_1:=2$ and $g_{k+1}:=g_k\cdot\rho_k$, then result that
      \begin{itemize}
        \item[] $g^\textrm{T}\rightarrow$(2\ \ 6\ \ 18\ \ 54\ \ 162\ \ 486\ \ 1458\ \ 4374\ \ 13122\ \ 39366\ \ 118098\ \ 354294\ \ 1062882\ \ 3188646\ \ 9565938\ \ 28697814\ \ 86093442\ \ 258280326\ \ 774840978\ \ 2324522934\ \ 6973568802)
      \end{itemize}
      which is a classical ascending geometric progression with $g_1=2$ and $\rho=3$;
  \item Let $L:=15$\ \ $k:=1..L$\ \ $\rho_k:=k$, then
      \begin{itemize}
        \item[] $\rho^\textrm{T}\rightarrow$(2\ \ 3\ \ 4\ \ 5\ \ 6\ \ 7\ \ 8\ \ 9\ \ 10\ \ 11\ \ 12\ \ 13\ \ 14\ \ 15\ \ 16)
      \end{itemize}
      if $g_1:=1$ and $g_{k+1}:=g_k\cdot\rho_k$, then result that
      \begin{itemize}
        \item[] $g^\textrm{T}\rightarrow$(1\ \ 2\ \ 6\ \ 24\ \ 120\ \ 720\ \ 5040\ \ 40320\ \ 362880\ \ 3628800\ \ 39916800\ \ 479001600\ \ 6227020800\ \ 87178291200\ \ 1307674368000\ \ 20922789888000)
      \end{itemize}
      which is factorial sequence that is a generalized ascending geometric progression but is not classical geometric progression;
  \item Let $L:=13$\ \ $k:=1..L$\ \ $\rho_k:=(k+1)^2$, then
      \begin{itemize}
        \item[] $\rho^\textrm{T}\rightarrow$(4\ \ 9\ \ 16\ \ 25\ \ 36\ \ 49\ \ 64\ \ 81\ \ 100\ \ 121\ \ 144\ \ 169\ \ 196)
      \end{itemize}
      if $g_1:=1$ and $g_{k+1}:=g_k\cdot\rho_k$, then we obtain sequence
      \begin{itemize}
        \item[] $g^\textrm{T}\rightarrow$(1\ 4\ 36\ 576\ 14400\ 518400\ 25401600\ 1625702400\ 131681894400\ 13168189440000\ 1593350922240000\ 229442532802560000\ 38775788043632640000\ 7600054456551997440000);
      \end{itemize}
      which is a generalized ascending geometric progression but is not classical geometric progression.
  \item Let $L:=10$\ \ $k:=1..L$\ \ $\rho_k:=(k+1)^3$, then
      \begin{itemize}
        \item[] $\rho^\textrm{T}\rightarrow$(8\ \ 27\ \ 64\ \ 125\ \ 216\ \ 343\ \ 512\ \ 729\ \ 1000\ \ 1331)
      \end{itemize}
      if $g_1:=7$ and $g_{k+1}:=g_k\cdot\rho_k$, then result that
      \begin{itemize}
        \item[] $g^\textrm{T}\rightarrow$(7\ \ 56\ \ 1512\ \ 96768\ \ 12096000\ \ 2612736000\ \ 896168448000\ \ 458838245376000\ \ 334493080879104000\ \ 334493080879104000000\ \ 445210290650087424000000);
      \end{itemize}
      which is a generalized ascending geometric progression but is not classical geometric progression.
  \item Let $L:=15$\ \ $k:=1..L$\ \ $\rho_k:=3+\mod(k-1,6)$, then
      \begin{itemize}
        \item[] $\rho^\textrm{T}\rightarrow$(3\ \ 4\ \ 5\ \ 6\ \ 7\ \ 8\ \ 3\ \ 4\ \ 5\ \ 6\ \ 7\ \ 8\ \ 3\ \ 4\ \ 5)
      \end{itemize}
      if $g_1:=11$ and $g_{k+1}:=g_k\cdot\rho_k$, then we obtain
      \begin{itemize}
        \item[] $g^\textrm{T}\rightarrow$(11\ \ 33\ \ 132\ \ 660\ \ 3960\ \ 27720\ \ 221760\ \ 665280\ \ 2661120\ \ 13305600\ \ 79833600\ \ 558835200\ \ 4470681600\ \ 13412044800\ \ 53648179200\ \ 268240896000);
      \end{itemize}
      which is a generalized ascending geometric progression but is not classical geometric progression.
  \item Let $L:=10$\ \ $k:=1..L$\ \ $\rho_k:=\frac{1}{3}$, then
      \begin{itemize}
        \item[] $\rho^\textrm{T}\rightarrow$\Bigg($\dfrac{1}{3}$\ \ $\dfrac{1}{3}$\ \ $\dfrac{1}{3}$\ \ $\dfrac{1}{3}$\ \ $\dfrac{1}{3}$\ \ $\dfrac{1}{3}$\ \ $\dfrac{1}{3}$\ \ $\dfrac{1}{3}$\ \ $\dfrac{1}{3}$\ \ $\dfrac{1}{3}$\Bigg)
      \end{itemize}
      if $g_1:=3^{10}$ and $g_{k+1}:=g_k\cdot\rho_k$, then result that
      \begin{itemize}
        \item[] $g^\textrm{T}\rightarrow$(59049\ \ 19683\ \ 6561\ \ 2187\ \ 729\ \ 243\ \ 81\ \ 27\ \ 9\ \ 3\ \ 1)
      \end{itemize}
      which is a classic descending geometric progression with $g_1=59049$ and $\rho=\frac{1}{3}$;
  \item Let $L:=10$\ \ $k:=1..L$\ \ $\rho_k:=2^{-k}$, then
      \begin{itemize}
        \item[] $\rho^\textrm{T}\rightarrow$\Bigg($\dfrac{1}{2}$\ \ $\dfrac{1}{4}$\ \ $\dfrac{1}{8}$\ \ $\dfrac{1}{16}$\ \ $\dfrac{1}{32}$\ \ $\dfrac{1}{64}$\ \ $\dfrac{1}{128}$\ \ $\dfrac{1}{256}$\ \ $\dfrac{1}{512}$\ \ $\dfrac{1}{1024}$\Bigg)
      \end{itemize}
      if $g_1:=2^{55}$ and $g_{k+1}:=g_k\cdot\rho_k$, then result that
      \begin{itemize}
        \item[] $g^\textrm{T}\rightarrow$(36028797018963968\ \ 18014398509481984\ \ 4503599627370496\ \ 562949953421312\ \ 35184372088832\ \ 1099511627776\ \ 17179869184\ \ 134217728\ \ 524288\ \ 1024\ \ 1);
      \end{itemize}
      which is a generalized descending geometric progression but is not classical geometric progression.
  \item Let $L:=15$\ \ $k:=1..L$\ \ $\rho_k:=2^{-[1+\mod(k-1,6)]}$, then
      \begin{itemize}
        \item[] $\rho^\textrm{T}\rightarrow$\Bigg($\dfrac{1}{2}$\ \ $\dfrac{1}{4}$\ \ $\dfrac{1}{8}$\ \ $\dfrac{1}{16}$\ \ $\dfrac{1}{32}$\ \ $\dfrac{1}{64}$\ \ $\dfrac{1}{2}$\ \ $\dfrac{1}{4}$\ \ $\dfrac{1}{8}$\ \ $\dfrac{1}{16}$\ \ $\dfrac{1}{32}$\ \ $\dfrac{1}{64}$\ \ $\dfrac{1}{2}$\ \ $\dfrac{1}{4}$\ \ $\dfrac{1}{8}$\Bigg)
      \end{itemize}
      if $g_1:=2^{48}$ and $g_{k+1}:=g_k\cdot\rho_k$, then result that
      \begin{itemize}
        \item[] $g^\textrm{T}\rightarrow$(281474976710656\ \ 140737488355328\ \ 35184372088832\ \ 4398046511104\ \ 274877906944\ \ 8589934592\ \ 134217728\ \ 67108864\ \ 16777216\ \ 2097152\ \ 131072\ \ 4096\ \ 64\ \ 32\ \ 8\ \ 1);
      \end{itemize}
      which is a generalized descending geometric progression but is not classical geometric progression.
\end{enumerate}

\section{Non--Geometric Progression}

If we have a series of real numbers $\set{g_k}$, $k=1,2,\ldots$ we say that the series $\rho_k=\frac{a_{k+1}}{a_k}$, $k=1,2,\ldots$ is the series of ratios\rq{} series $\set{g_k}$.

\begin{defn}\label{Definitia Non-generalized geometric progression}
  The series  $\set{g_k}$, $k=1,2,\ldots$ is a non--generalized geometric progression if ratios\rq{} series $\set{\rho_k}$ is a series of number not supraunitary (or subunitary).
\end{defn}

\begin{defn}\label{Definitia Non-classical geometric progression}
  The series $\set{g_k}$, $k=1,2,\ldots$ is a non--classical geometric progression if ratios\rq{} series $\set{\rho_k}$ is a series of inconstant numbers.
\end{defn}

\begin{obs}
  It is obvious that any classical geometric progression is a generalized geometric progression. Therefore, any series of non-generalized geometric progression is also non--classical geometric progression.
\end{obs}

\begin{prog}\label{Program PVgp} for geometric progression testing. We consider the following texts\rq{} assignments.
  \begin{itemize}
    \item[] $t_1:=$"\emph{Classical increasing geometric progression}";
    \item[] $t_2:=$"\emph{Classical decreasing geometric progression}";
    \item[] $t_3:=$"\emph{Generalized increasing geometric progression but not classical}";
    \item[] $t_4:=$"\emph{Generalized decreasing geometric progression but not classical}";
    \item[] $t_5:=$"\emph{Non--generalized geometric progression}";
  \end{itemize}
  \begin{tabbing}
    $\emph{PVgp}(g):=$\=\vline\ $f$\=$\emph{or}\ \ k\in1..\emph{last(g)}-1$\\
    \>\vline\ \>\vline\ $\rho_k\leftarrow \dfrac{g_{k+1}}{g_k}$\\
    \>\vline\ $f$\=$or\ \ k\in2..\emph{last(q)}$\\
    \>\vline\ \>\vline\ $\emph{npac}\leftarrow \emph{npac}+1\ \emph{if}\ \rho_k\neq \rho_1$\\
    \>\vline\ \>\vline\ $\emph{npag}\leftarrow \emph{npag}+1\ \emph{if}\ \neg[(\rho_1>1\wedge \rho_k>1)\vee(\rho_1<1\wedge \rho_k<1)]$\\
    \>\vline\ $\emph{if}\ \ \emph{npac}\textbf{=}0$\\
    \>\vline\ \>\vline\ $\emph{return}\ t_1\ \ \emph{if}\ \ \rho_1>1$\\
    \>\vline\ \>\vline\ $\emph{return}\ t_2\ \ \emph{if}\ \ 0<\rho_1<1$\\
    \>\vline\ \>\vline\ $\emph{return}\ "\emph{Error.}"\ \ \emph{if}\ \ \rho_1\le0\vee \rho_1\textbf{=}1$\\
    \>\vline\ $\emph{if}\ \ \emph{npac}\neq0\wedge \emph{npag}\textbf{=}0$\\
    \>\vline\ \>\vline\ $\emph{return}\ t_3\ \ \emph{if}\ \ \rho_1>1$\\
    \>\vline\ \>\vline\ $\emph{return}\ t_4\ \ \emph{if}\ \ 0<\rho_1<1$\\
    \>\vline\ \>\vline\ $\emph{return}\ "\emph{Error.}"\ \emph{if}\ \rho_1\le0\vee \rho_1\textbf{=}1$\\
    \>\vline\ $\emph{return}\  t_5\ \emph{if}\ \ \emph{npag}\neq0$\\
  \end{tabbing}
\end{prog}

Examples:
\begin{enumerate}
  \item Let $\rho1_1:=\frac{6}{5}$ $L:=10$\ \ $k:=2..L$\ \ $\rho1_k:=\frac{6}{5}$, then
  \begin{itemize}
    \item[] $\rho1^\textrm{T}\rightarrow$\Big($\dfrac{6}{5}$\ \ $\dfrac{6}{5}$\ \ $\dfrac{6}{5}$\ \ $\dfrac{6}{5}$\ \ $\dfrac{6}{5}$\ \ $\dfrac{6}{5}$\ \ $\dfrac{6}{5}$\ \ $\dfrac{6}{5}$\ \ $\dfrac{6}{5}$\ \ $\dfrac{6}{5}$\Big)
  \end{itemize}
  if $w1_1:=1$ and $w1_k:=w1_{k-1}\cdot \rho1_{k-1}$, then result
  \begin{itemize}
    \item[] $w1^\textrm{T}\rightarrow$\Big(1\ \ $\dfrac{6}{5}$\ \ $\dfrac{36}{25}$\ \ $\dfrac{216}{125}$\ \ $\dfrac{1296}{625}$\ \ $\dfrac{7776}{3125}$\ \ $\dfrac{46656}{15625}$\ \ $\dfrac{279936}{78125}$\ \ $\dfrac{1679616}{390625}$\ \ $\dfrac{10077696}{1953125}$\Big)
  \end{itemize}
  and $\emph{PVgp}(w1)=$"\emph{Classical increasing geometric progression}".
  \item Let $\rho2_1:=\frac{9}{10}$ $L:=10$\ \ $k:=2..L$\ \ $\rho2_k:=\frac{9}{10}$, then
  \begin{itemize}
    \item[] $\rho2^\textrm{T}\rightarrow$\Big($\dfrac{9}{10}$\ \ $\dfrac{9}{10}$\ \ $\dfrac{9}{10}$\ \ $\dfrac{9}{10}$\ \ $\dfrac{9}{10}$\ \ $\dfrac{9}{10}$\ \ $\dfrac{9}{10}$\ \ $\dfrac{9}{10}$\ \ $\dfrac{9}{10}$\ \ $\dfrac{9}{10}$\Big)
  \end{itemize}
  if $w2_1:=10^9$ and $w2_k:=w2_{k-1}\cdot \rho2_{k-1}$, then result
  \begin{itemize}
    \item[] $w2^\textrm{T}\rightarrow$(1000000000\ \ 900000000\ \ 810000000\ \ 729000000\ \ 656100000\ \ 590490000\ \ 531441000\ \ 478296900\ \ 430467210\ \ 387420489)
  \end{itemize}
  and $\emph{PVgp}(w2)=$"\emph{Classical decreasing geometric progression}".
  \item Let $\rho3_1:=2$ $L:=17$\ \ $k:=2..L$\ \ $\rho3_k:=q3_{k-1}+1$, then
  \begin{itemize}
    \item[] $\rho3^\textrm{T}\rightarrow$(2\ \ 3\ \ 4\ \ 5\ \ 6\ \ 7\ \ 8\ \ 9\ \ 10\ \ 11\ \ 12\ \ 13\ \ 14\ \ 15\ \ 16\ \ 17\ \ 18)
  \end{itemize}
  if $w3_1:=1$ and $w3_k:=w3_{k-1}\cdot \rho3_{k-1}$, then result
  \begin{itemize}
    \item[] $w3^\textrm{T}\rightarrow$(1\ \ 2\ \ 6\ \ 24\ \ 120\ \ 720\ \ 5040\ \ 40320\ \ 362880\ \ 3628800\ \ 39916800\ \ 479001600\ \ 6227020800\ \ 87178291200\ \ 1307674368000\ \ 20922789888000\ \ 355687428096000)
  \end{itemize}
  and $\emph{PVgp}(w3)=$"\emph{Generalized increasing geometric progression but not classical}". Note that the series $w3$ is the series of factorials up to $17!$.
  \item Let $\rho4_1:=\dfrac{1}{2}$ $L:=10$\ \ $k:=2..L$\ \ $\rho4_k:=\rho4_{k-1}\cdot\dfrac{1}{2}$, then
  \begin{itemize}
    \item[] $\rho4^\textrm{T}\rightarrow$\Big($\dfrac{1}{2}$\ \ $\dfrac{1}{4}$\ \ $\dfrac{1}{8}$\ \ $\dfrac{1}{16}$\ \ $\dfrac{1}{32}$\ \ $\dfrac{1}{64}$\ \ $\dfrac{1}{128}$\ \ $\dfrac{1}{256}$\ \ $\dfrac{1}{512}$\ \ $\dfrac{1}{1024}$\Big)
  \end{itemize}
  if $w4_1:=1$ and $w4_k:=w4_{k-1}\cdot \rho4_{k-1}$, then result
  \begin{itemize}
    \item[] $w4^\textrm{T}\rightarrow$\Big(1\ \ $\dfrac{1}{2}$\ \ $\dfrac{1}{8}$\ \ $\dfrac{1}{64}$\ \ $\dfrac{1}{1024}$\ \ $\dfrac{1}{32768}$\ \ $\dfrac{1}{2097152}$\ \ $\dfrac{1}{268435456}$\ \ $\dfrac{1}{68719476736}$\ \ $\dfrac{1}{35184372088832}$\Big)
  \end{itemize}
  and $\emph{PVgp}(w4)=$"\emph{Generalized decreasing geometric progression but not classical}".
  \item Let $\rho5_1:=\dfrac{1}{2}$ $L:=10$\ \ $k:=2..L$\ \ $\rho5_k:=\rho4_{k-1}\cdot3^{(-1)^k}$, then
  \begin{itemize}
    \item[] $\rho5^\textrm{T}\rightarrow$\Big($\dfrac{1}{2}$\ \ $\dfrac{3}{2}$\ \ $\dfrac{1}{2}$\ \ $\dfrac{3}{2}$\ \ $\dfrac{1}{2}$\ \ $\dfrac{3}{2}$\ \ $\dfrac{1}{2}$\ \ $\dfrac{3}{2}$\ \ $\dfrac{1}{2}$\ \ $\dfrac{3}{2}$\Big)
  \end{itemize}
  if $w5_1:=1$ and $w5_k:=w5_{k-1}\cdot \rho5_{k-1}$, then result
  \begin{itemize}
    \item[] $w5^\textrm{T}\rightarrow$\Big(1\ \ $\dfrac{1}{2}$\ \ $\dfrac{3}{4}$\ \ $\dfrac{3}{8}$\ \ $\dfrac{9}{16}$\ \ $\dfrac{9}{32}$\ \ $\dfrac{27}{64}$\ \ $\dfrac{27}{128}$\ \ $\dfrac{81}{256}$\ \ $\dfrac{81}{512}$\Big)
  \end{itemize}
  and $\emph{PVgp}(w5)=$"\emph{Non-generalized geometric progression}".
\end{enumerate}

\chapter{Special numbers}

\section{Numeration Bases}

\subsection{Prime Base}

We def{}ined over the set of natural numbers the following inf{}inite base: $p_0=1$, and for $k\in\Ns$ $p_k=prime_k$ is the $k$-th prime number. We proved that every positive integer $a\in\Ns$ may be uniquely written in the prime base as:
\[
 a=\overline{a_m\ldots a_1a_0}_{(pb)}=\sum_{k=0}^ma_kp_k~,
\]
where $a_k=0$ or $1$ for $k=0,1,\ldots,m-1$ and of course $a_m=1$, in the following way:
\begin{itemize}
  \item if $p_m\le a<p_{m+1}$ then $a=p_m+r_1$;
  \item if $p_k\le r_1<p_{k+1}$ then $r_1=p_k+r_2$, $k<m$;
  \item and so on until one obtains a rest $r_j=0$.
\end{itemize}
Therefore, any number may be written as a sum of prime numbers $+e$, where $e=0$ or $1$. Thus we have
\begin{eqnarray*}
  2_{(10)} &=& 1\cdot2+0\cdot1=10_{(pb)}~, \\
  3_{(10)} &=& 1\cdot3+0\cdot2+0\cdot1=100_{(pb)}~, \\
  4_{(10)} &=& 1\cdot3+0\cdot2+1\cdot1=101_{(pb)}~, \\
  5_{(10)} &=& 1\cdot5+0\cdot3+0\cdot2+0\cdot1=1000_{(pb)}~, \\
\end{eqnarray*}
\begin{eqnarray*}
  6_{(10)} &=& 1\cdot5+0\cdot3+0\cdot2+1\cdot1=1001_{(pb)}~, \\
  7_{(10)} &=& 1\cdot7+0\cdot5+0\cdot3+0\cdot2+0\cdot1=10000_{(pb)}~, \\
  8_{(10)} &=& 1\cdot7+0\cdot5+0\cdot3+0\cdot2+1\cdot1=10001_{(pb)}~, \\
  9_{(10)} &=& 1\cdot7+0\cdot5+0\cdot3+1\cdot2+0\cdot1=10010_{(pb)}~, \\
  10_{(10)} &=& 1\cdot7+0\cdot5+1\cdot3+0\cdot2+0\cdot1=10100_{(pb)}~. \\
\end{eqnarray*}

If we use the $\emph{ipp}$ function, given by  \ref{Programipp}, then $a$ is written in the prime base as:
\[
 a=\emph{ipp}(a)+\emph{ipp}(a-\emph{ipp}(a))+\emph{ipp}(a-\emph{ipp}(a)-\emph{ipp}(a-\emph{ipp}(a)))+\ldots~,
\]
or
\[
 a=\emph{ipp}(a)+\emph{ipp}(\emph{ppi}(a))+\emph{ipp}(\emph{ppi}(a)-\emph{ipp}(\emph{ppi}(a)))+\ldots
\]
where the function $\emph{ppi}$ given by \ref{Functia ppi}.
\begin{exem} Let $a=35$, then
 \begin{multline*}
   \emph{ipp}(35)+\emph{ipp}(35-\emph{ipp}(35))+\emph{ipp}(35-\emph{ipp}(35)-\emph{ipp}(35-\emph{ipp}(35)))\\
   =31+3+1=35
 \end{multline*}
or
\[
 \emph{ipp}(35)+\emph{ipp}(\emph{ppi}(35))+\emph{ipp}(\emph{ppi}(35)-\emph{ipp}(\emph{ppi}(35)))=31+3+1=35.
\]
\end{exem}
This base is important for partitions with primes.
\begin{prog}\label{Program PB} number generator based numeration of prime numbers, denoted $(pb)$.
  \begin{tabbing}
    $\emph{PB}(n):=$\=\vline\ $\emph{return}\ \ 1\ \ \emph{if}\ \ n\textbf{=}1$\\
    \>\vline\ $v_{\pi(\emph{ipp}(n))+1}\leftarrow1$\\
    \>\vline\ $r\leftarrow \emph{ppi}(n)$\\
    \>\vline\ $w$\=$\emph{hile}\ r\neq1\wedge r\ne0$\\
    \>\vline\ \>\vline\ $v_{\pi(\emph{ipp}(r))+1}\leftarrow1$\\
    \>\vline\ \>\vline\ $r\leftarrow \emph{ppi}(r)$\\
    \>\vline\ $v_1\leftarrow1\ \ \emph{if}\ \ r\textbf{=}1$\\
    \>\vline\ $\emph{return}\ \ \emph{reverse}(v)^\textrm{T}$
   \end{tabbing}
   The program uses the programs: $\pi$ of counting the prime numbers, \ref{Functia Pi}, $\emph{ipp}$ inferior prime part \ref{Programipp}, $\emph{ppi}$, inferior prime complements, \ref{Functia ppi}, utilitarian function Mathcad $\emph{reverse}$.
\end{prog}
Using the sequence $n:=1..25$, $v_n=PB(n)$ the vector $v$ was generated, which contains the numbers from 1 to 25 written on the basis $(\emph{pb})$:
\begin{center}
 \begin{longtable}{|l|}
   \caption{Numbers in base $(\emph{pb})$}\\
   \hline
   $n_{(10)}=n_{(\emph{pb})}$ \\
   \hline
  \endfirsthead
   \hline
   $n_{(10)}=n_{(\emph{pb})}$ \\
   \hline
  \endhead
   \hline \multicolumn{1}{r}{\textit{Continued on next page}} \\
  \endfoot
   \hline
  \endlastfoot
     1=1\\
     2=10\\
     3=100\\
     4=101\\
     5=1000\\
     6=1001\\
     7=10000\\
     8=10001\\
     9=10010\\
     10=10100\\
     11=100000\\
     12=100001\\
     13=1000000\\
     14=1000001\\
     15=1000010\\
     16=1000100\\
     17=10000000\\
     18=10000001\\
     19=100000000\\
     20=100000001\\
     21=100000010\\
     22=100000100\\
     23=1000000000\\
     24=1000000001\\
     25=1000000010\\
     26=1000000100\\
     27=1000000101\\
     28=1000001000\\
     29=10000000000\\
     30=10000000001\\
     31=100000000000\\
     32=100000000001\\
     33=100000000010\\
     34=100000000100\\
     35=100000000101\\
     36=100000001000\\
     37=1000000000000\\
     38=1000000000001\\
     39=1000000000010\\
     40=1000000000100\\
     41=10000000000000\\
     42=10000000000001\\
     43=100000000000000\\
     44=100000000000001\\
     45=100000000000010\\
     46=100000000000100\\
     47=1000000000000000\\
     \hline
 \end{longtable}
\end{center}

\subsection{Square Base}

We def{}ined over the set of natural numbers the following inf{}inite base: for $k\in\Na$, $s_k=k^2$, denoted $(\emph{sb})$.

Each number $a\in\Na$ can be written in the square base $(sb)$. We proved that every positive integer $a$ may be uniquely written in the square base as:
\[
 a=\overline{a_m\ldots a_1a_0}_{(sb)}=\sum_{k=0}^ma_k\cdot s_k~,
\]
with $a_k=0 \vee a_k=1$ for $k\ge2$, $a_1\in\set{0,1,2}$, $a_0\in\set{0,1,2,3}$ and of course $a_m=1$, in the following way:
\begin{itemize}
  \item if $s_m\le a<s_{m+1}$, then $a=s_m+r_1$;
  \item if $s_k\le r_1<s_{k+1}$, then $r_1=s_k+r_2$, $k<m$ and so on until one obtains a rest $r_j=0$, $j<m$.
\end{itemize}
Therefore, any number may be written as a sum of squares$+e$ (1 not counted as a square -- being obvious), where $e\in\set{0,1,2,3}$.
Examples: $4=2^2+0$, $5=2^2+1$, $6=2^2+2$, $7=2^2+3$, $8=2\cdot2^2+0$, $9=3^2+0$~.

\begin{prog}\label{Program SB} for transforming a number written in base $(10)$ based on the numeration $(\emph{sb})$.
  \begin{tabbing}
    $\emph{SB}(n):=$\=\vline\ $\emph{return}\ \ 0\ \ \emph{if}\ \ n\textbf{=}0$\\
    \>\vline\ $v_{\sqrt{isp(n)}}\leftarrow1$\\
    \>\vline\ $r\leftarrow \emph{spi}(n)$\\
    \>\vline\ $k\leftarrow\sqrt{\emph{isp}(r)}$\\
    \>\vline\ $w$\=$\emph{hile}\ r>3$\\
    \>\vline\ \>\vline\ $v_k\leftarrow v_k+1$\\
    \>\vline\ \>\vline\ $r\leftarrow \emph{spi}(r)$\\
    \>\vline\ \>\vline\ $k\leftarrow\sqrt{\emph{isp}(r)}$\\
    \>\vline\ $v_1\leftarrow v_1+r$\\
    \>\vline\ $\emph{return}\ \ reverse(v)\cdot Vb(10,last(v))$\\
  \end{tabbing}
  The program uses the following user functions: $\emph{isp}$ given by \ref{Functia isp}, $\emph{spi}$ given by \ref{Functia spi}, $\emph{Vb}$ which returns the vector $(b^m\ b^{m-1}\ \ldots b^0)^\textrm{T}$ and the utilitarian function Mathcad $\emph{reverse}$.
\end{prog}
The numbers from 1 to 100 generated by the program $\emph{SB}$ are: 1, 2, 3, 10, 11, 12, 13, 20, 100, 101, 102, 103, 110, 111, 112, 1000, 1001, 1002, 1003, 1010, 1011, 1012, 1013, 1020, 10000, 10001, 10002, 10003, 10010, 10011, 10012, 10013, 10020, 10100, 10101, 100000, 100001, 100002, 100003, 100010, 100011, 100012, 100013, 100020, 100100, 100101, 100102, 100103, 1000000, 1000001, 1000002, 1000003, 1000010, 1000011, 1000012, 1000013, 1000020, 1000100, 1000101, 1000102, 1000103, 1000110, 1000111, 10000000, 10000001, 10000002, 10000003, 10000010, 10000011, 10000012, 10000013, 10000020, 10000100, 10000101, 10000102, 10000103, 10000110, 10000111, 10000112, 10001000, 100000000, 100000001, 100000002, 100000003, 100000010, 100000011, 100000012, 100000013, 100000020, 100000100, 100000101, 100000102, 100000103, 100000110, 100000111, 100000112, 100001000, 100001001, 100001002, 1000000000~.

\subsection{Cubic Base}

We def{}ined over the set of natural numbers the following inf{}inite base: for $k\in\Na$, $c_k=k^3$, denoted $(cb)$.

Each number $a\in\Na$ can be written in the square base $(cb)$. We proved that every positive integer $a$ may be uniquely written in the cubic base as:
\[
 a=\overline{a_m\ldots a_1a_0}_{(cb)}=\sum_{k=0}^ma_k\cdot c_k~,
\]
with $a_k=0 \vee a_k=1$ for $k\ge2$, $a_1\in\set{0,1,2}$, $a_0\in\set{0,1,2,\ldots,7}$ and of course $a_m=1$, in the following way:
\begin{itemize}
  \item if $c_m\le a<c_{m+1}$, then $a=c_m+r_1$;
  \item if $c_k\le r_1<c_{k+1}$, then $r_1=c_k+r_2$, $k<m$ and so on until one obtains a rest $r_j=0$, $j<m$.
\end{itemize}
Therefore, any number may be written as a sum of cub$+e$ (1 not counted as a square -- being obvious), where $e\in\set{0,1,2,\ldots,7}$.

Examples: $9=2^3+1$, $10=2^3+1$, $11=2^3+2$, $12=2^3+3$, $13=2^3+4$, $14=2^3+5$, $15=2^3+6$, $16=2\cdot2^3$, $17=2\cdot2^3+1$, $18=2\cdot2^3+2$, $19=2\cdot2^3+3$, $20=2\cdot2^3+4$, $21=2\cdot2^3+5$, $22=2\cdot2^3+6$, $23=2\cdot2^3+7$, $24=3\cdot2^3$, $25=3\cdot2^3+1$, $26=3\cdot2^3+2$, $27=3^3$~.

\begin{prog}\label{Program CB} for transforming a number written in base $(10)$ based on the numeration $(\emph{cb})$.
  \begin{tabbing}
    $\emph{CB}(n):=$\=\vline\ $\emph{return}\ \ 0\ \ \emph{if}\ \ n\textbf{=}0$\\
    \>\vline\ $k\leftarrow\sqrt[3]{\emph{icp}(n)}$\\
    \>\vline\ $v_k\leftarrow1$\\
    \>\vline\ $r\leftarrow cpi(n)$\\
    \>\vline\ $k\leftarrow\sqrt[3]{\emph{icp}(r)}$\\
    \>\vline\ $w$\=$\emph{hile}\ r>7$\\
    \>\vline\ \>\vline\ $v_k\leftarrow v_k+1$\\
    \>\vline\ \>\vline\ $r\leftarrow \emph{cpi}(r)$\\
    \>\vline\ \>\vline\ $k\leftarrow\sqrt[3]{\emph{icp}(r)}$\\
    \>\vline\ $v_1\leftarrow v_1+r$\\
    \>\vline\ $\emph{return}\ \ \emph{reverse}(v)\cdot \emph{Vb}(10,\emph{last(v)})$\\
  \end{tabbing}
The program uses the following user functions: $\emph{icp}$ given by \ref{Functia icp}, $\emph{cpi}$ given by \ref{Functia cpi}, $\emph{Vb}$ which returns the vector $(b^m\ b^{m-1}\ \ldots b^0)^\textrm{T}$ and the utilitarian function
Mathcad $\emph{reverse}$.
\end{prog}
The natural numbers from 1 to 64 generated by the program $\emph{CB}$ are: 1, 2, 3, 4, 5, 6, 7, 10, 11, 12, 13, 14, 15, 16, 17, 20, 21, 22, 23, 24, 25, 26, 27, 30, 31, 32, 100, 101, 102, 103, 104, 105, 106, 107, 110, 111, 112, 113, 114, 115, 116, 117, 120, 121, 122, 123, 124, 125, 126, 127, 130, 131, 132, 200, 201, 202, 203, 204, 205, 206, 207, 210, 211, 1000~.

\subsection{Factorial Base}

We def{}ined over the set of natural numbers the following inf{}inite base: for $k\in\Ns$, $f_k=k!$, denoted $(fb)$. We proved that every positive integer $a$ may be uniquely written in the factorial base as:
\[
 a=\overline{a_m\ldots a_1a_0}_{(fb)}=\sum_{k=0}^ma_k\cdot f_k~,
\]
with all $a_k=0,1,\ldots,k$, for $k\in\Ns$, in the following way:
\begin{eqnarray*}
  1_{(10)} &=& 1\cdot1!=1_{(fb)} \\
  2_{(10)} &=& 1\cdot2!+0\cdot1!=10_{(fb)} \\
  3_{(10)} &=& 1\cdot2!+1\cdot1!=11_{(fb)} \\
  4_{(10)} &=& 2\cdot2!+0\cdot1!=20_{(fb)} \\
  5_{(10)} &=& 2\cdot2!+1\cdot1!=21_{(fb)} \\
  6_{(10)} &=& 1\cdot3!+0\cdot2!+0\cdot1!=100_{(fb)} \\
  7_{(10)} &=& 1\cdot3!+0\cdot2!+1\cdot1!=101_{(fb)} \\
  8_{(10)} &=& 1\cdot3!+1\cdot2!+0\cdot1!=110_{(fb)} \\
  9_{(10)} &=& 1\cdot3!+1\cdot2!+1\cdot1!=111_{(fb)} \\
  10_{(10)} &=& 1\cdot3!+2\cdot2!+0\cdot1!=120_{(fb)} \\
  11_{(10)} &=& 1\cdot3!+2\cdot2!+1\cdot1!=121_{(fb)} \\
  12_{(10)} &=& 2\cdot3!+0\cdot2!+0\cdot1!=200_{(fb)} \\
  13_{(10)} &=& 2\cdot3!+0\cdot2!+1\cdot1!=201_{(fb)} \\
  14_{(10)} &=& 2\cdot3!+1\cdot2!+0\cdot1!=210_{(fb)} \\
  15_{(10)} &=& 2\cdot3!+1\cdot2!+1\cdot1!=211_{(fb)} \\
  16_{(10)} &=& 2\cdot3!+2\cdot2!+0\cdot1!=220_{(fb)} \\
  17_{(10)} &=& 2\cdot3!+2\cdot2!+1\cdot1!=221_{(fb)} \\
  18_{(10)} &=& 3\cdot3!+0\cdot2!+0\cdot1!=300_{(fb)} \\
  19_{(10)} &=& 3\cdot3!+0\cdot2!+1\cdot1!=301_{(fb)} \\
  20_{(10)} &=& 3\cdot3!+1\cdot2!+0\cdot1!=301_{(fb)} \\
  21_{(10)} &=& 3\cdot3!+1\cdot2!+1\cdot1!=311_{(fb)} \\
  22_{(10)} &=& 3\cdot3!+2\cdot2!+0\cdot1!=320_{(fb)} \\
  23_{(10)} &=& 3\cdot3!+2\cdot2!+1\cdot1!=321_{(fb)} \\
  24_{(10)} &=& 1\cdot4!+0\cdot3!+0\cdot2!+0\cdot1!=1000_{(fb)} \\
\end{eqnarray*}
\begin{prog}\label{Program FB} for transforming a number written in base $(10)$ based on the numeration $(\emph{fb})$.
  \begin{tabbing}
    $\emph{FB}(n):=$\=\vline\ $\emph{return}\ \ 0\ \ \emph{if}\ \ n\textbf{=}0$\\
    \>\vline\ $k\leftarrow \emph{ifpk}(n)$\\
    \>\vline\ $v_k\leftarrow1$\\
    \>\vline\ $r\leftarrow \emph{fpi}(n)$\\
    \>\vline\ $k\leftarrow \emph{ifpk}(r)$\\
    \>\vline\ $w$\=$\emph{hile}\ r>1$\\
    \>\vline\ \>\vline\ $v_k\leftarrow v_k+1$\\
    \>\vline\ \>\vline\ $r\leftarrow \emph{fpi}(r)$\\
    \>\vline\ \>\vline\ $k\leftarrow \emph{ifpk}(r)$\\
    \>\vline\ $v_1\leftarrow v_1+r$\\
    \>\vline\ $\emph{return}\ \ \emph{reverse(v)}\cdot \emph{Vb(10,last(v))}$\\
  \end{tabbing}
\end{prog}
The program \ref{Program FB} uses the following user functions:
\begin{prog}\label{Program ifpk} provides the largest number $k-1$ for which $k!>x$.
  \begin{tabbing}
    $\emph{ifpk}(x):=$\=\vline\ $f$\=$\emph{or}\ k\in1..18$\\
    \>\vline\ \>\ $\emph{return}\ k-1\ \emph{if}\ x<k!$\\
  \end{tabbing}
and $\emph{fpi}$ given by \ref{Functia fpi}, $Vb$ which returns the vector $(b^m\ b^{m-1}\ \ldots b^0)^\textrm{T}$ and the
utilitarian function Mathcad $\emph{reverse}$.
\end{prog}

\subsection{Double Factorial Base}

We def{}ined over the set of natural numbers the following inf{}inite base: for $k\in\Ns$, $f_k=k!!$, denoted $(\emph{dfb})$, then 1, 2, 3, 8, 15, 48, 105, 384, 945, 3840, \ldots~. We proved that every positive integer $a$ may be uniquely written in the \emph{double factorial base} as:
\[
 a=\overline{a_m\ldots a_1a_0}_{(dfb)}=\sum_{k=0}^ma_k\cdot f_k~,
\]
with all $a_k=0,1,\ldots,k$, for $k\in\Ns$, in the following way:
\begin{eqnarray*}
  1_{(10)} &=& 1\cdot1!!=1_{(dfb)} \\
  2_{(10)} &=& 1\cdot2!!+0\cdot1!!=10_{(dfb)} \\
  3_{(10)} &=& 1\cdot3!+0\cdot2!!+0\cdot1!!=100_{(dfb)} \\
  4_{(10)} &=& 1\cdot3!+0\cdot2!!+1\cdot1!!=101_{(dfb)} \\
  5_{(10)} &=& 1\cdot3!+1\cdot2!!+0\cdot1!!=110_{(dfb)} \\
  6_{(10)} &=& 2\cdot3!+0\cdot2!!+0\cdot1!!=200_{(dfb)} \\
  7_{(10)} &=& 2\cdot3!+0\cdot2!!+1\cdot1!!=201_{(dfb)} \\
  8_{(10)} &=& 1\cdot4!!+0\cdot3!+0\cdot2!!+0\cdot1!!=1000_{(dfb)} \\
  9_{(10)} &=& 1\cdot4!!+0\cdot3!+0\cdot2!!+1\cdot1!!=1001_{(dfb)} \\
  10_{(10)} &=& 1\cdot4!!+0\cdot3!+1\cdot2!!+0\cdot1!!=1010_{(dfb)} \\
\end{eqnarray*}
and so on 1100, 1101, 1110, 1200, 10000, 10001, 10010, 10100, 10101, 10110, 10200, 10201, 11000, 11001, 11010, 11100, 11101, 11110, 11200, 20000, 20001, 20010, 20100, 20101, 20110, 20200, \ldots~.

The programs transforming the numbers in base $(10)$ based on $(\emph{dfb})$ are:
\begin{prog}\label{Program idfp} for the determination of inferior double factorial part.
  \begin{tabbing}
    $\emph{idfp}(x):=$\=\vline\ $\emph{return}\ "\emph{undef{}ined}"\ \ \emph{if}\ \ x<0\vee x>\emph{kf}(28,2)$\\
    \>\vline\ $f$\=$\emph{or}\ k\in1..28$\\
    \>\vline\ \>\ $\emph{return}\ \ \emph{kf}(k-1,2)\ \ \emph{if}\ \ x< \emph{kf}(k,2)$\\
    \>\vline\ $\emph{return}\ \ "\emph{Error.}"$\\
   \end{tabbing}
   Note that the number $28!!=\emph{kf}(28,2)=14283291230208$ is smaller than $10^{16}$.
\end{prog}
\begin{func}\label{Functia dfpi} calculates the dif{}ference between $x$ and the inferior double factorial part,
\[
 \emph{dfpi}(x)=x-\emph{idfp}(x)~.
\]
\end{func}

\begin{prog}\label{Program idfpk} for determining $k-1$ for which $x<k!!$.
  \begin{tabbing}
    $\emph{idfpk}(x):=$\=\vline\ $\emph{return}\ \ "\emph{undef{}ined}"\ \ \emph{if}\ \ x<0\vee x>\emph{kf}(28,2)$\\
    \>\vline\ $f$\=$or\ k\in1..28$\\
    \>\vline\ \>\ $\emph{return}\ k-1\ \emph{if}\ x< \emph{kf}(k,2)$\\
    \>\vline\ $\emph{return}\ \ "\emph{Error.}"$\\
   \end{tabbing}
\end{prog}

\begin{prog}\label{Program DFB} for transforming a number written in base $(10)$ based on numeration $(\emph{dfb})$.
  \begin{tabbing}
    $\emph{DFB}(n):=$\=\vline\ $\emph{return}\ \ 0\ \ \emph{if}\ \ n\textbf{=}0$\\
    \>\vline\ $k\leftarrow \emph{idfpk}(n)$\\
    \>\vline\ $v_k\leftarrow1$\\
    \>\vline\ $r\leftarrow \emph{dfpi}(n)$\\
    \>\vline\ $k\leftarrow \emph{idfpk}(r)$\\
    \>\vline\ $w$\=$\emph{hile}\ r>1$\\
    \>\vline\ \>\vline\ $v_k\leftarrow v_k+1$\\
    \>\vline\ \>\vline\ $r\leftarrow \emph{dfpi}(r)$\\
    \>\vline\ \>\vline\ $k\leftarrow \emph{idfpk}(r)$\\
    \>\vline\ $v_1\leftarrow v_1+r$\\
    \>\vline\ $\emph{return}\ \ \emph{reverse}(v)\cdot \emph{Vb}(10,\emph{last(v)})$\\
  \end{tabbing}

   The program \ref{Program DFB} calls the function Mathcad $\emph{reverse}$ and the program $\emph{Vb}$  which provides the vector $(b^m\ b^{m-1}\ \ldots b^0)^\textrm{T}$.
\end{prog}

\subsection{Triangular Base}

Numbers written in the \emph{triangular base}, def{}ined as follows:
\[
t_k=\frac{k(k+1)}{2}~,
\]
for $k\in\Ns$, denoted $(\emph{tb})$, then 1, 3, 6, 10, 15, 21, 28, 36, 45, 55, \ldots~.  We proved that every positive integer $a$ may be uniquely written in the \emph{triangular base} as:
\[
 a=\overline{a_m\ldots a_1a_0}_{(\emph{tb})}=\sum_{k=0}^ma_k\cdot t_k~,
\]
with all $a_k=0,1,\ldots,k$, for $k\in\Ns$.

The series of natural numbers from 1 to 36 in base $(\emph{tb})$ is:1, 2, 10, 11, 12, 100, 101, 102, 110, 1000, 1001, 1002, 1010, 1011, 10000, 10001, 10002, 10010, 10011, 10012, 100000, 100001, 100002, 100010, 100011, 100012, 100100, 1000000, 1000001, 1000002, 1000010, 1000011, 1000012, 1000100, 1000101, 10000000~.

\subsection{Quadratic Base}

Numbers written in the \emph{quadratic base}, def{}ined as follows:
\[
 q_k=\frac{k(k+1)(2k+1)}{6}~,
\]
for $k\in\Ns$, denoted $((\emph{qb})$, then 1, 5, 14, 30, 55, 91, 140, 204, 285, 385, \ldots~. We proved that every positive integer $a$ may be uniquely written in the \emph{quadratic base} as:
\[
 a=\overline{a_m\ldots a_1a_0}_{(\emph{qb})}=\sum_{k=0}^ma_k\cdot q_k~,
\]
with all $a_k=0,1,\ldots,k$, for $k\in\Ns$.

The series of natural numbers from 1 to 36 in base $(\emph{qb})$ is: 1, 2, 3, 4, 10, 11, 12, 13, 14, 20, 21, 22, 23, 100, 101, 102, 103, 104, 110, 111, 112, 113, 114, 120, 121, 122, 123, 200, 201, 1000, 1001, 1002, 1003, 1004, 1010, 1011~.

\subsection{Pentagon Base}

Numbers written in the \emph{pentagon base}, def{}ined as follows:
\[
 pe_k=\frac{k^2(k+1)^2}{4}~,
\]
for $k\in\Ns$, denoted $(\emph{peb})$, then 1, 9, 36, 100, 225, 441, 784, 1296, 2025, 3025, \ldots~. We proved that every positive integer $a$ may be uniquely written in the \emph{pentagon base} as:
\[
 a=\overline{a_m\ldots a_1a_0}_{(\emph{\emph{peb}})}=\sum_{k=0}^ma_k\cdot \emph{pe}_k~,
\]
with all $a_k=0,1,\ldots,k$, for $k\in\Ns$.

The series of natural numbers from 1 to 100 in base $(\emph{peb})$ is: 1, 2, 3, 4, 5, 6, 7, 8, 10, 11, 12, 13, 14, 15, 16, 17, 18, 20, 21, 22, 23, 24, 25, 26, 27, 28, 30, 31, 32, 33, 34, 35, 36, 37, 38, 100, 101, 102, 103, 104, 105, 106, 107, 108, 110, 111, 112, 113, 114, 115, 116, 117, 118, 120, 121, 122, 123, 124, 125, 126, 127, 128, 130, 131, 132, 133, 134, 135, 136, 137, 138, 200, 201, 202, 203, 204, 205, 206, 207, 208, 210, 211, 212, 213, 214, 215, 216, 217, 218, 220, 221, 222, 223, 224, 225, 226, 227, 228, 230, 1000~.

\subsection{Fibonacci Base}

Numbers written in the \emph{Fibonacci base}, def{}ined as follows:
\[
 f_{k+2}=f_{k+1}+f_k~,
\]
with $f_1=1$, $f_2=2$, for $k\in\Ns$, denoted $(\emph{Fb})$, then 1, 2, 3, 5, 8, 13, 21, 34, 55, 89, \ldots~. We proved that every positive integer $a$ may be uniquely written in the \emph{Fibonacci base} as:
\[
 a=\overline{a_m\ldots a_1a_0}_{(\emph{Fb})}=\sum_{k=0}^ma_k\cdot f_k~,
\]
with all $a_k=0,1,\ldots,k$, for $k\in\Ns$. With programs like programs \ref{Program idfp}, \ref{Functia dfpi}, \ref{Program idfpk} and \ref{Program DFB} we can generate natural numbers up to 50 in base $(\emph{Fb})$: 1, 10, 100, 101, 1000, 1001, 1010, 10000, 10001, 10010, 10100, 10101, 100000, 100001, 100010, 100100, 100101, 101000, 101001, 101010, 1000000, 1000001, 1000010, 1000100, 1000101, 1001000, 1001001, 1001010, 1010000, 1010001, 1010010, 1010100, 1010101, 10000000, 10000001, 10000010, 10000100, 10000101, 10001000, 10001001, 10001010, 10010000, 10010001, 10010010, 10010100, 10010101, 10100000, 10100001, 10100010, 10100100~.

\subsection{Tribonacci Base}

Numbers written in the \emph{Tribonacci base}, def{}ined as follows:
\[
 t_{k+3}=t_{k+2}+t_{k+1}+t_k~,
\]
with $t_1=1$, $t_2=2$, $t_3=3$, for $k\in\Ns$, denoted $(\emph{Tb})$, then 1, 2, 3, 6, 11, 20, 37, 68, 125, 230 \ldots~. We proved that every positive integer $a$ may be uniquely written in the \emph{Tribonacci base} as:
\[
 a=\overline{a_m\ldots a_1a_0}_{(\emph{Tb})}=\sum_{k=0}^ma_k\cdot t_k~,
\]
with all $a_k=0,1,\ldots,k$, for $k\in\Ns$. With programs like programs \ref{Program idfp}, \ref{Functia dfpi}, \ref{Program idfpk} and \ref{Program DFB} we can generate natural numbers up to 50 in base $(\emph{Tb})$: 1, 10, 100, 101, 110, 1000, 1001, 1010, 1100, 1101, 10000, 10001, 10010, 10100, 10101, 10110, 11000, 11001, 11010, 100000, 100001, 100010, 100100, 100101, 100110, 101000, 101001, 101010, 101100, 101101, 110000, 110001, 110010, 110100, 110101, 110110, 1000000, 1000001, 1000010, 1000100, 1000101, 1000110, 1001000, 1001001, 1001010, 1001100, 1001101, 1010000, 1010001, 1010010~.

\section{Smarandache Numbers}

\emph{Smaranadche numbers} are generated with commands: $n:=1..65$, $S(n,1)=$, where the function $S$ is given by \ref{Program Smarandache}:
1, 2, 3, 4, 5, 3, 7, 4, 6, 5, 11, 4, 13, 7, 5, 6, 17, 6, 19, 5, 7, 11, 23, 4, 10, 13, 9, 7, 29, 5, 31, 8, 11, 17, 7, 6, 37, 19, 13, 5, 41, 7, 43, 11, 5, 23, 47, 6, 14, 10, 17, 13, 53, 9, 11, 7, 19, 29, 59, 5, 61, 31, 7, 8, 13, \ldots~.

\section{Smarandache Quotients}

\subsection{Smarandache Quotients of First Kind}

For each $n$ to f{}ind the smallest $k$ such that $n\cdot k$ is a factorial number.
\begin{prog}\label{Program SQ} calculation of the number Smarandache quotient.
  \begin{tabbing}
    $\emph{SQ}(n,k):=$\=\vline\ $f$\=$\emph{or}\ m\in1..n$\\
    \>\vline\ $\emph{return}\ \dfrac{\emph{kf}(m,k)}{n}\ \emph{if}\ mod\big(\emph{kf}(m,k),n\big)\textbf{=}0$\\
  \end{tabbing}
  The program $\emph{kf}$, \ref{Programkf}, calculates multifactorial.
\end{prog}

The f{}irst 30 numbers \emph{Smarandache quotients of f{}irst kind} are: 1, 1, 2, 6, 24, 1, 720, 3, 80, 12, 3628800, 2, 479001600, 360, 8, 45, 20922789888000, 40, 6402373705728000, 6, 240, 1814400, 1124000727777607680000, 1, 145152, 239500800, 13440, 180, 304888344611713860501504000000, 4~.

These numbers were obtained using the commands: $n:=1..30$, $\emph{sq}1_n:=\emph{SQ}(n,1)$ and $\emph{sq}1^\textrm{T}\rightarrow$, where $\emph{SQ}$ is the program \ref{Program SQ}.

\subsection{Smarandache Quotients of Second Kind}

For each $n$ to f{}ind the smallest $k$ such that $n\cdot k$ is a double factorial number. The f{}irst 30 numbers \emph{Smarandache quotients of second kind} are: 1, 1, 1, 2, 3, 8, 15, 1, 105, 384, 945, 4, 10395, 46080, 1, 3, 2027025, 2560, 34459425, 192, 5, 3715891200, 13749310575, 2, 81081, 1961990553600, 35, 23040, 213458046676875, 128~.

These numbers were obtained using the commands: $n:=1..30$, $\emph{sq}2_n:=\emph{SQ}(n,2)$ and $\emph{sq}2^\textrm{T}\rightarrow$, where $\emph{SQ}$ is the program \ref{Program SQ}.

\subsection{Smarandache Quotients of Third Kind}

For each $n$ to f{}ind the smallest $k$ such that $n\cdot k$ is a triple factorial number. The f{}irst 30 numbers \emph{Smarandache quotients of third kind} are: 1, 1, 1, 1, 2, 3, 4, 10, 2, 1, 80, 162, 280, 2, 1944, 5, 12320, 1, 58240, 4, 524880, 40, 4188800, 81, 167552, 140, 6, 1, 2504902400, 972~.

These numbers were obtained using the commands: $n:=1..30$, $\emph{sq}3_n:=SQ(n,3)$ and $\emph{sq}3^\textrm{T}\rightarrow$, where $\emph{SQ}$ is the program \ref{Program SQ}.

\section{Primitive Numbers}

\subsection{Primitive Numbers of Power 2}

$\emph{S2}(n)$ is the smallest integer such that $\emph{S2}(n)!$ is divisible by $2^n$. The f{}irst primitive numbers (of power 2) are: 2, 4, 4, 6, 8, 8, 8, 10, 12, 12, 14, 16, 16, 16, 16, 18, 20, 20, 22, 24, 24, 24, 26, 28, 28, 30, 32, 32, 32, 32, 32, 34, 36, 36, 38, 40, 40, 40, 42, 44, 44, 46, 48, 48, 48, 48, 50, 52, 52, 54, 56, 56, 56, 58, 60, 60, 62, 64, 64, 64, 64, 64, 64, 66, \ldots~. This sequence was generated with the program $\emph{Spk}$, given by \ref{Program Spk}.

Curious property: This is the sequence of even numbers, each number being repeated as many times as its exponent (of power 2) is.

This is one of irreducible functions, noted $\emph{S2}(k)$, which helps to calculate the Smarandache function, \ref{Program Smarandache}.

\subsection{Primitive Numbers of Power 3}

$\emph{S3}(n)$ is the smallest integer such that $\emph{S3}(n)!$ is divisible by $3^n$. The f{}irst primitive numbers (of power 3) obtain with command $\emph{Spk}(n,3)\rightarrow$ 3, 6, 9, 9, 12, 15, 18, 18, 21, 24, 27, 27, 27, 30, 33, 36, 36, 39, 42, 45, 45, 48, 51, 54, 54, 54,57, 60, 63, 63, 66, 69, 72, 72, 75, 78, 81, 81, 81, 81, 84, 87, 90, 90, 93, 96, 99, 99, 102, 105, 108, 108, 108, 111, \ldots~. The program $\emph{Spk}$ is given by \ref{Program Spk}.

Curious property: this is the sequence of multiples of 3, each number being repeated as many times as its exponent (of power 3) is.

This is one of irreducible functions, noted $\emph{S3}(k)$, which helps to calculate the Smarandache function, \ref{Program Smarandache}.

\subsection{Primitive Numbers of Power Prime}

Let $p\in\NP{2}$, then $m=\emph{Spk}(n,p,k)$ is the smallest integer such that $m!!\ldots!$ ($k-$factorial) is divisible by $p^n$.
\begin{prog}\label{Program Spk} for generated primitive numbers of power $p$ and factorial $k$.
  \begin{tabbing}
    $\emph{Spk}(n,p,k):=$\=\vline\ $f$\=$\emph{or}\ m\in1..n\cdot p$\\
    \>\vline\ \>\ $\emph{return}\ m\ \emph{if}\ \emph{mod}\big(\emph{kf}(m,k),p^n\big)=0$\\
    \>\vline\ $\emph{return}\ -1$\\
   \end{tabbing}
\end{prog}
\begin{prop}
  For every $m> n\cdot p$, $p^n\nmid m!!\ldots!$ ($k-$factorial).
\end{prop}
\begin{proof}
  Case $m!$. Let $m=(n+1)\cdot$, then $m!=1\cdots p\cdots2p\cdots(n+1)p$, i.e. we have $n+1$ of $p$ in factorial, then $p^n\nmid m!$.

  Case $m!!$. Let $m=(n+1)p$, where $n$ is odd, then $m!!=1\cdot3\cdots p\cdots3p\cdots n\cdotp$, i.e. number of $p$ in the factorial product is $<n$, then $p^n\nmid m!!$. If $n$ is even, then $m!!=1\cdot3\cdots p\cdots3p\cdots(n+1)p$ and now number of $p$ in the factorial product is $<n$, then $p^n\nmid m!!$.

  Cases $m!!!$, \ldots, $m!\ldots!$ ($k$ times). These cases prove analogous.
\end{proof}

We consider command $n:=1..40$, then
\begin{description}
  \item[$\emph{Spk}(n,2,1)\rightarrow$]  2, 4, 4, 6, 8, 8, 8, 10, 12, 12, 14, 16, 16, 16, 16, 18, 20, 20, 22, 24, 24, 24, 26, 28, 28, 30, 32, 32, 32, 32, 32, 34, 36, 36, 38, 40, 40, 40, 42, 44~.
  \item[$\emph{Spk}(n,2,2)\rightarrow$]  2, 4, 4, 6, 8, 8, 8, 10, 12, 12, 14, 16, 16, 16, 16, 18, 20, 20, 22, 24, 24, 24, 26, 28, 28, 30, 32, 32, 32, 32, 32, 34, 36, 36, 38, 40, 40, 40, 42, 44~.
  \item[$\emph{Spk}(n,2,3)\rightarrow$] 2, 4, $-1$, 8, $-1$, $-1$, $-1$, $-1$, $-1$, $-1$, $-1$, $-1$, $-1$, $-1$, $-1$, $-1$, $-1$, $-1$, $-1$, $-1$, $-1$, $-1$, $-1$, $-1$, $-1$, $-1$, $-1$, $-1$, $-1$, $-1$, $-1$, $-1$, $-1$, $-1$, $-1$, $-1$, $-1$, $-1$, $-1$, $-1$~. In general, this case does not make sense.
  \item[$\emph{Spk}(n,2,4)\rightarrow$] 2, 4, $-1$, 8, 8, 12, 12, 16, 16, 16, 16, 20, 20, 24, 24, 24, 28, 28, 32, 32, 32, 32, 32, 36, 36, 40, 40, 40, 44, 44, 48, 48, 48, 48, 52, 52, 56, 56, 56, 60~.
  \item[$\emph{Spk}(n,3,1)\rightarrow$] 3, 6, 9, 9, 12, 15, 18, 18, 21, 24, 27, 27, 27, 30, 33, 36, 36, 39, 42, 45, 45, 48, 51, 54, 54, 54, 57, 60, 63, 63, 66, 69, 72, 72, 75, 78, 81, 81, 81, 81~.
  \item[$\emph{Spk}(n,3,2)\rightarrow$] 3, $-1$, 9, $-1$, $-1$, $-1$, $-1$, $-1$, $-1$, $-1$, $-1$, $-1$, $-1$, $-1$, $-1$, $-1$, $-1$, $-1$, $-1$, $-1$, $-1$, $-1$, $-1$, $-1$, $-1$, $-1$, $-1$, $-1$, $-1$, $-1$, $-1$, $-1$, $-1$, $-1$, $-1$, $-1$, $-1$, $-1$, $-1$, $-1$~. In general, this case does not make sense.
  \item[$\emph{Spk}(n,3,3)\rightarrow$] 3, 6, 9, 9, 12, 15, 18, 18, 21, 24, 27, 27, 27, 30, 33, 36, 36, 39, 42, 45, 45, 48, 51, 54, 54, 54, 57, 60, 63, 63, 66, 69, 72, 72, 75, 78, 81, 81, 81, 81~.
  \item[$\emph{Spk}(n,5,1)\rightarrow$] 5, 10, 15, 20, 25, 25, 30, 35, 40, 45, 50, 50, 55, 60, 65, 70, 75, 75, 80, 85, 90, 95, 100, 100, 105, 110, 115, 120, 125, 125, 125, 130, 135, 140, 145, 150, 150, 155, 160, 165~.
  \item[$\emph{Spk}(n,5,2)\rightarrow$] 5, $-1$, $-1$, $-1$, $-1$, $-1$, $-1$, $-1$, $-1$, $-1$, $-1$, $-1$, $-1$, $-1$, $-1$, $-1$, $-1$, $-1$, $-1$, $-1$, $-1$, $-1$, $-1$, $-1$, $-1$, $-1$, $-1$, $-1$, $-1$, $-1$, $-1$, $-1$, $-1$, $-1$, $-1$, $-1$, $-1$, $-1$, $-1$, $-1$~. In general, this case does not make sense.
  \item[$\emph{Spk}(n,5,3)\rightarrow$] 5, $-1$, $-1$, $-1$, $-1$, $-1$, $-1$, $-1$, $-1$, $-1$, $-1$, $-1$, $-1$, $-1$, $-1$, $-1$, $-1$, $-1$, $-1$, $-1$, $-1$, $-1$, $-1$, $-1$, $-1$, $-1$, $-1$, $-1$, $-1$, $-1$, $-1$, $-1$, $-1$, $-1$, $-1$, $-1$, $-1$, $-1$, $-1$, $-1$~. In general, this case does not make sense.
  \item[$\emph{Spk}(n,5,4)\rightarrow$] 5, $-1$, $-1$, $-1$, $-1$, $-1$, $-1$, $-1$, $-1$, $-1$, $-1$, $-1$, $-1$, $-1$, $-1$, $-1$, $-1$, $-1$, $-1$, $-1$, $-1$, $-1$, $-1$, $-1$, $-1$, $-1$, $-1$, $-1$, $-1$, $-1$, $-1$, $-1$, $-1$, $-1$, $-1$, $-1$, $-1$, $-1$, $-1$, $-1$~. In general, this case does not make sense.
  \item[$\emph{Spk}(n,5,5)\rightarrow$] 5, 10, 15, 20, 25, 25, 30, 35, 40, 45, 50, 50, 55, 60, 65, 70, 75, 75, 80, 85, 90, 95, 100, 100, 105, 110, 115, 120, 125, 125, 125, 130, 135, 140, 145, 150, 150, 155, 160, 165~.
\end{description}

\section{$m-$Power Residues}

\subsection{Square Residues}

For $n\in\Ns$ square residues (denoted by $s_r$) is: if $n=\desp[\alpha]{s}$, then $s_r(n)=p_1^{min\set{1,\alpha_1}}\cdot p_2^{min\set{1,\alpha_2}}\cdots p_s^{min\set{1,\alpha_s}}$. The sequence numbers square residues is: 1, 2, 3, 2, 5, 6, 7, 2, 3, 10, 11, 6, 13, 14, 15, 2, 17, 6, 19, 10, 21, 22, 23, 6, 5, 26, 3, 14, 29, 30, 31, 2, 33, 34, 35, 6, 37, 38, 39, 10, 41, 42, 43, 22, 15, 46, 47, 6, 7, 10, 51, 26, 53, 6, 14, 57, 58, 59, 30, 61, 62, 21, \ldots ~.

\subsection{Cubic Residues}

For $n\in\Ns$ cubic residues (denoted by $c_r$) is: if $n=\desp[\alpha]{s}$, then $c_r(n)=p_1^{\min\set{2,\alpha_1}}\cdot p_2^{\min\set{2,\alpha_2}}\cdots p_s^{\min\set{2,\alpha_s}}$. The sequence numbers cubic residues is: 1, 2, 3, 4, 5, 6, 7, 4, 9, 10, 11, 12, 13, 14, 15, 4, 17, 18, 19, 20, 21, 22, 23, 12, 25, 26, 9, 28, 29, 30, 31, 4, 33, 34, 35, 36, 37, 38, 39, 20, 41, 42, 43, 44, 45, 46, 47, 12, 49, 50, 51, 52, 53, 18, 55, 28, \ldots~.

\subsection{$m-$Power Residues}

For $n\in\Ns$ $m$--power residues (denoted by $m_r$) is: if $n=\desp[\alpha]{s}$, then $m_r(n)=p_1^{\min\set{m-1,\alpha_1}}\cdot p_2^{\min\set{m-1,\alpha_2}}\cdots p_s^{\min\set{m-1,\alpha_s}}$.

\section{Exponents of Power $m$}

\subsection{Exponents of Power 2}

For $n\in\Ns$, $e_2(n)$ is the largest exponent of power $2$ which divides $n$ or $e_2(n)=k$ if $2^k$ divides $n$ but $2^{k+1}$ does not.

\begin{prog}\label{Program Exp} for calculating the number $e_b(n)$.
   \begin{tabbing}
     $\emph{Exp}(b,n):=$\=\vline\ $a\leftarrow0$\\
     \>\vline\ $k\leftarrow1$\\
     \>\vline\ $w$\=$\emph{hile}\ b^k\le n$\\
     \>\vline\ \>\vline\ $a\leftarrow k\ \emph{if}\ \mod(n,b^k)\textbf{=}0$\\
     \>\vline\ \>\vline\ $k\leftarrow k+1$\\
     \>\vline\ $\emph{return}\ \ a$\\
   \end{tabbing}
\end{prog}

For $n=1,2,\ldots,200$ and $\emph{Exp}(2,n)=$, the program \ref{Program Exp}, one obtains: 0, 1, 0, 2, 0, 1, 0, 3, 0, 1, 0, 2, 0, 1, 0, 4, 0, 1, 0, 2, 0, 1, 0, 3, 0, 1, 0, 2, 0, 1, 0, 5, 0, 1, 0, 2, 0, 1, 0, 3, 0, 1, 0, 2, 0, 1, 0, 4, 0, 1, 0, 2, 0, 1, 0, 3, 0, 1, 0, 2, 0, 1, 0, 6, 0, 1, 0, 2, 0, 1, 0, 3, 0, 1, 0, 2, 0, 1, 0, 4, 0, 1, 0, 2, 0, 1, 0, 3, 0, 1, 0, 2, 0, 1, 0, 5, 0, 1, 0, 2, 0, 1, 0, 3, 0, 1, 0, 2, 0, 1, 0, 4, 0, 1, 0, 2, 0, 1, 0, 3, 0, 1, 0, 2, 0, 1, 0, 7, 0, 1, 0, 2, 0, 1, 0, 3, 0, 1, 0, 2, 0, 1, 0, 4, 0, 1, 0, 2, 0, 1, 0, 3, 0, 1, 0, 2, 0, 1, 0, 5, 0, 1, 0, 2, 0, 1, 0, 3, 0, 1, 0, 2, 0, 1, 0, 4, 0, 1, 0, 2, 0, 1, 0, 3, 0, 1, 0, 2, 0, 1, 0, 6, 0, 1, 0, 2, 0, 1, 0, 3.

\subsection{Exponents of Power 3}

For $n\in\Ns$, $e_3(n)$ is the largest exponent of power $3$ which divides $n$ or $e_3(n)=k$ if $3^k$ divides $n$ but $3^{k+1}$ does not.

For $n=1,2,\ldots,200$ and $\emph{Exp}(3,n)=$, the program \ref{Program Exp}, one obtains: 0, 0, 1, 0, 0, 1, 0, 0, 2, 0, 0, 1, 0, 0, 1, 0, 0, 2, 0, 0, 1, 0, 0, 1, 0, 0, 3, 0, 0, 1, 0, 0, 1, 0, 0, 2, 0, 0, 1, 0, 0, 1, 0, 0, 2, 0, 0, 1, 0, 0, 1, 0, 0, 3, 0, 0, 1, 0, 0, 1, 0, 0, 2, 0, 0, 1, 0, 0, 1, 0, 0, 2, 0, 0, 1, 0, 0, 1, 0, 0, 4, 0, 0, 1, 0, 0, 1, 0, 0, 2, 0, 0, 1, 0, 0, 1, 0, 0, 2, 0, 0, 1, 0, 0, 1, 0, 0, 3, 0, 0, 1, 0, 0, 1, 0, 0, 2, 0, 0, 1, 0, 0, 1, 0, 0, 2, 0, 0, 1, 0, 0, 1, 0, 0, 3, 0, 0, 1, 0, 0, 1, 0, 0, 2, 0, 0, 1, 0, 0, 1, 0, 0, 2, 0, 0, 1, 0, 0, 1, 0, 0, 4, 0, 0, 1, 0, 0, 1, 0, 0, 2, 0, 0, 1, 0, 0, 1, 0, 0, 2, 0, 0, 1, 0, 0, 1, 0, 0, 3, 0, 0, 1, 0, 0, 1, 0, 0, 2, 0, 0.

\subsection{Exponents of Power $b$}

For $n\in\Ns$, $e_b(n)$ is the largest exponent of power $b$ which divides $n$ or $e_b(n)=k$ if $b^k$ divides $n$ but $b^{k+1}$ does not.

For $n=1,2,\ldots,60$ and $\emph{Exp}(5,n)=$, the program \ref{Program Exp}, one obtains: 0, 0, 0, 0, 1, 0, 0, 0, 0, 1, 0, 0, 0, 0, 1, 0, 0, 0, 0, 1, 0, 0, 0, 0, 2, 0, 0, 0, 0, 1, 0, 0, 0, 0, 1, 0, 0, 0, 0, 1, 0, 0, 0, 0, 1, 0, 0, 0, 0, 2, 0, 0, 0, 0, 1, 0, 0, 0, 0, 1.

\section{Almost Prime}

\subsection{Almost Primes of First Kind}

Let $a_1\ge2$, and for $k\ge1$, $a_{k+1}$ is the smallest number that is not divisible by any of the previous terms (of the sequence) $a_1$, $a_2$, \ldots, $a_k$. If one starts by $a_1=2$, it obtains the complete prime sequence and only it.

If one starts by $a_1>2$, it obtains after a rank $r$, where $a_r=spp(a_1)^2$ with $spp(x)$, \ref{Programspp}, the strictly superior prime part of $x$, i.e. the largest prime strictly less than $x$, the prime sequence:
\begin{itemize}
  \item between $a_1$ and $a_r$, the sequence contains all prime numbers of this interval and some composite numbers;
  \item from $a_{r+1}$ and up, the sequence contains all prime numbers greater than $a_r$ and no composite numbers.
\end{itemize}

\begin{prog}\label{Program AP1} for generating the numbers \emph{almost primes of f{}irst kind} de la $n$ la $L$.
  \begin{tabbing}
    $\emph{AP1}(n,L):=$\=\vline\ $j\leftarrow1$\\
    \>\vline\ $a_j\leftarrow n$\\
    \>\vline\ $f$\=$\emph{or}\ m\in n+1..L$\\
    \>\vline\ \>\vline\ $\emph{sw}\leftarrow0$\\
    \>\vline\ \>\vline\ $f$\=$\emph{or}\ k\in1..j$\\
    \>\vline\ \>\vline\ \>\ $i$\=$f\ \mod(m,a_k)\textbf{=}0$\\
    \>\vline\ \>\vline\ \>\ \>\vline\ $\emph{sw}\leftarrow1$\\
    \>\vline\ \>\vline\ \>\ \>\vline\ $\emph{break}$\\
    \>\vline\ \>\vline\ $i$\=$f\ sw\textbf{=}0$\\
    \>\vline\ \>\vline\ \>\vline\ $j\leftarrow j+1$\\
    \>\vline\ \>\vline\ \>\vline\ $a_j\leftarrow m$\\
    \>\vline\ $\emph{return}\ \ a$
  \end{tabbing}
\end{prog}

The f{}irst numbers \emph{almost prime of f{}irst kind} given by $\emph{AP1}(10,10^3)$ are: 10, 11, 12, 13, 14, 15, 16, 17, 18, 19, 21, 23, 25, 27, 29, 31, 35, 37, 41, 43, 47, 49, 53, 59, 61, 67, 71, 73, 79, 83, 89, 97, 101, 103, 107, 109, 113, 127, 131, 137, 139, 149, 151, 157, 163, 167, 173, 179, 181, 191, 193, 197, 199, 211, 223, 227, 229, 233, 239, 241, 251, 257, 263, 269, 271, 277, 281, 283, 293, 307, 311, 313, 317, 331, 337, 347, 349, 353, 359, 367, 373, 379, 383, 389, 397, 401, 409, 419, 421, 431, 433, 439, 443, 449, 457, 461, 463, 467, 479, 487, 491, 499, 503, 509, 521, 523, 541, 547, 557, 563, 569, 571, 577, 587, 593, 599, 601, 607, 613, 617, 619, 631, 641, 643, 647, 653, 659, 661, 673, 677, 683, 691, 701, 709, 719, 727, 733, 739, 743, 751, 757, 761, 769, 773, 787, 797, 809, 811, 821, 823, 827, 829, 839, 853, 857, 859, 863, 877, 881, 883, 887, 907, 911, 919, 929, 937, 941, 947, 953, 967, 971, 977, 983, 991, 997~(See Figure \ref{MathcadAlmostaPrimes}).

\subsection{Almost Prime of Second Kind}

Let $a_1\ge2$, and for $k\ge1$, $a_{k+1}$ is the smallest number that is coprime ($a$ is coprime $b$ $\Leftrightarrow$ $gcd(a,b)=1$ with all of the previous terms (of the sequence), $a_1$, $a_2$, \ldots, $a_k$.

This second kind sequence merges faster to the prime numbers than the f{}irst kind sequence. If one starts by $a_1=2$, it obtains the complete prime sequence and only it.

If one starts by $a_1>2$, it obtains after a rank $r$, where $a_1=p_i\cdot p_j$ with $p_i$ and $p_j$ prime numbers strictly less than and not dividing $a_1$, the prime sequence:
\begin{itemize}
  \item between $a_1$ and $a_r$, the sequence contains all prime numbers of this interval and some composite numbers;
  \item from $a_{r+1}$ and up, the sequence contains all prime numbers greater than $a_r$ and no composite numbers.
\end{itemize}

\begin{prog}\label{Program AP2} for generating the numbers \emph{almost primes of second kind}.
  \begin{tabbing}
    $\emph{AP2}(n,L):=$\=\vline\ $j\leftarrow1$\\
    \>\vline\ $a_j\leftarrow n$\\
    \>\vline\ $f$\=$\emph{or}\ m\in n+1..L$\\
    \>\vline\ \>\vline\ $\emph{sw}\leftarrow0$\\
    \>\vline\ \>\vline\ $f$\=$\emph{or}\ k\in1..j$\\
    \>\vline\ \>\vline\ \>\ $i$\=$f\ \emph{gcd}(m,a_k)\neq1$\\
    \>\vline\ \>\vline\ \>\ \>\vline\ $\emph{sw}\leftarrow1$\\
    \>\vline\ \>\vline\ \>\ \>\vline\ $\emph{break}$\\
    \>\vline\ \>\vline\ $i$\=$f\ sw\textbf{=}0$\\
    \>\vline\ \>\vline\ \>\vline\ $j\leftarrow j+1$\\
    \>\vline\ \>\vline\ \>\vline\ $a_j\leftarrow m$\\
    \>\vline\ $\emph{return}\ \ a$
  \end{tabbing}
\end{prog}

The f{}irst numbers almost prime of second kind given by the program $\emph{AP2}(10,10^3)$ are: 10, 11, 13, 17, 19, 21, 23, 29, 31, 37, 41, 43, 47, 53, 59, 61, 67, 71, 73, 79, 83, 89, 97, 101, 103, 107, 109, 113, 127, 131, 137, 139, 149, 151, 157, 163, 167, 173, 179, 181, 191, 193, 197, 199, 211, 223, 227, 229, 233, 239, 241, 251, 257, 263, 269, 271, 277, 281, 283, 293, 307, 311, 313, 317, 331, 337, 347, 349, 353, 359, 367, 373, 379, 383, 389, 397, 401, 409, 419, 421, 431, 433, 439, 443, 449, 457, 461, 463, 467, 479, 487, 491, 499, 503, 509, 521, 523, 541, 547, 557, 563, 569, 571, 577, 587, 593, 599, 601, 607, 613, 617, 619, 631, 641, 643, 647, 653, 659, 661, 673, 677, 683, 691, 701, 709, 719, 727, 733, 739, 743, 751, 757, 761, 769, 773, 787, 797, 809, 811, 821, 823, 827, 829, 839, 853, 857, 859, 863, 877, 881, 883, 887, 907, 911, 919, 929, 937, 941, 947, 953, 967, 971, 977, 983, 991, 997~(See Figure \ref{MathcadAlmostaPrimes}).

\section{Pseudo--Primes}

\subsection{Pseudo--Primes of First Kind}

\begin{defn}\label{Definitia Pseudo-Prime 1}
  A number is a \emph{pseudo--prime of f{}irst kind} if exist a permutation of the digits that is a prime number, including the identity permutation.
\end{defn}

The matrices $Per2$, $Per3$ and $Per4$ contain all the permutations from 2, 3 and 4.
\begin{equation}\label{Per2}
  Per2:=\left(
          \begin{array}{cc}
            1 & 2 \\
            2 & 1 \\
          \end{array}
        \right)
\end{equation}
\begin{equation}\label{Per3}
  Per3:=\left(
          \begin{array}{cccccc}
            1 & 1 & 2 & 2 & 3 & 3 \\
            2 & 3 & 1 & 3 & 1 & 2 \\
            3 & 2 & 3 & 1 & 2 & 1 \\
          \end{array}
        \right)
\end{equation}

\begin{multline}\label{Per4}
  Per4:=\left(\begin{array}{cccccccccccc}
                1 & 1 & 1 & 1 & 1 & 1 & 2 & 2 & 2 & 2 & 2 & 2 \\
                2 & 2 & 3 & 3 & 4 & 4 & 1 & 1 & 3 & 3 & 4 & 4 \\
                3 & 4 & 2 & 4 & 2 & 3 & 3 & 4 & 1 & 4 & 1 & 3 \\
                4 & 3 & 4 & 2 & 3 & 2 & 4 & 3 & 4 & 1 & 3 & 1 \\
              \end{array}\right.\\
        \left.\begin{array}{cccccccccccc}
                3 & 3 & 3 & 3 & 3 & 3 & 4 & 4 & 4 & 4 & 4 & 4\\
                1 & 1 & 2 & 2 & 4 & 4 & 1 & 1 & 2 & 2 & 3 & 3\\
                2 & 4 & 1 & 4 & 1 & 2 & 2 & 3 & 1 & 3 & 1 & 2\\
                4 & 2 & 4 & 1 & 2 & 1 & 3 & 2 & 3 & 1 & 2 & 1\\
           \end{array}\right)
\end{multline}

\begin{prog}\label{Program PP} for counting the primes obtained by the permutation of number\rq{s} digits.
  \begin{tabbing}
    $\emph{PP}(n,q):=$\=\vline\ $m\leftarrow \emph{nrd}(n,10)$\\
    \>\vline\ $d\leftarrow \emph{dn}(n,10)$\\
    \>\vline\ $\emph{np}\leftarrow 1\ \ \emph{if}\ \ m\textbf{=}1$\\
    \>\vline\ $\emph{np}\leftarrow \emph{cols(Per2)}\ \emph{if}\ m\textbf{=}2$\\
    \>\vline\ $\emph{np}\leftarrow \emph{cols(Per3)}\ \emph{if}\ m\textbf{=}3$\\
    \>\vline\ $\emph{np}\leftarrow \emph{cols(Per4)}\ \emph{if}\ m\textbf{=}4$\\
    \>\vline\ $\emph{sw}\leftarrow0$\\
    \>\vline\ $f$\=$\emph{or}\ j\in q..\emph{max}(q,np)$\\
    \>\vline\ \>\vline\ $f$\=$\emph{or}\ k\in1..m$\\
    \>\vline\ \>\vline\ \>\vline\ $\emph{pd}\leftarrow d\ \emph{if}\ m\textbf{=}1$\\
    \>\vline\ \>\vline\ \>\vline\ $\emph{pd}_k\leftarrow d_{(\emph{Per2}_{k,j})}\ \emph{if}\ m\textbf{=}2$\\
    \>\vline\ \>\vline\ \>\vline\ $\emph{pd}_k\leftarrow d_{(\emph{Per3}_{k,j})}\ \emph{if}\ m\textbf{=}3$\\
    \>\vline\ \>\vline\ \>\vline\ $\emph{pd}_k\leftarrow d_{(\emph{Per4}_{k,j})}\ \emph{if}\ m\textbf{=}4$\\
    \>\vline\ \>\vline\ $\emph{nn}\leftarrow \emph{pd}\cdot \emph{Vb}(10,m)$\\
    \>\vline\ \>\vline\ $\emph{sw}\leftarrow \emph{sw}+1\ \emph{if}\ \emph{TS}(nn)\textbf{=}1$\\
    \>\vline\ $\emph{return}\ \ \emph{sw}$\\
  \end{tabbing}
The program uses the subprograms: $\emph{nrd}$ given by \ref{FunctionNrd}, $\emph{dn}$ given by \ref{ProgramDn}, $\emph{Vb}(b,m)$ which returns the vector $(b^m\ b^{m-1}\ \ldots b^0)^\textrm{T}$, $\emph{TS}$ Smarandache primality test def{}ined at \ref{Program TS}. Also, the program entails the matrices (\ref{Per2}), (\ref{Per3}) and (\ref{Per4}) which
contain all the permutation of sets $\set{1,2}$, $\set{1,2,3}$ and $\set{1,2,3,4}$.
\end{prog}

The f{}irst 457 of numbers \emph{pseudo-prime of f{}irst kinds} are: 2, 3, 5, 7, 11, 13, 14, 16, 17, 19, 20, 23, 29, 30, 31, 32, 34, 35, 37, 38, 41, 43, 47, 50, 53, 59, 61, 67, 70, 71, 73, 74, 76, 79, 83, 89, 91, 92, 95, 97, 98, 101, 103, 104, 106, 107, 109, 110, 112, 113, 115, 118, 119, 121, 124, 125, 127, 128, 130, 131, 133, 134, 136, 137, 139, 140, 142, 143, 145, 146, 149, 151, 152, 154, 157, 160, 163, 164, 166, 167, 169, 170, 172, 173, 175, 176, 179, 181, 182, 188, 190, 191, 193, 194, 196, 197, 199, 200, 203, 209, 211, 214, 215, 217, 218, 223, 227, 229, 230, 232, 233, 235, 236, 238, 239, 241, 251, 253, 257, 263, 269, 271, 272, 275, 277, 278, 281, 283, 287, 289, 290, 292, 293, 296, 298, 299, 300, 301, 302, 304, 305, 307, 308, 310, 311, 313, 314, 316, 317, 319, 320, 322, 323, 325, 326, 328, 329, 331, 332, 334, 335, 337, 338, 340, 341, 343, 344, 346, 347, 349, 350, 352, 353, 356, 358, 359, 361, 362, 364, 365, 367, 368, 370, 371, 373, 374, 376, 377, 379, 380, 382, 383, 385, 386, 388, 389, 391, 392, 394, 395, 397, 398, 401, 403, 407, 409, 410, 412, 413, 415, 416, 419, 421, 430, 431, 433, 434, 436, 437, 439, 443, 449, 451, 457, 461, 463, 467, 470, 473, 475, 476, 478, 479, 487, 490, 491, 493, 494, 497, 499, 500, 503, 509, 511, 512, 514, 517, 521, 523, 527, 530, 532, 533, 536, 538, 539, 541, 547, 557, 563, 569, 571, 572, 574, 575, 577, 578, 583, 587, 589, 590, 593, 596, 598, 599, 601, 607, 610, 613, 614, 616, 617, 619, 623, 629, 631, 632, 634, 635, 637, 638, 641, 643, 647, 653, 659, 661, 670, 671, 673, 674, 677, 679, 683, 691, 692, 695, 697, 700, 701, 703, 704, 706, 709, 710, 712, 713, 715, 716, 719, 721, 722, 725, 727, 728, 730, 731, 733, 734, 736, 737, 739, 740, 743, 745, 746, 748, 749, 751, 752, 754, 755, 757, 758, 760, 761, 763, 764, 767, 769, 772, 773, 775, 776, 778, 779, 782, 784, 785, 787, 788, 790, 791, 793, 794, 796, 797, 799, 803, 809, 811, 812, 818, 821, 823, 827, 829, 830, 832, 833, 835, 836, 838, 839, 847, 853, 857, 859, 863, 872, 874, 875, 877, 878, 881, 883, 887, 890, 892, 893, 895, 901, 902, 904, 905, 907, 908, 910, 911, 913, 914, 916, 917, 919, 920, 922, 923, 926, 928, 929, 931, 932, 934, 935, 937, 938, 940, 941, 943, 944, 947, 949, 950, 953, 956, 958, 959, 961, 962, 965, 967, 970, 971, 973, 974, 976, 977, 979, 980, 982, 983, 985, 991, 992, 994, 995, 997~. This numbers obtain with command $\emph{APP1}(2,999)$, where program $\emph{APP1}$ is:
\begin{prog}\label{Program APP1} for displaying the Pseudo--Primes of First Kind numbers.
   \begin{tabbing}
     $\emph{APP1}(a,b):=$\=\vline\ $j\leftarrow1$\\
     \>\vline\ $f$\=$\emph{or}\ n\in a..b$\\
     \>\vline\ \>\ $i$\=$f\ \emph{PP}(n,1)\ge1$\\
     \>\vline\ \>\ \>\vline\ $\emph{pp}_j\leftarrow n$\\
     \>\vline\ \>\ \>\vline\ $j\leftarrow j+1$\\
     \>\vline\ $\emph{return}\ \ \emph{pp}$
   \end{tabbing}
\end{prog}

\subsection{Pseudo--Primes of Second Kind}

\begin{defn}\label{Definitia Pseudo-Prime 2}
  A composite number is a \emph{pseudo--prime of second kind} if exist a permutation of the digits that is a prime number.
\end{defn}

The f{}irst 289 of numbers \emph{pseudo--prime of second kinds} are: 14, 16, 20, 30, 32, 34, 35, 38, 50, 70, 74, 76, 91, 92, 95, 98, 104, 106, 110, 112, 115, 118, 119, 121, 124, 125, 128, 130, 133, 134, 136, 140, 142, 143, 145, 146, 152, 154, 160, 164, 166, 169, 170, 172, 175, 176, 182, 188, 190, 194, 196, 200, 203, 209, 214, 215, 217, 218, 230, 232, 235, 236, 238, 253, 272, 275, 278, 287, 289, 290, 292, 296, 298, 299, 300, 301, 302, 304, 305, 308, 310, 314, 316, 319, 320, 322, 323, 325, 326, 328, 329, 332, 334, 335, 338, 340, 341, 343, 344, 346, 350, 352, 356, 358, 361, 362, 364, 365, 368, 370, 371, 374, 376, 377, 380, 382, 385, 386, 388, 391, 392, 394, 395, 398, 403, 407, 410, 412, 413, 415, 416, 430, 434, 436, 437, 451, 470, 473, 475, 476, 478, 490, 493, 494, 497, 500, 511, 512, 514, 517, 527, 530, 532, 533, 536, 538, 539, 572, 574, 575, 578, 583, 589, 590, 596, 598, 610, 614, 616, 623, 629, 632, 634, 635, 637, 638, 670, 671, 674, 679, 692, 695, 697, 700, 703, 704, 706, 710, 712, 713, 715, 716, 721, 722, 725, 728, 730, 731, 734, 736, 737, 740, 745, 746, 748, 749, 752, 754, 755, 758, 760, 763, 764, 767, 772, 775, 776, 778, 779, 782, 784, 785, 788, 790, 791, 793, 794, 796, 799, 803, 812, 818, 830, 832, 833, 835, 836, 838, 847, 872, 874, 875, 878, 890, 892, 893, 895, 901, 902, 904, 905, 908, 910, 913, 914, 916, 917, 920, 922, 923, 926, 928, 931, 932, 934, 935, 938, 940, 943, 944, 949, 950, 956, 958, 959, 961, 962, 965, 970, 973, 974, 976, 979, 980, 982, 985, 992, 994, 995~.
This numbers obtain with command $\emph{APP2}(2,999)$, where program $\emph{APP2}$ is:
\begin{prog}\label{Program APP2} of displaying the Pseudo--Primes of Second Kind numbers.
   \begin{tabbing}
     $\emph{APP2}(a,b):=$\=\vline\ $j\leftarrow1$\\
     \>\vline\ $f$\=$\emph{or}\ n\in a..b$\\
     \>\vline\ \>\ $i$\=$f\ \emph{TS}(n)\textbf{=}0\wedge \emph{PP}(n,1)\ge1$\\
     \>\vline\ \>\ \>\vline\ $\emph{pp}_j\leftarrow n$\\
     \>\vline\ \>\ \>\vline\ $j\leftarrow j+1$\\
     \>\vline\ $\emph{return}\ \ \emph{pp}$
   \end{tabbing}
\end{prog}

\subsection{Pseudo--Primes of Third Kind}

\begin{defn}\label{Definitia Pseudo-Prime 3}
  A number is a \emph{pseudo--prime of third kind} if exist a nontrivial permutation of the digits that is a prime number.
\end{defn}

The f{}irst 429 of numbers \emph{pseudo-prime of third kinds} are: 11, 13, 14, 16, 17, 20, 30, 31, 32, 34, 35, 37, 38, 50, 70, 71, 73, 74, 76, 79, 91, 92, 95, 97, 98, 101, 103, 104, 106, 107, 109, 110, 112, 113, 115, 118, 119, 121, 124, 125, 127, 128, 130, 131, 133, 134, 136, 137, 139, 140, 142, 143, 145, 146, 149, 151, 152, 154, 157, 160, 163, 164, 166, 167, 169, 170, 172, 173, 175, 176, 179, 181, 182, 188, 190, 191, 193, 194, 196, 197, 199, 200, 203, 209, 211, 214, 215, 217, 218, 223, 227, 229, 230, 232, 233, 235, 236, 238, 239, 241, 251, 253, 271, 272, 275, 277, 278, 281, 283, 287, 289, 290, 292, 293, 296, 298, 299, 300, 301, 302, 304, 305, 307, 308, 310, 311, 313, 314, 316, 317, 319, 320, 322, 323, 325, 326, 328, 329, 331, 332, 334, 335, 337, 338, 340, 341, 343, 344, 346, 347, 349, 350, 352, 353, 356, 358, 359, 361, 362, 364, 365, 367, 368, 370, 371, 373, 374, 376, 377, 379, 380, 382, 383, 385, 386, 388, 389, 391, 392, 394, 395, 397, 398, 401, 403, 407, 410, 412, 413, 415, 416, 419, 421, 430, 433, 434, 436, 437, 439, 443, 449, 451, 457, 461, 463, 467, 470, 473, 475, 476, 478, 479, 490, 491, 493, 494, 497, 499, 500, 503, 509, 511, 512, 514, 517, 521, 527, 530, 532, 533, 536, 538, 539, 547, 557, 563, 569, 571, 572, 574, 575, 577, 578, 583, 587, 589, 590, 593, 596, 598, 599, 601, 607, 610, 613, 614, 616, 617, 619, 623, 629, 631, 632, 634, 635, 637, 638, 641, 643, 647, 653, 659, 661, 670, 671, 673, 674, 677, 679, 683, 691, 692, 695, 697, 700, 701, 703, 704, 706, 709, 710, 712, 713, 715, 716, 719, 721, 722, 725, 727, 728, 730, 731, 733, 734, 736, 737, 739, 740, 743, 745, 746, 748, 749, 751, 752, 754, 755, 757, 758, 760, 761, 763, 764, 767, 769, 772, 773, 775, 776, 778, 779, 782, 784, 785, 787, 788, 790, 791, 793, 794, 796, 797, 799, 803, 809, 811, 812, 818, 821, 823, 830, 832, 833, 835, 836, 838, 839, 847, 857, 863, 872, 874, 875, 877, 878, 881, 883, 887, 890, 892, 893, 895, 901, 902, 904, 905, 907, 908, 910, 911, 913, 914, 916, 917, 919, 920, 922, 923, 926, 928, 929, 931, 932, 934, 935, 937, 938, 940, 941, 943, 944, 947, 949, 950, 953, 956, 958, 959, 961, 962, 965, 967, 970, 971, 973, 974, 976, 977, 979, 980, 982, 983, 985, 991, 992, 994, 995, 997~.
This numbers obtain with command $\emph{APP3}(2,999)$, where program $\emph{APP3}$ is:
\begin{prog}\label{Program APP3} of displaying the Pseudo--Primes of Third Kind numbers.
   \begin{tabbing}
     $\emph{APP3}(a,b):=$\=\vline\ $j\leftarrow1$\\
     \>\vline\ $f$\=$\emph{or}\ n\in a..b$\\
     \>\vline\ \>\ $i$\=$f\ \emph{PP}(n,2)\ge1\wedge n>10$\\
     \>\vline\ \>\ \>\vline\ $\emph{pp}_j\leftarrow n$\\
     \>\vline\ \>\ \>\vline\ $j\leftarrow j+1$\\
     \>\vline\ $\emph{return}\ \ \emph{pp}$
   \end{tabbing}
\end{prog}

Questions:
\begin{enumerate}
  \item How many \emph{pseudo--primes of third kind} are prime numbers? (We conjecture: an inf{}inity).
  \item There are primes which are not \emph{pseudo--primes of third kind}, and the reverse: there are \emph{pseudo--primes of third kind} which are not primes.
\end{enumerate}

\section{Permutation--Primes}

\subsection{Permutation--Primes of type 1}

Let the permutations of 3
\[per3=\left(\begin{array}{cccccc}
               1 & 1 & 2 & 2 & 3 & 3 \\
               2 & 3 & 1 & 3 & 1 & 2 \\
               3 & 2 & 3 & 1 & 2 & 1 \\
             \end{array}\right)~.
\]
We denote $\emph{per3}_k(\overline{d_1d_2d_3})$, $k=1,2,\ldots,6$, a permutation of the number with the digits $d_1$, $d_2$, $d_3$. E.g. $\emph{per3}_2(\overline{d_1d_2d_3})=\overline{d_1d_3d_2}$. It is obvious that for a number of $m$ digits one can apply a permutation of the order $m$.

\begin{defn}\label{Definitia Permutation-Primes 1}
  We say that $n\in\Ns$ is a \emph{permutation--prime of type 1} if there exists at least a permutation for which the resulted number is prime.
\end{defn}

\begin{prog}\label{Program APP} of displaying the \emph{permutation--primes}.
  \begin{tabbing}
    $\emph{APP}(a,b,k):=$\=\vline\ $j\leftarrow1$\\
    \>\vline\ $f$\=$\emph{or}\ n\in a..b$\\
    \>\vline\ \>\vline\ $\emph{sw}\leftarrow \emph{PP}(n,1)$\\
    \>\vline\ \>\vline\ $i$\=$f\ \emph{sw}\textbf{=}k$\\
    \>\vline\ \>\vline\ \>\vline\ $\emph{pp}_j\leftarrow n$\\
    \>\vline\ \>\vline\ \>\vline\ $j\leftarrow j+1$\\
    \>\vline\ $\emph{return}\ \ \emph{pp}$\\
  \end{tabbing}
  The program using the subprogram $\emph{PP}$ given by \ref{Program PP}.
\end{prog}

There are 122 \emph{permutation--primes of type 1} from 2 to 999: 2, 3, 5, 7, 14, 16, 19, 20, 23, 29, 30, 32, 34, 35, 38, 41, 43, 47, 50, 53, 59, 61, 67, 70, 74, 76, 83, 89, 91, 92, 95, 98, 134, 143, 145, 154, 203, 209, 230, 235, 236, 253, 257, 263, 269, 275, 278, 287, 289, 290, 296, 298, 302, 304, 308, 314, 320, 325, 326, 340, 341, 352, 358, 362, 380, 385, 403, 407, 409, 413, 415, 430, 431, 451, 470, 478, 487, 490, 514, 523, 527, 532, 538, 541, 572, 583, 589, 598, 623, 629, 632, 692, 704, 725, 728, 740, 748, 752, 782, 784, 803, 827, 829, 830, 835, 847, 853, 859, 872, 874, 892, 895, 902, 904, 920, 926, 928, 940, 958, 962, 982, 985~. This numbers are obtained with the command $\emph{APP}(2,999,1)^\textrm{T}=$, where $\emph{APP}$ is the program \ref{Program APP}.

\subsection{Permutation--Primes of type 2}

\begin{defn}\label{Definitia Permutation-Primes 2}
  We say that $n\in\Ns$ is a \emph{permutation--primes of type 2} if there exists only two permutations for which the resulted numbers are primes.
\end{defn}
There are 233 \emph{permutation--primes of type 2} from 2 to 999: 11, 13, 17, 31, 37, 71, 73, 79, 97, 104, 106, 109, 112, 115, 121, 124, 125, 127, 128, 139, 140, 142, 146, 151, 152, 160, 164, 166, 169, 172, 182, 188, 190, 193, 196, 200, 211, 214, 215, 217, 218, 223, 227, 229, 232, 233, 238, 239, 241, 251, 271, 272, 281, 283, 292, 293, 299, 300, 305, 319, 322, 323, 328, 329, 332, 334, 335, 338, 343, 344, 346, 347, 349, 350, 353, 356, 364, 365, 367, 368, 374, 376, 377, 382, 383, 386, 388, 391, 392, 394, 401, 410, 412, 416, 421, 433, 434, 436, 437, 439, 443, 449, 457, 461, 463, 467, 473, 475, 476, 479, 493, 494, 497, 499, 500, 503, 509, 511, 512, 521, 530, 533, 536, 547, 557, 563, 569, 574, 575, 578, 587, 590, 596, 599, 601, 607, 610, 614, 616, 619, 634, 635, 637, 638, 641, 643, 647, 653, 659, 661, 670, 673, 674, 677, 679, 683, 691, 695, 697, 700, 706, 712, 721, 722, 734, 736, 737, 743, 745, 746, 749, 754, 755, 758, 760, 763, 764, 767, 769, 773, 776, 785, 788, 794, 796, 799, 809, 812, 818, 821, 823, 832, 833, 836, 838, 857, 863, 875, 878, 881, 883, 887, 890, 901, 905, 908, 910, 913, 916, 922, 923, 929, 931, 932, 934, 943, 944, 947, 949, 950, 956, 959, 961, 965, 967, 974, 976, 979, 980, 992, 994, 995, 997~. This numbers are obtained with the command $\emph{APP}(2,999,2)^\textrm{T}=$, where $\emph{APP}$ is the program \ref{Program APP}.

\subsection{Permutation--Primes of type 3}

\begin{defn}\label{Definitia Permutation--Primes 3}
  We say $n\in\Ns$ is a \emph{permutation--primes of type 3} if there exists only three permutations for which the resulted numbers are primes.
\end{defn}
There are 44 \emph{permutation--primes of type 3} from 2 to 999: 103, 130, 136, 137, 157, 163, 167, 173, 175, 176, 301, 307, 310, 316, 317, 359, 361, 370, 371, 389, 395, 398, 517, 539, 571, 593, 613, 617, 631, 671, 703, 713, 715, 716, 730, 731, 751, 761, 839, 893, 935, 938, 953, 983. This numbers are obtained with the command $\emph{APP}(2,999,3)^\textrm{T}=$, where $\emph{APP}$ is the program \ref{Program APP}.

\subsection{Permutation--Primes of type $m$}

\begin{defn}\label{Definitia Permutation--Primes m}
  We say $n\in\Ns$ is a \emph{permutation--primes of type $m$} if there exists only $m$ permutations for which the resulted numbers are primes.
\end{defn}

There are 49 \emph{permutation--primes of type 4} from 2 to 999: 101, 107, 110, 118, 119, 133, 149, 170, 179, 181, 191, 194, 197, 277, 313, 331, 379, 397, 419, 491, 577, 701, 709, 710, 719, 727, 739, 757, 772, 775, 778, 779, 787, 790, 791, 793, 797, 811, 877, 907, 911, 914, 917, 937, 941, 970, 971, 973, 977~. This numbers are obtained with the command $\emph{APP}(2,999,4)^\textrm{T}=$, where $\emph{APP}$ is the program \ref{Program APP}.

There are not \emph{permutation--primes of type 5} from 2 to 999, but there are 9 \emph{permutation--primes of type 6} from 2 to 999: 113, 131, 199, 311, 337, 373, 733, 919, 991. This numbers are obtained with the command $\emph{APP}(2,999,6)^\textrm{T}=$, where $\emph{APP}$ is the program \ref{Program APP}.

\section{Pseudo--Squares}

\subsection{Pseudo--Squares of First Kind}

\begin{defn}\label{Definitia Pseudo-Square 1}
  A number is a \emph{pseudo--square of f{}irst kind} if some permutation of the digits is a perfect square, including the identity permutation.
\end{defn}
Of course, all perfect squares are pseudo--squares of f{}irst kind, but not the reverse!

\begin{prog}\label{Program PSq} for counting the squares obtained by digits permutation.
  \begin{tabbing}
    $\emph{PSq}(n,i):=$\=\vline\ $m\leftarrow \emph{nrd}(n,10)$\\
    \>\vline\ $d\leftarrow dn(n,10)$\\
    \>\vline\ $\emph{np}\leftarrow 1\ \emph{if}\ m\textbf{=}1$\\
    \>\vline\ $\emph{np}\leftarrow \emph{cols(Per2)}\ \ \emph{if}\ \ m\textbf{=}2$\\
    \>\vline\ $\emph{np}\leftarrow \emph{cols(Per3)}\ \ \emph{if}\ \ m\textbf{=}3$\\
    \>\vline\ $\emph{np}\leftarrow \emph{cols(Per4)}\ \ \emph{if}\ \ m\textbf{=}4$\\
    \>\vline\ $\emph{sw}\leftarrow0$\\
    \>\vline\ $f$\=$\emph{or}\ j\in i..\emph{np}$\\
    \>\vline\ \>\vline\ $f$\=$\emph{or}\ k\in1..m$\\
    \>\vline\ \>\vline\ \>\vline\ $\emph{pd}\leftarrow d\ \emph{if}\ m\textbf{=}1$\\
    \>\vline\ \>\vline\ \>\vline\ $\emph{pd}_k\leftarrow d_{(\emph{Per2}_{k,j})}\ \ \emph{if}\ \ m\textbf{=}2$\\
    \>\vline\ \>\vline\ \>\vline\ $\emph{pd}_k\leftarrow d_{(\emph{Per3}_{k,j})}\ \ \emph{if}\ \ m\textbf{=}3$\\
    \>\vline\ \>\vline\ \>\vline\ $\emph{pd}_k\leftarrow d_{(\emph{Per4}_{k,j})}\ \ \emph{if}\ \ m\textbf{=}4$\\
    \>\vline\ \>\vline\ $\emph{nn}\leftarrow \emph{pd}\cdot\emph{Vb}(10,m)$\\
    \>\vline\ \>\vline\ $\emph{sw}\leftarrow \emph{sw}+1\ \ \emph{if}\ \ \emph{isp}(\emph{nn})\textbf{=}\emph{nn}$\\
    \>\vline\ $\emph{return}\ \ \emph{sw}$\\
  \end{tabbing}
The program uses the subprograms: $\emph{nrd}$ given by \ref{FunctionNrd}, $\emph{dn}$ given by \ref{ProgramDn}, $\emph{Vb}(b,m)$ which returns the vector $(b^m\ b^{m-1}\ \ldots b^0)^\textrm{T}$ and $\emph{isp}$ def{}ined at \ref{Functia isp}. Also, the program calls the matrices $\emph{Per2}$, $\emph{Per3}$ and $\emph{Per4}$ which contains all permutations
of sets $\set{1,2}$, $\set{1,2,3}$ and $\set{1,2,3,4}$.
\end{prog}

\begin{prog}\label{Program APSq1} of displaying the \emph{pseudo--squares of f{}irst kind}.
   \begin{tabbing}
     $APSq1(a,b):=$\=\vline\ $j\leftarrow1$\\
     \>\vline\ $f$\=$or\ n\in a..b$\\
     \>\vline\ \>\ $i$\=$f\ PSq(n,1)\ge1$\\
     \>\vline\ \>\ \>\vline\ $psq_j\leftarrow n$\\
     \>\vline\ \>\ \>\vline\ $j\leftarrow j+1$\\
     \>\vline\ $\emph{return}\ \ psq$
   \end{tabbing}
\end{prog}

One listed all (are 121) \emph{pseudo--squares of f{}irst kind} up to 1000: 1, 4, 9, 10, 16, 18, 25, 36, 40, 46, 49, 52, 61, 63, 64, 81, 90, 94, 100, 106, 108, 112, 121, 136, 144, 148, 160, 163,169, 180, 184, 196, 205, 211, 225, 234, 243, 250, 252, 256, 259, 265, 279, 289, 295, 297, 298, 306, 316, 324 ,342, 360, 361, 400, 406, 409, 414, 418, 423, 432, 441, 448, 460, 478, 481, 484, 487, 490, 502, 520, 522, 526, 529, 562, 567, 576, 592, 601, 603, 604, 610, 613, 619, 625, 630, 631, 640, 652, 657, 667, 675, 676, 691, 729, 748, 756, 765, 766, 784, 792, 801, 810, 814, 829, 841, 844, 847, 874, 892, 900, 904, 916, 925, 927, 928, 940, 952, 961, 972, 982, 1000.  This numbers are obtained with the command $\emph{APSq1}(1,10^3)^\textrm{T}=$, where $\emph{APSq1}$ is the program \ref{Program APSq1}.

\subsection{Pseudo--Squares of Second Kind}

\begin{defn}\label{Definitia Pseudo-Square 2}
  A non--square number is a \emph{pseudo--squares of second kind} if exist a permutation of the digits is a square.
\end{defn}

\begin{prog}\label{Program APSq2} of displaying \emph{pseudo--squares of second kind}.
   \begin{tabbing}
     $\emph{APSq2}(a,b):=$\=\vline\ $j\leftarrow1$\\
     \>\vline\ $f$\=$\emph{or}\ n\in a..b$\\
     \>\vline\ \>\ $i$\=$f\ \emph{PSq}(n,1)\ge1\wedge\emph{isp}(n)\neq n$\\
     \>\vline\ \>\ \>\vline\ $\emph{psq}_j\leftarrow n$\\
     \>\vline\ \>\ \>\vline\ $j\leftarrow j+1$\\
     \>\vline\ $\emph{return}\ \ \emph{psq}$
   \end{tabbing}
\end{prog}

Let us list all (there are 90) \emph{pseudo--squares of second kind} up to 1000: 10, 18, 40, 46, 52, 61, 63, 90, 94, 106, 108, 112, 136, 148, 160, 163, 180, 184, 205, 211, 234, 243, 250, 252, 259, 265, 279, 295, 297, 298, 306, 316, 342, 360, 406, 409, 414, 418, 423, 432, 448, 460, 478, 481, 487, 490, 502, 520, 522, 526, 562, 567, 592, 601, 603, 604, 610, 613, 619, 630, 631, 640, 652, 657, 667, 675, 691, 748, 756, 765, 766, 792, 801, 810, 814, 829, 844, 847, 874, 892, 904, 916, 925, 927, 928, 940, 952, 972, 982, 1000~. This numbers are obtained with the command $\emph{APSq2}(1,10^3)^\textrm{T}=$, where $\emph{APSq2}$ is the program \ref{Program APSq2}.

\subsection{Pseudo--Squares of Third Kind}

\begin{defn}\label{Definitia Pseudo-Square 3}
  A number is a \emph{pseudo--square of third kind} if exist a nontrivial permutation of the digits is a square.
\end{defn}

\begin{prog}\label{Program APSq3} of displaying \emph{pseudo--squares of third kind}.
   \begin{tabbing}
     $\emph{APSq3}(a,b):=$\=\vline\ $j\leftarrow1$\\
     \>\vline\ $f$\=$\emph{or}\ n\in a..b$\\
     \>\vline\ \>\ $i$\=$f\ \emph{PSq}(n,2)\ge1\wedge n>9$\\
     \>\vline\ \>\ \>\vline\ $\emph{psq}_j\leftarrow n$\\
     \>\vline\ \>\ \>\vline\ $j\leftarrow j+1$\\
     \>\vline\ $\emph{return}\ \ \emph{psq}$
   \end{tabbing}
\end{prog}

Let us list all (there are 104) \emph{pseudo--squares of third kind} up to 1000: 10, 18, 40, 46, 52, 61, 63, 90, 94, 100, 106, 108, 112, 121, 136, 144, 148, 160, 163, 169, 180, 184, 196, 205, 211, 225, 234, 243, 250, 252, 256, 259, 265, 279, 295, 297, 298, 306, 316, 342, 360, 400, 406, 409, 414, 418, 423, 432, 441, 448, 460, 478, 481, 484, 487, 490, 502, 520, 522, 526, 562, 567, 592, 601, 603, 604, 610, 613, 619, 625, 630, 631, 640, 652, 657, 667, 675, 676, 691, 748, 756, 765, 766, 792, 801, 810, 814, 829, 844, 847, 874, 892, 900, 904, 916, 925, 927, 928, 940, 952, 961, 972, 982, 1000~. This numbers are obtained with the command $\emph{APSq3}(1,10^3)^\textrm{T}=$, where $\emph{APSq3}$ is the program \ref{Program APSq3}.

Question:
\begin{enumerate}
  \item How many pseudo--squares of third kind are square numbers? We conjecture: an inf{}inity.
  \item There are squares which are not pseudo--squares of third kind, and the reverse: there are pseudo--squares of third kind which are not squares.
\end{enumerate}

\section{Pseudo--Cubes}

\subsection{Pseudo--Cubes of First Kind}

\begin{defn}\label{Definitia Pseudo-Cube 1}
  A number is a \emph{pseudo--cube of f{}irst kind} if some permutation of the digits is a cube, including the identity permutation.
\end{defn}

Of course, all perfect cubes are \emph{pseudo--cubes of f{}irst kind}, but not the reverse!

With programs similar to $\emph{PSq}$, \ref{Program PSq}, $\emph{APSq1}$, \ref{Program APSq1}, $\emph{APSq2}$, \ref{Program APSq1} and $\emph{APSq3}$, \ref{Program APSq3} can list the pseudo--cube numbers.

Let us list all (there are 40) \emph{pseudo--cubes of f{}irst kind} up to 1000: 1, 8, 10, 27, 46, 64, 72, 80, 100, 125, 126, 152, 162, 207, 215, 216, 251, 261, 270, 279, 297, 334, 343, 406, 433, 460, 512, 521, 604, 612, 621, 640, 702, 720, 729, 792, 800, 927, 972, 1000~.

\subsection{Pseudo--Cubes of Second Kind}

\begin{defn}\label{Definitia Pseudo-Cube 2}
  A non--cube number is a \emph{pseudo--cube of second kind} if some permutation of the digits is a cube.
\end{defn}
Let us list all (there are 30) pseudo--cubes of second kind up to 1000: 10, 46, 72, 80, 100, 126, 152, 162, 207, 215, 251, 261, 270, 279, 297, 334, 406, 433, 460, 521, 604, 612, 621, 640, 702, 720, 792, 800, 927, 972~.

\subsection{Pseudo--Cubes of Third Kind}

\begin{defn}\label{Definitia Pseudo-Cube 3}
  A number is a \emph{pseudo--cube of third kind} if exist a nontrivial permutation of the digits is a cube.
\end{defn}

Let us list all (there are 34) \emph{ pseudo--cubes of third kind} up to 1000: 10, 46, 72, 80, 100, 125, 126, 152, 162, 207, 215, 251, 261, 270, 279, 297, 334, 343, 406, 433, 460, 512, 521, 604, 612, 621, 640, 702, 720, 792, 800, 927, 972, 1000~.

Question:
\begin{enumerate}
  \item How many pseudo--cubes of third kind are cubes? We conjecture: an inf{}inity.
  \item There are cubes which are not pseudo--cubes of third kind, and the reverse: there are pseudo--cubes of third kind which are not cubes.
\end{enumerate}

\section{Pseudo--$m$--Powers}

\subsection{Pseudo--$m$--Powers of First Kind}

\begin{defn}
  A number is a \emph{pseudo--$m$--power of f{}irst kind} if exist a permutation of the digits is an $m$--power, including the identity permutation; $m\ge 2$.
\end{defn}

\subsection{Pseudo--$m$--Powers of Second kind}

\begin{defn}
  A non $m$--power number is a \emph{pseudo--$m$--power of second kind} if exist a permutation of the digits is an $m$--power; $m\ge2$.
\end{defn}

\subsection{Pseudo--$m$--Powers of Third Kind}

\begin{defn}
  A number is a \emph{pseudo--$m$--power of third kind} if exist a nontrivial permutation of the digits is an $m$--power; $m\ge2$.
\end{defn}

Question:
\begin{enumerate}
  \item How many pseudo--$m$--powers of third kind are $m$--power numbers? We conjecture: an inf{}inity.
  \item There are $m$--powers which are not \emph{pseudo--$m$--powers of third kind}, and the reverse: there are \emph{pseudo--$m$--powers of third kind} which are not $m$--powers.
\end{enumerate}

\section{Pseudo--Factorials}

\subsection{Pseudo--Factorials of First Kind}

\begin{defn}\label{Definitia Peseudo-Factorial 1}
  A number is a \emph{pseudo-factorial of f{}irst kind} if exist a permutation of the digits is a factorial number, including the identity permutation.
\end{defn}

One listed all \emph{pseudo--factorials of f{}irst kind} up to 1000: 1, 2, 6, 10, 20, 24, 42, 60, 100, 102, 120, 200, 201, 204, 207, 210, 240, 270, 402, 420, 600, 702, 720, 1000, 1002, 1020, 1200, 2000, 2001, 2004, 2007, 2010, 2040, 2070, 2100, 2400, 2700, 4002, 4005, 4020, 4050, 4200, 4500, 5004, 5040, 5400, 6000, 7002, 7020, 7200~. In this list there are 37 numbers.

\subsection{Pseudo--Factorials of Second Kind}

\begin{defn}\label{Definitia Pseudo-Factorial 2}
  A non--factorial number is a \emph{pseudo--factorial of second kind} if exist a permutation of the digits is a factorial number.
\end{defn}

One listed all \emph{pseudo--factorials of second kind} up to 1000: 10, 20, 42, 60, 100, 102, 200, 201, 204, 207, 210, 240, 270, 402, 420, 600, 702, 1000, 1002, 1020, 1200, 2000,2001,2004, 2007, 2010, 2040, 2070, 2100, 2400, 2700, 4002, 4005, 4020, 4050, 4200, 4500, 5004, 5400, 6000, 7002, 7020, 7200~. In this list there are 31 numbers.

\subsection{Pseudo--Factorials of Third Kind}

\begin{defn}
  A number is a \emph{pseudo--factorial of third kind} if exist nontrivial permutation of the digits is a factorial number.
\end{defn}

One listed all \emph{pseudo--factorials of third kind} up to 1000: 10, 20, 42, 60, 100, 102, 200, 201, 204, 207, 210, 240, 270, 402, 420, 600, 702, 1000, 1002, 1020, 1200, 2000,2001,2004, 2007, 2010, 2040, 2070, 2100, 2400, 2700, 4002, 4005, 4020, 4050, 4200, 4500, 5004, 5400, 6000, 7002, 7020, 7200~. In this list there are 31 numbers.

Unfortunately, the second and third kinds of pseudo--factorials coincide.

Question:
\begin{enumerate}
  \item How many \emph{pseudo--factorials of third kind} are factorial numbers?
  \item We conjectured: none! \ldots that means the\emph{pseudo--factorials of second kind} set and \emph{pseudo--factorials of third kind} set coincide!
\end{enumerate}

\section{Pseudo--Divisors}

\subsection{Pseudo--Divisors of First Kind}

\begin{defn}
  A number is a \emph{pseudo--divisor of f{}irst kind} of $n$ if exist a permutation of the digits is a divisor of $n$, including the identity permutation.
\end{defn}
\begin{center}
 \begin{longtable}{|r|l|}
   \caption{Pseudo--divisor of f{}irst kind of $n\le12$}\\
   \hline
   $n$ &  \emph{pseudo--divisors} $<1000$ of $n$ \\
   \hline
  \endfirsthead
   \hline
   $n$ &  \emph{pseudo--divisors} $<1000$ of $n$ \\
   \hline
  \endhead
   \hline \multicolumn{2}{r}{\textit{Continued on next page}} \\
  \endfoot
   \hline
  \endlastfoot
  1 & 1, 10, 100 \\
  2 & 1, 2, 10, 20, 100, 200  \\
  3 & 1 ,3, 10, 30, 100, 300 \\
  4 & 1, 2, 4, 10, 20, 40, 100, 200, 400  \\
  5 & 1, 5, 10, 50, 100, 500  \\
  6 & 1, 2, 3, 6, 10, 20, 30, 60, 100, 200, 300, 600 \\
  7 & 1, 7, 10, 70, 100, 700 \\
  8 & 1, 2, 4, 8, 10, 20, 40, 80, 100, 200, 400, 800 \\
  9 & 1, 3, 9, 10, 30, 90, 100, 300, 900 \\
  10 & 1, 2, 5, 10, 20, 50, 100, 200, 500 \\
  11 & 1, 11, 101, 110 \\
  12 & 1, 2, 3, 4, 6, 10, 12, 20, 30, 40, 60, 100, 120, 200, 300, 400, 600\\
  \hline
\end{longtable}
\end{center}

\subsection{Pseudo--Divisors of Second Kind}

\begin{defn}
  A non--divisor of $n$ is a \emph{pseudo--divisor of second kind} of $n$ if exist a permutation of the digits is a divisor of $n$.
\end{defn}
\begin{center}
 \begin{longtable}{|r|l|}
   \caption{Pseudo--divisor of second kind of $n\le12$}\\
   \hline
   $n$ &  \emph{pseudo--divisors} $<1000$ of $n$ \\
   \hline
  \endfirsthead
   \hline
   $n$ &  \emph{pseudo--divisors} $<1000$ of $n$ \\
   \hline
  \endhead
   \hline \multicolumn{2}{r}{\textit{Continued on next page}} \\
  \endfoot
   \hline
  \endlastfoot
  1 & 10, 100 \\
  2 & 10, 20, 100, 200  \\
  3 & 10, 30, 100, 300 \\
  4 & 10, 20, 40, 100, 200, 400  \\
  5 & 10, 50, 100, 500  \\
  6 & 10, 20, 30, 60, 100, 200, 300, 600 \\
  7 & 10, 70, 100, 700 \\
  8 & 10, 20, 40, 80, 100, 200, 400, 800 \\
  9 & 10, 30, 90, 100, 300, 900 \\
  10 & 10, 20, 50, 100, 200, 500 \\
  11 & 101, 110 \\
  12 & 10, 20, 30, 40, 60, 100, 120, 200, 300, 400, 600\\
  \hline
\end{longtable}
\end{center}

\subsection{Pseudo--Divisors of Third Kind}

\begin{defn}
  A number is a \emph{pseudo--divisor of third kind} of $n$ if exist a nontrivial permutation of the digits is a divisor of $n$.
\end{defn}
\begin{center}
 \begin{longtable}{|r|l|}
   \caption{Pseudo--divisor of third kind of $n\le12$}\\
   \hline
   $n$ &  \emph{pseudo--divisors} $<1000$ of $n$ \\
   \hline
  \endfirsthead
   \hline
   $n$ &  \emph{pseudo--divisors} $<1000$ of $n$ \\
   \hline
  \endhead
   \hline \multicolumn{2}{r}{\textit{Continued on next page}} \\
  \endfoot
   \hline
  \endlastfoot
  1 & 10, 100 \\
  2 & 10, 20, 100, 200  \\
  3 & 10, 30, 100, 300 \\
  4 & 10, 20, 40, 100, 200, 400  \\
  5 & 10, 50, 100, 500  \\
  6 & 10, 20, 30, 60, 100, 200, 300, 600 \\
  7 & 10, 70, 100, 700 \\
  8 & 10, 20, 40, 80, 100, 200, 400, 800 \\
  9 & 10, 30, 90, 100, 300, 900 \\
  10 & 10, 20, 50, 100, 200, 500 \\
  11 & 101, 110 \\
  12 & 10, 20, 30, 40, 60, 100, 120, 200, 300, 400, 600\\
  \hline
\end{longtable}
\end{center}

\section{Pseudo--Odd Numbers}

\begin{prog}\label{Program Po} of counting the odd numbers obtained by digits permutation of the number.
  \begin{tabbing}
    $\emph{Po}(n,i):=$\=\vline\ $m\leftarrow \emph{nrd}(n,10)$\\
    \>\vline\ $d\leftarrow \emph{dn}(n,10)$\\
    \>\vline\ $\emph{np}\leftarrow 1\ \ \emph{if}\ \ m\textbf{=}1$\\
    \>\vline\ $\emph{np}\leftarrow \emph{cols(Per2)}\ \ \emph{if}\ \ m\textbf{=}2$\\
    \>\vline\ $\emph{np}\leftarrow \emph{cols(Per3)}\ \ \emph{if}\ \ m\textbf{=}3$\\
    \>\vline\ $\emph{np}\leftarrow \emph{cols(Per4)}\ \ \emph{if}\ \ m\textbf{=}4$\\
    \>\vline\ $\emph{sw}\leftarrow0$\\
    \>\vline\ $f$\=$\emph{or}\ j\in i..np$\\
    \>\vline\ \>\vline\ $f$\=$\emph{or}\ k\in1..m$\\
    \>\vline\ \>\vline\ \>\vline\ $\emph{pd}\leftarrow d\ \emph{if}\ m\textbf{=}1$\\
    \>\vline\ \>\vline\ \>\vline\ $\emph{pd}_k\leftarrow d_{(\emph{Per2}_{k,j})}\ \ \emph{if}\ \ m\textbf{=}2$\\
    \>\vline\ \>\vline\ \>\vline\ $\emph{pd}_k\leftarrow d_{(\emph{Per3}_{k,j})}\ \ \emph{if}\ \ m\textbf{=}3$\\
    \>\vline\ \>\vline\ \>\vline\ $\emph{pd}_k\leftarrow d_{(\emph{Per4}_{k,j})}\ \ \emph{if}\ \ m\textbf{=}4$\\
    \>\vline\ \>\vline\ $\emph{nn}\leftarrow \emph{pd}\cdot\emph{Vb}(10,m)$\\
    \>\vline\ \>\vline\ $\emph{sw}\leftarrow \emph{sw}+1\ \ \emph{if}\ \ \mod(\emph{nn},2)\textbf{=}1$\\
    \>\vline\ $\emph{return}\ \ \emph{sw}$\\
  \end{tabbing}
  The program uses the matrices $\emph{Per2}$ (\ref{Per2}), $\emph{Per3}$ (\ref{Per3}) and $\emph{Per4}$ (\ref{Per4}) which contains all the permutation of the sets $set{1,2}$, $\set{1,2,3}$ and $\set{1,2,3,4}$.
\end{prog}

\subsection{Pseudo--Odd Numbers of First Kind}

\begin{defn}
   A number is a \emph{pseudo--odd of f{}irst kind} if exist a permutation of digits is an odd number.
\end{defn}

\begin{prog}\label{Program APo1} of displaying the \emph{pseudo--odd of f{}irst kind}.
   \begin{tabbing}
     $\emph{APo1}(a,b):=$\=\vline\ $j\leftarrow1$\\
     \>\vline\ $f$\=$\emph{or}\ n\in a..b$\\
     \>\vline\ \>\ $i$\=$f\ \emph{Po}(n,1)\ge1$\\
     \>\vline\ \>\ \>\vline\ $\emph{po}_j\leftarrow n$\\
     \>\vline\ \>\ \>\vline\ $j\leftarrow j+1$\\
     \>\vline\ $\emph{return}\ \ \emph{po}$
   \end{tabbing}
   This program calls the program $\emph{Po}$, \ref{Program Po}.
\end{prog}

\emph{Pseudo--odd numbers of f{}irst kind} up to 199 are 175: 1, 3, 5, 7, 9, 10, 11, 12, 13, 14, 15, 16, 17, 18, 19, 21, 23, 25, 27, 29, 30, 31, 32, 33, 34, 35, 36, 37, 38, 39, 41, 43, 45, 47, 49, 50, 51, 52, 53, 54, 55, 56, 57, 58, 59, 61, 63, 65, 67, 69, 70, 71, 72, 73, 74, 75, 76, 77, 78, 79, 81, 83, 85, 87, 89, 90, 91, 92, 93, 94, 95, 96, 97, 98, 99, 100, 101, 102, 103, 104, 105, 106, 107, 108, 109, 110, 111, 112, 113, 114, 115, 116, 117, 118, 119, 120, 121, 122, 123, 124, 125, 126, 127, 128, 129, 130, 131, 132, 133, 134, 135, 136, 137, 138, 139, 140, 141, 142, 143, 144, 145, 146, 147, 148, 149, 150, 151, 152, 153, 154, 155, 156, 157, 158, 159, 160, 161, 162, 163, 164, 165, 166, 167, 168, 169, 170, 171, 172, 173, 174, 175, 176, 177, 178, 179, 180, 181, 182, 183, 184, 185, 186, 187, 188, 189, 190, 191, 192, 193, 194, 195, 196, 197, 198, 199~.

\subsection{Pseudo--Odd Numbers of Second Kind}

\begin{defn}
  Even numbers such that exist a permutation of digits is an odd number.
\end{defn}

\begin{prog}\label{Program APo2} of displaying the \emph{pseudo--odd of second kind}.
   \begin{tabbing}
     $\emph{APo2}(a,b):=$\=\vline\ $j\leftarrow1$\\
     \>\vline\ $f$\=$\emph{or}\ n\in a..b$\\
     \>\vline\ \>\ $i$\=$f\ \emph{Po}(n,1)\ge1\wedge \mod(n,2)\textbf{=}0$\\
     \>\vline\ \>\ \>\vline\ $\emph{po}_j\leftarrow n$\\
     \>\vline\ \>\ \>\vline\ $j\leftarrow j+1$\\
     \>\vline\ $\emph{return}\ \emph{po}$
   \end{tabbing}
   This program calls the program $\emph{Po}$, \ref{Program Po}.
\end{prog}

\emph{Pseudo--odd numbers of second kind} up to 199 are 75: 10, 12, 14, 16, 18, 30, 32, 34, 36, 38, 50, 52, 54, 56, 58, 70, 72, 74, 76, 78, 90, 92, 94, 96, 98, 100, 102, 104, 106, 108, 110, 112, 114, 116, 118, 120, 122, 124, 126, 128, 130, 132, 134, 136, 138, 140, 142, 144, 146, 148, 150, 152, 154, 156, 158, 160, 162, 164, 166, 168, 170, 172, 174, 176, 178, 180, 182, 184, 186, 188, 190, 192, 194, 196, 198~.

\subsection{Pseudo--Odd Numbers of Third Kind}

\begin{defn}
  A number is a \emph{pseudo--odd of third kind} if exist a nontrivial permutation of digits is an odd.
\end{defn}

\begin{prog}\label{Program APo3} of displaying the \emph{pseudo--odd of third kind}.
   \begin{tabbing}
     $\emph{APo3}(a,b):=$\=\vline\ $j\leftarrow1$\\
     \>\vline\ $f$\=$\emph{or}\ n\in a..b$\\
     \>\vline\ \>\ $i$\=$f\ \ \emph{Po}(n,2)\ge1\wedge n>9$\\
     \>\vline\ \>\ \>\vline\ $\emph{po}_j\leftarrow n$\\
     \>\vline\ \>\ \>\vline\ $j\leftarrow j+1$\\
     \>\vline\ $\emph{return}\ \ \emph{po}$
   \end{tabbing}
   This program calls the program $\emph{Po}$, \ref{Program Po}.
\end{prog}

\emph{Pseudo--odd numbers of third kind} up to 199 are 150: 10, 11, 12, 13, 14, 15, 16, 17, 18, 19, 30, 31, 32, 33, 34, 35, 36, 37, 38, 39, 50, 51, 52, 53, 54, 55, 56, 57, 58, 59, 70, 71, 72, 73, 74, 75, 76, 77, 78, 79, 90, 91, 92, 93, 94, 95, 96, 97, 98, 99, 100, 101, 102, 103, 104, 105, 106, 107, 108, 109, 110, 111, 112, 113, 114, 115, 116, 117, 118, 119, 120, 121, 122, 123, 124, 125, 126, 127, 128, 129, 130, 131, 132, 133, 134, 135, 136, 137, 138, 139, 140, 141, 142, 143, 144, 145, 146, 147, 148, 149, 150, 151, 152, 153, 154, 155, 156, 157, 158, 159, 160, 161, 162, 163, 164, 165, 166, 167, 168, 169, 170, 171, 172, 173, 174, 175, 176, 177, 178, 179, 180, 181, 182, 183, 184, 185, 186, 187, 188, 189, 190, 191, 192, 193, 194, 195, 196, 197, 198, 199~.

\section{Pseudo--Triangular Numbers}

A triangular number has the general form $n(n+1)/2$. The list f{}irst 44 triangular numbers is: $t^\textrm{T}=$(1, 3, 6, 10, 15, 21, 28, 36, 45, 55, 66, 78, 91, 105, 120, 136, 153, 171, 190, 210, 231, 253, 276, 300, 325, 351, 378, 406, 435, 465, 496, 528, 561, 595, 630, 666, 703, 741, 780, 820, 861, 903, 946, 990)~.
\begin{prog}\label{Program IT} for determining if $n$ is a triangular number or not.
  \begin{tabbing}
    $\emph{IT}(n):=$\=\vline\ $f$\=$\emph{or}\ k\in1..\emph{last}(t)$\\
    \>\vline\ \>\vline\ $\emph{return}\ \ 0\ \ \emph{if}\ t_k>n$\\
    \>\vline\ \>\vline\ $\emph{return}\ \ 1\ \ \emph{if}\ t_k=n$\\
    \>\vline\ $\emph{return}\ \ 0$\\
  \end{tabbing}
\end{prog}

\begin{prog}\label{Program PT} for counting the triangular numbers obtained by digits permutation of the number.
  \begin{tabbing}
    $\emph{PT}(n,i):=$\=\vline\ $m\leftarrow\emph{nrd}(n,10)$\\
    \>\vline\ $d\leftarrow \emph{dn}(n,10)$\\
    \>\vline\ $\emph{np}\leftarrow 1\ \emph{if}\ \ m\textbf{=}1$\\
    \>\vline\ $\emph{np}\leftarrow \emph{cols(Per2)}\ \ \emph{if}\ \ m\textbf{=}2$\\
    \>\vline\ $\emph{np}\leftarrow \emph{cols(Per3)}\ \ \emph{if}\ \ m\textbf{=}3$\\
    \>\vline\ $\emph{np}\leftarrow \emph{cols(Per4)}\ \ \emph{if}\ \ m\textbf{=}4$\\
    \>\vline\ $\emph{sw}\leftarrow0$\\
    \>\vline\ $f$\=$\emph{or}\ j\in i..\emph{np}$\\
    \>\vline\ \>\vline\ $f$\=$\emph{or}\ k\in1..m$\\
    \>\vline\ \>\vline\ \>\vline\ $\emph{pd}\leftarrow d\ \emph{if}\ m\textbf{=}1$\\
    \>\vline\ \>\vline\ \>\vline\ $\emph{pd}_k\leftarrow d_{(\emph{Per2}_{k,j})}\ \ \emph{if}\ \ m\textbf{=}2$\\
    \>\vline\ \>\vline\ \>\vline\ $\emph{pd}_k\leftarrow d_{(\emph{Per3}_{k,j})}\ \ \emph{if}\ \ m\textbf{=}3$\\
    \>\vline\ \>\vline\ \>\vline\ $\emph{pd}_k\leftarrow d_{(\emph{Per4}_{k,j})}\ \ \emph{if}\ \ m\textbf{=}4$\\
    \>\vline\ \>\vline\ $\emph{nn}\leftarrow \emph{pd}\cdot\emph{Vb}(10,m)$\\
    \>\vline\ \>\vline\ $\emph{sw}\leftarrow \emph{sw}+1\ \ \emph{if}\ \ \emph{IT(nn)}\textbf{=}1$\\
    \>\vline\ $\emph{return}\ \ \emph{sw}$\\
  \end{tabbing}
\end{prog}

\subsection{Pseudo--Triangular Numbers of First Kind}

\begin{defn}
  A number is a \emph{pseudo--triangular of f{}irst kind} if exist a permutation of digits is a triangular number.
\end{defn}

\begin{prog}\label{Program APT1} for displaying the \emph{pseudo--triangular of f{}irst kind}.
   \begin{tabbing}
     $\emph{APT1}(a,b):=$\=\vline\ $j\leftarrow1$\\
     \>\vline\ $f$\=$\emph{or}\ n\in a..b$\\
     \>\vline\ \>\ $i$\=$f\ \ \emph{PT(n,1)}\ge1$\\
     \>\vline\ \>\ \>\vline\ $\emph{pt}_j\leftarrow n$\\
     \>\vline\ \>\ \>\vline\ $j\leftarrow j+1$\\
     \>\vline\ $\emph{return}\ \ \emph{pt}$
   \end{tabbing}
   The program calls the program $\emph{PT}$, \ref{Program PT}.
\end{prog}

\emph{Pseudo--triangular numbers of f{}irst kind} up to 999 are 156: 1, 3, 6, 10, 12, 15, 19, 21, 28, 30, 36, 45, 51, 54, 55, 60, 63, 66, 78, 82, 87, 91, 100, 102, 105, 109, 117, 120, 123, 132, 135, 136, 147, 150, 153, 156, 163, 165, 168, 171, 174, 186, 190, 201, 208, 210, 213, 231, 235, 253, 258, 267, 276, 280, 285, 300, 306, 307, 309, 312, 315, 316, 321, 325, 345, 351, 352, 354, 360, 361, 370, 378, 387, 390, 405, 406, 417, 435, 450, 453, 456, 460, 465, 469, 471, 496, 501, 504, 505, 510, 513, 516, 523, 528, 531, 532, 534, 540, 543, 546, 550, 559, 561, 564, 582, 595, 600, 603, 604, 606, 613, 615, 618, 627, 630, 631, 640, 645, 649, 651, 654, 660, 666, 672, 681, 694, 703, 708, 711, 714, 726, 730, 738, 741, 762, 780, 783, 802, 807, 816, 820, 825, 837, 852, 861, 870, 873, 901, 903, 909, 910, 930, 946, 955, 964, 990~. This numbers obtain with the command $\emph{APT1}(1,999)^\textrm{T}=$, where $\emph{APT1}$ is the program \ref{Program APT1}.

\subsection{Pseudo--Triangular Numbers of Second Kind}

\begin{defn}
  A non--triangular number is a \emph{pseudo--triangular of second kind} if exist a permutation of the digits is a triangular number.
\end{defn}

\begin{prog}\label{Program APT2} for displaying the \emph{pseudo--triangular of second kind}.
   \begin{tabbing}
     $\emph{APT2}(a,b):=$\=\vline\ $j\leftarrow1$\\
     \>\vline\ $f$\=$\emph{or}\ n\in a..b$\\
     \>\vline\ \>\ $i$\=$f\ \ \emph{IT(n)}\textbf{=}0\wedge \emph{PT(n,1)}\ge1$\\
     \>\vline\ \>\ \>\vline\ $\emph{pt}_j\leftarrow n$\\
     \>\vline\ \>\ \>\vline\ $j\leftarrow j+1$\\
     \>\vline\ $\emph{return}\ \ \emph{pt}$
   \end{tabbing}
   The program calls the programs $\emph{IT}$, \ref{Program IT} and $\emph{PT}$, \ref{Program PT}.
\end{prog}

\emph{Pseudo--triangular numbers of second kind} up to 999 are 112: 12, 19, 30, 51, 54, 60, 63, 82, 87, 100, 102, 109, 117, 123, 132, 135, 147, 150, 156, 163, 165, 168, 174, 186, 201, 208, 213, 235, 258, 267, 280, 285, 306, 307, 309, 312, 315, 316, 321, 345, 352, 354, 360, 361, 370, 387, 390, 405, 417, 450, 453, 456, 460, 469, 471, 501, 504, 505, 510, 513, 516, 523, 531, 532, 534, 540, 543, 546, 550, 559, 564, 582, 600, 603, 604, 606, 613, 615, 618, 627, 631, 640, 645, 649, 651, 654, 660, 672, 681, 694, 708, 711, 714, 726, 730, 738, 762, 783, 802, 807, 816, 825, 837, 852, 870, 873, 901, 909, 910, 930, 955, 964~. This numbers are obtained with the command $\emph{APT2}(1,999)^\textrm{T}=$, where $\emph{APT2}$ is the program \ref{Program APT2}.

\subsection{Pseudo--Triangular Numbers of Third Kind}

\begin{defn}
   A number is a \emph{pseudo--triangular of third kind} if exist a nontrivial permutation of the digits is a triangular number.
\end{defn}

\begin{prog}\label{Program APT3} for displaying the \emph{pseudo--triangular of third kind}.
   \begin{tabbing}
     $\emph{APT3}(a,b):=$\=\vline\ $j\leftarrow1$\\
     \>\vline\ $f$\=$\emph{or}\ n\in a..b$\\
     \>\vline\ \>\ $i$\=$f\ \ \emph{PT(n,2)}\ge1\wedge n>9$\\
     \>\vline\ \>\ \>\vline\ $pt_j\leftarrow n$\\
     \>\vline\ \>\ \>\vline\ $j\leftarrow j+1$\\
     \>\vline\ $\emph{return}\ \ \emph{pt}$
   \end{tabbing}
   The program calls the program $\emph{PT}$, \ref{Program PT}.
\end{prog}

\emph{Pseudo--triangular numbers of third kind} up to 999 are 133: 10, 12, 19, 30, 51, 54, 55, 60, 63, 66, 82, 87, 100, 102, 105, 109, 117, 120, 123, 132, 135, 147, 150, 153, 156, 163, 165, 168, 171, 174, 186, 190, 201, 208, 210, 213, 235, 253, 258, 267, 280, 285, 300, 306, 307, 309, 312, 315, 316, 321, 325, 345, 351, 352, 354, 360, 361, 370, 387, 390, 405, 417, 450, 453, 456, 460, 469, 471, 496, 501, 504, 505, 510, 513, 516, 523, 531, 532, 534, 540, 543, 546, 550, 559, 564, 582, 595, 600, 603, 604, 606, 613, 615, 618, 627, 630, 631, 640, 645, 649, 651, 654, 660, 666, 672, 681, 694, 708, 711, 714, 726, 730, 738, 762, 780, 783, 802, 807, 816, 820, 825, 837, 852, 870, 873, 901, 909, 910, 930, 946, 955, 964, 990~. This numbers are obtained with the command $\emph{APT3}(1,999)^\textrm{T}=$, where $\emph{APT3}$ is the program \ref{Program APT3}.

\section{Pseudo--Even Numbers}

With similar programs with programs $\emph{Po}$, \ref{Program Po}, $\emph{Apo1}$, \ref{Program APo1}, $\emph{APo2}$, \ref{Program APo2} and $\emph{APo3}$, \ref{Program APo3} can get \emph{pseudo--even numbers}.

\subsection{Pseudo--even Numbers of First Kind}

\begin{defn}
   A number is a \emph{pseudo--even of f{}irst kind} if exist a permutation of digits is an even number.
\end{defn}

\emph{Pseudo--even numbers of f{}irst kind} up to 199 are 144: 2, 4, 6, 8, 10, 12, 14, 16, 18, 20, 21, 22, 23, 24, 25, 26, 27, 28, 29, 30, 32, 34, 36, 38, 40, 41, 42, 43, 44, 45, 46, 47, 48, 49, 50, 52, 54, 56, 58, 60, 61, 62, 63, 64, 65, 66, 67, 68, 69, 70, 72, 74, 76, 78, 80, 81, 82, 83, 84, 85, 86, 87, 88, 89, 90, 92, 94, 96, 98, 100, 101, 102, 103, 104, 105, 106, 107, 108, 109, 110, 112, 114, 116, 118, 120, 121, 122, 123, 124, 125, 126, 127, 128, 129, 130, 132, 134, 136, 138, 140, 141, 142, 143, 144, 145, 146, 147, 148, 149, 150, 152, 154, 156, 158, 160, 161, 162, 163, 164, 165, 166, 167, 168, 169, 170, 172, 174, 176, 178, 180, 181, 182, 183, 184, 185, 186, 187, 188, 189, 190, 192, 194, 196, 198.

\subsection{Pseudo--Even Numbers of Second Kind}

\begin{defn}
  Odd numbers such that exist a permutation of digits is an even number.
\end{defn}

\emph{Pseudo--even numbers of second kind} up to 199 are 45: 21, 23, 25, 27, 29, 41, 43, 45, 47, 49, 61, 63, 65, 67, 69, 81, 83, 85, 87, 89, 101, 103, 105, 107, 109, 121, 123, 125, 127, 129, 141, 143, 145, 147, 149, 161, 163, 165, 167, 169, 181, 183, 185, 187, 189.

\subsection{Pseudo--Even Numbers of Third Kind}

\begin{defn}
  A number is a \emph{pseudo--even of third kind} if exist a nontrivial permutation of digits is an even.
\end{defn}

\emph{Pseudo--even numbers of third kind} up to 199 are 115: 20, 21, 22, 23, 24, 25, 26, 27, 28, 29, 40, 41, 42, 43, 44, 45, 46, 47, 48, 49, 60, 61, 62, 63, 64, 65, 66, 67, 68, 69, 80, 81, 82, 83, 84, 85, 86, 87, 88, 89, 100, 101, 102, 103, 104, 105, 106, 107, 108, 109, 110, 112, 114, 116, 118, 120, 121, 122, 123, 124, 125, 126, 127, 128, 129, 130, 132, 134, 136, 138, 140, 141, 142, 143, 144, 145, 146, 147, 148, 149, 150, 152, 154, 156, 158, 160, 161, 162, 163, 164, 165, 166, 167, 168, 169, 170, 172, 174, 176, 178, 180, 181, 182, 183, 184, 185, 186, 187, 188, 189, 190, 192, 194, 196, 198.

\section{Pseudo--Multiples of Prime}

\subsection{Pseudo--Multiples of First Kind of Prime}

\begin{defn}\label{Definitia Pseudo-Multiple 1 of p}
  A number is a \emph{pseudo--multiple of f{}irst kind} of $\emph{prime}$ if exist a permutation of the digits is a multiple of $p$, including the identity permutation.
\end{defn}

\emph{Pseudo--Multiples of f{}irst kind of 5} up to 199 are 63: 5, 10, 15, 20, 25, 30, 35, 40, 45, 50, 51, 52, 53, 54, 55, 56, 57, 58, 59, 60, 65, 70, 75, 80, 85, 90, 95, 100, 101, 102, 103, 104, 105, 106, 107, 108, 109, 110, 115, 120, 125, 130, 135, 140, 145, 150, 151, 152, 153, 154, 155, 156, 157, 158, 159, 160, 165, 170, 175, 180, 185, 190, 195.

\emph{Pseudo--Multiples of f{}irst kind of 7} up to 199 are 80: 7, 12, 14, 19, 21, 24, 28, 35, 36, 41, 42, 48, 49, 53, 56, 63, 65, 70, 77, 82, 84, 89, 91, 94, 98, 102, 103, 104, 105, 109, 112, 115, 116, 119, 120, 121, 123, 126, 127, 128, 130, 132, 133, 134, 135, 137, 139, 140, 143, 144, 145, 147, 150, 151, 153, 154, 156, 157, 158, 161, 162, 165, 166, 168, 169, 172, 173, 174, 175, 179, 182, 185, 186, 189, 190, 191, 193, 196, 197, 198.

\subsection{Pseudo--Multiples of Second Kind of Prime}

\begin{defn}\label{Definitia Pseudo-Multiple 2 of p}
  A non-multiple of $p$ is a \emph{pseudo--multiple of second kind} of $p$ (prime) if exist permutation of the digits is a multiple of $p$.
\end{defn}

\emph{Pseudo--Multiples of second kind of 5} up to 199 are 24: 51, 52, 53, 54, 56, 57, 58, 59, 101, 102, 103, 104, 106, 107, 108, 109, 151, 152, 153, 154, 156, 157, 158, 159.

\emph{Pseudo--Multiples of second kind of 7} up to 199 are 52: 12, 19, 24, 36, 41, 48, 53, 65, 82, 89, 94, 102, 103, 104, 109, 115, 116, 120, 121, 123, 127, 128, 130, 132, 134, 135, 137, 139, 143, 144, 145, 150, 151, 153, 156, 157, 158, 162, 165, 166, 169, 172, 173, 174, 179, 185, 186, 190, 191, 193, 197, 198.

\subsection{Pseudo--Multiples of Third Kind of Prime}

\begin{defn}\label{Definitia Pseudo-Multiple 3 of p}
  A number is a \emph{pseudo--multiple of third kind} of $p$ (prime) if exist a nontrivial permutation of the digits is a multiple of $p$.
\end{defn}

\emph{Pseudo--Multiples of third kind of 5} up to 199 are 46: 50, 51, 52, 53, 54, 55, 56, 57, 58, 59, 100, 101, 102, 103, 104, 105, 106, 107, 108, 109, 110, 115, 120, 125, 130, 135, 140, 145, 150, 151, 152, 153, 154, 155, 156, 157, 158, 159, 160, 165, 170, 175, 180, 185, 190, 195.

\emph{Pseudo--Multiples of third kind of 7} up to 199 are 63: 12, 19, 24, 36, 41, 48, 53, 65, 70, 77, 82, 89, 94, 102, 103, 104, 109, 112, 115, 116, 119, 120, 121, 123, 127, 128, 130, 132, 133, 134, 135, 137, 139, 140, 143, 144, 145, 147, 150, 151, 153, 156, 157, 158, 161, 162, 165, 166, 168, 169, 172, 173, 174, 179, 182, 185, 186, 189, 190, 191, 193, 197, 198.

\section{Progressions}

How many primes do the following progressions contain:
\begin{enumerate}
  \item The sequence $\set{a\cdot p_n+b}$, $n=1,2,\ldots$ where $(a,b)=1$, i.e. $\emph{gcd(a,b)}=1$, and $p_n$ is $n$-th prime?

  Example: $a:=3$ $b:=10$ $n:=25$ $k:=1..n$ $q_k:=a\cdot p_k+b$, then $q^\textrm{T}\rightarrow$($2^4$, $19$, $5^2$, $31$, $43$, $7^2$, $61$, $67$, $79$, $97$, $103$, $11^2$, $7\cdot19$, $139$, $151$, $13^2$, $11\cdot17$, $193$, $211$, $223$, $229$, $13\cdot19$, $7\cdot37$, $277$, $7\cdot43$). Therefore in 25 terms 15 are prime numbers (See Figure \ref{MathcadProgressions}).
  \item The sequence $\set{a^n+b}$, $n=1,2,\ldots$, where $(a,b)=1$, and $a\neq\pm1$ and $a\neq0$?

  Example: $a:=3$ $b:=10$ $n:=25$ $k:=1..n$ $q_k:=a^k+b$, then in sequence $q$ are 6 prime numbers: 13, 19, 37, 739, 65571 and 387420499 (See \ref{MathcadProgressions}).
  \item The sequence $\set{n^n\pm1}$, $n=1,2,\ldots$?
  \begin{enumerate}
    \item First 10 terms from the sequence $\set{n^n+1}$ are: 2, 5, 28, 257, 3126, 46657, 823544, 16777217, 387420490, 10000000001, of which 2, 5 and 257 are primes (See Figure \ref{MathcadProgressions}).
    \item First 10 terms from the sequence $\set{n^n-1}$ are: 0, 3, 26, 255, 3124, 46655, 823542, 16777215, 387420488, 9999999999, of which 3 is prime (See Figure \ref{MathcadProgressions}).
  \end{enumerate}
  \item The sequence $\set{p_n\#\pm1}$, $n=1,2,\ldots$, where $p_n$ is $n$-th prime?
  \begin{enumerate}
    \item First 10 terms from the sequence $\set{p_n\#+1}$ are: 3, 7, 31, 211, 2311, 30031, 510511, 9699691, 223092871, 6469693231 of which 3, 7, 31 and 211 are primes (See Figure \ref{MathcadProgressions}).
    \item First 10 terms from the sequence $\set{p_n\#-1}$ are: 1, 5, 29, 209, 2309, 30029, 510509, 9699689, 223092869, 6469693229 of which 5, 29, 2309 and 30029 are primes (See Figure \ref{MathcadProgressions}).
  \end{enumerate}
  \item The sequence $\set{p_n\#\#\pm1}$, $n=1,2,\ldots$, where $p_n$ is $n$-th prime?
  \begin{enumerate}
    \item First 10 terms from the sequence $\set{p_n\#\#+1}$ are: 3, 4, 11, 8, 111, 92, 1871, 1730, 43011, 1247291 of which 3, 11, 1871 and 1247291 are primes (See Figure \ref{MathcadProgressions}).
    \item First 10 terms from the sequence $\set{p_n\#\#-1}$ are: 1, 2, 9, 6, 109, 90, 1869, 1728, 43009, 1247289 of which 2 and 109 are primes (See Figure \ref{MathcadProgressions}).
  \end{enumerate}
  \item The sequence $\set{p_n\#\#\#\pm2}$, $n=1,2,\ldots$, where $p_n$ is $n$-th prime?
  \begin{enumerate}
    \item First 10 terms from the sequence $\set{p_n\#\#\#+2}$ are: 4, 5, 7, 23, 233, 67, 1107, 4391, 100949, 32047 of which 5, 7, 23, 233, 67 and 4391 are primes (See Figure \ref{MathcadProgressions}).
    \item First 10 terms from the sequence $\set{p_n\#\#\#-2}$ are: 0, 1, 3, 19, 229, 63, 1103, 4387, 100945, 32043 of which 3, 19, 229 and 1103 are primes (See Figure \ref{MathcadProgressions}).
  \end{enumerate}
  \item The sequence $\set{n!!\pm2}$ and $\set{n!!!\pm1}$, $n=1,2,\ldots$?
  \begin{enumerate}
    \item First 17 terms from the sequence $\set{n!!+2}$ are: 3, 4, 5, 10, 17, 50, 107, 386, 947, 3842, 10397, 46082, 135137, 645122, 2027027, 10321922, 34459427 of which 3, 5, 17, 107 and 947 are primes (See Figure \ref{MathcadProgressions}).
    \item First 17 terms from the sequence $\set{n!!-2}$ are: -1, 0, 1, 6, 13, 46, 103, 382, 943, 3838, 10393, 46078, 135133, 645118, 2027023, 10321918, 34459423 of which 13, 103, 2027023 and 34459423 are primes
        (See Figure \ref{MathcadProgressions}).
    \item First 22 terms from the sequence $\set{n!!!+1}$ are: 2, 3, 4, 5, 11, 19, 29, 81, 163, 281, 881, 1945, 3641, 12321, 29161, 58241, 209441, 524881, 1106561, 4188801, 11022481, 24344321 of which 2, 3, 5, 11, 19, 29, 163, 281, 881, and 209441 are primes (See Figure \ref{MathcadProgressions}).
    \item First 22 terms from the sequence $\set{n!!!-1}$ are: 0, 1, 2, 3, 9, 17, 27, 79, 161, 279, 879, 1943, 3639, 12319, 29159, 58239, 209439, 524879, 1106559, 4188799, 11022479, 24344319 of which 2, 3, 17, 79 and 4188799 are primes (See Figure \ref{MathcadProgressions}).
  \end{enumerate}
  \item The sequences $\set{2^n\pm1}$ (Mersenne primes) and $\set{n!\pm1}$ (factorial primes) are well studied.
\end{enumerate}

\section{Palindromes}

\subsection{Classical Palindromes}

A palindrome of one digit is a number (in some base $b$) that is the same when written forwards or backwards, i.e.  of the form $\overline{d_1d_2\ldots d_2d_1}$. The f{}irst few palindrome in base 10 are 0, 1, 2, 3, 4, 5, 6, 7, 8, 9, 11, 22, 33, 44, 55, 66, 77, 88, 99, 101, 111, 121, \ldots \cite[A002113]{SloaneOEIS}.

The numbers of palindromes less than 10, $10^2$, $10^3$, \ldots are 9, 18, 108, 198, 1098, 1998, 10998, \ldots \cite[A050250]{SloaneOEIS}.

\begin{prog}\label{Program gP} palindrome generator in base $b$.
  \begin{tabbing}
    $\emph{gP}(v,b):=$\=\vline\ $f$\=$\emph{or}\ k\in1..\emph{last}(v)$\\
    \>\vline\ \>\vline\ $n\leftarrow n+v_k\cdot b^m$\\
    \>\vline\ \>\vline\ $m\leftarrow m+\emph{nrd}(v_k,b)$\\
    \>\vline\ $\emph{return}\ \ n$\\
  \end{tabbing}
  The program use the function $\emph{nrd}$, \ref{FunctionNrd}.
\end{prog}

\begin{prog}\label{Program PgP} of generate palindromes and primality verif{}ication.
  \begin{tabbing}
    $\emph{PgP}(\alpha,\beta,r,b,t,\emph{IsP}):=$\=\vline\ $j\leftarrow0$\\
    \>\vline\ $f$\=$\emph{or}\ k_1\in\alpha,\alpha+r..\beta$\\
    \>\vline\ \>\vline\ $v_1\leftarrow k_1$\\
    \>\vline\ \>\vline\ $f$\=$\emph{or}\ k_2\in\alpha,\alpha+r..\beta$\\
    \>\vline\ \>\vline\ \>\vline\ $v_2\leftarrow k_2$\\
    \>\vline\ \>\vline\ \>\vline\ $f$\=$\emph{or}\ k_3\in\alpha,\alpha+r..\beta$\\
    \>\vline\ \>\vline\ \>\vline\ \>\vline\ $v_3\leftarrow k_3$\\
    \>\vline\ \>\vline\ \>\vline\ \>\vline\ $n\leftarrow \emph{gP(v,b)}$\\
    \>\vline\ \>\vline\ \>\vline\ \>\vline\ $i$\=$f\ IsP=1$\\
    \>\vline\ \>\vline\ \>\vline\ \>\vline\ \>\vline\ $i$\=$f\ \emph{IsPrime(n)}\textbf{=}1$\\
    \>\vline\ \>\vline\ \>\vline\ \>\vline\ \>\vline\ \>\vline\ $j\leftarrow j+1$\\
    \>\vline\ \>\vline\ \>\vline\ \>\vline\ \>\vline\ \>\vline\ $S_j\leftarrow n$\\
    \>\vline\ \>\vline\ \>\vline\ \>\vline\ $o$\=$\emph{therwise}$\\
    \>\vline\ \>\vline\ \>\vline\ \>\vline\ \>\vline\ $j\leftarrow j+1$\\
    \>\vline\ \>\vline\ \>\vline\ \>\vline\ \>\vline\ $S_j\leftarrow n$\\
    \>\vline\ $\emph{return}\ \ \emph{sort(S)}$\\
  \end{tabbing}
  This program use subprogram $\emph{gP}$, \ref{Program gP} and Mathcad programs $\emph{IsPrime}$ and $\emph{sort}$.
\end{prog}

There are 125  palindromes of 5--digits, in base 10, made only with numbers 1, 3, 5, 7 and 9, which are obtained by running $\emph{Pgp}(1,9,2,10,0,0)^\textrm{T}\rightarrow$ 11111, 11311, 11511, 11711, 11911, 13131, 13331, 13531, 13731, 13931, 15151, 15351, 15551, 15751, 15951, 17171, 17371, 17571, 17771, 17971, 19191, 19391, 19591, 19791, 19991, 31113, 31313, 31513, 31713, 31913, 33133, 33333, 33533, 33733, 33933, 35153, 35353, 35553, 35753, 35953, 37173, 37373, 37573, 37773, 37973, 39193, 39393, 39593, 39793, 39993, 51115, 51315, 51515, 51715, 51915, 53135, 53335, 53535, 53735, 53935, 55155, 55355, 55555, 55755, 55955, 57175, 57375, 57575, 57775, 57975, 59195, 59395, 59595, 59795, 59995, 71117, 71317, 71517, 71717, 71917, 73137, 73337, 73537, 73737, 73937, 75157, 75357, 75557, 75757, 75957, 77177, 77377, 77577, 77777, 77977, 79197, 79397, 79597, 79797, 79997, 91119, 91319, 91519, 91719, 91919, 93139, 93339, 93539, 93739, 93939, 95159, 95359, 95559, 95759, 95959, 97179, 97379, 97579, 97779, 97979, 99199, 99399, 99599, 99799, 99999.

Of these we have 25 prime numbers, which are obtained by running $\emph{Pgp}(1,9,2,10,0,1)^\textrm{T}\rightarrow$ 11311, 13331, 13931, 15551, 17971, 19391, 19991, 31513, 33533, 35153, 35353, 35753, 37573, 71317, 71917, 75557, 77377, 77977, 79397, 79997, 93139, 93739, 95959, 97379, 97579~.

\begin{prog}\label{Program RePal} the palindromes recognition in base $b$.
  \begin{tabbing}
    $\emph{RePal}(n,b):=$\=\vline\ $d\leftarrow dn(n,b)$\\
    \>\vline\ $u\leftarrow\emph{last(d)}$\\
    \>\vline\ $\emph{return}\ \ 1\ \emph{if}\ n\textbf{=}1$\\
    \>\vline\ $m\leftarrow \emph{floor}\big(\frac{u}{2}\big)$\\
    \>\vline\ $f$\=$\emph{or}\ k\in1..m$\\
    \>\vline\ \>\ $\emph{return}\ 0\ \ \emph{if}\ \ \emph{dk}\neq d_{u-k+1}$\\
    \>\vline\ $\emph{return}\ \ 1$\\
  \end{tabbing}
\end{prog}

\begin{prog}\label{Program NrPa} the palindromes counting.
  \begin{tabbing}
    $\emph{NrPa}(m,B):=$\=\vline\ $f$\=$\emph{or}\ b\in2..B$\\
    \>\vline\ \>\vline\ $f$\=$\emph{or}\ k\in1..b^m$\\
    \>\vline\ \>\vline\ \>\ $v_k\leftarrow \emph{RePal}(k,b)$\\
    \>\vline\ \>\vline\ $f$\=$\emph{or}\ \mu\in1..m$\\
    \>\vline\ \>\vline\ \>\ $\emph{NP}_{b-1,\mu}\leftarrow\sum\emph{submatrix}(v,1,b^\mu,1,1)$\\
    \>\vline\ $\emph{return}\ \ \emph{NP}$\\
  \end{tabbing}
\end{prog}

The number of palindromes of one digit in base $b$ is given in Table \ref{NumberPolindromes} and was obtained with the command $\emph{NrPa}(6,16)$.
\begin{table}[h]
  \centering
  \begin{tabular}{|r|r|r|r|r|r|r|}
     \hline
     $b\backslash b^k$ & $b$ & $b^2$ & $b^3$ & $b^4$ & $b^5$ & $b^6$ \\ \hline
     2 & 1 & 2 & 4 & 6 & 10 & 14\\
     3 & 2 & 4 & 10 & 16 & 34 & 52\\
     4 & 3 & 6 & 18 & 30 & 78 & 126\\
     5 & 4 & 8 & 29 & 49 & 149 & 250\\
     6 & 5 & 10 & 41 & 71 & 251 & 432\\
     7 & 6 & 12 & 54 & 96 & 390 & 684\\
     8 & 7 & 14 & 70 & 126 & 574 & 1022\\
     9 & 8 & 16 & 88 & 160 & 808 & 1456\\
     10 & 9 & 18 & 108 & 198 & 1098 & 1998\\
     11 & 10 & 20 & 130 & 240 & 1451 & 2661\\
     12 & 11 & 22 & 154 & 286 & 1871 & 3455\\
     13 & 12 & 24 & 180 & 336 & 2364 & 4392\\
     14 & 13 & 26 & 208 & 390 & 2938 & 5486\\
     15 & 14 & 28 & 239 & 449 & 3600 & 6751\\
     16 & 15 & 30 & 270 & 510 & 4350 & 8190\\
     \hline
   \end{tabular}
  \caption{Number of palindromes of one digit in base $b$}\label{NumberPolindromes}
\end{table}

\subsection{Palindromes with Groups of $m$ Digits}

\begin{enumerate}
  \item Palindromes with groups of one digit in base $b$ are classical palindromes.
  \item Palindromes with groups of 2 digits, in base $b$, are:
      \[
       \overline{d_1d_2d_3d_4\ldots d_{n-3}d_{n-2}d_{n-1}d_nd_{n-1}d_nd_{n-3}d_{n-2}\ldots d_3d_4d_1d_2}
      \]
      or
      \[
        \overline{d_1d_2d_3d_4\ldots d_{n-3}d_{n-2}d_{n-1}d_nd_{n-3}d_{n-2}\ldots d_3d_4d_1d_2}~,
      \]
      where $d_k\in\set{0,1,2,\ldots,b-1}$ and $b\in\Ns$, $b\ge2$.

      Examples: 345534, 78232378, 782378, 105565655510, 1055655510, 3334353636353433, 33343536353433.
  \item Palindromes with groups of 3 digits in base $b$, are:
      \[
       \overline{d_1d_2d_3d_4d_5d_6\ldots d_{n-2}d_{n-1}d_nd_{n-2}nd_{n-1}d_n\ldots d_4d_5d_6d_1d_2d_3}
      \]
      or
      \[
       \overline{d_1d_2d_3\ldots d_{n-5}d_{n-4}d_{n-3}d_{n-2}d_{n-1}d_nd_{n-5}d_{n-4}d_{n-3}\ldots d_1d_2d_3}
      \]
      where $d_k\in\set{0,1,2,\ldots,b-1}$ and $b\in\Ns$, $b\ge2$. Examples: 987987, 456567678678567456, 456567678567456, 123321123, 123234234123, 676767808808767676.
  \item and so on~.
\end{enumerate}

Examples of palindromes with groups of 2 digits in base $b=3$ are: 1, 2, $30=1010_{(3)}$, $40=1111_{(3)}$, $50=1212_{(3)}$, $60=2020_{(3)}$, $70=2121_{(3)}$, $80=2222_{(3)}$, or in base $b=4$ are: 1, 2, 3, $68=1010_{(4)}$, $85=1111_{(4)}$, $102=1212=_{(4)}$, $119=1313_{(4)}$, $136=2020_{(4)}$, $153=2121_{(4)}$, $170=2222_{(4)}$, $187=2323_{(4)}$, $204=3030_{(4)}$, $221=3131_{(4)}$, $238=3232_{(4)}$, $255=3333_{(4)}$. It is noted that $1111_{(3)}$, $2222_{(3)}$, $1111_{(4)}$, $2222_{(4)}$ and $3333_{(4)}$ are palindromes and a single digit.

Numbers of palindromes with groups of one and two digits, in base $b$, $b=2,3,\ldots,16$, for the numbers $1,2,\ldots,b^m$, where $m=1,2,\ldots,6$ are found in Table \ref{NumberPolindromes1and2digits}.
\begin{table}[h]
  \centering
  \begin{tabular}{|r|r|r|r|r|r|r|}
     \hline
     $b\backslash b^k$ & $b$ & $b^2$ & $b^3$ & $b^4$ & $b^5$ & $b^6$ \\ \hline
     2 & 1 & 2 & 4 & 7 & 13 & 23\\
     3 & 2 & 4 & 10 & 20 & 50 & 116\\
     4 & 3 & 6 & 18 & 39 & 123 & 351\\
     5 & 4 & 8 & 29 & 65 & 245 & 826\\
     6 & 5 & 10 & 41 & 96 & 426 & 1657\\
     7 & 6 & 12 & 54 & 132 & 678 & 2988\\
     8 & 7 & 14 & 70 & 175 & 1015 & 4991\\
     9 & 8 & 16 & 88 & 224 & 1448 & 7856\\
     10 & 9 & 18 & 108 & 279 & 1989 & 11799\\
     11 & 10 & 20 & 130 & 340 & 2651 & 17061\\
     12 & 11 & 22 & 154 & 407 & 3444 & 23904\\
     13 & 12 & 24 & 180 & 480 & 4380 & 32616\\
     14 & 13 & 26 & 208 & 559 & 5473 & 43511\\
     15 & 14 & 28 & 239 & 645 & 6736 & 56927\\
     16 & 15 & 30 & 270 & 735 & 8175 & 73215\\
     \hline
   \end{tabular}
  \caption{Number of palindromes of one and two digits in base $b$}\label{NumberPolindromes1and2digits}
\end{table}

Numbers of palindromes with groups of one, two and three digits, in base $b=2,3,\ldots,16$, for the numbers $1,2,\ldots,b^m$ where $m=1,2,\ldots,6$ are found in Table \ref{NumberPolindromes1and3digits}.
\begin{table}[h]
  \centering
  \begin{tabular}{|r|r|r|r|r|r|r|}
     \hline
     $b\backslash b^k$ & $b$ & $b^2$ & $b^3$ & $b^4$ & $b^5$ & $b^6$ \\ \hline
     2 & 1 & 2 & 4 & 7 & 13 & 25\\
     3 & 2 & 4 & 10 & 20 & 50 & 128\\
     4 & 3 & 6 & 18 & 39 & 123 & 387\\
     5 & 4 & 8 & 29 & 65 & 245 & 906\\
     6 & 5 & 10 & 41 & 96 & 426 & 1807\\
     7 & 6 & 12 & 54 & 132 & 678 & 3240\\
     8 & 7 & 14 & 70 & 175 & 1015 & 5383\\
     9 & 8 & 16 & 88 & 224 & 1448 & 8432\\
     10 & 9 & 18 & 108 & 279 & 1989 & 12609\\
     11 & 10 & 20 & 130 & 340 & 2651 & 18161\\
     12 & 11 & 22 & 154 & 407 & 3444 & 25356\\
     13 & 12 & 24 & 180 & 480 & 4380 & 34488\\
     14 & 13 & 26 & 208 & 559 & 5473 & 45877\\
     15 & 14 & 28 & 239 & 645 & 6736 & 59867\\
     16 & 15 & 30 & 270 & 735 & 8175 & 76815\\
     \hline
   \end{tabular}
  \caption{Number of palindromes of one, two and three digits in base $b$}\label{NumberPolindromes1and3digits}
\end{table}

Unsolved research problem: The interested readers can study the $m$-digits palindromes that are prime, considering special classes of $m$-digits palindromes.

\subsection{Generalized Smarandache Palindrome}

A generalized Smarandache palindrome (GSP) is a number of the concatenated form:
\[
 \overline{a_1a_2\ldots a_{n-1}a_na_{n-1}\ldots a_2a_1}
\]
with $n\ge2$ (GSP1), or
\[
 \overline{a_1a_2\ldots a_{n-1}a_na_na_{n-1}\ldots a_2a_1}
\]
with $n\ge1$ (GSP2), where all $a_1$, $a_2$, \ldots, $a_n$ are positive integers in base $b$ of various number of digits, \citep{Khoshnevisan2003,Khoshnevisan2003a,Evans+Pinter+Libis2004}, \cite[A082461]{SloaneOEIS}, \citep{WeissteinPalindromicNumber,WeissteinPalindromicPrime}.

We agree that any number with a single digit, in base of numeration $b$ is palindrome GSP1 and palindrome GSP2.

Examples:
\begin{enumerate}
  \item The number $123567567312_{(10)}$ is a GSP2 because we can group it as $(12)(3)(567)(567)(3)(12)$ i.e. ABCCBA.
  \item The number $23523_{(8)}=10067_{(10)}$ is also a GSP1 since we can group it as $(23)(5)(23)$, i.e. ABA.
  \item The number $abcddcba_{(16)}=2882395322_{(10)}$ is a GSP2.
\end{enumerate}

\begin{prog}\label{Functia GSP1}
  \[
   \emph{GSP1}(v,b):=\emph{gP}(\emph{stack}(v,\emph{submatrix}(\emph{reverse}(v),2,\emph{last}(v),1,1)),b)~,
  \]
  where $\emph{stack}$, $\emph{submatrix}$ and $\emph{reverse}$ are Mathcad functions.
\end{prog}

\begin{prog}\label{Functia GSP2}
  \[
   \emph{GSP2}(v,b):=\emph{gP}(\emph{stack}(v,\emph{reverse}(v)),b)~,
  \]
  where $\emph{stack}$ and $\emph{reverse}$ are Mathcad functions.
\end{prog}

Examples:
\begin{enumerate}
  \item If $v:=\left(\begin{array}{ccc}
                   17 & 3 & 567 \\
                 \end{array}\right)^\textrm{T}$, then
      \begin{enumerate}
        \item $\emph{GSP1}(v,10)=173567317_{(10)}$ and $\emph{GSP2}(v,10)=173567567317_{(10)}$~,
        \item $\emph{GSP1}(v,8)=291794641_{(10)}=173567317_{(8)}$ and \\$\emph{GSP2}(v,8)=1195190283985_{(10)}=173567567317_{(8)}$~.
      \end{enumerate}
  \item If $u:=\left(\begin{array}{cccc}
                       31 & 3 & 201 & 1013 \\
                     \end{array}\right)^\textrm{T}$ then
      \begin{enumerate}
        \item $\emph{GSP1}(u,10)=3132011013201331_{(10)}$ and \\$\emph{GSP2}(u,10)=31320110131013201331_{(10)}$~,
        \item
        \begin{multline*}
          \emph{GSP1}(u,5)=120790190751031_{(10)}=3132011013201331_{(5)}\ \textnormal{and}\\
          \emph{GSP2}(u,5)=31320110131013201331_{(10)}\\
          =31320110131013201331_{(5)}~.
        \end{multline*}
      \end{enumerate}
\end{enumerate}

\begin{prog}\label{Program PgGSP} for generating the palindrome in base $b$ of type GSP1 or GSP2 and eventually checking the primality.
  \begin{tabbing}
    $\emph{PgGSP}(\alpha,\beta,\rho,b,f,\emph{IsP}):=$\=\vline\ $j\leftarrow0$\\
    \>\vline\ $f$\=$\emph{or}\ k_1\in\alpha,\alpha+\rho..\beta$\\
    \>\vline\ \>\vline\ $v_1\leftarrow k_1$\\
    \>\vline\ \>\vline\ $f$\=$\emph{or}\ k_2\in\alpha,\alpha+\rho..\beta$\\
    \>\vline\ \>\vline\ \>\vline\ $v_2\leftarrow k_2$\\
    \>\vline\ \>\vline\ \>\vline\ $f$\=$\emph{or}\ k_3\in\alpha,\alpha+\rho..\beta$\\
    \>\vline\ \>\vline\ \>\vline\ \>\vline\ $v_3\leftarrow k_3$\\
    \>\vline\ \>\vline\ \>\vline\ \>\vline\ $n\leftarrow f(n,b)$\\
    \>\vline\ \>\vline\ \>\vline\ \>\vline\ $i$\=$f\ \emph{IsPrime}(n)\textbf{=}1\ \ \emph{if}\ \ \emph{IsP}=1$\\
    \>\vline\ \>\vline\ \>\vline\ \>\vline\ \>\vline\ $j\leftarrow j+1$\\
    \>\vline\ \>\vline\ \>\vline\ \>\vline\ \>\vline\ $S_j\leftarrow n$\\
    \>\vline\ \>\vline\ \>\vline\ \>\vline\ $\emph{otherwise}$\\
    \>\vline\ \>\vline\ \>\vline\ \>\vline\ \>\vline\ $j\leftarrow j+1$\\
    \>\vline\ \>\vline\ \>\vline\ \>\vline\ \>\vline\ $S_j\leftarrow n$\\
    \>\vline\ $\emph{return}\ \ \emph{sort}(S)$\\
  \end{tabbing}
\end{prog}

Examples:
\begin{enumerate}
  \item All palindromes, in base of numeration $b=10$, of 5 numbers from the set $\set{1,3,5,7,9}$ is obtained with the command
      \[
       z=\emph{PgGSP}(1,9,2,10,\emph{GSP1},0)
      \]
      and to display the vector $z$: $z^\textrm{T}\rightarrow$ 11111, 11311, 11511, 11711, 11911, 13131, 13331, 13531, 13731, 13931, 15151, 15351, 15551, 15751, 15951, 17171, 17371, 17571, 17771, 17971, 19191, 19391, 19591, 19791, 19991, 31113, 31313, 31513, 31713, 31913, 33133, 33333, 33533, 33733, 33933, 35153, 35353, 35553, 35753, 35953, 37173, 37373, 37573, 37773, 37973, 39193, 39393, 39593, 39793, 39993, 51115, 51315, 51515, 51715, 51915, 53135, 53335, 53535, 53735, 53935, 55155, 55355, 55555, 55755, 55955, 57175, 57375, 57575, 57775, 57975, 59195, 59395, 59595, 59795, 59995, 71117, 71317, 71517, 71717, 71917, 73137, 73337, 73537, 73737, 73937, 75157, 75357, 75557, 75757, 75957, 77177, 77377, 77577, 77777, 77977, 79197, 79397, 79597, 79797, 79997, 91119, 91319, 91519, 91719, 91919, 93139, 93339, 93539, 93739, 93939, 95159, 95359, 95559, 95759, 95959, 97179, 97379, 97579, 97779, 97979, 99199, 99399, 99599, 99799, 99999 and $length(z)\rightarrow125$.
  \item  All prime palindromes, in base $b=10$, of 5 numbers from the set $\set{1,3,5,7,9}$ is obtained with the command
      \[
       \emph{zp}=\emph{PgGSP}(1,9,2,10,\emph{GSP1},1)
      \]
      and to display the vector $zp$: $zp^\textrm{T}\rightarrow$ 11311, 13331, 13931, 15551, 17971, 19391, 19991, 31513, 33533, 35153, 35353, 35753, 37573, 71317, 71917, 75557, 77377, 77977, 79397, 79997, 93139, 93739, 95959, 97379, 97579 and $length(zp)\rightarrow25$
  \item  All prime palindromes, in base $b=10$, of 5 numbers from the set $\set{1,4,7,10,13}$ is obtained with the command
      \[
       \emph{sp}=\emph{PgGSP}(1,13,3,10,\emph{GSP1},1)
      \]
      and to display the vector $sp$: $sp^\textrm{T}\rightarrow$ 11411, 14741, 17471, 74747, 77477, 141041, 711017, 711317, 741347, 1104101, 1107101, 1131131, 1314113, 1347413, 1374713, 1377713, 7104107, 7134137, 13410413, 131371313 and $length(sp)\rightarrow20$.
\end{enumerate}

\begin{prog}\label{Program RecGSP} of recognition $\emph{GSP}$ the number $n$ in base $b$.
  \begin{tabbing}
    $\emph{RecGSP}(n,b):=$\=\vline\ $d\leftarrow \emph{dn}(n,b)$\\
    \>\vline\ $m\leftarrow \emph{length}(d)$\\
    \>\vline\ $\emph{return}\ \ 1\ \ \emph{if}\ \ m\textbf{=}1$\\
    \>\vline\ $\emph{jm}\leftarrow \emph{floor}(\frac{m}{2})$\\
    \>\vline\ $f$\=$\emph{or}\ j\in1..jm$\\
    \>\vline\ \>\vline\ $d_1\leftarrow \emph{submatrix}(d,1,j,1,1)$\\
    \>\vline\ \>\vline\ $d_2\leftarrow \emph{submatrix}(d,m+1-j,m,1,1)$\\
    \>\vline\ \>\vline\ $\emph{return}\ \ 1\ \ \emph{if}\ \ d_1=d_2$\\
    \>\vline\ $\emph{return}\ \ 0$\\
  \end{tabbing}
\end{prog}

\begin{prog}\label{Program PGSP} of search palindromes GSP ($y=1$) or not palindromes GSP ($y=0$) from $\alpha$ to $\beta$ in base $b$.
  \begin{tabbing}
    $\emph{PGSP}(\alpha,\beta,b,y):=$\=\vline\ $j\leftarrow 0$\\
    \>\vline\ $f$\=$\emph{or}\ n\in\alpha..\beta$\\
    \>\vline\ \>\ $i$\=$f\ \emph{RecGSP}(n,b)\textbf{=}y$\\
    \>\vline\ \>\ \>\vline\ $j\leftarrow j+1$\\
    \>\vline\ \>\ \>\vline\ $S_{j,1}\leftarrow n$\\
    \>\vline\ \>\ \>\vline\ $S_{j,2}\leftarrow \emph{dn}(n,b)$\\
    \>\vline\ $\emph{return}\ S$\\
  \end{tabbing}
\end{prog}

With this program can display palindromes $GSP$, in base $b=2$, from 1 by 16:
\[
 \emph{PGSP}(1,2^4,2,1)=\left[\begin{array}{cc}
                         1 & (1) \\
                         3 & (1\ \ 1) \\
                         5 & (1\ \ 0\ \ 1) \\
                         7 & (1\ \ 1\ \ 1) \\
                         9 & (1\ \ 0\ \ 0\ \ 1) \\
                         10 & (1\ \ 0\ \ 1\ \ 0) \\
                         11 & (1\ \ 0\ \ 1\ \ 1) \\
                         13 & (1\ \ 1\ \ 0\ \ 1) \\
                         15 & (1\ \ 1\ \ 1\ \ 1) \\
                       \end{array}\right]~.
\]

\begin{prog}\label{Program NrGSP} the $GSP$ palindromes counting.
  \begin{tabbing}
    $\emph{NrGSP}(m,B):=$\=\vline\ $f$\=$\emph{or}\ b\in2..B$\\
    \>\vline\ $f$\=$\emph{or}\ k\in1..b^m$\\
    \>\vline\ \>\vline\ $v_k\leftarrow 1\ \emph{if}\ \emph{RecGSP}(k,b)\textbf{=}1$\\
    \>\vline\ \>\vline\ $v_k\leftarrow0\ \emph{otherwise}$\\
    \>\vline\ \>\vline\ $f$\=$\emph{or}\ \mu\in1..m$\\
    \>\vline\ \>\vline\ \>\ $\emph{NP}_{b-1,\mu}\leftarrow\sum\emph{submatrix}(v,1,b^\mu,1,1)$\\
    \>\vline\ \>\vline\ $v\leftarrow0$\\
    \>\vline\ $\emph{return}\ \ \emph{NP}$\\
  \end{tabbing}
\end{prog}

The number of palindromes $GSP$, in base $b$, are given in Table \ref{NumberPolindromesGSP}, using the command $\emph{NrGSP}(6,16)$:
\begin{table}
  \centering
  \begin{tabular}{|r|r|r|r|r|r|r|}
     \hline
     $b\backslash b^k$ & $b$ & $b^2$ & $b^3$ & $b^4$ & $b^5$ & $b^6$ \\ \hline
     2 & 1 & 2 & 4 & 9 & 19 & 41\\
     3 & 2 & 4 & 10 & 32 & 98 & 308\\
     4 & 3 & 6 & 18 & 75 & 303 & 1251\\
     5 & 4 & 8 & 29 & 145 & 725 & 3706\\
     6 & 5 & 10 & 41 & 246 & 1476 & 9007\\
     7 & 6 & 12 & 54 & 384 & 2694 & 19116\\
     8 & 7 & 14 & 70 & 567 & 4543 & 36743\\
     9 & 8 & 16 & 88 & 800 & 7208 & 65456\\
     10 & 9 & 18 & 108 & 1089 & 10899 & 109809\\
     11 & 10 & 20 & 130 & 1440 & 15851 & 175461\\
     12 & 11 & 22 & 154 & 1859 & 22320 & 269292\\
     13 & 12 & 24 & 180 & 2352 & 30588 & 399528\\
     14 & 13 & 26 & 208 & 2925 & 40963 & 575861\\
     15 & 14 & 28 & 239 & 3585 & 53776 & 809567\\
     16 & 15 & 30 & 270 & 4335 & 69375 & 1113615\\
     \hline
   \end{tabular}
  \caption{Number of palindromes GSP in base $b$}\label{NumberPolindromesGSP}
\end{table}

Unsolved research problem:  To study the number of prime GSPs for given classes of GSPs.

\section{Smarandache--Wellin Primes}

\begin{enumerate}
  \item Special prime digital subsequence: 2, 3, 5, 7, 23, 37, 53, 73, 223, 227, 233, 257, 277, 337, 353, 373, 523, 557, 577, 727, 733, 757, 773, 2237, 2273, 2333, 2357, 2377, 2557, 2753, 2777, 3253, 3257, 3323, 3373, 3527, 3533, 3557, 3727, 3733, 5227, 5233, 5237, 5273, 5323, 5333, 5527, 5557, 5573, 5737, 7237, 7253, 7333, 7523, 7537, 7573, 7577, 7723, 7727, 7753, 7757 \ldots, i.e. the prime numbers whose digits are all primes (they are called \emph{Smarandache--Wellin primes}). For all primes up to $10^7$, which are in number 664579, 1903 are \emph{Smarandache--Wellin primes}.

      Conjecture: this sequence is inf{}inite.
      \begin{prog}\label{Program Wellin} of generate primes Wellin.
        \begin{tabbing}
          $\emph{Wellin}(p,b,L):=$\=\vline\ $f$\=$\emph{or}\ k\in1..L$\\
          \>\vline\ \>\vline\ $d\leftarrow \emph{dn}(p_k,b)$\\
          \>\vline\ \>\vline\ $\emph{sw1}\leftarrow0$\\
          \>\vline\ \>\vline\ $f$\=$\emph{or}\ j\in1..\emph{last(d)}$\\
          \>\vline\ \>\vline\ \>\vline\ $h\leftarrow1$\\
          \>\vline\ \>\vline\ \>\vline\ $\emph{sw2}\leftarrow0$\\
          \>\vline\ \>\vline\ \>\vline\ $w$\=$\emph{hile}\ p_h\le b$\\
          \>\vline\ \>\vline\ \>\vline\ \>\vline\ $\emph{if}$\=$\ \ p_h\le b$\\
          \>\vline\ \>\vline\ \>\vline\ \>\vline\ \>\vline\ $\emph{sw2}\leftarrow1$\\
          \>\vline\ \>\vline\ \>\vline\ \>\vline\ \>\vline\ $\emph{break}$\\
          \>\vline\ \>\vline\ \>\vline\ \>\vline\ $h\leftarrow h+1$\\
          \>\vline\ \>\vline\ \>\vline\ $\emph{sw1}\leftarrow\emph{sw1}+1\ \ \emph{if}\ \ \emph{sw2}=1$\\
          \>\vline\ \>\vline\ $i$\=$f\ \ \emph{sw1}=\emph{last(d)}$\\
          \>\vline\ \>\vline\ \>\vline\ $i\leftarrow i+1$\\
          \>\vline\ \>\vline\ \>\vline\ $w_i\leftarrow p_k$\\
          \>\vline\ $\emph{return}\ \ w$\\
        \end{tabbing}
      \end{prog}
      The list \emph{Smarandache-Wellin primes} generate with commands $L:=1000$ $p:=\emph{submatrix}(\emph{prime},1,L,1,1)$ and $\emph{Wellin}(p,10,L)=$.
  \item \emph{Cira--Smarandache--Wellin primes in octal base}, are that have digits only primes up to 8, i.e. digits are:  2, 3, 5 and 7. For the f{}irst 1000 primes, exist 82 of \emph{Cira--Smarandache--Wellin primes in octal base}: 2, 3, 5, 7, 23, 27, 35, 37, 53, 57, 73, 75, 225, 227, 235, 255, 277, 323, 337, 357, 373, 533, 535, 557, 573, 577, 723, 737, 753, 775, 2223, 2235, 2275, 2325, 2353, 2375, 2377, 2527, 2535, 2725, 2733, 2773, 3235, 3255, 3273, 3323, 3337, 3373, 3375, 3525, 3527, 3555, 3723, 3733, 3753, 3755, 5223, 5227, 5237, 5253, 5275, 5355, 5527, 5535, 5557, 5573, 5735, 5773, 7225, 7233, 7325, 7333, 7355, 7357, 7523, 7533, 7553, 7577, 7723, 7757, 7773, 7775~. Where, for example, $7775_{(8)}=4093_{(10)}$. The list \emph{Smarandache-Wellin primes in octal base} generate with commands $L:=1000$ $p:=\emph{submatrix}(\emph{prime},1,L,1,1)$ and $\emph{Wellin}(p,8,L)=$.
  \item \emph{Cira--Smarandache--Wellin primes in hexadecimal base}, are that have digits only primes up to 16, i.e. digits are:  2, 3, 5, 7, $b$ and $d$. For the f{}irst 1000 primes, exist 68 of \emph{Cira--Smarandache--Wellin primes in hexadecimal base}: 2, 3, 5, 7, $b$, $d$, 25, $2b$, 35, $3b$, $3d$, 53, $b3$, $b5$, $d3$, 223, $22d$, 233, $23b$, 257, 277, $2b3$, $2bd$, $2d7$, $2dd$, $32b$, 335, 337, $33b$, $33d$, 355, $35b$, 373, 377, $3b3$, $3d7$, 527, 557, $55d$, 577, $5b3$, $5b5$, $5db$, 727, 737, 755, 757, 773, $7b5$, $7bb$, $7d3$, $7db$, $b23$, $b2d$, $b57$, $b5d$, $b7b$, $bb7$, $bdd$, $d2b$, $d2d$, $d3d$, $d55$, $db7$, $dbd$, $dd3$, $dd5$, $ddb$. Where, for example, $ddb_{(16)}=3547_{(10)}$. The list \emph{Smarandache-Wellin primes in hexadecimal base} generate with commands $L:=1000$ $p:=\emph{submatrix}(\emph{prime},1,L,1,1)$ and $\emph{Wellin}(p,16,L)=$.
  \item The primes that have numbers of 2 digits primes are \emph{Cira--Smarandache--Wellin primes of second order}. The list the \emph{Cira--Smarandache--Wellin primes of second order}, from 1000 primes, is: 2, 3, 5, 7, 11, 13, 17, 19, 23, 29, 31, 37, 41, 43, 47, 53, 59, 61, 67, 71, 73, 79, 83, 89, 97, 211, 223, 229, 241, 271, 283, 307, 311, 313, 317, 331, 337, 347, 353, 359, 367, 373, 379, 383, 389, 397, 503, 523, 541, 547, 571, 719, 743, 761, 773, 797, 1103, 1117, 1123, 1129, 1153, 1171, 1303, 1307, 1319, 1361, 1367, 1373, 1723, 1741, 1747, 1753, 1759, 1783, 1789, 1907, 1913, 1931, 1973, 1979, 1997, 2311, 2341, 2347, 2371, 2383, 2389, 2903, 2917, 2953, 2971, 3119, 3137, 3167, 3719, 3761, 3767, 3779, 3797, 4111, 4129, 4153, 4159, 4337, 4373, 4397, 4703, 4723, 4729, 4759, 4783, 4789, 5303, 5323, 5347, 5903, 5923, 5953, 6113, 6131, 6143, 6173, 6197, 6703, 6719, 6737, 6761, 6779, 7103, 7129, 7159, 7307, 7331, 7907, 7919~. The total \emph{primes Cira--Smarandache-Wellin of second order}, from 664579 primes, i.e. all primes up to $10^7$ are 12629. The list \emph{Smarandache-Wellin primes in hexadecimal base} generate with commands $L:=1000$ $p:=\emph{submatrix}(\emph{prime},1,L,1,1)$ and $\emph{Wellin}(p,100,L)=$.
  \item Primes that have numbers of 3 digits primes are \emph{Cira--Smarandache--Wellin primes of third order}. The list the \emph{Cira--Smarandache--Wellin primes of third order}, from $10^3$ primes, is: 2, 3, 5, 7, 11, 13, 17, 19, 23, 29, 31, 37, 41, 43, 47, 53, 59, 61, 67, 71, 73, 79, 83, 89, 97, 101, 103, 107, 109, 113, 127, 131, 137, 139, 149, 151, 157, 163, 167, 173, 179, 181, 191, 193, 197, 199, 211, 223, 227, 229, 233, 239, 241, 251, 257, 263, 269, 271, 277, 281, 283, 293, 307, 311, 313, 317, 331, 337, 347, 349, 353, 359, 367, 373, 379, 383, 389, 397, 401, 409, 419, 421, 431, 433, 439, 443, 449, 457, 461, 463, 467, 479, 487, 491, 499, 503, 509, 521, 523, 541, 547, 557, 563, 569, 571, 577, 587, 593, 599, 601, 607, 613, 617, 619, 631, 641, 643, 647, 653, 659, 661, 673, 677, 683, 691, 701, 709, 719, 727, 733, 739, 743, 751, 757, 761, 769, 773, 787, 797, 809, 811, 821, 823, 827, 829, 839, 853, 857, 859, 863, 877, 881, 883, 887, 907, 911, 919, 929, 937, 941, 947, 953, 967, 971, 977, 983, 991, 997, 2003, 2011, 2017, 2029, 2053, 2083, 2089, 2113, 2131, 2137, 2179, 2239, 2251, 2269, 2281, 2293, 2311, 2347, 2383, 2389, 2467, 2503, 2521, 2557, 2593, 2617, 2647, 2659, 2677, 2683, 2719, 2797, 2857, 2887, 2953, 2971, 3011, 3019, 3023, 3037, 3041, 3061, 3067, 3079, 3083, 3089, 3109, 3137, 3163, 3167, 3181, 3191, 3229, 3251, 3257, 3271, 3307, 3313, 3331, 3347, 3359, 3373, 3389, 3433, 3449, 3457, 3461, 3463, 3467, 3491, 3499, 3541, 3547, 3557, 3571, 3593, 3607, 3613, 3617, 3631, 3643, 3659, 3673, 3677, 3691, 3701, 3709, 3719, 3727, 3733, 3739, 3761, 3769, 3797, 3821, 3823, 3853, 3863, 3877, 3881, 3907, 3911, 3919, 3929, 3947, 3967, 5003, 5011, 5023, 5059, 5101, 5107, 5113, 5167, 5179, 5197, 5227, 5233, 5281, 5347, 5419, 5431, 5443, 5449, 5479, 5503, 5521, 5557, 5563, 5569, 5641, 5647, 5653, 5659, 5683, 5701, 5743, 5821, 5827, 5839, 5857, 5881, 5953, 7013, 7019, 7043, 7079, 7103, 7109, 7127, 7151, 7193, 7211, 7229, 7283, 7307, 7331, 7349, 7433, 7457, 7487, 7499, 7523, 7541, 7547, 7577, 7607, 7643, 7673, 7691, 7727, 7757, 7823, 7829, 7853, 7877, 7883, 7907, 7919~. The total \emph{Cira--Smarandache--Wellin primes of third order}, from 664579 primes, i.e. all primes up to $10^7$ are 22716. The list \emph{Smarandache-Wellin primes in hexadecimal base} generate with commands $L:=1000$ $p:=\emph{submatrix}(\emph{prime},1,L,1,1)$ and $\emph{Wellin}(p,1000,L)=$.
\end{enumerate}

In the same general conditions of a given sequence, one screens it selecting only its terms whose groups of digits hold the property (or relationship involving the groups of digits) $p$. A group of digits may contain one or more digits, but not the whole term.

\chapter{Sequences Applied in Science}

\section{Unmatter Sequences}

Unmatter is formed by combinations of matter and antimatter that bind together, or by long--range mixture of matter and antimatter forming a weakly--coupled phase.

And Unmmatter Plasma is a novel form of plasma, exclusively made of matter and its antimatter counterpart.

\subsection{Unmatter Combinations}

Unmatter combinations as pairs of quarks ($q$) and antiquarks ($a$), for $q\ge1$ and $a\ge1$. Each combination has $n=q+a\ge2$ quarks and antiquarks which preserve the colorless, \citep{Smarandache2004a,Smarandache2004b,Smarandache2005}, \cite[A181633]{SloaneOEIS}.
\begin{enumerate}
  \item if $n=2$, we have: $qa$ (biquark -- for example the mesons and antimessons), so the pair is $(1,1)$;
  \item if $n=3$ we have no unmatter combination, so the pair is $(0,0)$;
  \item if $n=4$, we have $qqaa$ (tetraquark), the pair is $(2,2)$;
  \item if $n=5$, we have $qqqqa$, $qaaaa$ (pentaquark), so the pairs are $(4,1)$ and $(1,4)$;
  \item if $n=6$, we have $qqqaaa$ (hexaquark), whence $(3,3)$;
  \item if $n=7$, we have $qqqqqaa$, $qqaaaaa$ (septiquark), whence $(5,2)$, $(2,5)$;
  \item if $n=8$, we have $qqqqqqqa$, $qqqqaaaa$, $qaaaaaaa$ (octoquark), whence $(7,1)$, $(4,4)$, $(1,7)$;
  \item if $n=9$, we have $qqqqqqaaa$, $qqqaaaaaa$ (nonaquark), whence $(6,3)$, $(3,6)$;
  \item if $n=10$, we have $qqqqqqqqaa$, $qqqqqaaaaa$, $qqaaaaaaaa$ (decaquark), whence $(8,2)$, $(5,5)$, $(2,8)$;
\end{enumerate}

From the conditions
\begin{equation}\label{ConditionsQA1}
  \left\{\begin{array}{l}
           q+a=n \\
           q-a=3k
         \end{array}\right.
\end{equation}
result the solutions
\begin{equation}\label{SolutionQA1}
  \left\{\begin{array}{l}
           a=\dfrac{n-3k}{2} \\
           q=\dfrac{n+3k}{2}
         \end{array}\right.~,
\end{equation}
that must be $a,q\in\Ns$, then result that
\begin{equation}\label{FormulakQA1}
  -\left\lfloor\frac{n-2}{3}\right\rfloor\le k\le\left\lfloor\frac{n-2}{3}\right\rfloor\ \ \textnormal{and}\ \ k\in\mathbb{Z}~.
\end{equation}

\begin{prog}\label{Program UC} for generate the unmatter combinations.
  \begin{tabbing}
    $\emph{UC}(n,z):=$\=\vline\ $\emph{return}\ \ "\emph{Error.}"\ \ \emph{if}\ \ n<2$\\
    \>\vline\ $\emph{return}\ \ (1\ \ 1)^\textrm{T}\ \ \emph{if}\ \ n\textbf{=}2$\\
    \>\vline\ $\emph{return}\ \ (0\ \ 0)^\textrm{T}\ \ \emph{if}\ \ n\textbf{=}3\wedge z\textbf{=}1$\\
    \>\vline\ $i\leftarrow \emph{floor}\left(\dfrac{n}{3}\right)\ \ \emph{if}\ \ z\textbf{=}0$\\
    \>\vline\ $i\leftarrow \emph{floor}\left(\dfrac{n-2}{3}\right)\ \ \emph{if}\ \ z\textbf{=}1$\\
    \>\vline\ $j\leftarrow1$\\
    \>\vline\ $f$\=$or\ k\in -i..i$\\
    \>\vline\ \>\vline\ $a\leftarrow\dfrac{n-3k}{2}$\\
    \>\vline\ \>\vline\ $i$\=$f\ a=\emph{trunc}(a)$\\
    \>\vline\ \>\vline\ \>\vline\ $\emph{qa}_j\leftarrow a$\\
    \>\vline\ \>\vline\ \>\vline\ $j\leftarrow j+1$\\
    \>\vline\ \>\vline\ $q\leftarrow\dfrac{n+3k}{2}$\\
    \>\vline\ \>\vline\ $i$\=$f\ q=\emph{trunc}(q)$\\
    \>\vline\ \>\vline\ \>\vline\ $\emph{qa}_j\leftarrow q$\\
    \>\vline\ \>\vline\ \>\vline\ $j\leftarrow j+1$\\
    \>\vline\ $\emph{return}\ \ \emph{qa}$\\
  \end{tabbing}
  In this program was taken into account formulas \ref{SolutionQA1}, \ref{FormulakQA1} and \ref{FormulakQA0}.
\end{prog}

\begin{prog}\label{Program UCS} for generate the unmatter sequences, for $n=\alpha,\alpha+1,\ldots,\beta$, where $\alpha,\beta\in\Ns$, $\alpha<\beta$.
  \begin{tabbing}
    $UCS(\alpha,\beta,z):=$\=\vline\ $S\leftarrow \emph{UC}(\alpha,z)$\\
    \>\vline\ $f$\=$or\ n\in\alpha+1..\beta$\\
    \>\vline\ \>\ $S\leftarrow \emph{stack}\big(S,\emph{UC}(n,z)\big)$\\
    \>\vline\ $\emph{return}\ \ S$\\
  \end{tabbing}
\end{prog}

For $\alpha=2$ and $\beta=30$, the unmattter sequence is: $\emph{UCS}(\alpha,\beta,1)^\textrm{T}\rightarrow$ 1, 1, 0, 0, 2, 2, 4, 1, 1, 4, 3, 3, 5, 2, 2, 5, 7, 1, 4, 4, 1, 7, 6, 3, 3, 6, 8, 2, 5, 5, 2, 8, 10, 1, 7, 4, 4, 7, 1, 10, 9, 3, 6, 6, 3, 9, 11, 2, 8, 5, 5, 8, 2, 11, 13, 1, 10, 4, 7, 7, 4, 10, 1, 13, 12, 3, 9, 6, 6, 9, 3, 12, 14, 2, 11, 5, 8, 8, 5, 11, 2, 14, 16, 1, 13, 4, 10, 7, 7, 10, 4, 13, 1, 16, 15, 3, 12, 6, 9, 9, 6, 12, 3, 15, 17, 2, 14, 5, 11, 8, 8, 11, 5, 14, 2, 17, 19, 1, 16, 4, 13, 7, 10, 10, 7, 13, 4, 16, 1, 19~.

\subsection{Unmatter Combinations of Quarks and Antiquarks}

Unmatter combinations of quarks and antiquarks of length $n\ge1$ that preserve the colorless.

There are 6 types of quarks: Up, Down, Top, Bottom, Strange, Charm and 6 types of antiquarks: $\emph{Up}^\wedge$, $\emph{Down}^\wedge$, $\emph{Top}^\wedge$, $\emph{Bottom}^\wedge$, $\emph{Strange}^\wedge$, $\emph{Charm}^\wedge$.
\begin{enumerate}
  \item For $n=1$, we have no unmatter combination;
  \item For combinations of 2 we have: $\emph{qa}$ (unmatter biquark), [mesons and antimesons]; the number of all possible unmatter combinations will be $6\times6=36$, but not all of them will bind together. It is possible to combine an entity with its mirror opposite and still bound them, such as: $\emph{uu}^\wedge$, $\emph{dd}^\wedge$, $\emph{ss}^\wedge$, $\emph{cc}^\wedge$, $\emph{bb}^\wedge$ which form mesons. It is possible to combine, $\emph{unmatter}+\emph{unmatter}=\emph{unmatter}$, as in $\emph{ud}^\wedge+\emph{us}^\wedge=\emph{uudd}^\wedge \emph{ss}^\wedge$ (of course if they bind together)
  \item For combinations of 7 we have: $\emph{qqqqqaa}$, $\emph{qqaaaaa}$ (unmatter septiquarks); the number of all possible unmatter combinations will be $6^5\times6^2+6^2\times6^5=559872$, but not all of them will bind together.
  \item For combinations of 8 we have: $\emph{qqqqaaaa}$, $\emph{qqqqqqqa}$, $\emph{qaaaaaaa}$ (unmatter octoquarks); the number of all possible unmatter combinations will be $6^7\times6^1+6^4\times6^4+6^1\times6^7=5038848$, but not all of them will bind together.
  \item For combinations of 9 we have: $\emph{qqqqqqaaa}$, $\emph{qqqaaaaaa}$ (unmatter nonaquarks); the number of all possible unmatter combinations will be $6^6\times6^3+6^3\times6^6=2\times6^9=20155392$, but not all of them will bind together.
  \item For combinations of 10 we have: $\emph{qqqqqqqqaa}$, $\emph{qqqqqaaaaa}$, $\emph{qqaaaaaaaa}$ (unmatter decaquarks); the number of all possible unmatter combinations will be $3\times6^{10}=181398528$, but not all of them will bind together.
  \item Etc.
\end{enumerate}

\begin{prog}\label{Program UCqa} for generate the sequence of  unmatter combinations of quarks and antiquarks.
  \begin{tabbing}
    $\emph{UCqa}(\alpha,\beta,z):=$\=\vline\ $j\leftarrow2$\\
    \>\vline\ $f$\=$or\ n\in\alpha..\beta$\\
    \>\vline\ \>\vline\ $qa\leftarrow \emph{UC}(n,z)$\\
    \>\vline\ \>\vline\ $t_j\leftarrow0$\\
    \>\vline\ \>\vline\ $f$\=$or\ k\in1,3..\emph{last}(qa)$\\
    \>\vline\ \>\vline\ \>\ $t_j\leftarrow t_j+6^{qa_{k}+qa_{k+1}}$\\
    \>\vline\ \>\vline\ $j\leftarrow j+1$\\
    \>\vline\ $t_3\leftarrow0\ \ \emph{if}\ \ z\textbf{=}1$\\
    \>\vline\ $\emph{return}\ \ t$\\
  \end{tabbing}
\end{prog}

For $\alpha=2$ and $\beta=30$, the sequence of  unmatter combinations of quarks and antiquarks is:
\begin{multline*}
  UCqa(\alpha,\beta,1)^\textrm{T}\rightarrow 0,\ 36,\  0,\  1296,\  15552,\  46656,\  559872,\  5038848,\\
  20155392,\  181398528,\  1451188224,\  6530347008,\  52242776064,\\
  391820820480,\  1880739938304,\  14105549537280,\  101559956668416,\\
  507799783342080,\  3656158440062976,\  25593109080440832,\\
  131621703842267136,\  921351926895869952,\  6317841784428822528,\\
  33168669368251318272,\  227442304239437611008,\\
  1535235553616203874304,\  8187922952619753996288,\\
  55268479930183339474944,\  368456532867888929832960,\\
  1989665277486600221097984~.
\end{multline*}

I wonder if it is possible to make inf{}initely many combinations of quarks / antiquarks and leptons / antileptons \ldots~. Unmatter can combine with matter and / or antimatter and the result may be any of these three. Some unmatter could be in the strong force, hence part of hadrons.

\subsection[Colorless Combinations as Pairs of Quarks and Antiquarks]{Colorless Combinations as Pairs \\of Quarks and Antiquarks}

Colorless combinations as pairs of quarks and antiquarks, for $q,a\ge0$;
\begin{enumerate}
  \item if $n=2$, we have: $\emph{qa}$ (biquark -- for example the mesons and antimessons), whence the pair $(1,1)$;
  \item if $n=3$, we have: $\emph{qqq}$, $\emph{aaa}$ (triquark -- for example the baryons and antibaryons), whence the pairs $(3,0)$, $(0,3)$;
  \item if $n=4$, we have $\emph{qqaa}$ (tetraquark), whence the pair $(2,2)$;
  \item if $n=5$, we have $\emph{qqqqa}$, $\emph{qaaaa}$ (pentaquark), whence the pairs $(4,1)$, $(1,4)$;
  \item if $n=6$, we have $\emph{qqqqqq}$, $\emph{qqqaaa}$, $\emph{aaaaaa}$ (hexaquark), whence the pairs $(6,0)$, $(3,3)$, $(0,6)$;
  \item if $n=7$, we have $\emph{qqqqqaa}$, $\emph{qqaaaaa}$ (septiquark), whence the pairs $(5,2)$, $(2,5)$;
  \item if $n=8$, we have $\emph{qqqqqqqa}$, $\emph{qqqqaaaa}$, $\emph{qaaaaaaa}$ (octoquark), whence the pairs $(7,1)$, $(4,4)$, $(1,7)$;
  \item if $n=9$, we have $\emph{qqqqqqqqq}$, $\emph{qqqqqqaaa}$, $\emph{qqqaaaaaa}$, $\emph{aaaaaaaaa}$ (nonaquark), whence the pairs $(9,0)$, $(6,3)$, $(3,6)$, $(0,9)$;
  \item if $n=10$, we have $\emph{qqqqqqqqaa}$, $\emph{qqqqqaaaaa}$, $\emph{qqaaaaaaaa}$ (decaquark), whence the pairs $(8,2)$, $(5,5)$, $(2,8)$; There are symmetric pairs.
\end{enumerate}

From the conditions \ref{ConditionsQA1} result the solutions \ref{SolutionQA1}, that must be $a,q\in\Na$, then result that
\begin{equation}\label{FormulakQA0}
  -\left\lfloor\frac{n}{3}\right\rfloor\le k\le\left\lfloor\frac{n}{3}\right\rfloor\ \ \textnormal{and}\ \ k\in\mathbb{Z}~.
\end{equation}

For $\alpha=2$ and $\beta=30$, the unmattter sequence is: $\emph{UCS}(\alpha,\beta,0)^\textrm{T}\rightarrow$ 1, 1, 3, 0, 0, 3, 2, 2, 4, 1, 1, 4, 6, 0, 3, 3, 0, 6, 5, 2, 2, 5, 7, 1, 4, 4, 1, 7, 9, 0, 6, 3, 3, 6, 0, 9, 8, 2, 5, 5, 2, 8, 10, 1, 7, 4, 4, 7, 1, 10, 12, 0, 9, 3, 6, 6, 3, 9, 0, 12, 11, 2, 8, 5, 5, 8, 2, 11, 13, 1, 10, 4, 7, 7, 4, 10, 1, 13, 15, 0, 12, 3, 9, 6, 6, 9, 3, 12, 0, 15, 14, 2, 11, 5, 8, 8, 5, 11, 2, 14, 16, 1, 13, 4, 10, 7, 7, 10, 4, 13, 1, 16, 18, 0, 15, 3, 12, 6, 9, 9, 6, 12, 3, 15, 0, 18, 17, 2, 14, 5, 11, 8, 8, 11, 5, 14, 2, 17, 19, 1, 16, 4, 13, 7, 10, 10, 7, 13, 4, 16, 1, 19~, where $\emph{UCS}$ is the program \ref{Program UCS}.

In order to save the colorless combinations prevailed in the Theory of Quantum Chromodynamics (QCD) of quarks and antiquarks in their combinations when binding, we devised the following formula, \citep{Smarandache2004a,Smarandache2005}: $q$ is congruent with $a$, modulo 3; where $q=$ number of quarks and $a=$ number of antiquarks. To justify this formula we mention that 3 quarks form a colorless combination and any multiple of three combination of quarks too, i.e. 6, 9, 12, etc. quarks. In a similar way, 3 antiquarks form a colorless combination and any multiple of three combination of antiquarks too, i.e. 6, 9, 12, etc. antiquarks.
\begin{itemize}
  \item If $n$ is even, $n=2k$, then its pairs are: $(k+3m,k-3m)$, where $m$ is an integer such that both $k+3m\ge0$ and $k-3m\ge0$.
  \item If $n$ is odd, $n=2k+1$, then its pairs are: $(k+3m+2,k-3m-1)$, where $m$ is an integer such that both $k+3m+2\ge0$ and $k-3m-1\ge0$.
\end{itemize}

\subsection[Colorless Combinations of Quarks and Qntiquarks of Length $n\ge1$]{Colorless Combinations \\of Quarks and Qntiquarks of Length $n\ge1$}

Colorless combinations of quarks and antiquarks of length $n\ge1$, for $q\ge0$ and $a\ge0$.

Comment:
\begin{itemize}
  \item If $n=1$ there is no colorless combination.
  \item If $n=2$ we have $\emph{qa}$ (quark antiquark), so a pair $(1,1)$; since a quark can be $\emph{Up}$, $\emph{Down}$, $\emph{Top}$, $\emph{Bottom}$, $\emph{Strange}$, $\emph{Charm}$ while an antiquark can be $\emph{Up}^\wedge$, $\emph{Down}^\wedge$, $\emph{Top}^\wedge$, $\emph{Bottom}^\wedge$, $\emph{Strange}^\wedge$, $\emph{Charm}^\wedge$ then we have $6\times6=36$ combinations.
  \item If $n=3$ we have $\emph{qqq}$ and $\emph{aaa}$, thus two pairs $(3,0)$, $(0,3)$, i.e. $2\times6^3=432$.
  \item If $n=4$, we have $\emph{qqaa}$, so the pair $(2,2)$, i.e. $6^4=1296$.
\end{itemize}

For $\alpha=2$ and $\beta=30$, the sequence of  unmatter combinations of quarks and antiquarks is:
\begin{multline*}
  \emph{UCqa}(\alpha,\beta,0)^\textrm{T}\rightarrow 0,\  36,\  432,\  1296,\  15552,\  139968,\  559872,\ \\
  5038848,\  40310784,\  181398528,\  1451188224,\  10883911680,\  \\
  52242776064,\  391820820480,\  2821109907456,\  14105549537280,\  \\
  101559956668416,\ 710919696678912,\  3656158440062976,\  \\
  25593109080440832,\ 175495605123022848,\  921351926895869952,\  \\
  6317841784428822528,\ 42645432044894552064,\  \\
  227442304239437611008,\ 1535235553616203874304,\ \\
  10234903690774692495360,\ 55268479930183339474944,\ \\
  368456532867888929832960,\ 2431813116928066936897536~,
\end{multline*}
where $\emph{UCqa}$ is the program \ref{Program UCqa}.

\section{Convex Polyhedrons}

A convex polyhedron can be def{}ined algebraically as the set of solutions to a system of linear inequalities
\[
 M\cdot x\le b
\]
where $M$ is a real $m\times3$ matrix and $b$ is a real $m$--vector. Although usage varies, most authors additionally require that a solution be bounded for it to qualify as a convex polyhedron. A convex polyhedron may be obtained from an arbitrary set of points by computing the convex hull of the points.

Explicit examples are given in the following table:
\begin{enumerate}
  \item Tetrahedron, $m=4$ and
     \[
      \left(\begin{array}{ccc}
              1 & 1 & 1 \\
              1 & -1 & -1 \\
              -1 & 1 & -1 \\
              -1 & -1 & 1 \\
             \end{array}\right)\cdot x\le\left(\begin{array}{c}
                                                 2 \\
                                                 0 \\
                                                 0 \\
                                                 0 \\
                                               \end{array}\right)~;
     \]
  \item Cube, $m=6$ and
      \[
      \left(\begin{array}{ccc}
              1 & 0 & 0 \\
              -1 & 0 & 0 \\
              0 & 1 & 0 \\
              0 & -1 & 0 \\
              0 & 0 & 1 \\
              0 & 0 & -1 \\
             \end{array}\right)\cdot x\le\left(\begin{array}{c}
                                                 1 \\
                                                 1 \\
                                                 1 \\
                                                 1 \\
                                                 1 \\
                                                 1 \\
                                               \end{array}\right)~;
      \]
  \item Octahedron, $m=8$ and
      \[
      \left(\begin{array}{ccc}
              1 & 1 & 1 \\
              1 & 1 & -1 \\
              1 & -1 & 1 \\
              1 & -1 & -1 \\
              -1 & 1 & 1 \\
              -1 & 1 & -1 \\
              -1 & -1 & 1 \\
              -1 & -1 & -1
             \end{array}\right)\cdot x\le\left(\begin{array}{c}
                                                 1 \\
                                                 1 \\
                                                 1 \\
                                                 1 \\
                                                 1 \\
                                                 1 \\
                                                 1 \\
                                                 1 \\
                                               \end{array}\right)~.
     \]
\end{enumerate}

Geometrically, a convex polyhedron can be def{}ined as a polyhedron for which a line connecting any two (noncoplanar) points on the surface always lies in the interior of the polyhedron. Every convex polyhedron can be represented in the plane or on the surface of a sphere by a 3--connected planar graph (called a polyhedral graph). Conversely, by a theorem of Steinitz\index{Steinitz E.} as restated by Gr\"{u}nbaum,\index{Gr\"{u}nbaum B.} every 3-connected planar graph can be realized as a convex polyhedron (Duijvestijn and Federico 1981)\index{Duijvestijn A. J. W.}\index{Federico P. J.}.

\begin{enumerate}
  \item Given $n$ points in space, four by four non-coplanar, f{}ind the maximum number $M(n)$ of points which constitute the vertexes of a convex polyhedron, \citep{Tomescu1983}. Of course, $M(n)\ge4$.
  \item Given $n$ points in space, four by four non-coplanar, f{}ind the minimum number $N(n)\ge5$ such that: any $N(n)$ points among these do not constitute the vertexes of a convex polyhedron. Of course, $N(n)$ may not exist.
\end{enumerate}

\chapter{Constants}

\section{Smarandache Constants}

In Mathworld website, \citep{WeissteinSmarandacheConstants}, one f{}inds the following constants related to the Smarandache function.

\begin{obs}
  All def{}initions use $S$ for denoting Smarandache function \ref{ProgramS}.
\end{obs}

The \emph{Smarandache constant} is the smallest solution to the generalized Andrica\rq{s} conjecture, $x\approx0.567148\ldots$~, \citep[A038458]{SloaneOEIS}.

Equation solutions
      \begin{equation}\label{EcSmEchiv}
        p^x-(p+g)^x=1~,\ \ p\in\NP{2}~,
      \end{equation}
where $g=g_n=p_{n+1}-p_n$ is the gap between two consecutive prime numbers.

The solutions to equation (\ref{EcSmEchiv}) in ascending order using the maximal gaps, \citep{Oliveira2014}, \citep{CiraIJMC2014}.
     \begin{center}
       \begin{longtable}{|r|r|r|}
         \caption{Equation (\ref{EcSmEchiv}) solutions}\label{SolGap}\\
         \hline
         $p$ & $g$ & solution for equation (\ref{EcSmEchiv}) \\
         \hline
         \endfirsthead
         \hline
         $p$ & $g$ & solution for equation (\ref{EcSmEchiv}) \\
         \hline
         \endhead
         \hline \multicolumn{3}{r}{\textit{Continued on next page}} \\
         \endfoot
         \hline
         \endlastfoot
         113 & 14 & 0.5671481305206224\ldots \\ \hline
         1327 & 34 & 0.5849080865740931\ldots \\ \hline
         7 & 4 & 0.5996694211239202\ldots \\ \hline
         23 & 6 & 0.6042842019286720\ldots \\ \hline
         523 & 18 & 0.6165497314215637\ldots \\ \hline
         1129 & 22 & 0.6271418980644412\ldots \\ \hline
         887 & 20 & 0.6278476315319166\ldots \\ \hline
         31397 & 72 & 0.6314206007048127\ldots \\ \hline
         89 & 8 & 0.6397424613256825\ldots \\ \hline
         19609 & 52 & 0.6446915279533268\ldots \\ \hline
         15683 & 44 & 0.6525193297681189\ldots \\ \hline
         9551 & 36 & 0.6551846556887808\ldots \\ \hline
         155921 & 86 & 0.6619804741301879\ldots \\ \hline
         370261 & 112 & 0.6639444999972240\ldots \\ \hline
         492113 & 114 & 0.6692774164975257\ldots \\ \hline
         360653 & 96 & 0.6741127001176469\ldots \\ \hline
         1357201 & 132 & 0.6813839139412406\ldots \\ \hline
         2010733 & 148 & 0.6820613370357171\ldots \\ \hline
         1349533 & 118 & 0.6884662952427394\ldots \\ \hline
         4652353 & 154 & 0.6955672852207547\ldots \\ \hline
         20831323 & 210 & 0.7035651178160084\ldots \\ \hline
         17051707 & 180 & 0.7088121412466053\ldots \\ \hline
         47326693 & 220 & 0.7138744163020114\ldots \\ \hline
         122164747 & 222 & 0.7269826061830018\ldots \\ \hline
         3 & 2 & 0.7271597432435757\ldots \\ \hline
         191912783 & 248 & 0.7275969819805509\ldots \\ \hline
         189695659 & 234 & 0.7302859105830866\ldots \\ \hline
         436273009 & 282 & 0.7320752818323865\ldots \\ \hline
         387096133 & 250 & 0.7362578381533295\ldots \\ \hline
         1294268491 & 288 & 0.7441766589716590\ldots \\ \hline
         1453168141 & 292 & 0.7448821415605216\ldots \\ \hline
         2300942549 & 320 & 0.7460035467176455\ldots \\ \hline
         4302407359 & 354 & 0.7484690049408947\ldots \\ \hline
         3842610773 & 336 & 0.7494840618593505\ldots \\ \hline
         10726904659 & 382 & 0.7547601234459729\ldots \\ \hline
         25056082087 & 456 & 0.7559861641728429\ldots \\ \hline
         42652618343 & 464 & 0.7603441937898209\ldots \\ \hline
         22367084959 & 394 & 0.7606955951728551\ldots \\ \hline
         20678048297 & 384 & 0.7609716068556747\ldots \\ \hline
         127976334671 & 468 & 0.7698203623795380\ldots \\ \hline
         182226896239 & 474 & 0.7723403816143177\ldots \\ \hline
         304599508537 & 514 & 0.7736363009251175\ldots \\ \hline
         241160624143 & 486 & 0.7737508697071668\ldots \\ \hline
         303371455241 & 500 & 0.7745991865337681\ldots \\ \hline
         297501075799 & 490 & 0.7751693424982924\ldots \\ \hline
         461690510011 & 532 & 0.7757580339651479\ldots \\ \hline
         416608695821 & 516 & 0.7760253389165942\ldots \\ \hline
         614487453523 & 534 & 0.7778809828805762\ldots \\ \hline
         1408695493609 & 588 & 0.7808871027951452\ldots \\ \hline
         1346294310749 & 582 & 0.7808983645683428\ldots \\ \hline
         2614941710599 & 652 & 0.7819658004744228\ldots \\ \hline
         1968188556461 & 602 & 0.7825687226257725\ldots \\ \hline
         7177162611713 & 674 & 0.7880214782837229\ldots \\ \hline
         13829048559701 & 716 & 0.7905146362137986\ldots \\ \hline
         19581334192423 & 766 & 0.7906829063252424\ldots \\ \hline
         42842283925351 & 778 & 0.7952277512573828\ldots \\ \hline
         90874329411493 & 804 & 0.7988558653770882\ldots \\ \hline
         218209405436543 & 906 & 0.8005126614171458\ldots \\ \hline
         171231342420521 & 806 & 0.8025304565279002\ldots \\ \hline
         1693182318746371 & 1132 & 0.8056470803187964\ldots \\ \hline
         1189459969825483 & 916 & 0.8096231085041140\ldots \\ \hline
         1686994940955803 & 924 & 0.8112057874892308\ldots \\ \hline
         43841547845541060 & 1184 & 0.8205327998695296\ldots \\ \hline
         55350776431903240 & 1198 & 0.8212591131062218\ldots \\ \hline
         80873624627234850 & 1220 & 0.8224041089823987\ldots \\ \hline
         218034721194214270 & 1248 & 0.8258811322716928\ldots \\ \hline
         352521223451364350 & 1328 & 0.8264955008480679\ldots \\ \hline
         1425172824437699300 & 1476 & 0.8267652954810718\ldots \\ \hline
         305405826521087900 & 1272 & 0.8270541728027422\ldots \\ \hline
         203986478517456000 & 1224 & 0.8271121951019150\ldots \\ \hline
         418032645936712100 & 1370 & 0.8272229385637846\ldots \\ \hline
         401429925999153700 & 1356 & 0.8272389079572986\ldots \\ \hline
         804212830686677600 & 1442 & 0.8288714147741382\ldots \\ \hline
         2 & 1 & 1 \\ \hline
      \end{longtable}
    \end{center}
\begin{enumerate}
  \item The f{}irst Smarandache constant is def{}ined as
      \begin{equation}\label{CS01}
        S_1=\sum_{n=2}^\infty\frac{1}{S(n)!}=1.09317\ldots~,
      \end{equation}
      \citep[A048799]{SloaneOEIS}. \cite{Cojocaru+Cojocaru1996a}\index{Cojocaru I.}\index{Cojocaru S.} prove that $S_1$ exists and is bounded by $0.717<S_1<1.253$.
  \item \cite{Cojocaru+Cojocaru1996b}\index{Cojocaru I.}\index{Cojocaru S.} prove that the second Smarandache constant
      \begin{equation}\label{CS02}
        S_2=\sum_{n=2}^\infty\frac{S(n)}{n!}\approx1.71400629359162\ldots~,
      \end{equation}
      \citep[A048834]{SloaneOEIS} is an irrational number.
  \item \cite{Cojocaru+Cojocaru1996c}\index{Cojocaru I.}\index{Cojocaru S.} prove that the series
      \begin{equation}\label{CS03}
        S_3=\sum_{n=2}^\infty\frac{1}{\displaystyle\prod_{m=2}^nS(m)}\approx 0.719960700043708
      \end{equation}
      converges to a number $0.71<S_3<1.01$.
  \item Series
      \begin{equation}\label{CS04}
        S_4(\alpha)=\sum_{n=2}^\infty\frac{n^\alpha}{\displaystyle\prod_{m=2}^nS(m)}~.
      \end{equation}
      converges for a f{}ixed real number $a\ge1$. The values for small a are
      \begin{eqnarray}
        S_4(1) & \approx & 1.72875760530223\ldots~; \\
        S_4(2) & \approx & 4.50251200619297\ldots~; \\
        S_4(3) & \approx & 13.0111441949445\ldots~, \\
        S_4(4) & \approx & 42.4818449849626\ldots~; \\
        S_4(5) & \approx & 158.105463729329\ldots~,
      \end{eqnarray}
      \cite[A048836, A048837, A048838]{SloaneOEIS}.
  \item \cite{Sandor1997}\index{Sandor J.} shows that the series
      \begin{equation}\label{CS05}
        S_5=\sum_{n=1}^\infty\frac{(-1)^{n-1}S(n)}{n!}
      \end{equation}
      converges to an irrational.
  \item \cite{Burton1995}\index{Burton E.} and \cite{Dumitrescu+Seleacu1996}\index{Dumitrescu C.}\index{Seleacu V.} show that the series
      \begin{equation}\label{CS06}
        S_6=\sum_{n=2}^\infty\frac{S(n)}{(n+1)!}
      \end{equation}
      converges.
  \item \cite{Dumitrescu+Seleacu1996}\index{Dumitrescu C.}\index{Seleacu V.} show that the series
      \begin{equation}\label{CS07}
        S_7=\sum_{n=r}^\infty\frac{S(n)}{(n+r)!}~,
      \end{equation}
      for $r\in\Na$, and
      \begin{equation}\label{CS08}
        S_8=\sum_{n=r}^\infty\frac{S(n)}{(n-r)!}
      \end{equation}
      for $r\in\Ns$, converges.
  \item \cite{Dumitrescu+Seleacu1996}\index{Dumitrescu C.}\index{Seleacu V.} show that
      \begin{equation}\label{CS09}
        S_9=\sum_{n=2}^\infty\frac{1}{\displaystyle\sum_{m=2}^n\frac{S(m)}{m!}}
      \end{equation}
      converges.
  \item \cite{Burton1995},  \cite{Dumitrescu+Seleacu1996} show that the series
      \begin{equation}\label{CS10}
        S_{10}=\sum_{n=2}^\infty\frac{1}{\big(S(n)\big)^\alpha\cdot\sqrt{S(n)!}}
      \end{equation}
      and
      \begin{equation}\label{CS11}
        S_{11}=\sum_{n=2}^\infty\frac{1}{\big(S(n)\big)^\alpha\cdot\sqrt{(S(n)+1)!}}
      \end{equation}
      converge for $\alpha\in\Na$, $\alpha>1$.
\end{enumerate}

\section{Erd\"{o}s--Smarandache Constants}

The authors did not prove the convergence towards each constant. We let it as possible research for the interested readers. With the program $\emph{ES}$, \ref{Program ES}, calculate vector top 100 terms numbers containing Erd\"{o}s--Smarandache, $\emph{es}=ES(2,130)$. The vector $es$ has 100 terms for $n:=last(es)=100$ and last term vector $es$ has the value 130, because $\emph{es}_{\emph{\emph{last}(es)}}=130$.
\begin{enumerate}
  \item The f{}irst constant Erd\"{o}s--Smarandache is def{}ined as
      \begin{equation}\label{CES01}
        \emph{ES}_1=\sum_{k=1}^\infty\frac{1}{es_k!}\approx\sum_{k=1}^n\frac{1}{es_k!}=0.6765876023854308\ldots~,
      \end{equation}
      it is well approximated because
      \[
       \frac{1}{\emph{es}_n!}=1.546\cdot10^{-220}~.
      \]
  \item The second constant Erd\"{o}s--Smarandache is def{}ined as
      \begin{equation}\label{CES02}
        \emph{ES}_2=\sum_{k=1}^\infty\frac{es_k}{k!}\approx\sum_{k=1}^n\frac{es_k}{k!}=4.658103698740189\ldots~,
      \end{equation}
      it is well approximated because
      \[
       \frac{es_n}{n!}=1.393\cdot10^{-156}~.
      \]
  \item The third constant Erd\"{o}s--Smarandache is def{}ined as
      \begin{equation}\label{CES03}
        ES_3=\sum_{k=1}^\infty\dfrac{1}{\displaystyle\prod_{j=1}^{k}es_j}\approx
        \sum_{k=1}^n\dfrac{1}{\displaystyle\prod_{j=1}^{k}es_j}=0.7064363838861719\ldots~,
      \end{equation}
      it is well approximated because
      \[
       \dfrac{1}{\displaystyle\prod_{j=1}^{n}es_j}=3.254\cdot10^{-173}~.
      \]
  \item Series
      \begin{equation}\label{CES04}
        ES_4(\alpha)=\sum_{k=1}^\infty\dfrac{k^\alpha}{\displaystyle\prod_{j=1}^kes_j}~,
      \end{equation}
      then
      \begin{itemize}
        \item The case $\alpha=1$
           \begin{multline*}
             ES_4(1)\approx\sum_{k=1}^n\dfrac{k}{\displaystyle\prod_{j=1}^k es_j}=0.9600553300834916\ldots\\
             \textnormal{it is well approximated because}\ \ \dfrac{n}{\displaystyle\prod_{j=1}^n es_j}=3.254\cdot10^{-171}~,
           \end{multline*}
        \item The case $\alpha=2$
            \begin{multline*}
              ES_4(2)\approx\sum_{k=1}^n\dfrac{k^2}{\displaystyle\prod_{j=1}^k es_j}=1.5786465190659933\ldots\\
              \textnormal{it is well approximated because}\ \ \dfrac{n^2}{\displaystyle\prod_{j=1}^n es_j}=3.254\cdot10^{-169}~,
            \end{multline*}
        \item The case $\alpha=3$
            \begin{multline*}
              ES_4(3)\approx\sum_{k=1}^n\dfrac{k^3}{\displaystyle\prod_{j=1}^k es_j}=3.208028767543241\ldots\\
              \textnormal{it is well approximated because}\ \ \dfrac{n^3}{\displaystyle\prod_{j=1}^n es_j}=3.254\cdot10^{-167}~,
            \end{multline*}
        \item The case $\alpha=4$
            \begin{multline*}
              ES_4(4)\approx\sum_{k=1}^n\dfrac{k^4}{\displaystyle\prod_{j=1}^k es_j}=7.907663276289289\ldots\\
              \textnormal{it is well approximated because}\ \ \dfrac{n^4}{\displaystyle\prod_{j=1}^n es_j}=3.254\cdot10^{-165}~,
            \end{multline*}
        \item The case $\alpha=5$
            \begin{multline*}
              ES_4(5)\approx\sum_{k=1}^n\dfrac{k^5}{\displaystyle\prod_{j=1}^k es_j}=22.86160508982205\ldots\\
              \textnormal{it is well approximated because}\ \ \dfrac{n^5}{\displaystyle\prod_{j=1}^n es_j}=3.254\cdot10^{-163}~.
            \end{multline*}
      \end{itemize}
  \item Series
      \begin{equation}\label{CES05}
        ES_5=\sum_{k=1}^\infty\frac{(-1)^{k+1}es_k}{k!}\approx\sum_{k=1}^n\frac{(-1)^{k+1}es_k}{k!}=1.1296727326811478
      \end{equation}
      it is well approximated because
      \[
       \frac{(-1)^{n+1}es_n}{n!}=-1.393\cdot10^{-156}~.
      \]
  \item Series
      \begin{equation}\label{CES06}
        ES_6=\sum_{k=1}^\infty\frac{es_k}{(k+1)!}\approx\sum_{k=1}^n\frac{es_k}{(k+1)!}=1.7703525971096077~,
      \end{equation}
      it is well approximated because
      \[
       \frac{es_n}{(n+1)!}=1.379\cdot10^{-158}~.
      \]
  \item Series
      \begin{equation}\label{CES07}
        ES_7(r)=\sum_{k=r}^\infty\frac{es_k}{(k+r)!}~,
      \end{equation}
      with $r\in\Ns$, then
      \begin{itemize}
        \item The case $r=1$ \big(be noticed as $ES_7(1)=ES_6$\big)
            \begin{multline*}
              ES_7(1)\approx\sum_{k=1}^n\frac{es_k}{(k+1)!}=1.7703525971096077~, \\
              \textnormal{it is well approximated because}\ \ \frac{es_n}{(n+1)!}=1.379\cdot10^{-158}~,
            \end{multline*}
        \item The case $r=2$
            \begin{multline*}
              ES_7(2)\approx\sum_{k=2}^n\frac{es_k}{(k+2)!}=0.17667118527354841~, \\
              \textnormal{it is well approximated because}\ \ \frac{es_n}{(n+2)!}=1.352\cdot10^{-160}~,
            \end{multline*}
        \item The case $r=3$
            \begin{multline*}
              ES_7(3)\approx\sum_{k=3}^n\frac{es_k}{(k+3)!}=0.0083394778946466~, \\
              \textnormal{it is well approximated because}\ \ \frac{es_n}{(n+3)!}=1.313\cdot10^{-162}~,
            \end{multline*}
      \end{itemize}
  \item Series
      \begin{equation}\label{CES08}
        ES_8(r)=\sum_{k=r}^\infty\frac{es_k}{(k-r)!}~,
      \end{equation}
      with $r\in\Ns$, then
      \begin{itemize}
        \item The case $r=1$
            \begin{multline*}
              ES_8(1)\approx\sum_{k=1}^n\frac{es_k}{(k-1)!}=8.893250907189714~, \\
              \textnormal{it is well approximated because}\ \ \frac{es_n}{(n-1)!}=1.393\cdot10^{-154}~,
            \end{multline*}
        \item The case $r=2$
            \begin{multline*}
              ES_8(2)\approx\sum_{k=2}^n\frac{es_k}{(k-2)!}=12.69625798917767~, \\
              \textnormal{it is well approximated because}\ \ \frac{es_n}{(n-2)!}=1.379\cdot10^{-152}~,
            \end{multline*}
        \item The case $r=3$
            \begin{multline*}
              ES_8(3)\approx\sum_{k=3}^n\frac{es_k}{(k-3)!}=16.756234041646312~, \\
              \textnormal{it is well approximated because}\ \ \frac{es_n}{(n-3)!}=1.351\cdot10^{-150}~.
            \end{multline*}
      \end{itemize}
  \item Series
      \begin{equation}\label{CES09}
        ES_9=\sum_{k=1}^\infty\dfrac{1}{\displaystyle\sum_{j=1}^k es_j!}
        \approx\sum_{k=1}^n\dfrac{1}{\displaystyle\sum_{j=1}^k es_j!}=0.6341618804985396~,
      \end{equation}
      it is well approximated because
      \[
       \dfrac{1}{\displaystyle\sum_{j=1}^n es_j!}=1.535\cdot10^{-220}~.
      \]
  \item Series
      \begin{equation}\label{CES10}
        ES_{10}(\alpha)=\sum_{k=1}^\infty\frac{1}{es_k^\alpha\sqrt{es_k!}}~,
      \end{equation}
      then
      \begin{itemize}
        \item The case $\alpha=1$
            \begin{multline*}
              ES_{10}(1)\approx\sum_{k=1}^n\frac{1}{es_k\sqrt{es_k!}}=0.5161853069935946~, \\
              \textnormal{it is well approximated because}\ \ \frac{1}{es_k\sqrt{es_k!}}=9.566\cdot10^{-113}~,
            \end{multline*}
        \item The case $\alpha=2$
            \begin{multline*}
              ES_{10}(2)\approx\sum_{k=1}^n\frac{1}{es_k^2\sqrt{es_k!}}=0.22711843820442665~, \\
              \textnormal{it is well approximated because}\ \ \frac{1}{es_k^2\sqrt{es_k!}}=7.358\cdot10^{-115}~,
            \end{multline*}
        \item The case $\alpha=3$
            \begin{multline*}
              ES_{10}(3)\approx\sum_{k=1}^n\frac{1}{es_k^3\sqrt{es_k!}}=0.10445320547192125~, \\
              \textnormal{it is well approximated because}\ \ \frac{1}{es_k^3\sqrt{es_k!}}=5.66\cdot10^{-117}~.
            \end{multline*}
      \end{itemize}
  \item Series
      \begin{equation}\label{CES11}
        ES_{11}(\alpha)=\sum_{k=1}^\infty\frac{1}{es_k^\alpha\sqrt{(es_k+1)!}}~,
      \end{equation}
      then
      \begin{itemize}
        \item The case $\alpha=1$
            \begin{multline*}
              ES_{11}(1)\approx\sum_{k=1}^n\frac{1}{es_k\sqrt{(es_k+1)!}}=0.28269850314464495~, \\
              \textnormal{it is well approximated because}\ \ \frac{1}{es_k\sqrt{(es_k+1)!}}=8.357\cdot10^{-114}~,
            \end{multline*}
        \item The case $\alpha=2$
            \begin{multline*}
              ES_{11}(2)\approx\sum_{k=1}^n\frac{1}{es_k^2\sqrt{(es_k+1)!}}=0.1267281413034069~, \\
              \textnormal{it is well approximated because}\ \ \frac{1}{es_k^2\sqrt{(es_k+1)!}}=6.429\cdot10^{-116}~,
            \end{multline*}
        \item The case $\alpha=3$
            \begin{multline*}
              ES_{11}(3)\approx\sum_{k=1}^n\frac{1}{es_k^3\sqrt{(es_k+1)!}}=0.05896925858439456~, \\
              \textnormal{it is well approximated because}\ \ \frac{1}{es_k^3\sqrt{(es_k+1)!}}=4.945\cdot10^{-118}~.
            \end{multline*}
      \end{itemize}
\end{enumerate}

\section{Smarandache--Kurepa Constants}

The authors did not prove the convergence towards each constant. We let it as possible research for the interested readers. With the program $\emph{SK}$, \ref{Program SK(k,p)}, calculate vectors top 25 terms numbers containing Smarandache--Kurepa, $\emph{sk1}=\emph{SK}(1,p)$, $\emph{sk2}=\emph{SK}(2,p)$ and $\emph{sk3}=\emph{SK}(3,p)$, where
\begin{multline*}
  p=(2\ \ 3\ \ 5\ \ 7\ \ 11\ \ 13\ \ 17\ \ 19\ \ 23\ \ 29\ \ 31\ \ 37\ \ 41\ \ 43\ \ 47\ \ 53\ \ 59\ \ 61\ \ 67\ \ 71\\
  73\ \ 79\ \ 83\ \ 89\ \ 97)^\textrm{T}~.
\end{multline*}
Vectors $\emph{sk1}$ (\ref{Vector sk1}), $\emph{sk2}$ (\ref{Vector sk2})  and $\emph{sk3}$ (\ref{Vector sk3}) has 25 terms.

\begin{enumerate}
  \item The f{}irst constant Smarandache--Kurepa is def{}ined as
      \begin{equation}\label{CSK01}
        SK_1=\sum_{k=1}^\infty\frac{1}{sk_k!}~.
      \end{equation}
      \begin{prog}\label{Program SK1} for the approximation of $SK_1$.
        \begin{tabbing}
          $SK_1(sk):=$\=\vline\ $SK\leftarrow0$\\
          \>\vline\ $f$\=$or\ k\in1..\emph{\emph{last}(sk)}$\\
          \>\vline\ \>\ $SK\leftarrow SK+\dfrac{1}{sk_k!}\ \ \emph{if}\ \ sk_k\neq-1$\\
          \>\vline\ $\emph{return}\ SK$\\
        \end{tabbing}
       \end{prog}
       Thus is obtained:
        \begin{itemize}
          \item $SK_1(sk1)\emph{float},20\rightarrow0.55317460526232666816\ldots$~, \\it is well approximated because
              \[
               \frac{1}{sk1_{\emph{last}(sk1)}!}=1.957\cdot10^{-20}~;
              \]
          \item $SK_1(sk2)\emph{float},20\rightarrow0.55855654987879293658\ldots$~, \\it is well approximated because
              \[
               \frac{1}{sk2_{\emph{last}(sk2)}!}=3.8\cdot10^{-36}~;
              \]
          \item $SK_1(sk3)\emph{float},20\rightarrow0.55161215327881994551\ldots$~, \\it is well approximated because
              \[
               \frac{1}{sk3_{\emph{last}(sk3)}!}=7.117\cdot10^{-52}~.
              \]
        \end{itemize}
  \item The second constant Smarandache--Kurepa is def{}ined as
      \begin{equation}\label{CSK02}
        SK_2=\sum_{k=1}^\infty\frac{sk_k}{k!}~.
      \end{equation}
      \begin{prog}\label{Program SK2} for the approximation of $SK_2$.
        \begin{tabbing}
          $SK_2(sk):=$\=\vline\ $SK\leftarrow0$\\
          \>\vline\ $f$\=$or\ k\in1..\emph{\emph{last}(sk)}$\\
          \>\vline\ \>\ $SK\leftarrow SK+\dfrac{sk_k}{k!}\ \ \emph{if}\ \ sk_k\neq-1$\\
          \>\vline\ $\emph{return}\ \ SK$\\
        \end{tabbing}
       \end{prog}
       Thus is obtained:
        \begin{itemize}
          \item $SK_2(sk1)\emph{float},20\rightarrow2.967851980516919686\ldots$~, \\it is well approximated because
              \[
               \frac{sk1_{\emph{last}(sk1)}}{\emph{last}(sk1)!}=1.354\cdot10^{-24}~;
              \]
          \item $SK_2(sk2)\emph{float},20\rightarrow5.5125891876109912425\ldots$~, \\it is well approximated because
              \[
               \frac{sk2_{\emph{last}(sk2)}}{\emph{last}(sk2)!}=2.063\cdot10^{-24}~;
              \]
          \item $SK_2(sk3)\emph{float},20\rightarrow5.222881245790957486\ldots$~, \\it is well approximated because
              \[
               \frac{sk3_{\emph{last}(sk3)}}{\emph{last}(sk3)!}=2.708\cdot10^{-24}~.
              \]
        \end{itemize}
  \item The third constant Smarandache--Kurepa is def{}ined as
      \begin{equation}\label{CSK03}
        SK_3=\sum_{k=1}^\infty\dfrac{1}{\displaystyle\prod_{j=1}^k sk_j}~.
      \end{equation}
      \begin{prog}\label{Program SK3} for the approximation of $SK_3$.
        \begin{tabbing}
          $SK_3(sk):=$\=\vline\ $SK\leftarrow0$\\
          \>\vline\ $f$\=$or\ k\in1..\emph{last}(sk)$\\
          \>\vline\ \>\ $i$\=$f\ sk_k\neq-1$\\
          \>\vline\ \>\ \>\vline\ $\emph{prod}\leftarrow1$\\
          \>\vline\ \>\ \>\vline\ $f$\=$or\ j\in1..k$\\
          \>\vline\ \>\ \>\vline\ \>\ $\emph{prod}\leftarrow \emph{prod}\cdot sk_j\ \ \emph{if}\ \ sk_j\neq-1$\\
          \>\vline\ \>\ \>\vline\ $SK\leftarrow SK+\dfrac{1}{\emph{prod}}$\\
          \>\vline\ $\emph{return}\ \ SK$\\
        \end{tabbing}
       \end{prog}
       Thus is obtained:
        \begin{itemize}
          \item $SK_3(sk1)\emph{float},20\rightarrow0.65011461681321770674\ldots$~, \\it is well approximated because
              \[
               \frac{1}{\abs{\displaystyle\prod_{j=1}^{\emph{last}(sk1)}sk1_j}}=4.301\cdot10^{-26}~;
              \]
          \item $SK_3(sk2)\emph{float},20\rightarrow0.62576709781381269162\ldots$~, \\it is well approximated because
              \[
               \frac{1}{\abs{\displaystyle\prod_{j=1}^{\emph{last}(sk2)}sk2_j}}=1.399\cdot10^{-29}~;
              \]
          \item $SK_3(sk3)\emph{float},20\rightarrow0.6089581283188629847\ldots$~, \\it is well approximated because
              \[
               \frac{1}{\abs{\displaystyle\prod_{j=1}^{\emph{last}(sk3)}sk3_j}}=4.621\cdot10^{-31}~.
              \]
        \end{itemize}
  \item Series
      \begin{equation}\label{CSK04}
        SK_4(\alpha)=\sum_{k=1}^\infty\dfrac{k^\alpha}{\displaystyle\prod_{j=1}^k sk_j}~.
      \end{equation}
      \begin{prog}\label{Program SK4} for the approximation of $SK_4(\alpha)$.
        \begin{tabbing}
          $SK_4(sk,\alpha):=$\=\vline\ $SK\leftarrow0$\\
          \>\vline\ $f$\=$or\ k\in1..\emph{last}(sk)$\\
          \>\vline\ \>\ $i$\=$f\ sk_k\neq-1$\\
          \>\vline\ \>\ \>\vline\ $\emph{prod}\leftarrow1$\\
          \>\vline\ \>\ \>\vline\ $f$\=$or\ j\in1..k$\\
          \>\vline\ \>\ \>\vline\ \>\ $\emph{prod}\leftarrow \emph{prod}\cdot sk_j\ \ \emph{if}\ \ sk_j\neq-1$\\
          \>\vline\ \>\ \>\vline\ $SK\leftarrow SK+\dfrac{k^\alpha}{prod}$\\
          \>\vline\ $\emph{return}\ \ SK$\\
        \end{tabbing}
      \end{prog}
      We def{}ine a function that is value the last term of the series (\ref{CSK04})
      \[
       U4(sk,\alpha):=\frac{\emph{last}(sk)^\alpha}{\abs{\displaystyle\prod_{j=1}^{\emph{last}(sk)}sk_j}}~.
      \]
      Thus is obtained:
      \begin{itemize}
        \item Case $\alpha=1$,
            \begin{itemize}
              \item $SK_4(sk1,\alpha)\emph{float},20\rightarrow0.98149043761308041099\ldots$~, \\it is well approximated because $U4(sk1,\alpha)=1.075\cdot10^{-24}$;
              \item $SK_4(sk2,\alpha)\emph{float},20\rightarrow0.78465913770543477708\ldots$~, \\it is well approximated because $U4(sk2,\alpha)=3.496\cdot10^{-28}$;
              \item $SK_4(sk3,\alpha)\emph{float},20\rightarrow0.77461420238539514113\ldots$~, \\it is well approximated because $U4(sk3,\alpha)=1.155\cdot10^{-29}$;
            \end{itemize}
        \item Case $\alpha=2$,
            \begin{itemize}
              \item $SK_4(sk1,\alpha)\emph{float},20\rightarrow2.08681505420554993\ldots$~, \\it is well approximated because $U4(sk1,\alpha)=2.688\cdot10^{-23}$;
              \item $SK_4(sk2,\alpha)\emph{float},20\rightarrow1.1883623850019734284\ldots$~, \\it is well approximated because $U4(sk2,\alpha)=8.741\cdot10^{-27}$;
              \item $SK_4(sk3,\alpha)\emph{float},20\rightarrow1.2937484108637316754\ldots$~, \\it is well approximated because $U4(sk3,\alpha)=2.888\cdot10^{-28}$;
            \end{itemize}
        \item Case $\alpha=3$,
            \begin{itemize}
              \item $SK_4(sk1,\alpha)\emph{float},20\rightarrow5.9433532880383150933\ldots$~, \\it is well approximated because $U4(sk1,\alpha)=6.721\cdot10^{-22}$;
              \item $SK_4(sk2,\alpha)\emph{float},20\rightarrow2.3331345867616929091\ldots$~, \\it is well approximated because $U4(sk2,\alpha)=2.185\cdot10^{-25}$;
              \item $SK_4(sk3,\alpha)\emph{float},20\rightarrow3.1599744540262403647\ldots$~, \\it is well approximated because $U4(sk3,\alpha)=7.22\cdot10^{-27}$;
            \end{itemize}
        \item Case $\alpha=4$,
            \begin{itemize}
              \item $SK_4(sk1,\alpha)\emph{float},20\rightarrow20.31367425449123713\ldots$~, \\it is well approximated because $U4(sk3,\alpha)=1.68\cdot10^{-20}$;
              \item $SK_4(sk2,\alpha)\emph{float},20\rightarrow6.0605166330133984862\ldots$~, \\it is well approximated because $U4(sk3,\alpha)=5.463\cdot10^{-24}$;
              \item $SK_4(sk3,\alpha)\emph{float},20\rightarrow10.61756990155963527\ldots$~, \\it is well approximated because $U4(sk3,\alpha)=1.805\cdot10^{-25}$.
            \end{itemize}
      \end{itemize}
  \item Series
      \begin{equation}\label{CSK05}
        SK_5=\sum_{k=1}^\infty\frac{(-1)^{k+1}sk_k}{k!}~.
      \end{equation}
      \begin{prog}\label{Program SK5} for the approximation of $SK_5$.
        \begin{tabbing}
          $SK_5(sk):=$\=\vline\ $SK\leftarrow0$\\
          \>\vline\ $f$\=$or\ k\in1..\emph{last}(sk)$\\
          \>\vline\ \>\ $SK\leftarrow SK+\dfrac{(-1)^{k+1}sk_k}{k!}\ \ \emph{if}\ \ sk_k\neq-1$\\
          \>\vline\ $\emph{return}\ \ SK$\\
        \end{tabbing}
      \end{prog}
      We def{}ine a function that is value the last term of the series (\ref{CSK05})
      \[
       U5(sk):=\frac{(-1)^{\emph{last}(sk)+1}\cdot sk_{\emph{last}(sk)}}{\emph{last}(sk)!}~.
      \]
      Thus is obtained:
      \begin{itemize}
        \item $SK_5(sk1)\emph{float},20\rightarrow2.4675046664369494917\ldots$~, \\it is well approximated because $U5(sk1)=1.354\cdot10^{-24}$;
        \item $SK_5(sk2)\emph{float},20\rightarrow0.15980474179291864895\ldots$~, \\it is well approximated because $U5(sk2)=2.063\cdot10^{-24}$;
        \item $SK_5(sk3)\emph{float},20\rightarrow-1.130480977547589544\ldots$~, \\it is well approximated because $U5(sk3)=2.708\cdot10^{-24}$.
      \end{itemize}
  \item Series
      \begin{equation}\label{CSK06}
        SK_6=\sum_{k=1}^\infty\frac{sk_k}{(k+1)!}~.
      \end{equation}
      \begin{prog}\label{Program SK6} for the approximation of $SK_6$.
        \begin{tabbing}
          $SK_6(sk):=$\=\vline\ $SK\leftarrow0$\\
          \>\vline\ $f$\=$or\ k\in1..\emph{last}(sk)$\\
          \>\vline\ \>\ $SK\leftarrow SK+\dfrac{sk_k}{(k+1)!}\ \ \emph{if}\ \ sk_k\neq-1$\\
          \>\vline\ $\emph{return}\ \ SK$\\
        \end{tabbing}
      \end{prog}
      We def{}ine a function that is value the last term of the series (\ref{CSK06})
      \[
       U6(sk):=\frac{sk_{\emph{last}(sk)}}{(last(sk)+1)!}~.
      \]
      Thus is obtained:
      \begin{itemize}
        \item $SK_6(sk1)\emph{float},20\rightarrow1.225145255840940818\ldots$~, \\it is well approximated because $U6(sk1)=5.207\cdot10^{-26}$;
        \item $SK_6(sk2)\emph{float},20\rightarrow2.0767449920846168598\ldots$~, \\it is well approximated because $U6(sk2)=7.935\cdot10^{-26}$;
        \item $SK_6(sk3)\emph{float},20\rightarrow2.0422612109916229114\ldots$~, \\it is well approximated because $U6(sk3)=1.041\cdot10^{-25}$.
      \end{itemize}
  \item Series
      \begin{equation}\label{CSK07}
        SK_7(r)=\sum_{k=r}^\infty\frac{sk_k}{(k+r)!}~.
      \end{equation}
      \begin{prog}\label{Program SK7} for the approximation of $SK_7(r)$.
        \begin{tabbing}
          $SK_7(sk,r):=$\=\vline\ $SK\leftarrow0$\\
          \>\vline\ $f$\=$or\ k\in r..\emph{last}(sk)$\\
          \>\vline\ \>\ $SK\leftarrow SK+\dfrac{sk_k}{(k+r)!}\ \ \emph{if}\ \ sk_k\neq-1$\\
          \>\vline\ $\emph{return}\ \ SK$\\
        \end{tabbing}
      \end{prog}
      We def{}ine a function that is value the last term of the series (\ref{CSK07})
      \[
       U7(sk,r):=\frac{sk_{\emph{last}(sk)}}{(last(sk)+r)!}~.
      \]
      Thus is obtained:
      \begin{itemize}
        \item Case $r=1$
            \begin{itemize}
              \item $SK_7(sk1,r)\emph{float},20\rightarrow1.225145255840940818\ldots$~, \\it is well approximated because $U7(sk1,r)=5.207\cdot10^{-26}$;
              \item $SK_7(sk2,r)\emph{float},20\rightarrow2.0767449920846168598\ldots$~, \\it is well approximated because $U7(sk2,r)=7.935\cdot10^{-26}$;
              \item $SK_7(sk3,r)\emph{float},20\rightarrow2.0422612109916229114\ldots$~, \\it is well approximated because $U7(sk3,r)=1.041\cdot10^{-25}$;
            \end{itemize}
        \item Case $r=2$
            \begin{itemize}
              \item $SK_7(sk1,r)\emph{float},20\rightarrow0.042873028085536375389\ldots$~, \\it is well approximated because $U7(sk1,r)=1.929\cdot10^{-27}$;
              \item $SK_7(sk2,r)\emph{float},20\rightarrow0.25576855146026900397\ldots$~, \\it is well approximated because $U7(sk2,r)=2.939\cdot10^{-27}$;
              \item $SK_7(sk3,r)\emph{float},20\rightarrow0.25678656434224640866\ldots$~, \\it is well approximated because $U7(sk3,r)=3.857\cdot10^{-27}$;
            \end{itemize}
        \item Case $\alpha=3$,
            \begin{itemize}
              \item $SK_7(sk1,r)\emph{float},20\rightarrow0.0068964092695284835701\ldots$~, \\it is well approximated because $U7(sk1,r)=6.888\cdot10^{-29}$;
              \item $SK_7(sk2,r)\emph{float},20\rightarrow0.0077613102362227910095\ldots$~, \\it is well approximated because $U7(sk2,r)=1.05\cdot10^{-28}$;
              \item $SK_7(sk3,r)\emph{float},20\rightarrow0.00094342579874230766729\ldots$~, \\it is well approximated because $U7(sk3,r)=1.378\cdot10^{-28}$.
            \end{itemize}
      \end{itemize}
  \item Series
      \begin{equation}\label{CSK08}
        SK_8(r)=\sum_{k=r}^\infty\frac{sk_k}{(k-r)!}~.
      \end{equation}
      \begin{prog}\label{Program SK8} for the approximation of $SK_8(r)$.
        \begin{tabbing}
          $SK_8(sk,r):=$\=\vline\ $SK\leftarrow0$\\
          \>\vline\ $f$\=$or\ k\in r..\emph{last}(sk)$\\
          \>\vline\ \>\ $SK\leftarrow SK+\dfrac{sk_k}{(k-r)!}\ \ \emph{if}\ \ sk_k\neq-1$\\
          \>\vline\ $\emph{return}\ \ SK$\\
        \end{tabbing}
      \end{prog}
      We def{}ine a function that is value the last term of the series (\ref{CSK08})
      \[
       U8(sk,r):=\frac{sk_{\emph{last}(sk)}}{(last(sk)-r)!}~.
      \]
      Thus is obtained:
      \begin{itemize}
        \item Case $r=1$
            \begin{itemize}
              \item $SK_8(sk1,r)\emph{float},20\rightarrow5.2585108358721020744\ldots$~, \\it is well approximated because $U8(sk1,r)=3.385\cdot10^{-23}$;
              \item $SK_8(sk2,r)\emph{float},20\rightarrow10.245251373119774594\ldots$~, \\it is well approximated because $U8(sk2,r)=5.1576\cdot10^{-23}$;
              \item $SK_8(sk3,r)\emph{float},20\rightarrow8.9677818084073938106\ldots$~, \\it is well approximated because $U7(sk3,r)=6.769\cdot10^{-23}$;
            \end{itemize}
        \item Case $r=2$
            \begin{itemize}
              \item $SK_8(sk1,r)\emph{float},20\rightarrow8.052817310007447497\ldots$~, \\it is well approximated because $U8(sk1,r)=8.123\cdot10^{-22}$;
              \item $SK_8(sk2,r)\emph{float},20\rightarrow12.414832179662003511\ldots$~, \\it is well approximated because $U7(sk2,r)=1.24\cdot10^{-21}$;
              \item $SK_8(sk3,r)\emph{float},20\rightarrow9.3374979438253485524\ldots$~, \\it is well approximated because $U7(sk3,r)=1.625\cdot10^{-21}$;
            \end{itemize}
        \item Case $\alpha=3$,
            \begin{itemize}
              \item $SK_8(sk1,r)\emph{float},20\rightarrow13.276751543252094323\ldots$~, \\it is well approximated because $U8(sk1,r)=1.868\cdot10^{-20}$;
              \item $SK_8(sk2,r)\emph{float},20\rightarrow10.795894008560432223\ldots$~, \\it is well approximated because $U8(sk2,r)=2.847\cdot10^{-20}$;
              \item $SK_8(sk3,r)\emph{float},20\rightarrow8.7740584270103739767$, \\it is well approximated because $U8(sk3,r)=3.737\cdot10^{-20}$.
            \end{itemize}
      \end{itemize}
  \item Series
      \begin{equation}\label{CSK09}
        SK_9=\sum_{k=1}^\infty\dfrac{1}{\displaystyle\sum_{j=1}^k sk_j!}~.
      \end{equation}
      \begin{prog}\label{Program SK9} for the approximation of $SK_9$.
        \begin{tabbing}
          $SK_9(sk):=$\=\vline\ $SK\leftarrow0$\\
          \>\vline\ $f$\=$or\ k\in1..\emph{last}(sk)$\\
          \>\vline\ \>\ $i$\=$f\ sk_k\neq-1$\\
          \>\vline\ \>\ \>\vline\ $\emph{sum}\leftarrow0$\\
          \>\vline\ \>\ \>\vline\ $f$\=$or\ j\in1..k$\\
          \>\vline\ \>\ \>\vline\ \>\ $\emph{sum}\leftarrow\emph{sum}+sk_j!\ \ \ \emph{if}\ \ \ sk_j\neq-1$\\
          \>\vline\ \>\ \>\vline\ $SK\leftarrow SK+\dfrac{1}{\emph{sum}}$\\
          \>\vline\ $\emph{return}\ \ SK$\\
        \end{tabbing}
       \end{prog}

       \begin{prog} that is value the last term of the series (\ref{CSK09}).
         \begin{tabbing}
           $U9(sk):=$\=\vline\ $sum\leftarrow0$\\
           \>\vline\ $f$\=$or\ j\in1..\emph{last}(sk)$\\
           \>\vline\ \>\ $\emph{sum}\leftarrow\emph{sum}+sk_j!\ \ \ \emph{if}\ \ \ sk_j\neq-1$\\
           \>\vline\ $\emph{return}\ \ \dfrac{1}{sum}$\\
         \end{tabbing}
       \end{prog}
       Thus is obtained:
        \begin{itemize}
          \item $SK_9(sk1)\emph{float},20\rightarrow0.54135130818666812662\ldots$~, \\it is well approximated because $U9(sk1)\rightarrow7.748\cdot10^{-21}$;
          \item $SK_9(sk2)\emph{float},20\rightarrow0.51627681711181976487\ldots$~, \\it is well approximated because $U9(sk2)\rightarrow6.204\cdot10^{-37}$;
          \item $SK_9(sk3)\emph{float},20\rightarrow0.50404957787232673113\ldots$~, \\it is well approximated because $U9(sk3)\rightarrow1.768\cdot10^{-52}$.
        \end{itemize}
  \item Series
      \begin{equation}\label{CSK10}
        SK_{10}=\sum_{k=1}^\infty\dfrac{1}{sk_k^\alpha\sqrt{sk_k!}}~.
      \end{equation}
      \begin{prog}\label{Program SK10} for the approximation of $SK_{10}$.
        \begin{tabbing}
          $SK_{10}(sk,\alpha):=$\=\vline\ $SK\leftarrow0$\\
          \>\vline\ $f$\=$or\ k\in1..\emph{last}(sk)$\\
          \>\vline\ \>\ $SK\leftarrow SK+\dfrac{1}{sk_k^\alpha\sqrt{s_k!}}\ \ \emph{if}\ \ sk_k\neq-1$\\
          \>\vline\ $\emph{return}\ \ SK$\\
        \end{tabbing}
       \end{prog}

       \begin{prog} that is value the last term of the series (\ref{CSK10}).
         \begin{tabbing}
           $U10(sk,\alpha):=$\=\vline\ $f$\=$or\ k=\emph{last}(sk)..1$\\
           \>\vline\ \>\ $\emph{return}\ \ \dfrac{1}{sk_k^\alpha\sqrt{sk_k!}}\ \ \ \emph{if}\ \ sk_k\neq-1$\\
         \end{tabbing}
       \end{prog}
       Thus is obtained:
      \begin{itemize}
        \item Case $\alpha=1$,
            \begin{itemize}
              \item $SK_{10}(sk1,\alpha)\emph{float},12\rightarrow0.439292810686\ldots$~, \\it is well approximated because $U10(sk1,\alpha)=6.662\cdot10^{-12}$;
              \item $SK_{10}(sk2,\alpha)\emph{float},20\rightarrow0.44373908470389981298\ldots$~, \\it is well approximated because $U10(sk2,\alpha)=6.092\cdot10^{-20}$;
              \item $SK_{10}(sk3,\alpha)\emph{float},20\rightarrow0.43171671029085234099\ldots$~, \\it is well approximated because $U10(sk3,\alpha)=6.352\cdot10^{-28}$;
            \end{itemize}
        \item Case $\alpha=2$,
            \begin{itemize}
              \item $SK_{10}(sk1,\alpha)\emph{float},13\rightarrow0.1958316244233\ldots$~, \\it is well approximated because $U10(sk1,\alpha)=3.172\cdot10^{-13}$;
              \item $SK_{10}(sk2,\alpha)\emph{float},20\rightarrow0.19720311371907892905\ldots$~, \\it is well approximated because $U10(sk2,\alpha)=1.904\cdot10^{-21}$;
              \item $SK_{10}(sk3,\alpha)\emph{float},20\rightarrow0.19458905271804637084\ldots$~, \\it is well approximated because $U10(sk3,\alpha)=1.512\cdot10^{-29}$;
            \end{itemize}
        \item Case $\alpha=3$,
            \begin{itemize}
              \item $SK_{10}(sk1,\alpha)\emph{float},14\rightarrow0.09273531642709\ldots$~, \\it is well approximated because $U10(sk1,\alpha)=1.511\cdot10^{-14}$;
              \item $SK_{10}(sk2,\alpha)\emph{float},20\rightarrow0.093089207952192019765\ldots$~, \\it is well approximated because $U10(sk2,\alpha)=5.949\cdot10^{-23}$;
              \item $SK_{10}(sk3,\alpha)\emph{float},20\rightarrow0.092531651675929962703\ldots$~, \\it is well approximated because $U10(sk3,\alpha)=3.601\cdot10^{-31}$;
            \end{itemize}
        \item Case $\alpha=4$,
            \begin{itemize}
              \item $SK_{10}(sk1,\alpha)\emph{float},16\rightarrow0.0452068407230367\ldots$~, \\it is well approximated because $U10(sk3,\alpha)=7.194\cdot10^{-16}$;
              \item $SK_{10}(sk2,\alpha)\emph{float},20\rightarrow0.045290737732775804922\ldots$~, \\it is well approximated because $U10(sk3,\alpha)=1.859\cdot10^{-24}$;
              \item $SK_{10}(sk3,\alpha)\emph{float},20\rightarrow0.045173453286382795647\ldots$~, \\it is well approximated because $U10(sk3,\alpha)=8.574\cdot10^{-33}$.
            \end{itemize}
      \end{itemize}
  \item Series
      \begin{equation}\label{CSK11}
        SK_{11}=\sum_{k=1}^\infty\dfrac{1}{sk_k^\alpha\sqrt{(sk_k+1)!}}~.
      \end{equation}
      \begin{prog}\label{Program SK11} for the approximation of $SK_{11}$.
        \begin{tabbing}
          $SK_{11}(sk,\alpha):=$\=\vline\ $SK\leftarrow0$\\
          \>\vline\ $f$\=$or\ k\in1..\emph{last}(sk)$\\
          \>\vline\ \>\ $SK\leftarrow SK+\dfrac{1}{sk_k^\alpha\sqrt{(s_k+1)!}}\ \ \emph{if}\ \ sk_k\neq-1$\\
          \>\vline\ $\emph{return}\ \ SK$\\
        \end{tabbing}
       \end{prog}

       \begin{prog} that is value the last term of the series (\ref{CSK10}).
         \begin{tabbing}
           $U11(sk,\alpha):=$\=\vline\ $f$\=$or\ k=last(sk)..1$\\
           \>\vline\ \>\ $\emph{return}\ \dfrac{1}{sk_k^\alpha\sqrt{(sk_k+1)!}}\ \ \emph{if}\ \ sk_k\neq-1$\\
         \end{tabbing}
       \end{prog}
       Thus is obtained:
      \begin{itemize}
        \item Case $\alpha=1$,
            \begin{itemize}
              \item $SK_{11}(sk1,\alpha)\emph{float},12\rightarrow0.240518730353\ldots$~, \\it is well approximated because $U11(sk1,\alpha)=6.662\cdot10^{-12}$;
              \item $SK_{11}(sk2,\alpha)\emph{float},20\rightarrow0.24277337011690480832\ldots$~, \\it is well approximated because $U11(sk2,\alpha)=6.092\cdot10^{-20}$;
              \item $SK_{11}(sk3,\alpha)\emph{float},20\rightarrow0.23767438713448743589\ldots$~, \\it is well approximated because $U11(sk3,\alpha)=6.352\cdot10^{-28}$;
            \end{itemize}
        \item Case $\alpha=2$,
            \begin{itemize}
              \item $SK_{11}(sk1,\alpha)\emph{float},13\rightarrow0.1102441365259\ldots$~, \\it is well approximated because $U11(sk1,\alpha)=3.172\cdot10^{-13}$;
              \item $SK_{11}(sk2,\alpha)\emph{float},20\rightarrow0.11087662573522063082\ldots$~, \\it is well approximated because $U11(sk2,\alpha)=1.904\cdot10^{-21}$;
              \item $SK_{11}(sk3,\alpha)\emph{float},20\rightarrow0.10977783145765213226\ldots$~, \\it is well approximated because $U11(sk3,\alpha)=1.512\cdot10^{-29}$;
            \end{itemize}
        \item Case $\alpha=3$,
            \begin{itemize}
              \item $SK_{11}(sk1,\alpha)\emph{float},14\rightarrow0.05291501051360\ldots$~, \\it is well approximated because $U11(sk1,\alpha)=1.511\cdot10^{-14}$;
              \item $SK_{11}(sk2,\alpha)\emph{float},20\rightarrow0.053071452693399642756\ldots$~, \\it is well approximated because $U11(sk2,\alpha)=5.949\cdot10^{-23}$;
              \item $SK_{11}(sk3,\alpha)\emph{float},20\rightarrow0.052838582271346066972\ldots$~, \\it is well approximated because $U11(sk3,\alpha)=3.601\cdot10^-31$;
            \end{itemize}
        \item Case $\alpha=4$,
            \begin{itemize}
              \item $SK_{11}(sk1,\alpha)\emph{float},16\rightarrow0.0259576233294432\ldots$~, \\it is well approximated because $U11(sk3,\alpha)=7.194\cdot10^{-16}$;
              \item $SK_{11}(sk2,\alpha)\emph{float},20\rightarrow0.025993845540992037418\ldots$~, \\it is well approximated because $U11(sk3,\alpha)=1.859\cdot10^{-24}$;
              \item $SK_{11}(sk3,\alpha)\emph{float},20\rightarrow0.025945091303465934837\ldots$~, \\it is well approximated because $U11(sk3,\alpha)=8.574\cdot10^{-33}$.
            \end{itemize}
      \end{itemize}
\end{enumerate}

\section{Smarandache--Wagstaf{}f Constants}

The authors did not prove the convergence towards each constant. We let it as possible research for the interested readers. With the program $SW$ \ref{Program SW(k,p)} calculate vectors top 25 terms numbers containing Smarandache--Wagstaf{}f, $\emph{sw1}=\emph{SW}(1,p)$, $\emph{sw2}=\emph{SW}(2,p)$ and $\emph{sw3}=\emph{SW}(3,p)$, where
\begin{multline*}
  p=(2\ \ 3\ \ 5\ \ 7\ \ 11\ \ 13\ \ 17\ \ 19\ \ 23\ \ 29\ \ 31\ \ 37\ \ 41\ \ 43\ \ 47\ \ 53\ \ 59\ \ 61\ \ 67\ \ 71\\
  73\ \ 79\ \ 83\ \ 89\ \ 97)^\textrm{T}~.
\end{multline*}
Vectors $\emph{sw1}$ (\ref{Vector sw1}), $\emph{sw2}$ (\ref{Vector sw2})  and $\emph{sw3}$ (\ref{Vector sw3}) has 25 terms.

In a similar manner with Smarandache--Kurepa constants were obtained Smarandache--Wagstaf{}f constants, which are found in following table.

\begin{center}
 \begin{longtable}{|l|l|l|}
   \caption{Smarandache--Wagstaf{}f constants}\\
   \hline
   Name  &  Constant value & Value the last term \\
   \hline
  \endfirsthead
   \hline
   Name  &  Constant value & Value the last term \\
   \hline
  \endhead
   \hline \multicolumn{3}{r}{\textit{Continued on next page}} \\
  \endfoot
   \hline
  \endlastfoot
   $SW_1(sw1)$ & 0.55158730367497730335\ldots & $1.389\cdot10^{-3}$ \\
   $SW_1(sw2)$ & 0.71825397034164395277\ldots & $1.216\cdot10^{-34}$ \\
   $SW_1(sw3)$ & 0.70994819257251365270\ldots & $2.9893\cdot10^{-50}$ \\ \hline
   $SW_2(sw1)$ & 1.0343637569611291909\ldots & $3.868\cdot10^{-25}$ \\
   $SW_2(sw2)$ & 4.2267823464172704922\ldots & $1.999\cdot10^{-24}$ \\
   $SW_2(sw3)$ & 5.1273356604617316278\ldots & $2.643\cdot10^{-24}$ \\ \hline
   $SW_3(sw1)$ & 0.65219770185345831168\ldots & $6.208\cdot10^{-23}$ \\
   $SW_3(sw2)$ & 0.54968878346863715478\ldots & $1.985\cdot10^{-30}$ \\
   $SW_3(sw3)$ & 0.54699912527156976558\ldots & $6.180\cdot10^{-32}$ \\ \hline
   $SW_4(sw1,1)$ & 1.8199032834559367993\ldots & $1.552\cdot10^{-21}$ \\
   $SW_4(sw2,1)$ & 0.87469626917369886975\ldots & $4.963\cdot10^{-29}$ \\
   $SW_4(sw3,1)$ & 0.81374377609424443467\ldots & $1.545\cdot10^{-30}$ \\ \hline
   $SW_4(sw1,2)$ & 6.5303985125207189262\ldots & $3.880\cdot10^{-20}$ \\
   $SW_4(sw2,2)$ & 1.8815309698588492643\ldots & $1.241\cdot10^{-27}$ \\
   $SW_4(sw3,2)$ & 1.4671407531614048561\ldots & $3.862\cdot10^{-29}$ \\ \hline
   $SW_4(sw1,3)$ & 29.836629842971767949\ldots & $9.700\cdot10^{-19}$ \\
   $SW_4(sw2,3)$ & 5.4602870746051287154\ldots & $3.102\cdot10^{-26}$ \\
   $SW_4(sw3,3)$ & 3.1799919620918477289\ldots & $9.656\cdot10^{-28}$ \\ \hline
   $SW_4(sw1,4)$ & 161.01437206742466933\ldots & $2.425\cdot10^{-17}$ \\
   $SW_4(sw2,4)$ & 19.652380984010238861\ldots & $7.755\cdot10^{-25}$ \\
   $SW_4(sw3,4)$ & 8.0309117554069796125\ldots & $2.414\cdot10^{-26}$ \\ \hline
   $SW_5(sw1)$ & $-0.96564681546640958928\ldots$ & $3.868\cdot10^{-25}$ \\
   $SW_5(sw2)$ & 1.8734063576918609871\ldots & $1.999\cdot10^{-24}$ \\
   $SW_5(sw3)$ & 2.3298081028442529042\ldots & $2.643\cdot10^{-24}$ \\ \hline
   $SW_6(sw1)$ & 0.33901668392958325553\ldots & $1.488\cdot10^{-26}$ \\
   $SW_6(sw2)$ & 1.8764315154663871518\ldots & $7.687\cdot10^{-26}$ \\
   $SW_6(sw3)$ & 2.0885585895284377234\ldots & $1.017\cdot10^{-25}$ \\ \hline
   $SW_7(sw1,1)$ & 0.33901668392958325553\ldots & $1.48\cdot10^{-26}$ \\
   $SW_7(sw2,1)$ & 1.8764315154663871518\ldots & $7.687\cdot10^{-26}$ \\
   $SW_7(sw3,1)$ & 2.0885585895284377234\ldots & $1.0167\cdot10^{-25}$ \\ \hline
   $SW_7(sw1,2)$ & 0.084141103378415065393\ldots & $5.510\cdot10^{-28}$ \\
   $SW_7(sw2,2)$ & 0.090257150078331834732\ldots & $2.847\cdot10^{-27}$ \\
   $SW_7(sw3,2)$ & 0.13102847128229505811\ldots & $3.765\cdot10^{-27}$ \\ \hline
   $SW_7(sw1,3)$ & 0.00010061233843047075978\ldots & $1.968\cdot10^{-29}$ \\
   $SW_7(sw2,3)$ & 0.0009621231898758737862\ldots & $1.017\cdot10^{-28}$ \\
   $SW_7(sw3,3)$ & 0.0075668414805812168739\ldots & $1.345\cdot10^{-28}$ \\ \hline
   $SW_8(sw1,1)$ & 2.173961760187995128\ldots & $9.670\cdot10^{-24}$ \\
   $SW_8(sw2,1)$ & 5.9782465282928175363\ldots & $4.996\cdot10^{-23}$ \\
   $SW_8(sw3,1)$ & 8.9595673801550634215\ldots & $6.608\cdot10^{-23}$ \\ \hline
   $SW_8(sw1,2)$ & 2.7111929711924074628\ldots & $2.320\cdot10^{-22}$ \\
   $SW_8(sw2,2)$ & 5.315290260964323833\ldots & $1.199\cdot10^{-21}$ \\
   $SW_8(sw3,2)$ & 12.533730523303497135\ldots & $1.586\cdot10^{-21}$ \\ \hline
   $SW_8(sw1,3)$ & 2.2288452763809560596\ldots & $5.338\cdot10^{-21}$ \\
   $SW_8(sw2,3)$ & 8.3156791068197439503\ldots & $2.758\cdot10^{-20}$ \\
   $SW_8(sw3,3)$ & 20.136561600709999638\ldots & $3.648\cdot10^{-20}$ \\ \hline
   $SW_9(sw1)$ & 0.54531085561770668291\ldots & $5.872\cdot10^{-19}$ \\
   $SW_9(sw2)$ & 0.32441910792206133261\ldots & $3.028\cdot10^{-36}$ \\
   $SW_9(sw3)$ & 0.32292213990139703594\ldots & $4.921\cdot10^{-51}$ \\ \hline
   $SW_{10}(sw1,1)$ & 0.4310692254141283029\ldots & $6.211\cdot10^{-3}$ \\
   $SW_{10}(sw2,1)$ & 0.56715198854304553087\ldots & $3.557\cdot10^{-19}$ \\
   $SW_{10}(sw3,1)$ & 0.54969992215171005647\ldots & $4.217\cdot10^{-27}$ \\ \hline
   $SW_{10}(sw1,2)$ & 0.19450893949849271072\ldots & $1.035\cdot10^{-3}$ \\
   $SW_{10}(sw2,2)$ & 0.23986986064238086601\ldots & $1.148\cdot10^{-20}$ \\
   $SW_{10}(sw3,2)$ & 0.23631653581644449211\ldots & $1.029\cdot10^{-28}$ \\ \hline
   $SW_{10}(sw1,3)$ & 0.092521713489107819791\ldots & $1.726\cdot10^{-4}$ \\
   $SW_{10}(sw2,3)$ & 0.107642020542344188\ldots & $3.702\cdot10^{-22}$ \\
   $SW_{10}(sw3,3)$ & 0.1069237140845950118\ldots & $2.509\cdot10^{-30}$ \\ \hline
   $SW_{10}(sw1,4)$ & 0.045172218019205576526\ldots & $2.876\cdot10^{-5}$ \\
   $SW_{10}(sw2,4)$ & 0.050212320370560015101\ldots & $1.194\cdot10^{-23}$ \\
   $SW_{10}(sw3,4)$ & 0.050067728660537151051\ldots & $6.119\cdot10^{-32}$ \\ \hline
   $SW_{11}(sw1,1)$ & 0.23745963081186272553\ldots & $6.211\cdot10^{-3}$ \\
   $SW_{11}(sw2,1)$ & 0.30550101247582884341\ldots & $3.557\cdot10^{-19}$ \\
   $SW_{11}(sw3,1)$ & 0.29831281390164911686\ldots & $4.217\cdot10^{-27}$ \\ \hline
   $SW_{11}(sw1,2)$ & 0.10975122611929470037\ldots & $1.035\cdot10^{-3}$ \\
   $SW_{11}(sw2,2)$ & 0.13243168669643861815\ldots & $1.148\cdot10^{-20}$ \\
   $SW_{11}(sw3,2)$ & 0.13097335030519627439\ldots & $1.029\cdot10^{-28}$ \\ \hline
   $SW_{11}(sw1,3)$ & 0.052835278683252704210\ldots & $1.726\cdot10^{-4}$ \\
   $SW_{11}(sw2,3)$ & 0.060395432210142668039\ldots & $3.702\cdot10^{-22}$ \\
   $SW_{11}(sw3,3)$ & 0.060101248246962299788\ldots & $2.509\cdot10^{-30}$ \\ \hline
   $SW_{11}(sw1,4)$ & 0.025944680389944764723\ldots & $2.876\cdot10^{-5}$ \\
   $SW_{11}(sw2,4)$ & 0.028464731565636192167\ldots & $1.194\cdot10^{-23}$ \\
   $SW_{11}(sw3,4)$ & 0.028405588029145682932\ldots & $6.119\cdot10^{-32}$ \\
  \hline
\end{longtable}
\end{center}

\section{Smarandache Ceil Constants}

The authors did not prove the convergence towards each constant. We let it as possible research for the interested readers. With the program $Sk$ \ref{Program Sk} calculate vectors top 100 terms numbers containing Smarandache Ceil, $sk1:=Sk(100,1)$, $sk2:=Sk(100,2)$ , $sk3:=Sk(100,3)$, $sk4:=Sk(100,4)$, $sk5:=Sk(100,5)$ and $sk6=Sk(100,6)$ given on page \pageref{sk1-sk6}.

\begin{equation}\label{Sk1}
  Sk_1(sk):=\sum_{k=1}^{\emph{last}(sk)}\frac{1}{sk_k!}\approx\sum_{k=1}^{\infty}\frac{1}{sk_k!}~.
\end{equation}

\begin{equation}\label{Sk2}
  Sk_2(sk):=\sum_{k=1}^{\emph{last}(sk)}\frac{sk_k}{k!}\approx\sum_{k=1}^{\infty}\frac{sk_k}{k!}~.
\end{equation}

\begin{equation}\label{Sk3}
  Sk_3(sk):=\sum_{k=1}^{\emph{last}(sk)}\frac{1}{\displaystyle\prod_{j=1}^k sk_j}\approx
  \sum_{k=1}^{\infty}\frac{1}{\displaystyle\prod_{j=1}^k sk_j}~.
\end{equation}

\begin{equation}\label{Sk4}
  Sk_4(sk,\alpha):=\sum_{k=1}^{\emph{last}(sk)}\frac{k^\alpha}{\displaystyle\prod_{j=1}^k sk_j}\approx
  \sum_{k=1}^{\infty}\frac{k^\alpha}{\displaystyle\prod_{j=1}^k sk_j}~.
\end{equation}

\begin{equation}\label{Sk5}
  Sk_5(sk):=\sum_{k=1}^{\emph{last}(sk)}\frac{(-1)^{k+1}sk_k}{k!}\approx\sum_{k=1}^{\infty}\frac{(-1)^{k+1}sk_k}{k!}~.
\end{equation}

\begin{equation}\label{Sk6}
  Sk_6(sk):=\sum_{k=1}^{\emph{last}(sk)}\frac{sk_k}{(k+1)!}\approx\sum_{k=1}^{\infty}\frac{sk_k}{(k+1)!}~.
\end{equation}

\begin{equation}\label{Sk7}
  Sk_7(sk,r):=\sum_{k=r}^{\emph{last}(sk)}\frac{sk_k}{(k+r)!}\approx\sum_{k=r}^{\infty}\frac{sk_k}{(k+r)!}~.
\end{equation}

\begin{equation}\label{Sk8}
  Sk_8(sk,r):=\sum_{k=r}^{\emph{last}(sk)}\frac{sk_k}{(k-r)!}\approx\sum_{k=r}^{\infty}\frac{sk_k}{(k-r)!}~.
\end{equation}

\begin{equation}\label{Sk9}
  Sk_9(sk):=\sum_{k=1}^{\emph{last}(sk)}\frac{1}{\displaystyle\sum_{j=1}^k sk_j!}\approx
  \sum_{k=1}^{\infty}\frac{1}{\displaystyle\sum_{j=1}^k sk_j!}~.
\end{equation}

\begin{equation}\label{Sk10}
  Sk_{10}(sk,\alpha):=\sum_{k=1}^{\emph{last}(sk)}\frac{1}{sk_k^\alpha\sqrt{sk_k!}}\approx
  \sum_{k=1}^{\infty}\frac{1}{sk_k^\alpha\sqrt{sk_k!}}~.
\end{equation}

\begin{equation}\label{Sk11}
  Sk_{11}(sk,\alpha):=\sum_{k=1}^{\emph{last}(sk)}\frac{1}{sk_k^\alpha\sqrt{(sk_k+1)!}}\approx
  \sum_{k=1}^{\infty}\frac{1}{sk_k^\alpha\sqrt{(sk_k+1)!}}~.
\end{equation}
In the formulas (\ref{Sk1}--\ref{Sk11}) will replace $sk$ with $sk1$ or $sk2$ or \ldots $sk6$.

\begin{center}
 \begin{longtable}{|l|l|l|}
   \caption{Smarandache ceil constants}\\
   \hline
   Name  &  Constant value \\
   \hline
  \endfirsthead
   \hline
   Name  &  Constant value \\
   \hline
  \endhead
   \hline \multicolumn{2}{r}{\textit{Continued on next page}} \\
  \endfoot
   \hline
  \endlastfoot
  $Sk_1(sk1)$ & 1.7182818284590452354\\
  $Sk_1(sk2)$ & 2.4393419627293664997\\
  $Sk_1(sk3)$ & 3.1517898773198280566\\
  $Sk_1(sk4)$ & 3.7781762839623110876\\
  $Sk_1(sk5)$ & 4.2378985040968576110\\
  $Sk_1(sk6)$ & 4.6962318374301909444\\ \hline
  $Sk_2(sk1)$ & 2.7182818284590452354\\
  $Sk_2(sk2)$ & 2.6348327418583415529\\
  $Sk_2(sk3)$ & 2.6347831386837383783\\
  $Sk_2(sk4)$ & 2.6347831386836427887\\
  $Sk_2(sk5)$ & 2.6347831386836427887\\
  $Sk_2(sk6)$ & 2.6347831386836427887\\ \hline
  $Sk_3(sk1)$ & 1.7182818284590452354\\
  $Sk_3(sk2)$ & 1.7699772067340966537\\
  $Sk_3(sk3)$ & 1.7701131436269234662\\
  $Sk_3(sk4)$ & 1.7701131436367350613\\
  $Sk_3(sk5)$ & 1.7701131436367350613\\
  $Sk_3(sk6)$ & 1.7701131436367350613\\ \hline
  $Sk_4(sk1,1)$ & 2.7182818284590452354\\
  $Sk_4(sk2,1)$ & 2.9372394121769037872\\
  $Sk_4(sk3,1)$ & 2.9383677132426964633\\
  $Sk_4(sk4,1)$ & 2.9383677134004179370\\
  $Sk_4(sk5,1)$ & 2.9383677134004179370\\
  $Sk_4(sk6,1)$ & 2.9383677134004179370\\ \hline
  $Sk_4(sk1,2)$ & 5.4365636569180904707\\
  $Sk_4(sk2,2)$ & 6.3788472114323090813\\
  $Sk_4(sk3,2)$ & 6.3882499784201737207\\
  $Sk_4(sk4,2)$ & 6.3882499809564429861\\
  $Sk_4(sk5,2)$ & 6.3882499809564429861\\
  $Sk_4(sk6,2)$ & 6.3882499809564429861\\ \hline
  $Sk_4(sk1,3)$ & 13.591409142295226177\\
  $Sk_4(sk2,3)$ & 17.731481467469518346\\
  $Sk_4(sk3,3)$ & 17.810185157161258977\\
  $Sk_4(sk4,3)$ & 17.810185197961806356\\
  $Sk_4(sk5,3)$ & 17.810185197961806356\\
  $Sk_4(sk6,3)$ & 17.810185197961806356\\ \hline
  $Sk_5(sk1)$ & 0.3678794411714423216\\
  $Sk_5(sk2)$ & 0.45129545898907722102\\
  $Sk_5(sk3)$ & 0.45134506216368039562\\
  $Sk_5(sk4)$ & 0.45134506216377598517\\
  $Sk_5(sk5)$ & 0.45134506216377598517\\
  $Sk_5(sk6)$ & 0.45134506216377598517\\ \hline
  $Sk_6(sk1)$ & 1.0\\
  $Sk_6(sk2)$ & 0.98332065600291383328\\
  $Sk_6(sk3)$ & 0.98331514453906903611\\
  $Sk_6(sk4)$ & 0.98331514453906341319\\
  $Sk_6(sk5)$ & 0.98331514453906341319\\
  $Sk_6(sk6)$ & 0.98331514453906341319\\ \hline
  $Sk_7(sk1,2)$\footnote{$SK_7(sk1,1)=SK_6(sk1)$\ldots$SK_7(sk6,1)=SK_6(sk6)$} & 0.11505150487428809797\\
  $Sk_7(sk2,2)$ & 0.11227247442226469528\\
  $Sk_7(sk3,2)$ & 0.11227192327588021556\\
  $Sk_7(sk4,2)$ & 0.11227192327587990317\\
  $Sk_7(sk5,2)$ & 0.11227192327587990317\\
  $Sk_7(sk6,2)$ & 0.11227192327587990317\\ \hline
  $Sk_7(sk1,3)$ & 0.0051030097485761959461\\
  $Sk_7(sk2,3)$ & 0.0047060716126746675110\\
  $Sk_7(sk3,3)$ & 0.0047060215084578966275\\
  $Sk_7(sk4,3)$ & 0.0047060215084578801863\\
  $Sk_7(sk5,3)$ & 0.0047060215084578801863\\
  $Sk_7(sk6,3)$ & 0.0047060215084578801863\\ \hline
  $Sk_7(sk1,4)$ & 0.000114832083181754236600\\
  $Sk_7(sk2,4)$ & 0.000065219594046380693870\\
  $Sk_7(sk3,4)$ & 0.000065215418694983120250\\
  $Sk_7(sk4,4)$ & 0.000065215418694982298187\\
  $Sk_7(sk5,4)$ & 0.000065215418694982298187\\
  $Sk_7(sk6,4)$ & 0.000065215418694982298187\\ \hline
  $Sk_8(sk1,1)$ & 5.4365636569180904707\\
  $Sk_8(sk2,1)$ & 5.1022877129454360911\\
  $Sk_8(sk3,1)$ & 5.1018908875486106943\\
  $Sk_8(sk4,1)$ & 5.1018908875470812615\\
  $Sk_8(sk5,1)$ & 5.1018908875470812615\\
  $Sk_8(sk6,1)$ & 5.1018908875470812615\\ \hline
  $Sk_8(sk1,2)$ & 8.1548454853771357061\\
  $Sk_8(sk2,2)$ & 7.1480978000537264745\\
  $Sk_8(sk3,2)$ & 7.1453200222759486967\\
  $Sk_8(sk4,2)$ & 7.1453200222530072055\\
  $Sk_8(sk5,2)$ & 7.1453200222530072055\\
  $Sk_8(sk6,2)$ & 7.1453200222530072055\\ \hline
  $Sk_8(sk1,3)$ & 10.873127313836180941\\
  $Sk_8(sk2,3)$ & 8.8314441108416899115\\
  $Sk_8(sk3,3)$ & 8.8147774441750232447\\
  $Sk_8(sk4,3)$ & 8.8147774438538423680\\
  $Sk_8(sk5,3)$ & 8.8147774438538423680\\
  $Sk_8(sk6,3)$ & 8.8147774438538423680\\ \hline
  $Sk_9(sk1)$ & 1.4826223630822915238\\
  $Sk_9(sk2)$ & 1.5446702350540098246\\
  $Sk_9(sk3)$ & 1.5446714960716397966\\
  $Sk_9(sk4)$ & 1.5446714960716397966\\
  $Sk_9(sk5)$ & 1.5446714960716397966\\
  $Sk_9(sk6)$ & 1.5446714960716397966\\ \hline
  $Sk_{10}(sk1,1)$ & 1.5680271290107037107\\
  $Sk_{10}(sk2,1)$ & 2.1485705607791708605\\
  $Sk_{10}(sk3,1)$ & 2.7064911009876044790\\
  $Sk_{10}(sk4,1)$ & 3.1511717572816964677\\
  $Sk_{10}(sk5,1)$ & 3.4599016039136594552\\
  $Sk_{10}(sk6,1)$ & 3.7624239581989503403\\ \hline
  $Sk_{10}(sk1,2)$ & 1.2399748241535239012\\
  $Sk_{10}(sk2,2)$ & 1.4820295340881653857\\
  $Sk_{10}(sk3,2)$ & 1.7198593859973659544\\
  $Sk_{10}(sk4,2)$ & 1.9302588979239212805\\
  $Sk_{10}(sk5,2)$ & 2.0953127335005767356\\
  $Sk_{10}(sk6,2)$ & 2.2593316697202178974\\ \hline
  $Sk_{10}(sk1,3)$ & 1.1076546756800267505\\
  $Sk_{10}(sk2,3)$ & 1.2156553766098991247\\
  $Sk_{10}(sk3,3)$ & 1.3228498484138567230\\
  $Sk_{10}(sk4,3)$ & 1.4233398208190790660\\
  $Sk_{10}(sk5,3)$ & 1.5087112383659205930\\
  $Sk_{10}(sk6,3)$ & 1.5939101462449901038\\ \hline
  $Sk_{11}(sk1,1)$ & 1.0128939498871834093\\
  $Sk_{11}(sk2,1)$ & 1.3234047869734872073\\
  $Sk_{11}(sk3,1)$ & 1.6249765356359900199\\
  $Sk_{11}(sk4,1)$ & 1.8766243840544798818\\
  $Sk_{11}(sk5,1)$ & 2.0602733507289191532\\
  $Sk_{11}(sk6,1)$ & 2.2415757227314687400\\ \hline
  $Sk_{11}(sk1,2)$ & 0.83957287300941603913\\
  $Sk_{11}(sk2,2)$ & 0.97282437595501178612\\
  $Sk_{11}(sk3,2)$ & 1.10438931993466493290\\
  $Sk_{11}(sk4,2)$ & 1.22381269885269296020\\
  $Sk_{11}(sk5,2)$ & 1.32056051527626897540\\
  $Sk_{11}(sk6,2)$ & 1.41691714458488924910\\ \hline
  $Sk_{11}(sk1,3)$ & 0.76750637030682730577\\
  $Sk_{11}(sk2,3)$ & 0.82803667284162935790\\
  $Sk_{11}(sk3,3)$ & 0.88824280769856038385\\
  $Sk_{11}(sk4,3)$ & 0.94547228007452707047\\
  $Sk_{11}(sk5,3)$ & 0.99514216074174737828\\
  $Sk_{11}(sk6,3)$ & 1.04474683622289388520\\
  \hline
\end{longtable}
\end{center}

\section{Smarandache--Mersenne Constants}

The authors did not prove the convergence towards each constant. We let it as possible research for the interested readers. With the program $\emph{SML}$ \ref{Program SML} and $\emph{SMR}$ \ref{Program SMR} calculate vectors top 40 terms numbers containing Smarandache--Mersenne, $n:=1..40$, $\emph{smlp}:=\emph{SML}(\emph{prime}_n)$, $\emph{sml}\omega:=\emph{SML}(2n-1)$, given on page \pageref{SML(prime)+SML(2n-1)} and $\emph{smrp}:=\emph{SMR}(\emph{prime}_n)$, $\emph{smr}\omega:=\emph{SMR}(2n-1)$, given on page \pageref{SMR(prime)+SMR(2n-1)}. Using programs similar to programs $\emph{SK1}$ \ref{Program SK1}--$\emph{SK11}$ \ref{Program SK11} was calculated constants Smarandache--Mersenne.

\begin{equation}\label{SM1}
  SM_1(sm)=\sum_{k=1}^{\emph{last}(sm)}\dfrac{1}{sm_k!}~.
\end{equation}

\begin{equation}\label{SM2}
  SM_2(sm)=\sum_{k=1}^{\emph{last}(sm)}\dfrac{sm_k}{k!}~.
\end{equation}

\begin{equation}\label{SM3}
  SM_3(sm)=\sum_{k=1}^{\emph{last}(sm)}\dfrac{1}{\prod_{j=1}^k sm_j}~.
\end{equation}

\begin{equation}\label{SM4}
  SM_4(sm,\alpha)=\sum_{k=1}^{\emph{last}(sm)}\dfrac{k^\alpha}{\prod_{j=1}^k sm_j}~,\ \ \textnormal{where}\ \ \alpha\in\Ns~.
\end{equation}

\begin{equation}\label{SM5}
  SM_5(sm)=\sum_{k=1}^{\emph{last}(sm)}\dfrac{(-1)^{k+1}sm_k}{k!}
\end{equation}

\begin{equation}\label{SM6}
  SM_6(sm)=\sum_{k=1}^{\emph{last}(sm)}\dfrac{sm_k}{(k+1)!}
\end{equation}

\begin{equation}\label{SM7}
  SM_7(sm,r)=\sum_{k=r}^{\emph{last}(sm)}\dfrac{sm_k}{(k+r)!}~,\ \ \textnormal{where}\ \ r\in\Ns~.
\end{equation}

\begin{equation}\label{SM8}
  SM_8(sm,r)=\sum_{k=r}^{\emph{last}(sm)}\dfrac{sm_k}{(k-r)!}~,\ \ \textnormal{where}\ \ r\in\Ns~.
\end{equation}

\begin{equation}\label{SM9}
  SM_9(sm)=\sum_{k=1}^{\emph{last}(sm)}\dfrac{1}{\sum_{j=1}^k sm_k!}
\end{equation}

\begin{equation}\label{SM10}
  SM_{10}(sm,\alpha)=\sum_{k=1}^{\emph{last}(sm)}\dfrac{1}{sm_k^\alpha\sqrt{sm_k!}}~,\ \ \textnormal{where}\ \ \alpha\in\Ns~.
\end{equation}

\begin{equation}\label{SM11}
  SM_{11}(sm,\alpha)=\sum_{k=1}^{\emph{last}(sm)}\dfrac{1}{sm_k^\alpha\sqrt{(sm_k+1)!}}~,\ \ \textnormal{where}\ \ \alpha\in\Ns~.
\end{equation}
In the formulas (\ref{SM1}--\ref{SM11}) will replace $\emph{sm}$ with $\emph{smlp}$ or $\emph{sml}\omega$ or $\emph{smrp}$ or $\emph{smr}\omega$.

\begin{center}
 \begin{longtable}{|l|l|l|}
   \caption{Smarandache--Mersenne constants}\\
   \hline
   Name  &  Constant value & Value the last term \\
   \hline
  \endfirsthead
   \hline
   Name  &  Constant value & Value the last term \\
   \hline
  \endhead
   \hline \multicolumn{3}{r}{\textit{Continued on next page}} \\
  \endfoot
   \hline
  \endlastfoot
  $SM_1(smlp)$ & 0.71689296446162367907\ldots & $4.254\cdot10^{-79}$ \\
  $SM_1(sml\omega)$ & 1.7625529455548681481\ldots & $4.902\cdot10^{-47}$ \\
  $SM_1(smrp)$ & 1.5515903329153580293\ldots & $4.254\cdot10^{-79}$ \\
  $SM_1(smr\omega)$ & 2.7279850088298130535\ldots & $1.406\cdot10^{-75}$ \\ \hline
  $SM_2(smlp)$ & 1.8937386297354390132\ldots & $7.109\cdot10^{-47}$ \\
  $SM_2(sml\omega)$ & 2.8580628971756003018\ldots & $4.780\cdot10^{-47}$ \\
  $SM_2(smrp)$ & 0.88435409570308110516\ldots & $7.106\cdot10^{-47}$ \\
  $SM_2(smr\omega)$ & 1.8664817587745350501\ldots & $6.863\cdot10^{-47}$ \\ \hline
  $SM_3(smlp)$ & 0.67122659824508457394\ldots & $5.222\cdot10^{-53}$ \\
  $SM_3(sml\omega)$ & 1.6743798207846741489\ldots & $6.111\cdot10^{-44}$ \\
  $SM_3(smrp)$ & 1.6213314132399371377\ldots & $8.770\cdot10^{-49}$ \\
  $SM_3(smr\omega)$ & 2.7071160606980893623\ldots & $2.243\cdot10^{-31}$ \\ \hline
  $SM_4(smlp,1)$ & 1.5649085118184007572\ldots & $2.089\cdot10^{-51}$ \\
  $SM_4(sml\omega,1)$ & 2.5810938731370101206\ldots & $2.444\cdot10^{-42}$ \\
  $SM_4(smrp,1)$ & 4.1332258765524751774\ldots & $3.508\cdot10^{-47}$ \\
  $SM_4(smr\omega,1)$ & 5.5864923376104430121\ldots & $8.974\cdot10^{-30}$ \\ \hline
  $SM_4(smlp,2)$ & 3.9106331251114174037\ldots & $8.356\cdot10^{-50}$ \\
  $SM_4(sml\omega,2)$ & 4.9941981422394427409\ldots & $9.777\cdot10^{-41}$ \\
  $SM_4(smrp,2)$ & 11.837404185675137231\ldots & $1.403\cdot10^{-45}$ \\
  $SM_4(smr\omega,2)$ & 15.26980028696103804\ldots & $3.590\cdot10^{-28}$ \\ \hline
  $SM_4(smlp,3)$ & 10.653791515005141146\ldots & $3.342\cdot10^{-48}$ \\
  $SM_4(sml\omega,3)$ & 12.088423092839839475\ldots & $3.911\cdot10^{-39}$ \\
  $SM_4(smrp,3)$ & 39.302927041913622498\ldots & $5.613\cdot10^{-44}$ \\
  $SM_4(smr\omega,3)$ & 53.644863796864365278\ldots & $1.436\cdot10^{-26}$ \\ \hline
  $SM_5(smlp)$ & --0.3905031434783375128\ldots & $-7.109\cdot10^{-47}$ \\
  $SM_5(sml\omega)$ & 0.58007673972219526962\ldots & $-4.780\cdot10^{-47}$ \\
  $SM_5(smrp)$ & --0.13276679090380714341\ldots & $-7.109\cdot10^{-47}$ \\
  $SM_5(smr\omega)$ & 0.85258790936171242297\ldots & $-6.863\cdot10^{-47}$ \\ \hline
  $SM_6(smlp)$ & 0.54152160636801684581\ldots & $1.734\cdot10^{-48}$ \\
  $SM_6(sml\omega)$ & 1.0356287723533406243\ldots & $1.166\cdot10^{-48}$ \\
  $SM_6(smrp)$ & 0.25825928231171221647\ldots & $1.734\cdot10^{-48}$ \\
  $SM_6(smr\omega)$ & 0.75530886756654392948\ldots & $1.674\cdot10^{-48}$ \\ \hline
  $SM_7(smlp,2)$ & 0.12314242079796645308\ldots & $4.128\cdot10^{-50}$ \\
  $SM_7(sml\omega,2)$ & 0.12230623557364785581\ldots & $2.776\cdot10^{-50}$ \\
  $SM_7(smrp,2)$ & 0.059487738877545345581\ldots & $4.128\cdot10^{-50}$ \\
  $SM_7(smr\omega,2)$ & 0.059069232762245306011\ldots & $3.986\cdot10^{-50}$ \\ \hline
  $SM_7(smlp,3)$ & 0.0064345613772149429216\ldots & $9.600\cdot10^{-52}$ \\
  $SM_7(sml\omega,3)$ & 0.0063305872508483357372\ldots & $6.455\cdot10^{-52}$ \\
  $SM_7(smrp,3)$ & 0.0029196501300934203105\ldots & $9.600\cdot10^{-52}$ \\
  $SM_7(smr\omega,3)$ & 0.0028676244349701518512\ldots & $9.269\cdot10^{-52}$ \\ \hline
  $SM_8(smlp,1)$ & 5.0317015083535379435\ldots & $2.843\cdot10^{-45}$ \\
  $SM_8(sml\omega,1)$ & 5.8510436406486760859\ldots & $1.912\cdot10^{-45}$ \\
  $SM_8(smrp,1)$ & 2.2657136563817159215\ldots & $2.843\cdot10^{-45}$ \\
  $SM_8(smr\omega,1)$ & 3.1751240304194110976\ldots & $2.745\cdot10^{-45}$ \\ \hline
  $SM_8(smlp,2)$ & 9.7612345513619487842\ldots & $1.109\cdot10^{-43}$ \\
  $SM_8(sml\omega,2)$ & 9.0242760612601416584\ldots & $7.457\cdot10^{-44}$ \\
  $SM_8(smrp,2)$ & 4.1295191146631057347\ldots & $1.109\cdot10^{-43}$ \\
  $SM_8(smr\omega,2)$ & 3.7593504699750039665\ldots & $1.071\cdot10^{-43}$ \\ \hline
  $SM_8(smlp,3)$ & 14.504395800227427856\ldots & $4.214\cdot10^{-42}$ \\
  $SM_8(sml\omega,3)$ & 12.21486768404280473\ldots & $2.834\cdot10^{-42}$ \\
  $SM_8(smrp,3)$ & 5.7444969832482504712\ldots & $4.214\cdot10^{-42}$ \\
  $SM_8(smr\omega,3)$ & 4.5906776136611764786\ldots & $4.069\cdot10^{-42}$ \\ \hline
  $SM_9(smlp)$ & 0.56971181817608535663\ldots & $1.937\cdot10^{-84}$ \\
  $SM_9(sml\omega)$ & 1.4020017036535156398\ldots & $1.202\cdot10^{-82}$ \\
  $SM_9(smrp)$ & 1.3438058059670733965\ldots & $1.970\cdot10^{-84}$ \\
  $SM_9(smr\omega)$ & 1.8600161722383592123\ldots & $6.567\cdot10^{-85}$ \\ \hline
  $SM_{10}(smlp,1)$ & 0.56182922629094674981\ldots & $1.125\cdot10^{-41}$ \\
  $SM_{10}(sml\omega,1)$ & 1.630157684938493886\ldots & $1.795\cdot10^{-25}$ \\
  $SM_{10}(smrp,1)$ & 1.4313028466067383441\ldots & $1.125\cdot10^{-41}$ \\
  $SM_{10}(smr\omega,1)$ & 2.5922754787881733157\ldots & $6.697\cdot10^{-40}$ \\ \hline
  $SM_{10}(smlp,2)$ & 0.23894083955413619831\ldots & $1.939\cdot10^{-43}$ \\
  $SM_{10}(sml\omega,2)$ & 1.2545998023407273118\ldots & $4.603\cdot10^{-27}$ \\
  $SM_{10}(smrp,2)$ & 1.1945344006412840488\ldots & $1.939\cdot10^{-43}$ \\
  $SM_{10}(smr\omega,2)$ & 2.2446282430004759916\ldots & $1.1959\cdot10^{-41}$ \\ \hline
  $SM_{10}(smlp,3)$ & 0.10748225265502766995\ldots & $3.343\cdot10^{-45}$ \\
  $SM_{10}(sml\omega,3)$ & 1.1111584723169680866\ldots & $1.180\cdot10^{-28}$ \\
  $SM_{10}(smrp,3)$ & 1.0925244916164133302\ldots & $3.343\cdot10^{-45}$ \\
  $SM_{10}(smr\omega,3)$ & 2.1085527095304115107\ldots & $2.136\cdot10^{-43}$ \\ \hline
  $SM_{11}(smlp,1)$ & 0.30344340711709933226\ldots & $1.464\cdot10^{-42}$ \\
  $SM_{11}(sml\omega,1)$ & 1.0399268570536074792\ldots & $2.839\cdot10^{-26}$ \\
  $SM_{11}(smrp,1)$ & 0.9446396368342379653\ldots & $1.464\cdot10^{-42}$ \\
  $SM_{11}(smr\omega,1)$ & 1.7297214313345287702\ldots & $8.870\cdot10^{-41}$ \\ \hline
  $SM_{11}(smlp,2)$ & 0.13207517402450262085\ldots & $2.524\cdot10^{-44}$ \\
  $SM_{11}(sml\omega,2)$ & 0.84598708313977015251\ldots & $7.279\cdot10^{-28}$ \\
  $SM_{11}(smrp,2)$ & 0.81686599189767927686\ldots & $2.524\cdot10^{-44}$ \\
  $SM_{11}(smr\omega,2)$ & 1.5485497608283970434\ldots & $1.58\cdot10^{-42}$ \\ \hline
  $SM_{11}(smlp,3)$ & 0.060334409371851792881\ldots & $4.352\cdot10^{-46}$ \\
  $SM_{11}(sml\omega,3)$ & 0.76905206513210718626\ldots & $1.866\cdot10^{-29}$ \\
  $SM_{11}(smrp,3)$ & 0.75994293154985848019\ldots & $4.352\cdot10^{-46}$ \\
  $SM_{11}(smr\omega,3)$ & 1.4749748193334142711\ldots & $2.829\cdot10^{-44}$ \\
  \hline
\end{longtable}
\end{center}

\section{Smarandache Near to Primorial Constants}

The authors did not prove the convergence towards each constant. We let it as possible research for the interested readers. With the program $\emph{SNtkP}$ \ref{Program SNtkP} calculate vectors top 45 terms numbers containing Smarandache Near to $k$ Primorial, $n:=1..45$, $\emph{sntp}:=\emph{SNtkP}(n,1)$, $\emph{sntdp}:=\emph{SNtkP}(n,2)$ and  $\emph{snttp}:=\emph{SNtkP}(n,3)$ given on page \pageref{Program SNtkP} and by (\ref{sntp}--\ref{snttp}). Using programs similar to programs $\emph{SK1}-\emph{SK11}$, \ref{Program SK1}--\ref{Program SK11}, was calculated constants Smarandache near to $k$ primorial.

\begin{multline}\label{sntp}
  \emph{sntp}:=(1\ \ 2\ \ 2\ \ -1\ \ 3\ \ 3\ \ 3\ \ -1\ \ -1\ \ 5\ \ 7\ \ -1\ \ 13\ \ 7\ \ 5\ \ 43\ \ 17\ \ 47\\
  7\ \ 47\ \ 7\ \ 11\ \ 23\ \ 47\ \ 47\ \ 13\ \ 43\ \ 47\ \ 5\ \ 5\ \ 5\ \ 47\ \ 11\ \ 17\ \ 7\ \ 47\\
  23\ \ 19\ \ 13\ \ 47\ \ 41\ \ 7\ \ 43\ \ 47\ \ 47)^\textrm{T}
\end{multline}

\begin{multline}\label{sntdp}
  \emph{sntdp}:=(2\ \ 2\ \ 2\ \ 3\ \ 5\ \ -1\ \ 7\ \ 13\ \ 5\ \ 5\ \ 5\ \ 83\ \ 13\ \ 83\ \ 83\ \ 13\ \ 13\ \ 83\\
  19\ \ 7\ \ 7\ \ 7\ \ 23\ \ 83\ \ 37\ \ 83\ \ 23\ \ 83\ \ 29\ \ 83\ \ 31\ \ 83\ \ 89\ \ 13\ \ 83\ \ 83\\
  11\ \ 97\ \ 13\ \ 71\ \ 23\ \ 83\ \ 43\ \ 89\ \ 89)^\textrm{T}
\end{multline}

\begin{multline}\label{snttp}
  \emph{snttp}:=(2\ \ 2\ \ 2\ \ 3\ \ 5\ \ 5\ \ 7\ \ 11\ \ 23\ \ 43\ \ 11\ \ 89\ \ 7\ \ 7\ \ 7\ \ 11\ \ 11\ \ 23\ \ 19\\
  71\ \ 37\ \ 13\ \ 23\ \ 89\ \ 71\ \ 127\ \ 97\ \ 59\ \ 29\ \ 127\ \ 31\ \ 11\ \ 11\ \ 11\ \ 127\ \ 113\ \ 37\\
  103\ \ 29\ \ 131\ \ 41\ \ 37\ \ 31\ \ 23\ \ 131)^\textrm{T}
\end{multline}

\begin{equation}\label{SNtP1}
  SNtP_1(s)=\sum_{k=1}^{\emph{last}(s)}\dfrac{1}{s_k!}~.
\end{equation}

\begin{equation}\label{SNtP2}
  SNtP_2(s)=\sum_{k=1}^{\emph{last}(s)}\dfrac{s_k}{k!}~.
\end{equation}

\begin{equation}\label{SNtP3}
  SNtP_3(s)=\sum_{k=1}^{\emph{last}(s)}\dfrac{1}{\prod_{j=1}^k s_j}~.
\end{equation}

\begin{equation}\label{SNtP4}
  SNtP_4(s,\alpha)=\sum_{k=1}^{\emph{last}(s)}\dfrac{k^\alpha}{\prod_{j=1}^k s_j}~,\ \ \textnormal{where}\ \ \alpha\in\Ns~.
\end{equation}

\begin{equation}\label{SNtP5}
  SNtP_5(s)=\sum_{k=1}^{\emph{last}(s)}\dfrac{(-1)^{k+1}s_k}{k!}
\end{equation}

\begin{equation}\label{SNtP6}
  SNtP_6(s)=\sum_{k=1}^{\emph{last}(s)}\dfrac{s_k}{(k+1)!}
\end{equation}

\begin{equation}\label{SNtP7}
  SNtP_7(s,r)=\sum_{k=r}^{\emph{last}(s)}\dfrac{s_k}{(k+r)!}~,\ \ \textnormal{where}\ \ r\in\Ns~.
\end{equation}

\begin{equation}\label{SNtP8}
  SNtP_8(s,r)=\sum_{k=r}^{\emph{last}(s)}\dfrac{s_k}{(k-r)!}~,\ \ \textnormal{where}\ \ r\in\Ns~.
\end{equation}

\begin{equation}\label{SNtP9}
  SNtP_9(s)=\sum_{k=1}^{\emph{last}(s)}\dfrac{1}{\sum_{j=1}^k s_k!}
\end{equation}

\begin{equation}\label{SNtP10}
  SNtP_{10}(s,\alpha)=\sum_{k=1}^{\emph{last}(s)}\dfrac{1}{s_k^\alpha\sqrt{s_k!}}~,\ \ \textnormal{where}\ \ \alpha\in\Ns~.
\end{equation}

\begin{equation}\label{SNtP11}
  SNtP_{11}(s,\alpha)=\sum_{k=1}^{\emph{last}(s)}\dfrac{1}{s_k^\alpha\sqrt{(s_k+1)!}}~,\ \ \textnormal{where}\ \ \alpha\in\Ns~.
\end{equation}
In the formulas (\ref{SNtP1}--\ref{SNtP11}) will replace $s$ with $\emph{sntp}$ or $\emph{sntdp}$ or $\emph{snttp}$.

\begin{center}
 \begin{longtable}{|l|l|l|}
   \caption{Smarandache near to $k$ primorial constants}\\
   \hline
   Name  &  Constant value & Value the last term \\
   \hline
  \endfirsthead
   \hline
   Name  &  Constant value & Value the last term \\
   \hline
  \endhead
   \hline \multicolumn{3}{r}{\textit{Continued on next page}} \\
  \endfoot
   \hline
  \endlastfoot
  $SNtP_1(sntp)$ & 2.5428571934431365743\ldots & $3.867\cdot10^{-60}$ \\
  $SNtP_1(sntdp)$ & 1.7007936768093018175\ldots & $6.058\cdot10^{-137}$ \\
  $SNtP_1(snttp)$ & 1.6841271596523332717\ldots & $1.180\cdot10^{-222}$ \\ \hline
  $SNtP_2(sntp)$ & 2.3630967934998628805\ldots & $3.929\cdot10^{-55}$ \\
  $SNtP_2(sntdp)$ & 3.5017267676908976891\ldots & $7.440\cdot10^{-55}$ \\
  $SNtP_2(snttp)$ & 3.5086818448618034565\ldots & $1.095\cdot10^{-54}$ \\ \hline
  $SNtP_3(sntp)$ & 1.8725106254184368208\ldots & $1.902\cdot10^{-46}$ \\
  $SNtP_3(sntdp)$ & 0.92630477141852643895\ldots & $3.048\cdot10^{-60}$ \\
  $SNtP_3(snttp)$ & 0.92692737191576132359\ldots & $7.637\cdot10^{-62}$ \\ \hline
  $SNtP_4(sntp,1)$ & 3.4198909215383519666\ldots & $8.560\cdot10^{-45}$ \\
  $SNtP_4(sntdp,1)$ & 1.5926089060696178054\ldots & $1.371\cdot10^{-58}$ \\
  $SNtP_4(snttp,1)$ & 1.5951818710249031148\ldots & $3.437\cdot10^{-60}$ \\ \hline
  $SNtP_4(sntp,2)$ & 9.0083762775802033621\ldots & $3.852\cdot10^{-43}$ \\
  $SNtP_4(sntdp,2)$ & 3.5661339881230530766\ldots & $6.171\cdot10^{-57}$ \\
  $SNtP_4(snttp,2)$ & 3.5731306093562134618\ldots & $1.547\cdot10^{-58}$ \\ \hline
  $SNtP_4(sntp,3)$ & 33.601268780655137928\ldots & $1.733\cdot10^{-41}$ \\
  $SNtP_4(sntdp,3)$ & 10.05655962385479256\ldots & $2.777\cdot10^{-55}$ \\
  $SNtP_4(snttp,3)$ & 10.036792885247337658\ldots & $6.959\cdot10^{-57}$ \\ \hline
  $SNtP_5(sntp)$ & 0.35476070426989163902\ldots & $3.929\cdot10^{-55}$ \\
  $SNtP_5(sntdp)$ & 1.2510788222295556551\ldots & $7.440\cdot10^{-55}$ \\
  $SNtP_5(snttp)$ & 1.2442232499898234684\ldots & $1.095\cdot10^{-54}$ \\ \hline
  $SNtP_6(sntp)$ & 0.92150311621958362854\ldots & $8.541\cdot10^{-57}$ \\
  $SNtP_6(sntdp)$ & 1.4488220738476987965\ldots & $1.617\cdot10^{-56}$ \\
  $SNtP_6(snttp)$ & 1.4498145515325115306\ldots & $2.381\cdot10^{-56}$ \\ \hline
  $SNtP_7(sntp,2)$ & 0.10067792162571842448\ldots & $1.817\cdot10^{-58}$ \\
  $SNtP_7(sntdp,2)$ & 0.10518174020177436646\ldots & $3.441\cdot10^{-58}$ \\
  $SNtP_7(snttp,2)$ & 0.10530572828545769371\ldots & $5.065\cdot10^{-58}$ \\ \hline
  $SNtP_7(sntp,3)$ & 0.0028612773389160206688\ldots & $3.786\cdot10^{-60}$ \\
  $SNtP_7(sntdp,3)$ & 0.0034992898623021328169\ldots & $7.170\cdot10^{-60}$ \\
  $SNtP_7(snttp,3)$ & 0.0035130621711970344353\ldots & $1.055\cdot10^{-59}$ \\ \hline
  $SNtP_8(sntp,1)$ & 4.1541824026937241559\ldots & $1.768\cdot10^{-53}$ \\
  $SNtP_8(sntdp,1)$ & 5.7207762058536591328\ldots & $3.348\cdot10^{-53}$ \\
  $SNtP_8(snttp,1)$ & 5.762598971655286527\ldots & $4.928\cdot10^{-53}$ \\ \hline
  $SNtP_8(sntp,2)$ & 4.6501436396625714979\ldots & $7.779\cdot10^{-52}$ \\
  $SNtP_8(sntdp,2)$ & 6.4108754573319424904\ldots & $1.473\cdot10^{-51}$ \\
  $SNtP_8(snttp,2)$ & 6.6209627683988696642\ldots & $2.168\cdot10^{-51}$ \\ \hline
  $SNtP_8(sntp,3)$ & 4.126169449994612004\ldots & $3.345\cdot10^{-50}$ \\
  $SNtP_8(sntdp,3)$ & 7.9082950789773488585\ldots & $6.334\cdot10^{-50}$ \\
  $SNtP_8(snttp,3)$ & 8.7576630556464774578\ldots & $9.324\cdot10^{-50}$ \\ \hline
  $SNtP_9(sntp)$ & 1.4214338438480314719\ldots & $3.866\cdot10^{-61}$ \\
  $SNtP_9(sntdp)$ & 1.0466424860358234656\ldots & $1.040\cdot10^{-152}$ \\
  $SNtP_9(snttp)$ & 1.055311342837214453\ldots & $5.902\cdot10^{-223}$ \\ \hline
  $SNtP_{10}(sntp,1)$ & 2.218747505497683865\ldots & $4.184\cdot10^{-32}$ \\
  $SNtP_{10}(sntdp,1)$ & 1.2778419356685079191\ldots & $8.745\cdot10^{-71}$ \\
  $SNtP_{10}(snttp,1)$ & 1.2414085582609377035\ldots & $8.294\cdot10^{-114}$ \\ \hline
  $SNtP_{10}(sntp,2)$ & 1.5096212187703265236\ldots & $8.902\cdot10^{-34}$ \\
  $SNtP_{10}(sntdp,2)$ & 0.59144856965778964572\ldots & $9.826\cdot10^{-73}$ \\
  $SNtP_{10}(snttp,2)$ & 0.58415307582503530822\ldots & $6.331\cdot10^{-116}$ \\ \hline
  $SNtP_{10}(sntp,3)$ & 1.2260357559936064516\ldots & $1.894\cdot10^{-35}$ \\
  $SNtP_{10}(sntdp,3)$ & 0.28337095760325627718\ldots & $1.104\cdot10^{-74}$ \\
  $SNtP_{10}(snttp,3)$ & 0.28191104877850708976\ldots & $4.833\cdot10^{-118}$ \\ \hline
  $SNtP_{11}(sntp,1)$ & 1.3610247806676889023\ldots & $6.039\cdot10^{-33}$ \\
  $SNtP_{11}(sntdp,1)$ & 0.71307955647334199639\ldots & $9.218\cdot10^{-72}$ \\
  $SNtP_{11}(snttp,1)$ & 0.69819605640188698911\ldots & $7.219\cdot10^{-115}$ \\ \hline
  $SNtP_{11}(sntp,2)$ & 0.98733649403689727251\ldots & $1.285\cdot10^{-34}$ \\
  $SNtP_{11}(sntdp,2)$ & 0.33523656456978691915\ldots & $1.036\cdot10^{-73}$ \\
  $SNtP_{11}(snttp,2)$ & 0.33225730607103958582\ldots & $5.510\cdot10^{-117}$ \\ \hline
  $SNtP_{11}(sntp,3)$ & 0.83342721544339356637\ldots & $2.734\cdot10^{-36}$ \\
  $SNtP_{11}(sntdp,3)$ & 0.16190395280891818924\ldots & $1.164\cdot10^{-75}$ \\
  $SNtP_{11}(snttp,3)$ & 0.16130786627737647051\ldots & $4.206\cdot10^{-119}$ \\
  \hline
\end{longtable}
\end{center}

\section{Smarandache--Cira constants}

The authors did not prove the convergence towards each constant. We let it as possible research for the interested readers. With the program $\emph{SC}$ \ref{Program Smarandache-Cira Order k} calculate vectors top 113 terms numbers containing Smarandache--Cira sequences of order two and three, $n:=1..113$, $\emph{sc}2:=\emph{SC}(n,2)$ and $\emph{sc}3:=\emph{SC}(n,3)$ given on page \pageref{Program Smarandache-Cira Order k}. We note with $m:=\emph{last}(sc2)=\emph{last}(sc3)$.

\begin{multline*}
  \emph{sc}2^\textrm{T}\rightarrow(1\ \ 2\ \ 3\ \ 2\ \ 5\ \ 3\ \ 7\ \ 4\ \ 3\ \ 5\ \ 11\ \ 3\ \ 13\ \ 7\ \ 5\ \ 4\ \ 17\ \ 3\ \ 19\ \ 5\ \ 7\ \ 11\ \ 23\ \ 4\ \ 5\ \ 13\\
  6\ \ 7\ \ 29\ \ 5\ \ 31\ \ 4\ \ 11\ \ 17\ \ 7\ \ 3\ \ 37\ \ 19\ \ 13\ \ 5\ \ 41\ \ 7\ \ 43\ \ 11\ \ 5\ \ 23\ \ 47\ \ 4\ \ 7\ \ 5\ \ 17\ \ 13\\
  53\ \ 6\ \ 11\ \ 7\ \ 19\ \ 29\ \ 59\ \ 5\ \ 61\ \ 31\ \ 7\ \ 4\ \ 13\ \ 11\ \ 67\ \ 17\ \ 23\ \ 7\ \ 71\ \ 4\ \ 73\ \ 37\ \ 5\ \ 19\ \ 11\\
  13\ \ 79\ \ 5\ \ 6\ \ 41\ \ 83\ \ 7\ \ 17\ \ 43\ \ 29\ \ 11\ \ 89\ \ 5\ \ 13\ \ 23\ \ 31\ \ 47\ \ 19\ \ 4\ \ 97\ \ 7\ \ 11\ \ 5\ \ 101\\
  17\ \ 103\ \ 13\ \ 7\ \ 53\ \ 107\ \ 6\ \ 109\ \ 11\ \ 37\ \ 7\ \ 113)
\end{multline*}

\begin{multline*}
  \emph{sc}3^\textrm{T}\rightarrow(1\ \ 2\ \ 3\ \ 2\ \ 5\ \ 3\ \ 7\ \ 2\ \ 3\ \ 5\ \ 11\ \ 3\ \ 13\ \ 7\ \ 5\ \ 4\ \ 17\ \ 3\ \ 19\ \ 5\ \ 7\ \ 11\ \ 23\ \ 3\ \ 5\ \ 13\\
  3\ \ 7\ \ 29\ \ 5\ \ 31\ \ 4\ \ 11\ \ 17\ \ 7\ \ 3\ \ 37\ \ 19\ \ 13\ \ 5\ \ 41\ \ 7\ \ 43\ \ 11\ \ 5\ \ 23\ \ 47\ \ 4\ \ 7\ \ 5\ \ 17\ \ 13\\
  53\ \ 3\ \ 11\ \ 7\ \ 19\ \ 29\ \ 59\ \ 5\ \ 61\ \ 31\ \ 7\ \ 4\ \ 13\ \ 11\ \ 67\ \ 17\ \ 23\ \ 7\ \ 71\ \ 3\ \ 73\ \ 37\ \ 5\ \ 19\ \ 11\\
  13\ \ 79\ \ 5\ \ 6\ \ 41\ \ 83\ \ 7\ \ 17\ \ 43\ \ 29\ \ 11\ \ 89\ \ 5\ \ 13\ \ 23\ \ 31\ \ 47\ \ 19\ \ 4\ \ 97\ \ 7\ \ 11\ \ 5\ \ 101\\
  17\ \ 103\ \ 13\ \ 7\ \ 53\ \ 107\ \ 3\ \ 109\ \ 11\ \ 37\ \ 7\ \ 113)
\end{multline*}

\[
  SC_1(sc2)=\sum_{k=1}^m\dfrac{1}{sc2_k!}\emph{float},20\rightarrow3.4583335851391576045~,
\]

\[
  SC_1(sc3)=\sum_{k=1}^m\dfrac{1}{sc3_k!}\emph{float},20\rightarrow4.6625002518058242712~,
\]

\[
  SC_2(sc2)=\sum_{k=1}^m\dfrac{sc2_k}{k!}\emph{float},20\rightarrow2.6306646909747437367~,
\]

\[
  SC_2(sc3)= \sum_{k=1}^m\dfrac{sc3_k}{k!}\emph{float},20\rightarrow2.6306150878001405621~,
\]

\[
  SC_3(sc2)=\sum_{k=1}^m\dfrac{1}{\displaystyle\prod_{j=1}^k sc2_k}\emph{float},20\rightarrow1.7732952904854629675,
\]

\[
  SC_31(sc3)=\sum_{k=1}^m\dfrac{1}{\displaystyle\prod_{j=1}^k sc3_k}\emph{float},20\rightarrow1.7735747079550529244,
\]

\[
  SC_4(sc2,1)=\sum_{k=1}^m\dfrac{k}{\displaystyle\prod_{j=1}^k sc2_k}\emph{float},20\rightarrow2.9578888874303295232,
\]

\[
  SC_4(sc3,1)=\sum_{k=1}^m\dfrac{k}{\displaystyle\prod_{j=1}^k sc3_k}\emph{float},20\rightarrow2.9602222193051036183,
\]

\[
  SC_4(sc2,2)=\sum_{k=1}^m\dfrac{k^2}{\displaystyle\prod_{j=1}^k sc2_k}\emph{float},20\rightarrow6.5084767524852643905,
\]

\[
  SC_4(sc3,2)=\sum_{k=1}^m\dfrac{k^2}{\displaystyle\prod_{j=1}^k sc3_k}\emph{float},20\rightarrow6.5280646160816429818,
\]

\[
  SC_4(sc2,3)=\sum_{k=1}^m\dfrac{k^3}{\displaystyle\prod_{j=1}^k sc2_k}\emph{float},20\rightarrow18.554294952927603195,
\]

\[
  SC_4(sc3,3)=\sum_{k=1}^m\dfrac{k^3}{\displaystyle\prod_{j=1}^k sc3_k}\emph{float},20\rightarrow18.719701016966392423,
\]

\[
  SC_5(sc2)=\sum_{k=1}^m\dfrac{(-1)^{k+1}sc2_k}{k!}\emph{float},20\rightarrow0.45546350985738070916~,
\]

\[
  SC_5(sc3)=\sum_{k=1}^m\dfrac{(-1)^{k+1}sc3_k}{k!}\emph{float},20\rightarrow0.45551311303198388376~,
\]

\[
  SC_6(sc2)=\sum_{k=1}^m\dfrac{sc2_k}{(k+1)!}\emph{float},20\rightarrow0.98272529215953150861~,
\]

\[
  SC_6(sc3)=\sum_{k=1}^m\dfrac{sc3_k}{(k+1)!}\emph{float},20\rightarrow0.98271978069568671144~,
\]

\[
  SC_7(sc2,2)=\sum_{k=2}^m\dfrac{sc2_k}{(k+2)!}\emph{float},20\rightarrow0.11219805918720652342~,
\]

\[
  SC_7(sc2,3)=\sum_{k=2}^m\dfrac{sc3_k}{(k+2)!}\emph{float},20\rightarrow0.1121975080408220437~,
\]

\[
  SC_7(sc2,3)=\sum_{k=3}^m\dfrac{sc2_k}{(k+3)!}\emph{float},20\rightarrow0.0046978036116398886027~,
\]

\[
  SC_7(sc3,3)=\sum_{k=3}^m\dfrac{sc3_k}{(k+3)!}\emph{float},20\rightarrow0.0046977535074231177193~,
\]

\[
  SC_8(sc2,1)=\sum_{k=1}^m\dfrac{sc2_k}{(k-1)!}\emph{float},20\rightarrow5.0772738578906499358~,
\]

\[
  SC_8(sc3,1)=\sum_{k=1}^m\dfrac{sc3_k}{(k-1)!}\emph{float},20\rightarrow5.076877032493824539~,
\]

\[
  SC_8(sc2,2)=\sum_{k=2}^m\dfrac{sc2_k}{(k-2)!}\emph{float},20\rightarrow7.0229729491778632583~,
\]

\[
  SC_8(sc3,2)=\sum_{k=2}^m\dfrac{sc3_k}{(k-2)!}\emph{float},20\rightarrow7.0201951714000854806~,
\]

\[
  SC_8(sc2,3)=\sum_{k=3}^m\dfrac{sc2_k}{(k-3)!}\emph{float},20\rightarrow8.3304435839100330540~,
\]

\[
  SC_8(sc3,3)=\sum_{k=3}^m\dfrac{sc3_k}{(k-3)!}\emph{float},20\rightarrow8.3137769172433663874~,
\]

\[
  SC_9(sc2)=\sum_{k=1}^m\dfrac{1}{\displaystyle\sum_{j=1}^k sc2_k!}\emph{float},20\rightarrow1.5510516488142853476,
\]

\[
  SC_9(sc3)=\sum_{k=1}^m\dfrac{1}{\displaystyle\sum_{j=1}^k sc3_k!}\emph{float},20\rightarrow1.5510540589248305324,
\]

\[
  SC_{10}(sc2,1)=\sum_{k=1}^m\dfrac{1}{sc2_k\sqrt{sc2_k!}}\emph{float},20\rightarrow3.2406242284919649811,
\]

\[
  SC_{10}(sc3,1)=\sum_{k=1}^m\dfrac{1}{sc3_k\sqrt{sc3_k!}}\emph{float},20\rightarrow4.1028644277885635587,
\]

\[
  SC_{10}(sc2,2)=\sum_{k=1}^m\dfrac{1}{sc2_k^2\sqrt{sc2_k!}}\emph{float},20\rightarrow1.7870808548071955299,
\]

\[
  SC_{10}(sc3,2)=\sum_{k=1}^m\dfrac{1}{sc3_k^2\sqrt{sc3_k!}}\emph{float},20\rightarrow2.1492832287173527766,
\]

\[
  SC_{10}(sc2,3)=\sum_{k=1}^m\dfrac{1}{sc2_k^3\sqrt{sc2_k!}}\emph{float},20\rightarrow1.3045045248690711166,
\]

\[
  SC_{10}(sc3,3)=\sum_{k=1}^m\dfrac{1}{sc3_k^3\sqrt{sc3_k!}}\emph{float},20\rightarrow1.4584084801526035601,
\]

\[
  SC_{11}(sc2,1)=\sum_{k=1}^m\dfrac{1}{sc2_k\sqrt{(sc2_k+1)!}}\emph{float},20\rightarrow1.8299218542898376035,
\]

\[
  SC_{11}(sc3,1)=\sum_{k=1}^m\dfrac{1}{sc3_k\sqrt{(sc3_k+1)!}}\emph{float},20\rightarrow2.2987446364307715563,
\]

\[
  SC_{11}(sc2,2)=\sum_{k=1}^m\dfrac{1}{sc2_k^2\sqrt{(sc2_k+1)!}}\emph{float},20\rightarrow1.1168191325442037858,
\]

\[
  SC_{11}(sc3,2)=\sum_{k=1}^m\dfrac{1}{sc3_k^2\sqrt{(sc3_k+1)!}}\emph{float},20\rightarrow1.3139933527909721081,
\]

\[
  SC_{11}(sc2,3)=\sum_{k=1}^m\dfrac{1}{sc2_k^3\sqrt{(sc2_k+1)!}}\emph{float},20\rightarrow0.87057913102041407346,
\]

\[
  SC_{11}(sc3,3)=\sum_{k=1}^m\dfrac{1}{sc3_k^3\sqrt{(sc3_k+1)!}}\emph{float},20\rightarrow0.95493621492446577773.
\]

\section{Smarandache--X-nacci constants}

Let $n:=1..80$ be and the commands $\emph{sf}_n:=\emph{SF}(n)$, $\emph{str}_n:=\emph{STr}(n)$ and $\emph{ste}_n:=\emph{STe}(n)$, then
\begin{multline*}
  \emph{sf}^\textrm{T}\rightarrow(1\ \ 3\ \ 4\ \ 6\ \ 5\ \ 12\ \ 8\ \ 6\ \ 12\ \ 15\ \ 10\ \ 12\ \ 7\ \ 24\ \ 20\ \ 12\ \ 9\ \ 12\ \ 18\ \ 30\\
  8\ \ 30\ \ 24\ \ 12\ \ 25\ \ 21\ \ 36\ \ 24\ \ 14\ \ 60\ \ 30\ \ 24\ \ 20\ \ 9\ \ 40\ \ 12\ \ 19\ \ 18\ \ 28\ \ 30\\
  20\ \ 24\ \ 44\ \ 30\ \ 60\ \ 24\ \ 16\ \ 12\ \ 56\ \ 75\ \ 36\ \ 42\ \ 27\ \ 36\ \ 10\ \ 24\ \ 36\ \ 42\ \ 58\ \ 60\\
  15\ \ 30\ \ 24\ \ 48\ \ 35\ \ 60\ \ 68\ \ 18\ \ 24\ \ 120\ \ 70\ \ 12\ \ 37\ \ 57\ \ 100\ \ 18\ \ 40\ \ 84\ \ 78\ \ 60)
\end{multline*}

\begin{multline*}
  \emph{str}^\textrm{T}\rightarrow(1\ \ 3\ \ 7\ \ 4\ \ 14\ \ 7\ \ 5\ \ 7\ \ 9\ \ 19\ \ 8\ \ 7\ \ 6\ \ 12\ \ 52\ \ 15\ \ 28\ \ 12\ \ 18\ \ 31\ \   12\ \ 8\\
  29\ \ 7\ \ 30\ \ 39\ \ 9\ \ 12\ \ 77\ \ 52\ \ 14\ \ 15\ \ 35\ \ 28\ \ 21\ \ 12\ \ 19\ \ 28\ \ 39\ \ 31\ \ 35\ \ 12\ \ 82\\
  8\ \ 52\ \ 55\ \ 29\ \ 64\ \ 15\ \ 52\ \ 124\ \ 39\ \ 33\ \ 35\ \ 14\ \ 12\ \ 103\ \ 123\ \ 64\ \ 52\ \ 68\ \ 60\\
  12\ \ 15\ \ 52\ \ 35\ \ 100\ \ 28\ \ 117\ \ 31\ \ 132\ \ 12\ \ 31\ \ 19\ \ 52\ \ 28\ \ 37\ \ 39\ \ 18\ \ 31)
\end{multline*}

\begin{multline*}
  \emph{ste}^\textrm{T}\rightarrow(1\ \ 3\ \ 6\ \ 4\ \ 6\ \ 9\ \ 8\ \ 5\ \ 9\ \ 13\ \ 20\ \ 9\ \ 10\ \ 8\ \ 6\ \ 10\ \ 53\ \ 9\ \ 48\ \ 28\ \ 18\ \ 20\ \ 35\\
  18\ \ 76\ \ 10\ \ 9\ \ 8\ \ 7\ \ 68\ \ 20\ \ 15\ \ 20\ \ 53\ \ 30\ \ 9\ \ 58\ \ 48\ \ 78\ \ 28\ \ 19\ \ 18\ \ 63\ \ 20\ \ 68\\
  35\ \ 28\ \ 18\ \ 46\ \ 108\ \ 76\ \ 10\ \ 158\ \ 9\ \ 52\ \ 8\ \ 87\ \ 133\ \ 18\ \ 68\ \ 51\ \ 20\ \ 46\ \ 35\ \ 78\\
  20\ \ 17\ \ 138\ \ 35\ \ 30\ \ 230\ \ 20\ \ 72\ \ 58\ \ 76\ \ 48\ \ 118\ \ 78\ \ 303\ \ 30)
\end{multline*}

\begin{equation}\label{SX1}
  SX_1(s)=\sum_{k=1}^m\dfrac{1}{s_k!}~.
\end{equation}

\begin{equation}\label{SX2}
  SX_2(s)=\sum_{k=1}^m\dfrac{s_k}{k!}~.
\end{equation}

\begin{equation}\label{SX3}
  SX_3(s)=\sum_{k=1}^m\dfrac{1}{\prod_{j=1}^k s_j}~.
\end{equation}

\begin{equation}\label{SX4}
  SX_4(s,\alpha)=\sum_{k=1}^m\dfrac{k^\alpha}{\prod_{j=1}^k s_j}~,\ \ \textnormal{where}\ \ \alpha\in\Ns~.
\end{equation}

\begin{equation}\label{SX5}
  SX_5(s)=\sum_{k=1}^m\dfrac{(-1)^{k+1}s_k}{k!}
\end{equation}

\begin{equation}\label{SX6}
  SX_6(s)=\sum_{k=1}^m\dfrac{s_k}{(k+1)!}
\end{equation}

\begin{equation}\label{SX7}
  SX_7(s,r)=\sum_{k=r}^m\dfrac{s_k}{(k+r)!}~,\ \ \textnormal{where}\ \ r\in\Ns~.
\end{equation}

\begin{equation}\label{SX8}
  SX_8(s,r)=\sum_{k=r}^m\dfrac{s_k}{(k-r)!}~,\ \ \textnormal{where}\ \ r\in\Ns~.
\end{equation}

\begin{equation}\label{SX9}
  SX_9(s)=\sum_{k=1}^m\dfrac{1}{\sum_{j=1}^k s_k!}
\end{equation}

\begin{equation}\label{SX10}
  SX_{10}(s,\alpha)=\sum_{k=1}^m\dfrac{1}{s_k^\alpha\sqrt{s_k!}}~,\ \ \textnormal{where}\ \ \alpha\in\Ns~.
\end{equation}

\begin{equation}\label{SX11}
  SX_{11}(s,\alpha)=\sum_{k=1}^m\dfrac{1}{s_k^\alpha\sqrt{(s_k+1)!}}~,\ \ \textnormal{where}\ \ \alpha\in\Ns~.
\end{equation}
In the formulas (\ref{SX1}--\ref{SX11}) will replace $s$ with $\emph{sf}$ or $\emph{str}$ or $\emph{ste}$ and $m:=\emph{last}(sf)=\emph{last}(str)=\emph{last}(ste)$.

The authors did not prove the convergence towards each constant. We let it as possible research for the interested readers.

\begin{center}
 \begin{longtable}{|l|l|l|}
   \caption{Smarandache--X-nacci constants}\\
   \hline
   Name  &  Constant value & Value the last term \\
   \hline
  \endfirsthead
   \hline
   Name  &  Constant value & Value the last term \\
   \hline
  \endhead
   \hline \multicolumn{3}{r}{\textit{Continued on next page}} \\
  \endfoot
   \hline
  \endlastfoot
  $SX_1(sf)$ & 1.2196985417298194908\ldots & $1.202\cdot10^{-82}$ \\
  $SX_1(str)$ & 1.2191275540999213351\ldots & $1.216\cdot10^{-34}$ \\
  $SX_1(ste)$ & 1.2211513449585368892\ldots & $3.770\cdot10^{-33}$ \\ \hline
  $SX_2(sf)$ & 3.4767735904805818975\ldots & $8.383\cdot10^{-118}$ \\
  $SX_2(str)$ & 3.9609181504757024515\ldots & $4.331\cdot10^{-118}$ \\
  $SX_2(ste)$ & 3.7309068817634133077\ldots & $4.192\cdot10^{-118}$ \\ \hline
  $SX_3(sf)$ & 1.4335990041360201401\ldots & $1.513\cdot10^{-108}$ \\
  $SX_3(str)$ & 1.3938571352678434235\ldots & $7.402\cdot10^{-108}$ \\
  $SX_3(ste)$ & 1.4053891469804807777\ldots & $5.461\cdot10^{-109}$ \\ \hline
  $SX_4(sf,1)$ & 1.9877450439442829197\ldots & $1.210\cdot10^{-106}$ \\
  $SX_4(str,1)$ & 1.8623249618930151417\ldots & $5.922\cdot10^{-106}$ \\
  $SX_4(ste,1)$ & 1.9022896785778318923\ldots & $4.369\cdot10^{-107}$ \\ \hline
  $SX_4(sf,2)$ & 3.3850953926486127438\ldots & $9.681\cdot10^{-105}$ \\
  $SX_4(str,2)$ & 2.9794588765640621423\ldots & $4.737\cdot10^{-104}$ \\
  $SX_4(ste,2)$ & 3.1247358165606852605\ldots & $3.495\cdot10^{-105}$ \\ \hline
  $SX_4(sf,3)$ & 7.2154954684533914439\ldots & $7.74510^{-103}$ \\
  $SX_4(str,3)$ & 5.8572332489153350088\ldots & $3.790\cdot10^{-102}$ \\
  $SX_4(ste,3)$ & 6.4153627205469027224\ldots & $2.796\cdot10^{-103}$ \\ \hline
  $SX_5(sf)$ & --0.056865679752101086683\ldots & $-8.383\cdot10^{-118}$ \\
  $SX_5(str)$ & 0.6077826491902020422\ldots & $-4.331\cdot10^{-118}$ \\
  $SX_5(ste)$ & 0.37231832989141262735\ldots & $-4.192\cdot10^{-118}$ \\ \hline
  $SX_6(sf)$ & 1.2262107161454250477\ldots & $1.035\cdot10^{-119}$ \\
  $SX_6(str)$ & 1.3459796054481592511\ldots & $5.347\cdot10^{-120}$ \\
  $SX_6(ste)$ & 1.2936674214665211769\ldots & $5.175\cdot10^{-120}$ \\ \hline
  $SX_7(sf,2)$ & 0.16798038219142954409\ldots & $1.262\cdot10^{-121}$ \\
  $SX_7(str,2)$ & 0.19185625195487853146\ldots & $6.521\cdot10^{-122}$ \\
  $SX_7(ste,2)$ & 0.18199292568493984364\ldots & $6.311\cdot10^{-122}$ \\ \hline
  $SX_7(sf,3)$ & 0.0069054909490509753782\ldots & $1.52\cdot10^{-123}$ \\
  $SX_7(str,3)$ & 0.010883960530019286262\ldots & $7.857\cdot10^{-124}$ \\
  $SX_7(ste,3)$ & 0.0093029461977309121111\ldots & $7.604\cdot10^{-124}$ \\ \hline
  $SX_8(sf,1)$ & 7.3209769507255575585\ldots & $6.707\cdot10^{-116}$ \\
  $SX_8(str,1)$ & 8.8169446348716671749\ldots & $3.465\cdot10^{-116}$ \\
  $SX_8(ste,1)$ & 8.0040346392438852968\ldots & $3.353\cdot10^{-116}$ \\ \hline
  $SX_8(sf,2)$ & 11.411117402547284927\ldots & $5.298\cdot10^{-114}$ \\
  $SX_8(str,2)$ & 14.678669992110622025\ldots & $2.737\cdot10^{-114}$ \\
  $SX_8(ste,2)$ & 12.450777109170763294\ldots & $2.649\cdot10^{-114}$ \\ \hline
  $SX_8(sf,3)$ & 14.903259849189180761\ldots & $4.133\cdot10^{-112}$ \\
  $SX_8(str,3)$ & 19.449822942788797955\ldots & $2.135\cdot10^{-112}$ \\
  $SX_8(ste,3)$ & 14.890603169310654296\ldots & $2.066\cdot10^{-112}$ \\ \hline
  $SX_9(sf)$ & 1.1775948782312684824\ldots & $2.103\cdot10^{-123}$ \\
  $SX_9(str)$ & 1.1432524801852870116\ldots & $1.340\cdot10^{-95}$ \\
  $SX_9(ste)$ & 1.1462530152136221219\ldots & $4.886\cdot10^{-88}$ \\ \hline
  $SX_{10}(sf,1)$ & 1.2215596514691068605\ldots & $1.827\cdot10^{-43}$ \\
  $SX_{10}(str,1)$ & 1.2239155500269214276\ldots & $3.557\cdot10^{-19}$ \\
  $SX_{10}(ste,1)$ & 1.2300096122076512655\ldots & $2.047\cdot10^{-18}$ \\ \hline
  $SX_{10}(sf,2)$ & 1.0643380628436674107\ldots & $3.045\cdot10^{-45}$ \\
  $SX_{10}(str,2)$ & 1.0645200726839585165\ldots & $1.148\cdot10^{-20}$ \\
  $SX_{10}(ste,2)$ & 1.0656390763277979581\ldots & $6.822\cdot10^{-20}$ \\ \hline
  $SX_{10}(sf,3)$ & 1.0194514801361011603\ldots & $5.075\cdot10^{-47}$ \\
  $SX_{10}(str,3)$ & 1.0194518909553334484\ldots & $3.702\cdot10^{-22}$ \\
  $SX_{10}(ste,3)$ & 1.0196556717297023297\ldots & $2.274\cdot10^{-21}$ \\ \hline
  $SX_{11}(sf,1)$ & 0.81140316439268935525\ldots & $2.340\cdot10^{-44}$ \\
  $SX_{11}(str,1)$ & 0.81207729582854199505\ldots & $6.289\cdot10^{-20}$ \\
  $SX_{11}(ste,1)$ & 0.81447979678562282739\ldots & $3.676\cdot10^{-19}$ \\ \hline
  $SX_{11}(sf,2)$ & 0.73793638461692431348\ldots & $3.899\cdot10^{-46}$ \\
  $SX_{11}(str,2)$ & 0.7379744330358397038\ldots & $2.029\cdot10^{-21}$ \\
  $SX_{11}(ste,2)$ & 0.73841432833681481939\ldots & $6.822\cdot10^{-20}$ \\ \hline
  $SX_{11}(sf,3)$ & 0.71654469002152894246\ldots & $6.498\cdot10^{-48}$ \\
  $SX_{11}(str,3)$ & 0.71654048115486167686\ldots & $6.544\cdot10^{-23}$ \\
  $SX_{11}(ste,3)$ & 0.716620239259230028\ldots & $2.274\cdot10^{-21}$ \\
  \hline
\end{longtable}
\end{center}

\section{The Family of Metallic Means}

The family of \emph{Metallic Means} (whom most prominent members are the \emph{Golden Mean}, \emph{Silver Mean}, \emph{Bronze Mean}, \emph{Nickel Mean}, \emph{Copper Mean}, etc.) comprises every quadratic irrational number that is the positive solution of one of the algebraic equations
\[
 x^2-n\cdot x-1=0\ \ \textnormal{or}\ \ x^2-x-n=0~,
\]
where $n\in\Na$. All of them are closely related to quasi-periodic dynamics, being therefore important basis of musical and architectural proportions. Through the analysis of their common mathematical properties, it becomes evident that they interconnect dif{}ferent human f{}ields of knowledge, in the sense def{}ined in "\emph{Paradoxist Mathematics}". Being irrational numbers, in applications to dif{}ferent scientif{}ic disciplines, they have to be approximated by ratios of integers -- which is the goal of this paper, \citep{Spinadel1998}.

The solutions of equation  $n^2-n\cdot x-1=0$ are:
\begin{equation}\label{SolutiaS1}
  x^2-n\cdot x-1\ solve,x\rightarrow\left(\begin{array}{c}
                                           \dfrac{n+\sqrt{n^2+4}}{2} \\
                                           \dfrac{n-\sqrt{n^2+4}}{2} \\
                                         \end{array}\right)~.
\end{equation}
If we denote by $s_1(n)$ positive solution, then for $n:=1..10$ we have the solutions:
\[s_1(n)\rightarrow\left(\begin{array}{c}
                           \dfrac{\sqrt{5}+2}{2} \\
                           \sqrt{2}+1 \\
                           \dfrac{\sqrt{13}+3}{2} \\
                           \sqrt{5}+2 \\
                           \dfrac{\sqrt{29}+5}{2} \\
                           \sqrt{10}+3 \\
                           \dfrac{\sqrt{53}+7}{2} \\
                           \sqrt{17}+4 \\
                           \dfrac{\sqrt{85}+9}{2} \\
                           \sqrt{26}+5 \\
                         \end{array}\right)=
                   \left(\begin{array}{r}
                           1.618033988749895\\
                           2.414213562373095\\
                           3.302775637731995\\
                           4.236067977499790\\
                           5.192582403567252\\
                           6.162277660168380\\
                           7.140054944640259\\
                           8.123105625617661\\
                           9.109772228646444\\
                          10.099019513592784\\
                     \end{array}\right)~.
\]

The solutions of equation  $n^2-x-n=0$ are:
\begin{equation}\label{SolutiaS2}
  x^2-x-n\ solve,x\rightarrow\left(\begin{array}{c}
                                    \dfrac{1+\sqrt{4n+1}}{2} \\
                                    \dfrac{1-\sqrt{4n+1}}{2} \\
                                  \end{array}\right)~.
\end{equation}
If we denote by $s_2(n)$ positive solution, then for $n:=1..10$ we have the solutions:
\[s_2(n)\rightarrow\left(\begin{array}{c}
                           \dfrac{\sqrt{5}+2}{2} \\
                           2 \\
                           \dfrac{\sqrt{13}+1}{2} \\
                           \dfrac{\sqrt{17}+1}{2} \\
                           \dfrac{\sqrt{21}+1}{2} \\
                           3 \\
                           \dfrac{\sqrt{29}+1}{2} \\
                           \dfrac{\sqrt{33}+1}{2} \\
                           \dfrac{\sqrt{37}+1}{2} \\
                           \dfrac{\sqrt{41}+5}{2} \\
                         \end{array}\right)=
                   \left(\begin{array}{r}
                           1.6180339887498950\\
                           2.0000000000000000\\
                           2.3027756377319950\\
                           2.5615528128088303\\
                           2.7912878474779200\\
                           3.0000000000000000\\
                           3.1925824035672520\\
                           3.3722813232690143\\
                           3.5413812651491097\\
                           3.7015621187164243\\
                         \end{array}\right)~.
\]

\chapter{Numerical Carpet}

\section{Generating Cellular Matrices}
\begin{func}\label{Functia conc} Concatenation function of two numbers in the base on numeration 10.
  \begin{tabbing}
    $\emph{conc}(n,m):=$\=\vline\ $\emph{return}\ n\cdot10\ \ \emph{if}\ \ m\textbf{=}0$\\
    \>\vline\ $\emph{return}\ n\cdot10^{\emph{nrd}(m,10)}+m\ \emph{otherwise}$
  \end{tabbing}
\end{func}
Examples of calling the function $\emph{conc}$: $\emph{conc}(123,78)\rightarrow12378$, $\emph{conc}(2,3)\rightarrow23$, $\emph{conc}(2,35)\rightarrow235$, $\emph{conc}(23,5)\rightarrow235$, $\emph{conc}(0,12)\rightarrow12$, $conc(13,0)\rightarrow130$.

\begin{prog}\label{Program concM} Concatenation program in base 10 of all the elements on a line, all the matrix lines. The result is a vector. The origin of indexes is 1, i.e.
\[
 ORIGIN:=1
\]
  \begin{tabbing}
    $\emph{concM}(M):=$\=\vline\ $c\leftarrow \emph{cols}(M)$\\
    \>\vline\ $f$\=$or\ k\in1..\emph{rows}(M)$\\
    \>\vline\ \>\vline\ $v_k\leftarrow M_{k,1}$\\
    \>\vline\ \>\vline\ $f$\=$\emph{or}\ j\in2..c-1$\\
    \>\vline\ \>\vline\ \>\vline\ $v_k\leftarrow \emph{conc}(v_k,M_{k,j})\ \ \emph{if}\ \ M_{k,j}\neq0$\\
    \>\vline\ \>\vline\ \>\vline\ $o$\=$\emph{therwise}$\\
    \>\vline\ \>\vline\ \>\vline\ \>\vline\ $sw\leftarrow0$\\
    \>\vline\ \>\vline\ \>\vline\ \>\vline\ $f$\=$or\ i\in j+1..c$\\
    \>\vline\ \>\vline\ \>\vline\ \>\vline\ \>\ $i$\=$f\ M_{k,i}\neq0$\\
    \>\vline\ \>\vline\ \>\vline\ \>\vline\ \>\ \>\vline\ $sw\leftarrow1$\\
    \>\vline\ \>\vline\ \>\vline\ \>\vline\ \>\ \>\vline\ $\emph{break}$\\
    \>\vline\ \>\vline\ \>\vline\ \>\vline\ $v_k\leftarrow\emph{conc}(v_k,M_{k,j})\ \ \emph{if}\ \ sw\textbf{=}1$\\
    \>\vline\ \>\vline\ $v_k\leftarrow conc(v_k,M_{k,c})\ \ \emph{if}\ \ M_{k,c}\neq0$\\
    \>\vline\ $\emph{return}\ \ v$\\
  \end{tabbing}
  Matrix concatenation program does not concatenate on zero if on that line after the zero we only have zeros on every column. Obviously, zeros preceding a non--zero number have no value.
\end{prog}

Examples of calling the program $\emph{concM}$:
\[
 M:=\left(
      \begin{array}{ccc}
        1 & 3 & 5 \\
 & 11 & 13 \\
        17 & & 23 \\
      \end{array}
    \right)\ \ \emph{concM}(M)=\left(
                          \begin{array}{c}
                            135 \\
                            1113 \\
                            17023 \\
                          \end{array}
                        \right)~,
\]
\[
 M:=\left(
      \begin{array}{ccc}
 & 1 & \\
        1 & 2 & 1 \\
 & 1 & \\
      \end{array}
    \right)\ \ \emph{concM}(M)=\left(
                          \begin{array}{c}
                            1 \\
                            121 \\
                            1 \\
                          \end{array}
                        \right)~.
\]
Using the function $\emph{conc}$ and the matrix concatenation program one can generate carpet numbers. For generating carpet numbers we present a program which generates \emph{cellular matrices}. Let a vector $v$ of $m$ size smaller than the size of the cell matrix that is generated.
\begin{defn}\label{Definitia Matrici Celulare}
  By \emph{cellular matrix} we understand a square matrix having an odd number of lines. Matrix values are only the values of the vector $v$ and eventually 0. The display of vector values $v$ is made by a rule that is based on a function $f$.
\end{defn}

\begin{prog}\label{Program GMC} Program for generating cellular matrices, of $n\times n$ ($n$ odd) size, with the values of the vector $v$ following the rule imposed by the function $f$.
  \begin{tabbing}
    $\emph{GMC}(v,n,f):=$\=\vline\ $\emph{return}\ \ "\emph{Nr.}\ \emph{cols}\ and\ \emph{rows}\ \emph{odd}"\ \ \emph{if}\ \ \mod(n,2)=0$\\
    \>\vline\ $m\leftarrow\frac{n+1}{2}$\\
    \>\vline\ $A_{n,n}\leftarrow0$\\
    \>\vline\ $f$\=$or\ k\in1..n$\\
    \>\vline\ \>\ $f$\=$or\ j\in1..n$\\
    \>\vline\ \>\ \>\vline\ $q\leftarrow\abs{k-m}$\\
    \>\vline\ \>\ \>\vline\ $s\leftarrow\abs{j-m}$\\
    \>\vline\ \>\ \>\vline\ $A_{k,j}\leftarrow v_{f(q,s)+1}\ \ \emph{if}\ \ f(q,s)+1\le last(v)$\\
    \>\vline\ $\emph{return}\ A$\\
  \end{tabbing}
\end{prog}

\begin{exem} for calling the program to generate cellular matrices. Let the vector $v=(13\ \ 7\ \ 1)^\textrm{T}$ and function $f_1(q,s):=q+s$. Thus we have:
\[
 M:=\emph{GMC}(v,7,f_1)=\left(\begin{array}{ccccccc}
 0 & 0 & 0 & 0 & 0 & 0 & 0\\
 0 & 0 & 0 & 1 & 0 & 0 & 0\\
 0 & 0 & 1 & 7 & 1 & 0 & 0\\
 0 & 1 & 7 & 13 & 7 & 1 & 0\\
 0 & 0 & 1 & 7 & 1 & 0 & 0\\
 0 & 0 & 0 & 1 & 0 & 0 & 0\\
 0 & 0 & 0 & 0 & 0 & 0 & 0\\
                       \end{array}\right)~,
\]
\[
 N:=\emph{concM}(M)\rightarrow\left(\begin{array}{c}
                              0 \\
                              1 \\
                              171 \\
                              171371 \\
                              171 \\
                              1 \\
                              0 \\
                            \end{array}\right)~.
\]
With the command sequence: $k:=1..\emph{last}(N)$,
\[
 \emph{IsPrime}(N_k)\rightarrow\left(\begin{array}{c}
 0\\
 0\\
 0\\
 0\\
 0\\
 0\\
 0\\
                              \end{array}\right)~,\ \
 \emph{factor} N_k\rightarrow\left(\begin{array}{c}
                              0 \\
                              1 \\
                              3^2\cdot19\\
                              409\cdot419 \\
                              3^2\cdot19 \\
                              1 \\
                              0 \\
                            \end{array}\right)~,
\]
we can study the nature of these numbers by using functions Mathcad $\emph{factor}$, $\emph{IsPrime}$, etc.
\end{exem}
As it turns, the function $f_1(q,s)=q+s$ will generate the matrix
\[
 M:=\left(\begin{array}{ccccccc}
 0 & 0 & 0 & 0 & 0 & 0 & 0\\
 0 & 0 & 0 & v_3 & 0 & 0 & 0\\
 0 & 0 & v_3 & v_2 & v_3 & 0 & 0\\
 0 & v_3 & v_2 & v_1 & v_2 & v_3 & 0\\
 0 & 0 & v_3 & v_2 & v_3 & 0 & 0\\
 0 & 0 & 0 & v_3 & 0 & 0 & 0\\
 0 & 0 & 0 & 0 & 0 & 0 & 0\\
          \end{array}\right)~,
\]
therefore if the vector has $m$ ($m\le n$) elements and the generated matrix size is $2n+1$ it will result a matrix that has in its center $v_1$, then, around $v_2$ and so on. By concatenating the matrix, the following carpet number will result:
\[
 concM(M)=\left(
   \begin{array}{c}
     0 \\
     v_3 \\
     \overline{v_3v_2v_3} \\
     \overline{v_3v_2v_1v_2v_3} \\
     \overline{v_3v_2v_3} \\
     v_3 \\
     0 \\
   \end{array}
 \right)
\]
or, if we concatenate the matrix $M_1:=\emph{submatrix}(M,1,\emph{rows}(M),1,n+1)$ will result:
\[
 \emph{concM}(M_1)=\left(
   \begin{array}{c}
     0 \\
     v_3 \\
     \overline{v_3v_2} \\
     \overline{v_3v_2v_1} \\
     \overline{v_3v_2} \\
     v_3 \\
     0 \\
   \end{array}
 \right)~,
\]
or, if we concatenate the matrix $M_2:=\emph{submatrix}(M,1,n+1,1,n+1)$ will result:
\[
 \emph{concM}(M_2)=\left(
   \begin{array}{c}
     0 \\
     v_3 \\
     \overline{v_3v_2} \\
     \overline{v_3v_2v_1} \\
   \end{array}
 \right)~.
\]

\section{Carpet Numbers Study}

As seen, the function $\emph{conc}$, the program $\emph{GCM}$ with the function $f$ and the subprogram $\emph{concM}$ allow us to generate a great diversity of carpet numbers. We of{}fer a list of functions $f$ for generating \emph{cellular matrices} with interesting structures.:
\begin{eqnarray}
  f_1(k,j) &=& k+j~, \label{f1}\\
  f_2(k,j) &=& \abs{k-j}~, \label{f2}\\
  f_3(k,j) &=& \max(k,j)~, \label{f3}\\
  f_4(k,j) &=& \min(k,j)~, \label{f4}\\
  f_5(k,j) &=& k\cdot j~, \label{f5}\\
  f_6(k,j) &=& \left\lfloor\frac{k+1}{j+1}\right\rfloor~, \label{f6}\\
  f_7(k,j) &=& \left\lceil\frac{k}{j+1}\right\rceil~, \label{f7}\\
  f_8(k,j) &=& \min(\abs{k-j},k,j)~, \label{f8}\\
  f_9(k,j) &=& \left\lceil\frac{\max(k,j)+1}{\min(k,j)+1}\right\rceil~, \label{f9}\\
  f_{10}(k,j) &=& \min(S(k+1),S(j+1)~, \label{f10}\\
  f_{11}(k,j) &=& \left\lceil k\cdot\sin(j)^3+j\cos(k)^3\right\rceil~, \label{f11}\\
  f_{12}(k,j) &=& \left\lfloor\frac{k+2j}{3}\right\rfloor~. \label{f12}
\end{eqnarray}
where $S$ is the Smarandache function.

Let the vector $v:=(1\ \ 3\ \ 9\ \ 7)^\textrm{T}$, therefore the vector size is $m=4$, cell matrices that we generate to be of size $9$.
\begin{enumerate}
  \item The case for generating function $f_1$ given by formula (\ref{f1}).
\[M_1:=\emph{GMC}(v,9,f_1)=\left(\begin{array}{ccccccccc}
 0 & 0 & 0 & 0 & 0 & 0 & 0 & 0 & 0\\
 0 & 0 & 0 & 0 & 7 & 0 & 0 & 0 & 0\\
 0 & 0 & 0 & 7 & 9 & 7 & 0 & 0 & 0\\
 0 & 0 & 7 & 9 & 3 & 9 & 7 & 0 & 0\\
 0 & 7 & 9 & 3 & 1 & 3 & 9 & 7 & 0\\
 0 & 0 & 7 & 9 & 3 & 9 & 7 & 0 & 0\\
 0 & 0 & 0 & 7 & 9 & 7 & 0 & 0 & 0\\
 0 & 0 & 0 & 0 & 7 & 0 & 0 & 0 & 0\\
 0 & 0 & 0 & 0 & 0 & 0 & 0 & 0 & 0\\
                          \end{array}\right)
\]
\[N_1:=\emph{concM}(M_1)\rightarrow\left(\begin{array}{c}
                                    0\\
                                    7\\
                                    797\\
                                    79397\\
                                    7931397\\
                                    79397\\
                                    797\\
                                    7\\
                                    0\\
                                  \end{array}\right)
\]
\[k:=1..\emph{last}(N_1)~,\ \
N_{1_k}\emph{factor}\rightarrow\left(\begin{array}{c}
                                  0\\
                                  7\\
                                  797\\
                                  79397\\
                                  3\cdot53\cdot83\cdot601\\
                                  79397\\
                                  797\\
                                  7\\
                                  0\\
                                \end{array}\right)~.
\]
\item The case for generating function $f_2$ given by formula (\ref{f2}).
\[M_2:=\emph{GMC}(v,9,f_2)=\left(\begin{array}{ccccccccc}
                            1 & 3 & 9 & 7 & 0 & 7 & 9 & 3 & 1\\
                            3 & 1 & 3 & 9 & 7 & 9 & 3 & 1 & 3\\
                            9 & 3 & 1 & 3 & 9 & 3 & 1 & 3 & 9\\
                            7 & 9 & 3 & 1 & 3 & 1 & 3 & 9 & 7\\
                            0 & 7 & 9 & 3 & 1 & 3 & 9 & 7 & 0\\
                            7 & 9 & 3 & 1 & 3 & 1 & 3 & 9 & 7\\
                            9 & 3 & 1 & 3 & 9 & 3 & 1 & 3 & 9\\
                            3 & 1 & 3 & 9 & 7 & 9 & 3 & 1 & 3\\
                            1 & 3 & 9 & 7 & 0 & 7 & 9 & 3 & 1\\
                          \end{array}\right)
\]
\[N_2:=\emph{concM}(M_2)\rightarrow\left(\begin{array}{c}
                                    139707931\\
                                    313979313\\
                                    931393139\\
                                    793131397\\
                                    7931397\\
                                    793131397\\
                                    931393139\\
                                    313979313\\
                                    139707931\\
                                  \end{array}\right)
\]
\[k:=1..\emph{last}(N_2)~,\ \
N_{2_k}\emph{factor}\rightarrow\left(\begin{array}{c}
                                11^2\cdot19\cdot67\cdot907\\
                                3\cdot19\cdot2347^2\\
                                601\cdot1549739\\
                                793131397\\
                                3\cdot53\cdot83\cdot601\\
                                793131397\\
                                601\cdot1549739\\
                                3\cdot19\cdot2347^2\\
                                11^2\cdot19\cdot67\cdot907\\
                              \end{array}\right)~.
\]
\item The case for generating function $f_3$ given by formula (\ref{f3}).
\[M_3:=\emph{GMC}(v,9,f_3)=\left(\begin{array}{ccccccccc}
 0 & 0 & 0 & 0 & 0 & 0 & 0 & 0 & 0\\
 0 & 7 & 7 & 7 & 7 & 7 & 7 & 7 & 0\\
 0 & 7 & 9 & 9 & 9 & 9 & 9 & 7 & 0\\
 0 & 7 & 9 & 3 & 3 & 3 & 9 & 7 & 0\\
 0 & 7 & 9 & 3 & 1 & 3 & 9 & 7 & 0\\
 0 & 7 & 9 & 3 & 3 & 3 & 9 & 7 & 0\\
 0 & 7 & 9 & 9 & 9 & 9 & 9 & 7 & 0\\
 0 & 7 & 7 & 7 & 7 & 7 & 7 & 7 & 0\\
 0 & 0 & 0 & 0 & 0 & 0 & 0 & 0 & 0\\
                          \end{array}\right)
\]
\[N_3:=\emph{concM}(M_3)\rightarrow\left(\begin{array}{c}
                                    0\\
                                    7777777\\
                                    7999997\\
                                    7933397\\
                                    7931397\\
                                    7933397\\
                                    7999997\\
                                    7777777\\
                                    0\\
                                  \end{array}\right)
\]
\[k:=1..\emph{last}(N_3)~,\ \
N_{3_k}\emph{factor}\rightarrow\left(\begin{array}{c}
                                0\\
                                7\cdot239\cdot4649\\
                                73\cdot109589\\
                                7933397\\
                                3\cdot53\cdot83\cdot601\\
                                7933397\\
                                73\cdot109589\\
                                7\cdot239\cdot4649\\
                                0\\
                              \end{array}\right)~.
\]
\item The case for generating function $f_4$ given by formula (\ref{f4}).
\[M_4:=\emph{GMC}(v,9,f_4)=\left(\begin{array}{ccccccccc}
                            0 & 7 & 9 & 3 & 1 & 3 & 9 & 7 & 0\\
                            7 & 7 & 9 & 3 & 1 & 3 & 9 & 7 & 7\\
                            9 & 9 & 9 & 3 & 1 & 3 & 9 & 9 & 9\\
                            3 & 3 & 3 & 3 & 1 & 3 & 3 & 3 & 3\\
                            1 & 1 & 1 & 1 & 1 & 1 & 1 & 1 & 1\\
                            3 & 3 & 3 & 3 & 1 & 3 & 3 & 3 & 3\\
                            9 & 9 & 9 & 3 & 1 & 3 & 9 & 9 & 9\\
                            7 & 7 & 9 & 3 & 1 & 3 & 9 & 7 & 7\\
                            0 & 7 & 9 & 3 & 1 & 3 & 9 & 7 & 0\\
                          \end{array}\right)
\]
\[N_4:=\emph{concM}(M_4)\rightarrow\left(\begin{array}{c}
                                    7931397\\
                                    779313977\\
                                    999313999\\
                                    333313333\\
                                    111111111\\
                                    333313333\\
                                    999313999\\
                                    779313977\\
                                    7931397\\
                                  \end{array}\right)
\]
\[k:=1..\emph{last}(N_4)~,\ \
N_{4_k}\emph{factor}\rightarrow\left(\begin{array}{c}
                                3\cdot53\cdot83\cdot601\\
                                13\cdot5657\cdot10597\\
                                263\cdot761\cdot4993\\
                                19\cdot31\cdot61\cdot9277\\
                                3^2\cdot37\cdot333667\\
                                19\cdot31\cdot61\cdot9277\\
                                263\cdot761\cdot4993\\
                                13\cdot5657\cdot10597\\
                                3\cdot53\cdot83\cdot601\\
                              \end{array}\right)~.
\]
\item The case for generating function $f_5$ given by formula (\ref{f5}).
\[M_5:=\emph{GMC}(v,9,f_5)=\left(\begin{array}{ccccccccc}
 0 & 0 & 0 & 0 & 1 & 0 & 0 & 0 & 0\\
 0 & 0 & 0 & 7 & 1 & 7 & 0 & 0 & 0\\
 0 & 0 & 0 & 9 & 1 & 9 & 0 & 0 & 0\\
 0 & 7 & 9 & 3 & 1 & 3 & 9 & 7 & 0\\
 1 & 1 & 1 & 1 & 1 & 1 & 1 & 1 & 1\\
 0 & 7 & 9 & 3 & 1 & 3 & 9 & 7 & 0\\
 0 & 0 & 0 & 9 & 1 & 9 & 0 & 0 & 0\\
 0 & 0 & 0 & 7 & 1 & 7 & 0 & 0 & 0\\
 0 & 0 & 0 & 0 & 1 & 0 & 0 & 0 & 0\\
                          \end{array}\right)
\]
\[N_5:=\emph{concM}(M_5)\rightarrow\left(\begin{array}{c}
                                    1\\
                                    717\\
                                    919\\
                                    7931397\\
                                    111111111\\
                                    7931397\\
                                    919\\
                                    717\\
                                    1\\
                                  \end{array}\right)
\]
\[k:=1..\emph{last}(N_5)~,\ \
N_{5_k}\emph{factor}\rightarrow\left(\begin{array}{c}
                                1\\
                                3\cdot239\\
                                919\\
                                3\cdot53\cdot83\cdot601\\
                                3^2\cdot37\cdot333667\\
                                3\cdot53\cdot83\cdot601\\
                                919\\
                                3\cdot239\\
                                1\\
                              \end{array}\right)~.
\]
\end{enumerate}

\section{Other Carpet Numbers Study}

Obviously, square matrices can be introduced using formulas or manually, and then to apply the concatenation program $\emph{concM}$. Using the vector containing carpet numbers we can proceed to study the numbers the way we did above.

\begin{enumerate}
  \item Carpet numbers generated by the series given by formula, \citep{SmarandacheArizona,Smarandache1995}:
  \[
   C(n,k)=4n\prod_{j=1}^k(4n-4j+1), \ \ \textnormal{for}\ \ 1\le k\le n
  \]
  and $C(n,0)=1$ for any $n\in\mathbb{N}$. 
\begin{prog}\label{Program GenM} to generate the matrix $M$. 
  \begin{tabbing}
    $\emph{GenM}(D):=$\=\vline\ $f$\=$or\ \ n\in1..D$\\
    \>\vline\ \>\ $f$\=$or\ k\in1..n$\\
    \>\vline\ \>\ \>\ $M_{n,n-k+1}\leftarrow C(n-1,k-1)$\\
    \>\vline\ $\emph{return}\ \ M$\\
  \end{tabbing}
\end{prog}
If $D:=7$ and we command $M:=\emph{GenM}(D)$, then we have:
\[M=\left(\begin{array}{ccccccccc}
                            1 & 0 & 0 & 0 & 0 & 0 & 0\\
                            4 & 1 & 0 & 0 & 0 & 0 & 0\\
                            40 & 8 & 1 & 0 & 0 & 0 & 0\\
                            504 & 108 & 12 & 1 & 0 & 0 & 0\\
                            9360 & 1872 & 208 & 16 & 1 & 0 & 0\\
                            198900 & 39780 & 4420 & 340 & 20 & 1 & 0\\
                            5012280 & 1002456 & 111384 & 8568 & 504 & 24 & 1\\
                          \end{array}\right)
\]
\[N:=\emph{concM}(M)\rightarrow\left(\begin{array}{c}
                                    1\\
                                    41\\
                                    4081\\
                                    504108121\\
                                    93601872208161\\
                                    198900397804420340201\\
                                    501228010024561113848568504241\\
                                  \end{array}\right)
\]
\[
 k:=1..\emph{last}(N)~,
\]
\[N_k\ \emph{factor}\rightarrow\left(\begin{array}{c}
                               1\\
                               41\\
                               7\cdot11\cdot53\\
                               11\cdot239\cdot191749\\
                               3^2\cdot17\cdot1367\cdot447532511\\
                               83\cdot2396390334993016147\\
                               31\cdot3169\cdot5341781\cdot955136233518518099\\
                              \end{array}\right)~.
\]
\item Carpet numbers generated by Pascal\index{Pascal B.} triangle. Let the matrix by Pascal\rq{s} triangle values:
\begin{prog}\label{Program Combinarilor} generating the matrix by Pascal\rq{s} triangle.
  \begin{tabbing}
    $\emph{Pascal}(n):=$\=\vline\ $f$\=$or\ k\in1..n+1$\\
    \>\vline\ \>\ $f$\=$or\ j\in1..k$\\
    \>\vline\ \>\ \>\ $M_{k,j}\leftarrow combin(k-1,j-1)$\\
    \>\vline\ $\emph{return}\ \ M$\\
  \end{tabbing}
\end{prog}

The Pascal program generates a matrix containing Pascal\rq{s} triangle.
\begin{multline*}
  M:=\emph{Pascal}(10)= \\
  \left(\begin{array}{ccccccccccc}
                    1 & 0 & 0 & 0 & 0 & 0 & 0 & 0 & 0 & 0 & 0\\
                    1 & 1 & 0 & 0 & 0 & 0 & 0 & 0 & 0 & 0 & 0\\
                    1 & 2 & 1 & 0 & 0 & 0 & 0 & 0 & 0 & 0 & 0\\
                    1 & 3 & 3 & 1 & 0 & 0 & 0 & 0 & 0 & 0 & 0\\
                    1 & 4 & 6 & 4 & 1 & 0 & 0 & 0 & 0 & 0 & 0\\
                    1 & 5 & 10 & 10 & 5 & 1 & 0 & 0 & 0 & 0 & 0\\
                    1 & 6 & 15 & 20 & 15 & 6 & 1 & 0 & 0 & 0 & 0\\
                    1 & 7 & 21 & 35 & 35 & 21 & 7 & 1 & 0 & 0 & 0\\
                    1 & 8 & 28 & 56 & 70 & 56 & 28 & 8 & 1 & 0 & 0\\
                    1 & 9 & 36 & 84 & 126 & 126 & 84 & 36 & 9 & 1 & 0\\
                    1 & 10 & 45 & 120 & 210 & 252 & 210 & 120 & 45 & 10 & 1\\
                  \end{array}\right)~.
\end{multline*}
Using the program $\emph{concM}$ for concatenating the components of the matrix, it results the carpet numbers:
\[N:=\emph{concM}(M)\rightarrow\left(\begin{array}{c}
                                1\\
                                11\\
                                121\\
                                1331\\
                                14641\\
                                15101051\\
                                1615201561\\
                                172135352171\\
                                18285670562881\\
                                193684126126843691\\
                                1104512021025221012045101\\
                          \end{array}\right)~,
\]
whose decomposition in prime factors is:
\[
 k:=1..\emph{last}(N)~,
\]
\[N_k\ \emph{factor}\rightarrow\left(\begin{array}{c}
                                1\\
                                11\\
                                11^2\\
                                11^3\\
                                11^4\\
                                7\cdot2157293\\
                                43\cdot37562827\\
                                29\cdot5935701799\\
                                18285670562881\\
                                5647\cdot34298587945253\\
                                13\cdot197\cdot4649\cdot92768668286052709\\
                              \end{array}\right)~.
\]
\item Carpet numbers generated by primes.
\begin{prog}\label{Program MPrime} for generating matrix by primes.
  \begin{tabbing}
    $\emph{MPrime}(n):=$\=\vline\ $f$\=$or\ k\in1..n+1$\\
    \>\vline\ \>\ $f$\=$or\ j\in1..k$\\
    \>\vline\ \>\ \>\ $M_{k,j}\leftarrow \emph{prime}_j$\\
    \>\vline\ $\emph{return}\ \ M$\\
  \end{tabbing}
\end{prog}

Before using the program $\emph{MPrime}$ we have to generate the vector of primes by instruction: $\emph{prime}:=\emph{SEPC}(100)$, where we call the program \ref{ProgramSEPC}.
\[M:=\emph{MPrime}(10)=\left(
                  \begin{array}{cccccccccc}
                    2 & 0 & 0 & 0 & 0 & 0 & 0 & 0 & 0 & 0\\
                    2 & 3 & 0 & 0 & 0 & 0 & 0 & 0 & 0 & 0\\
                    2 & 3 & 5 & 0 & 0 & 0 & 0 & 0 & 0 & 0\\
                    2 & 3 & 5 & 7 & 0 & 0 & 0 & 0 & 0 & 0\\
                    2 & 3 & 5 & 7 & 11 & 0 & 0 & 0 & 0 & 0\\
                    2 & 3 & 5 & 7 & 11 & 13 & 0 & 0 & 0 & 0\\
                    2 & 3 & 5 & 7 & 11 & 13 & 17 & 0 & 0 & 0\\
                    2 & 3 & 5 & 7 & 11 & 13 & 17 & 19 & 0 & 0\\
                    2 & 3 & 5 & 7 & 11 & 13 & 17 & 19 & 23 & 0\\
                    2 & 3 & 5 & 7 & 11 & 13 & 17 & 19 & 23 & 29\\
                  \end{array}
                \right)
\]
Using the program $\emph{concM}$ for concatenating the matrix components we have the carpet numbers:
\[N:=\emph{concM}(M)\rightarrow\left(\begin{array}{c}
                                2\\
                                23\\
                                235\\
                                2357\\
                                235711\\
                                23571113\\
                                2357111317\\
                                235711131719\\
                                23571113171923\\
                                2357111317192329\\
                          \end{array}\right)~,
\]
whose decomposition in prime factors is:
\[
 k:=1..\emph{last}(N)~,
\]
\[N_k\ \emph{factor}\rightarrow\left(\begin{array}{c}
                                2\\
                                23\\
                                5\cdot47\\
                                2357\\
                                7\cdot151\cdot223\\
                                23\cdot29\cdot35339\\
                                11\cdot214282847\\
                                7\cdot4363\cdot7717859\\
                                61\cdot478943\cdot806801\\
                                3\cdot4243\cdot185176472401\\
                              \end{array}\right)~.
\]

\item Carpet numbers generated by primes starting with 2, as it follows:

\[M:=\left(\begin{array}{cccccccccc}
             2 & 0 & 0 & 0 & 0 & 0 & 0 & 0 & 0\\
             3 & 5 & 0 & 0 & 0 & 0 & 0 & 0 & 0\\
             7 & 11 & 13 & 0 & 0 & 0 & 0 & 0 & 0\\
             17 & 19 & 23 & 29 & 0 & 0 & 0 & 0 & 0\\
             31 & 37 & 41 & 43 & 47 & 0 & 0 & 0 & 0\\
             53 & 59 & 61 & 67 & 71 & 73 & 0 & 0 & 0\\
             79 & 83 & 89 & 97 & 101 & 103 & 107 & 0 & 0\\
             109 & 113 & 127 & 131 & 137 & 139 & 149 & 151 & 0\\
             157 & 163 & 167 & 173 & 179 & 181 & 191 & 193 & 197\\
           \end{array}\right)
\]
Using the program $\emph{concM}$ for concatenating the matrix components we have the carpet numbers:
\[N:=\emph{concM}(M)\rightarrow\left(\begin{array}{c}
                                2\\
                                35\\
                                71113\\
                                17192329\\
                                3137414347\\
                                535961677173\\
                                79838997101103107\\
                                109113127131137139149151\\
                                157163167173179181191193197\\
                          \end{array}\right)~,
\]
whose decomposition in prime factors is:
\[
 k:=1..\emph{last}(N)~,
\]
\[N_k\ \emph{factor}\rightarrow\left(\begin{array}{c}
                                2\\
                                5\cdot7\\
                                7\cdot10159\\
                                7\cdot11\cdot223277\\
                                2903\cdot1080749\\
                                3\cdot13\cdot13742607107\\
                                7\cdot41\cdot3449\cdot80656613189\\
                                3\cdot857\cdot35039761\cdot1211194223021\\
                                10491377789\cdot14980221886391473\\
                              \end{array}\right)~.
\]

\item Carpet numbers generated by primes starting with 3, as it follows:

\[M:=\left(\begin{array}{cccccccccc}
             3 & 0 & 0 & 0 & 0 & 0 & 0 & 0 & 0\\
             5 & 7 & 0 & 0 & 0 & 0 & 0 & 0 & 0\\
             11 & 13 & 17 & 0 & 0 & 0 & 0 & 0 & 0\\
             19 & 23 & 29 & 31 & 0 & 0 & 0 & 0 & 0\\
             37 & 41 & 43 & 47 & 53 & 0 & 0 & 0 & 0\\
             59 & 61 & 67 & 71 & 73 & 79 & 0 & 0 & 0\\
             83 & 89 & 97 & 101 & 103 & 107 & 109 & 0 & 0\\
             113 & 127 & 131 & 137 & 139 & 149 & 151 & 157 & 0\\
             163 & 167 & 173 & 179 & 181 & 191 & 193 & 197 & 199\\
           \end{array}\right)
\]
Using the program $\emph{concM}$ for concatenating the matrix components we have the carpet numbers:
\[N:=\emph{concM}(M)\rightarrow\left(\begin{array}{c}
                                3\\
                                57\\
                                111317\\
                                19232931\\
                                3741434753\\
                                596167717379\\
                                838997101103107109\\
                                113127131137139149151157\\
                                163167173179181191193197199\\
                          \end{array}\right)~,
\]
whose decomposition in prime factors is:
\[
 k:=1..\emph{last}(N)~,
\]
\[N_k\ \emph{factor}\rightarrow\left(\begin{array}{c}
                                3\\
                                3\cdot19\\
                                111317\\
                                3\cdot6410977\\
                                7\cdot577\cdot926327\\
                                13\cdot45859055183\\
                                3251\cdot258073546940359\\
                                3\cdot41\cdot467\cdot1969449193731640277\\
                                7\cdot1931\cdot47123\cdot2095837\cdot122225561597\\
                              \end{array}\right)~.
\]
\end{enumerate}

\section{Ulam Matrix}

In the Ulam matrices, \citep{Ulam1930,Jech2003}, the natural numbers are placed on a spiral that starts from the center of the matrix. Primes to 169 are in red text. On the main diagonal of the matrix, there are the perfect squares, in blue text
\begin{prog}\label{Program Matrice Ulam} for generating Ulam matrix.
  \begin{tabbing}
    $\emph{MUlam}(n):=$\=\vline\ $\emph{return}\ "\emph{Error.}\ n\ \ \emph{if}\ \ \mod(n,2)=0\vee n\le1$\\
    \>\vline\ $A_{n,n}\leftarrow0$\\
    \>\vline\ $m\leftarrow\dfrac{n+1}{2}$\\
    \>\vline\ $I\leftarrow(m\ \ m)$\\
    \>\vline\ $k_f\leftarrow rows(I)$\\
    \>\vline\ $f$\=$or\ s\in1..n-m$\\
    \>\vline\ \>\vline\ $k_i\leftarrow k_f$\\
    \>\vline\ \>\vline\ $c\leftarrow s$\\
    \>\vline\ \>\vline\ $f$\=$or\ r\in s-1..-s$\\
    \>\vline\ \>\vline\ \>\ $I\leftarrow stack[I,(m+r\ \ m+c)]$\\
    \>\vline\ \>\vline\ $f$\=$or\ c\in s-1..-s$\\
    \>\vline\ \>\vline\ \>\ $I\leftarrow stack[I,(m+r\ \ m+c)]$\\
    \>\vline\ \>\vline\ $f$\=$or\ r\in -s+1..s$\\
    \>\vline\ \>\vline\ \>\ $I\leftarrow stack[I,(m+r\ \ m+c)]$\\
    \>\vline\ \>\vline\ $f$\=$or\ c\in -s+1..s$\\
    \>\vline\ \>\vline\ \>\ $I\leftarrow stack[I,(m+r\ \ m+c)]$\\
    \>\vline\ \>\vline\ $k_f\leftarrow rows(I)$\\
    \>\vline\ \>\vline\ $f$\=$or\ k\in k_i..k_f$\\
    \>\vline\ \>\vline\ \>\ $A_{(I_{k,1},I_{k,2})}\leftarrow k$\\
    \>\vline\ $\emph{return}\ \ A$\\
  \end{tabbing}
  For exemplif{}ication, we generate the Ulam matrix of 13 lines and 13 columns by using command $\emph{MUlam}(13)$.

\begin{multline*}
  U:=\emph{MUlam(13)}=\\
  \left[\begin{array}{ccccccccccccc}
    145 & \textcolor{blue}{144} & 143 & 142 & 141 & 140 & \textcolor{red}{139} & 138 & \textcolor{red}{137} & 136 & 135 & 134 & 133\\
    146 & \textcolor{red}{101} & \textcolor{blue}{100} & 99 & 98 & \textcolor{red}{97} & 96 & 95 & 94 & 93 & 92 & 91 & 132\\
    147 & 102 & 65 & \textcolor{blue}{64} & 63 & 62 & \textcolor{red}{61} & 60 & \textcolor{red}{59} & 58 & 57 & 90 & \textcolor{red}{131}\\
    148 & \textcolor{red}{103} & 66 & \textcolor{red}{37} & \textcolor{blue}{36} & 35 & 34 & 33 & 32 & \textcolor{red}{31} & 56 & \textcolor{red}{89} & 130\\
    \textcolor{red}{149} & 104 & \textcolor{red}{67} & 38 & \textcolor{red}{17} & \textcolor{blue}{16} & 15 & 14 & \textcolor{red}{13} & 30 & 55 & 88 & 129\\
    150 & 105 & 68 & 39 & 18 & \textcolor{red}{5} & \textcolor{blue}{4} & \textcolor{red}{3} & 12 & \textcolor{red}{29} & 54 & 87 & 128\\
    \textcolor{red}{151} & 106 & 69 & 40 & \textcolor{red}{19} & 6 & 1 & \textcolor{red}{2} & \textcolor{red}{11} & 28 & \textcolor{red}{53} & 86 & \textcolor{red}{127}\\
    152 & \textcolor{red}{107} & 70 & \textcolor{red}{41} & 20 & \textcolor{red}{7} & 8 & \textcolor{blue}{9} & 10 & 27 & 52 & 85 & 126\\
    153 & 108 & \textcolor{red}{71} & 42 & 21 & 22 & \textcolor{red}{23} & 24 & \textcolor{blue}{25} & 26 & 51 & 84 & 125\\
    154 & \textcolor{red}{109} & 72 & \textcolor{red}{43} & 44 & 45 & 46 & \textcolor{red}{47} & 48 & \textcolor{blue}{49} & 50 & \textcolor{red}{83} & 124\\
    155 & 110 & \textcolor{red}{73} & 74 & 75 & 76 & 77 & 78 & \textcolor{red}{79} & 80 & \textcolor{blue}{81} & 82 & 123\\
    156 & 111 & 112 & \textcolor{red}{113} & 114 & 115 & 116 & 117 & 118 & 119 & 120 & \textcolor{blue}{121} & 122\\
    \textcolor{red}{157} & 158 & 159 & 160 & 161 & 162 & \textcolor{red}{163} & 164 & 165 & 166 & \textcolor{red}{167} & 168 & \textcolor{blue}{169}\\
  \end{array}\right]
\end{multline*}
\end{prog}

Using the command $\emph{concM}(\emph{submatrix}(U,1,7,1,13)\rightarrow$, for concatenating the components of submatrix $U$, we get the carpet Ulam numbers:
\[ \left(\begin{array}{c}
           145144143142141140139138137136135134133\\
           146101100999897969594939291132\\
           14710265646362616059585790131\\
           14810366373635343332315689130\\
           14910467381716151413305588129\\
           15010568391854312295487128\\
           15110669401961211285386127
     \end{array}\right)
\]

Ulam matrix only with primes, then concatenated and factorized:
\[
  \emph{Up}:=\left[\begin{array}{ccccccccccccc}
          0 & 0 & 0 & 0 & 0 & 0 & 139 & 0 & 137 & 0 & 0 & 0 & 0\\
          0 & 101 & 0 & 0 & 0 & 97 & 0 & 0 & 0 & 0 & 0 & 0 & 0\\
          0 & 0 & 0 & 0 & 0 & 0 & 61 & 0 & 59 & 0 & 0 & 0 & 131\\
          0 & 103 & 0 & 37 & 0 & 0 & 0 & 0 & 0 & 31 & 0 & 89 & 0\\
          149 & 0 & 67 & 0 & 17 & 0 & 0 & 0 & 13 & 0 & 0 & 0 & 0\\
          0 & 0 & 0 & 0 & 0 & 5 & 0 & 3 & 0 & 29 & 0 & 0 & 0\\
          151 & 0 & 0 & 0 & 19 & 0 & 0 & 2 & 11 & 0 & 53 & 0 & 127\\
          0 & 107 & 0 & 41 & 0 & 7 & 0 & 0 & 0 & 0 & 0 & 0 & 0\\
          0 & 0 & 71 & 0 & 0 & 0 & 23 & 0 & 0 & 0 & 0 & 0 & 0\\
          0 & 109 & 0 & 43 & 0 & 0 & 0 & 47 & 0 & 0 & 0 & 83 & 0\\
          0 & 0 & 73 & 0 & 0 & 0 & 0 & 0 & 79 & 0 & 0 & 0 & 0\\
          0 & 0 & 0 & 113 & 0 & 0 & 0 & 0 & 0 & 0 & 0 & 0 & 0\\
          157 & 0 & 0 & 0 & 0 & 0 & 163 & 0 & 0 & 0 & 167 & 0 & 0\\
     \end{array}\right]
\]
\[
 \emph{concM}(Up)\rightarrow
\]
\[
 \left(\begin{array}{c}
     1390137\\
     10100097\\
     61059000131\\
     1030370000031089\\
     14906701700013\\
     503029\\
     15100019002110530127\\
     10704107\\
     7100023\\
     1090430004700083\\
     730000079\\
     113\\
     15700000163000167\\
   \end{array}\right)
  \emph{factor}\rightarrow
   \left(\begin{array}{c}
           3\cdot7\cdot53\cdot1249\\
           3^2\cdot7\cdot160319\\
           \fbox{61059000131}\\
           53\cdot1117153\cdot17402221\\
           3\cdot107\cdot46438323053\\
           41\cdot12269\\
           3\cdot1021\cdot3203\cdot1539123815843\\
           467\cdot22921\\
           7\cdot31\cdot32719\\
           3\cdot83\cdot599\cdot7310913133\\
           \fbox{730000079}\\
           \fbox{113}\\
           7\cdot17\cdot29\cdot43\cdot163\cdot19949\cdot32537\\
     \end{array}\right)
\]

\chapter{Conjectures}

\begin{enumerate}
  \item Coloration conjecture: Anyhow all points of an $m$--dimensional Euclidian space are colored with a f{}inite number of colors, there exists a color which fulf{}ills all distances.
  \item Primes: Let $a_1$, $a_2$, \ldots, $a_n$, be distinct digits, $1\le n\le9$. How many primes can we construct from all these digits only (eventually repeated)?
  \item More generally: when $a_1$, $a_2$, \ldots, $a_n$, and $n$ are positive integers. Conjecture: Inf{}initely many!
  \item Back concatenated prime sequence: 2, 32, 532, 7532, 117532, 13117532, 1713117532, 191713117532, 23191713117532, \ldots~.  Conjecture: There are inf{}initely many primes among the f{}irst sequence numbers!
  \item Back concatenated odd sequence: 1, 31, 531, 7531, 97531, 1197531, 131197531, 15131197531, 1715131197531, \ldots~. Conjecture: There are inf{}initely many primes among these numbers!
  \item Back concatenated even sequence: 2, 42, 642, 8642, 108642, 12108642, 1412108642, 161412108642, \ldots~. Conjecture: None of them is a perfect power!
  \item Wrong numbers: A number $n=\overline{a_1a_2\ldots a_k}$, of at least two digits, with the property: the sequence $a_1$, $a_2$, \ldots, $a_k$, $b_{k+1}$, $b_{k+2}$, \ldots (where $b_{k+i}$ is the product of the previous $k$ terms, for any $i\ge1$) contains $n$ as its term.) The authors conjectured that there is no wrong number (!) Therefore, this sequence is empty.
  \item Even Sequence is generated by choosing $G=\set{2,4,6,8,10,12, \ldots}$, and it is: 2, 24, 246, 2468, 246810, 24681012, \ldots~. Searching the f{}irst 200 terms of the sequence we didn\rq{t} f{}ind any $n$-th perfect power among them, no perfect square, nor even of the form $2p$, where $p$ is a prime or pseudo--prime. Conjecture: There is no $n$-th perfect power term!
  \item Prime-digital sub-sequence "Personal Computer World" Numbers Count of February 1997 presented some of the Smarandache Sequences and related open problems. One of them def{}ines the prime--digital sub--sequence as the ordered set of primes whose digits are all primes: 2, 3, 5, 7, 23, 37, 53, 73, 223, 227, 233, 257, 277, \ldots~. We used a computer program in Ubasic to calculate the f{}irst 100 terms of the sequence. The 100-th term is 33223. \cite{Smith1996} conjectured that the sequence is inf{}inite. In this paper we will prove that this sequence is in fact inf{}inite.
  \item Concatenated Fibonacci sequence: 1, 11, 112, 1123, 11235, 112358, 11235813, 1123581321, 112358132134, \ldots~.
  \item Back concatenated Fibonacci sequence: 1, 11, 211, 3211, 53211, 853211, 13853211, 2113853211, 342113853211, \ldots~. Does any of these numbers is a Fibonacci number? \citep{Marimutha1997}
  \item Special expressions.
      \begin{enumerate}
        \item Perfect powers in special expressions $x^y+y^x$, where $gcd(x,y)=1$, \citep{Castini1995,Castillo1996}. For $x=1,2,\ldots,20$ and $y=1,2,\ldots,20$ one obtains 127 of numbers and  following numbers are primes: 2, 3, 5, 7, 11, 13, 17, 19, 23, 29, 31, 593, 32993, 2097593, 59604644783353249. \cite{Kashihara1996}\index{Kashihara K.} announced that there are only f{}initely many numbers of the above form which are products of factorials. In this note we propose the following conjecture: Let $a$, $b$, and $c$ three integers with $a\cdot b$ nonzero. Then the equation:
            \[
             a\cdot x^y+b\cdot y^x=c\cdot z^n~,
            \]
            with $x$, $y$, $n\ge2$, and $gcd(x,y)=1$, has f{}initely many solutions $(x, y, z, n)$. And we prove some particular cases of it, \cite{Luca1997a,Luca1997b}\index{Luca F.}.
        \item Products of factorials in special expressions. \cite{Castillo1996}\index{Castillo J.} asked how many primes are there in the $n$-expression
            \begin{equation}\label{SpecialExpresion}
              x_1^{x_2}+x_2^{x_3}+\ldots+x_n^{x_1}~,
            \end{equation}
            where $n,x_1,x_2,\ldots,x_n>1$, and $gcd(x_1,x_2,\ldots,x_n)=1$?
            For $n=3$ expression $x_1^{x_2}+x_2^{x_3}+x_3^{x_1}$ has 51 prime numbers: 3, 5, 7, 11, 13, 19, 31, 61, 67, 71, 89, 103, 181, 347, 401, 673, 733, 773, 1301, 2089, 2557, 12497, 33049, 46663, 78857, 98057, 98929, 135329, 262151, 268921, 338323, 390721, 531989, 552241, 794881, 1954097, 2165089, 2985991, 4782977, 5967161, 9765757, 17200609, 35835953, 40356523, 48829699, 387420499, 430513649, 2212731793, 1000000060777, 1000318307057, 1008646564753, where $x_1,x_2,x_3\in\set{1,2,\ldots,12}$. These results were obtained with the following programs:
            \begin{prog}\label{Program SpecialExpresion3} of f{}inding the numbers of the form (\ref{SpecialExpresion}) for $n=3$.
              \begin{tabbing}
                $P3(a_x,b_x,a_y,b_y,a_z,b_z):=$\=\vline\ $j\leftarrow1$\\
                \>\vline\ $f$\=$or\ x\in a_x..b_x$\\
                \>\vline\ \>\ $f$\=$or\ y\in a_y..b_y$\\
                \>\vline\ \>\ \>\ $f$\=$or\ z\in a_z..b_z$\\
                \>\vline\ \>\ \>\ \>\ $i$\=$f\ \emph{gcd}(x,y,z)=1$\\
                \>\vline\ \>\ \>\ \>\ \>\vline\ $\emph{se}_j\leftarrow x^y+y^z+z^x$\\
                \>\vline\ \>\ \>\ \>\ \>\vline\ $j\leftarrow j+1$\\
                \>\vline\ $\emph{sse}\leftarrow\emph{sort}(\emph{se})$\\
                \>\vline\ $k\leftarrow1$\\
                \>\vline\ $s_k\leftarrow \emph{sse}_1$\\
                \>\vline\ $f$\=$or\ j\in2..\emph{last}(\emph{sse})$\\
                \>\vline\ \>\ $i$\=$f\ s_k\neq \emph{sse}_j$\\
                \>\vline\ \>\ \>\ \vline\ $k\leftarrow k+1$\\
                \>\vline\ \>\ \>\ \vline\ $s_k\leftarrow \emph{sse}_j$\\
                \>\vline\ $\emph{return}\ \ s$\\
              \end{tabbing}
              The program uses Mathcad function $\emph{gcd}$, greatest common divisor.
            \end{prog}
            \begin{prog}\label{Program IP} of extraction the prime numbers from a sequences.
              \begin{tabbing}
                $\emph{IP}(s):=$\=\vline\ $j\leftarrow0$\\
                \>\vline\ $f$\=$or\ k\in1..\emph{last}(s)$\\
                \>\vline\ \>\ $i$\=$f\ \emph{IsPrime}(s_k)=1$\\
                \>\vline\ \>\ \>\vline\ $j\leftarrow j+1$\\
                \>\vline\ \>\ \>\vline\ $\emph{ps}_j\leftarrow s_k$\\
                \>\vline\ $\emph{return}\ \ ps$\\
              \end{tabbing}
            The program uses Mathcad function $\emph{IsPrime}$.
            \end{prog}
      \end{enumerate}
      For $n=4$ expression $x_1^{x_2}+x_2^{x_3}+x_3^{x_4}+x_4^{x_1}$ has 50 primes: 5, 7, 11, 13, 23, 29, 37, 43, 47, 71, 89, 103, 107, 109, 113, 137, 149, 157, 193, 199, 211, 257, 271, 277, 293, 313, 631, 677, 929, 1031, 1069, 1153, 1321, 1433, 2017, 2161, 3163, 4057, 4337, 4649, 4789, 5399, 6337, 16111, 18757, 28793, 46727, 54521, 64601, 93319, where $x_1,x_2,x_3,x_4\in\set{1,2,\ldots,5}$. These results have been obtained with a program similar to
      \ref{Program SpecialExpresion3}.
  \item There are inf{}initely many primes which are generalized Smarandache palindromic number GSP1 or GSP2.
\end{enumerate}

\chapter{Algorithms}

\section{Constructive Set}

\subsection{Constructive Set of Digits 1 and 2 }

\begin{defn}\label{Definitia Constructive Set 1,2}\
  \begin{enumerate}
     \item 1, 2 belong to $S$;
     \item if $a$, $b$ belong to $S$, then ab belongs to $S$ too;
     \item only elements obtained by rules 1. and 2. applied a f{}inite number of times belong to $S$.
   \end{enumerate}
\end{defn}

Numbers formed by digits 1 and 2 only: 1, 2, 11, 12, 21, 22, 111, 112, 121, 122, 211, 212, 221, 222, 1111, 1112, 1121, 1122, 1211, 1212, 1221, 1222, 2111, 2112, 2121, 2122, 2211, 2212, 2221, 2222, \ldots~.

\begin{rem}\label{RemConstructiveSet12}\
  \begin{enumerate}
    \item there are $2^k$ numbers of k digits in the sequence, for $k=1,2,3,\ldots$ ;
    \item to obtain from the $k$--digits number group the $(k+1)$--digits number group, just put f{}irst the digit 1 and second the digit 2 in the front of all $k$--digits numbers.
  \end{enumerate}
\end{rem}

\subsection{Constructive Set of Digits 1, 2 and 3}

\begin{defn}\label{Definitia Constructive Set 1,2,3}\
  \begin{enumerate}
     \item 1, 2, 3 belong to S;
     \item if $a$, $b$ belong to $S$, then ab belongs to $S$ too;
     \item only elements obtained by rules 1. and 2. applied a f{}inite number of times belong to $S$.
   \end{enumerate}
\end{defn}

Numbers formed by digits 1, 2, and 3 only: 1, 2, 3, 11, 12, 13, 21, 22, 23, 31, 32, 33, 111, 112, 113, 121, 122, 123, 131, 132, 133, 211, 212, 213, 221, 222, 223, 231, 232, 233, 311, 312, 313, 321, 322, 323, 331, 332, 333, \ldots~.

\begin{rem}\label{RemConstructiveSet123}\
   \begin{enumerate}
     \item there are $3^k$ numbers of $k$ digits in the sequence, for $k=1,2,3, \ldots$;
     \item to obtain from the $k$--digits number group the $(k+1)$--digits number group, just put f{}irst the digit 1, second the digit 2, and third the digit 3 in the front of all $k$--digits numbers.
   \end{enumerate}
\end{rem}

\subsection{Generalized Constructive Set}

\begin{prog}\label{Program Cset} for generating the numbers between limits $\alpha$ and $\beta$ that have the digits from the vector $w$.
  \begin{tabbing}
    $\emph{Cset}(\alpha,\beta,w):=$\=\vline\ $b\leftarrow \emph{last}(w)$\\
    \>\vline\ $j\leftarrow1$\\
    \>\vline\ $f$\=$or\ n\in\alpha..\beta$\\
    \>\vline\ \>\vline\ $d\leftarrow \emph{dn}(n,b)$\\
    \>\vline\ \>\vline\ $f$\=$or\ k\in 1..\emph{last}(d)$\\
    \>\vline\ \>\vline\ \>\ $\emph{wd}_k\leftarrow w_{(d_k+1)}$\\
    \>\vline\ \>\vline\ $\emph{cs}_j\leftarrow \emph{wd}\cdot \emph{Vb}(10,k)$\\
    \>\vline\ \>\vline\ $j\leftarrow j+1$\\
    \>\vline\ $\emph{return}\ \ \emph{cs}$\\
  \end{tabbing}
  The program uses the subprograms $dn$, \ref{ProgramDn}, and function $\emph{Vb}(b,m)$ that returns the vector $(b^m\ b^{m-1}\ \ldots b^0)^\textrm{T}$.
\end{prog}

\begin{enumerate}
  \item The f{}irst 26 numbers from 0 to 25, with digits 3 to 7 are: 3, 7, 73, 77, 733, 737, 773, 777, 7333, 7337, 7373, 7377, 7733, 7737, 7773, 7777, 73333, 73337, 73373, 73377, 73733, 73737, 73773, 73777, 77333, 77337.
  \item The numbers from 0 to 30, with digits 1, 3 and 7 are: 1, 3, 7, 31, 33, 37, 71, 73, 77, 311, 313, 317, 331, 333, 337, 371, 373, 377, 711, 713, 717, 731, 733, 737, 771, 773, 777, 3111, 3113, 3117, 3131.
  \item The numbers from 3 to 70, with digits 1, 3, 7 and 9 are: 9, 31, 33, 37, 39, 71, 73, 77, 79, 91, 93, 97, 99, 311, 313, 317, 319, 331, 333, 337, 339, 371, 373, 377, 379, 391, 393, 397, 399, 711, 713, 717, 719, 731, 733, 737, 739, 771, 773, 777, 779, 791, 793, 797, 799, 911, 913, 917, 919, 931, 933, 937, 939, 971, 973, 977, 979, 991, 993, 997, 999, 3111, 3113, 3117, 3119, 3131, 3133, 3137.
  \item The numbers from 227 to 280, with digits 1, 2, 3, 7 and 9 are: 2913, 2917, 2919, 2921, 2922, 2923, 2927, 2929, 2931, 2932, 2933, 2937, 2939, 2971, 2972, 2973, 2977, 2979, 2991, 2992, 2993, 2997, 2999, 3111, 3112, 3113, 3117, 3119, 3121, 3122, 3123, 3127, 3129, 3131, 3132, 3133, 3137, 3139, 3171, 3172, 3173, 3177, 3179, 3191, 3192, 3193, 3197, 3199, 3211, 3212, 3213, 3217, 3219, 3221.
\end{enumerate}

\section{Romanian Multiplication}

Another algorithm to multiply two integer numbers, $a$ and $b$:
\begin{itemize}
  \item let $k$ be an integer $\ge2$;
  \item write $a$ and $b$ on two dif{}ferent vertical columns: $col(a)$, respectively $col(b)$;
  \item multiply $a$ by $k$, and write the product $a_1$ on the column $col(a)$;
  \item divide $b$ by $k$, and write the integer part of the quotient $b_1$ on the column $col(b)$;
  \item \ldots and so on with the new numbers $a_1$ and $b_1$, until we get a $b_i<k$ on the column $col(b)$;
  \item[] Then:
  \item write another column $\emph{col}(r)$, on the right side of $\emph{col}($b), such that: for each number of column $\emph{col}(b)$, which may be a multiple of $k$ plus the rest $r$ (where $r\in\set{0,1,2, \ldots, k-1}$), the corresponding number on $\emph{col}(r)$ will be $r$;
  \item multiply each number of column $a$ by its corresponding $r$ of $col(r)$, and put the new products on another column $\emph{col}(p)$ on the right side of $\emph{col}(r)$;
  \item f{}inally add all numbers of column $\emph{col}(p)$, $a\times b=$ the sum of all numbers of $\emph{col}(p)$.
\end{itemize}

\begin{rem}
  that any multiplication of integer numbers can be done only by multiplication with 2, 3, \ldots, $k$, divisions by $k$, and additions.
\end{rem}

\begin{rem}
  This is a generalization of Russian multiplication (when $k=2$); we call it \emph{Romanian Multiplication}.

  This special multiplication is useful when $k$ is very small, the best values being for $k=2$ (Russian multiplication -- known since Egyptian time), or $k=3$. If $k$ is greater than or equal to $\min\set{10,b}$, this multiplication is trivial (the obvious multiplication).
\end{rem}

\begin{prog}\label{Program RM} for Romanian Multiplication.
  \begin{tabbing}
    $\emph{RM}(a,b,k):=$\=\vline\ $w\leftarrow (a\ \ b\ \ "="\ \ 0)$\\
    \>\vline\ $r\leftarrow \mod(b,k)$\\
    \>\vline\ $Q\leftarrow(a\ \ b\ \ r\ \ a\cdot r)$\\
    \>\vline\ $w$\=$hile\ \ b>1$\\
    \>\vline\ \>\vline\ $a\leftarrow a\cdot k$\\
    \>\vline\ \>\vline\ $b\leftarrow \emph{floor}\left(\dfrac{b}{k}\right)$\\
    \>\vline\ \>\vline\ $r\leftarrow \mod(b,k)$\\
    \>\vline\ \>\vline\ $Q\leftarrow \emph{stack}[Q,(a\ \ b\ \ r\ \ a\cdot r)]$\\
    \>\vline\ $w_{1,4}\leftarrow\sum Q^{\langle4\rangle}$\\
    \>\vline\ $\emph{return}\ stack(Q,w)$\\
  \end{tabbing}
\end{prog}

\[RM(73,97,2)=\left(\begin{array}{rrcr}
                      73 & 97 & 1 & 73\\
                      146 & 48 & 0 & 0\\
                      292 & 24 & 0 & 0\\
                      584 & 12 & 0 & 0\\
                      1168 & 6 & 0 & 0\\
                      2336 & 3 & 1 & 2336\\
                      4672 & 1 & 1 & 4672\\ \hline
                      73 & \times97 & "=" & 7081\\
                    \end{array}\right)~,
\]
\[
 RM(73,97,3)=\left(\begin{array}{rrcr}
                     73 & 97 & 1 & 73\\
                     219 & 32 & 2 & 438\\
                     657 & 10 & 1 & 657\\
                     1971 & 3 & 0 & 0\\
                     5913 & 1 & 1 & 5913\\ \hline
                     73 & \times97 & "=" & 7081
                   \end{array}\right)~,
\]
\[
 \vdots
\]
\[RM(73,97,10)=\left(\begin{array}{rrcr}
                       73 & 97 & 7 & 511\\
                       730 & 9 & 9 & 6570\\
                       7300 & 0 & 0 & 0\\ \hline
                       73 & \times97 & "=" & 7081\\
                     \end{array}\right)~.
\]
\begin{multline*}
  RM(2346789,345793,10)=\\
  \left(\begin{array}{rrcr}
          2346789 & 345793 & 3 & 7040367\\
          23467890 & 34579 & 9 & 211211010\\
          234678900 & 3457 & 7 & 1642752300\\
          2346789000 & 345 & 5 & 11733945000\\
          23467890000 & 34 & 4 & 93871560000\\
          234678900000 & 3 & 3 & 704036700000\\
          2346789000000 & 0 & 0 & 0\\ \hline
          2346789 & \times345793 & "=" & 811503208677\\
        \end{array}\right)
\end{multline*}

\begin{rem}
  that any multiplication of integer numbers can be done only by multiplication with 2, 3, \ldots, 9, 10, divisions by 10, and additions -- hence we obtain just the obvious multiplication!
\end{rem}

\begin{prog}\label{Program rm} is the variant that displays the intermediate values of the multiplication process.
  \begin{tabbing}
    $rm(a,b,k):=$\=\vline\ $s\leftarrow a\cdot \mod(b,k)$\\
    \>\vline\ $w$\=$hile\ \ b>1$\\
    \>\vline\ \>\vline\ $a\leftarrow a\cdot k$\\
    \>\vline\ \>\vline\ $b\leftarrow\emph{floor}\left(\dfrac{b}{k}\right)$\\
    \>\vline\ \>\vline\ $s\leftarrow s+a\cdot \mod(b,k)$\\
    \>\vline\ $\emph{return}\ \ s$\\
  \end{tabbing}
\end{prog}

Example of calling: $rm(2346789,345793,10)=811503208677$.

\begin{rem}
  that any multiplication of integer numbers can be done only by multiplication with 2, 3, \ldots, 9, 10, divisions by 10, and  additions -- hence we obtain just the obvious multiplication!
\end{rem}

\section{Division with $k$ to the Power $n$}

Another algorithm to divide an integer number $a$ by $k^n$, where $k$, $n$ are integers $ge2$, \cite{Bouvier+Michel1979}:
\begin{itemize}
  \item write $a$ and $k^n$ on two dif{}ferent vertical columns: $col(a)$, respectively $col(k^n)$;
  \item divide $a$ by $k$, and write the integer quotient $a_1$ on the column $col(a)$;
  \item divide $k^n$ by $k$, and write the quotient $q_1= k^{n-1}$ on the column $col(k^n)$;
  \item \ldots and so on with the new numbers $a_1$ and $q_1$, until we get $q_n=1$ ($=k^0$) on the column $col(k^n)$;
  \item[] Then:
  \item write another column $col(r)$, on the left side of $col(a)$, such that: for each number of column $col(a)$, which may be a multiple of $k$ plus the rest $r$ (where $r\in\set{0, 1, 2, ..., k-1}$), the corresponding number on $col(r)$ will be $r$;
  \item write another column $col(p)$, on the left side of $col(r)$, in the following way: the element on line $i$ (except the last line which is 0) will be $k^{n-1}$;
  \item multiply each number of column $col(p)$ by its corresponding $r$ of $col(r)$, and put the new products on another column $col(r)$ on the left side of $col(p)$;
  \item f{}inally add all numbers of column $col(r)$ to get the f{}inal rest $r^n$, while the f{}inal quotient will be stated in front of $col(k^n)$\rq{s} 1. Therefore:
  \[
   \frac{a}{k^n}=a_n\ \ \textnormal{and rest}\ \ r_n~.
  \]
\end{itemize}

\begin{rem}
  that any division of an integer number by $k^n$ can be done only by divisions to $k$, calculations of powers of $k$,
  multiplications with 1, 2, \ldots, $k-1$, additions.
\end{rem}

\begin{prog}\label{Program Dkn} for division calculation of a positive integer number $k$ of power $n$ where $k,n$ are integers $\ge2$.
  \begin{tabbing}
    $\emph{Dkn}(a,k,n):=$\=\vline\ $c_{1,1}\leftarrow a$\\
    \>\vline\ $f$\=$or\ j\in1..n$\\
    \>\vline\ \>\vline\ $c_{j,2}\leftarrow \mod(c_{j,1},k)$\\
    \>\vline\ \>\vline\ $c_{j,3}\leftarrow k^{j-1}\cdot c_{j,2}$\\
    \>\vline\ \>\vline\ $c_{j+1,1}\leftarrow \emph{floor}\left(\dfrac{c_{j,1}}{k}\right)$\\
    \>\vline\ $c_{j+1,2}\leftarrow "\emph{rest}"$\\
    \>\vline\ $c_{j+1,3}\leftarrow\sum c^{\langle3\rangle}$\\
    \>\vline\ $\emph{return}\ \ c$\\
   \end{tabbing}
\end{prog}
The program call $\emph{Dkn}$, \ref{Program Dkn}, for dividing 13537 to $2^7$:
\[\emph{Dkn}(13537,2,7)=\left(\begin{array}{rcr}
                         13537 & 1 & 1\\
                         6768 & 0 & 0\\
                         3384 & 0 & 0\\
                         1692 & 0 & 0\\
                         846 & 0 & 0\\
                         423 & 1 & 32\\
                         211 & 1 & 64\\ \hline
                         105 & "rest" & 97\\
                       \end{array}\right)
\]
The program call $\emph{Dkn}$, \ref{Program Dkn}, for dividing 21345678901 to $3^9$:
\[\emph{Dkn}(21345678901,3,9)=\left(\begin{array}{rcr}
                               21345678901 & 1 & 1\\
                               7115226300 & 0 & 0\\
                               2371742100 & 0 & 0\\
                               790580700 & 0 & 0\\
                               263526900 & 0 & 0\\
                               87842300 & 2 & 486\\
                               29280766 & 1 & 729\\
                               9760255 & 1 & 2187\\
                               3253418 & 2 & 13122\\ \hline
                               1084472 & "rest" & 16525
                             \end{array}\right)
\]
The program call $\emph{Dkn}$, \ref{Program Dkn}, using symbolic computation, for dividing 2536475893647585682919172 to $11^{13}$:
\begin{multline*}
  \emph{Dkn}(2536475893647585682919172,11,13)\rightarrow \\
  \left(\begin{array}{rcr}
          2536475893647585682919172 & 2 & 2\\
          230588717604325971174470 & 10 & 110\\
          20962610691302361015860 & 2 & 242\\
          1905691881027487365078 & 2 & 2662\\
          173244716457044305916 & 7 & 102487\\
          15749519677913118719 & 6 & 966306\\
          1431774516173919883 & 9 & 15944049\\
          130161319652174534 & 9 & 175384539\\
          11832847241106775 & 4 & 857435524\\
          1075713385555161 & 1 & 2357947691\\
          97792125959560 & 10 & 259374246010\\
          8890193269050 & 5 & 1426558353055\\
          808199388095 & 0 & 0\\ \hline
          73472671645 & "\emph{rest}" & 1689340382677\\
        \end{array}\right)
\end{multline*}

\begin{prog}\label{Program dkn} for dividing an integer with $k^n$, where $k,n\in\Ns$, $k,n\ge2$, without
displaying intermediate results of the division.
  \begin{tabbing}
    $\emph{dkn}(a,k,n):=$\=\vline\ $R\leftarrow0$\\
    \>\vline\ $f$\=$or\ j\in1..n$\\
    \>\vline\ \>\vline\ $r\leftarrow \mod(a,k)$\\
    \>\vline\ \>\vline\ $R\leftarrow R+k^{j-1}\cdot r$\\
    \>\vline\ \>\vline\ $a\leftarrow \emph{floor}\left(\dfrac{a}{k}\right)$\\
    \>\vline\ $\emph{return}\ \ (a\ \ "\emph{rest}"\ \ R)$\\
  \end{tabbing}
\end{prog}
Examples of dialing the program $\emph{dkn}$, \ref{Program dkn}:
\[
 dkn(13537,2,7)=(105\ \ "rest"\ \ 97)~,
\]
\begin{multline*}
  dkn(2536475893647585682919172,11,13)\rightarrow \\
  (73472671645\ \ "\emph{rest}"\ \ 1689340382677)
\end{multline*}

\section{Generalized Period}

Let $M$ be a number in a base $b$. All distinct digits of $M$ are named generalized period of $M$. For example, if $M=104001144$, its generalized period is $g(M)=\set{0,1,4}$. Of course, $g(M)$ is included in $\set{0,1,2, \ldots, b-1}$.

The number of generalized periods of $M$ is equal to the number of the groups of $M$ such that each group contains all distinct digits of $M$. For example, $n_g(M)=2$ because
\[
 M=\underbrace{104}_1\underbrace{001144}_2~.
\]

Length of generalized period is equal to the number of its distinct digits. For example, $l_g(M)=3$.

Questions:
\begin{enumerate}
  \item Find $n_g$, $l_g$ for $p_n$ , $n!$, $n^n$, $\sqrt[n]{n}$.
  \item For a given $k\ge1$, is there an inf{}inite number of primes $p_n$, or $n!$, or $n^n$, or $\sqrt[n]{n}$ which have a generalized period of length $k$? Same question such that the number of generalized periods be equal to $k$.
  \item Let $a_1$, $a_2$ , \ldots, $a_h$ be distinct digits. Is there an inf{}inite number of primes $p_n$, or $n!$, or $n^n$, or $\sqrt[n]{n}$ which have as a generalized period the set $\set{a_1, a_2, \ldots, a_h}$ ?
\end{enumerate}

\section{Prime Equation Conjecture}

Let $k>0$ be an integer. There is only a f{}inite number of solutions in integers $p$, $q$, $x$, $y$, each greater than 1, to the equation
\begin{equation}\label{xlapminusylaqegalk}
  x^p-y^q=k~.
\end{equation}
For k = 1 this was conjectured by \cite{Casseles1953}\index{Casseles J. W. S.} and proved by \cite{Tijdeman1976}\index{Tijdeman R.}, \citep{Smarandache1993,Ibstedt1997}.

\begin{lem}\label{Lemma6.1}
  Let $,q\ge2$ be integers and suppose that $x,y$ are nonzero integer that are a solution to equation $x^p-y^q=1$. Then $p$ and $q$ are necessarily distinct.
\end{lem}
Cassells\rq{}\index{Casseles J. W. S.} theorem is concerned with Catalan\rq{s}\index{Catalan E.} equation for the odd prime exponents $p$ and $q$. We f{}irst prove the easy part of this result.
\begin{prop}\label{Proposition6.2}
  Let $p>q$ be two odd primes and suppose that $x,y$ are nonzero integers for which $x^p-y^q=1$. Then both of the following hold:
  \begin{enumerate}
    \item $q\mid x$~;
    \item $\abs{x}\ge q+q^{p-1}$~.
  \end{enumerate}
\end{prop}

\begin{prog}\label{Program Pec} for determining all solutions of the equation (\ref{xlapminusylaqegalk}) for $p$ and $q$ give and $k\in\set{a_k,a_k+1,\ldots,b_k}$, $y\in\set{a_y,a_y+1,\ldots,b_y}$.
  \begin{tabbing}
    $\emph{Pec}(p,a_y,b_y,q,a_k,b_k):=$\=\vline\ $\emph{return}\ "\emph{Error.}"\ \ \emph{if}\ \ p\le q\vee a_y\le b_y\vee a_k\le b_k$\\
    \>\vline\ $S\leftarrow("x"\ \ "p"\ \ "y"\ \ "q"\ \ "k")$\\
    \>\vline\ $f$\=$or\ k\in a_k..b_k$\\
    \>\vline\ \>\ $f$\=$or\ y\in a_y..b_y$\\
    \>\vline\ \>\ \>\vline\ $xr\leftarrow\sqrt[p]{k+y^q}$\\
    \>\vline\ \>\ \>\vline\ $f$\=$or\ \emph{floor}(xr)..\emph{ceil}(xr)$\\
    \>\vline\ \>\ \>\vline\ \>\ $S\leftarrow \emph{stack}[S,(x\ \ p\ \ y\ \ q\ \ k)]\ \ \emph{if}\ \ x^p-y^q\textbf{=}k$\\
    \>\vline\ $\emph{return}\ \ S$\\
  \end{tabbing}
\end{prog}

Calling the program $\emph{Pec}$ by command $\emph{Pec}(5,2,10^3,3,19,2311)=$:
\begin{center}
 \begin{longtable}{|r|r|r|r|r|}
   \caption{The solutions of the equation (\ref{xlapminusylaqegalk})}\\
   \hline
   "x" & "p" & "y" & "q" & "k"\\
   \hline
  \endfirsthead
   \hline
   "x" & "p" & "y" & "q" & "k"\\
   \hline
  \endhead
   \hline \multicolumn{5}{r}{\textit{Continued on next page}} \\
  \endfoot
   \hline
  \endlastfoot
          2 & 5 & 2 & 3 & 24\\
          4 & 5 & 10 & 3 & 24\\
          3 & 5 & 6 & 3 & 27\\
          3 & 5 & 5 & 3 & 118\\
          3 & 5 & 4 & 3 & 179\\
          3 & 5 & 3 & 3 & 216\\
          3 & 5 & 2 & 3 & 235\\
          4 & 5 & 9 & 3 & 295\\
          5 & 5 & 14 & 3 & 381\\
          4 & 5 & 8 & 3 & 512\\
          4 & 5 & 7 & 3 & 681\\
          4 & 5 & 6 & 3 & 808\\
          4 & 5 & 5 & 3 & 899\\
          6 & 5 & 19 & 3 & 917\\
          5 & 5 & 13 & 3 & 928\\
          4 & 5 & 4 & 3 & 960\\
          4 & 5 & 3 & 3 & 997\\
          4 & 5 & 2 & 3 & 1016\\
          7 & 5 & 25 & 3 & 1182\\
          5 & 5 & 12 & 3 & 1397\\
          23 & 5 & 186 & 3 & 1487\\
          5 & 5 & 11 & 3 & 1794\\
          6 & 5 & 18 & 3 & 1944\\
          5 & 5 & 10 & 3 & 2125\\
 \end{longtable}
\end{center}

\subsection{Generalized Prime Equation Conjecture}

Let $m\ge2$ be a positive integer. The Diophantine equation
\begin{equation}\label{EcuatiaGeneralizataCatalan}
  y=2\cdot x_1\cdot x_2\cdots x_m+k~,
\end{equation}
has inf{}initely many solutions in distinct primes $y$, $x_1$, $x_2$, \ldots, $x_m$.

Let us remark that $y\in2\Ns+1$ and the unknowns $x_1$, $x_2$, \ldots, $x_m$ have a similar role.

\begin{prog}\label{Program Pecg3} for complete solving the equation (\ref{EcuatiaGeneralizataCatalan}) for $m=3$.
  \begin{tabbing}
    $\emph{Pecg3}(y,k):=$\=\vline\ $S\leftarrow("x1"\ \ "x2"\ \ "x3")$\\
    \>\vline\ $f$\=$or\ x_1\in1..y-k-2$\\
    \>\vline\ \>\ $f$\=$or\ x_2\in1+x_1..y-k-1$\\
    \>\vline\ \>\ \>\ $f$\=$or\ x_3\in1+x_2..y-k$\\
    \>\vline\ \>\ \>\ \>\ $S\leftarrow \emph{stack}[S,(x_1\ \ x_2\ \ x_3)]\ \ \emph{if}\ \ 2\cdot x_1\cdot x_2\cdot x_3+k=y$\\
    \>\vline\ $\emph{return}\ \ S$\\
   \end{tabbing}
   The program is so designed as to avoid getting trivial solutions (for example $x_1=1$, $x_2=1$ and $x_3=(y-k)/2$) and symmetrical solutions (for example for $y=13$ we would have the solutions $x_1=1$, $x_2=2$ and $x_3=3$ but any permutation between these values would be solutions of the equation $2x_1x_2x_3+1=13$).
\end{prog}

Examples of calling the program $\emph{Pecg3}$:
\[Pecg3(649,1)=\left(\begin{array}{ccc}
                     "x1" & "x2" & "x3"\\
                     1 & 2 & 162\\
                     1 & 3 & 108\\
                     1 & 4 & 81\\
                     1 & 6 & 54\\
                     1 & 9 & 36\\
                     1 & 12 & 27\\
                     2 & 3 & 54\\
                     2 & 6 & 27\\
                     2 & 9 & 18\\
                     3 & 4 & 27\\
                     3 & 6 & 18\\
                     3 & 9 & 12\\
                   \end{array}\right)
\]
\[Pecg3(649,19)=\left(\begin{array}{ccc}
                       "x1" & "x2" & "x3"\\
                       1 & 3 & 105\\
                       1 & 5 & 63\\
                       1 & 7 & 45\\
                       1 & 9 & 35\\
                       1 & 15 & 21\\
                       3 & 5 & 21\\
                       3 & 7 & 15\\
                       5 & 7 & 9\\
                     \end{array}\right)
\]

For $k=4$, $k=5$, \ldots\ similar programs can be written to determine all untrite and unsymmetrical solutions.

\chapter{Documents Mathcad}
\begin{figure}
  \centering
  \rotatebox{90}{\includegraphics[scale=0.7]{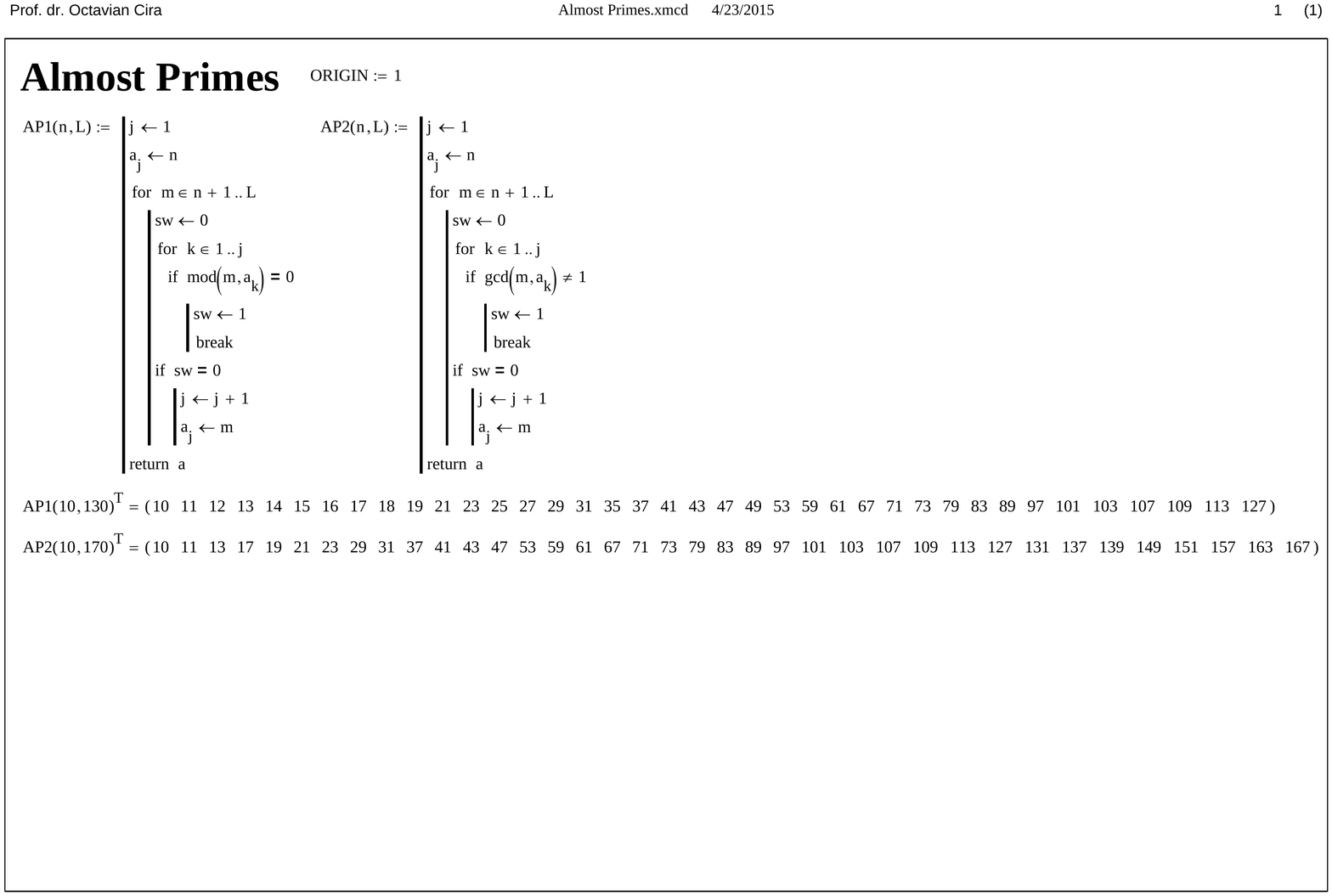}}\\
  \caption{The document Mathcad Almost Primes}\label{MathcadAlmostaPrimes}
\end{figure}
\newpage
\begin{figure}
  \centering
  \rotatebox{90}{\includegraphics[scale=0.7]{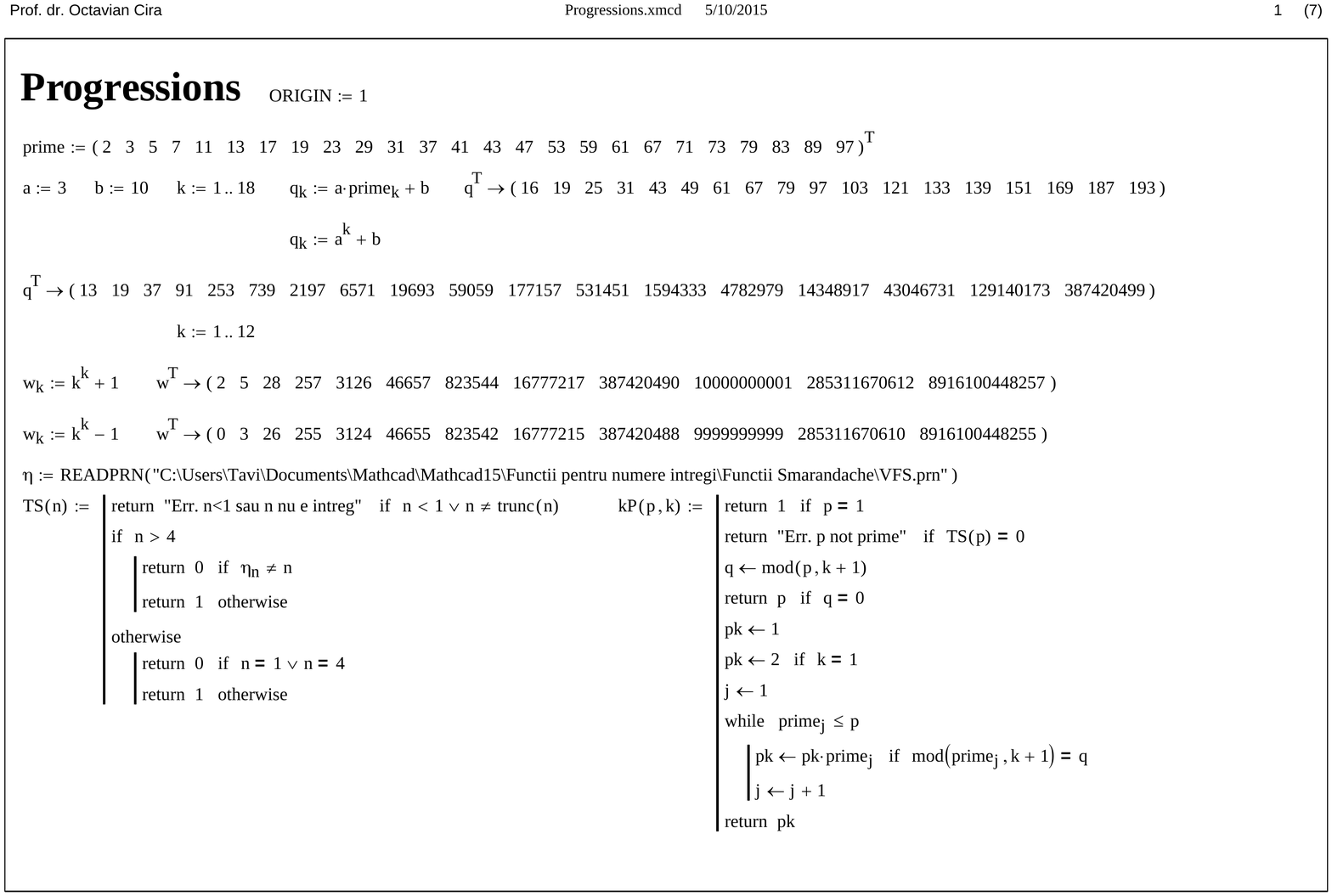}}\\
  \caption{The document Mathcad Progression}\label{MathcadProgressions}
\end{figure}
\newpage
\begin{figure}
  \centering
  \rotatebox{90}{\includegraphics[scale=0.7]{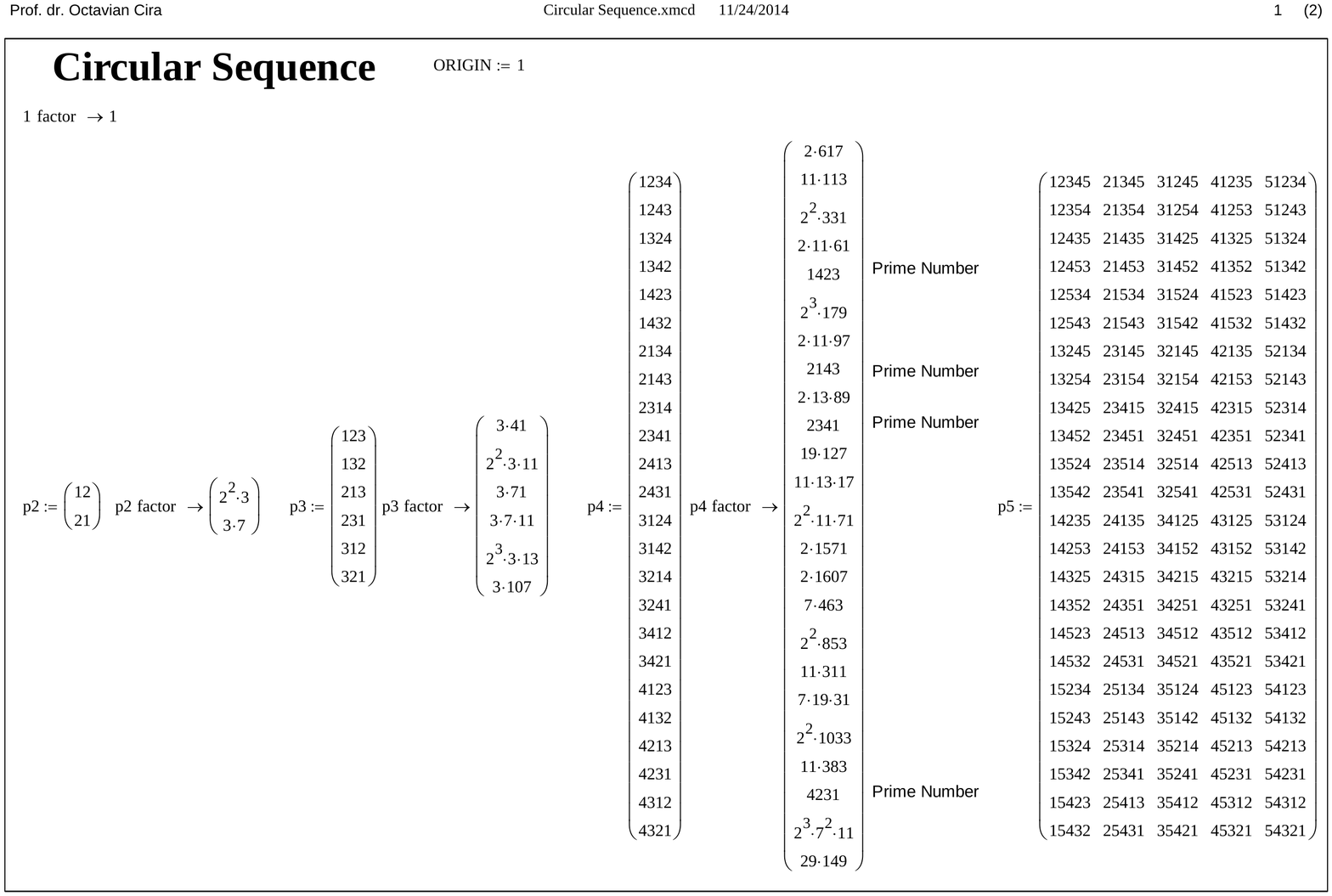}}\\
  \caption{The document Mathcad Circular Sequence}\label{MathcadCircularSequence}
\end{figure}
\newpage
\begin{figure}
  \centering
  \rotatebox{90}{\includegraphics[scale=0.7]{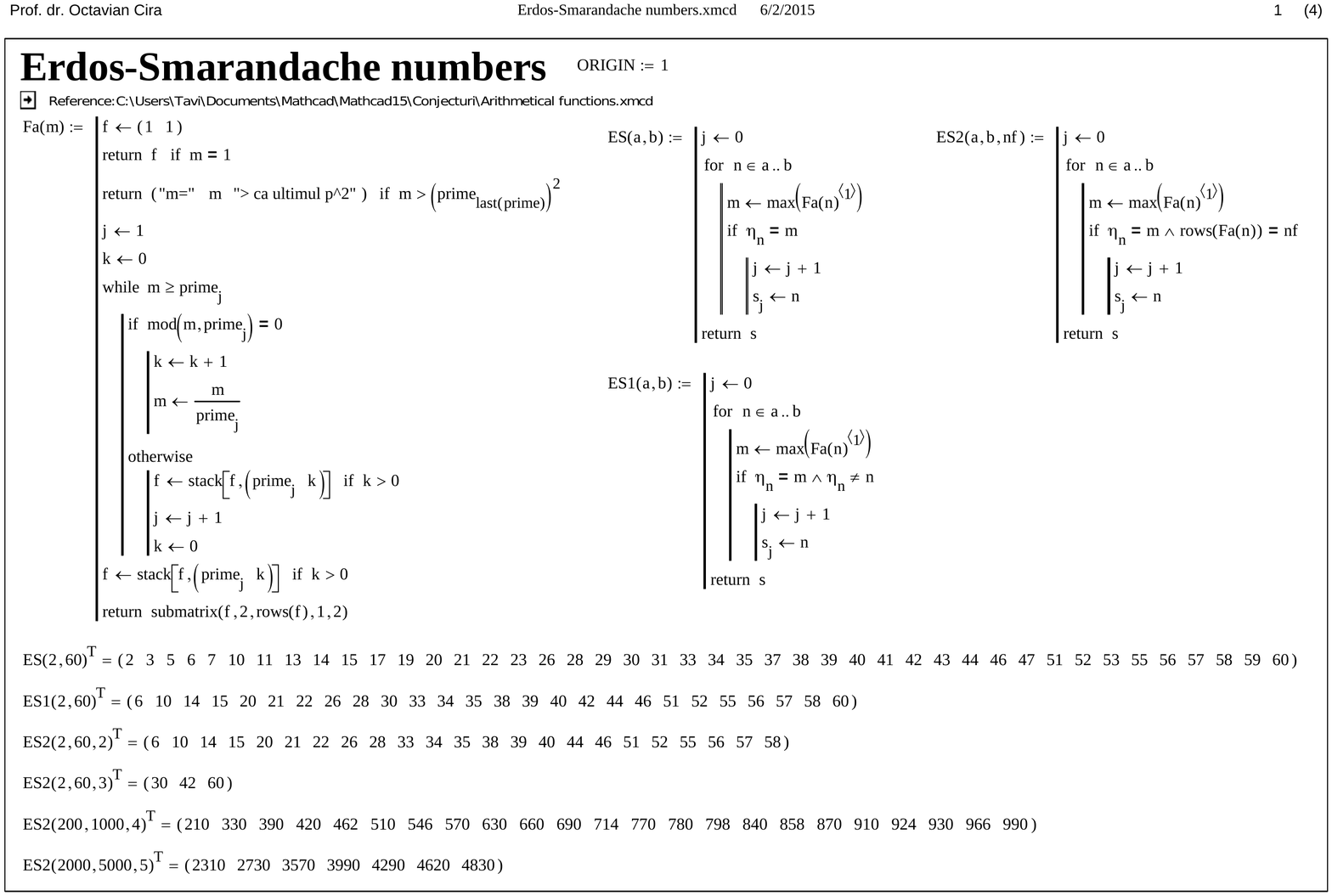}}\\
  \caption{The document Mathcad Erdos--Smarandache numbers}\label{Mathcad-Erdos-Smarandachenumbers}
\end{figure}
\newpage
\begin{figure}
  \centering
  \rotatebox{90}{\includegraphics[scale=0.7]{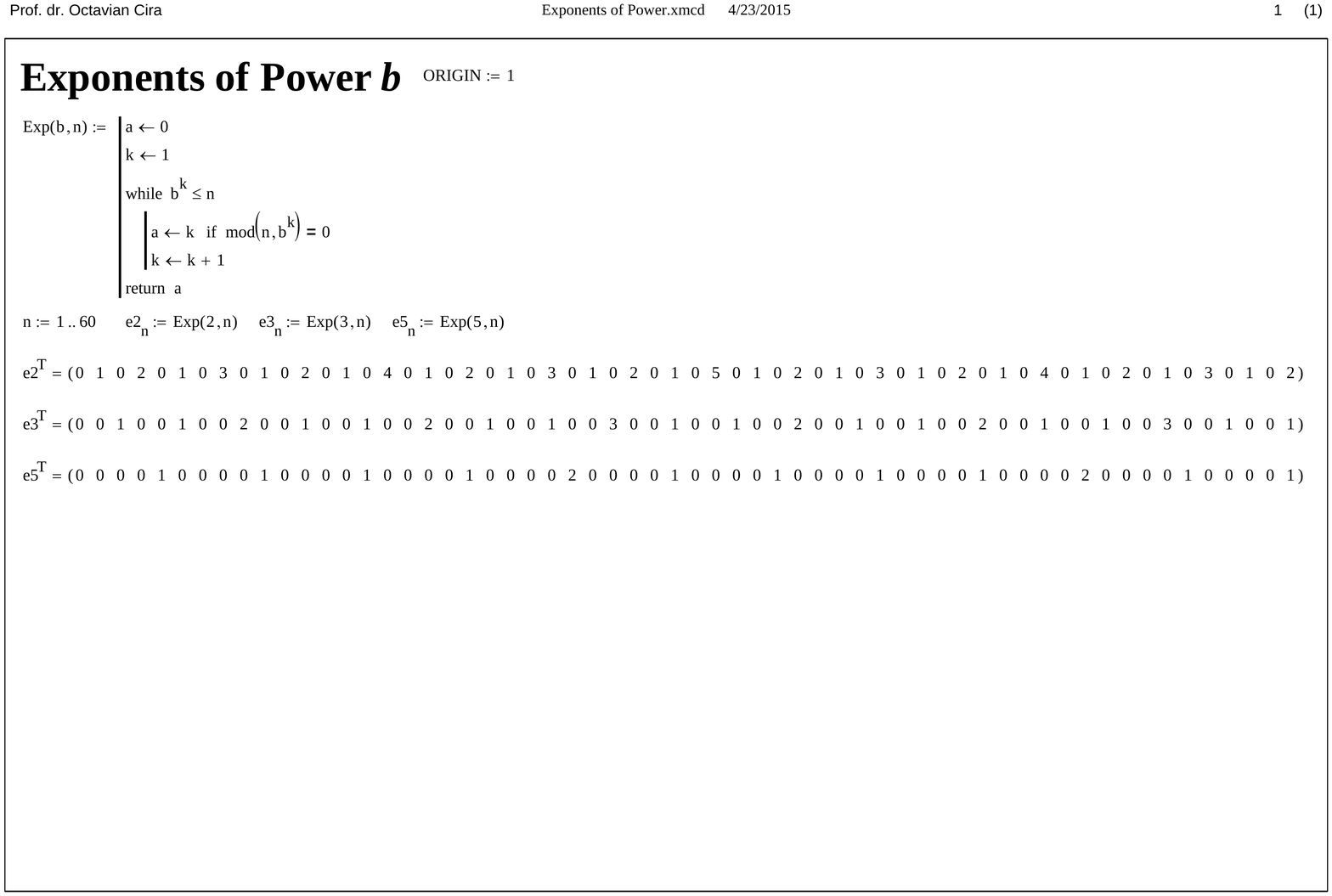}}\\
  \caption{The document Mathcad Exponents of Power}\label{Mathcad-ExponentsofPower}
\end{figure} 

\backmatter
\chapter{Indexes}

\section*{Index of notations}
\small\begin{description}
  \item $\Na=\set{0,1,2,\ldots}$;
  \item $\Ns=\Na\setminus\set{0}=\set{1,2,\ldots}$;
  \item $\NP{2}=\set{2,3,5,7,11,13,\ldots}$, $\NP{3}=\set{3,5,7,11,13,\ldots}$;
  \item $\Ind{s}=\set{1,2,\ldots,s}$ : the set of indexes;
  \item $\Real$ : the real numbers;
  \item $\pi(x)$ : the number of prime numbers up to $x$;
  \item $[x]$ : the integer part of number $x$;
  \item $\{x\}$ : the fractional part of $x$;
  \item $\sigma_k(n)$ : the sum of the powers of order $k$ of the divisors of $n$;
  \item $\sigma(n)$ : the sum of the divisors of $n$; $\sigma(n)=\sigma_1(n)$;
  \item $s(n)$ : the sum of the divisors of $n$ without $n$; $s(n)=\sigma(n)-n$;
  \item $\lfloor a\rfloor$ : the lower integer part of $a$; the greatest integer, smaller than $a$;
  \item $\lceil a\rceil$ : the upper integer part of $a$; the smallest integer, greater than $a$;
  \item $n\mid m$ : $n$ divides $m$;
  \item $n\nmid m$ : $n$ does not divide $m$;
  \item $(m,n)$ : the greatest common divisor of $m$ and $n$; $(m,n)=gcd(m,n)$;
  \item $[m,n]$ : the smallest common multiple of $m$ and $n$; $[m,n]=lcd(m,n)$;
\end{description}

\newpage
\section*{Mathcad Utility Functions}

{\small\begin{description}
  \item[]
  \item[]
  \item $augment(M,N)$ : concatenates matrices $M$ and $N$ that have the same number of lines;
  \item $ceil(x)$ : the upper integer part function;
  \item $cols(M)$ : the number of columns of matrix $M$;
  \item $eigenvals(M)$ : the eigenvalues of matrix $M$;
  \item $eigenvec(M,\lambda)$ : the eigenvector of matrix $M$ relative to the eigenvalue $\lambda$;
  \item $eigenvecs(M)$ : the matrix of the eigenvectors of matrix $M$;
  \item $n\ factor\rightarrow$ : symbolic computation function that factorizes $n$;
  \item $floor(x)$ : the lower integer part function;
  \item $gcd(n_1,n_2,\ldots)$ : the function which computes the greatest common divisor of $n_1,n_2,\ldots$;
  \item $last(v)$ : the last index of vector $v$;
  \item $lcm(n_1,n_2,\ldots)$ : the function which computes the smallest common multiple of $n_1,n_2,\ldots$;
  \item $max(v)$ : the maximum of vector $v$;
  \item $min(v)$ : the minimum of vector $v$;
  \item $\mod(m,n)$ : the rest of the division of $m$ by $n$;
  \item $ORIGIN$ : the variable dedicated to the origin of indexes, $0$ being an implicit value;
  \item $rref(M)$ : determines the matrix \emph{row-reduced echelon form};
  \item $reverse(M)$ : reverses the order of elements in a vector, or of rows in a matrix $M$.;
  \item $rows(M)$ : the number of lines of matrix $M$;
  \item $solve$ : the function of symbolic solving the equations;
  \item $stack(M,N)$ : concatenates matrices $M$ and $N$ that have the same care number of columns;
  \item $submatrix(M,k_r,j_r,k_c,j_c)$ : extracts from matrix $M$, from line $k_r$ to line $j_r$ and from column $k_c$ to column $j_c$, a submatrix;
  \item $trunc(x)$ : the truncation function;
  \item $\sum v$ : the function that sums the components of vector $v$~.
\end{description}}

\newpage
\section*{Mathcad User Arithmetical Functions}

{\small\begin{description}
  \item[]
  \item[]
  \item $\emph{conc}$ : concatenation function of two numbers in numeration base 10, \ref{Functia conc};
  \item $\emph{concM}$ : concatenation program in base 10 of all elements on a line, for all of matrix lines, \ref{Program concM};
  \item $\emph{ConsS}$ : program for generating series of consecutive sieve, \ref{ProgramConsS};
  \item $\emph{cpi}$ : function inferior fractional cubic part, \ref{Functia cpi};
  \item $\emph{cps}$ : function superior fractional cubic part, \ref{Functia cps};
  \item $\emph{dks}$ : function of digit--summing in base $b$ of power $k$ of the number $n$ written in base $10$, \ref{Functia dks};
  \item $\emph{dn}$ : program for providing digits in base $b$, \ref{ProgramDn};
  \item $\emph{dp}$ : function for calculation of the digit--product of the number $n_{(b)}$, \ref{ProgramDp};
  \item $\emph{fpi}$ : the inferior factorial dif{}ference part function, \ref{Functia fpi};
  \item $\emph{fps}$ : the superior factorial dif{}ference part function, \ref{Functia fps};
  \item $\emph{kConsS}$ : program for generating the series of $k-$ary consecutive sieve, \ref{ProgramConS};
  \item $\emph{kf}$ : function for calculating the multifactorial, \ref{FunctiaKF};
  \item $\emph{icp}$ : function inferior cubic part, \ref{Functia icp};
  \item $\emph{ifp}$ : function inferior factorial part, \ref{Programifp};
  \item $\emph{isp}$ : function inferior square part, \ref{Functia isp};
  \item $\emph{ip}$ : function inferior function part, \ref{ProgramIp};
  \item $\emph{ipp}$ : inferior prime part function, \ref{Programipp};
  \item $\emph{nfd}$ : program for counting unit of digits of prime numbers, \ref{Programnfd};
  \item $\emph{nPS}$ : program for generating the series $n-$ary power sieve; \ref{ProgramnPS};
  \item $\emph{nrd}$ : function for counting the digits of the number $n_{(10)}$ in base $b$, \ref{FunctionNrd}
  \item $\emph{TS}$ : the program for $S$ (Smarandache function) primality test, \ref{Program TS};
  \item $\emph{pL}$ : the program for determining prime numbers Luhn of the order $o$, \ref{ProgramLuhn};
  \item $\emph{ppi}$ : f{}irst inferior dif{}ference part function, \ref{Functia ppi};
  \item $\emph{pps}$ : f{}irst superior dif{}ference part function, \ref{Functia pps};
  \item $\emph{Psp}$ : program for determining the numbers $sum-product$ in base $b$, \ref{ProgramPsp};
  \item $P_d$ : function product of all positive divisors of $n_{(10)}$;
  \item $P_{\emph{dp}}$ : function product of all positive proper divisor of $n_{(10)}$;
  \item $\emph{Reverse}$ : function returning the inverse of number $n_{(10)}$ in base $b$, \ref{FunctiaReverse};
  \item $\emph{SAOC}$ : sieve of Atkin Optimized by Cira for generating prime numbers, \ref{ProgAtkin};
  \item $\emph{Sgm}$ : program providing the series of maximal gaps;
  \item $\emph{SEPC}$ : sieve of Erathostenes, linear version of Prithcard, optimized of Cira for generating prime numbers, \ref{ProgramSEPC};
  \item $\emph{sp}$ : function digital product in base $b$ of number $n_{(10)}$, \ref{FunctiaSum-Product};
  \item $\emph{scp}$ : function superior cubic part, \ref{Functia scp};
  \item $\emph{sfp}$ : function superior factorial part, \ref{Programsfp};
  \item $\emph{spi}$ : function inferior fractional square part, \ref{Functia spi};
  \item $\emph{sps}$ : function superior fractional square part, \ref{Functia sps};
  \item $\emph{spp}$ : function superior prime part, \ref{Programspp};
  \item $\emph{SS}$ : sieve of Sundaram for generating prime numbers, \ref{ProgSundram};
  \item $\emph{ssp}$ : function superior square part, \ref{Functia ssp};
  \item $Z_1$ : function pseudo-Smarandache of the order 1, \ref{ProgramZ1};
  \item $Z_2$ : function pseudo-Smarandache of the order 2, \ref{ProgramZ2};
  \item $Z_3$ : function pseudo-Smarandache of the order 3, \ref{ProgramZ3};
\end{description}}

\newpage
\section*{Mathcad Engendering Programs}

{\small\begin{description}
  \item[]
  \item[]
  \item $\emph{GMC}$ : program \ref{Program GMC} for generating cellular matrices;
  \item $\emph{GR}$ : program  \ref{ProgramGR} for generating general residual numbers;
  \item $\emph{GSt}$ : program \ref{Program GSt} for generating the Goldbach–-Smarandache table;
  \item $\emph{mC}$ : program \ref{Program m-power Complements} for generating free series of numbers of power $m$;
  \item $\emph{NGSt}$ : program \ref{Program NGSt} which determines all the possible combinations (irrespective of
addition commutativity) of sums of two primes that are equal to the given even number;
  \item $\emph{NVSt}$ : program \ref{Program NVSt} for counting of decompositions of $n$ (natural odd number $n\ge3$) in sums of three primes;
  \item $\emph{Pascal}$ : program \ref{Program Combinarilor} for generating the triangle of Pascal matrix;
  \item $\emph{mC}$ : program \ref{Program m-power Complements} for generating $m-$power complements\rq{} numbers;
  \item $\emph{mfC}$ : program \ref{Program m-factorial complements} for generating the series of $m-$factorial complements;
  \item $\emph{MPrime}$ : program \ref{Program MPrime} for generating primes matrix;
  \item $\emph{paC}$ : program \ref{Program prime additive Complements} for generating the series of additive complements primes;
  \item $\emph{P2Z1}$ : program  \ref{Program P2Z1} for determining Diophantine Z equations\rq{} solutions $Z_1^2(n)=n$;
  \item $\emph{SVSt}$ : program \ref{Program SVSt} for determining all possible combinations in the Vinogradov–-Smarandache tables, such as the odd number to be written as a sum of 3 prime numbers
  \item $\emph{VSt}$ : program \ref{Program VSt} for generating Vinogradov--Smarandache table;
\end{description}}

\renewcommand\indexname{\LARGE{Index of name}}
\printindex

\bibliographystyle{plainnat}
\bibliography{VAFA}
\end{document}